%% file: root.tex
\documentclass[12pt,oneside,final]{thesis}
\usepackage{enumitem}
\usepackage[titletoc]{appendix}

\usepackage{algcompatible}

\usepackage{rotating}

\usepackage{natbib}

\usepackage{afterpage}

\usepackage{calc}

\usepackage{booktabs}
\usepackage{siunitx}

\makeatletter
\newcommand{\DESCRIPTION@original@item}{}
\let\DESCRIPTION@original@item\item
\newcommand*{\DESCRIPTION@envir}{DESCRIPTION}
\newlength{\DESCRIPTION@totalleftmargin}
\newlength{\DESCRIPTION@linewidth}
\newcommand{\DESCRIPTION@makelabel}[1]{\llap{#1}}%
\newcommand{\DESCRIPTION@item}[1][]{%
  \setlength{\@totalleftmargin}%
       {\DESCRIPTION@totalleftmargin+\widthof{\textbf{#1 }}-\leftmargin}%
  \setlength{\linewidth}
       {\DESCRIPTION@linewidth-\widthof{\textbf{#1 }}+\leftmargin}%
  \par\parshape \@ne \@totalleftmargin \linewidth
  \DESCRIPTION@original@item[\textbf{#1}]%
}
\newenvironment{DESCRIPTION}
  {\list{}{\setlength{\labelwidth}{0cm}%
           \let\makelabel\DESCRIPTION@makelabel}%
   \setlength{\DESCRIPTION@totalleftmargin}{\@totalleftmargin}%
   \setlength{\DESCRIPTION@linewidth}{\linewidth}%
   \renewcommand{\item}{\ifx\@currenvir\DESCRIPTION@envir
                           \expandafter\DESCRIPTION@item
                        \else
                           \expandafter\DESCRIPTION@original@item
                        \fi}}
  {\endlist}
\makeatother

\usepackage{mathtools}
\DeclarePairedDelimiter\floor{\lfloor}{\rfloor}
\DeclarePairedDelimiter\ceil{\lceil}{\rceil}
\DeclareMathOperator{\Var}{Var}

\usepackage[english]{babel}
\usepackage{blindtext}

\usepackage{tikz}
\usepackage{caption}
\usepackage{subcaption}
\usetikzlibrary{positioning,patterns, arrows.meta}
\usetikzlibrary{positioning}
\usepackage{pgfplots}
  \pgfplotsset{compat=newest}
  \usetikzlibrary{plotmarks}
  \usetikzlibrary{arrows.meta}
  \usepgfplotslibrary{patchplots}
  \usepackage{grffile}
\usetikzlibrary{calc,fadings}
\usetikzlibrary{
    shapes.geometric,
    decorations.pathreplacing
    }
    
    \tikzset{
    elli/.style args={#1:#2and#3}{
        draw,
        shape=ellipse,
        rotate=#1,
        minimum width=2*#2,
        minimum height=2*#3,
        outer sep=0pt,
    },
    /pgf/decoration/raise/.append code={
        \def\tikzdecorationsbrace{#1}
    },
    elli node/.style={
        circle,
        black,
        draw=none,
        midway,
        anchor=#1-90,
        inner sep=0pt,
        shift=(#1+90:\tikzdecorationsbrace+\pgfdecorationsegmentamplitude)
    },
    eigen/.style 2 args={
        decorate,
        decoration={
            brace,
            amplitude=#1,
            mirror,
            raise=#2,
        },
    },
    eigen/.default={15pt}{4pt},
    axis/.style={
        line width=.5mm,
        ->,
    },
    normal axis/.style={
        axis,
        dashed,
    }
}

\pgfplotsset{compat=1.12}

\newcommand\pgfmathsinandcos[3]{%
  \pgfmathsetmacro#1{sin(#3)}%
  \pgfmathsetmacro#2{cos(#3)}%
}

\newcommand\LatitudePlane[3][current plane]{%
  \pgfmathsinandcos\sinEl\cosEl{#2} 
  \pgfmathsinandcos\sint\cost{#3} 
  \pgfmathsetmacro\yshift{\cosEl*\sint}
  \tikzset{#1/.style={cm={\cost,0,0,\cost*\sinEl,(0,\yshift)}}} %
}

\newcommand\DrawBackLatitudeCircle[2][1]{
  \LatitudePlane{\angEl}{#2}
  \tikzset{current plane/.prefix style={scale=#1}}
  \pgfmathsetmacro\sinVis{sin(#2)/cos(#2)*sin(\angEl)/cos(\angEl)}
  \pgfmathsetmacro\angVis{asin(min(1,max(\sinVis,-1)))}
  \draw[current plane,dashed,black, ultra thick] (180-\angVis:1) arc (180-\angVis:\angVis:1);
}

\newcommand\DrawFrontLatitudeCircle[2][1]{
  \LatitudePlane{\angEl}{#2}
  \tikzset{current plane/.prefix style={scale=#1}}
  \pgfmathsetmacro\sinVis{sin(#2)/cos(#2)*sin(\angEl)/cos(\angEl)}
  \pgfmathsetmacro\angVis{asin(min(1,max(\sinVis,-1)))}
  \draw[current plane,black, ultra thick] (\angVis:1) arc (\angVis:-\angVis-180:1);
}

\tikzset{%
  >=latex,
  inner sep=0pt,%
  outer sep=2pt,%
  mark coordinate/.style={inner sep=0pt,outer sep=0pt,minimum size=3pt,
    fill=black,circle}%
    }

\usepackage{cancel}

\usepackage{cite}
\usepackage{amsmath,amsfonts}

\usepackage{amsthm}
\theoremstyle{definition}
\newtheorem{theorem}{Theorem}
\newtheorem{corollary}{Corollary}
\newtheorem{lemma}{Lemma}

\newtheorem{proposition}{Proposition}
\newtheorem{definition}{Definition}
\theoremstyle{remark}

\theoremstyle{remark}

\usepackage{graphicx}
\graphicspath{{./figs/}}
\usepackage{array}
\usepackage{wrapfig} 
\usepackage{newcent}
\usepackage{subcaption}

\usepackage[linktocpage=true]{hyperref} 

\usepackage{upgreek} 
\usepackage{bm}
\usepackage{hyperref}
\usepackage{setspace}
\usepackage{booktabs}
\usepackage{multirow}
\usepackage{longtable}
\usepackage[font=singlespacing, labelfont=bf]{caption}

\usepackage{color}
\definecolor{mygreen}{RGB}{156, 32, 32}

\usepackage{enumitem}
\newlist{inlinelist}{enumerate*}{1}
\setlist*[inlinelist,1]{%
  label=(\arabic*),
}

\usepackage{fancyhdr}      

\fancypagestyle{plain}{
\fancyhf{} 
\fancyfoot[C]{\thepage} 

}

\usepackage{amsthm}

\DeclareMathOperator{\trace}{trace}

\usepackage{algcompatible}
\usepackage{algorithm}

\newcommand{\vertiii}[1]{{\left\vert\kern-0.25ex\left\vert\kern-0.25ex\left\vert #1 
    \right\vert\kern-0.25ex\right\vert\kern-0.25ex\right\vert}}

\pdfoutput=1
\usepackage{hyperref}
\hypersetup{
  pdfinfo={
    Title={Cyclic Stochastic Optimization: Generalizations, Convergence, and Applications in Multi-Agent Systems},
    Author={Karla Hernandez Cuevas}
      }
}
\usepackage{pdfpages}

\begin{document}

\title{CYCLIC STOCHASTIC OPTIMIZATION:\\ GENERALIZATIONS, CONVERGENCE,\\ AND APPLICATIONS IN MULTI-AGENT SYSTEMS}
\author{Karla Hern\'andez Cuevas}
\degreemonth{August}
\degreeyear{2017} 
\dissertation
\doctorphilosophy
\copyrightnotice

\input{chapter0}
\input{chapter1}

\input{chapter1p}

\input{chapter1pp}

\input{chapter2}

\input{chapter3}

\input{chapter4p}

\input{chapterNUMERICS}

\input{chapter5}

\input{chapter6}

\begin{appendices}

\input{AppendixFabian}

\input{AppendixSensor}

\end{appendices}

\input{notation}

\cleardoublepage

\phantomsection

\addcontentsline{toc}{chapter}{Bibliography}

\fancyhead[L]{BIBLIOGRAPHY} 

\newcommand{\newblock}{}

\bibliographystyle{apa-good}
\bibliography{thesis} 

\cleardoublepage

\fancyhead[L]{BIBLIOGRAPHY} 



\end{document}

%% file: chapter0.tex

\begin{frontmatter}

\maketitle

\begin{abstract}

Stochastic approximation (SA) is a powerful class of iterative algorithms for nonlinear root-finding that can be used for minimizing a  loss function, $L(\bm\uptheta)$, with respect to a parameter vector $\bm\uptheta$, when only noisy observations of $L(\bm\uptheta)$ or its gradient are available (through the natural connection between root-finding and minimization); SA algorithms can be thought of as stochastic line search methods where the entire parameter vector is updated at each iteration. 
The cyclic approach to SA is a variant of SA procedures where $\bm\uptheta$ is divided into multiple subvectors that are updated one at a time in a cyclic manner. 

This dissertation focuses on studying the asymptotic properties of cyclic SA and of the generalized cyclic SA (GCSA) algorithm, a variant of cyclic SA where the subvector to update may be selected according to a random variable or according to a predetermined pattern, and where the noisy update direction can be based on the updates of any SA algorithm (e.g., stochastic gradient, Kiefer--Wolfowitz, or simultaneous perturbation SA). The convergence of GCSA, asymptotic normality of GCSA (related to rate of convergence), and efficiency of GCSA relative to its non-cyclic counterpart are investigated both analytically and numerically. Specifically, conditions are obtained for the convergence with probability one of the GCSA iterates and for the asymptotic normality of the normalized iterates of a special case of GCSA. Further, an analytic expression is given for the asymptotic relative efficiency (when efficiency is defined in terms of mean squared error) between a special case of GCSA and its non-cyclic counterpart. Finally, an application of the cyclic SA scheme to a multi-agent stochastic optimization problem is investigated.

This dissertation also contains two appendices. The first appendix generalizes Theorem 2.2 in Fabian (1968) (a seminal paper in the SA literature that derives general conditions for the asymptotic normality of SA procedures) to make the result more applicable to some modern applications of SA including (but not limited to) the GCSA algorithm, certain root-finding SA algorithms, and certain second-order SA algorithms. The second appendix considers the problem of determining the presence and location of a static object within an area of interest by combining information from multiple sensors using a maximum-likelihood-based approach.

\vfill
\noindent {\bf{Primary Reader and Advisor:}} James C. Spall\\
{\bf{Second Reader:}} Raman Arora
\end{abstract}

\begin{dedication}
 \begin{center}
{\it{To my parents.}}
\end{center}

\end{dedication}

\begin{acknowledgment}

 I would first like to thank my advisor, James Spall, for his invaluable guidance throughout my graduate studies. His dedication to his students, patience, admirable work ethic, and vast knowledge were instrumental in making my experience as a graduate student fruitful, enjoyable, and enriching.

I would also like to give special thanks to Stephen Lee, my husband and best friend, for his friendship, support, and for the {\it{many}} interesting discussions we've shared (of a mathematical and non-mathematical nature) throughout our time together. I would also like to mention that the proof of Lemma \ref{lem:devil} was the result of a discussion with Stephen. 

I would also like to thank fellow graduate student Jingyi Zhu (currently a Ph.D. student at the department of Applied Mathematics \& Statistics at Johns Hopkins) for reading Chapter \ref{sec:cyclicseesaw} of my dissertation and for her detailed feedback which helped me improve the presentation of the chapter.

I would also like to express my gratitude to Raman Arora, Danial Naiman, Daniel Robinson, Maxim Bichuch, and John Wierman, the members of my dissertation defense committee, for their time with special thanks to Raman Arora, my second reader; reviewing such a lengthy dissertation is no small task and I sincerely appreciate his help.

I also want to thank all the professors whose courses I've had the pleasure of taking. Although I have taken several courses during my studies at Hopkins, the courses taught by professors James Fill, James Spall, and John Wierman stand out as having been instrumental to my research.

Last but not least, I would like to thank my family with special thanks to my parents. Without their support I would not be where I am today.

This work was financially supported by the National Council of Science and Technology (CONACYT) of Mexico, the Office of Naval Research via Navy contract N00024-13-D6400, the Acheson J. Duncan Fund for the Advancement of Research in Statistics, and Dr. James Spall's JHU/APL sabbatical professorship
at Johns Hopkins' Whiting School of Engineering.
\end{acknowledgment}

\tableofcontents

\listoftables

\listoffigures

\end{frontmatter}

%% file: chapter1.tex

\chapter{Introduction}
\label{chap:intro}

\section{Motivation}

The objective of unconstrained optimization problems is to minimize a real-valued loss function $L(\bm\uptheta)$ with respect to a parameter vector $\bm\uptheta\in \mathbb{R}^p$. 
An important limitation in many practical problems is the fact that the loss function itself is unknown in the sense that $L(\bm\uptheta)$ may only be observable in the presence of noise. For example, suppose that the loss function represents the expected output of a complex stochastic system that depends on $\bm\uptheta$. In this case, obtaining a closed form expression for $L(\bm\uptheta)$ (or even evaluating $L(\bm\uptheta)$ at a given $\bm\uptheta$) may be impossible since computing the expected output would require detailed knowledge of the stochastic process governing the output of the system. In this light, this dissertation makes a distinction between deterministic optimization, where the loss function is known and the optimization process is entirely deterministic, and stochastic optimization, where the optimization process involves some type of randomness (e.g., the optimization algorithm uses measurements of $L(\bm\uptheta)$ that are corrupted by random noise or there is a random choice made in the search direction as the algorithm iterates towards a solution). Stochastic approximation (SA) is a powerful class of iterative algorithms for nonlinear root-finding. These algorithms can also be used for stochastic optimization (through the natural connection between root-finding and minimization). The cyclic approach to SA is a particular variant of these iterative procedures in which only a subset of the parameter vector (referred to as a {\it{subvector}}) is updated at any given time. This dissertation focuses on studying the asymptotic properties (e.g., convergence and rate of convergence) of cyclic SA (and generalizations) for stochastic optimization.

  In the cyclic approach, the full parameter vector is divided into two or more subvectors and the process proceeds by sequentially updating each of the subvectors, while holding the remaining parameters at their most recent values. Thus, the cyclic approach is iterative in nature and at each iteration an estimate for the minimizer of $L(\bm\uptheta)$ is obtained. To illustrate the idea behind the cyclic approach, let $\hat{\bm{\uptheta}}_k$ denote the estimate obtained during the $k$th iteration. Furthermore, consider the special case of the cyclic approach where there are only two subvectors, $\hat{\bm{\uptheta}}_k^{[1]}$ and $\hat{\bm{\uptheta}}_k^{[2]}$, so that $\hat{\bm{\uptheta}}_k=[(\hat{\bm{\uptheta}}_k^{[1]})^\top,(\hat{\bm{\uptheta}}_k^{[2]})^\top]^\top$. Additionally, for simplicity assume the two subvectors are updated in a strictly-alternating manner.
Here, the first step of an iteration of the cyclic approach would consist of updating $\hat{\bm{\uptheta}}_k^{[1]}$ while keeping $\hat{\bm{\uptheta}}_k^{[2]}$ fixed. This would give rise to the vector $[(\hat{\bm{\uptheta}}_{k+1}^{[1]})^\top,(\hat{\bm{\uptheta}}_k^{[2]})^\top]^\top$. In the second step $\hat{\bm{\uptheta}}_k^{[2]}$ would be updated while holding $\hat{\bm{\uptheta}}_{k+1}^{[1]}$ fixed, this would give rise to the vector $\hat{\bm{\uptheta}}_{k+1}=[(\hat{\bm{\uptheta}}_{k+1}^{[1]})^\top,(\hat{\bm{\uptheta}}_{k+1}^{[2]})^\top]^\top$.

  Independent of whether an optimization algorithm is stochastic or deterministic, cyclic schemes can help reduce the problem's complexity by focusing only on a subset of the parameter vector at any given time. For example, conditional optimization with respect to subsets of $\bm\uptheta$ may prove more tractable than
simultaneous (unconditional) optimization. Lee and Park (2008)\nocite{leenpark2008}, for example, use cyclic optimization to approach a structure-from-motion problem (presented as a deterministic optimization problem), where minimizing $L(\bm\uptheta)$ is difficult but where it is possible to obtain closed form expressions for the minimizers $L(\bm\uptheta)$ with respect to different subsets of $\bm\uptheta$. Lee and Park (2008)\nocite{leenpark2008} propose an algorithm that consists of sequentially minimizing the loss function with respect to the different subsets of $\bm\uptheta$. Although the authors do not provide any theoretical guarantees, the algorithm appears to exhibit good numerical performance. Another example lies in the area of communication networks. Here, one may be interested in jointly optimizing congestion control and scheduling (see, for example, Andrews 2006)\nocite{andrews2006congestion} and it may be the case that conditional solutions are more readily available.

 In the cyclic approach, the full parameter vector is divided into two or more subvectors and the process proceeds by sequentially optimizing each of the subvectors, while holding the remaining parameters at their most recent values. The cyclic approach is of special interest where there is need to extend a model to include new parameters since it allows for the preservation of resources dedicated to optimizing the original problem (e.g., the preservation of expensive software). In theory, one could alternate between updating the new parameters and the parameters from the original problem (existing resources could be reused with this approach).

A convenient feature of cyclic line search methods like block coordinate descent is that computing partial update directions (e.g., computing only a part of the gradient when the update directions are gradient-based) is often significantly less expensive than computing the full update direction. In this vein, Wright (2015)\nocite{wright2015} discusses how coordinate descent algorithms have grown in popularity because of their usefulness in data analysis, and machine learning. Wright (2015)\nocite{wright2015} gives a useful literature review of coordinate descent methods that are mostly limited to deterministic optimization algorithms, with the exception of randomized coordinate descent methods which can be seen as a special case of the stochastic gradient algorithm (a type of SA algorithm).

In the area of multi-agent optimization, the vector $\bm\uptheta$ can be divided into different subsets each of which is updated by a different agent (see, for example, Peterson et al. 2014\nocite{spie2014} and Botts et al. 2016, where $\bm\uptheta$ is a vector related to the positions of the agents)\nocite{bottsetal2016}. In general, agents may operate in a synchronous manner (i.e., agents synchronize the timing of their respective updates in some way, such as by taking turns performing updates) or in an asynchronous manner (i.e., agents do not synchronize the timing of their respective updates in any way). Cyclic optimization algorithms apply given
 the manner in which agents operate and, under appropriate conditions, allow for rigorous theoretical analysis of multi-agent optimization algorithms. 

Despite the aforementioned desirable properties of cyclic optimization, convergence to a minimizer of $L(\bm\uptheta)$ (either global or local) is not generally guaranteed for cyclic algorithms. To illustrate this fact, consider an example where $L(\bm\uptheta)=-(\uptau_1+\uptau_2)$ and $\bm\uptheta\in \mathcal{S}\equiv \{[\uptau_1,\uptau_2]^\top{\text{such that $\uptau_1\in [0,1]$ and $\uptau_2\in[0,2(1-\uptau_1)]$}}\}$ (this set defines a region bounded by a right triangle with base corresponding to the interval $[0,1]$ on the $\uptau_1$-axis, right angle at $[0,0]^\top$, and height 2). This constrained optimization problem has a unique global minimum at $\bm\uptheta^\ast=[0,2]^\top$. Suppose now that a cyclic algorithm is implemented given that one alternates between minimizing $L(\bm\uptheta)$ with respect to $\uptau_1$ and $\uptau_2$, beginning with $\uptau_1$. Furthermore, suppose the algorithm is initialized at $\hat{\bm{\uptheta}}_0=[0,0]^\top$. In its first step, the algorithm will arrive at the point $\hat{\bm{\uptheta}}_1=[1,0]^\top$. Afterwards, the algorithm becomes ``stuck'' at this point since no update in the value of $\uptau_2$ can further educe the loss function value. Thus, this cyclic algorithm would never reach the global minimum at $\bm\uptheta^\ast$. Moreover, the point $\hat{\bm{\uptheta}}_1=[1,0]^\top$ is not even a {\it{local}} minimum. This can be seen by noting that the point $\bm\uptheta_\upepsilon= [1-\upepsilon, 2\upepsilon]^\top$ with $\upepsilon\in [0,1]$ is in the set $\mathcal{S}$
 and that $L(\bm\uptheta_\upepsilon)=-(1+\upepsilon)< -1=L(\hat{\bm{\uptheta}}_1)$. For both deterministic- and stochastic optimization algorithms, knowing that a non-cyclic algorithm converges to a minimizer (either local or global) of $L(\bm\uptheta)$ does not automatically imply that their cyclic counterparts also converge to a minimizer or $L(\bm\uptheta)$.

This dissertation introduces the generalized cyclic SA (GCSA) algorithm, a cyclic algorithm where the subvectors are updated using SA. Loosely speaking, at each time increment a subvector of the parameter vector is selected and updated 
according to a direction that
is obtained using SA (this includes many popular stochastic update directions including the those used in the Robbins--Monro, Kiefer--Wolfowitz, and simultaneous perturbation SA algorithms). In the GCSA algorithm the subvector to update is not necessarily selected following a deterministic pattern (the term ``generalized'' in ``generalized cyclic SA'' is used precisely to emphasize this fact). A special case of the GCSA algorithm, for example, allows the subvector that is to be updated
to be selected according to a random variable that may depend on the iteration number. The following section reviews the literature for existing results on cyclic optimization. We cover both deterministic and stochastic implementations.

Before surveying the literature we first review the basic form of SA algorithms for nonlinear root-finding. This allows us to be more specific when comparing existing results to the results developed in this dissertation.
Given a vector $\bm\uptheta$ and a vector-valued function $\bm{f}(\bm\uptheta)$, the  basic SA algorithm for nonlinear root-finding attempts to find a solution to $\bm{f}(\bm\uptheta)=\bm{0}$ in an iterative manner using the recursion $\hat{\bm{\uptheta}}_{k+1}=\hat{\bm{\uptheta}}_k-a_k{\bm{Y}}_k(\hat{\bm{\uptheta}}_k)$, where $a_k>0$ is the gain sequence and $-{\bm{Y}}_k(\hat{\bm{\uptheta}}_k)$ is a vector-valued random variable representing a noisy observation of $-\bm{f}(\hat{\bm{\uptheta}}_k)$, the desired update direction. In the analysis of SA algorithms it is often useful to express the vector ${\bm{Y}}_k(\hat{\bm{\uptheta}}_k)$ as follows:
\begin{align}
\label{eq:imastandardupdaedirection}
{\bm{Y}}_k(\hat{\bm{\uptheta}}_k)= {\bm{f}}(\hat{\bm{\uptheta}}_k)+{\bm{\upbeta}}_k(\hat{\bm{\uptheta}}_k)+{\bm{\upxi}}_k(\hat{\bm{\uptheta}}_k).
\end{align}
Typically, $\bm\upbeta_k(\hat{\bm{\uptheta}}_k)\equiv E\big[{\bm{Y}}_k(\hat{\bm{\uptheta}}_k)-\bm{f}(\hat{\bm{\uptheta}}_k)\big|\mathcal{F}_k\big]$, ${\bm{\upxi}}_k(\hat{\bm{\uptheta}}_k)\equiv {\bm{Y}}_k(\hat{\bm{\uptheta}}_k)-E\big[{\bm{Y}}_k(\hat{\bm{\uptheta}}_k)\big|\mathcal{F}_k\big]$, $\mathcal{F}_k$ is some representation of the history of the process, and $E[\mathcal{X}]$ represents the expected value of the random variable $\mathcal{X}$.
In the special case where $\bm{f}(\bm\uptheta)=\bm{g}(\bm\uptheta)$ is the gradient of $L(\bm\uptheta)$, that is when SA is used for stochastic optimization via nonlinear root-finding, $\bm{Y}_k(\hat{\bm{\uptheta}}_k)$ denotes a noisy estimate of the gradient. Here, it is common to replace the notation $\bm{Y}_k(\hat{\bm{\uptheta}}_k)$ with $\hat{\bm{g}}_k(\hat{\bm{\uptheta}}_k)$. Therefore, we write $\hat{\bm{\uptheta}}_{k+1}=\hat{\bm{\uptheta}}_k-a_k\hat{\bm{g}}_k(\hat{\bm{\uptheta}}_k)$,
(Section \ref{sec:dummieyouask} discusses SA in greater detail).

\section{Literature Review}
\label{sec:literaturesurveyrev}

In the 1980s, 1990s, and early 2000s,  Bertsekas and Tsitsiklis (1989)\nocite{bertsekasandtsitsiklis1989}, Luo and Tseng (1992)\nocite{luptseng1992}, Luo and Tseng (1993)\nocite{luptseng1993}, and Tseng (2001)\nocite{tseng2001} made important contributions to understanding the convergence properties of cyclic optimization procedures. However, their analysis focused on deterministic optimization problems, not on the stochastic optimization setting considered in this dissertation. A few more recent references investigating cyclic implementations in the area of deterministic optimization are Tseng and Yung (2009)\nocite{tsengandyun2009}, who investigate the convergence of block-coordinate gradient descent methods for linearly constrained non-smooth separable loss functions, and Spall (2012)\nocite{spallcyclicseesaw2012}, who investigates the convergence properties of general cyclic algorithms. The interested reader can also find a useful literature review of coordinate descent methods (largely for deterministic optimization) in Wright (2015)\nocite{wright2015}. 
 A few applications of cyclic procedures for deterministic optimization can be found in the papers by Canutescu and Dunbrack (2003)\nocite{canutescu2003}, who use cyclic procedures for a control problem in robotics, Lee and Park (2008)\nocite{leenpark2008}, who use cyclic optimization for a problem in computer vision, Li and Osher (2009)\nocite{liandosher2009}, who consider a compressed sensing problem, and Li and Petropulu (2014)\nocite{liandpetriouli2014}, who use the alternating directions method of multipliers (a cyclic procedure) for target location estimation. 

The previous references are concerned with cyclic schemes for solving deterministic optimization problems. In the area of stochastic optimization, one class of SA algorithms which at first appears to be related to GCSA is the class of two time-scales SA algorithm studied in Borkar (1997)\nocite{borkar1997} (see also Konda and Tsitsiklis 2004\nocite{kondaandtsitsiklis2004} for a similar setting). Here, a parameter vector, $\bm{X}(k)$, is updated using SA. In Borkar's (1997)\nocite{borkar1997} formulation, obtaining the SA update direction requires a second SA step. Specifically, using Borkar's (1997)\nocite{borkar1997} notation, there exists a vector $\bm{U}(k)$ such that:
\begin{subequations}
\begin{align}
\bm{X}(k+1)&=\bm{X}(k)+ a(k)\left[\bm{R}(\bm{X}(k),\bm{U}(k))+\bm{M}(k+1)\right],\label{eq:sandhoneygrow0}\\
\bm{U}(k+1)&=\bm{U}(k)+b(k)\left[\bm{G}(\bm{X}(k),\bm{U}(k))+\bm{N}(k+1)\right],
\label{eq:sandhoneygrow}
\end{align}
\end{subequations}
where $a(k)$ and $b(k)$ are real sequences, $\bm{R}(\cdot,\cdot)$ and $\bm{G}(\cdot,\cdot)$ are deterministic vector valued functions, and $\bm{M}(k+1)$ and $\bm{N}(k+1)$ are random noise terms. Thus, the terms in square brackets in (\ref{eq:sandhoneygrow0},b) are SA update directions for $\bm{X}(k)$ and $\bm{U}(k)$. 
By letting $\bm\uptheta=[\bm{X}^\top,\bm{U}^\top]^\top$ and setting $\bm{R}(\cdot,\cdot)=\bm{G}(\cdot,\cdot)$ it may appear that any strictly alternating cyclic SA algorithm with disjoint subvectors is a special case of (\ref{eq:sandhoneygrow0},b), but this is not true.
To see that (\ref{eq:sandhoneygrow0},b) is not a cyclic recursion, note that the recursion in (\ref{eq:sandhoneygrow0},b) implies that the values of $\bm{X}(k+1)$ and $\bm{U}(k+1)$ are both obtained based on the vector $[\bm{X}(k)^\top,\bm{U}(k)^\top]^\top$. In a cyclic algorithm, however, the values of $\bm{X}(k+1)$ and $\bm{U}(k+1)$ would be obtained based on the vectors $[\bm{X}(k)^\top,\bm{U}(k)^\top]^\top$ and $[\bm{X}(k)^\top,\bm{U}(k+1)^\top]^\top$, respectively (assuming $\bm{X}(k)$ is updated first). Therefore, the pair in (\ref{eq:sandhoneygrow0},b) is not cyclic.  The remainder of this section discusses the results that most closely resemble GCSA.

Tsitsiklis (1994)\nocite{tsitsiklis1994} considers an asynchronous SA algorithm for finding a solution to $\bm{D}(\bm{x})=\bm{x}$, where $\bm{x}\in \mathbb{R}^p$ is a vector and $\bm{D}(\cdot)$ is a vector valued function, using possibly outdated information. Specifically, using Tsitsiklis' notation, the $i$th entry of $\bm{x}$, denoted by $x_i$, is updated according to the following recursion:
\begin{align}
x_i(t+1)&=x_i(t)-\upalpha_i(t)[D_i(\bm{x}^{i}(t))-x_i(t)+w_i(t)] \text{\ \ \ for $t\in T^i$},\notag\\
x_i(t+1)&=x_i(t)\text{\ \ \ for $t\notin T^i$},
\label{eq:singoayatailinchin}
\end{align}
where $T^{i}$ is a random set of (nonnegative integer) times at which $x_i$ is updated, $\upalpha_i(t)\in [0,1]$, $w_i(t)$ is a mean-zero noise term, $D_i$ denotes the $i$th entry of $\bm{D}$, and $\bm{x}^{i}(t)$ is a vector of (possibly) outdated components of $\bm{x}$, that is $\bm{x}^i=[x_1(\uptau_1^i(t)),\dots,x_p(\uptau_p^i(t))]^\top$, where $\uptau_j^i(t)\leq t$ is a nonnegative integer. If no information is outdated then $\bm{x}^i(t)=\bm{x}(t)$. It is in the case of delayed information that the resemblance to cyclic optimization becomes apparent. To see this, consider the special case of (\ref{eq:singoayatailinchin}) where $p$ is even, 
$T^{i}$ is the set of nonnegative odd numbers for $i\leq p/2$, $T^{i}$ is the set of even numbers for $i>p/2$, $\bm{x}^i(t)=\bm{x}(t)$ for $i\leq p/2$, and $\bm{x}^i(t)=[x_1(t+1), \dots, x_{p/2}(t+1), x_{p/2+1}(t),\dots, x_p(t)]^\top$ for $i>p/2$. The resulting algorithm is a special case of cyclic SA in which two subvectors of $\bm{x}(t)$ are updated in a strictly alternating manner.

One convenient property of (\ref{eq:singoayatailinchin}) is that it may be implemented in an asynchronous manner, that is the entries of $\bm{x}(t)$ that are updated at time $t$ can each be updating using outdated parameter vectors collected at possibly different times. The GCSA algorithm, in contrast, is synchronous in nature since the entries of the parameter vector that are updated at any given time must all be updated using the latest value of the parameter vector. Despite the advantage that an asynchronous algorithm presents over a synchronous algorithm, there are a few reasons why the theory regarding the convergence of (\ref{eq:singoayatailinchin}) does not apply to GCSA as is discussed next.

One of the conditions that Tsitsiklis (1994)\nocite{tsitsiklis1994} imposes for the convergence of (\ref{eq:singoayatailinchin}) requires the second moment of $w_i(t)$ to be bounded above by a function that depends on the magnitude of $\bm{x}(t)$ (see Assumption 2.e in Tsitsiklis 1992)\nocite{tsitsiklis1994}. When $\bm{x}(t)$ converges to some finite vector w.p.1 (e.g., to the solution of $\bm{F}(\bm{x})=\bm{x}$), Tsitsiklis' (1994)\nocite{tsitsiklis1994} assumption requires $w_i(t)$ to have bounded variance w.p.1. However, because 
(\ref{eq:singoayatailinchin}) can be thought of as a special case of (\ref{eq:imastandardupdaedirection}):
 \begin{align}
 \label{eq:latesincity}
&{\text{Update direction in first line of (\ref{eq:singoayatailinchin})}}= D_i(\bm{x}^{i}(t))-x_i(t)+w_i(t)\notag\\
&= {\underbrace{D_i(\bm{x}(t))-x_i(t)}_{\text{$i$th entry of $\bm{f}$ term in (\ref{eq:imastandardupdaedirection})}}}+{\underbrace{D_i(\bm{x}^{i}(t))-D_i(\bm{x}(t))+w_i(t)}_{{\text{$i$th entry of $\bm\upbeta+\bm\upxi$ term in (\ref{eq:imastandardupdaedirection})}}}},
 \end{align}
requiring $w_i(t)$ to have bounded second moment is equivalent to requiring the diagonal entries of the second moment matrix of ${\bm{\upbeta}}_k(\hat{\bm{\uptheta}}_k)+{\bm{\upxi}}_k(\hat{\bm{\uptheta}}_k)$ in (\ref{eq:imastandardupdaedirection}) to have bounded magnitude, an  assumption that is incompatible with many SA algorithms (such as the simultaneous perturbation SA algorithm) for which the variance of the noise can increase with time.
 In this dissertation we allow the variance of the noise to increase as the iteration number increases in order to accommodate a more general class of SA update directions. 
 Another observation we make is that the vector ${\bm{\upbeta}}_k(\hat{\bm{\uptheta}}_k)+{\bm{\upxi}}_k(\hat{\bm{\uptheta}}_k)$ from (\ref{eq:imastandardupdaedirection}) has a very specific form in (\ref{eq:latesincity}).  In the GCSA algorithm, ${\bm{\upbeta}}_k(\hat{\bm{\uptheta}}_k)+{\bm{\upxi}}_k(\hat{\bm{\uptheta}}_k)$ term is allowed to have a form that is more appropriate for general SA procedures. 
 In conclusion, the assumptions imposed by Tsitsiklis (1994)\nocite{tsitsiklis1994} imply that the class of algorithms studied in Tsitsiklis (1994)\nocite{tsitsiklis1994} intersects the class of algorithms that fit into the GCSA framework although neither class of algorithms contains the other. Moreover, Tsitsiklis (1994)\nocite{tsitsiklis1994} does not discuss rate of convergence or provide results on asymptotic normality, both contributions of this dissertation.

Another algorithm closely related to GCSA appears in Borkar (1998)\nocite{borkar1998}. Here, the author investigates the asymptotic behavior of a distributed, asynchronous SA scheme in terms of a limiting
 differential equation. 
One important assumption made by Borkar (1998)\nocite{borkar1998} is that the update direction for the $i$th entry of the parameter vector is an unbiased estimate of the $i$th entry of $\bm{f}(\bm\uptheta)$, the function whose root is to be found (see equation 2.9 in Borkar 1998)\nocite{borkar1998}. This assumption is usually only valid for certain SA algorithms (e.g., the stochastic gradient algorithm) or deterministic optimization problems. The theory in this dissertation does not require the availability of unbiased estimates of the entries of $\bm{f}(\bm\uptheta)$. Another assumption made by Borkar (1998)\nocite{borkar1998} requires $\bm{f}(\bm\uptheta)$ to be Lipschitz continuous. In our setting (which focuses on stochastic optimization via root-finding so that $\bm{f}(\bm\uptheta)=\bm{g}(\bm\uptheta)$) this would be equivalent to requiring the gradient of $L(\bm\uptheta)$ to be Lipschitz continuous. We do not make such an assumption in the theory for convergence w.p.1 of the GCSA iterates. Additionally, Borkar (1998)\nocite{borkar1998} does not discuss rate of convergence or provide results on asymptotic normality, both contributions of this dissertation.

In this dissertation, the theory behind the GCSA algorithms allows the noisy update directions to be biased estimates of $\bm{f}(\bm\uptheta)$ and allows the noise term in (\ref{eq:imastandardupdaedirection}) (i.e., ${\bm{\upxi}}_k(\hat{\bm{\uptheta}}_k)$) to have a second moment matrix whose entries can increase in magnitude as a function of the iteration number. Algorithms that are related to GCSA but require the availability of unbiased estimates of $\bm{f}(\bm\uptheta)$ can be found in the papers by Borkar and Meyn (2000)\nocite{borkarmeyn2000}, Aboundai et al. (2002)\nocite{aboundaietal2002}, Bhatnagar (2011)\nocite{bhatnagar2011}, Bianchi and Jakubowicz (2013)\nocite{bianchiandjiakubowicz2013}, and Singh et al. (2014)\nocite{singhetal2014} (technically, the estimate for $\bm{f}(\bm\uptheta)$ in this last reference is and unbiased estimate of a vector that has the same roots as $\bm{f}(\bm\uptheta)$). Algorithms that are related to GCSA but require the term ${\bm{\upxi}}_k(\hat{\bm{\uptheta}}_k)$ in (\ref{eq:imastandardupdaedirection}) to have bounded variance can be found in the papers by Tsitsiklis (1984)\nocite{tsitsiklis1984}, Tsitsiklis et al. (1986)\nocite{tsitsiklisetal1986}, Solodov and Zavriev (1998)\nocite{solodovzavriev1998}, Ram et al. (2009b)\nocite{rametal2009b}, Nedi\'c and Bertsekas (2010)\nocite{nedicandbertsekas2010}, Ram et al. (2010)\nocite{rametal2010}, Xu and Yin (2015)\nocite{xyandyin2015}, and in Exercise 1.7.1 on p. 34 in Benveniste et al. (1990)\nocite{benveniste1990}. The distributed algorithms in Chapter 12 of Kushner and Yin (1997)\nocite{kushnyin1997} are also closely related to GCSA. However, the results in said chapter require the sequence of noisy update directions (i.e., the sequence ${\bm{Y}}_k(\hat{\bm{\uptheta}}_k)$) to be uniformly integrable (see Kushner and Yin 1997, Chapter 12, Assumptions A3.1 and A3.1$'$), an assumption that is too strong for some SA algorithms such as finite difference SA where the noise in the update direction generally makes the ${\bm{Y}}_k(\hat{\bm{\uptheta}}_k)$ vectors not uniformly integrable. This dissertation does not make the assumption of uniform integrability of the noisy update directions. 
 Lastly, a few results related to GCSA concerned with optimizing convex functions can be found in the papers by Ram et al. (2009a)\nocite{rametal2009a}, Necoara (2013)\nocite{necoara2013}, and Necoara and Petrascu (2014)\nocite{recoaraandpetrascu2014} (the theory in this dissertation does not assume convexity).

\section{This Work's Contribution}
This work's main contribution can be summarized by the following points:
\begin{enumerate}
\item	We derive conditions for the convergence with probability one of the GCSA algorithm to a root of the gradient of $L(\bm\uptheta)$.  The main convergence result is stated in Theorem \ref{thm:hoeshoo} of Section \ref{sec:convo}. A few corollaries based on special cases of GCSA are derived in the same section. Numerical results supporting the theory on convergence are provided in Section \ref{sec:convergencenumericos}.
\item We provide a generalization to Theorem 2.2 in Fabian (1968) (see Theorem \ref{thm:generalizefabian}) that allows us to show asymptotic normality for a special case of GCSA. Appendix \ref{sec:fabiansecgeneralize} explains how the generalization to Fabian's theorem also extends the theorem's result to include a broader range of SA algorithms of practical interest including certain second order SA and root-finding SA algorithms. 
\item We show the asymptotic normality of the normalized iterates of a special case of GCSA in which the subvector to update is selected according to a deterministic pattern (see Theorem \ref{thm:fnogg}). The result on asymptotic normality helps us define the asymptotic rate of convergence for the iterates of this special case of GCSA. Numerical results supporting the theory on asymptotic normality are provided in Section \ref{sec:normalitycenumericos}.
\item We discuss the importance of defining the cost of implementation when comparing the performances of two optimization algorithms. When cost is a measure of the number of basic arithmetic computations required, we discuss the type of arithmetic operations that can result in a significant difference between the per-iteration cost of implementing an algorithm
 and the per-iteration cost of implementing its cyclic counterpart. 
\item We provide an analytical estimate for the asymptotic efficiency of a special case of GCSA relative to the efficiency of its non-cyclic counterpart after taking into consideration the per-iteration costs of implementation. Here, efficiency is defined in terms of the mean-squared estimation errors. We show how the expression for asymptotic relative efficiency implies that either algorithm (cyclic or non-cyclic) can be more efficient. Numerical experiments computing the relative efficiency between cyclic and non-cyclic algorithms under different definitions of cost are provided in Section \ref{sec:Numericssimples}. 
\item In Chapter \ref{chap:multiagent} we apply the cyclic SA approach to a multi-agent optimization problem for tracking and surveillance. 
\item This dissertation also contains two appendices. Appendix \ref{sec:fabiansecgeneralize} contains a more detailed proof of Theorem \ref{thm:generalizefabian}, the theorem in Chapter \ref{sec:ROC} that generalizes Theorem 2.2 in Fabian (1968), and discusses a few applications of the generalized theorem. Appendix \ref{sec:ROC} considers the problem of determining the presence and location of a static object within an area of interest by combining information from multiple sensors using a maximum-likelihood-based approach.

\end{enumerate}
Two noteworthy assumptions made throughout this dissertation are that the loss function is differentiable and that the optimization problem is unconstrained. Future work could focus on investigating the constrained- and non-differentiable settings.

\section{Overview of Contents}
The main part of this work (Chapters \ref{chap:prems8}--\ref{chap:multiagent}) is organized as follows. First, Chapter \ref{chap:prems8} reviews the general stochastic optimization setting as well as the idea behind SA algorithms for stochastic optimization. 
Chapter \ref{chap:descriptionofGCSA} introduces the cyclic seesaw SA algorithm and the generalized cyclic SA algorithm (GCSA) that is the focus of this work.
 Chapter \ref{sec:cyclicseesaw} focuses on deriving conditions for the convergence with probability one of the GCSA algorithm's iterates to a root of the gradient of the function to minimize. 
  Chapter \ref{sec:ROC} generalizes a well-known result in the SA literature (Fabian 1968, Theorem 2.2) and uses the resulting generalization to prove asymptotic normality of the scaled iterates for a special case of GCSA. 
  Chapter \ref{chap:imefficient} is concerned with computing the asymptotic efficiency (defined as the asymptotic mean-squared-error) of a special case of GCSA relative to that of its non-cyclic counterpart. The chapter begins by discussing the importance of defining the cost of implementation before attempting to compare any two algorithms. The chapter goes on to provide a comparison of the cost of implementing an SA algorithm in a cyclic manner versus the cost if implemented in a non-cyclic manner; a few definitions of cost are considered. Using the asymptotic normality result from Chapter \ref{sec:ROC}, Chapter \ref{chap:imefficient} computes the asymptotic relative efficiency between a special case of GCSA and its non-cyclic counterpart. Chapter \ref{chap:numericschap} contains numerical results that illustrate the theory of Chapters \ref{sec:cyclicseesaw}--\ref{chap:imefficient}. Chapter \ref{chap:multiagent} applies the cyclic approach to a multi-agent optimization problem where the loss function is time-varying and corrupted by noise. While the theory from Chapters \ref{sec:cyclicseesaw} and \ref{sec:ROC} (which is concerned with the optimization of a loss function that is not time-varying) is not fully applicable to this multi-agent problem (the main condition in Chapters \ref{sec:cyclicseesaw}--\ref{sec:ROC} not satisfied is the decaying gain sequence of SA; as a tracking problem the gains are not allowed to decay to zero), the purpose of the numerical example in this chapter is threefold. First, it addresses an important area surveillance and tracking problem. Second, it studies the performance of a cyclic implementation when the conditions from Chapters \ref{sec:cyclicseesaw} and \ref{sec:ROC} do not fully hold. Third, it serves to motivate directions for future work.

  Following Chapter \ref{chap:multiagent} this work is organized as follows. Appendix \ref{sec:fabiansecgeneralize} contains a more detailed proof of Theorem \ref{thm:generalizefabian}, a theorem in Chapter \ref{sec:ROC} that generalizes Theorem 2.2 in Fabian (1968). Appendix \ref{sec:fabiansecgeneralize} also discusses how the generalization makes Theorem 2.2 in Fabian (1968) applicable to a broader range of SA algorithms that extend beyond cyclic SA. In contrast to the content of Chapters \ref{chap:intro}--\ref{chap:multiagent} and Appendix \ref{sec:fabiansecgeneralize}, Appendix \ref{appen:systemid} is concerned with a topic that is unrelated to SA: it considers the problem of determining the presence and location of a static object within an area of interest by combining information from multiple sensors. A simple maximum-likelihood-based approach is investigated. Lastly, a list of frequently used notation is included at the end of this dissertation.

%% file: chapter1p.tex

\chapter{Preliminaries}
\label{chap:prems8}

This chapter lays the groundwork for our study of stochastic optimization based on nonlinear root-finding stochastic approximation (SA) algorithms. Section \ref{sec:stafff} formally introduces the general stochastic optimization setting and briefly reviews the relationship between root-finding and optimization. Section \ref{sec:dummieyouask} describes the basic form of SA algorithms for nonlinear root-finding. Lastly, Sections \ref{sec:sgfosa} and \ref{sec:coffee100} give three examples of well-known SA algorithms: stochastic gradient, finite difference SA, and simultaneous perturbation SA.
\section{Stochastic Optimization}
\label{sec:stafff}

The idea behind {\it{unconstrained optimization}} problems is the minimization of a real-valued loss function $L(\bm\uptheta)$ with respect to a parameter vector $\bm\uptheta$. When $L(\bm\uptheta)$ is a smooth function, its gradient is denoted by $\bm{g}(\bm\uptheta)\equiv \partial L(\bm\uptheta)/\partial \bm\uptheta$ and, in this case, the optimization problem can be reformulated as a root-finding problem where one attempts to solve $\bm{g}(\bm\uptheta)=\bm{0}$ (this equation is referred to as the {\it{gradient equation}}  and any solution is referred to as a {\it{root of the gradient}}). Here, the usual caveat applies: the set of all roots of the gradient may contain vectors other than global minimizers of $L(\bm\uptheta)$. Still, regardless of this caveat, the close relationship between minimization and root-finding implies many optimization algorithms rely on being able to evaluate either $L(\bm\uptheta)$ or $\bm{g}(\bm\uptheta)$. Take the steepest descent algorithm (e.g., Nocedal and Wright 2006, Chapter 2)\nocite{nocedalnwright2006}, for example. This algorithm is an iterative line-search algorithm in which an estimate, $\hat{\bm{\uptheta}}_k$, for a solution to the gradient equation is obtained during the $k$th iteration
according to the recursion:
\begin{align}
\label{eq:exescriss}
\hat{\bm{\uptheta}}_{k+1}=\hat{\bm{\uptheta}}_k-a_k{\bm{g}}(\hat{\bm{\uptheta}}_k),
\end{align}
where $a_k$ is a strictly-positive scalar often referred to as the {\it{gain sequence}} (one valid choice for the gain sequence in the steepest descent algorithm is setting $a_k$ equal to a constant that does not depend on $k$). Under certain assumptions, $\hat{\bm{\uptheta}}_k$ converges to $\bm\uptheta^\ast$, a root of the gradient. It is apparent from (\ref{eq:exescriss}) that being able to evaluate $\bm{g}(\bm\uptheta)$ at $\{\hat{\bm{\uptheta}}_k\}_{k\geq0}$ is essential to the implementation of the steepest-descent algorithm. Similarly, many optimization algorithms rely on being able to evaluate $L(\bm\uptheta)$. Simple examples of this type of algorithm are the random search algorithms in Section 2.2 of Spall (2003). 

In practice, it is often the case that evaluating the loss function or its gradient is difficult. As a simple example, consider a setting where there exists a complex stochastic model whose output depends on a set of parameters denoted by $\bm\uptheta$. Furthermore, suppose one wishes to find the value of $\bm\uptheta$ that minimizes the expected output of the model. Here, finding a closed form expression for $L(\bm\uptheta)$ would require computing the expected output for each $\bm\uptheta$, which might be {{infeasible}} due to the complexity of the model. When dealing with physical processes in the optimization process, rather than mathematical models, computing the expected value of the measurement/output for any given $\bm\uptheta$ could easily be {{impossible}} (rather than simply infeasible) since physical processes are often governed by rules unknown to the observer. While traditional optimization techniques cannot be implemented when neither $L(\bm\uptheta)$ nor $\bm{g}(\bm\uptheta)$ are known, stochastic optimization algorithms can use noisy measurements of either the loss function or its gradient in the minimization process. Hereafter, the term ``{\it{stochastic optimization}}'' will be used to refer to optimization problems where there is random noise in the measurements of $L(\bm\uptheta)$ or $\bm{g}(\bm\uptheta)$, or where there is a random choice made in the search direction as the algorithm iterates towards a solution. The term ``{\it{deterministic optimization}}'' will be used to refer to classical optimization problems, like steepest descent or Newton's method, where the minimization process is entirely deterministic.

%
  
  As mentioned in the previous paragraph, one type of stochastic optimization problem pertains to the case where one wishes to minimize a function $L(\bm\uptheta)$ when only noisy measurements of this function are available. In other words, it is assumed that the loss function is unknown but that it is possible to obtain measurements of a random variable $Q(\bm\uptheta,\bm{V})$ such that:
\begin{align}
\label{eq:chanchan}
Q(\bm\uptheta,{\bm{V}})=L(\bm\uptheta)+\upvarepsilon(\bm\uptheta,{\bm{V}}),
\end{align}
where $\bm{V}$ denotes a multivariate random variable and $\upvarepsilon(\bm\uptheta,{\bm{V}})$ is a noise term. The term $\upvarepsilon(\bm\uptheta,{\bm{V}})$ can then be interpreted as the error in measuring the function to minimize. In the special case where the expected value of $\upvarepsilon(\bm\uptheta,\bm{V})$ at $\bm\uptheta$ is equal to zero, a consequence of the law of large numbers is that it is possible to approximate $L(\bm\uptheta)$ at any given $\bm\uptheta$ by averaging several independent and identically distributed (i.i.d.) measurements of $Q(\bm\uptheta,\bm{V)}$. In theory, one could average several i.i.d. noisy loss function measurements to obtain $\bar{L}(\bm\uptheta_1),\dots, \bar{L}(\bm\uptheta_N)$, a set of estimates for $L(\bm\uptheta)$ at the points $\bm\uptheta=\bm\uptheta_1,\dots,\bm\uptheta_N$ (selected from the domain of $L(\bm\uptheta)$ via some deterministic or random scheme).
Then, an approximation to $L(\bm\uptheta)$ for all $\bm\uptheta$, denoted by $\bar{L}(\bm\uptheta)$, could be obtained by interpolating the values of $\bar{L}(\bm\uptheta_i)$ (obtaining this interpolation may be a nontrivial task). Afterwards, one could attempt to minimize $\bar{L}(\bm\uptheta)$ using deterministic optimization algorithms with the hope that the minimizer of $\bar{L}(\bm\uptheta)$ is close to the minimizer or $L(\bm\uptheta)$ (the response surface methodology strategy, introduced by Box and Wilson in 1951\nocite{boxwilson1951}, is a more sophisticated variant of this approach). Such an approach, however, is not always practical since obtaining a good approximation to the loss function (an approximation to the loss function would be considered ``good'' if its minimizer is close to the minimizer of $L(\bm\uptheta)$) could require a prohibitive number of noisy function measurements. The following section describes the stochastic approximation setting, a general framework for stochastic nonlinear root-finding that is more appropriate for minimizing (\ref{eq:chanchan}) than the deterministic optimization approach discussed above. 

\section[Stochastic Approximation (SA)]{Stochastic Approximation}
\label{sec:dummieyouask}

There exist many stochastic optimization algorithms. Random search, genetic algorithms, simulated annealing, stochastic gradient, and simultaneous perturbation SA, are a few examples.
This work will focus on {\it{stochastic approximation}} (SA) algorithms for stochastic optimization via nonlinear root-finding. SA algorithms are closely related to line-search methods and  can be used for stochastic optimization. This section formally defines SA and presents three important examples of SA algorithms: the {\it{stochastic gradient}} (SG) form of SA, {\it{finite difference}} SA (FDSA), and {\it{simultaneous perturbation}} SA (SPSA). Throughout the remainder of this work it is assumed that $\bm\uptheta=[\uptau_1,\dots,\uptau_p]^\top\in \mathbb{R}^p$ and $\bm\uptheta^\ast$ will denote a solution to $\bm{g}(\bm\uptheta)=\bm{0}$.

The basic SA algorithm for nonlinear root-finding is known as the Robbins--Monro algorithm (Robbins and Monro 1951)\nocite{robbinsmonro1951}.  Given a vector-valued function $\bm{f}(\bm\uptheta)$, the Robbins--Monro algorithm attempts to find a solution to $\bm{f}(\bm\uptheta)=\bm{0}$ in an iterative manner using the following recursion:
\begin{align}
\label{eq:formof}
\hat{\bm{\uptheta}}_{k+1}=\hat{\bm{\uptheta}}_k-a_k{\bm{Y}}_k(\hat{\bm{\uptheta}}_k),
\end{align}
where ${\bm{Y}}_k(\hat{\bm{\uptheta}}_k)$ is a vector-valued random variable representing a noisy observation of $\bm{f}(\hat{\bm{\uptheta}}_k)$, and $a_k>0$ is the gain sequence (step size). Unlike most deterministic optimization algorithms, the gain sequence in (\ref{eq:formof}) typically satisfies $a_k\rightarrow 0$. This dissertation is concerned with the special case where $\bm{f}(\bm\uptheta)=\bm{g}(\bm\uptheta)$. In this context, $\bm{Y}_k(\hat{\bm{\uptheta}}_k)$ denotes a noisy estimate of the gradient and it is common to replace the notation $\bm{Y}_k(\hat{\bm{\uptheta}}_k)$ in (\ref{eq:formof}) with $\hat{\bm{g}}_k(\hat{\bm{\uptheta}}_k)$. Therefore, we write:
\begin{align}
\label{eq:youknownothing}
\hat{\bm{\uptheta}}_{k+1}=\hat{\bm{\uptheta}}_k-a_k\hat{\bm{g}}_k(\hat{\bm{\uptheta}}_k),
\end{align}
where $\hat{\bm{g}}(\hat{\bm{\uptheta}}_k)$ is a noisy gradient measurement. 

Because $\hat{\bm{g}}(\hat{\bm{\uptheta}}_k)$ can be thought of as an estimate of $\bm{g}(\hat{\bm{\uptheta}}_k)$, the theory of SA for stochastic optimization often relies on rewriting the vector $\hat{\bm{g}}(\hat{\bm{\uptheta}}_k)$ as:
\begin{align}
\hat{\bm{g}}_k(\hat{\bm{\uptheta}}_k)&= {\bm{g}}(\hat{\bm{\uptheta}}_k)+{\bm{\upbeta}}_k(\hat{\bm{\uptheta}}_k)+{\bm{\upxi}}_k(\hat{\bm{\uptheta}}_k),
\label{eq:tennisey}
\end{align}
a special case of (\ref{eq:imastandardupdaedirection}), where $\bm\upbeta_k(\hat{\bm{\uptheta}}_k)= E\big[\hat{\bm{g}}_k(\hat{\bm{\uptheta}}_k)-\bm{g}(\hat{\bm{\uptheta}}_k)\big|\mathcal{F}_k\big]$, ${\bm{\upxi}}_k(\hat{\bm{\uptheta}}_k)= \hat{\bm{g}}_k(\hat{\bm{\uptheta}}_k)-E\big[\hat{\bm{g}}_k(\hat{\bm{\uptheta}}_k)\big|\mathcal{F}_k\big]$, $\mathcal{F}_k$ is some representation of the history of the process (the precise definition of $\mathcal{F}_k$ may vary from algorithm to algorithm), and $E[\mathcal{X}]$ represents the expected value of the random variable $\mathcal{X}$. One common choice of $\mathcal{F}_k$ is $\mathcal{F}_k=\hat{\bm{\uptheta}}_0,\dots,\hat{\bm{\uptheta}}_k$. In this case, $\bm\upbeta_k(\hat{\bm{\uptheta}}_k)$ represents the {\it{bias}} of $\hat{\bm{g}}_k(\hat{\bm{\uptheta}}_k)$ (e.g., Bickel and Doksum 2007)\nocite{bickelndoksum2007} as an estimator of $\bm{g}(\hat{\bm{\uptheta}}_k)$ and the vector ${\bm{\upxi}}_k(\hat{\bm{\uptheta}}_k)$ is often referred to as the {\it{noise}} term. In the special case where $E\big[\hat{\bm{g}}_k(\hat{\bm{\uptheta}}_k)\big|\hat{\bm{\uptheta}}_k\big]= \bm{g}(\hat{\bm{\uptheta}}_k)$, $\hat{\bm{g}}_k(\hat{\bm{\uptheta}}_k)$ is said to be an unbiased estimate of the gradient at $\hat{\bm{\uptheta}}_k$. It is important to mention that the decomposition in (\ref{eq:tennisey}) is used only for theoretical purposes and, in practice, the bias and noise terms are never computed. Sections \ref{sec:sgfosa} and \ref{sec:coffee100} discuss special cases of (\ref{eq:youknownothing}) that differ in the way $\hat{\bm{g}}_k(\hat{\bm{\uptheta}}_k)$ is computed.

\section[Stochastic Gradient (SG) Form of SA]{Stochastic Gradient Form of SA}
\label{sec:sgfosa}

The {\it{stochastic gradient}} (SG) algorithm for stochastic approximation (e.g., Spall 2003, Chapter 5) is a special case of the Robbins-Monro algorithm that requires the availability of a vector $\hat{\bm{g}}_k^{\text{SG}}(\hat{\bm{\uptheta}}_k)$ such that $E\big[\hat{\bm{g}}_k^{\text{SG}}(\hat{\bm{\uptheta}}_k)\big|\hat{\bm{\uptheta}}_k\big]=\bm{g}(\hat{\bm{\uptheta}}_k)$ (i.e., the noisy gradient measurement must be an unbiased measurement of the gradient at $\hat{\bm{\uptheta}}_k$). The recursion defining $\hat{\bm{\uptheta}}_k$ in the SG algorithm is as follows:
\begin{align*}
\hat{\bm{\uptheta}}_{k+1}=\hat{\bm{\uptheta}}_k-a_k\hat{\bm{g}}_k^{\text{SG}}(\hat{\bm{\uptheta}}_k),
\end{align*}
a special case of (\ref{eq:youknownothing}). Because $\hat{\bm{g}}_k^{\text{SG}}(\hat{\bm{\uptheta}}_k)$ is an unbiased estimate of $\bm{g}(\hat{\bm{\uptheta}}_k)$ we write $\hat{\bm{g}}_k^{\text{SG}}(\hat{\bm{\uptheta}}_k)=\bm{g}(\hat{\bm{\uptheta}}_k)+{\bm{\upxi}}_k^{\text{SG}}(\hat{\bm{\uptheta}}_k)$, where $E[{\bm{\upxi}}_k^{\text{SG}}(\hat{\bm{\uptheta}}_k)|\mathcal{F}_k]=\bm{0}$. In the area of simulation-based optimization, two popular algorithms that are special cases of the SG algorithm are the pure infinitesimal perturbation analysis (IPA) algorithm and the pure likelihood ratio function algorithm (Spall 2003, p. 418)\nocite{ISSO}. For both these algorithms, the random vector $\hat{\bm{g}}_k^{\text{SG}}(\hat{\bm{\uptheta}}_k)$ can sometimes be obtained by assuming that is is possible to
differentiate $Q(\bm\uptheta,\bm{V})$ with respect to $\bm\uptheta$ (e.g., Spall 2003, p. 134\nocite{ISSO}).
 This approach is formally discussed next.

First note that (\ref{eq:chanchan}) implies $E[Q(\bm\uptheta,\bm{V})|\bm\uptheta]=L(\bm\uptheta)+E[\upvarepsilon(\bm\uptheta,{\bm{V}})|\bm\uptheta]$ so that:
\begin{align}
\label{eq:goodworkteak}
\frac{\partial}{\partial \bm\uptheta}E[Q(\bm\uptheta,\bm{V})|\bm\uptheta]={\bm{g}}(\bm\uptheta)+\frac{\partial E[\upvarepsilon(\bm\uptheta,{\bm{V}})|\bm\uptheta]}{\partial \bm\uptheta}.
\end{align}
 If $E[\bm\upvarepsilon(\bm\uptheta,\bm{V})|\bm\uptheta]$ does not depend on $\bm\uptheta$ (e.g., if $\bm\upvarepsilon$ is i.i.d. noise) then (\ref{eq:goodworkteak}) implies:
\begin{align}
\label{eq:chungie}
{\bm{g}}(\bm\uptheta)=\frac{\partial}{\partial \bm\uptheta}E[Q(\bm\uptheta,\bm{V})|\bm\uptheta]=\frac{\partial}{\partial \bm\uptheta}\int_{\Omega_{\bm{V}}}Q(\bm\uptheta,\bm{v})\ dP,
\end{align}
where $P$ denotes a probability measure and $\Omega_{\bm{V}}$ is the domain of $\bm{V}$. Consider the case where $\bm{V}$ is a continuous random variable with probability density function (pdf) $p_{{\bm{V}}}({\bm{v}}|\bm\uptheta)$ that may depend on $\bm\uptheta$. 
Then, (\ref{eq:chungie}) becomes:
\begin{align*}
{\bm{g}}(\bm\uptheta)=\frac{\partial}{\partial \bm\uptheta}\int_{\Omega_{\bm{V}}}Q(\bm\uptheta,\bm{v})p_{\bm{V}}(\bm{v}|\bm\uptheta)\ d\bm{v}.
\end{align*}
If the interchange of differentiation and integration is justified we obtain:
\begin{align}
{\bm{g}}(\bm\uptheta)=
\int_{\Omega_{\bm{V}}}\Bigg[\frac{\partial Q(\bm\uptheta,{\bm{v}})}{\partial \bm\uptheta}+Q(\bm\uptheta,{\bm{v}})\frac{\partial \log(p_{{\bm{V}}}({\bm{v}}|\bm\uptheta))}{\partial \bm\uptheta}\Bigg]p_{{\bm{V}}}({\bm{v}}|\bm\uptheta)\ d{\bm{v}}.\label{eq:bias}
\end{align}
\noindent From  (\ref{eq:bias}) we see that
\begin{align}
\label{eq:dieforyou}
\hat{\bm{g}}_k^{\text{SG}}(\hat{\bm{\uptheta}}_k)=\frac{\partial Q(\bm\uptheta,{\bm{V}})}{\partial \bm\uptheta}\Big|_{\bm\uptheta=\hat{\bm{\uptheta}}_k}+Q(\hat{\bm{\uptheta}}_k,{\bm{V}})\frac{\partial \log[p_{{\bm{V}}}({\bm{V}}|\bm\uptheta)]}{\partial \bm\uptheta}\Big|_{\bm\uptheta=\hat{\bm{\uptheta}}_k},
\end{align}
 is an unbiased measurement of $\bm{g}(\hat{\bm{\uptheta}}_k)$ by construction. 
The following condition is sufficient for the validity of the interchange of differentiation and integration used in (\ref{eq:bias}).

 \begin{theorem}[{\bf{Interchange of Differentiation and Integration}}](Spall 2003, Theorem 15.1)\nocite{ISSO}. Assume $Q(\bm\uptheta,{\bm{v}})p_{{\bm{V}}}({\bm{v}}|\bm\uptheta)$ and $\partial [Q(\bm\uptheta,{\bm{v}})p_{{\bm{V}}}({\bm{v}}|\bm\uptheta)]/\partial \bm\uptheta$ are continuous on $\mathbb{R}^{p}\times \Omega_{\bm{V}}$. Suppose there exist  two nonnegative integrable functions $q_{0}({\bm{v}})$ and $q_{1}({\bm{v}})$ such that for all pairs $(\bm\uptheta,{\bm{v}})\in \mathbb{R}^p\times \Omega_{\bm{V}}$:
\vspace{-.05in}
\begin{align*}
\left|Q(\bm\uptheta,{\bm{v}})p_{{\bm{V}}}({\bm{v}}|\bm\uptheta)\right|\leq q_{0}({\bm{v}});\ \ \  \left\|\frac{\partial [Q(\bm\uptheta,{\bm{v}})p_{{\bm{V}}}({\bm{v}}|\bm\uptheta)]}{\partial \bm\uptheta}\right\|\leq q_{1}({\bm{v}}).
\end{align*}
\noindent Then:
\begin{align*}
{\bm{g}}(\bm\uptheta)=\frac{\partial}{\partial \bm\uptheta}\int_{\Omega_{\bm{V}}}Q(\bm\uptheta,{\bm{V}})p_{{\bm{V}}}({\bm{v}}|\bm\uptheta)\ d{\bm{v}}=\int_{\Omega_{\bm{V}}}\frac{\partial}{\partial \bm\uptheta}\left[Q(\bm\uptheta,{\bm{V}})p_{{\bm{V}}}({\bm{v}}|\bm\uptheta)\right]\ d{\bm{v}}.
\end{align*}
\end{theorem}
 
 When $\bm{V}$ is a discrete random variable taking the values $1,\dots, N$, a relationship analogous to (\ref{eq:bias}) can be obtained. To see this, let $p_{\bm{V}}(\bm{v}|\bm\uptheta)$ denote the probability mass function (pmf) of $\bm{V}$. In this case, 
 \begin{align}
 \label{eq:frordddas}
{\bm{g}}(\bm\uptheta)=\frac{\partial}{\partial \bm\uptheta}\sum_{\bm{v}=1}^{N}Q(\bm\uptheta,\bm{v})p_{\bm{V}}(\bm{v}|\bm\uptheta).
\end{align}
If $N<\infty$, the interchange of change of summation and differentiation in (\ref{eq:frordddas}) is justified (provided the required derivatives exist). This may not be the case, however, when $N=\infty$ (the sum of the gradients of the individual summands may fail to converge or may converge to something other than $\bm{g}(\bm\uptheta)$). When $N=\infty$, sufficient conditions for the interchange of  differentiation and summation can easily be obtained using Theorem 7.17 in Rudin (1976)\nocite{Rudin1976} (although this theorem pertains only to the case where $\bm\uptheta$ is real, the result of the theorem is easily generalized to the multi-dimensional case by applying the theorem to each entry of the gradient vectors). When the exchange of differentiation and summation in (\ref{eq:frordddas}) is valid,
 \begin{align*}
 {\bm{g}}(\bm\uptheta)=  \sum_{\bm{v}=1}^{N}\Bigg[\frac{\partial Q(\bm\uptheta,{\bm{v}})}{\partial \bm\uptheta}+Q(\bm\uptheta,{\bm{v}})\frac{\partial \log[p_{{\bm{V}}}({\bm{v}}|\bm\uptheta)]}{\partial \bm\uptheta}\Bigg]p_{{\bm{V}}}({\bm{v}}|\bm\uptheta).
 \end{align*}
From this,  we deduce that (\ref{eq:dieforyou}) once again represents an unbiased measurement of the gradient.
 
While (\ref{eq:dieforyou}) gives a theoretically valid expression for $\hat{\bm{g}}_k^{\text{SG}}(\hat{\bm{\uptheta}}_k)$, it requires information that may not be available. For example, if $p_{{\bm{V}}}({\bm{v}}|\bm\uptheta)$ is unknown we would not be able to compute ${\partial \log{[p_{{\bm{V}}}({\bm{V}}|\bm\uptheta)]}}/{\partial \bm\uptheta}$. However, if $p_{{\bm{V}}}({\bm{v}}|\bm\uptheta)$ is independent of $\bm\uptheta$ (a common assumption) then (\ref{eq:dieforyou}) simplifies to:
\begin{align}
\label{eq:imanIPA}
\hat{\bm{g}}_k^{\text{SG}}(\hat{\bm{\uptheta}}_k)= \frac{\partial Q(\bm\uptheta,{\bm{V}})}{\partial \bm\uptheta }\Big|_{\bm\uptheta=\hat{\bm{\uptheta}}_k}.
\end{align}
Then, $\hat{\bm{g}}_k^{\text{SG}}(\hat{\bm{\uptheta}}_k)$ could be obtained through direct measurement of $\partial Q(\bm\uptheta,{\bm{V}})/\partial \bm\uptheta$ evaluated at ${\bm\uptheta=\hat{\bm{\uptheta}}_k}$. In the area of simulation-based optimization, the Robbins--Monro algorithm with $\hat{\bm{g}}_k(\hat{\bm{\uptheta}}_k)=\hat{\bm{g}}_k^{\text{SG}}(\hat{\bm{\uptheta}}_k)$ where $\hat{\bm{g}}_k^{\text{SG}}(\hat{\bm{\uptheta}}_k)$ is as in (\ref{eq:imanIPA}) is commonly referred to as the pure IPA algorithm. 

\section[Simultaneous Perturbation SA (SPSA)]{Simultaneous Perturbation SA}
\label{sec:coffee100}

SA is a powerful tool for stochastic optimization. Section \ref{sec:sgfosa} discussed the SG algorithm which used a noisy but unbiased measurement of $\bm{g}(\hat{\bm{\uptheta}}_k)$ to update $\hat{\bm{\uptheta}}_k$. Section \ref{sec:sgfosa} also gave conditions under which a direct measurement of $\partial Q(\bm\uptheta,\bm{V})/\partial \bm\uptheta$ could be used as the noisy unbiased measurement of the gradient. Direct measurement of $\partial Q(\bm\uptheta,\bm{V})/\partial \bm\uptheta$ is certainly a feasible approach in some applications (e.g., Widrow and Stearns 1985). However, measuring $\partial Q(\bm\uptheta,\bm{V})/\partial \bm\uptheta$ is not always possible and is particularly a problem in black-box settings where the form of $Q(\bm\uptheta,\bm{V})$ is unknown. With this motivation, several SA approximate $\bm{g}(\bm\uptheta)$ only using noisy measurements of $Q(\bm\uptheta,\bm{V})$ and, for this reason, these SA algorithms are often said to be {\it{gradient-free}}.

Let us briefly discuss the oldest gradient-free SA method: {\it{finite difference stochastic approximation}} (FDSA). See, for example, Dennis and Schnabel (1989)\nocite{dennisnschnabel1989}. FDSA requires measuring $Q(\bm\uptheta,\bm{V})$ at different values of $\bm\uptheta$. Specifically,
$\hat{\bm{\uptheta}}_{k+1}=\hat{\bm{\uptheta}}_k-a_k \hat{\bm{g}}^{\text{FD}}_k(\hat{\bm{\uptheta}}_k)$,
where the random vector $\hat{\bm{g}}^{\text{FD}}_k(\hat{\bm{\uptheta}}_k)$ is defined as follows:
\begin{align}
\label{eq:marcella}
\hat{\bm{g}}^{\text{FD}}(\hat{\bm{\uptheta}}_k)\equiv \left[\begin{array}{c}\displaystyle{\frac{Q(\hat{\bm\uptheta}_k+c_k\bm{e}_1,\bm{V}^{1+})-Q(\hat{\bm\uptheta}_k-c_k\bm{e}_1,\bm{V}^{1-})}{2c_k}} \\\vdots \\\displaystyle{\frac{Q(\hat{\bm\uptheta}_k+c_k\bm{e}_p,\bm{V}^{p+})-Q(\hat{\bm\uptheta}_k-c_k\bm{e}_p,\bm{V}^{p-})}{2c_k}}\end{array}\right],
\end{align}
where $\bm{e}_i$ denotes the $i$th standard-basis vector in $\mathbb{R}^p$, $c_k>0$ satisfies $c_k\rightarrow 0$, and where $\bm{V}^{1+},\dots, \bm{V}^{p+}, \bm{V}^{1-},\dots, \bm{V}^{p-}$ denote $2p$ different realizations of $\bm{V}$. The update direction in (\ref{eq:marcella}) is referred to as the two-sided FDSA update direction and requires a total of $2p$ noisy loss function measurements at each iteration (recall that $p$ is the dimension of the parameter space). When $p$ is large and if obtaining noisy loss function measurements is costly, requiring $2p$ noisy loss function measurements per-iteration could be prohibitive (this is also an issue for a one-sided version of the FDSA algorithm requiring $p+1$ noisy loss measurements per-iteration).

The {\it{simultaneous perturbation stochastic approximation}} (SPSA) algorithm (e.g., Spall 1992\nocite{spall1992}; Bhatnagar et al. 2013\nocite{bhatnagaretal2013}) is similar to FDSA but requires only two noisy loss function measurements per iteration (independently of $p$). In SPSA, the gradient approximation in (\ref{eq:marcella}) is replaced with the random vector:
\begin{align}
\label{eq:pocres}
\hat{\bm{g}}^{\text{SP}}_k(\hat{\bm{\uptheta}}_k)\equiv \left[\begin{array}{c}\displaystyle{\frac{Q(\hat{\bm\uptheta}_k+c_k\bm\Delta_k,\bm{V}^+)-Q(\hat{\bm\uptheta}_k-c_k\bm\Delta_k,\bm{V}^-)}{2c_k\Delta_{k1}}} \\\vdots \\\displaystyle{\frac{Q(\hat{\bm\uptheta}_k+c_k\bm\Delta_k,\bm{V}^+)-Q(\hat{\bm\uptheta}_k-c_k\bm\Delta_k,\bm{V}^-)}{2c_k\Delta_{kp}}}\end{array}\right],
\end{align}
where $\bm\Delta_k=[\Delta_{k1},\dots,\Delta_{kp}]^\top$ is a random vector, $c_k>0$ is a sequence satisfying $c_k\rightarrow 0$, and $\bm{V}^+$ and $\bm{V}^-$ are two different realizations of $\bm{V}$. The SPSA recursion governing $\{\hat{\bm{\uptheta}}_k\}_{k\geq 0}$ is then
$\hat{\bm{\uptheta}}_{k+1}=\hat{\bm{\uptheta}}_k-a_k\hat{\bm{g}}^{\text{SP}}_k(\hat{\bm{\uptheta}}_k)$.
By defining $\mathcal{F}_k\equiv \{\hat{\bm{\uptheta}}_0,\dots,\hat{\bm{\uptheta}}_k,\bm\Delta_0,\dots,\bm\Delta_{k-1}\}$, $\bm\upbeta_k^{\text{SP}}(\hat{\bm{\uptheta}}_k)\equiv E[\hat{\bm{g}}^{\text{SP}}_k(\hat{\bm{\uptheta}}_k)-\bm{g}(\hat{\bm{\uptheta}}_k)|\mathcal{F}_k]$, and $\bm\upxi_k^{\text{SP}}(\hat{\bm{\uptheta}}_k)\equiv \hat{\bm{g}}^{\text{SP}}_k(\hat{\bm{\uptheta}}_k)-\bm{g}(\hat{\bm{\uptheta}}_k)-\bm\upbeta_k^{\text{SP}}(\hat{\bm{\uptheta}}_k)$,  the connection between the SPSA algorithm and (\ref{eq:youknownothing})--(\ref{eq:tennisey}) becomes apparent. 
%
%

Let us briefly discuss an important property of $\bm\upbeta_k^{\text{SP}}(\hat{\bm{\uptheta}}_k)$ and $\bm\upxi_k^{\text{SP}}(\hat{\bm{\uptheta}}_k)$, the bias and noise terms associated with the SPSA algorithm. In (\ref{eq:pocres}), the vector $\hat{\bm{g}}^{\text{SP}}_k(\hat{\bm{\uptheta}}_k)$ was obtained by adding a perturbation to $\hat{\bm{\uptheta}}_k$ in two opposite directions and collecting a noisy loss measurement at these new locations. 
 Spall (1992)\nocite{spall1992} derives conditions under which the bias of the gradient estimate decreases w.p.1 as $c_k$ decreases. Spall (1992)\nocite{spall1992} also derives conditions under which 
 $\bm\upbeta_k(\hat{\bm{\uptheta}}_k^{\text{SP}})=O(c_k^2) $ w.p.1, where $O(\cdot)$ is the standard big-$O$ notation. This result implies that the bias must decrease at least as fast as $c_k^2$ w.p.1. In other words, as the magnitude of the perturbation decreases, the expected update direction approaches the negative gradient of $L(\bm\uptheta)$. In contrast, the noise term $\bm\upxi_k^{\text{SP}}(\hat{\bm{\uptheta}}_k)$ tends to grow as $c_k$ decreases. More specifically, the magnitude of the covariance matrix of $\bm\upxi_k^{\text{SP}}(\hat{\bm{\uptheta}}_k)$ often increases at least as fast as $1/c_k^2$ (e.g., p. 391 in Spall 2003). Here, having $a_k\rightarrow 0$ serves to dampen the effect of the noise's growing variance.

 \section{Concluding Remarks}
 
 This chapter formally described the stochastic optimization problem and presented the general form of SA algorithms. Through the connection between root-finding and optimization, SA algorithms can be used for stochastic optimization. Although this dissertation focuses on stochastic optimization, the results of Chapters \ref{sec:cyclicseesaw}--\ref{chap:imefficient} also apply to the general root-finding setting (where the update directions are of the form in (\ref{eq:imastandardupdaedirection})) to the extent that the assumptions in Chapters \ref{sec:cyclicseesaw}--\ref{chap:imefficient} remain valid. The following chapter introduces the SA-based algorithm to be the focus of this dissertation.

%% file: chapter1pp.tex

\chapter{Basic and Generalized Cyclic SA}
\label{chap:descriptionofGCSA}

In the basic cyclic optimization algorithm the parameter vector $\bm\uptheta$ is represented in terms of two subvectors:
\begin{align}
\setstretch{1.5} 
\bm\uptheta=\left[\begin{array}{c}\bm\uptheta^{[{1}]} \\\bm\uptheta^{[{2}]}\end{array}\right]\in \mathbb{R}^{p},
\label{eq:analgoysDRAGON}
\end{align}
\noindent where subvector $\bm\uptheta^{[{1}]}$ has length $p'$ with $0<p'<p$. In other words, if $\bm\uptheta=[\uptau_1,\dots,\uptau_p]^\top$ then $\bm\uptheta^{[{1}]}=[\uptau_1,\dots,\uptau_{p'}]^\top$ and $\bm\uptheta^{[{2}]}=[\uptau_{p'+1},\dots,\uptau_p]^\top$. The idea behind cyclic optimization methods is to alternate between updating one of the subvectors while holding the other fixed, the iterative process continues until a stopping criterion has been satisfied. Since $\bm\uptheta$ is divided into two subvectors we refer to this process as cyclic {\it{seesaw}} optimization due to its back-and-forth nature. In this chapter we introduce the {{cyclic seesaw SA}} algorithm, a cyclic implementation of SA procedures for nonlinear root-finding. This chapter also introduces the generalization to cyclic seesaw SA which will be the focus of this work (see Section \ref{sec:wunderbar}). One extension permitted by the generalization pertains to a version of the algorithm that is not necessarily strictly alternating (not cyclic seesaw). Here, there may be more than two subvectors possibly with shared components and it is known that each will be updated infinitely often as the iteration count grows to infinity. A special case of this extension arises, for example, when at each iteration a random variable dictates which subvector to update. Section \ref{sec:keatonmodasd} formally describes the cyclic seesaw SA algorithm (with two special cases given in Sections \ref{sec:ay} and \ref{sect:cspsa}) and Section \ref{sec:wunderbar} describes the generalized cyclic SA algorithm. Section \ref{sec:cocnremarksGCSA} contains concluding remarks.

\section{Cyclic Seesaw SA}
\label{sec:keatonmodasd}  

We first introduce the following notation: given any vector $\bm{v}\in \mathbb{R}^p$, let the vectors ${\bm{v}}^{(1)}$ and ${\bm{v}}^{(2)}$ be the result of replacing the last $p-p'$ coordinates or the first $p'$ coordinates of ${\bm{v}}$ with zeros, respectively. Specifically,
 \begin{align}
 \setstretch{1.5} 
  \label{eq:electricpulsar}
 \bm{v}^{(1)}=\left[\begin{array}{c}v_1  \\\vdots \\v_{p'}\\{\bm{0}}\end{array}\right]\in \mathbb{R}^p, \ \ \ \bm{v}^{(2)}=\left[\begin{array}{c}{\bm{0}} \\v_{p'+1}  \\\vdots \\v_{p}\end{array}\right]\in\mathbb{R}^p,
 \end{align} 
where $v_i$ denotes the $i$th entry of $\bm{v}$ and $\bm{0}$ denotes a vector of zeros. Thus, for example, $\bm{g}(\bm\uptheta)=\bm{g}^{(1)}(\bm\uptheta)+\bm{g}^{(2)}(\bm\uptheta)$. Using the notation in (\ref{eq:electricpulsar}), the following recursion defines cyclic seesaw SA:
\begin{align}
\bm{{\hat{\uptheta}}}_{k}^{(I)}=\bm{{\hat{\uptheta}}}_{k}-a_{k}^{(1)}\hat{\bm{g}}^{(1)}_k(\hat{\bm\uptheta}_k),\ \ \ \bm{{\hat{\uptheta}}}_{k+1}=\bm{{\hat{\uptheta}}}_{k}^{(I)}-a_{k}^{(2)}\hat{\bm{g}}^{(2)}_k(\hat{\bm\uptheta}_k^{(I)}),\label{eq:propoto2}
\end{align}
where $a_{k}^{(1)}$ and $a_{k}^{(2)}$ are sequences of positive scalars and $\hat{\bm{g}}^{(1)}_k(\hat{\bm\uptheta}_k)$ and $\hat{\bm{g}}^{(2)}_k(\hat{\bm\uptheta}_k^{(I)})$ denote noisy approximations of $\bm{g}^{(1)}(\hat{\bm{\uptheta}}_k)$ and $\bm{g}^{(2)}(\hat{\bm{\uptheta}}_k^{(I)})$, respectively. 
Throughout our description of the cyclic seesaw algorithm we use the superscript ``$(I)$'' to denote the intermediate step of an iteration, this is done to emphasize the seesaw nature of the algorithm. Using (\ref{eq:tennisey}), (\ref{eq:propoto2}) may also be rewritten as:
\begin{subequations}
\begin{align}
\bm{{\hat{\uptheta}}}_{k}^{(I)}&=\bm{{\hat{\uptheta}}}_{k}-a_{k}^{(1)}\left[{\bm{g}}^{(1)}(\bm{{\hat{\uptheta}}}_{k})+{\bm{\upbeta}}_k^{(1)}(\bm{{\hat{\uptheta}}}_{k})+{\bm{\upxi}}_k^{(1)}(\bm{{\hat{\uptheta}}}_{k})\right],\label{eq:byestablishing0}\\
\bm{{\hat{\uptheta}}}_{k+1}&=\bm{{\hat{\uptheta}}}_{k}^{(I)}-a_{k}^{(2)}\left[{\bm{g}}^{(2)}(\bm{{\hat{\uptheta}}}_{k}^{(I)})+{\bm{\upbeta}_k^{(2)}}(\bm{{\hat{\uptheta}}}_{k}^{(I)})+{\bm{\upxi}_k^{(2)}}(\bm{{\hat{\uptheta}}}_{k}^{(I)})\right].\label{eq:byestablishing}
\end{align}
\end{subequations}
The following section describes the cyclic seesaw SG algorithm, a special case of (\ref{eq:byestablishing0},b) where the gradient estimates are unbiased (see Section \ref{sec:sgfosa}).

\section{Cyclic Seesaw SG}
\label{sec:ay}

The cyclic seesaw SG algorithm produces its updates according to (\ref{eq:propoto2}) by using unbiased noisy gradient measurements to obtain the update directions $\hat{\bm{g}}^{(1)}_k(\hat{\bm\uptheta}_k)$ and $\hat{\bm{g}}^{(2)}_k(\hat{\bm\uptheta}_k^{(I)})$. Specifically, 
cyclic seesaw SG recursion is as follows:
\begin{align}
\bm{{\hat{\uptheta}}}_{k}^{(I)}=\bm{{\hat{\uptheta}}}_{k}-a_{k}^{(1)}\left[\hat{\bm{g}}^{\text{SG}}(\bm{{\hat{\uptheta}}}_{k})\right]^{(1)},\ \ \ \bm{{\hat{\uptheta}}}_{k+1}=\bm{{\hat{\uptheta}}}_{k}^{(I)}-a_{k}^{(2)}\left[\hat{\bm{g}}^{\text{SG}}(\bm{{\hat{\uptheta}}}_{k}^{(I)})\right]^{(2)},
 \label{eq:goonnowharderdinosaur}
\end{align}
where $\hat{\bm{g}}^{\text{SG}}(\bm{{\hat{\uptheta}}}_{k})$ and $\hat{\bm{g}}^{\text{SG}}(\bm{{\hat{\uptheta}}}_{k}^{(I)})$ are the SG estimates of $\bm{g}(\bm\uptheta)$ at $\hat{\bm{\uptheta}}_k$ and $\hat{\bm{\uptheta}}_k^{(I)}$, respectively (see Section \ref{sec:sgfosa}). Because the SG estimates are unbiased (by assumption), there exist vectors ${\bm{\upxi}}_k^{\text{SG}}(\bm{{\hat{\uptheta}}}_{k})$ and ${\bm{\upxi}}_k^{\text{SG}}(\bm{{\hat{\uptheta}}}_{k}^{(I)})$ with $E[{\bm{\upxi}}_k^{\text{SG}}(\bm{{\hat{\uptheta}}}_{k})|\hat{\bm{\uptheta}}_k]=\bm{0}$ and $E[{\bm{\upxi}}_k^{\text{SG}}(\bm{{\hat{\uptheta}}}_{k}^{(I)})|\hat{\bm{\uptheta}}_k^{(I)}]=\bm{0}$ such that:
\begin{align}
\hat{\bm{g}}^{\text{SG}}(\bm{{\hat{\uptheta}}}_{k})={\bm{g}}(\bm{{\hat{\uptheta}}}_{k})+{\bm{\upxi}}_k^{\text{SG}}(\bm{{\hat{\uptheta}}}_{k}), \ \ \  \hat{\bm{g}}^{\text{SG}}(\bm{{\hat{\uptheta}}}_{k}^{(I)})={\bm{g}}(\bm{{\hat{\uptheta}}}_{k}^{(I)})+{\bm{\upxi}}_k^{\text{SG}}(\bm{{\hat{\uptheta}}}_{k}^{(I)}).
 \label{eq:uglypeppym}
\end{align}
Cyclic seesaw SG is then a special case of (\ref{eq:byestablishing0},b) in which the bias terms are zero. It is important to note that the vectors $\hat{\bm{g}}^{\text{SG}}(\bm{{\hat{\uptheta}}}_{k})$ and $\hat{\bm{g}}^{\text{SG}}(\bm{{\hat{\uptheta}}}_{k}^{(I)})$ in (\ref{eq:uglypeppym}) are defined only for theoretical purposes and, in practice, (\ref{eq:goonnowharderdinosaur}) implies it is unnecessary to compute all the entries in $\hat{\bm{g}}^{\text{SG}}(\bm{{\hat{\uptheta}}}_{k})$ and $\hat{\bm{g}}^{\text{SG}}(\bm{{\hat{\uptheta}}}_{k}^{(I)})$ since several entries are replaced with zeros in (\ref{eq:goonnowharderdinosaur}). The following section introduces the cyclic seesaw SPSA algorithm, another special case of cyclic seesaw SA.


\section{Cyclic Seesaw SPSA}
\label{sect:cspsa}

In the SPSA algorithm, two noisy loss function measurements are used to update the vector $\hat{\bm\uptheta}_k$ (see Section \ref{sec:coffee100}). To implement SPSA in a cyclic seesaw manner, one possibility is to replace the noisy gradient estimates in (\ref{eq:propoto2}) with gradient approximations of the form in (\ref{eq:pocres}). In other words, cyclic SPSA could be implemented as follows:
\begin{align}
\bm{{\hat{\uptheta}}}_{k}^{(I)}=\bm{{\hat{\uptheta}}}_{k}-a_{k}^{(1)}[\hat{\bm{g}}^{\text{SP}}_k(\hat{\bm\uptheta}_k)]^{(1)},\ \ \ \bm{{\hat{\uptheta}}}_{k+1}=\bm{{\hat{\uptheta}}}_{k}^{(I)}-a_{k}^{(2)}[\hat{\bm{g}}^{\text{SP}}_k(\hat{\bm\uptheta}_k^{(I)})]^{(2)},
\label{eq:moumousika}
\end{align}
where $\hat{\bm{g}}^{\text{SP}}_k(\hat{\bm\uptheta}_k)$ and $\hat{\bm{g}}^{\text{SP}}_k(\hat{\bm\uptheta}_k^{(I)})$ are noisy estimates of $\bm{g}(\hat{\bm{\uptheta}}_k)$ and $\bm{g}(\hat{\bm{\uptheta}}_k^{(I)})$, respectively, obtained using SPSA (see Section \ref{sec:coffee100}). Specifically,
%
%
\begin{align}
\hat{\bm{g}}^{\text{SP}}_k(\hat{\bm{\uptheta}}_k)&= \left[\begin{array}{c}\displaystyle{\frac{Q(\hat{\bm\uptheta}_k+c_k^{(1)}\bm\Delta_k, \bm{V}^{+}_k)-Q(\hat{\bm\uptheta}_k-c_k^{(1)}\bm\Delta_k, \bm{V}^{-}_k)}{2c_k^{(1)}\Delta_{k1}}} \\\vdots \\ \displaystyle{\frac{Q(\hat{\bm\uptheta}_k+c_k^{(1)}\bm\Delta_k,\bm{V}^{+}_k)-Q(\hat{\bm\uptheta}_k-c_k^{(1)}\bm\Delta_k,\bm{V}^{-}_k)}{2c_k^{(1)}\Delta_{kp}}}\end{array}\right],
\label{eq:shmasdfg}
\end{align}
for some sequence $\{c_k^{(1)}\}_{k\geq0}$ and random vector $\bm\Delta_k=[\Delta_{k1},\dots,\Delta_{kp}]^\top$, and
\begin{align}
\hat{\bm{g}}^{\text{SP}}_k(\hat{\bm{\uptheta}}_k^{(I)})&= \left[\begin{array}{c}\displaystyle{\frac{Q(\hat{\bm\uptheta}_k^{(I)}+c_k^{(2)}\bm\Delta_k^{(I)},\bm{V}^{(I)+}_k)-Q(\hat{\bm\uptheta}_k^{(I)}-c_k^{(2)}\bm\Delta_k^{(I)},\bm{V}^{(I)-}_k)}{2c_k^{(2)}\Delta_{k1}^{(I)}}} \\\vdots \\ \displaystyle{\frac{Q(\hat{\bm\uptheta}_k^{(I)}+c_k^{(2)}\bm\Delta_k^{(I)},\bm{V}^{(I)+}_k)-Q(\hat{\bm\uptheta}_k^{(I)}-c_k^{(2)}\bm\Delta_k^{(I)},\bm{V}^{(I)-}_k)}{2c_k^{(2)}\Delta_{kp}^{(I)}}}\end{array}\right],\label{eq:splice}
\end{align}
for a sequence $\{c_k^{(2)}\}_{k\geq0}$ (not necessarily equal to $\{c_k^{(1)}\}_{k\geq0}$) and a random vector $\bm\Delta_k^{(I)}=[\Delta_{k1}^{(I)},\dots,\Delta_{kp}^{(I)}]^\top$, where $\bm{V}^{+}_k, \bm{V}^{-}_k, \bm{V}^{(I)+}_k, \bm{V}^{(I)-}_k$ denote four different realizations of $\bm{V}$. Note that the definitions of $\hat{\bm{g}}^{\text{SP}}_k(\hat{\bm{\uptheta}}_k)$ and $\hat{\bm{g}}^{\text{SP}}_k(\hat{\bm{\uptheta}}_k^{(I)})$ given in 
%
%
(\ref{eq:shmasdfg}) and (\ref{eq:splice}) 
 allow adding a small perturbation to {\it{all}} elements of the vectors $\hat{\bm{\uptheta}}_k$ and $\hat{\bm{\uptheta}}_k^{(I)}$, respectively.  An alternative definition of cyclic seesaw SPSA can be obtained by only perturbing the entries to be updated. This is discussed next.
 
From (\ref{eq:moumousika}) it follows that there is no need to estimate the last $p'$ entries of $\hat{\bm{g}}^{\text{SP}}_k(\hat{\bm\uptheta}_k)$ or the first $p-p'$ entries of $\hat{\bm{g}}^{\text{SP}}_k(\hat{\bm\uptheta}_k^{(I)})$. 
%
%
 Thus, an alternative definition of cyclic seesaw SPSA can be obtained by setting: 
\begin{align}
\hat{\bm{g}}^{\text{SP}}_k(\hat{\bm{\uptheta}}_k)= \left[\begin{array}{c}\displaystyle{\frac{Q(\hat{\bm\uptheta}_k+c_k^{(1)}\bm\Delta_k,\bm{V}_k^+)-Q(\hat{\bm\uptheta}_k-c_k^{(1)}\bm\Delta_k, \bm{V}_k^-)}{2c_k^{(1)}\Delta_{k1}}} \\\vdots \\\displaystyle{\frac{Q(\hat{\bm\uptheta}_k+c_k^{(1)}\bm\Delta_k,\bm{V}_k^+)-Q(\hat{\bm\uptheta}_k-c_k^{(1)}\bm\Delta_k,\bm{V}_k^-)}{2c_k^{(1)}\Delta_{kp'}}}\\ \bm{0}\end{array}\right],\label{eq:howeardspos1}
\end{align}
where $\bm\Delta_k=[\Delta_{k1},\dots,\Delta_{kp'},0,\dots,0]^\top$, and:
\begin{align}
\hat{\bm{g}}^{\text{SP}}_k(\hat{\bm{\uptheta}}_k^{(I)})= \left[\begin{array}{c}\bm{0}\\ \displaystyle{\frac{Q(\hat{\bm\uptheta}_k^{(I)}+c_k^{(2)}\bm\Delta_k^{(I)},\bm{V}_k^{(I)+})-Q(\hat{\bm\uptheta}_k^{(I)}-c_k^{(2)}\bm\Delta_k^{(I)},\bm{V}_k^{(I)-})}{2c_k^{(2)}\Delta_{k(p'+1)}^{(I)}}} \\\vdots \\\displaystyle{\frac{Q(\hat{\bm\uptheta}_k^{(I)}+c_k^{(2)}\bm\Delta_k^{(I)},\bm{V}_k^{(I)+})-Q(\hat{\bm\uptheta}_k^{(I)}-c_k^{(2)}\bm\Delta_k^{(I)},\bm{V}_k^{(I)-})}{2c_k^{(2)}\Delta_{kp}^{(I)}}}\end{array}\right],\label{eq:howeardspos2}
\end{align}
where $\bm\Delta_k^{(I)}=[0,\dots,0,\Delta_{k(p'+1)}^{(I)},\dots,\Delta_{kp}^{(I)}]^\top$. 

%
Regardless of whether $\hat{\bm{g}}^{\text{SP}}_k(\hat{\bm\uptheta}_k)$ and $\hat{\bm{g}}^{\text{SP}}_k(\hat{\bm\uptheta}_k^{(I})$ in  (\ref{eq:moumousika}) are obtained according to  (\ref{eq:shmasdfg})--(\ref{eq:splice}) or according to (\ref{eq:howeardspos1})--(\ref{eq:howeardspos2}), the resulting algorithm is a special case of (\ref{eq:byestablishing0},b). To see this, define $\mathcal{F}_k\equiv \{\hat{\bm{\uptheta}}_0,\dots,\hat{\bm{\uptheta}}_k,\bm\Delta_0,\dots,\bm\Delta_{k-1},\bm\Delta_0^{(I)},\dots,\bm\Delta_{k-1}^{(I)}\}$ and $\mathcal{F}_k^{(I)}\equiv \{\mathcal{F}_k, \hat{\bm{\uptheta}}_k^{(I)} , \bm\Delta_{k}\}$.
 Then, define the bias terms in (\ref{eq:byestablishing0},b) as follows:
\begin{subequations}
\begin{align}
{\bm{\upbeta}}^{(1)}_k(\hat{\bm{\uptheta}}_k)&=E\left[\hat{\bm{g}}^{\text{SP}}_k(\hat{\bm{\uptheta}}_k)-{\bm{g}}(\hat{\bm{\uptheta}}_k)\Big|\mathcal{F}_k\right]^{(1)},\label{eq:savemebias0}\\
 {\bm{\upbeta}}^{(2)}_k(\hat{\bm{\uptheta}}_k^{(I)})&=E\left[\hat{\bm{g}}^{\text{SP}}_k(\hat{\bm{\uptheta}}_k^{(I)})-{\bm{g}}^{(2)}(\hat{\bm{\uptheta}}_k^{(I)})\Big|\mathcal{F}_k^{(I)}\right]^{(2)},
\label{eq:savemebias}
 \end{align}
 \end{subequations}
and define the noise terms in (\ref{eq:byestablishing0},b) as follows:
 \begin{subequations}
 \begin{align}
 {\bm{\upxi}}^{(1)}_k(\hat{\bm{\uptheta}}_k)&=\left(\hat{\bm{g}}^{\text{SP}}_k(\hat{\bm{\uptheta}}_k)-E[\hat{\bm{g}}^{\text{SP}}_k(\hat{\bm{\uptheta}}_k)\Big|\mathcal{F}_k]\right)^{(1)}, \label{eq:savemenoise0}\\
   {\bm{\upxi}}^{(2)}_k(\hat{\bm{\uptheta}}_k^{(I)})&=\left(\hat{\bm{g}}^{\text{SP}}_k(\hat{\bm{\uptheta}}_k^{(I)})-E\left[\hat{\bm{g}}^{\text{SP}}_k(\hat{\bm{\uptheta}}_k^{(I)})\Big|\mathcal{F}_k^{(I)}\right]\right)^{(2)}.
   \label{eq:savemenoise}
\end{align}
 \end{subequations}
We may then write $[\hat{\bm{g}}^{\text{SP}}_k(\hat{\bm{\uptheta}}_k)]^{(1)}={\bm{g}}^{(1)}(\bm{{\hat{\uptheta}}}_{k})+{\bm{\upbeta}}_k^{(1)}(\bm{{\hat{\uptheta}}}_{k})+{\bm{\upxi}}_k^{(1)}(\bm{{\hat{\uptheta}}}_{k})$ and, similarly, $[\hat{\bm{g}}^{\text{SP}}_k(\hat{\bm{\uptheta}}_k^{(I)})]^{(2)}={\bm{g}}^{(2)}(\bm{{\hat{\uptheta}}}_{k}^{(I)})+{\bm{\upbeta}_k^{(2)}}(\bm{{\hat{\uptheta}}}_{k}^{(I)})+{\bm{\upxi}_k^{(2)}}(\bm{{\hat{\uptheta}}}_{k}^{(I)})$. Therefore, cyclic seesaw SPSA is a special case of (\ref{eq:byestablishing0},b). The following section introduces the generalized cyclic SA algorithm, the algorithm to be the focus of this dissertation.

\section[Generalized Cyclic SA (GCSA)]{Generalized Cyclic SA}
\label{sec:wunderbar}

While the cyclic seesaw algorithm alternates between updating each of two subvectors of $\bm\uptheta$, it is easy to conceive of an algorithm that does not exhibit this strictly alternating nature. For example, suppose that the vector to update is chosen according to a Bernoulli (i.e., binary) random variable with parameter $q_k$ (``success'' in the Bernoulli trial could mean the first subvector is updated in which case ``failure'' would mean the second subvector is updated). Formally,
\begin{align}
&{\text{Let $\mathcal{X}_k\sim \text{Bernoulli}(q_k)$}},  \notag \\
&{\text{Let $j=1$ if $\mathcal{X}_k=1$ and $j=2$ otherwise}},\notag \\
&\bm{{\hat{\uptheta}}}_{k+1}=\bm{{\hat{\uptheta}}}_{k}-\tilde{a}_{k}^{(j)}\hat{\bm{g}}^{(j)}_k(\hat{\bm\uptheta}_k)\label{eq:beer149}
\end{align}
where $\tilde{a}_k^{(j)}$ is the first unused element of a sequence $\{a_i^{(j)}\}_{i\geq0}$ with $a_i^{(j)}>0$, $q_k$ is a value in the interval $(0,1)$, and where $\mathcal{X}_k\sim \text{Bernoulli}(q_k)$ implies $\mathcal{X}_k$ has a Bernoulli distribution with success probability $q_k$.

Let us draw attention to two important differences between (\ref{eq:propoto2}) and (\ref{eq:beer149}). First, because there is no intermediate step per se in (\ref{eq:beer149}), the superscript ``$(I)$'' is avoided altogether. Second, the deterministic gain $a_k^{(j)}$ in (\ref{eq:propoto2}) has been replaced by the random gain $\tilde{a}_k^{(j)}$. Since the convergence theory for SA procedures typically requires $a_k\rightarrow 0$, using $\tilde{a}_k^{(j)}$ in place of $a_k^{(j)}$ is done with the intention of ensuring that the gain sequence does not become too small too soon. Suppose, for example, that $\mathcal{X}_k=1$ for $k=0,\cdots,99$ and $\mathcal{X}_{100}=0$ so that the first time the second subvector is updated is during iteration 101. 
Setting
\begin{align*}
\bm{{\hat{\uptheta}}}_{101}=\bm{{\hat{\uptheta}}}_{100}-{a}_{100}^{(2)}\ \hat{\bm{g}}^{(2)}_{100}(\hat{\bm\uptheta}_{100})
\end{align*}
in this case might prove disadvantageous since $a_{100}^{(2)}$ could be very small, thus resulting in only a very small modification to the second subvector even when $\|\hat{\bm{g}}^{(2)}_{100}(\hat{\bm\uptheta}_{100})\|$ is not close to zero. This is less of an issue when using $\tilde{a}_k^{(j)}$ since $\tilde{a}_{100}^{(2)}=a_0^{(2)}$, a gain sequence value that is likely much larger than $a_{100}^{(2)}$.


As another example of how the cyclic seesaw algorithm might be generalized, the subvectors could be updated according to some predetermined pattern (cyclic seesaw is a particular case). For example, suppose that at each iteration the parameter vector is updated according to the following repeating pattern: two updates to $\bm\uptheta^{[1]}$ followed by a single update to $\bm\uptheta^{[2]}$ and, lastly, two updates to $\bm\uptheta^{[1]}$ (see (\ref{eq:analgoysDRAGON}) for the definition of $\bm\uptheta^{[j]}$). This algorithm could be written as:
\begin{subequations}
\begin{align}
&\bm{{\hat{\uptheta}}}_{k}^{(I_{1,1})}\equiv\bm{{\hat{\uptheta}}}_{k}-{a}_{k}^{(1)}\ \hat{\bm{g}}^{(1)}_k(\hat{\bm\uptheta}_k),\label{eq:believeeyes0}\\
&\bm{{\hat{\uptheta}}}_{k}^{(I_{1,2})}\equiv\bm{{\hat{\uptheta}}}_{k}^{(I_{1,1})}-{a}_{k}^{(1)}\ \hat{\bm{g}}^{(1)}_k\left(\hat{\bm\uptheta}_k^{(I_{1,1})}\right),\label{eq:believeeyes1}\\
&\bm{{\hat{\uptheta}}}_{k}^{(I_{2,1})}\equiv\bm{{\hat{\uptheta}}}_{k}^{(I_{1,2})}-{a}_{k}^{(2)}\ \hat{\bm{g}}^{(2)}_k\left(\hat{\bm\uptheta}_k^{(I_{1,2})}\right),\label{eq:believeeyes2}\\
&\hat{\bm\uptheta}_{k}^{(I_{3,1})}\equiv\bm{{\hat{\uptheta}}}_{k}^{(I_{2,1})}-{a}_{k}^{(1)}\ \hat{\bm{g}}^{(1)}_k\left(\hat{\bm\uptheta}_k^{(I_{2,1})}\right),\label{eq:believeeyes3}\\
&\hat{\bm\uptheta}_{k}^{(I_{3,2})}\equiv\bm{{\hat{\uptheta}}}_{k}^{(I_{3,1})}-{a}_{k}^{(1)}\ \hat{\bm{g}}^{(1)}_k\left(\hat{\bm\uptheta}_k^{(I_{3,1})}\right),\label{eq:believeeyes4}\\
&\hat{\bm{\uptheta}}_{k+1}= \hat{\bm\uptheta}_{k}^{(I_{3,2})}.
\label{eq:believeeyes}
\end{align}
\end{subequations}
Another way to visualize (\ref{eq:believeeyes0}--f) is given below:
\begin{align}
\hat{\bm{\uptheta}}_k\overset{(1)}{\xrightarrow{\hspace*{0.8cm}}} {{\hat{\bm\uptheta}_k^{(I_{1,1})}}}\overset{(1)}{\xrightarrow{\hspace*{0.8cm}}} \hat{\bm\uptheta}_k^{(I_{1,2})}\overset{(2)}{\xrightarrow{\hspace*{0.8cm}}}\hat{\bm\uptheta}_k^{(I_{2,1})}\overset{(1)}{\xrightarrow{\hspace*{0.8cm}}} \hat{\bm\uptheta}_k^{(I_{3,1})}\overset{(1)}{\xrightarrow{\hspace*{0.8cm}}} {\underbrace{\hat{\bm\uptheta}_{k+1}}_{=\hat{\bm\uptheta}_k^{(I_{3,2})}}},\label{eq:lifetime}
\end{align}
where the numbers above the arrows indicate which subvector was updated. In (\ref{eq:lifetime}) there are three main update ``blocks'': the first block consists of two updates to the first subvector (lines (\ref{eq:believeeyes0},b)), the second block consists of a single update to the second subvector (line (\ref{eq:believeeyes2})), and the third block consists of two more updates to the second subvector (lines (\ref{eq:believeeyes3},e)). Next we provide a slight generalization of  (\ref{eq:believeeyes0}--f) to allow for any general (deterministic) pattern for selecting the subvector to update. We begin by giving the following formal definition of what constitutes a ``block'' of subvector updates.

\begin{definition}
\label{def:block}
 Let a block of subvector updates be defined as a {\it{maximal consecutive sequence of updates on the same subvector}}. 
 \end{definition}
 
Let $s\geq 1$ denote the number of blocks in a single iteration (in the previous example we have $s=3$) and let $n(m)\geq 1$ denote the number of updates in the $m$th block for $m=1,\dots,s$ (in the previous example, $n(1)=2$, $n(2)=1$, and $n(3)=2$). Then, the vectors of the form $\hat{\bm{\uptheta}}_k^{(I_{m,i})}$ in (\ref{eq:believeeyes0}--f) can be interpreted as the vector obtained during the $m$th block of the $(k+1)$st iteration after having performed exactly $i\geq 0$ updates within the block. Note that this interpretation 
 implies $\hat{\bm{\uptheta}}_k^{(I_{m,n(m)})}=\hat{\bm{\uptheta}}_k^{(I_{m+1,0})}$ when $m<s$.  Including this observation in (\ref{eq:lifetime}):\label{def:intermediate}
 \begin{align*}
{\underbrace{\hat{\bm{\uptheta}}_k}_{=\hat{\bm{\uptheta}}_k^{(I_{1,0})}}}\overset{(1)}{\xrightarrow{\hspace*{0.8cm}}} {{\hat{\bm\uptheta}_k^{(I_{1,1})}}}\overset{(1)}{\xrightarrow{\hspace*{0.8cm}}} {\underbrace{\hat{\bm\uptheta}_k^{(I_{1,2})}}_{=\hat{\bm{\uptheta}}_k^{(I_{2,0})}}}\overset{(2)}{\xrightarrow{\hspace*{0.8cm}}}{\underbrace{\hat{\bm\uptheta}_k^{(I_{2,1})}}_{=\hat{\bm{\uptheta}}_k^{(I_{3,0})}}}\overset{(1)}{\xrightarrow{\hspace*{0.8cm}}} \hat{\bm\uptheta}_k^{(I_{3,1})}\overset{(1)}{\xrightarrow{\hspace*{0.8cm}}} {\underbrace{\hat{\bm\uptheta}_{k+1}}_{=\hat{\bm\uptheta}_k^{(I_{3,2})}}}.
\end{align*}
While at first it may seem convoluted to assign two different labels to the same vector, having $\hat{\bm{\uptheta}}_k^{(I_{m,n(m)})}=\hat{\bm{\uptheta}}_k^{(I_{m+1,0})}$ for $m<s$ serves to significantly simplify the presentation of the generalized cyclic SA algorithm of which (\ref{eq:beer149}) and (\ref{eq:believeeyes0}--f) are special cases.

An alternative way to write the algorithm in (\ref{eq:believeeyes0}--f) is as follows:
\begin{align}
&\bm{F}_k\equiv \left[{a}_{k}^{(1)}\ \hat{\bm{g}}^{(1)}_k(\hat{\bm\uptheta}_k)+{a}_{k}^{(2)}\ \hat{\bm{g}}^{(2)}_k\left(\hat{\bm\uptheta}_k^{(I_{1,2})}\right)+\sum_{i=1}^3{a}_{k}^{(1)}\ \hat{\bm{g}}^{(1)}_k\left(\hat{\bm\uptheta}_k^{(I_{i,1})}\right)\right]/a_k,\notag \\
&\hat{\bm\uptheta}_{k+1}=\hat{\bm\uptheta}_k-a_k\bm{F}_k,\label{eq:william}
\end{align}
for some sequence $\{a_i\}_{i\geq 0}$ with $a_i>0$.
Similarly, the algorithm in (\ref{eq:beer149}) can also be written in the form $\hat{\bm\uptheta}_{k+1}=\hat{\bm\uptheta}_k-a_k\bm{F}_k$ after redefining the vector $\bm{F}_k$. Specifically, the algorithm in (\ref{eq:beer149}) can be written as follows:
\begin{align}
&{\text{Let $\mathcal{X}_{k}\sim \text{Bernoulli}(q_k)$}},\notag \\
&{\text{Let $j=1$ if $\mathcal{X}_{k}=1$ and $j=2$ otherwise}},\notag \\
&\bm{F}_k\equiv \left[\tilde{a}_{k}^{(j)}\hat{\bm{g}}^{(j)}_k(\hat{\bm\uptheta}_k)\right]/a_k,\notag\\
&\bm{{\hat{\uptheta}}}_{k+1}=\bm{{\hat{\uptheta}}}_{k}-a_k\bm{F}_k.\label{eq:beerabove}
\end{align}
In fact, (\ref{eq:beerabove}) is a variant of (\ref{eq:william}) in which $s=1$ (there is only one block per iteration), $n(1)=1$ (there is only one update within the block), and the subvector to update within the block is selected at random. 

In general, it is possible to conceive of a vector $\bm{F}_k$ with a more general form than that in (\ref{eq:william}) or (\ref{eq:beerabove}). The number of blocks in an iteration of (\ref{eq:believeeyes0}--f), for example, could be a random variable, $s_k$, depending on the iteration number. The number of updates within a block could also be a random variable, $n_k(m)$, depending on the iteration number as well as on the block number. Moreover, the choice of which subvector to update during the $m$th block could also be a random variable, $j_k(m)$, depending on $k$. These generalizations are captured by the generalized cyclic SA algorithm introduced next. 

Let $\{\mathcal{S}_j\}_{j=1}^d$ be (not necessarily disjoint) subsets of $\mathcal{S}\equiv \{1,\dots,p\}$ satisfying:
\begin{align}
\label{eq:notexclusive}
\bigcup_{j=1}^d\mathcal{S}_j=\mathcal{S}.
\end{align}
Next, for $j=1,\dots,d$ and any vector $\bm{v}=[v_1,\dots,v_p]^\top\in \mathbb{R}^p$ define the vector $\bm{v}^{(j)}$, with $i$th entry denoted by ${v}^{(j)}_i$, as follows:
\begin{align}
\label{eq:alonesong}
{v}^{(j)}_i=
\begin{cases} v_i &\mbox{if } i \in \mathcal{S}_j,  \\ 
0 & \mbox{otherwise, }  \end{cases} 
\end{align}
this is simply a generalization of the notation used in (\ref{eq:electricpulsar}) to the case where there are $d$ subvectors possibly with shared components. Note that it is possible that $\sum_{j=1}^d\bm{g}^{(j)}(\bm\uptheta)\neq\bm{g}(\bm\uptheta)$ since $\{\mathcal{S}_j\}_{j=1}^d$ are not necessarily disjoint. The generalized cyclic stochastic approximation (GCSA) algorithm is described next.

\begin{algorithm}[p]                     
\caption{Generalized Cyclic Stochastic Approximation (GCSA)}          
\label{findme}                          
\begin{algorithmic} [1]                   
\setstretch{1.5} 
\REQUIRE    $\hat{\bm{\uptheta}}_0$ and $\{a_k^{(j)}\}_{k\geq 0}$ for $j=1,\dots,d$. Set $k=0$.
\WHILE{stopping criterion has not been reached}
         \STATE{Let $s_k\in \mathbb{Z}^+$ be a random variable, where $\mathbb{Z}^+$ denotes the positive integers.
          }
         \FOR{$m=1,\dots,s_k$}
         \STATE{Let $j_k(m)\in \{1,\dots, d\}$ be a random variable, where $d$ is defined above (\ref{eq:notexclusive}) (recall $d$ is the number of subvectors that can be updated).
         } 
         \STATE{Let $n_k(m)\in \mathbb{Z}^+$ be a random variable.}
	\FOR{$i=1,\dots,n_k(m)$}\label{line:lineme}
	\STATE{Define:\label{eq:indianindinorange}
	\begin{align*}
	\hat{\bm\uptheta}_k^{(I_{m,i})}\equiv \hat{\bm\uptheta}_k&- \sum_{z=1}^{m-1}\sum_{\ell=0}^{n_k(z)-1} \left[\tilde{A}_{k}(z)\ \hat{\bm{g}}^{(j_k(z))}_k\left(\hat{\bm\uptheta}_k^{(I_{z,\ell})}\right)\right]-\sum_{\ell=0}^{i-1} \left[\tilde{A}_{k}(m)\ \hat{\bm{g}}^{(j_k(m))}_k\left(\hat{\bm\uptheta}_k^{(I_{m,\ell})}\right)\right],
	\end{align*}
	where $\tilde{A}_{k}(m)\equiv\sum_{j=1}^d\chi\{j_k(m)=j\}\tilde{a}_k^{(j)}$\label{eq:attaritapo}, $\tilde{a}_k^{(j)}$ is the first unused element of the predetermined sequence $\{a_k^{(j)}\}_{k\geq0}$, $\chi \{\mathcal{E}\}$ denotes the indicator function of the event $\mathcal{E}$, and $\sum_{i=a}^b (\cdot)_i=0$ whenever $b<a$.   A mathematical definition of the random sequence $\tilde{a}_k^{(j)}$ is given in (\ref{eq:saropian}) and Figure \ref{fig:answerscully} gives a visual representation of the process through which $\tilde{a}_k^{(j)}$ is updated.  
	}
	\ENDFOR \label{line:lineme2}
	 \ENDFOR
	 \STATE{Let:
	 \begin{align*}
	 \hat{\bm{\uptheta}}_{k+1}=\hat{\bm{\uptheta}}_k^{(I_{m,i})}{\text{ with $m=s_k$ and $i=n_k(s_k)$}}.
	 \end{align*}
	 }
	\STATE{set $k=k+1$}
\ENDWHILE

\end{algorithmic}

\end{algorithm}

Using the notation in (\ref{eq:alonesong}), the idea behind GCSA is as follows: given the current parameter estimate, $\hat{\bm{\uptheta}}_k$, a decision is first made on which of its coordinates to update by selecting the coordinates in one of the sets $\mathcal{S}_1,\dots,\mathcal{S}_d$ according to a random variable. The number of blocks, number of updates within each block, and the subvector to update within each block are also random variables. After $\hat{\bm{\uptheta}}_{k+1}$ has been obtained, the process is repeated until a stopping criterion is reached. The process needed to update $\hat{\bm{\uptheta}}_k$ to $\hat{\bm{\uptheta}}_{k+1}$ will constitute an {\it{iteration}} of GCSA. Figure \ref{eq:rextickles} gives a representation of the components of an iteration of GCSA and Algorithm \ref{findme} provides an outline of GCSA. 

%
\begin{figure}[t]
\centering
\begin{tikzpicture}
  \definecolor{mynewcoloring}{RGB}{102, 128, 153}

\node[inner sep=0pt] (russell) at (0,0)
    {\includegraphics[width=12mm]{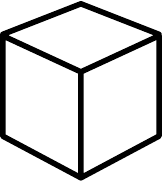}};
    \draw[->, ultra thick] (0.95,0)--(1.7,0);
    \node at (0,1.2) {\small{Block 1}};
    
	\begin{scope}[xshift=+75pt]
 	  \node[inner sep=0pt] (russell) at (0,0)
  	  {\includegraphics[width=12mm]{Graphics/cubewhitethin-eps-converted-to.pdf}};
	 \draw[->, dotted, ultra thick] (0.95,0)--(2.25,0);
	 \node at (0,1.2) {\small{Block 2}};
	 
		\begin{scope}[xshift=+90pt]
 		\node[inner sep=0pt] (russell) at (0,0)
  		{\includegraphics[width=12mm]{Graphics/cubewhitethin-eps-converted-to.pdf}};
		\draw[->, dotted, ultra thick] (0.95,0)--(2.25,0);
		\node at (0,1.2) {\small{Block $m$}};
		\draw[->, ultra thick] (0,-0.85)--(0,-1.85);
		\node[align=center] at (0,-2.9) {{Update the entries of $\bm\uptheta$}\\ {with indices in the set $\mathcal{S}_{j_k(m)}$}\\ {a total of $n_k(m)$ times.}\\};
		
			\begin{scope}[xshift=+90pt]
 			\node[inner sep=0pt] (russell) at (0,0)
  			{\includegraphics[width=12mm]{Graphics/cubewhitethin-eps-converted-to.pdf}};
			\node at (0,1.2) {\small{Block $s_k$}};
			
			\end{scope} 
		
   		\end{scope}  
		
	\end{scope}  

\end{tikzpicture}
\caption[The components of an iteration of GCSA]{The components of the $(k+1)$st iteration of the GCSA algorithm (Algorithm \ref{findme}). The numbers $j_k(m)$, $n_k(m)$, and $s_k$ are allowed to be random variables. Once the value of $j_k(m)$ has been obtained, the entries to update are determined by the indices in $\mathcal{S}_{j_k(m)}$. The updates performed within the $m$th block correspond to lines \ref{line:lineme}--\ref{line:lineme2} of Algorithm \ref{findme}. 
}
\label{eq:rextickles}
\end{figure}

Let us define the sequence $\tilde{a}_k^{(j)}$ in Algorithm \ref{findme} more precisely.
First, define:
\begin{align}
\label{eq:wheretheasaredefinedd}
\upvarphi_k^{(j)}\equiv\sum_{i=0}^k \chi\left\{\left(\sum_{m=1}^{s_i}\chi \{j_i(m)=j\}\right)>0\right\}-1,
\end{align}
%
%
%
%
 for each $j$. In other words, to compute $\upvarphi_k^{(j)}$ one must subtract one from the number of iterations up to and including iteration $k+1$ that use the sequence $\{a_i^{(j)}\}_{i\geq0}$.
 %
\begin{figure}[!t]
\centering
 \begin{tikzpicture}

  \definecolor{newblue}{RGB}{41, 67, 86}
 
  \definecolor{mycoloring}{RGB}{183,208,225}
   \definecolor{mynewcoloring}{RGB}{215,230,244}
     \definecolor{mygrey}{RGB}{238, 243, 246}

      
\fill[mynewcoloring] (-2.4,-2.4) rectangle (2.4,1.2);

\fill[mycoloring] (-1.2,-2.4) rectangle (1.2,-1.2);

\fill[mygrey] (-1.2,-4.8) rectangle (0,-2.4);
\fill[mygrey] (0,-3.6) rectangle (1.2,-2.4);
\fill[mygrey] (1.2,-6) rectangle (2.4,-2.4);

\draw[step=1.2cm,white,  line width=3pt] (-2.4,-6) grid (2.4,1.2);

\node[align=center] at (0,2.7) {Entries $\mathcal{S}_2$ and $\mathcal{S}_3$ are updated\\ when obtaining $\hat{\bm{\uptheta}}_{k+1}$ from $\hat{\bm{\uptheta}}_k$.};

\draw[->,ultra thick] (-0.6,1.95) -- (-0.6,1.45);
\draw[->,ultra thick] (0.6,1.95) -- (0.6,1.45);

\draw[newblue, ultra thick] (-2.63,-2.35) rectangle (2.6,-1.25);

\draw[->,ultra thick] (-3.4,-1.8) -- (-2.9,-1.8);

\node at (-3.85,-1.75) {$\tilde{a}_k^{(j)}$};

\node at (-1.8,-1.8) {$a_0^{(1)}$};
\node at (-0.6,-1.8) {$a_2^{(2)}$};
\node at (0.6,-1.8) {$a_1^{(3)}$};
\node at (1.8,-1.8) {$a_3^{(4)}$};

\node at (-0.6,-3) {${a_1^{(2)}}$};
\node at (-0.6,-4.2) {${a_0^{(2)}}$};

\node at (0.6,-3) {${a_0^{(3)}}$};

\node at (1.8,-3) {${a_2^{(4)}}$};
\node at (1.8,-4.2) {${a_1^{(4)}}$};
\node at (1.8,-5.4) {${a_0^{(4)}}$};

\node at (-1.8,-0.6) {$a_1^{(1)}$};
\node at (-1.8,0.6) {$\vdots$};

\node at (-0.6,-0.6) {$a_3^{(2)}$};
\node at (-0.6,0.6) {$\vdots$};

\node at (0.6,-0.6) {$a_2^{(3)}$};
\node at (0.6,0.6) {$\vdots$};

\node at (1.8,-0.6) {$a_4^{(4)}$};
\node at (1.8,0.6) {$\vdots$};


\begin{scope}[xshift=10pt]
\fill[mynewcoloring] (3.6,-2.4) rectangle (8.4,1.2);

\fill[mygrey] (4.8,-6) rectangle (6,-2.4);
\fill[mygrey] (6,-4.8) rectangle (7.2,-2.4);
\fill[mygrey] (7.2,-6) rectangle (8.4,-2.4);


\draw[step=1.2cm, white, line width=3pt] (3.6,-6) grid (8.4,1.2);

\draw[newblue, ultra thick] (3.37,-2.35) rectangle (8.6,-1.25);

\draw[->,ultra thick] (9.35,-1.8) -- (8.85,-1.8);

\node at (10,-1.75) {$\tilde{a}_{k+1}^{(j)}$};

\node at (4.2,-1.8) {$a_0^{(1)}$};
\node at (5.4,-1.8) {$a_3^{(2)}$};
\node at (6.6,-1.8) {$a_2^{(3)}$};
\node at (7.8,-1.8) {$a_3^{(4)}$};

\node at (5.4,-3) {${a_2^{(2)}}$};
\node at (5.4,-4.2) {${a_1^{(2)}}$};
\node at (5.4,-5.4) {${a_0^{(2)}}$};

\node at (6.6,-3) {${a_1^{(3)}}$};
\node at (6.6,-4.2) {${a_0^{(3)}}$};

\node at (7.8,-3) {${a_2^{(4)}}$};
\node at (7.8,-4.2) {${a_1^{(4)}}$};
\node at (7.8,-5.4) {${a_0^{(4)}}$};

\node at (4.2,-0.6) {$a_1^{(1)}$};
\node at (4.2,0.6) {$\vdots$};

\node at (5.4,-0.6) {$a_4^{(2)}$};
\node at (5.4,0.6) {$\vdots$};

\node at (6.6,-0.6) {$a_3^{(3)}$};
\node at (6.6,0.6) {$\vdots$};

\node at (7.8,-0.6) {$a_4^{(4)}$};
\node at (7.8,0.6) {$\vdots$};


\node[align=center] at (6,3) {Once $\hat{\bm{\uptheta}}_{k+1}$ has been obtained,\\ shift queues 2 and 3 down\\ and begin the new iteration.};

\draw[->,ultra thick] (5.4,1.95) -- (5.4,1.45);
\draw[->,ultra thick] (6.6,1.95) -- (6.6,1.45);
\end{scope}

 \end{tikzpicture}
 \caption[The process through which $\tilde{a}_k^{(j)}$ is updated]{This is an illustration of the process through which $\tilde{a}_k^{(j)}$ is updated. For $j=1,\dots, 4$, the $j$th column of each of the two ``tables'' above represents a queue containing the used and unused elements of $\{a_i^{(j)}\}_{i\geq0}$. 
 In the example above, initially  $\tilde{a}_k^{(1)}=a_0^{(1)}$, $\tilde{a}_k^{(2)}=a_2^{(2)}$, $\tilde{a}_k^{(3)}=a_1^{(3)}$, $\tilde{a}_k^{(4)}=a_3^{(4)}$. Then, the subvectors with entries in the sets $\mathcal{S}_2$ and $\mathcal{S}_3$ are updated when obtaining $\hat{\bm{\uptheta}}_{k+1}$ from $\hat{\bm{\uptheta}}_k$. Consequently,
 $\tilde{a}_{k+1}^{(1)}=a_0^{(1)}$, $\tilde{a}_{k+1}^{(2)}=a_3^{(2)}$, $\tilde{a}_{k+1}^{(3)}=a_2^{(3)}$, and $\tilde{a}_{k+1}^{(4)}=a_3^{(4)}$. 
  }
  \label{fig:answerscully}
\end{figure}
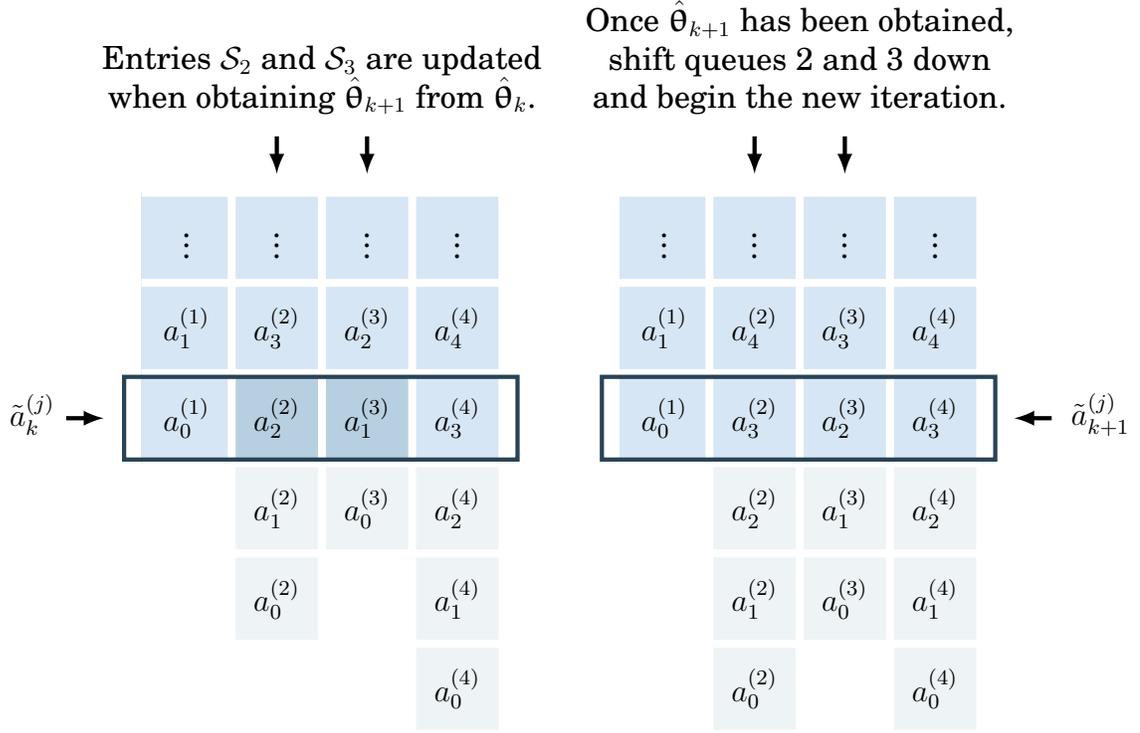
Then, for $k+1\geq 0$ the value of $\tilde{a}_{k+1}^{(j)}$ satisfies the following equation:
\begin{align}
\label{eq:saropian}
\tilde{a}_{k+1}^{(j)}&=a_0^{(j)}+\sum_{i=0}^k \chi\{\upvarphi_k^{(j)}\geq i\}(a_{i+1}^{(j)}-a_{i}^{(j)}).
\end{align}
Note that $\tilde{a}_0^{(j)}=a_0^{(j)}$ for all $j$ since $\sum_{a}^b(\cdot)=0$ if $a>b$. Note also that if the sequence $\{a_i^{(j)}\}_{i\geq 0}$ is not used during the $(k+1)$st iteration, that is if $j_k(m)\neq j$ for all $m$, then $\tilde{a}_{k+1}^{(j)}$ will be equal to $\tilde{a}_k^{(j)}$. This shows that $\tilde{a}_{k+1}^{(j)}$ and the indicator function $\chi\{j_k(m)\neq j\}$ may be {\it{dependent random variables}} (with the exception of the case where $\tilde{a}_{k}^{(j)}$ is a deterministic function of $k$, as in the case of (\ref{eq:william})).

The expression for $\hat{\bm{\uptheta}}_{k+1}=\hat{\bm{\uptheta}}_k^{(I_{s_k,n_k(s_k)})}$ in the GCSA algorithm (Algorithm \ref{findme}) is equivalent to the recursion:
\begin{align}
\hat{\bm\uptheta}_{k+1}= \hat{\bm\uptheta}_k-a_k\bm{F}_k,\text{ where }\bm{F}_k\equiv 
\left[ \sum_{z=1}^{s_k}\sum_{\ell=0}^{n_{k}(z)-1} \tilde{A}_{k}(z)\ \hat{\bm{g}}^{(j_k(z))}_k\left(\hat{\bm\uptheta}_k^{(I_{z,\ell})}\right)\right]/a_k,\label{eq:hand}
\end{align}
and where $a_k>0$. Therefore, an iteration of the GCSA algorithm can be written in the general form $\hat{\bm\uptheta}_{k+1}=\hat{\bm\uptheta}_k-a_k\bm{F}_k$ (the two algorithms defined by (\ref{eq:william}) and (\ref{eq:beerabove}) are special cases of (\ref{eq:hand})). Although (\ref{eq:hand}) resembles the general SA update in (\ref{eq:formof}), existing results on convergence of SA algorithms are not directly applicable to GCSA due to the increased complexity of $\bm{F}_k$ over $\bm{Y}_k$, the typical SA update direction.

%
%

\section{Concluding Remarks}
\label{sec:cocnremarksGCSA}

This chapter described the GCSA algorithm for stochastic optimization via SA, the  iterative algorithm for updating the parameter vector $\hat{\bm{\uptheta}}_k$ that will be the focus of this dissertation. The basic idea behind GCSA is to divide the vector of parameters into $d$ (possibly overlapping) subvectors. Then, at each time a decision is made regarding which subvector to update (this decision can be made according to a random variable or may be governed by a deterministic selection pattern). The subvector updates are performed using SA-based update directions (in general, it is impossible to guarantee that the update direction is a descent direction due to the presence of the noise and bias terms described in Section \ref{sec:dummieyouask}). In the GCSA algorithm, each subvector has an associated gain sequence that is a deterministic function of $k$ and is used to scale the update direction. Although the gain sequences for the different subvectors may be different, the convergence theory in the following chapter will require that all gain sequences converge to zero at the same rate.

%% file: chapter2.tex

\chapter{Convergence of GCSA}
\label{sec:cyclicseesaw}

This chapter derives conditions for the convergence w.p.1 of the GCSA iterates. Section \ref{sec:rewrite} first rewrites the GCSA algorithm as a stochastic time-dependent process. Section \ref{sec:12or15miles} then provides a detailed analysis of the process from Section \ref{sec:rewrite}. Section \ref{sec:convo} states the main theorem for this chapter (Theorem \ref{thm:hoeshoo}). Section \ref{sec:specialitay} states two corollaries regarding special cases of GCSA. Section \ref{sec:discussconvergence} discusses the validity of the assumptions of Theorem \ref{thm:hoeshoo}. Lastly, Section \ref{sec:concremarksconvergence} contains concluding remarks.

\section{Rewriting the GCSA Recursion}
\label{sec:rewrite}
The theory behind the convergence of GCSA relies on rewriting the algorithm as a stochastic time-dependent process. Loosely speaking, this section first rewrites
a realization of GCSA
 as a multi-dimensional, {{continuous}}, time-dependent function.
This continuous time-dependent function is then shown to satisfy a time-dependent version of the GCSA recursion in (\ref{eq:hand}), a fact which
will play a crucial role in the development of the convergence theory for GCSA in Sections \ref{sec:12or15miles} and \ref{sec:convo}.  The aforementioned time-dependent {\it{continuous function}} and a related time-dependent {\it{step function}} are constructed next.
%
%

We begin by defining the following time-dependent {\it{step function}}:
\begin{align}
\bar{\bm{Z}}_{0}(t)\equiv \begin{cases}  \hat{\bm{\uptheta}}_k &\mbox{if } t\in[t_{k},t_{k+1}){\text{ for }} k\geq 0,  \\ 
{\bar{\bm{Z}}_{0}}(t_0) &\mbox{if } t\leq t_0,
 \end{cases} 
 \label{eq:daviddaas}
\end{align}
 where $t_{k}\equiv \sum_{i=0}^{k-1}a_{i}$ for $k\geq 0$ and where $a_k>0$ is a deterministic sequence satisfying $a_k\rightarrow 0$ (Section \ref{sec:12or15miles} imposes other assumptions on $a_k$ relating it to the gain sequences $a_k^{(j)}$ from Algorithm \ref{findme}).
  Now, the time-dependent {\it{continuous function}} discussed in this section's opening paragraph
  will be defined as:
\begin{align}
\label{eq:anlogouskartmannedg}
\bm{Z}_0(t)&\equiv \frac{(t_{k+1}-t)}{a_{k}}\bar{\bm{Z}}_0(t_k)+\frac{(t-t_k)}{a_{k}}\bar{\bm{Z}}_0(t_{k+1})
\end{align}
 for $t\in[t_k,t_{k+1}]$, and $\bm{Z}_0(t)\equiv\bm{Z}_0(t_0)$ for $t\leq t_0$. The function $\bm{Z}_0(t)$ is then simply an interpolation of $\bar{\bm{Z}}_0(t)$ at the interpolation points $\{t_k\}_{k\geq 0}$ (see Figure \ref{fig:zeevszeebar}). The subindex ``0'' is used in anticipation of modified versions of $\bar{\bm{Z}}_0(t)$ and $\bm{Z}_0(t)$ to be introduced in Section \ref{sec:12or15miles}. Next we rewrite the GCSA recursion in (\ref{eq:hand}) by decomposing $\bm{F}_k$ into several terms, after which the resulting expression is used to construct a time-dependent analogue to (\ref{eq:hand}) involving $\bm{Z}_0(t)$.

By adding and subtracting select terms (effectively adding zero) to the definition of $\bm{F}_k$ in (\ref{eq:hand}) and using the fact that $\hat{\bm{g}}_k(\bm\uptheta)= {\bm{g}}(\bm\uptheta)+{\bm{\upbeta}}_k(\bm\uptheta)+{\bm{\upxi}}_k(\bm\uptheta)$ (see Section \ref{sec:dummieyouask}), the recursion in (\ref{eq:hand}) can first be rewritten as:
\begin{align}
\label{eq:paved}
\hat{\bm\uptheta}_{k+1}=\hat{\bm\uptheta}_k&-\sum_{m=1}^{s_k}\sum_{i=0}^{n_k(m)-1} \tilde{A}_{k}(m)\left[\bm\upbeta^{(j_k(m))}_k\left(\hat{\bm\uptheta}_k^{(I_{m,i})}\right)+\bm\upxi^{(j_k(m))}_k\left(\hat{\bm\uptheta}_k^{(I_{m,i})}\right)\right]\notag\\
&-\sum_{m=1}^{s_k}\sum_{i=0}^{n_k(m)-1} \tilde{A}_{k}(m)\left[ {\bm{g}}^{(j_k(m))}\left(\hat{\bm\uptheta}_k^{\left(I_{m,i}\right)}\right)- {\bm{g}}^{(j_k(m))}(\hat{\bm\uptheta}_k)\right]\notag\\
&-\sum_{m=1}^{s_k}\sum_{i=0}^{n_k(m)-1} \tilde{A}_{k}(m)\ {\bm{g}}^{(j_k(m))}(\hat{\bm\uptheta}_k).
\end{align}
Next,
%
\begin{figure}[!t]
\centering
\begin{tikzpicture}
  
  \definecolor{mynewcoloring}{RGB}{102, 128, 153}
  \definecolor{mygray2}{gray}{0.8}
    
    \draw[mygray2, ultra thick] (-.5,0)  -- (5.8,0);
    \draw[mygray2, ultra thick] (0,-.5) -- (0,4.5);
    \node at (0.01,4.8) {$\uptheta$};
    \node at (6.95,0) {$t$};
     \node at (-0.70,0.3) {$\uptheta={0}$};
    
    \draw[->,ultra thick] (5.75,0) -- (6.25,0) -- (6.75,0) ;
    
     \node at (.75,-.3) {$t_0=0$};
     \node at (2.5,-.3) {$t_1$};
     \node at (4,-.3) {$t_2$};
     \node at (4.8,-.3) {$t_3$};
     \node at (5.3,-.3) {$t_4$};
     
     \draw[help lines, mynewcoloring] (2.5,0) -- (2.5,4);

     \draw[ultra thick,dotted, mynewcoloring] (-.5,1)  -- (0,1); 
     \draw[ultra thick, dotted, mynewcoloring] (0,1)  -- (2.5,4);
     \draw[ultra thick, dotted, mynewcoloring] (2.5,4)  -- (4,3);
     \draw[ultra thick, dotted, mynewcoloring] (4,3)  -- (4.8,2);
     \draw[ultra thick, dotted, mynewcoloring] (4.8,2)  -- (5.3,4);

     \draw[ultra thick, mynewcoloring] (-.25,1)  -- (2.5,1); %
     \draw[ultra thick, mynewcoloring] (2.5,4) -- (4,4); %
     \draw[ultra thick, mynewcoloring] (4,3) -- (4.8,3); %
      \draw[ultra thick, mynewcoloring] (4.8,2) -- (5.3,2); %
      \draw[ultra thick, mynewcoloring] (5.3,4)  -- (5.75,4); %

     \draw[help lines, mynewcoloring] (4,0) -- (4,3);     
     \draw[help lines, mynewcoloring] (4.8,0) -- (4.8,2);
     \draw[help lines, mynewcoloring] (5.3,0) -- (5.3,4);   
      
      \node at (3.25,4.4) {$\bar{Z}_0(t)$};
      \node[rotate=51] at (1,2.8) {$Z_0(t)$};
   
   \draw[->] (-.6,1.6)--(-.1,1.1);
   \node at (-0.9,1.9) {$\hat{{\uptheta}}_0$};
   
   \draw[->] (1.9,4.6)--(2.4,4.1);
   \node at (1.6,4.9) {$\hat{{\uptheta}}_1$};
   
   \draw[->] (4.7,4.6)--(5.2,4.1);
   \node at (4.4,4.9) {$\hat{{\uptheta}}_4$};
    
    \node at (0,1) {\textbullet};
    \node at (2.5,4) {\textbullet};
    \node at (4,3) {\textbullet};
    \node at (4.8,2) {\textbullet};
    \node at (5.3,4) {\textbullet};

  \end{tikzpicture}
\caption[The difference between $\bar{\bm{Z}}_0(t)$ and $\bm{Z}_0(t)$.]{The difference between $\bar{Z}_0(t)$ and $Z_0(t)$, special cases of $\bar{\bm{Z}}_0(t)$ and $\bm{Z}_0(t)$ where both functions are real-valued. In general, $\bar{\bm{Z}}_0(t_k)=\bm{Z}_0(t_k)=\hat{\bm{\uptheta}}_k$.}
\label{fig:zeevszeebar}
\end{figure}
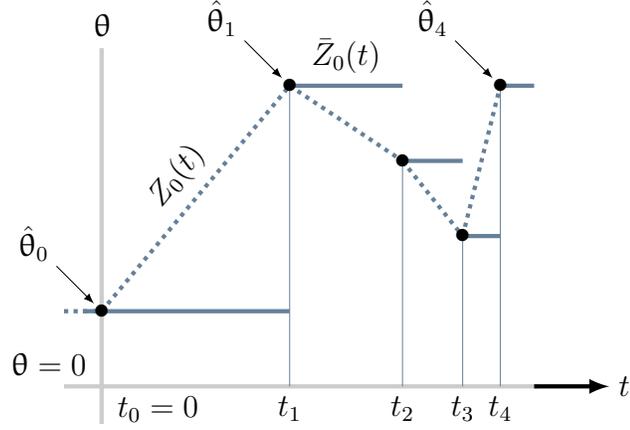
for $j=1,\dots, d$ define the random variables:
\begin{align}
\label{eq:tinkelpan}
x_k(j)\equiv \left(\frac{\tilde{a}_{k}^{(j)} }{a_{k}}\right)\sum_{m=1}^{s_k}\chi\{j_k(m)=j\}n_k(m).
\end{align}
 Assuming $\upmu_k(j)\equiv E[x_k(j)]$ exists 
and is finite and that $\upmu(j)\equiv \lim_{k\rightarrow \infty} \upmu_k(j)$ exists and is finite (these assumptions will be discussed in Section \ref{sec:discussconvergence}), define:
\begin{align}
\bm{h}_k(\bm\uptheta)\equiv \sum_{j=1}^d \upmu_k(j)\bm{g}^{(j)}(\bm\uptheta), \ \ \ \bm{h}(\bm\uptheta)\equiv \sum_{j=1}^d \upmu(j)\bm{g}^{(j)}(\bm\uptheta).\label{eq:unrealantartica}
\end{align}
By adding and subtracting $a_k\bm{h}(\hat{\bm{\uptheta}}_k)$ from the right-hand side of (\ref{eq:paved}):
\begin{align}
\hat{\bm\uptheta}_{k+1}=\hat{\bm\uptheta}_k&-a_k\bm{h}(\hat{\bm{\uptheta}}_k)-a_k\Bigg[\sum_{m=1}^{s_k}\sum_{i=0}^{n_k(m)-1} \left(\frac{\tilde{A}_{k}(m) }{a_{k}}\right){\bm{g}}^{(j_k(m))}(\hat{\bm\uptheta}_k)-\bm{h}(\hat{\bm{\uptheta}}_k)\Bigg]\notag\\
&-a_k\sum_{m=1}^{s_k}\sum_{i=0}^{n_k(m)-1} \left(\frac{\tilde{A}_{k}(m) }{a_{k}}\right)\left[\bm\upbeta_k^{(j_k(m))}\left(\hat{\bm\uptheta}_k^{(I_{m,i})}\right)+\bm\upxi_k^{(j_k(m))}\left(\hat{\bm\uptheta}_k^{(I_{m,i})}\right)\right]\notag\\
&-a_{k}\sum_{m=1}^{s_k}\sum_{i=0}^{n_k(m)-1}\left(\frac{\tilde{A}_{k}(m)}{a_{k}}\right)\left[ {\bm{g}}^{(j_k(m))}\left(\hat{\bm\uptheta}_k^{\left(I_{m,i}\right)}\right)- {\bm{g}}^{(j_k(m))}(\hat{\bm\uptheta}_k)\right].
\label{eq:pandabear}
\end{align}
Thus, when $\upmu_k(j)$ and $\upmu(j)$ exist and are finite, (\ref{eq:pandabear}) is equivalent to the GCSA recursion from (\ref{eq:hand}). Next we derive a time-dependent version of (\ref{eq:pandabear}).

We begin by defining the following terms:
\begin{subequations}
\begin{align}
\bar{\bm{B}}_0(t_k)\equiv&\ \sum_{r=0}^{k-1} a_r\sum_{m=1}^{s_r}\sum_{i=0}^{n_r(m)-1} \left(\frac{\tilde{A}_{r}(m) }{a_{r}}\right)\bm\upbeta_r^{(j_r(m))}\left(\hat{\bm\uptheta}_r^{(I_{m,i})}\right),\label{eq:pulcritud0}\\
\bar{\bm{M}}_0(t_k)\equiv&\ \sum_{r=0}^{k-1} a_r\sum_{m=1}^{s_r}\sum_{i=0}^{n_r(m)-1} \left(\frac{\tilde{A}_{r}(m) }{a_{r}}\right)\bm\upxi_r^{(j_r(m))}\left(\hat{\bm\uptheta}_r^{(I_{m,i})}\right),\label{eq:pulcritud1}\\
{\bar{\bm{N}}}_0(t_k)\equiv&\ \sum_{r=0}^{k-1} a_r\Bigg[\sum_{m=1}^{s_r}\sum_{i=0}^{n_r(m)-1} \left(\frac{\tilde{A}_{r}(m) }{a_{r}}\right){\bm{g}}^{(j_r(m))}(\hat{\bm\uptheta}_r)-h(\hat{\bm{\uptheta}}_r)\Bigg],\label{eq:pulcritud2}\\
\bar{\bm{W}}_0(t_k)\equiv&\ \sum_{r=0}^{k-1} a_{r}\sum_{m=1}^{s_r}\sum_{i=0}^{n_r(m)-1}\left(\frac{\tilde{A}_{r}(m)}{a_{r}}\right)\left[ {\bm{g}}^{(j_r(m))}\left(\hat{\bm\uptheta}_r^{\left(I_{m,i}\right)}\right)- {\bm{g}}^{(j_r(m))}(\hat{\bm\uptheta}_r)\right].\label{eq:pulcritud}
\end{align}
\end{subequations}
Then, using the notation from  (\ref{eq:pulcritud0}--d), an equivalent way to write (\ref{eq:pandabear}) is:
\begin{align}
\bar{\bm{Z}}_0(t_{k+1})=&\ \bar{\bm{Z}}_0(t_{0})-\bar{\bm{B}}_0(t_{k+1})-\bar{\bm{M}}_0(t_{k+1})\notag\\
&-\bar{\bm{N}}_0(t_{k+1})-\bar{\bm{W}}_0(t_{k+1})-\sum_{i=0}^ka_i\bm{h}(\bar{\bm{Z}}_0(t_i)).\label{eq:excited}
\end{align}
The expression in (\ref{eq:excited}) lays the foundation for constructing the continuous and time-dependent version of the GCSA recursion (\ref{eq:pandabear}). Specifically,
in a manner analogous to (\ref{eq:daviddaas}) let $\bar{\bm{B}}_0(t)$, $\bar{\bm{M}}_0(t)$, $\bar{\bm{N}}_0(t)$, and $\bar{\bm{W}}_0(t)$ be step functions on the interval $[t_k,t_{k+1})$ and, in a manner analogous to (\ref{eq:anlogouskartmannedg}),
let $\bm{B}_0(t)$, $\bm{M}_0(t)$, $\bm{N}_0(t)$, and $\bm{W}_0(t)$ denote their respective interpolation functions.
Then, using $t_0=0$:
\begin{align}
\label{eq:sandsss}
{\bm{Z}}_0(t)=&\ {\bm{Z}}_0(0)-{\bm{B}}_0(t)-{\bm{M}}_0(t)-{\bm{N}}_0(t)-{\bm{W}}_0(t)
-\int_{0}^t{\bm{h}}(\bar{\bm{Z}}_0(s))\ ds.
\end{align}
Alternatively, replacing the integrand (\ref{eq:sandsss}) with $\bm{h}(\bm{Z}_0(s))$ we obtain:
\begin{align}
\label{eq:levelred222}
{\bm{Z}}_0(t)=&\ {\bm{Z}}_0(0)-{\bm{B}}_0(t)-{\bm{M}}_0(t)-{\bm{N}}_0(t)-{\bm{W}}_0(t)
-\int_{0}^t{\bm{h}}({\bm{Z}}_0(s))\ ds +{\bm{\upzeta}}_{0}(t),
\end{align}
where $\bm\upzeta_0(t)$ is a vector representing the error introduced by using ${\bm{h}}({\bm{Z}}_0(s))$ in place of ${\bm{h}}(\bar{\bm{Z}}_0(s))$. Equation (\ref{eq:levelred222}) represents a continuous and time-dependent version of the GCSA recursion.
The following section proves a few lemmas regarding the GCSA algorithm when treated as the stochastic time-dependent, continuous process in (\ref{eq:levelred222}). 

\section{Analyzing the GCSA Recursion}
\label{sec:12or15miles}

This section proves several results regarding a set of shifted versions of (\ref{eq:sandsss}) and (\ref{eq:levelred222}). Specifically, 
define the {\it{shift functions}}:
\begin{align}
\label{eq:gochoire}
 \bar{\bm{Z}}_k(t)\equiv{\bar{\bm{Z}}}_0(t_k+t), \ \ \ \bm{Z}_k(t)\equiv{\bm{Z}}_0(t_k+t), 
 \end{align}
 and the {\it{shift-increment functions}}:
 \begin{subequations}
 \begin{align} 
\bm{B}_k(t)\equiv{\bm{B}}_0(t_k+t)-{\bm{B}}_0(t_k),\ \ \ \bm{M}_k(t)\equiv{\bm{M}}_0(t_k+t)-{\bm{M}}_0(t_k),\label{eq:easychair0}\\
\bm{N}_k(t)\equiv{\bm{N}}_0(t_k+t)-{\bm{N}}_0(t_k),\ \ \ \bm{W}_k(t)\equiv{\bm{W}}_0(t_k+t)-{\bm{W}}_0(t_k).\label{eq:easychair}
\end{align}
\end{subequations}
%
\begin{figure}[!t]
\centering
 \begin{tikzpicture}
  
  \definecolor{mynewcoloring}{RGB}{102, 128, 153}
  \definecolor{mygray2}{gray}{0.8}
  \definecolor{mygreen}{RGB}{80, 135, 96}

    \draw[mygray2, ultra thick] (-0.5,0)  -- (4,0);
    \draw[mygray2, ultra thick] (0,-.5) -- (0,4.5);
    
    
     \draw[ultra thick, mynewcoloring] (-.5,1)  -- (0,1); 
     \draw[ultra thick, mynewcoloring] (0,1)  -- (2,4);
      \draw[ultra thick, mynewcoloring] (2,4)  -- (3,3);
      \draw[ultra thick, mynewcoloring] (3,3)  -- (3.5,3.5);
      
      \draw[->,ultra thick] (3.5,0) -- (4.5,0) ;
            
     \draw[help lines, mynewcoloring] (2,0) -- (2,4);
     \draw[help lines, mynewcoloring] (3,0) -- (3,3);
     \draw[help lines, mynewcoloring] (3,0) -- (3,3);
   
     \node at (0.01,4.8) {$\uptheta$};
     \node at (4.75,0) {$t$};
         \node at (-0.70,0.3) {$\uptheta={0}$};
     \node at (.75,-.3) {$t_0=0$};
     \node at (2,-.3) {$t_1$};
     \node at (3,-.3) {$t_2$};
     \node at (-2.75,1.5) {${Z}_1(t)$};
     \node at (4.25,3.65) {$Z_0(t)$}; 
      \node at (-2.75,0) {${B}_1(t)$};  
      \node at (4.25,1.9) {$B_0(t)$};

     \draw[ultra thick,dotted, mynewcoloring] (-2.5,1)  -- (-2,1); 
     \draw[ultra thick, dotted,mynewcoloring] (-2,1)  -- (0,4);
     \draw[ultra thick, dotted, mynewcoloring] (0,4)  -- (1,3);     
     \draw[ultra thick, dotted, mynewcoloring] (1,3)  -- (1.5,3.4);

    \draw[->] (-0.6,4.6)--(-0.1,4.1);
    \node at (-0.9,4.9) {$\hat{{\uptheta}}_1$};
    
    \draw[->] (1.4,4.6)--(1.9,4.1);
    \node at (1.1,4.9) {$\hat{{\uptheta}}_1$};
     
     
      \draw[ultra thick, mygreen] (-.5,0)  -- (0,0); 
      \draw[ultra thick, mygreen] (0,0)  -- (2,0.5); 
      \draw[ultra thick, mygreen] (2,0.5)  -- (3,2.5); 
      \draw[ultra thick, mygreen] (3,2.5)  -- (3.5,2); 
      
      \draw[ultra thick,dotted,mygreen] (-2.5,-0.5)  -- (-2,-0.5);
      \draw[ultra thick,dotted,mygreen] (-2,-0.5)  -- (0,0);
      \draw[ultra thick,dotted,mygreen] (0,0)  -- (1,2);
      \draw[ultra thick,dotted,mygreen] (1,2)  -- (1.5,1.5);
        
     \node at (0,1) {\textbullet};
     \node at (2,4) {\textbullet}; 
     \node at (3,3) {\textbullet};
          
     \node at (-2,1) {\textbullet};
     \node at (0,4) {\textbullet}; 
     \node at (1,3) {\textbullet};

      \node at (2,0.5) {\textbullet};
      \node at (3,2.5) {\textbullet};
      
      \node at (-2,-0.5) {\textbullet};
      \node at (0,0) {\textbullet};
      \node at (1,2) {\textbullet};
      
  \end{tikzpicture}
  \caption[Comparing $\bm{Z}_0(t)$ to $\bm{Z}_1(t)$ and $\bm{B}_0(t)$ to $\bm{B}_1(t)$.]{Comparing $\bm{Z}_0(t)$ to $\bm{Z}_1(t)$ and $\bm{B}_0(t)$ to $\bm{B}_1(t)$ for the case where all functions have real-valued output. The variables in (\ref{eq:easychair0},b) are obtained by shifting a function up/down and to the left so that $\bm{B}_k(t_0)=\bm{M}_k(t_0)=\bm{N}_k(t_0)=\bm{W}_k(t_0)=\bm{0}$. In contrast, the functions in (\ref{eq:gochoire}) are obtained solely via left-shifts so that $ \bar{\bm{Z}}_k(t_0)=\bm{Z}_k(t_0)=\hat{\bm{\uptheta}}_k$.}
  \label{fig:ufosite}
\end{figure}
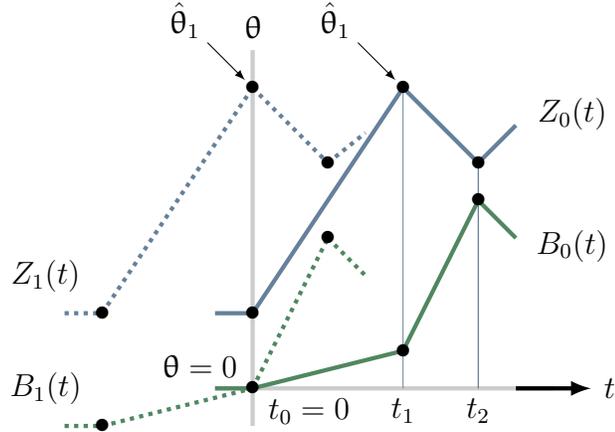 
(Figure \ref{fig:ufosite} illustrates the difference between a shift function and a shift-increment function.)
Using the notation from (\ref{eq:gochoire}) and (\ref{eq:easychair0},b), ${\bm{Z}}_k(t)$ satisfies:
\begin{align}
\label{eq:bondok}
{\bm{Z}}_k(t)=&\ {\bm{Z}}_k(0)-{\bm{B}}_k(t)-{\bm{M}}_k(t)-{\bm{N}}_k(t)-{\bm{W}}_k(t)
-\int_{0}^t{\bm{h}}(\bar{\bm{Z}}_k(s))\ ds.
\end{align}
Alternatively, replacing the integrand (\ref{eq:bondok}) with $\bm{h}(\bm{Z}_k(s))$:
\begin{align}
\label{eq:levelred}
{\bm{Z}}_k(t)=&\ {\bm{Z}}_k(0)-{\bm{B}}_k(t)-{\bm{M}}_k(t)-{\bm{N}}_k(t)-{\bm{W}}_k(t)
-\int_{0}^t{\bm{h}}({\bm{Z}}_k(s))\ ds +{\bm{\upzeta}}_{k}(t),
\end{align}
for a vector $\bm\upzeta_k(t)$ representing the error introduced by using ${\bm{h}}({\bm{Z}}_k(s))$ in place of ${\bm{h}}(\bar{\bm{Z}}_k(s))$.  Equation (\ref{eq:levelred}) is similar to equation (2.3.3) in Kushner and Clark (1978)\nocite{kushnclark1978}, though (\ref{eq:levelred}) is significantly more complex.
This section proves several lemmas pertaining to terms appearing in (\ref{eq:bondok}) and (\ref{eq:levelred}). These lemmas will then be used in Section \ref{sec:convo} to prove this chapter's main convergence theorem. Pages \pageref{chap:fun}-- \pageref{def:lastpage} contain a compilation of frequently used notation (including a section with GCSA-specific notation) which may be used as a quick reference.

The results in this section make use of Kolmogorov's Extension Theorem (\O ksendal 2003, Theorem 2.1.5)\nocite{oksendal2003}, which guarantees that any realization of GCSA can be seen as a random variable on a probability space $(\Omega,\mathcal{F},P)$, where $\Omega$ is the sample space, $\mathcal{F}$ is the $\upsigma$-field and $P$ is a probability measure. In other words, each $\upomega\in \Omega$ is assumed to determine (via some unknown function) an entire realization of GCSA. Thus, for example, the random sets $\{\tilde{A}_k^{(j_k(m))}\}_{k}$, $\{\hat{\bm\uptheta}_k^{(I_{m,i})}\}_{k,m,i}$, $\{\bm\upxi_k^{(j_k(m))}(\hat{\bm\uptheta}_k^{(I_{m,i})})\}_{k,m,i}$, and $\{\bm\upbeta_k^{(j_k(m))}(\hat{\bm\uptheta}_k^{(I_{m,i})})\}_{k,m,i}$ are fully determined by $\upomega$. 
  Next, we introduce a set of assumptions to be used throughout this section. Throughout these assumptions let $\upomega$ be as defined above.

\begin{DESCRIPTION}
\item[A0] Let $a_{k}>0$ and $a_k^{(j)}>0$ for all $k$ and $j$. Let $\sum_{k=0}^{\infty}a_{k}=\infty$ and $\sum_{k=0}^{\infty}a_{k}^2<\infty$. Next, let $E[\tilde{a}_k^{(j)}/a_k]\rightarrow r_j$ with $0\leq r_j<\infty$. Additionally, let there exist a set $\Omega_0\subset\Omega$ with $P(\Omega_0)=1$ such that $\tilde{a}_k^{(j)}/a_k\rightarrow r_j$ for all $\upomega\in\Omega_0$.  
\item [A1]  Let $s_k$ and $n_k(m)$ be bounded uniformly over $k$, $m$, and $\upomega$.  \label{cond:conditiona1}
\item[A2] Assume $\upmu_k(j)= E[x_k(j)]$ exists for all $k$ and $j$ and is finite (i.e., $\upmu_k(j)<\infty$) and that $\upmu(j)= \lim_{k\rightarrow \infty} \upmu_k(j)$ exists and is finite  ($x_k(j)$, $\upmu_k(j)$, and $\upmu(j)$ were first introduced in pp. \pageref{eq:tinkelpan}--\pageref{eq:unrealantartica}).
Additionally, define:
\begin{align*}
C_k(j)\equiv \sum_{m=1}^{s_k}\chi\{j_k(m)=j\}n_k(m),\ \ \ S_k\equiv \sum_{j=1}^dE\left[\frac{\tilde{a}_k^{(j)}}{a_k}\right](C_k(j)-E[C_k(j)]),
\end{align*}
and let $E[S_kS_\ell]=0$ for $k\neq \ell$. Furthermore, let $\tilde{a}_k^{(j)}$ and $C_k(j)$ be independent.
\item[A3] Let ${\bm{g}}(\bm\uptheta)$ be a continuous  function (i.e., $L(\bm\uptheta)$ is continuously differentiable).
\item[A4] For $k\geq0$, $m\leq s_k$, $i\leq n_k(m)$, let there exist a set $\Omega_1\subset\Omega$ with $P(\Omega_1)=1$ and a scalar $0<R_1(\upomega)<\infty$ such that the set $\{\hat{\bm\uptheta}_k^{(I_{m,i})}\}_{k,m,i}$ is contained within a $p$-dimensional ball of radius $R_1(\upomega)$ centered at the origin for all $\upomega\in\Omega_1$.  \label{cond:conditiona4}
\item[A5] For $k\geq0$, $m\leq s_k$, $i\leq n_k(m)$, let there exist a set $\Omega_2\subset\Omega$ with $P(\Omega_2)=1$ and a scalar $0<R_2(\upomega)<\infty$ such that the set $\{\bm\upbeta_k^{(j_k(m))}(\hat{\bm\uptheta}_k^{(I_{m,i})})\}_{k,m,i}$ is contained within a $p$-dimensional ball of radius $R_2(\upomega)$ centered at the origin for all $\upomega\in\Omega_2$.
Additionally, let $\bm\upbeta_k^{(j_k(m))}(\hat{\bm\uptheta}_k^{(I_{m,i})})$ converge to ${\bm{0}}$ w.p.1 as $k\rightarrow\infty$.
\item[A6] Define $\bm{D}_r\equiv \sum_{m=1}^{s_r}\sum_{i=0}^{n_r(m)-1} \left(\frac{\tilde{A}_{r}(m) }{a_{r}}\right)\bm\upxi^{(j_r(m))}_r\left(\hat{\bm\uptheta}_r^{(I_{m,i})}\right)$.
Assume for all $\upvarepsilon>0$ we have $ \lim_{k\rightarrow \infty}P\left(\sup_{m\geq k}\Big\|\sum_{r=k}^{m}a_{r}\bm{D}_{r}\Big\|\geq \upvarepsilon\right)=0$.
\item [A7] Let $\bm{\uptheta^{\ast}}$ be a locally asymptotically stable (in the sense of Lyapunov) solution of the differential equation:
\begin{align}
\label{eq:ode}
\dot{\bm{Z}}(t)\equiv\frac{d {\bm{Z}}(t)}{d t}=- \sum_{j=1}^d\upmu(j)\bm{g}^{(j)}(\bm{Z}(t))
\end{align}
with domain of attraction $DA(\bm{\uptheta^{\ast}})$. 
\item[A8] There is a compact subset $A$ of $DA(\bm{\uptheta^{\ast}})$ such that $\bm{{\hat{\uptheta}}}_{k} \in A$ infinitely often (here compactness is defined using the Euclidean topology).
\end{DESCRIPTION}

The validity of A0--A8 will be discussed 
in Sections \ref{sec:specialitay} and \ref{sec:discussconvergence}. We note that while $a_k$ is not used by the GCSA algorithm, it serves as a representation of the rate at which ${a}_k^{(j)}$ decreases. Following are several lemmas regarding the set of functions $\{{\bm{Z}}_{k}(t)\}_{k\geq0}$.  Throughout all proofs
 any unspecified probabilistic arguments are meant to hold w.p.1.

\begin{lemma}
\label{claim:brey}
Assume A0 and A1 hold. Using the definition of $x_k(j)$ in (\ref{eq:tinkelpan}) let: 
\begin{align}
\label{eq:apoint}
X_k\equiv \sum_{j=1}^dx_k(j).
\end{align}
 Then, there exists a constant $0<R_3(\upomega)<\infty$ such that $|X_k|<R_3(\upomega)$ for all $\upomega\in \Omega_3$ and all $k\geq 0$.
\end{lemma}

\begin{proof}
 By A1 it follows that $s_k$ and $n_k(m)$ are bounded (uniformly in $k$, $m$, and $\upomega$). Next, A0 (using the fact that $\tilde{a}_k^{(j)}/a_k\rightarrow r_j$ w.p.1 where $0\leq r_j<\infty$) implies that for $\upomega\in \Omega_0$ the sequence $\{\tilde{a}_k^{(j)}/a_k\}_{k\geq 0}$ is bounded in magnitude by a constant that depends on $\upomega$ but is independent of $k$. Combining this with the definition of $x_k(j)$ given in (\ref{eq:tinkelpan}) and the fact that $d<\infty$ yields the desired result.
\end{proof}

\begin{lemma}
\label{claim:trompeta} Assume condition A4 holds. Then, w.p.1 the set $\{\bm{Z}_k(t)\}_{k\geq0}$ is a set of continuous functions that are bounded over $t$ uniformly in $k$.
\end{lemma}

\begin{proof}
A direct consequence of condition A4 is that $\{\hat{\bm\uptheta}_k\}_{k\geq 0}$ must be a bounded set (although the magnitude of the bound may be a function of $\upomega$). Since $\bm{Z}_0(t)$ was obtained by interpolating the vectors $\{\hat{\bm\uptheta}_k\}_{k\geq 0}$ then $\bm{Z}_0(t)$ must be bounded uniformly over $t\in\mathbb{R}$ w.p.1. Since $\bm{Z}_k(t)$ is obtained by performing a sequence of left-shifts on $\bm{Z}_0(t)$, all functions in $\{\bm{Z}_k(t)\}_{k\geq 0}$ must be bounded uniformly in $k$ and $t$ w.p.1. Additionally, these functions are clearly continuous by construction (the functions are linear interpolations of bounded vectors). 
\end{proof}

The next lemma relies on the concept of the ``equicontinuity'' of set of functions. This concept  and two related concepts are defined next.

\begin{definition}
\label{def:newspaperchaingg}
A set $\{\bm{\uprho}_k(t)\}_{k\geq0}$ of functions from $\mathbb{R}$ to $\mathbb{R}^p$ is said to be {\it{equicontinuous}} at $t$ if for all $k$ and $\upepsilon>0$ there exists a $\updelta(t,\upepsilon)$ such that $\|\bm{\uprho}_k(t)-\bm{\uprho}_k(s)\|<\upepsilon$ if $|t-s|<\updelta(t,\upepsilon)$.  If the set of functions is equicontinuous at every $t$ it is said to be {\it{point-wise equicontinuous}}. The set of functions is said to be {\it{uniformly equicontinuous}} over a set $\mathcal{S}\subset \mathbb{R}$ if for all $k\geq 0$, $t\in \mathcal{S}$, and $\upepsilon>0$ there exists a $\updelta(\upepsilon)$ such that $\|\bm{\uprho}_k(t)-\bm{\uprho}_k(s)\|<\upepsilon$ whenever $|t-s|<\updelta(\upepsilon)$ and $s\in \mathcal{S}$ (note that $\updelta(\upepsilon)$ depends neither on $k$ nor on $t$).
\end{definition}

\begin{lemma}
\label{claim:loan}
 Assume conditions A0, A1, and A5 hold and for any finite $T>0$ define $I_{T}\equiv[-T,T]$. Then, the following statements hold for all $\upomega$ in a set w.p.1 and any finite $T>0$. The set $\{{\bm{B}}_{k}(t)\}_{k\geq0}$ is a set of functions, each of which is uniformly continuous over $ I_T$ (note that at this point we are not claiming equicontinuity). Additionally, the functions in $\{{\bm{B}}_{k}(t)\}_{k\geq0}$ are bounded over $t\in I_T$ uniformly in $k$. Furthermore, $\bm{B}_{k}(t)\rightarrow \bm{0}$ uniformly over $t\in I_T$ as $k\rightarrow \infty$. The former statements imply $\{{\bm{B}}_{k}(t)\}_{k\geq0}$ is uniformly equicontinuous over $ I_T$.
\end{lemma}

\begin{proof}
First, we prove that $\{\bm{B}_k(t)\}_{k\geq 0}$ is bounded over $t\in I_T$ uniformly in $k$. First and foremost, using the fact that all gain sequences are nonnegative (condition A0), the definition of $\tilde{A}_{r}(m)$ given in Algorithm \ref{findme}, and the definition of $\bar{\bm{B}}_0(t_k)$ in (\ref{eq:pulcritud0}) along with the triangle inequality gives rise to the following inequality:
\begin{align}
\label{eq:mussel}
\|\bar{\bm{B}}_0(t_k)\|\leq \sum_{r=0}^{k-1} a_r\sum_{m=1}^{s_r}\sum_{i=0}^{n_r(m)-1} \left(\frac{\tilde{A}_{r}(m) }{a_{r}}\right)\left\|\bm\upbeta_r^{(j_r(m))}\left(\hat{\bm\uptheta}_r^{(I_{m,i})}\right)\right\|.
\end{align}
Under A5, for all $\upomega\in\Omega_2$  there exists a constant $R_2(\upomega)$ such that the term $\|\bm\upbeta_k^{(j_k(m))}\left(\hat{\bm\uptheta}_k^{(I_{m,i})}\right)\|$ is bounded above by $R_2(\upomega)$.
Then, from (\ref{eq:mussel}) we have:
\begin{align*}
\|\bar{\bm{B}}_0(t_{k+1})-\bar{\bm{B}}_0(t_{k})\|&\leq a_k\sum_{m=1}^{s_k}\sum_{i=0}^{n_k(m)-1} \left(\frac{\tilde{A}_{k}(m) }{a_{k}}\right)R_2(\upomega)\notag\\
&= a_kX_kR_2(\upomega)
\end{align*}
for all $\upomega\in\Omega_2$, where $X_k$ was defined in (\ref{eq:apoint}).
 Using the result from Lemma \ref{claim:brey} we obtain the bound $|X_k|<R_3(\upomega)$ for $\upomega \in \Omega_3$, which implies:
\begin{align}
\label{eq:soeur}
\|\bar{\bm{B}}_0(t_{k+1})-\bar{\bm{B}}_0(t_{k})\|\leq a_kR_2(\upomega)R_3(\upomega)
\end{align}
for all $\upomega\in\Omega_2\cap\Omega_3$ (note that $P({\Omega_2}\cap{\Omega_3})=1$).
Because $\bm{B}_0(t)$ is a piecewise-linear interpolation of $\bar{\bm{B}}_0(t)$ at the points $\{t_k\}_{k\geq0}$,
it follows from (\ref{eq:soeur}) that for $t\in\mathbb{R}$:
\begin{align}
\label{eq:santiagokorman}
\|\bm{B}_k(t)\|=\|\bm{B}_0(t_k+t)-\bm{B}_0(t_k)\|\leq |t|R_2(\upomega)R_3(\upomega).
\end{align}
The inequality in (\ref{eq:santiagokorman}) implies that each function in $\{\bm{B}_k(t)\}_{k\geq0}$ is bounded in magnitude by $TR_2(\upomega)R_3(\upomega)$ for $t\in I_T$ and $\upomega\in\Omega_2\cap\Omega_3$ (independently of $k$).

Observe now that $\bm{B}_k(t)$ is continuous by construction. Since any continuous function on a compact set is uniformly continuous on that set, for each $k$ the function ${\bm{B}}_k(t)$ is uniformly continuous on $I_T$ (Rudin 1976, Theorem 4.19)\nocite{Rudin1976}. So far, we have proven that the functions in the set $\{\bm{B}_k(t)\}_{k\geq 0}$ are bounded uniformly (in $k$) for $t\in I_T$ and that, for each $k$, the function $\bm{B}_k(t)$ is uniformly continuous over $t\in I_T$. Next, we show $\bm{B}_k(t)$ converges to the zero vector uniformly over $t\in I_T$. 

By the last part of A5 (the convergence of the bias vector to zero), there exists a set $\Omega_4\subset\Omega$ with $P(\Omega_4)=1$ such that for any $\upomega\in\Omega_4$ and any $\upepsilon>0$ there exists a finite constant $K_1(\upomega,\upepsilon)$ such that $\|\bm\upbeta_k^{(j_k(m))}(\hat{\bm\uptheta}_k^{(I_{m,i})})\|\leq \upepsilon$
whenever $k\geq K_1(\upomega,\upepsilon)$. Therefore, using (\ref{eq:santiagokorman}) with $R_2(\upomega)$ replaced by $\upepsilon$ yields the bound:
\begin{align}
\label{eq:danceheart}
\|\bm{B}_k(t)\|=\|\bm{B}_0(t_k+t)-\bm{B}_0(t_k)\|\leq \upepsilon|t|R_3(\upomega),
\end{align}
provided $t_k+t > t_{K_1(\upomega,\upepsilon)}$. Note that having $k> K_1(\upomega, \upepsilon)$ does not guarantee that $t_k+t > t_{K_1(\upomega,\upepsilon)}$ since $t$ may be negative. However, by condition A0 (using the fact that $\sum_{k=0}^\infty a_k=\infty$) there exists a finite constant $K_2(T,K_1(\upomega,\upepsilon))$ with $K_2(T,K_1(\upomega,\upepsilon))\geq K_1(\upomega,\upepsilon)$ such that for $k\geq K_2(T,K_1(\upomega,\upepsilon))$ and $t\in I_T$ we have $t_k+t\geq t_k-T\geq t_{K_1(\upomega,\upepsilon)}$.
In other words, if $k$ is large enough (at least as large as $K_2(T,K_1)$), the value $t_k+t$ can also be made large enough so that the interpolated bias term at time $t_k+t$ is arbitrarily small. Let $k\geq K_2(T,K_1)$ and $t\in I_T$. Then, (\ref{eq:danceheart}) with $|t|$ replaced by $T$ implies:
\begin{align}
\label{eq:blonsky}
\left\|{\bm{B}}_k(t)\right\|=\left\|{\bm{B}}_0(t_k+t)-{\bm{B}}_0(t_k)\right\|\leq  \upepsilon TR_3(\upomega)
\end{align}
provided $\upomega \in \Omega_3\cap\Omega_4$. Since $K_1(\upomega,\upepsilon)$ and $K_2(T,K_1(\upomega,\upepsilon))$ do not depend on $t$ and since $\upepsilon$ can be taken to be as small as desired, ${\bm{B}}_k(t)$ converges uniformly (in $t$) to the zero vector for $t\in I_T$ and $\upomega\in\Omega_3\cap\Omega_4$ (note that $P(\Omega_3\cap\Omega_4)=1$). The fact that $\{{\bm{B}}_{k}(t)\}_{k\geq0}$ is uniformly equicontinuous for $t\in I_T$ now follows in a similar manner to the proof of Theorem 7.24 in Rudin (1976)\nocite{Rudin1976}.
\end{proof}

      The following lemma will be useful for 
      establishing some important properties regarding 
     the set $\{\bm{M}_k(t)\}_{k\geq0}$.

\begin{lemma}[Kushner and Clark 1978, Lemma 2.2.1]
\label{lem:opportunityxyz}
\nocite{kushnclark1978}
Let $\bm{\uprho}_k$ be a vector-valued random variable. For $t_k=\sum_{r=0}^{k-1}a_r$ define
$\bm{R}(t_k)\equiv \sum_{r=0}^{k-1}a_r\bm{\uprho}_r$ . For $t\in[0,\infty)$ let $\bm{R}(t)$ be the piecewise linear interpolation of $\{\bm{R}(t_k)\}_{k\geq 0}$. Additionally, for $t\in(-\infty,0]$ let $\bm{R}(t)\equiv \bm{R}(t_0)$. Assume:
\begin{align*}
 \lim_{k\rightarrow \infty}P\left(\sup_{m\geq k}\Big\|\sum_{r=k}^{m}a_{r}\bm{\uprho}_{r}\Big\|\geq \upvarepsilon\right)=0.
\end{align*}
Then, $\bm{R}(t)$ is uniformly continuous over $t\in \mathbb{R}$ w.p.1 and $\sum_{k=0}^\infty a_k\bm{\uprho}_k<\infty$ w.p.1.
\end{lemma}

The statement of the following Lemma is similar to the statement of Lemma \ref{claim:loan} with the exception that it pertains to $\{{\bm{M}}_{k}(t)\}_{k\geq0}$. The proof of the following lemma, however, is fundamentally different from the proof of Lemma \ref{claim:loan}.


\begin{lemma}
\label{claim:travaille}
 Assume conditions A0 and A6 hold and for any finite $T>0$ let $ I_T$ be defined as in Lemma \ref{claim:loan}. Then, the following statements hold for all $\upomega$ in a set w.p.1 and any finite $T>0$. The set $\{{\bm{M}}_{k}(t)\}_{k\geq0}$ is a set of functions each of which is uniformly continuous over $ I_T$. Additionally, the functions in $\{{\bm{M}}_{k}(t)\}_{k\geq0}$ are bounded over $t\in I_T$ uniformly in $k$. Furthermore, $\bm{M}_{k}(t)\rightarrow \bm{0}$ uniformly over $t\in I_T$ as $k\rightarrow \infty$. The former statements imply $\{{\bm{M}}_{k}(t)\}_{k\geq0}$ is uniformly equicontinuous over $ I_T$.
\end{lemma}

\begin{proof} 
By letting $\bm{\uprho}_r=\bm{D}_r$ and $\bm{R}(t)=\bm{M}_0(t)$ in Lemma \ref{lem:opportunityxyz}, condition A6 
implies that $\bm{M}_0(t)$ is uniformly continuous for $t\in\mathbb{R}$. Moreover, due to the manner in which $\bm{M}_k(t)$ can be obtained from $\bm{M}_0(t)$ (i.e., via shifting) we obtain the uniform continuity of $\bm{M}_k(t)$ for $t\in\mathbb{R}$ and, therefore, over $t\in I_T$. The fact that $\{\bm{M}_{k}(t)\}_{k\geq 0}$ is bounded over $t\in I_T$ uniformly in $k$ also follows from Lemma \ref{lem:opportunityxyz} as is shown next.

First, using the definition of $\bar{\bm{M}}_0(t_k)$ (see (\ref{eq:pulcritud1})) along with the definition of $\bm{M}_k(t)$ (see (\ref{eq:easychair0})) it follows that for any two nonnegative integers $n_1$ and $n_2$ satisfying $n_2\geq n_1+1$ the following holds:
\begin{align*}
\bm{M}_0(t_{n_2})-\bm{M}_0(t_{n_1})=\bar{\bm{M}}_0(t_{n_2})-\bar{\bm{M}}_0(t_{n_1})= \sum_{r=n_1}^{n_2-1}a_r\bm{D}_r
\end{align*}
(recall that $\bm{D}_r$ was defined in condition A6). 
Next, condition A6, Lemma \ref{lem:opportunityxyz}, and the Cauchy criterion for convergence (Rudin 1976, Theorem 3.11) imply that for any $\upepsilon>0$ there exists a constant $K_1(\upomega,\upepsilon)$ such that for $n_1, n_2\geq K_1(\upomega,\upepsilon)$:
 \begin{align}
 \label{eq:whowbuabuu}
 \Big\| \sum_{r=n_1}^{n_2-1}a_r\bm{D}_r\Big\|<\upepsilon
 \end{align}
 for all $\upomega$ in a set of probability one. Therefore, $\|\bm{M}_0(t_{n_2})-\bm{M}_0(t_{n_1})\|<\upepsilon$.
 In general, 
 for each $t\in I_T$ it follows from condition A0 (using the fact that $\sum_{k=0}^\infty a_k=\infty$) that
there exists a constant $K_2(\upomega,\upepsilon,T)$ such that if $k\geq K_2(\upomega,\upepsilon,T)$ then:
\begin{align}
\label{eq:spyplanesssarc}
t_{n}\leq t_k+t\leq t_{n+1}
\end{align}
 for some $n\geq K_1(\upomega,\upepsilon)$ (the value of $n$ depends on $k$ and on $t$ although this dependance has been omitted for simplicity since it does not impact our arguments). Next, note that (\ref{eq:whowbuabuu}) with $n_1=n$ and $n_2=n+1$ implies that $\|\bm{M}_0(t_{n+1})-\bm{M}_0(t_{n})\|<\upepsilon$. 
 Combining this with (\ref{eq:spyplanesssarc}) it follows that $\|\bm{M}_k(t)\|=\|\bm{M}_0(t_k+t)-\bm{M}_0(t_n)\|<\upepsilon$ whenever $t\in I_T$ and $k\geq K_2(\upomega,\upepsilon,T)$.
%
%
 Since $K_2(\upomega,\upepsilon,T)$ does not depend on $t$, $\bm{M}_k(t)$ must be bounded over $t\in I_T$ uniformly in $k$. Furthermore, since $\upepsilon$ can be arbitrarily small we also obtain the desired convergence of $\bm{M}_k(t)\rightarrow \bm{0}$ uniformly over $t\in I_T$. The uniform equicontinuity of $\{{\bm{M}}_{k}(t)\}_{k\geq0}$ follows immediately (see the end of the proof of Lemma \ref{claim:loan}).
\end{proof}

 The statement of the following lemma is similar to the statements of Lemmas \ref{claim:loan} and \ref{claim:travaille} with the exception that it pertains to the set $\{{\bm{N}}_{k}(t)\}_{k\geq0}$. Part of the proof of the following lemma is a combination of the proofs of Lemmas \ref{claim:loan} and \ref{claim:travaille} and, consequently, some details have been omitted.

\begin{lemma}
\label{claim:marseille}
 Assume conditions A0--A4 hold and for any finite $T>0$ let $ I_T$ be defined as in Lemma \ref{claim:loan}. Then, the following statements hold for all $\upomega$ in a set w.p.1 and any finite $T>0$. The set $\{{\bm{N}}_{k}(t)\}_{k\geq0}$ is a set of functions each of which is uniformly continuous over $ I_T$. Additionally, the functions in $\{{\bm{N}}_{k}(t)\}_{k\geq0}$ are bounded over $t\in I_T$ uniformly in $k$. Furthermore, $\bm{N}_{k}(t)\rightarrow \bm{0}$ uniformly over $t\in I_T$ as $k\rightarrow \infty$. The former statements imply $\{{\bm{N}}_{k}(t)\}_{k\geq0}$ is uniformly equicontinuous over $ I_T$.
\end{lemma}

\begin{proof}
 First, rewrite the expression for $\bar{\bm{N}}_0(t_k)$ given in (\ref{eq:pulcritud2}) as follows:
\begin{align}
\bar{\bm{N}}_0(t_k)
&=\sum_{r=0}^{k-1}a_r\left[\sum_{j=1}^d(x_r(j)-\upmu_r(j)+\upmu_r(j)-\upmu(j))\bm{g}^{(j)}(\hat{\bm{\uptheta}}_r)\right]=\sum_{r=0}^{k-1}a_r\bm\upnu_r,\label{eq:votedownsimb}
\end{align}
where $\bm\upnu_r\equiv\sum_{j=1}^d(x_r(j)-\upmu_r(j)+\upmu_r(j)-\upmu(j))\bm{g}^{(j)}(\hat{\bm{\uptheta}}_r)$ with $x_r(j)$, $\upmu_r(j)$, and $\upmu(j)$ defined in on pp. \pageref{eq:tinkelpan}--\pageref{eq:unrealantartica}. Using A0 (the fact that $\tilde{a}_k^{(j)}/a_k\rightarrow r_j$ w.p.1), A1, A2, A3, and A4, equation (\ref{eq:votedownsimb}) implies that $\bm\upnu_r$ is bounded in magnitude w.p.1 (for $\upomega\in\Omega_0\cap \Omega_1$) uniformly over $r$. Therefore, w.p.1 there exists a constant $R_4(\upomega)$ satisfying $\|\bm\upnu_r\|\leq R_4(\upomega)$ for $\upomega$ in a  set w.p.1. Then, via a derivation similar to that of (\ref{eq:santiagokorman}) it is possible to show that $\|\bm{N}_k(t)\|=\|\bm{N}_0(t_k+t)-\bm{N}_0(t_k)\|\leq |t|R_4(\upomega)$.
This implies $\{\bm{N}_k(t)\}_{k\geq0}$ is bounded over $t\in I_T$ uniformly in $k$.

 Because $\bm{N}_k(t)$ is a continuous function on $I_T$ it is clear that $\bm{N}_k(t)$ must also be uniformly continuous on $I_T$. It remains to show that  $\bm{N}_k(t)$ converges to the zero vector uniformly over $t\in I_T$. We begin by rewriting the expression for the vector $\bm\upnu_r$ appearing in (\ref{eq:votedownsimb}). First,
 using the last part of condition A2:
\begin{align}
x_r(j)-\upmu_r(j)=&\ E\left[\frac{\tilde{a}_r^{(j)}}{a_r}\right](C_r(j)-E[C_r(j)])\notag\\
&+\left(\frac{\tilde{a}_r^{(j)}}{a_r}-E\left[\frac{\tilde{a}_r^{(j)}}{a_r}\right]\right)C_r(j).\label{eq:pinottnoiress}
\end{align}
Using (\ref{eq:pinottnoiress}) let us express $\bm\upnu_r$ as $\bm\upnu_r= \bm\upnu_r'+\bm\upnu_r''$, where
\begin{align*}
\bm\upnu_r'&\equiv \sum_{j=1}^dE\left[\frac{\tilde{a}_r^{(j)}}{a_r}\right](C_r(j)-E[C_r(j)])\bm{g}^{(j)}(\hat{\bm{\uptheta}}_r),\notag\\
\bm\upnu_r''&\equiv \sum_{j=1}^d\left[\upeta_r(j)C_r(j)+\upmu_r(j)-\upmu(j)\right]\bm{g}^{(j)}(\hat{\bm{\uptheta}}_r),
\end{align*}
and $\upeta_k(j)\equiv{\tilde{a}_k^{(j)}}/{a_k}-E[{\tilde{a}_k^{(j)}}/{a_k}]$. Note that A0 implies $\upeta_k(j)\rightarrow 0$ w.p.1. Therefore,
since $C_k(j)$ is bounded (a consequence of condition A1) it follows that $\upeta_k(j)C_k(j)\rightarrow 0$ w.p.1. Next we study the asymptotic properties of $\bm\upnu_r'$ and $\bm\upnu_r''$.

Conditions A2, A3, and A4 along with the fact that $\upeta_k(j)C_k(j)\rightarrow 0$ w.p.1 imply $\bm\upnu_r''\rightarrow \bm{0}$ w.p.1.
The term $\bm\upnu_r'$, on the other hand, cannot be assumed to converge. However, it is possible to show $\sum_{k=0}^{\infty}a_{k}\bm{\upnu}_k'<\infty$, which we do next. Using the definition of $S_r$ given in condition A2,
conditions A0 and A1 imply that for a fixed $k$ the sequence $\{\sum_{r=k}^{m}a_rS_r\}_{m\geq k}$ (indexed by $m$)
is a martingale sequence. An implication of p. 315 in Doob (1953)\nocite{doob1953} is the following:
\begin{align}
P\left(\sup_{m\geq k}\Big\|\sum_{r=k}^{m}a_{r}S_r\Big\|\geq \upvarepsilon\right)&\leq\upvarepsilon^{-2}E\Big\|\sum_{r=k}^\infty a_{r}{S_r}\Big\|^2=\upvarepsilon^{-2}\sum_{r=k}^\infty a_{r}^{2}\Var{(S_r)},
\label{eq:ranchie2}
\end{align}
where equality holds by condition A2 (the fact that $E[S_kS_\ell]=0$ for $k\neq \ell$).
Moreover, since the $C_k(j)$ are bounded uniformly in $k$ and $j$ (condition A1) and since the sequence $E[{\tilde{a}_k^{(j)}}/{a_k}]$ converges (condition A0), the sequence $S_k$ must be bounded uniformly in $k$. Therefore, there exists an $\upsigma^2>0$ such that $\Var{(S_k)}\leq \upsigma^2$ for all $k$. Then, (\ref{eq:ranchie2}) along with condition A0 
implies:
\begin{align*}
\lim_{k\rightarrow\infty}P\left(\sup_{m\geq k}\Big\|\sum_{r=k}^{m}a_{r}S_r\Big\|\geq \upvarepsilon\right)\leq \lim_{k\rightarrow\infty}\frac{\upsigma^2}{\upvarepsilon^{2}}\sum_{r=k}^\infty a_{r}^{2}=0.
\end{align*}
Then, 
Lemma \ref{lem:opportunityxyz} implies:
\begin{align*}
\sum_{k=0}^{\infty}a_{k}\sum_{j=1}^dE\left[\frac{\tilde{a}_k^{(j)}}{a_k}\right](C_k(j)-E[C_k(j)])=\sum_{k=0}^{\infty}a_{k}S_k<\infty{\text{ w.p.1.}}
\end{align*}
Furthermore, since $\|\bm{g}^{(j)}(\hat{\bm{\uptheta}}_k)\|$ is bounded w.p.1 (conditions A3 and A4): 
\begin{align}
\label{eq:saxophone2}
\sum_{k=0}^{\infty}a_{k}\sum_{j=1}^dE\left[\frac{\tilde{a}_k^{(j)}}{a_k}\right](C_k(j)-E[C_k(j)])\bm{g}^{(j)}(\hat{\bm{\uptheta}}_k)=\sum_{k=0}^{\infty}a_{k}\bm{\upnu}_k'<\infty.
\end{align}

Following the ideas in the proofs of Lemmas \ref{claim:loan} and \ref{claim:travaille}, equation (\ref{eq:saxophone2}) and the fact that $\bm\upnu_k''\rightarrow \bm{0}$ w.p.1 can be used to show that $\bm{N}_{k}(t)$ converges to the zero vector uniformly over $t\in I_T$. Specifically, the term $\bm{\upnu}_r'$ can be related to the term $\bm{D}_r$ appearing in condition A6 and $\bm\upnu_r''$ can be related to the term:
\begin{align*}
\sum_{m=1}^{s_r}\sum_{i=0}^{n_r(m)-1} \left(\frac{\tilde{A}_{r}(m) }{a_{r}}\right)\left\|\bm\upbeta_r^{(j_r(m))}\left(\hat{\bm\uptheta}_r^{(I_{m,i})}\right)\right\|
\end{align*}
appearing in (\ref{eq:mussel}).
 Once again, the uniform equicontinuity of $\{{\bm{N}}_{k}(t)\}_{k\geq0}$ follows immediately (see the proof of Lemma \ref{claim:loan}).
 \end{proof}
 
  The statement of the following lemma is similar to the statements of Lemmas \ref{claim:loan}, \ref{claim:travaille}, and \ref{claim:marseille}, with the exception that it pertains to the set $\{{\bm{W}}_{k}(t)\}_{k\geq0}$.

\begin{lemma}
\label{claim:dasmaitien}
 Assume conditions A0, A1, and A3--A6 hold and for any finite $T>0$ let $ I_T$ be defined as in Lemma \ref{claim:loan}. Then, the following statements hold for all $\upomega$ in a set w.p.1 and any finite $T>0$. 
 The set $\{{\bm{W}}_{k}(t)\}_{k\geq0}$ is a set of functions each of which is uniformly continuous over $ I_T$. Additionally, the functions in $\{{\bm{W}}_{k}(t)\}_{k\geq0}$ are bounded over $t\in I_T$ uniformly in $k$. Furthermore, $\bm{W}_{k}(t)\rightarrow \bm{0}$ uniformly over $t\in I_T$ as $k\rightarrow \infty$. The former statements imply $\{{\bm{W}}_{k}(t)\}_{k\geq0}$ is uniformly equicontinuous over $ I_T$.
\end{lemma}

\begin{proof}
 First we prove the boundedness part of the lemma.  Recall the definition of $\bar{\bm{W}}_0(t_k)$ in (\ref{eq:pulcritud}) (we include it here for convenience):
\begin{align}
\label{eq:blackriverriverse}
\bar{\bm{W}}_0(t_k)= \sum_{r=0}^{k-1} a_{r}\sum_{m=1}^{s_r}\sum_{i=0}^{n_r(m)-1}\left(\frac{\tilde{A}_{r}(m)}{a_{r}}\right)\left[ {\bm{g}}^{(j_r(m))}\left(\hat{\bm\uptheta}_r^{\left(I_{m,i}\right)}\right)- {\bm{g}}^{(j_r(m))}(\hat{\bm\uptheta}_r)\right].
\end{align}
Conditions A3 and A4 imply all terms of the form ${\bm{g}}^{(j)}(\cdot)$ in (\ref{eq:blackriverriverse}) are bounded in magnitude by a constant $R_5(\upomega)$ for all $\upomega\in \Omega_1$. Thus, using Lemma \ref{claim:brey} and following a derivation similar to that of (\ref{eq:santiagokorman}) we obtain that $\|\bm{W}_k(t)\|=\|\bm{W}_0(t_k+t)-\bm{W}_0(t_k)\|\leq |t|R_3(\upomega)R_5(\upomega)$.
For $t\in I_T$ this implies $\|\bm{W}_k(t)\|\leq TR_3(\upomega)R_5(\upomega)$.
Thus, $\{{\bm{W}}_{k}(t)\}_{k\geq0}$ is bounded uniformly in $k$ for $t\in I_T$. The uniform continuity of each ${\bm{W}}_{k}(t)$ function on finite (i.e., closed and bounded) intervals follows by the continuity of ${\bm{W}}_{k}(t)$. 
Next we show uniform convergence of $\bm{W}_{k}(t)$ to zero.

First recall that $\hat{\bm\uptheta}_k^{(I_{m,i})}$ satisfies:
\begin{align}
\label{eq:romanee}
	\hat{\bm\uptheta}_k^{(I_{m,i})}=&\ \hat{\bm\uptheta}_k- \sum_{z=1}^{m-1}\sum_{\ell=0}^{n_k(z)-1} \left[\tilde{A}_{k}(z)\ \hat{\bm{g}}^{(j_k(z))}_k\left(\hat{\bm\uptheta}_k^{(I_{z,\ell})}\right)\right]-\sum_{\ell=0}^{i-1} \left[\tilde{A}_{k}(m)\ \hat{\bm{g}}^{(j_k(m))}_k\left(\hat{\bm\uptheta}_k^{(I_{m,\ell})}\right)\right]
\end{align}
(Line \ref{eq:indianindinorange} of Algorithm \ref{findme}). Next, recall that $\hat{\bm{g}}^{(j_k(z))}_k(\hat{\bm\uptheta}_k^{(I_{z,\ell})})$ can be written as the sum ${\bm{g}}^{(j_k(z))}(\hat{\bm\uptheta}_k^{(I_{z,\ell})})+{\bm{\upbeta}}^{(j_k(z))}_k(\hat{\bm\uptheta}_k^{(I_{z,\ell})})+{\bm{\upxi}}^{(j_k(z))}_k(\hat{\bm\uptheta}_k^{(I_{z,\ell})})$.
%
%
Furthermore, ${\bm{g}}^{(j_k(z))}(\hat{\bm\uptheta}_k^{(I_{z,\ell})})$ and ${\bm{\upbeta}}^{(j_k(z))}_k(\hat{\bm\uptheta}_k^{(I_{z,\ell})})$ are bounded w.p.1 (conditions A3--A5). Therefore, using condition A0 (the implication that $\tilde{a}_k^{(j)}\rightarrow 0$  w.p.1 for all $j$) and condition A1 we rewrite (\ref{eq:romanee}) as follows:
\begin{align}
\hat{\bm\uptheta}_k^{(I_{m,i})}=&\ \hat{\bm\uptheta}_k- \sum_{z=1}^{m-1}\sum_{\ell=0}^{n_k(z)-1} \tilde{A}_{k}(z)\ {\bm{\upxi}}^{(j_k(z))}_k\left(\hat{\bm\uptheta}_k^{(I_{z,\ell})}\right)-\sum_{\ell=0}^{i-1} \tilde{A}_{k}(m)\ {\bm{\upxi}}^{(j_k(m))}_k\left(\hat{\bm\uptheta}_k^{(I_{m,\ell})}\right)+ \bm{T}_k,\label{eq:digerrootoros}
\end{align}
where $\bm{T}_k$ is a vector such that $\bm{T}_k\rightarrow \bm{0}$ w.p.1. Additionally, condition A6 and Lemma \ref{lem:opportunityxyz} imply $\sum_{k=0}^\infty a_k\bm{D}_k<\infty$ w.p.1. Therefore, 
 \begin{align}
\sum_{z=1}^{m-1}\sum_{\ell=0}^{n_k(z)-1} \left[\tilde{A}_{k}(z)\ {\bm{\upxi}}^{(j_k(z))}_k\left(\hat{\bm\uptheta}_k^{(I_{z,\ell})}\right)\right]+\sum_{\ell=0}^{i-1} \left[\tilde{A}_{k}(m)\ {\bm{\upxi}}^{(j_k(m))}_k\left(\hat{\bm\uptheta}_k^{(I_{m,\ell})}\right)\right]\rightarrow \bm{0}{\text{ w.p.1}}\label{eq:casakakk}
\end{align}
(this is a consequence of Theorem 3.11 in Rudin 1976\nocite{Rudin1976}, also known as the Cauchy criterion for convergence). Using (\ref{eq:digerrootoros}) and (\ref{eq:casakakk}) and we can write $\hat{\bm\uptheta}_k^{(I_{m,i})}=\hat{\bm{\uptheta}}_k+\bm{T}'_k$, where $\bm{T}'_k$ is a vector such that $\bm{T}'_k\rightarrow \bm{0}$ w.p.1 as $k\rightarrow \infty$. Consequently, condition A3 implies $\|{\bm{g}}^{(j_r(m))}(\hat{\bm\uptheta}_r^{\left(I_{m,i}\right)})- {\bm{g}}^{(j_r(m))}(\hat{\bm\uptheta}_r)\|\rightarrow 0$ w.p.1
as $r\rightarrow \infty$ uniformly in $m$ and $i$.
Then, via a derivation similar to that of (\ref{eq:blonsky}) it follows that $\left\|{\bm{W}}_0(t_k+t)-{\bm{W}}_0(t_k)\right\|\leq T R_3(\upomega) \upepsilon$
for $k$ large enough and $t\in I_T$. 
Uniform equicontinuity then follows (see the end of the proof of Lemma \ref{claim:loan}).
\end{proof}

\begin{lemma}
\label{claim:musicaxyz} Assume conditions A2--A4 hold and for any finite $T>0$ let $ I_T$ be defined as in Lemma \ref{claim:loan}. Then, the following statements hold for all $\upomega$ in a set w.p.1 and any finite $T>0$. 
The set of functions $\{\int_0^t\bm{h}(\bar{\bm{Z}}_k(s))\ ds\}_{k\geq0}$ is bounded over $t\in I_T$ uniformly in $k$. Furthermore, $\{\int_0^t\bm{h}(\bar{\bm{Z}}_k(s))\ ds\}_{k\geq0}$ is uniformly equicontinuous over $ I_T$.
\end{lemma}

\begin{proof}
Recall from (\ref{eq:unrealantartica}) that $\bm{h}(\bm\uptheta)= \sum_{j=1}^d\upmu(j)\bm{g}^{(j)}(\bm\uptheta)$.
Then,
\begin{align}
\label{eq:buhai}
\int_0^t\bm{h}(\bar{\bm{Z}}_k(s))\ ds=\int_0^t \sum_{j=1}^d\upmu(j)\bm{g}^{(j)}(\bar{\bm{Z}}_k(s))\ ds.
\end{align}
Since $\bm{g}^{(j)}(\bar{\bm{Z}}_k(t))$ is bounded w.p.1 uniformly over $k$ and $t$ (conditions A3 and A4) and since $\upmu(j)$ must be bounded uniformly in $j$ (condition A2), then the magnitude of the integrand in (\ref{eq:buhai}) must be bounded uniformly over $t$ by some constant $R_6(\upomega)$ for $\upomega$ in a set w.p.1. Thus, the boundedness part of the lemma follows for $t\in I_T$. Next, for $t_1, t_2\in I_T$ with $t_1<t_2$:
\begin{align*}
\left\|\int_0^{t_2}\bm{h}(\bar{\bm{Z}}_k(s))\ ds-\int_0^{t_1}\bm{h}(\bar{\bm{Z}}_k(s))\ ds\right\|&\leq \int_{t_1}^{t_2}\left\|\bm{h}(\bar{\bm{Z}}_k(s))\right\|\ ds\leq |t_2-t_2|R_6(\upomega)
\end{align*}
independently of $k$. Therefore, the desired equicontinuity also follows.
\end{proof}

 The following Lemma presents a result similar to that of Lemmas \ref{claim:trompeta}--\ref{claim:musicaxyz} for the set of functions $\{{\bm{Z}}_k(t)\}_{k\geq0}$ with one important modification: parts of the result are now shown to hold for all $t\in \mathbb{R}$ rather than only for $t\in  I_T$.


\begin{lemma}
\label{claim:sinuses}
 Assume conditions A0--A6 hold. Then, the functions in the set $\{{\bm{Z}}_k(t)\}_{k\geq0}$ are bounded over $t\in \mathbb{R}$ uniformly in $k$ w.p.1. Additionally, $\{{\bm{Z}}_k(t)\}_{k\geq0}$ is uniformly equicontinuous for $t\in \mathbb{R}$ w.p.1.
\end{lemma}

\begin{proof}
The desired uniform boundedness was proven in Lemma \ref{claim:trompeta}. Next, recall from (\ref{eq:bondok}) that:
\begin{align*}
{\bm{Z}}_k(t)=&\ {\bm{Z}}_k(0)-{\bm{B}}_k(t)-{\bm{M}}_k(t)-{\bm{N}}_k(t)-{\bm{W}}_k(t)-\int_{0}^t{\bm{h}}(\bar{\bm{Z}}_k(s))\ ds.
\end{align*}
Here, Lemmas \ref{claim:loan}--\ref{claim:musicaxyz} and the fact that $\bm{Z}_k(t)$ is a shifted version of $\bm{Z}_0(t)$ imply that $\{{\bm{Z}}_k(t)\}_{k\geq0}$ is a set of uniformly equicontinuous functions for $t\in \mathbb{R}$. 
\end{proof}

\begin{lemma}
\label{claim:windowsxyz}
Assume conditions A0--A6 hold and for any finite $T>0$ let $ I_T$ be defined as in Lemma \ref{claim:loan}. Then, for any finite $T>0$ it follows that $\bm\upzeta_k(t)\rightarrow \bm{0}$ (recall $\bm\upzeta_k(t)$ was defined in (\ref{eq:levelred})) w.p.1 uniformly over $t\in  I_T$ as $k\rightarrow \infty$.
\end{lemma}

\begin{proof}
We remind the reader that any unspecified probabilistic arguments are meant w.p.1. The vector $\bm\upzeta_k(t)$ can be written as:
\begin{align}
\label{eq:haloodst}
\bm\upzeta_k(t)&=\int_{0}^t{\bm{h}}({\bm{Z}}_k(s))\ ds-\int_{0}^t{\bm{h}}(\bar{\bm{Z}}_k(s))\ ds\notag\\
&=\int_0^t \sum_{j=1}^d\upmu(j)\left[\bm{g}^{(j)}({\bm{Z}}_k(s))-\bm{g}^{(j)}(\bar{\bm{Z}}_k(s))\right]\ ds.
\end{align}
By condition A4, the vectors $\bar{\bm{Z}}_k(s)$ and ${\bm{Z}}_k(s)$ must lie in a $p$-dimensional ball of radius $R_1(\upomega)$ for all $\upomega\in\Omega_1$. Let $\Theta(\upomega)\subset \mathbb{R}^p$ denote the closure (using the Euclidean topology) of the aforementioned ball. Then, condition A3 implies $\bm{g}(\bm\uptheta)$ is uniformly continuous over $\bm\uptheta\in \Theta(\upomega)$. In other words, 
for any $\upepsilon>0$ there exists a $\updelta_1(\upomega,\upepsilon)>0$ such that $\|\bm{g}^{(j)}({\bm{Z}}_k(s))-\bm{g}^{(j)}(\bar{\bm{Z}}_k(s))\|<\upepsilon$ if $\|\bar{\bm{Z}}_k(s)-{\bm{Z}}_k(s)\|<\updelta_1(\upomega,\upepsilon)$.
Next, Lemma \ref{claim:sinuses} implies $\{\bm{Z}_k(t)\}_{k\geq0}$ is uniformly equicontinuous w.p.1 on $I_T$. Therefore, for $t_1,t_2\in I_T$ and any $\upepsilon>0$, there exists a $\updelta_2(\upepsilon,T,\upomega)>0$ such that $\|{\bm{Z}}_k(t_2)-{\bm{Z}}_k(t_1)\|<\updelta_1(\upomega,\upepsilon)$
whenever $|t_2-t_1|<\updelta_2(\upepsilon,T,\upomega)$ independently of $k$. Next, note that for any $\updelta_2(\upepsilon, T,\upomega)>0$ and $s\in I_T$ it is possible to take $k$ large enough (independently of $t$) so that $t_m\leq t_k+s\leq t_{m+1}$ where $t_{m+1}-t_m=a_m<\updelta_2(\upepsilon, T,\upomega)$ (this follows from condition A0). Then, for large enough $k$ and $s\in I_T$ we have $\|{\bm{Z}}_k(s)-\bar{\bm{Z}}_k(s)\|\leq \|{\bm{Z}}_k(t_{m+1})-{\bm{Z}}_k(t_m)\|<\updelta_1(\upepsilon,\upomega)$ and,
 therefore, $\|\bm{g}^{(j)}({\bm{Z}}_k(s))-\bm{g}^{(j)}(\bar{\bm{Z}}_k(s))\|<\upepsilon$ for all $s\in I_T$. Consequently,
  the maximum possible norm of the integral in (\ref{eq:haloodst}) for $t\in I_T$ is $\upepsilon T$, where $\upepsilon>0$ can be as small as desired provided $k$ is large (precisely how large is independent of $t$).
Thus, we have proven the desired uniform (in $t$) convergence to the zero vector.
\end{proof}

\section[A Convergence Theorem for GCSA]{A Convergence Theorem for GCSA}
\label{sec:convo}

We are now ready to prove the following theorem for convergence of the GCSA algorithm. Once again, all unspecified probabilistic arguments in the statement of the theorem as well as in its proof are meant w.p.1.

\begin{theorem}
\label{thm:hoeshoo} 
Let $\hat{\bm{\uptheta}}_k$ denote the GCSA iterates of Algorithm \ref{findme} and assume conditions A0--A6 hold. Then, there exists a subsequence 
$\{\bm{Z}_{k_i}(t)\}_{i\geq1}$ of $\{\bm{Z}_{k}(t)\}_{k\geq0}$ 
and a bounded (w.p.1) function $\bm{Z}(t)$ such that $\bm{Z}_{k_i}(t)\rightarrow \bm{Z}(t)$ uniformly over $t\in I_T$ for any finite $T>0$ as $i\rightarrow \infty$ ($I_T$ was defined in Lemma \ref{claim:loan}).
%
%
The aforementioned limit function, 
$\bm{Z}(t)$, satisfies the ODE in (\ref{eq:ode}). If in addition conditions A7 and A8 hold (i.e., if A0--A8 hold) then ${\bm{\bm{{\hat{\uptheta}}}}}_{k}\rightarrow {\bm\uptheta^{\ast}}$ w.p.1.
\end{theorem}

\begin{proof}The proof of this theorem consists of three main steps:

\begin{DESCRIPTION}
\item[Step 1:] Show that A0--A6 imply that for $t\in \mathbb{R}$ the functions $\{{\bm{Z}_{k}(t)\}}_{k\geq0}$ are bounded uniformly in $k$, and $\{{\bm{Z}_{k}(t)\}}_{k\geq0}$  is uniformly equicontinuous.
\item[Step 2:] Given Step 1, the Arzel\`a--Ascoli Theorem implies that the sequence $\{\bm{Z}_k(t)\}_{k\geq0}$ must have 
a subsequence $\{\bm{Z}_{k_i}(t)\}_{i\geq 1}$ that converges to some bounded function $\bm{Z}(t)$ uniformly over $t\in I_T$ for any finite $T>0$ as $i\rightarrow \infty$. This step consists of showing that $\bm{Z}(t)$ must satisfy the ODE in condition A7.
\item[Step 3:]  Using the asymptotic stability condition (condition A7) along with A8 show that $\bm{{\hat{\uptheta}}}_{k}$ must converge to $\bm\uptheta^\ast$.
\end{DESCRIPTION}
 Step 1 was completed in Lemma \ref{claim:sinuses}. Step 2 is an implication of  Lemmas \ref{claim:loan}--\ref{claim:dasmaitien}, \ref{claim:sinuses}, and \ref{claim:windowsxyz} as is shown next. Using the
 Arzel\`a--Ascoli Theorem (Kushner and Clark 1978, p. 20)\nocite{kushnclark1978} 
 along with the fact that $\bm{Z}_k(t)$ is a shifted version of $\bm{Z}_0(t)$
 it follows that there exists a subsequence $\{\bm{Z}_{k_i}(t)\}_{i\geq 1}$ and a bounded (w.p.1) function $\bm{Z}(t)$ such that $\bm{Z}_{k_i}(t)\rightarrow \bm{Z}(t)$ uniformly over $t\in I_T$ for any finite $T>0$ as $i\rightarrow \infty$. 
 Then, by (\ref{eq:levelred}) in conjunction with Lemmas \ref{claim:loan}--\ref{claim:dasmaitien} and \ref{claim:windowsxyz} (using the convergence to zero of the different functions) it follows that $\bm{Z}(t)$ satisfies:
\begin{align*}
\bm{Z}(t)= \lim_{i\rightarrow\infty}\bm{Z}_{k_i}(t)
= {\bm{Z}}(0)-\lim_{i\rightarrow \infty}\int_{0}^t{\bm{h}}({\bm{Z}}_{k_i}(s))\ ds,
\end{align*}
for $t\in I_T$.
Due to the uniform convergence of $\bm{Z}_{k_i}(t)$ on $I_T$ and using the continuity of $\bm{g}(\bm\uptheta)$, it follows that:
\begin{align}
\label{eq:alexgotseidankee}
\bm{h}({\bm{Z}}_{k_i}(t))\rightarrow  \sum_{j=1}^d\upmu(j)\bm{g}^{(j)}({\bm{Z}}(t))=\bm{h}({\bm{Z}}(t)){\text{ w.p.1}}
\end{align}
uniformly over $t\in I_T$ as $i\rightarrow \infty$. Using the uniform convergence in (\ref{eq:alexgotseidankee}) along with Lebesgue's dominated convergence theorem implies:
\begin{align}
\label{eq:ansdgwichtdil}
\bm{Z}(t)&= {\bm{Z}}(0)-\int_{0}^t\bm{h}({\bm{Z}}(s))\ ds
\end{align}
for $t\in I_T$. Because $T>0$ is arbitrary, (\ref{eq:ansdgwichtdil}) holds for all $t$. Therefore,
\begin{align*}
\dot{\bm{Z}}(t)&=-\sum_{j=1}^d\upmu(j)\bm{g}^{(j)}({\bm{Z}}(t)),
\end{align*}
which is the ODE in (\ref{eq:ode}).
 %
 The remainder of the proof (Step 3) is now identical to the last part of the proof of Theorem 2.3.1 in Kushner and Clark (1978). Therefore, only a brief outline of Step 3 is included below.

 By condition A8, we know there exists a subsequence $\{\bm{{\hat{\uptheta}}}_{k_i}\}_{i\geq1}$ such that $\bm{{\hat{\uptheta}}}_{k_i}\in A$ for all $i\geq 1$. Furthermore, by Step 2 we can assume without loss of generality that $\bm{Z}_{k_i}(t)\rightarrow\bm{Z}(t)$ where $\bm{Z}(t)$ satisfies the ODE in condition A7. Additionally, since $\bm{Z}_{k_i}(0)=\bm{{\hat{\uptheta}}}_{k_i}\in A$ where $A$ is a compact set then $\bm{Z}(0)\in A$ by construction. Because $\bm\uptheta^\ast$
 is an assymptotically stable (in the sense of Lyapunov) solution to the differential equation (\ref{eq:ode}), we know $ \lim_{t\rightarrow \infty}\bm{Z}(t)=\bm\uptheta^\ast$.
The stability properties of $\bm\uptheta^\ast$ then guarantee $\bm{{\hat{\uptheta}}}_{k}\rightarrow\bm\uptheta^\ast$ (see Kushner and Clark 1978, pp. 42--43).
\end{proof}

\begin{corollary}
\label{eq:idamagedittarantatan}
If the conditions of Theorem \ref{thm:hoeshoo} hold, then the sequence $\{\hat{\bm\uptheta}_k^{(I_{m,i})}\}$ of Algorithm \ref{findme} converges to $\bm\uptheta^\ast$ w.p.1 uniformly in $m$ and $i$ as $k\rightarrow \infty$.
\end{corollary}
 \begin{proof}
 This result follows immediately from Theorem \ref{thm:hoeshoo} along with the comment below (\ref{eq:casakakk}).
 \end{proof}
 The following section obtains Corollaries to Theorem \ref{thm:hoeshoo} pertaining to two special cases of the GCSA algorithm.
 
 \section[Two Special Cases of GCSA]{Two Special Cases of GCSA}
 \label{sec:specialitay}

 Theorem \ref{thm:hoeshoo} gave convergence conditions for the GCSA algorithm, two special cases of which were discussed in Section \ref{sec:wunderbar} (see (\ref{eq:beer149}) and (\ref{eq:beerabove})). In (\ref{eq:beerabove}), for example, the parameter vector was partitioned into two subvectors; at each iteration the subvector to update was chosen at random. A generalization of (\ref{eq:beerabove}) is obtained by partitioning $\bm\uptheta$ into $d$ subvectors. The resulting generalization of (\ref{eq:beerabove}) is presented in Algorithm \ref{beastwasdone} below. For simplicity, it is assumed the sets $\mathcal{S}_1,\dots,\mathcal{S}_d$ in (\ref{eq:notexclusive}) satisfy:
\begin{align}
\label{eq:scurry}
\bigcup_{i=1}^d\mathcal{S}_i=\mathcal{S}, \ \ \ \bigcap_{i=1}^d\mathcal{S}_i=\emptyset.
\end{align}
In order to study the behavior of Algorithm \ref{beastwasdone} we first prove two lemmas that are useful for understanding the behavior of the sequence $\tilde{a}_k^{(j)}/a_k$, a sequence whose relevance stems from condition A0 in Theorem \ref{thm:hoeshoo} (also note that this ratio affects the value of $\bm{h}(\bm\uptheta)$ through (\ref{eq:tinkelpan})). The following lemma gives sufficient conditions for the convergence w.p.1 of $\tilde{a}_k^{(j)}/a_k$.
 \begin{algorithm}[!t]                     
\caption{Randomized Subvector Selection}          
\label{beastwasdone}                          
\begin{algorithmic} [1]                   
\setstretch{1.5} 
\REQUIRE    $\hat{\bm{\uptheta}}_0$, $\{a_k^{(j)}\}_{k\geq 0}$ for $j=1,\dots,d$, and let $\mathcal{S}_1,\dots, \mathcal{S}_d$ be such that (\ref{eq:scurry}) holds. Set $k=0$.
\WHILE{stopping criterion has not been reached}
  \STATE{Let $j_k$ be a random variable with $P(j_k=j)=q(j)\neq0$ and $\sum_{j=1}^d q(j)=1$.}
	\STATE{Define:
	\begin{align*}
	\hat{\bm\uptheta}_{k+1}\equiv&\ \hat{\bm\uptheta}_k-\tilde{a}_k^{(j_{k})}\hat{\bm{g}}^{(j_{k})}(\hat{\bm{\uptheta}}_k).
	\end{align*}}
	\STATE{set $k=k+1$}
\ENDWHILE
\end{algorithmic}
\end{algorithm}

\begin{lemma}
\label{lem:devil}
Consider the setting of Algorithm \ref{beastwasdone}. Assume $a_k>0$ and that $a_k^{(j)}>0$ is a monotonically decreasing sequence with $a_k^{(j)}\rightarrow 0$. In addition, assume there exists a function $\uprho(j,q)$ with domain $\{1,\dots,j\}\times [0,1]$ and range in $\mathbb{R}$ such that if $\upepsilon>0$ is small enough that $0\leq q(j)\pm\upepsilon\leq 1$ then:
\begin{align*}
\lim_{k\rightarrow \infty}\frac{a^{(j)}_{\floor{k(q(j)-\upepsilon)}}}{a_{k}}=\uprho (j,q(j)-\upepsilon), \ \ \ \lim_{k\rightarrow \infty}\frac{a^{(j)}_{\ceil{k(q(j)+\upepsilon)}}}{a_{k}}=\uprho (j,q(j)+\upepsilon),
\end{align*}
 ($\floor{\cdot}$ and $\ceil{\cdot}$ denote the floor and ceiling functions, respectively). Assume there exist a constant $C<\infty$ such that $0<\uprho (j, q(j)\pm \upepsilon)<C$. Finally, let:
\begin{align*}
\lim_{\upepsilon\rightarrow 0}\uprho (j, q(j)\pm \upepsilon)=\uprho (j, q(j))\in [\uprho (j, q(j)+\upepsilon),\uprho (j, q(j)-\upepsilon)].
\end{align*}
Then, $\tilde{a}_k^{(j)}/a_k\rightarrow \uprho (j,q(j))$ w.p.1 with $0<\uprho (j,q(j))<C$.
\end{lemma}
\begin{proof}
Throughout this proof all unspecified probabilistic arguments are meant to hold w.p.1. From (\ref{eq:saropian}) we know that for $k>0$:
\begin{align}
\tilde{a}_{k+1}^{(j)}
&=a_0^{(j)}+\sum_{i=0}^k \chi\left\{\frac{\upvarphi_k^{(j)}}{k}-q(j)\geq \frac{i}{k}-q(j)\right\}(a_{i+1}^{(j)}-a_{i}^{(j)}),\label{eq:pianofortee}
\end{align}
where $0\leq \upvarphi_k^{(j)}\leq k$ was defined in (\ref{eq:wheretheasaredefinedd}). Furthermore, the strong law of large numbers implies that $\upvarphi_k^{(j)}/k-q(j)\rightarrow 0$ w.p.1 as $k\rightarrow \infty$. This implies (w.p.1) that for all $\upepsilon>0$ there exists an $N(\upepsilon)$ such that $|{\upvarphi_k^{(j)}}/{k}-q(j)|<\upepsilon$
for $k\geq N(\upepsilon)$ (the constant $N(\upepsilon)$ depends on $\upomega$ although the dependence has been omitted for simplicity). Therefore,
\begin{align}
\label{eq:samuraislicemetan}
\chi\left\{\frac{\upvarphi_k^{(j)}}{k}-q(j)\geq \frac{i}{k}-q(j)\right\}=
\begin{cases} 
0 &\mbox{if } {\text{$k\geq N(\upepsilon)$ and $i\geq \ceil{k(\upepsilon+q(j))}$}},\\
1 & \mbox{if } {\text{$k\geq N(\upepsilon)$ and $i\leq \floor{k(q(j)-\upepsilon)}$}}.
 \end{cases} 
\end{align}
Equation (\ref{eq:samuraislicemetan}) implies $\tilde{a}_{k+1}^{(j)}\in[a^{(j)}_{ \ceil{k(\upepsilon+q(j))}}, a^{(j)}_{\floor{k(q(j)-\upepsilon)}+1}]$ for $k\geq N(\upepsilon)$. Therefore:
\begin{align}
\tilde{a}_{k}^{(j)}\in\left[a_{\ceil{k(q(j)+\upepsilon)-(q(j)+\upepsilon)}}^{(j)},a_{\floor{k(q(j)-\upepsilon)-(q(j)-\upepsilon)}+1}^{(j)}\right]\label{eq:jing}
\end{align}
for $k$ large. From (\ref{eq:jing}) it follows that for large $k$ and $\upepsilon>0$ small enough that $0\leq q(j)\pm\upepsilon\leq 1$:
\begin{align*}
\frac{\tilde{a}_{k}^{(j)}}{a_k}\in \left[\frac{a^{(j)}_{\ceil{k(q(j)+\upepsilon)}}}{a_k},\frac{a^{(j)}_{\floor{k(q(j)-\upepsilon)}}}{a_k}\right].
\end{align*}
Therefore, ${\tilde{a}_{k}^{(j)}}/{a_k}$ must approach the interval $I(\upepsilon)\equiv [\uprho (j, q(j)+\upepsilon),\uprho (j, q(j)-\upepsilon)]$ as $k$ grows (i.e., the distance between ${\tilde{a}_{k}^{(j)}}/{a_k}$ and its closest point in the interval $I(\upepsilon)$ approaches zero).
Because ${I}(\upepsilon)$ contains the point $\uprho (j, q(j))$ for all $\upepsilon>0$ small enough (by assumption) and because the length of ${I}(\upepsilon)$ decreases as $\upepsilon$ decreases,
the sequence $\tilde{a}_{k}^{(j)}/a_k$ must converge w.p.1 to $0<\uprho (j, q(j))<C$.
\end{proof}

Let us comment on the assumptions of Lemma \ref{lem:devil}. The first requirement is that $q(j)>0$ depends neither on $k$ nor on $m$. In other words, when selecting the coordinates to update from one of the sets $\mathcal{S}_1,\dots,\mathcal{S}_d$, we require the probability of selecting $\mathcal{S}_j$ to be independent of the iteration and block numbers.
 The remaining assumptions are satisfied for sequences of the form $a_k^{(j)}=a^{(j)}/(k+A^{(j)}+1)^\upalpha$ and $a_k=a/(k+1)^\upalpha$ where $a^{(j)}>0$, $a>0$, $A^{(j)}>0$, and $0<\upalpha$. To see this, note that $\lim_{k\rightarrow \infty} a^{(j)}_k/a_k=a^{(j)}/a$ and that:
\begin{align*}
\lim_{k\rightarrow \infty}\frac{a^{(j)}_{\floor{ck}}}{a_{k}^{(j)}}=\lim_{k\rightarrow \infty}\left(\frac{k+1}{\floor{ck}+A^{(j)}+1}\right)^\upalpha=\frac{1}{c^\upalpha},
\end{align*}
for $c>0$. The last limit can be computed by using $c_k$ as an upper bound for $\floor{ck}$, using $ck-1$ as a lower bound for $\floor{ck}$, and observing that both the upper and lower bounds result in the same limit: $1/c^\upalpha$. Therefore, $\lim_{k\rightarrow \infty}{a^{(j)}_{\floor{ck}}}/{a_{k}}=a^{(j)}/a c^\upalpha$. Similarly, it is possible to show $\lim_{k\rightarrow \infty}{a^{(j)}_{\ceil{ck}}}/{a_{k}}=a^{(j)}/(a c^\upalpha)$. In the notation of Lemma \ref{lem:devil}: $\uprho (j,q(j)-\upepsilon)=a^{(j)}/[a (q(j)-\upepsilon)^\upalpha]$ for small $\upepsilon>0$ and $\uprho (j,q(j)+\upepsilon)=a^{(j)}/[a (q(j)+\upepsilon)^\upalpha]$. Then, all the conditions of Lemma \ref{lem:devil} are satisfied with $0<\uprho (j,q(j))=a^{(j)}/(aq(j)^\upalpha)<\infty$.

Lemma \ref{lem:devil} gave conditions for the convergence w.p.1 of $\tilde{a}_k^{(j)}/a_k$. It is also of interest to derive conditions for the convergence of $E[\tilde{a}_k^{(j)}/a_k]$ (condition A0). For this we first define the concept of uniform integrability.

\begin{definition}
\label{def:unifintegrable}
[Chung 2001\nocite{chung2001}]
Given a sequence $\mathcal{X}_k$ of random variables, the set $\{\mathcal{X}_k\}_{k\geq0}$ is said to be {\it{uniformly integrable}} if and only if:
\begin{align*}
\lim_{n\rightarrow \infty}\int_{|\mathcal{X}_k|> n}|\mathcal{X}_k|\ dP=0{\text{ uniformly in $k$.}}
\end{align*}
\end{definition}


\begin{lemma}
\label{lem:lemmylime}
Assume $a_k^{(j)}>0$ and $a_k>0$ are both strictly monotonically decreasing sequences and that all the assumptions of Lemma \ref{lem:devil} hold. Furthermore, let $\upepsilon>0$ be such that $q(j)-\upepsilon>0$ and assume $e^{-2\upepsilon^2k}/a_k\rightarrow 0$ as $k\rightarrow \infty$ (this last assumption holds when $a_k$ satisfies condition A0). Then, $E[\tilde{a}_k^{(j)}/a_k]\rightarrow \uprho (j,q(j))$.
\end{lemma}
\begin{proof}
For each $j=1,\dots, d$ we will show that the set $\{\tilde{a}_k^{(j)}/a_k\}_{k\geq 0}$ is uniformly integrable. Then, the result will follow from  Lemma \ref{lem:devil} along with Theorem 4.5.4 in Chung (2001)\nocite{chung2001}, which states that $E[\tilde{a}_k^{(j)}/a_k]\rightarrow \uprho (j,q(j))$ if the set $\{\tilde{a}_k^{(j)}/a_k\}_{k\geq 0}$ is uniformly integrable and $\tilde{a}_k^{(j)}/a_k\rightarrow \uprho (j,q(j))$ w.p.1.

For each $j$ define $\mathcal{X}_k^{(j)}\equiv\tilde{a}_k^{(j)}/a_k>0$ and $n_i\equiv a_0^{(j)}/a_i$. Then, $n_i\rightarrow \infty$ as $i\rightarrow \infty$. Furthermore, for each $i\geq 0$ and any probability measure $P$:
\begin{align}
\label{eq:anothgerway}
\int_{\mathcal{X}_k^{(j)}> n_i} \mathcal{X}_k^{(j)}\ dP\leq \sup_{k\geq 0}\left\{\int_{\mathcal{X}_k^{(j)}> n_i} \mathcal{X}_k^{(j)}\ dP\right\}.
\end{align}
Using the fact that $a_k$ and $a_k^{(j)}$ are both monotonically-decreasing sequences of strictly positive numbers, for each $k$ we have:
\begin{align*}
\int_{\mathcal{X}_k^{(j)}> n_i} \mathcal{X}_k^{(j)}\ dP\leq P(\mathcal{X}_k^{(j)}> n_i)\frac{a_0^{(j)}}{a_k}.
\end{align*}
Combining this bound with (\ref{eq:anothgerway}):
\begin{align}
\int_{\mathcal{X}_k^{(j)}> n_i} \mathcal{X}_k^{(j)}\ dP&\leq  \sup_{k\geq 0}\left\{P(\mathcal{X}_k^{(j)}> n_i)\frac{a_0^{(j)}}{a_k}\right\}
= a_0^{(j)}\sup_{k\geq i}\left\{\frac{P(\mathcal{X}_k^{(j)}> n_i)}{a_k}\right\}.
\label{eq:lepainabongout}
\end{align}
%
%
Because $a_{\floor{k(q(j)-\upepsilon)}}^{(j)}/a_k\rightarrow \uprho (j,(q(j)-\upepsilon))$ (by assumption), for each $\updelta>0$ there exists an $N(\updelta)>0$ such that $a_{\floor{k(q(j)-\upepsilon)}}^{(j)}/a_k<\uprho (j,q(j)-\upepsilon)+\updelta$
whenever $k\geq N(\updelta)$. 
  Let $M(\updelta)>0$ be such that $\uprho (j,q(j)-\upepsilon)+\updelta<a_0^{(j)}/a_i$ for $i\geq M(\updelta)$ ($M(\updelta)$ exists because $a_i$ is monotonically decreasing). Then, for $k\geq i\geq \max\{M(\updelta),N(\updelta)\}$:
\begin{align*}
\frac{a_{\floor{k(q(j)-\upepsilon)}}^{(j)}}{a_k}<\uprho (j,q(j)-\upepsilon)+\updelta\leq \frac{a_0^{(j)}}{a_i}=n_i.
\end{align*}
We then have the following bound when $k\geq i\geq \max\{M(\updelta),N(\updelta)\}$: 
\begin{align}
\label{eq:hoeffdingsboundme}
P(\mathcal{X}_k^{(j)}> n_i)\leq P\left(\frac{\tilde{a}_k^{(j)}}{a_k}> \frac{a_{\floor{k(q(j)-\upepsilon)}}^{(j)}}{a_k}\right)=P(\tilde{a}_k^{(j)}> a^{(j)}_{\floor{k(q(j)-\upepsilon)}}).
\end{align}
The term $P(\tilde{a}_k^{(j)}> a^{(j)}_{\floor{k(q(j)-\upepsilon)}})$ can be computed exactly. We do this next.

From (\ref{eq:pianofortee}) we can see that $\tilde{a}_k^{(j)}$ will be strictly greater than $a^{(j)}_{\floor{k(q(j)-\upepsilon)}}$ if the sequence $\{a_k^{(j)}\}_{k\geq0}$ has been used strictly less than $\floor{k(q(j)-\upepsilon)}$ times by the time $\tilde{a}_k^{(j)}$ is computed (this is due to the strictly monotonically decreasing nature of $a_k^{(j)}$). 
The number of times $\{a_k^{(j)}\}_{k\geq0}$ has been used can be modeled using a Binomial random variable with ``success'' probability $q(j)$. Therefore,
\begin{align}
\label{eq:hoeffdingsboundmeNEXT}
P(\tilde{a}_k^{(j)}> a^{(j)}_{\floor{k(q(j)-\upepsilon)}})&=\sum_{i=0}^{\floor{k(q(j)-\upepsilon)}-1} {{k}\choose{i}}q(j)^i (1-q(j))^{k-i}.
\end{align}
 Since $\floor{k(q(j)-\upepsilon)}-1\leq k(q(j)-\upepsilon)$, Hoeffding's inequality (e.g., Hoeffding 1963\nocite{hoeffding}) allows us to further bound the probability in (\ref{eq:hoeffdingsboundmeNEXT}) as follows:
\begin{align}
\label{eq:elleguteleriz}
P(\tilde{a}_k^{(j)}> a^{(j)}_{\floor{k(q(j)-\upepsilon)}})\leq e^{-2\upepsilon^2k}.
\end{align}
Combining (\ref{eq:hoeffdingsboundme}) with (\ref{eq:elleguteleriz}) yields $P(\mathcal{X}_k^{(j)}> n_i)\leq e^{-2\upepsilon^2k}$ for  $i\geq \max\{M(\updelta),N(\updelta)\}$ and $k\geq i$. Combining this bound with (\ref{eq:lepainabongout}) yields (for large values of $i$ and $k$):
\begin{align*}
\int_{\mathcal{X}_k^{(j)}> n_i} \mathcal{X}_k^{(j)}\ dP&\leq  a_0^{(j)}\sup_{k\geq i}\left\{\frac{e^{-2\upepsilon^2k}}{a_k}\right\}.
\end{align*}
Since $e^{-2\upepsilon^2k}/a_k\rightarrow 0$ (by assumption), we know that $\sup_{k\geq i}\{{e^{-2\upepsilon^2k}}/{a_k}\}\rightarrow 0$ as $i\rightarrow \infty$. Therefore, the set $\{\mathcal{X}_k\}_{k\geq 0}$ is uniformly integrable and Theorem 4.5.4 in (Chung 2001) implies $E[\tilde{a}_k^{(j)}/a_k]\rightarrow \uprho (j,q(j))$ as desired.
\end{proof}



The following corollary gives sufficient conditions for the iterates of Algorithm \ref{beastwasdone} to converge w.p.1 to $\bm\uptheta^\ast$, that is for $\hat{\bm{\uptheta}}_k\rightarrow \bm\uptheta^\ast$ w.p.1.

\begin{corollary}
\label{thm:hoeshoo49488} 
Let ${\bm{\bm{{\hat{\uptheta}}}}}_{k}$ be defined by according to Algorithm \ref{beastwasdone} and assume the following conditions hold:
\begin{DESCRIPTION}
\item[A0$'$] Let $a_{k}>0$, $a_k^{(j)}>0$, $ \sum_{k=0}^{\infty}a_{k}=\infty$ and $ \sum_{k=0}^{\infty}a_{k}^2<\infty$. Additionally, assume the conditions of Lemmas {\ref{lem:devil}} and \ref{lem:lemmylime} hold with $0<\uprho (j,q(j))<\infty$.
\item[A4$'$] For $k\geq0$, let there exist a set $\Omega_1\subset\Omega$ with $P(\Omega_1)=1$ and a scalar $0<R_1(\upomega)<\infty$ such that the set $\{\hat{\bm\uptheta}_k\}_{k\geq 0}$ is contained within a $p$-dimensional ball of radius $R_1(\upomega)$ centered at the origin for all $\upomega\in\Omega_1$.
\item[A5$'$] For $k\geq0$, let there exist a set $\Omega_2\subset\Omega$ with $P(\Omega_2)=1$ and a scalar $0<R_2(\upomega)<\infty$ such that the set $\{\bm\upbeta_k^{(j_k)}(\hat{\bm\uptheta}_k)\}_{k}$ is contained within a $p$-dimensional ball of radius $R_2(\upomega)$ centered at the origin for all $\upomega\in\Omega_2$.
Furthermore, let $\bm\upbeta_k^{(j_k)}(\hat{\bm\uptheta}_k)\rightarrow {\bm{0}}$ w.p.1.
%
\item[A6$'$] $\bm{D}_r\equiv \left(\frac{\tilde{a}_r^{(j_r)}}{a_r}\right)\bm\upxi_r^{(j_r)}(\hat{\bm\uptheta}_r)$
and $\lim_{k\rightarrow \infty}P\left(\sup_{m\geq k}\Big\|\sum_{r=k}^{m}a_{r}\bm{D}_{r}\Big\|\geq \upvarepsilon\right)=0$ for $\upvarepsilon>0$.
\item [A7$'$] Let $\bm{\uptheta^{\ast}}$ be a locally asymptotically stable (in the sense of Lyapunov) solution of the differential equation:
\begin{align}
\label{eq:ode333}
\dot{\bm{Z}}(t)=- \sum_{j=1}^dq(j)\uprho (j,q(j))\bm{g}^{(j)}(\bm{Z}(t))=-\bm{\Lambda}\bm{g}(\bm{Z}(t)),
\end{align}
where $\bm\Lambda$ is a diagonal matrix with $i$th diagonal entry $\Lambda_{ii}=q(j)\uprho (j,q(j))$ where $j$ is such that $i\in \mathcal{S}_j$. Let (\ref{eq:ode333}) have domain of attraction $DA(\bm{\uptheta^{\ast}})$. 
\item[A3$'$ and A8$'$] The same as A3 and A8, respectively.
\end{DESCRIPTION}
 Then, ${\bm{\bm{{\hat{\uptheta}}}}}_{k}\rightarrow {\bm\uptheta^{\ast}}$ w.p.1.
\end{corollary}

\begin{proof}
 We will show that conditions A0$'$ and A3$'$--A8$'$ imply A0--A8, the result will then follow from Theorem \ref{thm:hoeshoo}. First, note that A0$'$ along with the results from Lemmas {\ref{lem:devil}} and \ref{lem:lemmylime} imply that A0 holds with $r_j=\uprho (j,q(j))$. Additionally, condition A1 is satisfied because $s_k=1$ and $n_k(m)=1$. Next we show that condition A2 is satisfied. First, note that $x_k(j)=\chi\{j_k=j\}(\tilde{a}_k^{(j)}/a_k)$. Therefore, $\upmu_k(j)=E[\chi\{j_k=j\}(\tilde{a}_k^{(j)}/a_k)]$. However, because $\chi\{j_k=j\}$ and $\tilde{a}_k^{(j)}$ are independent in Algorithm \ref{beastwasdone}, $\upmu_k(j)=q(j)E[\tilde{a}_k^{(j)}/a_k]$
so that $\lim_{k\rightarrow \infty} \upmu_k(j)= q(j)\uprho (j,q(j))$. Then, the first part of A2 holds with $\upmu(j)=q(j)\uprho (j,q(j))$. Next, for Algorithm \ref{beastwasdone} we have $S_r=\sum_{j=1}^dE[\tilde{a}_r^{(j)}/a_r](\chi\{j_r=j\}-q(j))$. Therefore, by the independence of the variables $\{j_k\}_{k\geq 0}$, we see that the rest of condition A2 holds. Since A3$'$--A8$'$ are simply versions of A3--A8 rewritten in terms of the notation of Algorithm \ref{beastwasdone}, it follows that conditions A3--A8 are automatically satisfied. Therefore, the conclusion of Theorem \ref{thm:hoeshoo} holds.
\end{proof}

  Another special case of the GCSA algorithm pertains to the case where $\hat{\bm{\uptheta}}_k$ is updated via a deterministic subvector update pattern. This is described in Algorithm \ref{kirkey}. The basic idea is that the number of updates made during an iteration is a constant, $s$, and the $m$th block consists of $n(m)$ updates to subvector $j(m)$. In Algorithm \ref{kirkey} the variables $s$, $n(m)$, and $j(m)$ are deterministic (cyclic seesaw is a special case). The following corollary gives conditions for the convergence of the iterates of Algorithm \ref{kirkey}.
 \begin{algorithm}[!t]                     
\caption{Deterministic Pattern for Coordinate Selection}      
\label{kirkey}                
\begin{algorithmic} [1]                  
\setstretch{1.5} 
\REQUIRE    $\hat{\bm{\uptheta}}_0$, $\{a_k^{(j)}\}_{k\geq 0}$ for $j=1,\dots,d$, and let $\mathcal{S}_1,\dots, \mathcal{S}_d$ be such that (\ref{eq:scurry}) holds. Set $k=0$.
\WHILE{stopping criterion has not been reached}
         \STATE{Let $s\in \mathbb{Z}^+$ be a finite integer-valued constant.}\label{line:brothermove}
         \FOR{$m=1,\dots,s$}
         \STATE{Let $j(m)\in\{1,\dots,d\}$ and let $n(m)\in \mathbb{Z}^+$ be a finite integer.}
	\FOR{$i=1,\dots,n(m)$}
	\STATE{Define:
	\begin{align*}
	\hat{\bm\uptheta}_k^{(I_{m,i})}\equiv&\ \hat{\bm\uptheta}_k- \sum_{z=1}^{m-1}\sum_{\ell=0}^{n(z)-1} \left[{a}_{k}^{(j(z))}\ \hat{\bm{g}}^{(j(z))}\left(\hat{\bm\uptheta}_k^{(I_{z,\ell})}\right)\right]-\sum_{\ell=0}^{i-1} \left[{a}_{k}^{(j(m))}\ \hat{\bm{g}}^{(j(m))}\left(\hat{\bm\uptheta}_k^{(I_{m,\ell})}\right)\right]
	\end{align*}
	}
	\ENDFOR
	 \ENDFOR
	 \STATE{Let $\hat{\bm{\uptheta}}_{k+1}\equiv\hat{\bm{\uptheta}}_k^{(I_{m,i})}$ {\text{ with $m=s$ and $i=n(s)$}}.
	 }
	\STATE{set $k=k+1$}
\ENDWHILE
\end{algorithmic}
\end{algorithm}

\begin{corollary}
\label{thm:hoeshooHarry} 
Let ${\bm{\bm{{\hat{\uptheta}}}}}_{k}$ be defined as in Algorithm \ref{kirkey}. Assume the following hold:
\begin{DESCRIPTION}
\item[A0$''$] Let $a_{k}>0$, $a_k^{(j)}>0$, $ \sum_{k=0}^{\infty}a_{k}=\infty$ and $ \sum_{k=0}^{\infty}a_{k}^2<\infty$. Additionally, assume ${a}_k^{(j)}/a_k\rightarrow r_j$ with $0<r_j<\infty$.
\item[A4$''$] For $k\geq0$, $m\leq s$ ($s$ is defined in Line \ref{line:brothermove} of Algorithm \ref{kirkey}), $i\leq n(m)$, let there exist a set $\Omega_1\in\Omega$ with $P(\Omega_1)=1$ and a scalar $0<R_1(\upomega)<\infty$ such that the set $\{\hat{\bm\uptheta}_k^{(I_{m,i})}\}_{k,m,i}$ is contained within a $p$-dimensional ball of radius $R_1(\upomega)$ centered at the origin for all $\upomega\in\Omega_1$.
\item[A5$''$] For $k\geq0$, $m\leq s$, $i\leq n(m)$, let there exist a set $\Omega_2\in\Omega$ with $P(\Omega_2)=1$ and a scalar $0<R_2(\upomega)<\infty$ such that the set $\{\bm\upbeta_k^{(j(m))}(\hat{\bm\uptheta}_k^{(I_{m,i})})\}_{k,m,i}$ is contained within a $p$-dimensional ball of radius $R_2(\upomega)$ centered at the origin for all $\upomega\in\Omega_2$.
Furthermore, let $\bm\upbeta_k^{(j(m))}(\hat{\bm\uptheta}_k^{(I_{m,i})})\rightarrow {\bm{0}}$ w.p.1 uniformly in $m$ and $i$.
\item[A6$''$] Define $\bm{D}_r\equiv \sum_{m=1}^{s}\sum_{i=0}^{n(m)-1} \left(\frac{{a}_{r}^{(j(m))} }{a_{r}}\right)\bm\upxi^{(j(m))}_r\left(\hat{\bm\uptheta}_r^{(I_{m,i})}\right)$.
Assume for all $\upvarepsilon>0$ we have $\lim_{k\rightarrow \infty}P\left(\sup_{m\geq k}\Big\|\sum_{r=k}^{m}a_{r}\bm{D}_{r}\Big\|\geq \upvarepsilon\right)=0$.
\item [A7$''$] $\bm{\uptheta^{\ast}}$ is a locally asymptotically stable (in the sense of Lyapunov) solution of:
\begin{align}
\label{eq:leghearhere}
\dot{\bm{Z}}(t)=- \sum_{j=1}^d\left[r_j\sum_{m=1}^s\chi\{j(m)=j\}n(m)\right]\bm{g}^{(j)}(\bm{Z}(t))=-\bm\Lambda\bm{g}(\bm{Z}(t)),
\end{align}
where the $i$th diagonal entry of $\bm\Lambda$ equals $\Lambda_{ii}=r_j\sum_{m=1}^s\chi\{j(m)=j\}n(m)$
\label{eq:professiedadadna}
where $j$ is such that $i\in \mathcal{S}_j$. Let (\ref{eq:leghearhere}) have domain of attraction $DA(\bm{\uptheta^{\ast}})$.
\item[A3$''$ and A8$''$] The same as A3 and A8, respectively.
\end{DESCRIPTION}
 Then, ${\bm{\bm{{\hat{\uptheta}}}}}_{k}\rightarrow {\bm\uptheta^{\ast}}$ w.p.1.
\end{corollary}

\begin{proof}
 Because all gain sequences are deterministic (not random), condition A0$''$ implies A0. Additionally, since $s_k=s<\infty$ and since $n_k(m)=n(m)$ is a bounded deterministic function of $m$, then condition A1 is also satisfied. Next, note that for Algorithm {\ref{kirkey}}:
\begin{align*}
\upmu_k(j)=\left(\frac{{a}_{k}^{(j)} }{a_{k}}\right)\sum_{m=1}^{s}\chi\{j(m)=j\}n(m).
\end{align*}
Here, $\upmu_k(j)$ is a deterministic quantity since the subvector to update (and hence the indicator random variables in the previous equation) are deterministic functions of $m$. Since $a_k^{(j)}/a_k\rightarrow r_j$ then $\upmu(j)=r_j\sum_{m=1}^{s}\chi\{j(m)=j\}n(m)$. Thus, the first part of A2 holds. The second part of A2 holds since $S_r=0$ for all $r$. Because A3$''$--A8$''$ are versions of A3--A8 that have been rewritten in terms of Algorithm \ref{kirkey}, we have shown conditions A0--A8 are satisfied. 
\end{proof}

The following section discusses the validity of conditions A0--A8.

 \section[On the Convergence Conditions]{On the Convergence Conditions}
 \label{sec:discussconvergence}

Sections \ref{sec:12or15miles}--\ref{sec:specialitay} presented conditions for the convergence of the GCSA algorithm (conditions A0--A8).
 It is worthwhile to note that conditions A0--A8 closely resemble those in Kushner and Clark (1978) for the convergence of SA procedures. This is not surprising given that the proof of Theorem \ref{thm:hoeshoo} is based on principles similar to those in the proof of Theorem 2.3.1 in Kusher and Clark (1978) \nocite{kushnclark1978}. This section comments on the validity of A0--A8.
 
 
{\underline{\it{Condition A0}}}. Requiring $a_k>0$ and $a_k^{(j)}>0$ is a reasonable assumption when one has complete control over the specific form of the gain sequences (in the deterministic steepest descent algorithm, having the gain sequences be strictly positive guarantee that the update direction is in fact a descent direction). Furthermore, the requirement that $\sum_{k=0}^\infty a_k=\infty$ and $\sum a_k^2<\infty$ is easily satisfied by sequences of the form $a_k=a/(1+k+A)^\upalpha$ where $a>0$ and $A>0$ and where $0.5< \upalpha\leq 1$ (the non-convergence and square summability for series associated with $a_k$ is a long-standing requirement in SA). 
%
%
 %

Another assumptions in condition A0 is that the sequence $\tilde{a}_k^{(j)}/a_k$ converges in expectation and w.p.1 of to some finite, strictly positive constant. Because $\tilde{a}_k^{(j)}/a_k$ must converge (in a stochastic sense) to a non-zero constant for all $j$, all gain sequences must converge to zero at the same rate (in a stochastic sense). For Algorithm \ref{beastwasdone} (randomized subvector selection), Lemmas \ref{lem:devil} and \ref{lem:lemmylime} give conditions under which $\tilde{a}_j^{(j)}/a_k$ converges w.p.1 and in expectation to a finite constant. For Algorithm \ref{kirkey}, it suffices to have $a_k^{(j)}=a^{(j)}/(k+A^{(j)}+1)^\upalpha$ and $a_k=a/(k+1)^\upalpha$, where $a^{(j)}>0$, $a>0$, $A^{(j)}>0$, and $0<\upalpha\leq 1$ ($\upalpha$ must be the same for all $j$). It is important to note that the sequence $a_k$ is not used by the GCSA algorithm. Rather, $a_k$ serves only as a representation of the rate at which ${a}_k^{(j)}$ decreases. It is important to note that the requirement in condition A0 that $r_j$ be strictly positive is introduced only to guarantee that $\sum_{j=1}^d\upmu(j)\bm{g}^{(j)}(\bm\uptheta)=\bm{0}$ only when $\bm{g}(\bm\uptheta)=\bm{0}$. This implies that the only way in which the ODE in (\ref{eq:ode}) is equal to zero is if $\bm{Z}(t)=\bm\uptheta^\ast$. Note that if the sets $\mathcal{S}_1,\dots,\mathcal{S}_d$ overlap, it is still possible that the only solution to $\sum_{j=1}^d\upmu(j)\bm{g}^{(j)}(\bm\uptheta)=\bm{0}$ occurs at $\bm\uptheta=\bm\uptheta^\ast$ even when some of the $\upmu(j)$ are equal to zero (this occurs provided $\sum_{j=1}^d\upmu(j)\bm{g}^{(j)}(\bm\uptheta)=\bm\Lambda\bm{g}(\bm\uptheta)$ for some positive definite diagonal matrix $\bm\Lambda$).

{\underline{\it{Conditions A1 and A2}}}. Condition A1 requires that the number of blocks within the $k$ iteration, $s_k$, and the number of updates within each block, $n_k(m)$, be uniformly bounded over $k$, $m$, and $\upomega$. A special case where A1 holds is when a hard bound is imposed on $s_k$ and $n_k(m)$, as in the case of Algorithms \ref{beastwasdone} and \ref{kirkey}. Additionally, both these algorithms were shown to satisfy condition A2, which imposes a restriction on how $s_j$ and $n_j(m)$ relate to $s_i$ and $n_i(m)$ for $i\neq j$.

\label{page:onconditionA4444}{\underline{\it{Conditions A3 and A4}}}. Despite the fact that the gradient of $L(\bm\uptheta)$ is not explicitly used by the GCSA algorithm, the convergence theory of Section \ref{sec:convo} requires the existence of $\bm{g}(\bm\uptheta)$ and its continuity, a fact which may be hard to verify. Condition A4 is also difficult to verify in practice. Here, it is assumed that the algorithm's iterates are bounded w.p.1 (see Figure \ref{fig:spheres}). While the boundedness of the iterates is a common assumption throughout the SA, it remains somewhat controversial 
(e.g., Benveniste et al. 1990, p. 46)\nocite{benveniste1990}. 
Kushner and Clark (1978)\nocite{kushnclark1978}, however, mention that this boundedness is not necessarily a strong assumption since one typically imposes bounds on $\bm\uptheta$ in practice.
 Projected versions of GCSA could be the subject of future work and, with this in mind, the following quote seems fitting: 
\begin{quote}
There is no general scheme for showing $P(Q) = 1$ [the iterates of a realization are bounded w.p.1]. There are, however, problem-specific techniques for special problem classes. . . . One way to escape the boundedness issue is to alter the algorithm by projecting the iterates back onto a prescribed, large bounded set whenever they exit from the same. The trade-off is that the limiting ODE becomes more complicated. It is now confined to the said set and thus involves a ``reflection at the boundary'' of the same in an appropriate sense. (Borkar 1998\nocite{borkar1998}, on an asynchronous SA algorithm related to GCSA).
\end{quote}
\label{page:quotepage}

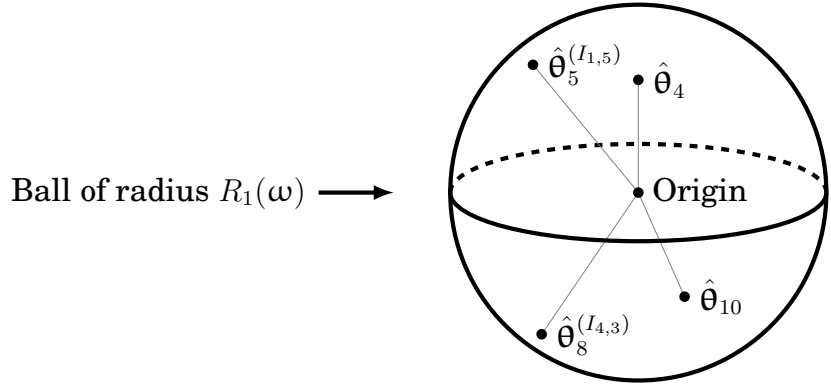
\begin{figure}[!t]
\centering
 \raisebox{1em}{\begin{tikzpicture}

\definecolor{mynewcoloring}{RGB}{102, 128, 153}
 \definecolor{mygreen}{RGB}{80, 135, 96}

\def\R{2.5} 
\def\angEl{15} 

\foreach \t in {0} { \DrawBackLatitudeCircle[\R]{\t} }


\coordinate (a) at (0,1.5);
\coordinate (b) at (1,-1,1);
\coordinate (c) at (-0.9,-1.5,1);
\coordinate (d) at (-1.4,1.7,0);

\coordinate (O) at (0,0);
\coordinate (r1) at (-3.25,0,0);
\coordinate (r2) at (-4.25,0,0);

\draw[->,ultra thick] (r2)--(r1);
\node at (-4.35,0,0)[left] {Ball of radius $R_1(\upomega)$};
\draw[help lines] (O)--(a);
\draw[help lines] (O)--(b);
\draw[help lines] (O)--(c);
\draw[help lines] (O)--(d);

\fill[black] (a) circle[radius=2pt] node[right] {\ $\hat{\bm{\uptheta}}_{4}$};
\fill[black] (b) circle[radius=2pt] node[right] {\ $\hat{\bm{\uptheta}}_{10}$};
\fill[black] (c) circle[radius=2pt] node[right] {\ $\hat{\bm{\uptheta}}_{8}^{(I_{4,3})}$};
\fill[black] (d) circle[radius=2pt] node[right] {\ $\hat{\bm{\uptheta}}_{5}^{(I_{1,5})}$};
\fill[black] (O) circle[radius=2pt] node[right] {\ Origin};

\foreach \t in {0} { \DrawFrontLatitudeCircle[\R]{\t} }

\draw[ultra thick] (0,0) circle (\R);
\end{tikzpicture}}
\caption[Condition A4 on the boundedness of the iterates]{Condition A4 states that the iterates produced by a single realization of GCSA must be contained within a ball of radius $R_1(\upomega)<\infty$ for $\upomega\in \Omega_1$}
\label{fig:spheres}
\end{figure}

{\underline{\it{Condition A5}}}. 
When the GCSA update directions are obtained using SG-type measurements, all bias terms are identically zero and condition A5 would automatically be satisfied. For the cyclic seesaw SPSA algorithm, Spall (1992)\nocite{spall1992} gives conditions under which the bias term is bounded (w.p.1) and converges to zero. In a manner analogous to Lemma 1 in Spall (1992)\nocite{spall1992} and under similar conditions (after a natural adaptation for the cyclic setting) it can be shown that the bias terms of the GCSA algorithm satisfy  A5.

{\underline{\it{Condition A6}}}. This condition restricts the relationship between $a_k$ and $\bm{D}_k$, a vector which represents the ``total noise'' of the $(k+1)$st iteration of GCSA (see the statement of condition A6 for a the precise definition of $\bm{D}_k$). Let us give an example of when condition A6 holds. Assume that for each $k\geq0$ the sequence  $\{\sum_{r=k}^m a_r \bm{D}_r\}_{m\geq k}$ (indexed by $m$) is a martingale sequence, that $E\|\bm{D}_k\|^2<\infty$, and that $E[\bm{D}_i^\top\bm{D}_j]=0$ for $i\neq j$. Then, by a relation on p. 315 in Doob (1953)\nocite{doob1953} we obtain the following inequality:
\begin{align}
\lim_{k\rightarrow\infty}P\left(\sup_{m\geq k}\Big\|\sum_{r=k}^{m}a_{r}\bm{D}_{r}\Big\|\geq \upvarepsilon\right)&\leq\lim_{k\rightarrow\infty}\upvarepsilon^{-2}\sum_{r=k}^\infty a_{r}^{2}E\|\bm{D}_{r}\|^2.\label{eq:ranch}
\end{align}
If in addition  we assume $\sum_{k=1}^\infty a_k^2<\infty$ and $E\|\bm{D}_{k}\|^2<\upsigma^2$ for some finite $\upsigma^2>0$, then the limit of the upper bound in (\ref{eq:ranch}) is equal to zero and A6 holds. Requiring the noise term to have a bounded variance could be a reasonable assumption when the update directions are obtained using SG-based gradient measurements. However, it is not always reasonable to assume the noise has bounded variance. For example, for the SPSA algorithm (which is a special case of GCSA) it is often the case that the variance of $\bm{D}_k$ grows at a rate proportional to $1/c_k^2$ (recall $c_k\rightarrow 0$). Therefore, in order for the limit in the upper bound of (\ref{eq:ranch}) to be zero, it would be necessary to have $a_k$ and $c_k$ satisfy $\sum_{k=1}^\infty a_k^2/c_k^2<\infty$ (see, for example, Spall 2003, p. 183)\nocite{ISSO}\label{page:anotherpagereffs}. 

{\underline{\it{Conditions A7 and A8}}}. Since the premise of stochastic optimization is that the function $L(\bm\uptheta)$ and its gradient are unknown, conditions A7 and A8 are likely to be impossible to verify in practice. Nevertheless, it is impractical to attempt to derive convergence conditions for an optimization algorithm without imposing restrictions on the properties of the function to minimize. Informally, conditions A7 and A8 force $\hat{\bm{\uptheta}}_k$ be inside the domain of attraction of $\bm\uptheta^\ast$ with enough frequency (for infinitely many $k$) that the iterates begin to benefit from the asymptotic stability of $\bm\uptheta^\ast$, which results in $\hat{\bm{\uptheta}}_k\rightarrow \bm\uptheta^\ast$ w.p.1. A few observations regarding the ODE in condition A7 are given next.

Understanding the ODE in (\ref{eq:ode}) requires understanding the terms $\upmu(j)$ (defined on p. \pageref{eq:tinkelpan}). When ${\tilde{a}_{k}^{(j)}}/{a_{k}}$ is independent of $s_k$, $n_k(m)$, and $j_k(m)$ (a valid assumption for Algorithms \ref{beastwasdone} and \ref{kirkey}),
$\upmu(j)$ can be rewritten as:
\begin{align*}
\upmu(j)=&\ r_j \times [\text{asymptotic average \# times entries $\mathcal{S}_j$ are updated per-iteration}].
\end{align*}
By construction, $\upmu(j)>0$ for all $j$. Thus, there exists a diagonal matrix $\bm\Lambda$ with strictly positive diagonal entries (determined by the $\upmu(j)$) such that:
\begin{align}
\label{eq:goodnewspillsconpaint}
\dot{\bm{Z}}(t)=- \bm\Lambda\bm{g}(\bm{Z}(t))
\end{align} 
(see p. \pageref{eq:ode333} and p. \pageref{eq:professiedadadna} for two special cases of $\bm\Lambda$). The ODE in (\ref{eq:goodnewspillsconpaint}) is a slight variation of the ODE that arises in the standard SA algorithm from (\ref{eq:formof}), where $\bm\Lambda=\bm{I}$ (see, for example, Spall 2003, Chapter 7)\nocite{ISSO}. 


At this point one might wonder whether $\bm\uptheta^\ast$, a minimizer of $L(\bm\uptheta)$, is necessarily a stable solution to (\ref{eq:goodnewspillsconpaint}). The answer, unfortunately, is no. While $\bm{Z}(t)=\bm\uptheta^\ast$ is certainly an {\it{equilibrium}} point of (\ref{eq:goodnewspillsconpaint}) (and hence a solution to (\ref{eq:goodnewspillsconpaint})), it is entirely possible for $\bm{Z}(t)=\bm\uptheta^\ast$ not to be a {\it{stable}} equilibrium point of (\ref{eq:goodnewspillsconpaint}) even when $L(\bm\uptheta)$ is differentiable for all degrees of differentiation) and $\bm\Lambda=\bm{I}$ (e.g., Absil and Kurdyka 2006, Proposition 2)\nocite{ansilandkurd2006}. Thus, condition A7 is not automatically satisfied when $\bm\uptheta^\ast$ is a minimizer of $L(\bm\uptheta)$. We will note, however, that if $\bm\uptheta^\ast$ is a Lyapunov stable solution to $\dot{\bm{Z}}(t)=-\bm{g}(\bm{Z}(t))$, if the entries of $\bm\Lambda$ have finite magnitudes, and if the smallest eigenvalue of $\bm\Lambda$ is strictly positive (more generally, if the smallest real part of an eigenvalue is strictly positive), then $\bm\uptheta^\ast$ is also a Lyapunov stable solution to (\ref{eq:goodnewspillsconpaint}) (e.g., Benveniste et al. 1990, pp. 111--112)\nocite{benveniste1990}. Therefore, asking for $\bm\uptheta^\ast$ to be a stable solution of (\ref{eq:goodnewspillsconpaint}) is not a stronger condition than the corresponding stability assumption when $\bm\Lambda=\bm{I}$.

 
 In general, it is impossible to know whether there is a unique solution to $\bm{g}(\bm\uptheta)=\bm{0}$. However, if $\bm\uptheta^{\ast,1}$ and $\bm\uptheta^{\ast,2}$ denote two distinct zeros of the gradient, it is impossible for both vectors to satisfy A7 and A8. Therefore, the GCSA iterates will converge to whichever of $\bm\uptheta^{\ast,1}$ and $\bm\uptheta^{\ast,2}$ (if any) satisfies both A7 and A8. Based on the idea in  Kushner and Clark (1978, p.39)\nocite{kushnclark1978}, however, condition A8 could be relaxed to allow all zeros of the gradient to be contained within some set that is stable (in the sense of Lyapunov), which would result in the GCSA algorithm's iterates converging to the aforementioned set w.p.1. This, however, would not guarantee convergence to a zero of $\bm{g}(\bm\uptheta)$.

\section{Concluding Remarks}
\label{sec:concremarksconvergence}

 This chapter gave a theorem for the convergence of the GCSA algorithm (Theorem \ref{thm:hoeshoo}) as well as two corollaries pertaining to special cases of GCSA (Corollaries \ref{thm:hoeshoo49488} and \ref{thm:hoeshooHarry}). The proof of Theorem \ref{thm:hoeshoo} is based on the proof of Theorem 2.3.1 in Kushner and Clark (1978)\nocite{kushnclark1978} which, in turn, is based on the ODE-based method for proving convergence of SA algorithms that was introduced by Ljung (1977)\nocite{ljung1977}. The time-dependent recursion in (\ref{eq:levelred}), which describes the evolution of the GCSA iterates, is similar to equation (2.3.3) in Kushner and Clark (1978)\nocite{kushnclark1978} although we point out that (\ref{eq:levelred}) is significantly more complex. As a result, some of the conditions of Theorem \ref{thm:hoeshoo} 
 resemble the conditions of Kushner and Clark's Theorem 2.3.1.\nocite{kushnclark1978}

 In contrast to this chapter's focus on convergence w.p.1 of the GCSA iterates, the following chapter focuses on the asymptotic normality of the normalized iterates from Algorithm \ref{kirkey}, a special case of GCSA in which the subvector to update is selected according to a deterministic pattern (see, for example, the algorithm specified by (\ref{eq:believeeyes0}--f), which is a special case of Algorithm \ref{kirkey}).

%% file: chapter3.tex

\chapter{Asymptotic Normality of GCSA}
\label{sec:ROC}

When applicable, asymptotic normality results can be used to construct approximate confidence regions for the SA iterates, to compute the relative efficiency between two SA algorithms by comparing their asymptotic mean squared errors (e.g., Spall 1992)\nocite{spall1992}, or even to define the rate of convergence of the vector-valued random sequence $\hat{\bm{\uptheta}}_k$.
Fabian (1968)\nocite{fabian1968}, a seminal paper in the SA literature, provides conditions for the asymptotic normality of the iterates of SA algorithms (after an appropriate centering and scaling). Fabian's theorem (Fabian 1968, Theorem 2.2)\nocite{fabian1968}, gives conditions under which there exists a constant $\upbeta>0$ such that $k^{\upbeta/2}(\hat{\bm{\uptheta}}_k-\bm\uptheta^\ast)$ has a limiting multivariate normal distribution with a theoretically computable mean and covariance. 
 One of the assumptions from Fabian's theorem, however, 
 is too strong for the GCSA algorithm. Specifically, Fabian (1968, Assumption 2.2.1)\nocite{fabian1968} requires a matrix that affects the mean and covariance matrix in the limiting distribution of $k^{\upbeta/2}(\hat{\bm{\uptheta}}_k-\bm\uptheta^\ast)$ to be symmetric. Section \ref{sec:genfabiangenfabian} generalizes Fabian's result by relaxing the symmetry assumption of the aforementioned matrix. The resulting generalization 
%
%
  expands the theorem's applicability to include a broader range of SA algorithms of practical interest (see Appendix \ref{sec:fabiansecgeneralize}) including a special case of GCSA.

The remainder of this chapter is organized as follows. First, Section \ref{sec:reviewfabian} reviews Fabian's theorem and its connection to SA algorithms. Then, Section \ref{subsec:primadonna} explains the nature of the incompatibility between Fabian's theorem and the GCSA algorithm. Section \ref{sec:genfabiangenfabian} provides a generalization of Fabian's theorem which is used in Section \ref{sec:upperpotomacshell} to derive an asymptotic normality result for a special case of GCSA. Section \ref{sec:bonestrailseason} comments on the validity of the assumptions in Section \ref{sec:upperpotomacshell}. Lastly, Section \ref{sec:dontmeanthings} contains concluding remarks.




\section[Reviewing Fabian's Theorem]{Reviewing Fabian's Theorem}
\label{sec:reviewfabian}

This section briefly reviews Fabian's result on asymptotic normality.


\begin{theorem}
[Fabian 1968, Theorem 2.2 or ``Fabian's Theorem''\nocite{fabian1968}]
\label{thm:fabian}
 For $k\geq 1$, let $\bm{V}_k$, $\bm{W}_k$, $\bm{T}_k$, and $\bm{T}$ be vectors in $\mathbb{R}^p$, and let $\bm\Gamma_k$, $\bm\Phi_k$, $\bm\Sigma$, $\bm\Gamma$, $\bm\Phi$, and $\bm{P}$ be matrices in $\mathbb{R}^{p\times p}$. Let $\bm{W}_k$ satisfy the recursion: 
\begin{align}
\label{eq:alien}
\bm{W}_{k+1}=(\bm{I}-k^{-\upalpha}\bm\Gamma_k)\bm{W}_k+\frac{\bm{T}_k}{k^{\upalpha+\upbeta/2}}+\frac{\bm\Phi_k\bm{V}_k}{k^{(\upalpha+\upbeta)/2}}.
\end{align}
Additionally, let $\mathcal{F}_k$ be a non-decreasing sequence of $\upsigma$-fields and assume there exists a set $\mathcal{S}$ such that $\mathcal{F}_k\subset \mathcal{S}$ for all $k$. Assume the following conditions hold:
\begin{DESCRIPTION}
\item[B0] $\bm\Gamma_k$, $\bm\Phi_{k-1}$, and $\bm{V}_{k-1}$ are $\mathcal{F}_k$-measurable. 
\item[B1]$\bm\Gamma_k\rightarrow\bm\Gamma$ w.p.1 where $\bm\Gamma=\bm{P}\bm\Lambda\bm{P}^\top$ for some real orthogonal matrix $\bm{P}$ and a diagonal matrix $\bm\Lambda$ with strictly positive eigenvalues (an important implication is that the matrix $\bm\Gamma$ must be symmetric).
\item[B2] $\bm\Phi_k\rightarrow \bm\Phi$ w.p.1.
\item[B3] Either $\bm{T}_k\rightarrow \bm{T}$ w.p.1 or $E\|\bm{T}_k-\bm{T}\|\rightarrow 0$.
\item[B4] $E[\bm{V}_k|\mathcal{F}_k]=\bm{0}$ and there exists $C$ such that $C>\|E[\bm{V}_k\bm{V}_k^\top|\mathcal{F}_k]-\bm\Sigma\|\rightarrow 0$.
\item[B5] $\upsigma_{k,r}^2\equiv E\chi\{\|\bm{V}_k\|^2\geq rk^\upalpha\}\|\bm{V}_k\|^2$ (recall $\chi\{\mathcal{E}\}$ denotes the indicator function of the event $\mathcal{E}$). For every $r>0$ assume that one of the following holds:
\begin{enumerate}[label=(\roman*)]
\item $\lim_{k\rightarrow \infty}\upsigma_{k,r}^2=0$.
\item $\upalpha=1$ and $\lim_{n\rightarrow \infty}n^{-1}\sum_{k=1}^n\upsigma^2_{k,r}=0$.
\end{enumerate}
\item[B6] Let $\uplambda\equiv \min_i{\{\Lambda_{ii}\}}$ where $\Lambda_{ii}$ is the $i$th diagonal entry of $\bm\Lambda$. Let $\upalpha$ and $\upbeta$ be constants such that $0<\upalpha\leq 1$ and $0\leq \upbeta$. Define $\upbeta_+\equiv\upbeta$ if $\upalpha=1$ and $\upbeta_+\equiv 0$ otherwise. Let $\upbeta_+<2\uplambda$.
\label{page:scratchless}
\end{DESCRIPTION}
 Then, the asymptotic distribution of $k^{\upbeta/2}\bm{W}_k$ is a multivariate normal random variable with mean $(\bm\Gamma-(\upbeta_+/2)\bm{I})^{-1}\bm{T}$ and covariance matrix $\bm{PQP}^\top$, where the $(i,j)$th entry of $\bm{Q}$ is equal to $(\bm{P}^\top\bm{\Phi}\bm\Sigma\bm\Phi^\top\bm{P})_{ij}(\Lambda_{ii}+\Lambda_{jj}-\upbeta_+)^{-1}$ with $(\bm{P}^\top\bm{\Phi}\bm\Sigma\bm\Phi^\top\bm{P})_{ij}$ denoting the $(i,j)$th entry of the matrix $\bm{P}^\top\bm{\Phi}\bm\Sigma\bm\Phi^\top\bm{P}$.
\end{theorem}

Let us discuss the connection between Fabian's theorem (Theorem \ref{thm:fabian}) and stochastic optimization.
 Consider a stochastic optimization algorithm that produces updates according to (\ref{eq:youknownothing}). Writing $\hat{\bm{g}}_k(\hat{\bm{\uptheta}}_k)$ as in (\ref{eq:tennisey}) 
and letting $\bm{W}_k= \hat{\bm{\uptheta}}_k-\bm\uptheta^\ast$, the algorithm in (\ref{eq:youknownothing}) can be rewritten as follows:
\begin{align}
\label{eq:husseinbolt}
\bm{W}_{k+1}=\bm{W}_k-a_k\left({\bm{g}}(\hat{\bm{\uptheta}}_k)+{\bm{\upbeta}}_k(\hat{\bm{\uptheta}}_k)+{\bm{\upxi}}_k(\hat{\bm{\uptheta}}_k)\right).
\end{align}
Asssume  $L(\bm\uptheta)$ is twice continuously differentiable and denote its Hessian matrix by
$\bm{H}(\bm\uptheta)$.
Then, using Taylor's theorem we may write $\bm{g}(\hat{\bm{\uptheta}}_k)=\tilde{\bm{H}}_k(\hat{\bm{\uptheta}}_k-\bm\uptheta^\ast)$,
where the $i$th row of $\tilde{\bm{H}}_k$ is equal to the $i$th row of $\bm{H}(\bm\uptheta)$ evaluated at $\bm\uptheta=(1-\uplambda_i)\hat{\bm{\uptheta}}_k+\uplambda_i\bm\uptheta^\ast$ for some $\uplambda_i\in[0,1]$ which depends on $\hat{\bm{\uptheta}}_k$.
By expanding $\bm{g}(\hat{\bm{\uptheta}}_k)$ around $\bm\uptheta^\ast$ in this way, we can rewrite (\ref{eq:husseinbolt}) in the following manner:\label{page:thispagehello}
\begin{align}
\bm{W}_{k+1}
&=\left(\bm{I}-k^{-\upalpha}k^{\upalpha}a_k\tilde{\bm{H}}_k\right)\bm{W}_k-\frac{k^{\upalpha+\upbeta/2}a_k\bm\upbeta_k(\hat{\bm{\uptheta}}_k)}{k^{\upalpha+\upbeta/2}}-\frac{k^{(\upalpha+\upbeta)/2}a_k\bm\upxi_k(\hat{\bm{\uptheta}}_k)}{k^{(\upalpha+\upbeta)/2}}.
\label{eq:crownuni}
\end{align}
Letting $\bm\Gamma_k=k^{\upalpha}a_k\tilde{\bm{H}}_k$, $\bm\Phi_k=\bm{I}$, $\bm{T}_k=-k^{\upalpha+\upbeta/2}a_k\bm\upbeta_k(\hat{\bm{\uptheta}}_k)$, and $\bm{V}_k=-k^{(\upalpha+\upbeta)/2}a_k\bm\upxi_k(\hat{\bm{\uptheta}}_k)$,
it follows that (\ref{eq:crownuni}), and therefore (\ref{eq:youknownothing}), can be rewritten in the form of (\ref{eq:alien}). Consequently, Fabian's theorem can be used to derive conditions for the asymptotic normality of $k^{\upbeta/2}(\hat{\bm{\uptheta}}_k-\bm\uptheta^\ast)$ when $\hat{\bm{\uptheta}}_k$ is obtained via the general SA recursion in  (\ref{eq:youknownothing}).
%

Undoubtedly, Fabian's theorem is already applicable to a variety of SA algorithms (see, for example, Zhou and Hu 2014,  Kar et al. 2013, Hu et al. 2012, and Zorin et al. 2000 to name a few recent applications). However, a critical assumption in Fabian's theorem is that $\bm\Gamma_k\rightarrow \bm\Gamma$ w.p.1 for some real, positive definite matrix $\bm\Gamma$ (condition B1). This assumption 
%
gives rise to an important complication
when attempting to write the GCSA algorithm in the form of (\ref{eq:alien}); the nature of this complication is explored in the following section. 

\section[A Limitation of Fabian's Theorem]{A Limitation of Fabian's Theorem}
\label{subsec:primadonna}
In this section we explore an issue that arises when attempting to write the recursion for the GCSA algorithm (see Line \ref{eq:indianindinorange} of Algorithm \ref{findme}) in the form of (\ref{eq:alien}), the recursion in Fabian's theorem. The nature of this issue is simple: the matrix $\bm\Gamma$ from condition B1 cannot generally be assumed to be symmetric for the GCSA algorithm, thus assumption B1 is not always satisfied and Fabian's theorem is not applicable. Using a simple example, we first explain why $\bm\Gamma$ cannot be assumed to be symmetric for the GCSA algorithm. Then, we discuss why it is not always possible to redefine $\bm\Gamma_k$, $\bm{T}_k$, $\bm{\Phi}_k$, and $\bm{V}_k$ in such a way that $\bm\Gamma$ (the limit w.p.1 of $\bm\Gamma_k$) is symmetric. Finally, we propose a slight generalization of Fabian's theorem that would allow for treatment of the special case of the GCSA algorithm where the subvector to update is selected according to a deterministic pattern (see, for example, the algorithm defined by (\ref{eq:william})).

Let us begin by rewriting the cyclic seesaw SA algorithm (a special case of GCSA) in the form of (\ref{eq:alien}). Recall that cyclic seesaw SA satisfies:
\begin{align}
\label{eq:santa}
\bm{{\hat{\uptheta}}}_{k+1}=\bm{{\hat{\uptheta}}}_{k}-a_{k}^{(1)}\hat{\bm{g}}_k^{(1)}(\hat{\bm\uptheta}_k)-a_{k}^{(2)}\hat{\bm{g}}_k^{(2)}(\hat{\bm\uptheta}_k^{(I)}),
\end{align}
where $\hat{\bm{\uptheta}}_k^{(I)}$ is defined in (\ref{eq:propoto2}). If $L(\bm\uptheta)$ is twice continuously differentiable then expanding $\bm{g}(\hat{\bm{\uptheta}}_k)$ around $\bm\uptheta^\ast$ and $\bm{g}(\hat{\bm{\uptheta}}_k^{(I)})$ around $\hat{\bm{\uptheta}}_k$:
\begin{align*}
\bm{g}^{(1)}(\hat{\bm{\uptheta}}_k)&=\bm{J}^{(1)}(\overline{\bm\uptheta}_k)(\hat{\bm{\uptheta}}_k-\bm\uptheta^\ast),\notag\\
\bm{g}^{(2)}(\hat{\bm{\uptheta}}_k)&=\bm{J}^{(2)}(\overline{\bm\uptheta}_k)(\hat{\bm{\uptheta}}_k-\bm\uptheta^\ast),\notag\\
\bm{g}^{(2)}(\hat{\bm{\uptheta}}_k^{(I)})&=\bm{g}^{(2)}(\hat{\bm{\uptheta}}_k)+\bm{J}^{(2)}(\overline{\bm\uptheta}_k^{(I)})(\hat{\bm{\uptheta}}_k^{(I)}-\hat{\bm{\uptheta}}_k),
\end{align*}
where $\bm{J}^{(j)}(\bm\uptheta)$ is the Jacobian of $\bm{g}^{(j)}(\bm\uptheta)$ with respect to $\bm\uptheta$,
  the $i$th row of $\bm{J}^{(j)}(\overline{\bm\uptheta}_k)$ is equal to the $i$th row of $\bm{J}^{(j)}(\bm\uptheta)$ evaluated at $\bm\uptheta=(1-\uplambda_i)\hat{\bm{\uptheta}}_k+\uplambda_i\bm\uptheta^\ast$ for some $\uplambda_i\in[0,1]$ which depends on $\hat{\bm{\uptheta}}_k$, and the $i$th row of $\bm{J}^{(2)}(\overline{\bm\uptheta}_k^{(I)})$ is the $i$th row of $\bm{J}^{(2)}(\bm\uptheta)$ with $\bm\uptheta=(1-\uplambda_i)\hat{\bm{\uptheta}}_k^{(I)}+\uplambda_i\hat{\bm{\uptheta}}_k$ for some $\uplambda_i\in[0,1]$ which depends on $\hat{\bm{\uptheta}}_k^{(I)}$. We now rewrite (\ref{eq:santa}) as:
\label{page:betweenersow}
\begin{align}
\bm{{\hat{\uptheta}}}_{k+1}=&\ \bm{{\hat{\uptheta}}}_{k}-a_{k}^{(1)}\bm{J}^{(1)}(\overline{\bm\uptheta}_k)(\hat{\bm{\uptheta}}_k-\bm\uptheta^\ast)-a_{k}^{(2)}\bm{J}^{(2)}(\overline{\bm\uptheta}_k)(\hat{\bm{\uptheta}}_k-\bm\uptheta^\ast)-a_{k}^{(2)}\bm{J}^{(2)}(\overline{\bm\uptheta}_k^{(I)})(\hat{\bm{\uptheta}}_k^{(I)}-\hat{\bm{\uptheta}}_k)\notag\\
&-a_k^{(1)}\left({\bm{\upbeta}}_k^{(1)}(\hat{\bm{\uptheta}}_k)+{\bm{\upxi}}_k^{(1)}(\hat{\bm{\uptheta}}_k)\right)-a_k^{(2)}\left({\bm{\upbeta}}_k^{(2)}(\hat{\bm{\uptheta}}_k^{(I)})+{\bm{\upxi}}_k^{(2)}(\hat{\bm{\uptheta}}_k^{(I)})\right).\label{eq:buckwild}
\end{align}
Using (\ref{eq:propoto2}) it follows that $-(\hat{\bm{\uptheta}}_k^{(I)}-\hat{\bm{\uptheta}}_k)/a_k^{(1)}=\bm{g}^{(1)}(\hat{\bm{\uptheta}}_k)+{\bm{\upbeta}}^{(1)}_k(\hat{\bm{\uptheta}}_k)+{\bm{\upxi}}^{(1)}_k(\hat{\bm{\uptheta}}_k)$. Combining this last observation with (\ref{eq:buckwild}) and using (\ref{eq:tennisey}), the cyclic seesaw SA algorithm can be written in the form of (\ref{eq:alien}) by letting $\bm{W}_k=\hat{\bm{\uptheta}}_k-\bm\uptheta^\ast$, 
\begin{subequations}
\begin{align}
\bm\Gamma_k=&\ {k^\upalpha}\left[a_k^{(1)}\bm{J}^{(1)}(\overline{\bm\uptheta}_k)+ a_k^{(2)}\bm{J}^{(2)}(\overline{\bm\uptheta}_k)\right],\label{eq:shmee0}\\
\bm{T}_k=& -{k^{\upalpha+\upbeta/2}}\Big[a_k^{(1)}\bm{\bm{\upbeta}}_k^{(1)}(\hat{\bm\uptheta}_k)+a_k^{(2)}\bm{\upbeta}_k^{(2)}(\hat{\bm\uptheta}_k^{(I)})-a_k^{(1)}a_k^{(2)}\bm{J}^{(2)}(\overline{\bm\uptheta}_k^{(I)})\bm{g}^{(1)}(\hat{\bm\uptheta}_k)\notag\\
&-a_k^{(1)}a_k^{(2)}\bm{J}^{(2)}(\overline{\bm\uptheta}_k^{(I)})\bm{\upbeta}_k^{(1)}(\hat{\bm\uptheta}_k)-a_k^{(1)}a_k^{(2)}E\Big(\bm{J}^{(2)}(\overline{\bm\uptheta}_k^{(I)})\bm{\upxi}_k^{(1)}(\hat{\bm\uptheta}_k)\Big|\mathcal{F}_k\Big)\Big],\label{eq:shmee1}\\
\bm{V}_k=& - k^{(\upalpha+\upbeta)/2}\Big[a_k^{(1)}\bm{\upxi}_k^{(1)}(\hat{\bm\uptheta}_k)+a_k^{(2)}\bm{\upxi}_k^{(2)}(\hat{\bm\uptheta}_k^{(I)})+a_k^{(1)}a_k^{(2)}E\Big(\bm{J}^{(2)}(\overline{\bm\uptheta}_k^{(I)})\bm{\upxi}_k^{(1)}(\hat{\bm\uptheta}_k)\Big|\mathcal{F}_k\Big)\notag\\
&-a_k^{(1)}a_k^{(2)}\bm{J}^{(2)}(\overline{\bm\uptheta}_k^{(I)})\bm{\upxi}^{(1)}_k(\hat{\bm\uptheta}_k)\Big],\label{eq:shmee}
\end{align}
\end{subequations}
where $\mathcal{F}_k\equiv \{\hat{\bm{\uptheta}}_0, \hat{\bm{\uptheta}}_0^{(I)},\hat{\bm{\uptheta}}_1, \hat{\bm{\uptheta}}_1^{(I)}, \dots, \hat{\bm{\uptheta}}_k\}$, and letting $\bm\Phi_k=\bm{I}$.

In general, the matrix $\bm\Gamma_k$ in (\ref{eq:shmee0}) cannot be assumed to converge to a symmetric matrix w.p.1. To see this consider the case where $k^\upalpha a_k^{(1)}\rightarrow a^{(1)}$, $k^\upalpha a_k^{(2)}\rightarrow a^{(2)}$, and $\bm{J}^{(j)}(\hat{\bm{\uptheta}}_k)\rightarrow \bm{J}^{(j)}(\bm\uptheta^\ast)$ w.p.1 (e.g., when $L(\bm\uptheta)$ is twice continuously differentiable and $\hat{\bm{\uptheta}}_k\rightarrow \bm\uptheta^\ast$ w.p.1). Then,
\begin{align}
\label{eq:lookatme}
\bm\Gamma_k\rightarrow \bm\Gamma=a^{(1)}\bm{J}^{(1)}(\bm\uptheta^\ast)+a^{(2)}\bm{J}^{(2)}(\bm\uptheta^\ast){\text{ w.p.1.}}
\end{align}
If $a^{(1)}=a^{(2)}=a>0$ then $\bm\Gamma=a\bm{H}(\bm\uptheta^\ast)$ is clearly a symmetric matrix. This is not generally the case, however, when $a^{(1)}\neq a^{(2)}$ (an exception being the case where $L(\bm\uptheta)$ is linearly separable in $\bm\uptheta^{(1)}$ and $\bm\uptheta^{(2)}$). Because there is no {\it{unique}} way to define the variables $\bm\Gamma_k$, $\bm{T}_k$, $\bm{\Phi}_k$, and $\bm{V}_k$ in (\ref{eq:alien}), it may be tempting to think that it is always possible to redefine these variables so that $\bm\Gamma$ is symmetric. Next we show that such a redefinition often leads to very strong assumptions on $\hat{\bm{\uptheta}}_k$.



Suppose an SA algorithm can be written in the form of (\ref{eq:alien}) and that all of the conditions of Fabian's theorem are satisfied {\it{with the exception that $\bm\Gamma$ is not symmetric}}. 
%
%
For simplicity, let us consider a special case where writing the algorithm in the form of (\ref{eq:alien}) can be done by letting $\upbeta\neq 0$, $\bm{T}_k=\bm{T}$, $\bm\Phi_k=\bm{I}$, and $\bm\Gamma_k=\bm\Gamma$ where $\bm\Gamma$ is not symmetric (note that we are assuming that $\bm{T}_k$, $\bm{\Phi}_k$, and $\bm{\Gamma}_k$ do not depend on $k$). 
Now, assume the matrices $\bm\Gamma_k'$, $\bm\Phi_k'$ and vectors $\bm{T}_k'$, $\bm{V}_k'$ provide an alternative way to write the SA algorithm in the form of (\ref{eq:alien}). Furthermore, assume that $\bm\Gamma_k'$ converges w.p.1 to a positive definite matrix. Then, since
\label{eq:anotherpagerefgoeshere}
\begin{align*}
\bm{W}_{k+1}=&\ (\bm{I}-k^{-\upalpha}\bm\Gamma_k')\bm{W}_k+\frac{\bm{T}}{k^{\upalpha+\upbeta/2}}+\frac{\bm{V}_k}{k^{(\upalpha+\upbeta)/2}}+k^{-\upalpha}(\bm\Gamma_k'-\bm\Gamma)\bm{W}_k,
\end{align*}
 it is known that either $\bm{T}_k'$ must depend on $(\bm\Gamma'_k-\bm\Gamma)\bm{W}_k$ or $\bm\Phi_k'\bm{V}_k'$ must depend on $(\bm\Gamma_k'-\bm\Gamma)\bm{W}_k$. However, the assumptions Fabian's theorem imposes on $\bm{T}_k'$, $\bm\Phi_k'$, and $\bm{V}_k'$ are typically incompatible with the term $k^{\upbeta/2}(\bm\Gamma'_k-\bm\Gamma)\bm{W}_k$. Say, for example, that we let $\bm\Phi_k'=\bm{I}$, $\bm{V}_k'=\bm{V}_k$, and $\bm{T}_k'=\bm{T}+k^{\upbeta/2}(\bm\Gamma'_k-\bm\Gamma)\bm{W}_k$.
 Here, having $\bm{T}_k'$ converge to some finite vector $\bm{T}'$ w.p.1 (as required by B3) would impose a priori conditions on the stochastic rate at which $\hat{\bm{\uptheta}}_k$ converges to $\bm\uptheta^\ast$. However, such a condition violates the very purpose of Fabian's theorem, which is to establish such a rate of convergence.
 Alternatively, having $\bm\Phi_k'\bm{V}_k'=\bm{V}_k+k^{(\upbeta-\upalpha)/2}(\bm\Gamma_k'-\bm\Gamma)\bm{W}_k$ does not lead to an appropriate definition of $\bm\Phi_k'$ and $\bm{V}_k'$
 given the restriction that $\bm{V}_k'$ must have mean zero conditionally on $\mathcal{F}_k$. 
By simply generalizing the theorem to relax the symmetry condition on $\bm\Gamma$ we avoid the need for imposing additional restrictions on $\hat{\bm{\uptheta}}_k$. Next we propose a generalization of Fabian's theorem that will allow us to show asymptotic normality for a special case of the GCSA algorithm.

Note that in (\ref{eq:lookatme}) we have
  $\bm\Gamma=\bm{A}\bm{H}(\bm\uptheta^\ast)$, where $\bm{A}$ is a diagonal matrix with $i$th diagonal entry equal to $a^{(1)}$ if $i\leq p'$ and with $i$th diagonal entry equal to $a^{(2)}$ if $i>p'$. While $\bm\Gamma=\bm{A}\bm{H}(\bm\uptheta^\ast)$ is not generally real symmetric and positive definite as required by Fabian's condition B1 (an exception being the case where $L(\bm\uptheta)$ is linearly separable in $\bm\uptheta^{(1)}$ and $\bm\uptheta^{(2)}$), it is entirely possible for this matrix to have strictly positive eigenvalues and be real-diagonalizable (i.e., $\bm{S}^{-1}\bm{AH(\bm\uptheta^\ast)S}=\bm\Lambda$ for a nonsingular real matrix $\bm{S}$ and a positive definite diagonal matrix $\bm\Lambda$); the following proposition formalizes this observation.

\begin{proposition}
\label{prop:sandpipercrossing}
Let $\bm{A}$ and $\bm{H}(\bm\uptheta^\ast)$ be real square matrices. Additionally, let $\bm{A}$ be a diagonal matrix with strictly positive diagonal entries. Finally, let  $\bm{H}(\bm\uptheta^\ast)$ be symmetric and positive definite (a common assumption in minimization problems). Then, there exist a nonsingular real matrix $\bm{S}$ and a positive definite diagonal matrix $\bm\Lambda$ such that $\bm{S}^{-1}\bm{AH(\bm\uptheta^\ast)S}=\bm\Lambda$.
\end{proposition}
\begin{proof}
Because $\bm{A}$ and $\bm{H}(\bm\uptheta^\ast)$ are both real and positive definite, Corollary 2.5.14 in Horn and Johnson (2010) implies there exist real, symmetric, positive definite matrices $\bm{A}^{1/2}$ and $\bm{H}(\bm\uptheta^\ast)^{1/2}$ such that $\bm{A}=\bm{A}^{1/2}\bm{A}^{1/2}$ and $\bm{H}(\bm\uptheta^\ast)=\bm{H}(\bm\uptheta^\ast)^{1/2}\bm{H}(\bm\uptheta^\ast)^{1/2}$. Therefore, 
$\bm{AH(\bm\uptheta^\ast)}=\bm{A}^{1/2}\bm{A}^{1/2}\bm{H}(\bm\uptheta^\ast)^{1/2}\bm{H}(\bm\uptheta^\ast)^{1/2}$. By multiplying the previous equation on the right by $\bm{I}=\bm{A}^{1/2}\bm{A}^{-1/2}$, we have:
\begin{align*}
\bm{AH(\bm\uptheta^\ast)}&=\bm{A}^{1/2}[\bm{H}(\bm\uptheta^\ast)^{1/2}\bm{A}^{1/2}]^\top[\bm{H}(\bm\uptheta^\ast)^{1/2}\bm{A}^{1/2}]\bm{A}^{-1/2}.
\end{align*}
Define $\bm{M}\equiv \bm{H}(\bm\uptheta^\ast)^{1/2}\bm{A}^{1/2}$. Then, $\bm{M}^\top\bm{M}$ is a positive definite real symmetric matrix. Thus,  Corollary 2.5.14 in Horn and Johnson (2010) once again implies that
we can write $\bm{M}=\bm{P}_0\bm\Lambda\bm{P}_0^{-1}$ for a nonsingular real orthogonal matrix $\bm{P}_0$ and a diagonal matrix, $\bm\Lambda$, with strictly positive eigenvalues. Then, $\bm{AH(\bm\uptheta^\ast)}= \bm{S\Lambda S}^{-1}$ with $\bm{S}=\bm{A}^{1/2}\bm{P}_0$. Note that while $\bm{S}$ is a real matrix, it is not necessarily an orthogonal matrix (i.e., $\bm{S}^{-1}$ is not necessarily equal to $\bm{S}^\top$), this follows from the fact that $\bm{S}^{-1}=\bm{P}_0^\top\bm{A}^{-1/2}$ and $\bm{S}^\top=\bm{P}_0^\top\bm{A}^{1/2}$ so that having $\bm{S}^{-1}=\bm{S}^\top$ would imply $\bm{A}^{-1/2}=\bm{A}^{1/2}$, which is only true of $\bm{A}=\bm{I}$.
\end{proof}

It follows from Proposition \ref{prop:sandpipercrossing} that generalizing Condition B1 in Fabian's theorem to allow $\bm\Gamma$ to be any real diagonalizable matrix with strictly positive eigenvalues (i.e., $\bm{S}^{-1}\bm{AH(\bm\uptheta^\ast)S}=\bm\Lambda$ with $\bm{S}$ and $\bm\Lambda$ defined as in Proposition \ref{prop:sandpipercrossing}) would allow for treatment of the cyclic seesaw SA algorithm. Furthermore, Section \ref{sec:upperpotomacshell} shows that the proposed generalization to Fabian's theorem would also allow for treatment of Algorithm \ref{kirkey} (deterministic pattern for coordinate selection), a special case of the GCSA algorithm (Algorithm \ref{findme}).

 In the following section we provide a generalization of Fabian's theorem which is slightly more general than the extension suggested by Proposition \ref{prop:sandpipercrossing}. The generalization only requires that $\bm\Gamma=\bm{SUS}^{-1}$ for a real non-singular matrix $\bm{S}$ and a real {{upper triangular}} matrix $\bm{U}$ with strictly positive diagonal entries (here $\bm\Gamma$ is said to be upper-triangularizable). Appendix \ref{sec:fabiansecgeneralize} discusses how the generalization to Fabian's theorem expands the theorem's applicability to include other SA algorithms (aside from Algorithm \ref{kirkey}) of practical interest.

\section{Generalizing Fabian's Theorem}
\label{sec:genfabiangenfabian}

This section contains
 a generalization to Fabian's theorem derived by replacing condition B1 with a slightly weaker assumption.  Specifically, for $\bm{W}_k$ and $\upbeta$ defined as in (\ref{eq:alien}) we show $k^{\upbeta/2}\bm{W}_k$ converges in distribution to a multivariate normal random variable under conditions B0, B2--B6, and a relaxed version of B1. Following the proof of Fabian's theorem (Fabian 1968), we begin by showing that $k^{\upbeta/2}\bm{W}_k$ is asymptotically normally distributed if and only if a much simpler process is also asymptotically normally distributed. After showing that the simpler process does, in fact, converge in distribution to a multivariate normal random variable, the parameters of its asymptotic distribution will uniquely determine the parameters of the asymptotic distribution of $k^{\upbeta/2}\bm{W}_k$. Next we introduce some notation and the generalized version of condition B1.

  Throughout this Chapter $M_{ij}$ and $M_{k(i,j)}$ denote the $(i,j)$th entries of the matrices $\bm{M}$ and $\bm{M}_k$, respectively, and ${v}_i$ and $v_{k(i)}$ denote the $i$th entries of the vectors $\bm{v}$ and $\bm{v}_k$, respectively. Furthermore, $k\geq 1$ denotes a strictly positive integer; $\xrightarrow{\text{\ dist\ }}$ means convergence in distribution; $\bm{V}_k$, $\bm{W}_k$, $\bm{T}_k$, and $\bm{T}$ are vectors in $\mathbb{R}^p$; $\bm\Gamma_k$, $\bm\Phi_k$, $\bm\Sigma$, $\bm\Gamma$, $\bm\Phi$, and $\bm{P}$ are matrices in $\mathbb{R}^{p\times p}$; and $\mathcal{N}(\bm\upmu,\bm{M})$ denotes a multivariate normal random variable with mean $\bm\upmu$ and covariance $\bm{M}$. Furthermore, throughout this section we assume the recursion for ${\bm{W}}_k$ given in (\ref{eq:alien}) satisfies the following conditions:
  \begin{DESCRIPTION}
\label{apge:definematrixindices}
\item[B1$'$] There exists an {\it{upper triangular}} matrix $\bm{U}\in \mathbb{R}^{p\times p}$ with strictly positive eigenvalues and a nonsingular $\bm{S}\in \mathbb{R}^{p\times p}$ such that $\bm\Gamma_k\rightarrow \bm\Gamma=\bm{S}\bm{U}\bm{S}^{-1}$ w.p.1.
\item[B6$'$] Define $\uplambda\equiv \min_i{\{U_{ii}\}}$. Let $\upalpha$ and $\upbeta$ be constants such that $0<\upalpha\leq 1$ and $0\leq \upbeta$. Define $\upbeta_+\equiv\upbeta$ if $\upalpha=1$ and $\upbeta_+\equiv 0$ otherwise. Then, $\upbeta_+<2\uplambda$.
\item[B0$'$ \& B2$'$--B5$'$]  The same as B0 and B2--B5, respectively.
\end{DESCRIPTION}
Conditions B0$'$ and B2$'$--B6$'$ are identical to the conditions of Fabian's theorem (note that B6$'$ is obtained by rewriting B6 in the notation of B1$'$). On the other hand, B1$'$ is the relaxed version of Fabian's  corresponding condition (condition B1 in Theorem \ref{thm:fabian}) requiring symmetry of $\bm\Gamma$. As discussed in the comment following Proposition \ref{prop:sandpipercrossing}, any real square matrix with real eigenvalues satisfies B1$'$. The following Theorem is our generalization to Fabian's theorem based on conditions B1$'$--B6$'$.

  \begin{theorem}[A Generalization of Fabian's Theorem]
  \label{thm:generalizefabian} 
  Assume the recursion for ${\bm{W}}_k$ given in (\ref{eq:alien}) satisfies B0$'$--B6$'$. Then, the asymptotic distribution of $k^{\upbeta/2}\bm{W}_k$ is a multivariate normal random variable with mean $\bm{S}\bm\upnu$ and covariance matrix $\bm{SQS}^\top$, where the entries of $\bm\upnu$ are the unique solution to:
\begin{align}
 \label{eq:okayfine}
 \upnu_i\equiv (\tilde{{U}}_{ii})^{-1}\tilde{{T}}_i-\Bigg[(\tilde{{U}}_{ii})^{-1}\sum_{j=i+1}^p \tilde{U}_{ij}\upnu_j\Bigg]
 \end{align}
with  $\tilde{\bm{T}}\equiv \bm{S}^{-1}\bm{T}$ and $\tilde{\bm{U}}\equiv \bm{U}-(\upbeta_+/2)\bm{I}$; and the entries of $\bm{Q}$ are the unique solution to:
\begin{align}
\label{eq:nolookback}
Q_{ij}=\frac{[\bm{S}^{-1}\bm{\Phi\Sigma\Phi}^\top(\bm{S}^{-1})^\top]_{ij}}{\tilde{U}_{ii}+\tilde{U}_{jj}}-\Bigg[\frac{\sum_{\ell=j+1}^{p} \tilde{U}_{j\ell}\ Q_{i\ell}+\sum_{\ell=i+1}^{p} \tilde{U}_{i\ell}\ Q_{\ell j}}{\tilde{U}_{ii}+\tilde{U}_{jj}}\Bigg].
\end{align}
 Using (\ref{eq:nolookback}), $Q_{ij}$ is a function of the elements of the set $\{Q_{mn}\}_{(m,n)\in \mathcal{G}}$ where $\mathcal{G}$ is the set of $(m,n)$ tuples such that either $m\geq i+1$ and $n\geq j+1$, $m=i$ and $n\geq j+1$, or $m\geq i+1$ and $n=j$. Therefore, the entries of $\bm{Q}$ can be computed sequentially beginning with $Q_{pp}$, for which (\ref{eq:nolookback}) gives a solution. Similarly, the entries of $\bm\upnu$ can be computed beginning with $\upnu_p$, for which (\ref{eq:okayfine}) gives a solution.
\end{theorem}

\begin{proof}
(Although the following proof is complete, a more detailed proof of the theorem can be found in Appendix \ref{sec:fabiansecgeneralize}). In order to compute the asymptotic distribution of $k^{\upbeta/2}\bm{W}_k$ we begin by constructing the following process:
\begin{align*}
\tilde{\bm{W}}_{k}=(k-1)^{\upbeta/2}\bm{S}^{-1}\bm{W}_k,
\end{align*}
where $\bm{S}$ is the matrix from B1$'$. Here, using Slutsky's theorem it follows that if $\tilde{\bm{W}}_k\xrightarrow{\text{\ dist\ }}\mathcal{N}(\bm\upmu,\bm{M})$,
 then $k^{\upbeta/2}\bm{W}_k\xrightarrow{\text{\ dist\ }}\mathcal{N}(\bm{S}\bm\upmu,\bm{SM S}^\top)$.
   Thus, proving that the process $\tilde{\bm{W}}_k$ is asymptotically normally distributed (with certain mean vector and covariance matrix) is sufficient for computing the asymptotic distribution of $k^{\upbeta/2}\bm{W}_k$. Moreover, after some algebraic manipulation it can be shown that:
\begin{align}
\label{eq:graham3884}
 \tilde{\bm{W}}_{k+1}=&\ (\bm{I}-k^{-\upalpha}\tilde{\bm{\Gamma}}_k)\tilde{\bm{W}}_k+\frac{\bm{S}^{-1}\bm{T}_k}{k^\upalpha}+\frac{\bm{S}^{-1}\bm\Phi_k\bm{V}_k}{k^{\upalpha/2}},
\end{align}
 where $\tilde{\bm{W}}_{1}=\bm{0}$, 
\begin{align*}
\tilde{\bm\Gamma}_k\equiv \left(\frac{k}{k-1}\right)^{\upbeta/2}\bm{S}^{-1}\bm\Gamma_k\bm{S}-\left[\left(\frac{k}{k-1}\right)^{\upbeta/2}-1\right]k^\upalpha\bm{I},
\end{align*}
and $\tilde{\bm\Gamma}_k\rightarrow \tilde{\bm{U}}= \bm{U}-(\upbeta_+/2)\bm{I}$ w.p.1 so that $\tilde{\bm{W}}_k$ is a special case of (\ref{eq:alien}). Using the same arguments as those in Fabian (1968, proof of Theorem 2.2) it can be shown that replacing ${\bm{T}}_k$ with ${\bm{T}}$ and $\bm\Phi_k$ with $\bm\Phi$  in (\ref{eq:graham3884}) does not change the asymptotic distribution of $\tilde{\bm{W}}_k$ (this result is not immediate due to the recursive nature of $ \tilde{\bm{W}}_{k}$). Therefore, in order to show that $\tilde{\bm{W}}_k$ is asymptotically normally distributed we may assume, without loss of generality (w.l.o.g.), that ${\bm{T}}_k={\bm{T}}$ and $\bm\Phi_k=\bm\Phi$ so that:
\begin{align}
\label{eq:jaredboothsssd}
 \tilde{\bm{W}}_{k+1}=(\bm{I}-k^{-\upalpha}\tilde{\bm{\Gamma}}_k)\tilde{\bm{W}}_k+k^{-\upalpha}\tilde{\bm{T}}+k^{-\upalpha/2}\tilde{\bm{V}}_k,
\end{align}
where $\tilde{\bm{T}}= \bm{S}^{-1}\bm{T}$ and $\tilde{\bm{V}}_k\equiv \bm{S}^{-1}\bm{\Phi V}_k$. Next, we relate the asymptotic distribution of $\tilde{\bm{W}}_k$, as described by (\ref{eq:jaredboothsssd}), to that of an even simpler process. 

Consider the process:
\begin{align}
\label{eq:norahjoness772}
 \tilde{\bm{W}}_{k+1}'=(\bm{I}-k^{-\upalpha}\tilde{\bm{\Gamma}}_k)\tilde{\bm{W}}_k'+k^{-\upalpha/2}\tilde{\bm{V}}_k,
\end{align}
obtained by removing the term $k^{-\upalpha}\tilde{\bm{T}}$ from (\ref{eq:jaredboothsssd}). Lemma 4.2 (Fabian 1967) implies $\tilde{\bm{W}}_k-\tilde{\bm{W}}_k'\rightarrow\bm\upnu$ w.p.1, where the entries of $\bm\upnu$ are given in (\ref{eq:okayfine}).
The significance of this observation is that if $\tilde{\bm{W}}_{k}'\xrightarrow{\text{\ dist\ }}\mathcal{N}(\bm\upmu,\bm{M})$,
 then $\tilde{\bm{W}}_{k}\xrightarrow{\text{\ dist\ }}\mathcal{N}(\bm\upmu+\bm\upnu,\bm{M})$. 
%
 At this point, the same arguments as those in Fabian (1968, proof of Theorem 2.2) can be used to show the following two results:
 \begin{enumerate}
 \item Replacing $\tilde{\bm{\Gamma}}_k$ with $\tilde{\bm{U}}$ in (\ref{eq:norahjoness772}) does not change the asymptotic distribution of $\tilde{\bm{W}}_k'$ which depends on the limit (w.p.1) of $\tilde{\bm{\Gamma}}_k$ but not on $\tilde{\bm{\Gamma}}_k$ itself. Therefore, w.l.o.g. we may assume that $\tilde{\bm{\Gamma}}_k=\tilde{\bm{U}}$ so that:
 \begin{align}
\label{eq:utildenew}
 \tilde{\bm{W}}_{k+1}'=(\bm{I}-k^{-\upalpha}\tilde{\bm{U}})\tilde{\bm{W}}_k'+k^{-\upalpha/2}\tilde{\bm{V}}_k.
\end{align}

 \item The characteristic function of the asymptotic distribution of $\tilde{\bm{W}}_{k}'$ evaluated at $\bm{t}\in \mathbb{R}^p$ depends on $\bm{t}$, $\upalpha$, $\tilde{\bm{U}}$, and on $\tilde{\bm{\Sigma}}\equiv \lim_{k\rightarrow \infty}\text{cov}[\tilde{\bm{V}}_k|\mathcal{F}_k]$ but is independent of other aspects of the distribution of $\tilde{\bm{V}}_k$ (provided conditions B0$'$--B6$'$ hold).  Therefore, we may assume w.l.o.g. that the vectors $\tilde{\bm{V}}_k$ are i.i.d. $\mathcal{N}(\bm{0},\tilde{\bm{\Sigma}})$.
 \end{enumerate}
 Next we derive the asymptotic distribution of the process $\tilde{\bm{W}}_k'$ in (\ref{eq:utildenew}).

First, note that $\tilde{\bm{\Sigma}}=\bm{S}^{-1}\bm{\Phi\Sigma\Phi}^\top(\bm{S}^{-1})^\top$ by condition B4$'$. Next, by point 2 in the previous paragraph we may assume w.l.o.g. that the vectors $\tilde{\bm{V}}_k$ are i.i.d. $\mathcal{N}(\bm{0},\tilde{\bm{\Sigma}})$. Consequently, the distribution of $\tilde{\bm{W}}_k'$ from (\ref{eq:utildenew}) must be a multivariate normal random variable with mean zero (since $\tilde{\bm{W}}_1'=\bm{0}$) and covariance $\bm{Q}_k\equiv E[\tilde{\bm{W}}_{k}'(\tilde{\bm{W}}_{k}')^\top]$
which satisfies:
\begin{align*}
Q_{k+1(ij)}=&\ \left(1-k^{-\upalpha}\left[\tilde{U}_{ii}+\tilde{U}_{jj}-k^{-\upalpha} \tilde{U}_{ii}\hspace{.2em} \tilde{U}_{jj}\right]\right)Q_{k(ij)}\notag\\
&-k^{-\upalpha}\left[\sum_{\ell=j+1}^{p} \tilde{U}_{j\ell}\hspace{.2em} Q_{k(i\ell)}+\sum_{\ell=i+1}^{p} \tilde{U}_{i\ell}\hspace{.2em} Q_{k(\ell j)}\right]+k^{-\upalpha}\tilde{\Sigma}_{ij}\notag\\
&+k^{-2\upalpha}\sum_{\ell=i}^p \tilde{U}_{i\ell}\sum_{s=j+1}^p\tilde{U}_{js}\hspace{.2em} Q_{\ell s}+k^{-2\upalpha}\sum_{\ell=i+1}^p \tilde{U}_{i\ell}\hspace{.2em} \tilde{U}_{jj}\hspace{.2em} Q_{\ell j}.
\end{align*}
 Here, Lemma 4.2 in Fabian (1967) implies $\bm{Q}\equiv \lim_{k\rightarrow \infty}\bm{Q}_k$ is a matrix whose entries are the unique solution to (\ref{eq:nolookback}).
Therefore, $\tilde{\bm{W}}_k'\xrightarrow{\text{\ dist\ }}\mathcal{N}(\bm{0},\bm{Q})$ which implies ${\bm{W}}_{k}\xrightarrow{\text{\ dist\ }}\mathcal{N}(\bm{S}\bm\upnu,\bm{S}\bm{Q}\bm{S}^\top)$.
\end{proof}

Note that if $\bm\Gamma$ is real symmetric then the terms in square brackets in (\ref{eq:okayfine}) and (\ref{eq:nolookback}) disappear since $\tilde{\bm{U}}$ may be taken to be a diagonal matrix. This implies that when $\bm\Gamma$ is a real and symmetric matrix (as in Fabian's theorem) then:
\begin{align}
\label{eq:fishytailshastrain}
\bm{S}\bm\upnu=\bm{S}\tilde{\bm{U}}^{-1}\tilde{\bm{T}}=\bm{S}[\bm{U}-(\upbeta_+/2)\bm{I}]^{-1}\bm{S}^{-1}\bm{T}=(\bm\Gamma-(\upbeta_+/2)\bm{I})^{-1}\bm{T}.
\end{align}
The mean vector in (\ref{eq:fishytailshastrain}) is the same as the mean of the asymptotic distribution in Fabian's theorem (Theorem \ref{thm:fabian}). Similarly, when $\bm{U}$ is a diagonal matrix, we can assume $\bm{S}=\bm{P}$ for some orthogonal matrix $\bm{P}$ and
 (\ref{eq:nolookback}) reduces to $Q_{ij}=(\bm{P}^\top\bm{\Phi}\bm\Sigma\bm\Phi^\top\bm{P})_{ij}(U_{ii}+U_{jj}-\upbeta_+)^{-1}$ as in Fabian's theorem.

\section{Asymptotic Normality of Algorithm \ref{kirkey}}
\label{sec:upperpotomacshell}
This section gives a set of conditions for the asymptotic normality of the normalized iterates from Algorithm \ref{kirkey} (a special case of GCSA).
%
%
First, following a derivation similar to that of (\ref{eq:buckwild}), Algorithm \ref{kirkey} may be written as:
\begin{align}
 \hat{\bm{\uptheta}}_{k+1}=&\ \hat{\bm\uptheta}_k-\sum_{z=1}^{s}\sum_{\ell=0}^{n(z)-1} {a}_{k}^{(j(z))} \bm{J}^{(j(z))}(\overline{\bm\uptheta}_k)(\hat{\bm{\uptheta}}_k-\bm\uptheta^\ast)\notag\\
 &+\sum_{z=1}^{s}\sum_{\ell=0}^{n(z)-1} {a}_{k}^{(j(z))} \bm{J}^{(j(z))}\left(\overline{\bm\uptheta}_k^{(I_{z,\ell})}\right)[\bm\Delta_{\bm{g}}^{(I_{z,\ell})}+\bm\Delta_{\bm{\upbeta}}^{(I_{z,\ell})}+\bm\Delta_{\bm{\upxi}}^{(I_{z,\ell})}]\notag\\
 &-\sum_{z=1}^{s}\sum_{\ell=0}^{n_(z)-1} {a}_{k}^{(j(z))} {\bm{\upbeta}}_k^{(j(z))}\left(\hat{\bm\uptheta}_k^{\left(I_{z,\ell}\right)}\right)\notag\\
 &-\sum_{z=1}^{s}\sum_{\ell=0}^{n(z)-1} {a}_{k}^{(j(z))} {\bm{\upxi}}_k^{(j(z))}\left(\hat{\bm\uptheta}_k^{\left(I_{z,\ell}\right)}\right),\label{eq:eleafanimalstrees}
\end{align}
where $\hat{\bm{\uptheta}}_k$ denotes an iterate of Algorithm \ref{kirkey} and
\begin{subequations}
\begin{align}
 \bm\Delta_{\bm{g}}^{(I_{z,\ell})}\equiv&\ \sum_{m=1}^{z-1}\sum_{i=0}^{n(m)-1} \left[ {a}_{k}^{(j(m))} {\bm{g}}^{(j(m))}\left(\hat{\bm\uptheta}_k^{(I_{m,i})}\right)\right]+\sum_{i=0}^{\ell-1} \left[ {a}_{k}^{(j(z))}{\bm{g}}^{(j(z))}\left(\hat{\bm\uptheta}_k^{(I_{z,i})}\right)\right],\label{eq:futuramasohs0}\\
 \bm\Delta_{\bm\upbeta}^{(I_{z,\ell})}\equiv&\ \sum_{m=1}^{z-1}\sum_{i=0}^{n_k(m)-1} \left[ {a}_{k}^{(j(m))} {\bm{\upbeta}}_k^{(j(m))}\left(\hat{\bm\uptheta}_k^{(I_{m,i})}\right)\right]+\sum_{i=0}^{\ell-1} \left[ {a}_{k}^{(j(z))} {\bm{\upbeta}}_k^{(j(z))}\left(\hat{\bm\uptheta}_k^{(I_{z,i})}\right)\right],\label{eq:futuramasohs1}\\
 \bm\Delta_{\bm\upxi}^{(I_{z,\ell})}\equiv&\ \sum_{m=1}^{z-1}\sum_{i=0}^{n_k(m)-1} \left[ {a}_{k}^{(j(m))} {\bm{\upxi}}_k^{(j(m))}\left(\hat{\bm\uptheta}_k^{(I_{m,i})}\right)\right]+\sum_{i=0}^{\ell-1} \left[ {a}_{k}^{(j(z))} {\bm{\upxi}}_k^{(j(z))}\left(\hat{\bm\uptheta}_k^{(I_{z,i})}\right)\right].\label{eq:futuramasohs}
\end{align}
\end{subequations}
Then, in a manner analogous to (\ref{eq:shmee0}--c) we let $\upalpha>0$, $\upbeta\geq 0$, $\bm{W}_k=\hat{\bm{\uptheta}}_k=\bm\uptheta^\ast$, and let $\mathcal{F}_k$ be the sigma field generated by $\{\hat{\bm{\uptheta}}_\ell\}_{\ell=0}^k$ as well as by any random variables generated  by the algorithm in the production of $\hat{\bm{\uptheta}}_k$. Additionally,
let 
\begin{align}
\bm{T}_k=& \ - k^{\upalpha+\upbeta/2}\sum_{z=1}^{s}\sum_{\ell=0}^{n(z)-1} {a}_{k}^{(j(z))} {\bm{\upbeta}}_k^{(j(z))}\left(\hat{\bm\uptheta}_k^{\left(I_{z,\ell}\right)}\right)\notag\\
&+k^{\upalpha+\upbeta/2}\sum_{z=1}^{s}\sum_{\ell=0}^{n(z)-1} {a}_{k}^{(j(z))} \bm{J}^{(j(z))}\left(\overline{\bm\uptheta}_k^{(I_{z,\ell})}\right)[\bm\Delta_{\bm{g}}^{(I_{z,\ell})}+\bm\Delta_{\bm{\upbeta}}^{(I_{z,\ell})}]\notag\\
&+k^{\upalpha+\upbeta/2}\sum_{z=1}^{s}\sum_{\ell=0}^{n(z)-1} {a}_{k}^{(j(z))} E\Big[\bm{J}^{(j(z))}\left(\overline{\bm\uptheta}_k^{(I_{z,\ell})}\right) \bm\Delta_{\bm\upxi}^{(I_{z,\ell})}\Big|\mathcal{F}_k\Big],
\label{eq:goingcrazyhim}
\end{align}
let $\bm\Phi_k=\bm{I}$, let $\bm{V}_k$ be given by
\begin{align}
\bm{V}_k=&\ -k^{(\upalpha+\upbeta)/2}\sum_{z=1}^{s}\sum_{\ell=0}^{n(z)-1} {a}_{k}^{(j(z))} {\bm{\upxi}}_k^{(j(z))}\left(\hat{\bm\uptheta}_k^{\left(I_{z,\ell}\right)}\right)\notag\\
&+k^{(\upalpha+\upbeta)/2}\sum_{z=1}^{s}\sum_{\ell=0}^{n(z)-1} {a}_{k}^{(j(z))} \bm{J}^{(j(z))}\left(\overline{\bm\uptheta}_k^{(I_{z,\ell})}\right) \bm\Delta_{\bm\upxi}^{(I_{z,\ell})}\notag\\
&-k^{(\upalpha+\upbeta)/2}\sum_{z=1}^{s}\sum_{\ell=0}^{n(z)-1} {a}_{k}^{(j(z))} E\Big[\bm{J}^{(j(z))}\left(\overline{\bm\uptheta}_k^{(I_{z,\ell})}\right) \bm\Delta_{\bm\upxi}^{(I_{z,\ell})}\Big|\mathcal{F}_k\Big],
\label{eq:othingcoodsd}
\end{align}
and let
\begin{align}
\bm\Gamma_k=k^\upalpha\sum_{z=1}^{s}\sum_{\ell=0}^{n(z)-1} {a}_{k}^{(j(z))} \bm{J}^{(j(z))}(\overline{\bm\uptheta}_k).
\label{eq:pantyponcho}
\end{align}
With this notation, the recursion in (\ref{eq:eleafanimalstrees}) defining an iteration of Algorithm \ref{kirkey} can be written in the form of (\ref{eq:alien}), the recursion of Theorem \ref{thm:generalizefabian}. 

Because (\ref{eq:eleafanimalstrees}) is a special case of (\ref{eq:alien}), conditions B0$'$--B6$'$ would be sufficient for the asymptotic normality of $k^{\upbeta/2}(\hat{\bm{\uptheta}}_k-\bm\uptheta^\ast)$ when $\hat{\bm{\uptheta}}_k$ denotes an iterate of Algorithm \ref{kirkey}. In practice, however, the validity of conditions B0$'$--B6$'$ (or of the slightly stronger conditions B0--B6) can be difficult to verify for {\it{any}} algorithm which may be written in the form of (\ref{eq:alien}). 
%
%
In the case of Algorithm \ref{kirkey} this is complicated by the fact that the terms $\bm\Gamma_k$, $\bm{T}_k$, $\bm{\Phi}_k$, and $\bm{V}_k$ have more complex forms than the corresponding variables for classical SA algorithms (compare the definitions below (\ref{eq:crownuni}) to (\ref{eq:goingcrazyhim})--(\ref{eq:pantyponcho})). 
%
%
 In an effort to obtain a set of conditions that are easier to understand than B0$'$--B6$'$, this section derives a set of conditions which imply B0$'$--B6$'$ hold for Algorithm \ref{kirkey}. To do this, we first present a useful lemma related to condition B5$'$.
%

\begin{lemma}
\label{lem:piccardo}
Let $\bm{V}_k$ be a vector-valued sequence with $\bm{V}_k= \sum_{i=1}^N\bm{\Psi}_k(i)$, where $N$ is an integer such that $2\leq N<\infty$ and $\bm{\Psi}_k(i)$ denotes a vector-valued random variable.
Additionally,
%
%
 for some $\upalpha>0$ and for every $r>0$ let the values $\upsigma_{k,r}^2(i)\equiv E\chi \{\|\bm{\Psi}_k(i)\|^2\geq rk^\upalpha\}\|\bm{\Psi}_k(i)\|^2$ satisfy
$\lim_{k\rightarrow \infty}\upsigma_{k,r}^2(i)=0$ for all $i$.
Then, the sequence $\upsigma_{k,r}^2\equiv E\chi\{\|\bm{V}_k\|^2\geq rk^\upalpha\}\|\bm{V}_k\|^2$ also satisfies $\lim_{k\rightarrow \infty}\upsigma_{k,r}^2=0$.
\end{lemma}
\begin{proof}
For simplicity, we assume $N=2$ so that $\bm{V}_k=\bm\Psi_k(1)+\bm\Psi_k(2)$ (the proof for $N>2$ is essentially identical).
First,
%
%
 using the triangle inequality we know $\|\bm{V}_k\|\leq \|\bm{\Psi}_k(1)\|+\|\bm{\Psi}_k(2)\|$. Therefore, in order to have $\|\bm{V}_k\|\geq \sqrt{rk^\upalpha}$ at least one of $\|\bm{\Psi}_k(1)\|$ or $\|\bm{\Psi}_k(2)\|$ must be greater or equal to $\sqrt{rk^\upalpha}/2$, which implies:
 \begin{align*}
 \chi\{\|\bm{V}_k\|\geq \sqrt{rk^\upalpha}\}\|\bm{V}_k\|\leq\sum_{i=1}^2 2\chi\{\|\bm{\Psi}_k(i)\|\geq \sqrt{rk^\upalpha}/2\}\|\bm{\Psi}_k(i)\|
 \end{align*}
(recall that $\chi\{\mathcal{E}\}$ is the indicator function of the event $\mathcal{E}$). Equivalently,
  \begin{align}
  \label{eq:masstwilighteffext}
 \chi\{\|\bm{V}_k\|^2\geq {rk^\upalpha}\}\|\bm{V}_k\|\leq\sum_{i=1}^2 2\chi\{\|\bm{\Psi}_k(i)\|^2\geq {rk^\upalpha}/4\}\|\bm{\Psi}_k(i)\|.
 \end{align}
 Taking the square on both sides of (\ref{eq:masstwilighteffext}) gives:
   \begin{align}
& \chi\{\|\bm{V}_k\|^2\geq {rk^\upalpha}\}\|\bm{V}_k\|^2\leq \sum_{i=1}^2 4\chi\{\|\bm{\Psi}_k(i)\|^2\geq {rk^\upalpha}/4\}\|\bm{\Psi}_k(i)\|^2\notag\\
 &+8\chi\{\|\bm{\Psi}_k(1)\|^2\geq {rk^\upalpha}/4\}\chi\{\|\bm{\Psi}_k(2)\|^2\geq {rk^\upalpha}/4\}\|\bm{\Psi}_k(1)\|\|\bm{\Psi}_k(2)\|.
 \label{eq:ffrustratingg}
 \end{align}
Taking the expectation on both sides of (\ref{eq:ffrustratingg}) and using the definition of $\upsigma_{k,r}^2$:
 \begin{align}
&\upsigma_{k,r}^2\leq  \sum_{i=1}^2 4E\Big[\chi\{\|\bm{\Psi}_k(i)\|^2\geq {rk^\upalpha}/4\}\|\bm{\Psi}_k(i)\|^2\Big]\notag\\
&+8E\Big[\chi\{\|\bm{\Psi}_k(1)\|^2\geq {rk^\upalpha}/4\}\chi\{\|\bm{\Psi}_k(2)\|^2\geq {rk^\upalpha}/4\}\|\bm{\Psi}_k(1)\|\|\bm{\Psi}_k(2)\|\Big].
\label{eq:kriganwreckgerudo}
\end{align}
By assumption, however, $\lim_{k\rightarrow \infty}E\left[\chi\{\|\bm{\Psi}_k(i)\|^2\geq {rk^\upalpha}/4\}\|\bm{\Psi}_k(i)\|^2\right]= 0$. Therefore, (\ref{eq:kriganwreckgerudo}) becomes:
  \begin{align}
  \label{eq:cawcawsquak}
\upsigma_{k,r}^2\leq o(1)+8E\Big[\chi\{\|\bm{\Psi}_k(1)\|^2\geq {rk^\upalpha}/4\}\chi\{\|\bm{\Psi}_k(2)\|^2\geq {rk^\upalpha}/4\}\|\bm{\Psi}_k(1)\|\|\bm{\Psi}_k(2)\|\Big],
\end{align}
where $o(\cdot)$ denotes the standard little-$o$ notation. Next, using the Cauchy--Schwarz inequality:
\begin{align}
&E\Big[\chi\{\|\bm{\Psi}_k(1)\|^2\geq {rk^\upalpha}/4\}\chi\{\|\bm{\Psi}_k(2)\|^2\geq {rk^\upalpha}/4\}\|\bm{\Psi}_k(1)\|\|\bm{\Psi}_k(2)\|\Big]\leq\notag\\
&\left(E\Big[\chi\{\|\bm{\Psi}_k(1)\|^2\geq {rk^\upalpha}/4\}\|\bm{\Psi}_k(1)\|^2\Big]\right)^{1/2}\left(E\Big[\chi\{\|\bm{\Psi}_k(2)\|^2\geq {rk^\upalpha}/4\}\|\bm{\Psi}_k(2)\|^2\Big]\right)^{1/2}.
\label{eq:onceagainherewegooooo}
\end{align}
Once again, since $\lim_{k\rightarrow \infty}E\left[\chi\{\|\bm{\Psi}_k(i)\|^2\geq {rk^\upalpha}/4\}\|\bm{\Psi}_k(i)\|^2\right]= 0$, combining (\ref{eq:cawcawsquak}) and (\ref{eq:onceagainherewegooooo}) implies $\upsigma_{k,r}^2=o(1)$, as desired.
\end{proof}

The following theorem gives conditions for the asymptotic normality of the normalized iterates from Algorithm \ref{kirkey}.

\begin{theorem}
\label{thm:fnogg}
Let $\hat{\bm{\uptheta}}_k$ be generated according to Algorithm \ref{kirkey}. Additionally, let $\mathcal{F}_k$ denote the sigma field generated by $\{\hat{\bm{\uptheta}}_\ell\}_{\ell=0}^k$ as well as by any random variables generated  by the algorithm in the production of $\hat{\bm{\uptheta}}_k$. Let $\bm{W}_k=\hat{\bm{\uptheta}}_k-\bm\uptheta^\ast$, let $\upbeta$, $\upalpha$, $\bm\Gamma_k$, $\bm{T}_k$, $\bm\Phi_k$, and $\bm{V}_k$ be defined as in (\ref{eq:goingcrazyhim})--(\ref{eq:pantyponcho}) and assume these variables have real-valued entries. Additionally, assume the following conditions hold:
\begin{DESCRIPTION}
\item[C0] $L(\bm\uptheta)$ is twice continuously differentiable, $\bm{H}(\bm\uptheta)$ (the Hessian of $L(\bm\uptheta)$) has bounded entries (i.e., $\bm{g}(\bm\uptheta)$ is Lipschitz continuous), and $\hat{\bm\uptheta}_k^{(I_{z,i})}\rightarrow \bm\uptheta^\ast$ w.p.1 uniformly over $z$ and $i$.
\item[C1] For each $j$,  $a_k^{(j)}>0$ and $k^\upalpha {a}_k^{(j)}\rightarrow r_j<\infty$.
\item[C2] There exists  an upper triangular matrix $\bm{U}\in \mathbb{R}^{p\times p}$ with strictly positive eigenvalues and a nonsingular matrix $\bm{S}\in \mathbb{R}^{p\times p}$ such that
\begin{align}
\label{eq:macnealhuh}
\bm\Gamma\equiv \sum_{z=1}^{s}\sum_{\ell=0}^{n(z)-1} r_{j(z)} \bm{J}^{(j(z))}(\bm\uptheta^\ast)=\bm{SUS}^{-1},
\end{align}
where $\bm{J}^{(j)}(\bm\uptheta)$ denotes the Jacobian of $\bm{g}^{(j)}(\bm\uptheta)$.
\item[C3] $\upbeta$ and $\upalpha$ satisfy $0<\upbeta\leq \upalpha$. Additionally, for all $j$ there exist a finite vector $\bm{b}^{(j)}\in \mathbb{R}^p$ such that $k^{\upbeta/2}{\bm{\upbeta}}_k^{(j(z))}\big(\hat{\bm\uptheta}_k^{(I_{z,i})}\big)\rightarrow \bm{b}^{(j(z))}$ w.p.1 uniformly over $z$ and $i$.
\item[C4] For all $z$ and $i$,
  $E\left[{\bm{\upxi}}_k^{(j(z))}\big(\hat{\bm\uptheta}_k^{(I_{z,i})}\big)\Big|\hat{\bm\uptheta}_k^{(I_{z,i})}\right]=\bm{0}$.
\item[C5]  For all $z$ there exists a constant $C>0$ and a matrix $\bm\Sigma^{(j)}$ such that one of the following holds w.p.1 (convergence is meant uniformly over $i$):
\begin{enumerate}[label=(\roman*)]
\item $\upbeta=\upalpha$, $C>\|E[{\bm{\upxi}}_k^{(j(z))}\big(\hat{\bm\uptheta}_k^{(I_{z,i})}\big)[{\bm{\upxi}}_k^{(j(z))}\big(\hat{\bm\uptheta}_k^{(I_{z,i})}\big)]^\top|\mathcal{F}_k]-\bm\Sigma^{(j(z))}\|\rightarrow 0$.
\item $\upbeta<\upalpha$, $C>\|k^{\upbeta-\upalpha}E[{\bm{\upxi}}_k^{(j(z))}\big(\hat{\bm\uptheta}_k^{(I_{z,i})}\big)[{\bm{\upxi}}_k^{(j(z))}\big(\hat{\bm\uptheta}_k^{(I_{z,i})}\big)]^\top|\mathcal{F}_k]-\bm\Sigma^{(j(z))}\|\rightarrow 0$.
\end{enumerate}
Additionally, if $\bm{v}_1$ and $\bm{v}_2$ are two different elements of $\{{\bm{\upxi}}_k^{(j(z))}\big(\hat{\bm\uptheta}_k^{(I_{z,i})}\big)\}_{z,i}$, then one of the following holds w.p.1:
\begin{enumerate}[label=(\roman*)]
\setcounter{enumi}{2}
\item C5-(i) holds and $C>\|E[\bm{v}_1\bm{v}_2^\top|\mathcal{F}_k]\|\rightarrow 0$.
\item C5-(ii) holds and $C>\|k^{\upbeta-\upalpha}E[\bm{v}_1\bm{v}_2^\top|\mathcal{F}_k]\|\rightarrow 0$.
\end{enumerate}
%
%
\item[C6] $\bm\Sigma\equiv \sum_{z=1}^{s}\sum_{\ell=0}^{n(z)-1} r_{j(z)}^2 {\bm\Sigma}^{(j(z))}$ is a positive definite matrix.
\item[C7] For each $z=1,\dots, s$ and $i=1,\dots, n(z)$ one of the following holds:
\begin{enumerate}[label=(\roman*)]
\item$\upbeta=\upalpha$ and B5-(i) holds with $\bm{V}_k$ replaced by ${\bm{\upxi}}_k^{(j(z))}\big(\hat{\bm\uptheta}_k^{(I_{z,i})}\big)$.
\item $\upbeta<\upalpha$ and B5-(i) holds with $\bm{V}_k$ replaced by $k^{(\upbeta-\upalpha)/2}{\bm{\upxi}}_k^{(j(z))}\big(\hat{\bm\uptheta}_k^{(I_{z,i})}\big)$.
\end{enumerate}
\item[C8] The same as B6$'$.
\end{DESCRIPTION}
Then, 
the asymptotic distribution of $k^{\upbeta/2}\bm{W}_k$ is a multivariate normal random variable with mean $\bm{S}\bm\upnu$ and covariance matrix $\bm{SQS}^\top$, where the entries of $\bm\upnu$ are the unique solution to (\ref{eq:okayfine})
with $\tilde{\bm{T}}=-\bm{S}^{-1}\left[\sum_{z=1}^{s}\sum_{\ell=0}^{n_(z)-1}r_{j(z)}\bm{b}^{(j(z))}\right]$ and $\tilde{\bm{U}}=\bm{U}-(\upbeta_+/2)\bm{I}$, and the entries of $\bm{Q}$ are the unique solution to (\ref{eq:nolookback}) with $\tilde{\bm{U}}$ defined as above and $\bm\Phi=\bm{I}$.
\end{theorem}

\begin{proof}
We show that conditions B0$'$--B6$'$ hold. The result then follows from Theorem \ref{thm:generalizefabian}. First, note that condition B0$'$ holds by the definition of $\bm\Phi_k$, $\bm{\Gamma}_k$, $\bm{V}_k$, and $\mathcal{F}_k$. Next, we show that B1$'$ holds. First, by C0 and C1:
\begin{align*}
\lim_{k\rightarrow \infty}\bm\Gamma_k&=\lim_{k\rightarrow \infty}k^\upalpha\sum_{z=1}^{s}\sum_{\ell=0}^{n(z)-1} {a}_{k}^{(j(z))} \bm{J}^{(j(z))}(\overline{\bm\uptheta}_k)=\sum_{z=1}^{s}\sum_{\ell=0}^{n(z)-1} r_{j(z)}\bm{J}^{j(z)}(\bm\uptheta^\ast){\text{ w.p.1.}}
\end{align*}
Then, C2 guarantees B1$'$ holds with $\bm\Gamma$ equal to the matrix in (\ref{eq:macnealhuh}). The fact that B2$'$ holds follows immediately from the fact that $\bm\Phi_k=\bm{I}$ for all $k$, so that $\bm\Phi$ in B1$'$ is equal to the identity matrix. Next we show that B3$'$ holds.

First, using C1 and C3 it follows that
\begin{align}
- \lim_{k\rightarrow \infty}k^{\upalpha+\upbeta/2}\sum_{z=1}^{s}\sum_{\ell=0}^{n(z)-1} {a}_{k}^{(j(z))} {\bm{\upbeta}}_k^{(j(z))}\left(\hat{\bm\uptheta}_k^{\left(I_{z,\ell}\right)}\right)=-\sum_{z=1}^{s}\sum_{\ell=0}^{n(z)-1}r_{j(z)}\bm{b}^{(j(z))}{\text{ w.p.1.}}
\label{eq:timetowos}
\end{align}
Next, C0, C1, and C3 imply that
\begin{align}
&\lim_{k\rightarrow \infty}k^{\upalpha+\upbeta/2}\sum_{z=1}^{s}\sum_{\ell=0}^{n(z)-1} {a}_{k}^{(j(z))} \bm{J}^{(j(z))}\left(\overline{\bm\uptheta}_k^{(I_{z,\ell})}\right)[\bm\Delta_{\bm{g}}^{(I_{z,\ell})}+\bm\Delta_{\bm{\upbeta}}^{(I_{z,\ell})}]\notag\\
&=\lim_{k\rightarrow \infty} k^{\upalpha +\upbeta/2}\sum_{z=1}^{s}\sum_{\ell=0}^{n(z)-1} O(k^\upalpha) \bm{J}^{(j(z))}\left(\overline{\bm\uptheta}_k^{(I_{z,\ell})}\right)o(k^\upalpha)\notag\\
&=\bm{0}{\text{ w.p.1.}}
\label{eq:mosestron}
\end{align}
Lastly, by C0 (the boundedness of $\bm{H}(\bm\uptheta)$), C1, C3 (using $0<\upbeta\leq \upalpha$), and C5:
\begin{align}
&\lim_{k\rightarrow \infty}k^{\upalpha+\upbeta/2}\Bigg\|\sum_{z=1}^{s}\sum_{\ell=0}^{n(z)-1} {a}_{k}^{(j(z))} E\Big[\bm{J}^{(j(z))}\left(\overline{\bm\uptheta}_k^{(I_{z,\ell})}\right) \bm\Delta_{\bm\upxi}^{(I_{z,\ell})}\Big|\mathcal{F}_k\Big]\Bigg\|\notag\\
&\leq \lim_{k\rightarrow \infty}k^{\upalpha+\upbeta/2}\sum_{z=1}^{s}\sum_{\ell=0}^{n(z)-1} {a}_{k}^{(j(z))} \left(E\Big[\Big\|\bm{J}^{(j(z))}\left(\overline{\bm\uptheta}_k^{(I_{z,\ell})}\right) \bm\Delta_{\bm\upxi}^{(I_{z,\ell})}\Big\|^2\Big|\mathcal{F}_k\Big]\right)^{1/2}\notag\\
&\leq\lim_{k\rightarrow \infty}k^{\upalpha+\upbeta/2}\sum_{z=1}^{s}\sum_{\ell=0}^{n(z)-1} O(k^{\upalpha})\left(E\left[\Big\|\bm{J}^{(j(z))}\left(\overline{\bm\uptheta}_k^{(I_{z,\ell})}\right)\Big\|^2\Big\| \bm\Delta_{\bm\upxi}^{(I_{z,\ell})}\Big\|^2\Big|\mathcal{F}_k\right]\right)^{1/2}\notag\\
&\leq \lim_{k\rightarrow \infty}k^{\upalpha+\upbeta/2}\sum_{z=1}^{s}\sum_{\ell=0}^{n(z)-1}O(k^{\upalpha})\Big[E\Big\| \bm\Delta_{\bm\upxi}^{(I_{z,\ell})}\Big\|^2\Big|\mathcal{F}_k\Big]^{1/2}\notag\\
&={0}{\text{ w.p.1}},
\label{eq:listeningtozeldeeee}
\end{align}
where the first inequality in (\ref{eq:listeningtozeldeeee}) follows from the triangle inequality and the fact that $\|E[\bm{v}]\|^2\leq E\|\bm{v}\|^2$ for $\bm{v}\in \mathbb{R}^p$ so that $\|E[\bm{v}]\|\leq \sqrt{E\|\bm{v}\|^2}$, the third inequality follows from the submultplicativity of the Frobenius norm ($\|\cdot\|$ applied to a matrix is the Frobenius norm), and the last inequality follows from the boundedness of the the Hessian (which implies the boundedness of $\bm{J}^{(j)}(\bm\uptheta)$). 
%
Thus, the definition of $\bm{T}_k$ in (\ref{eq:goingcrazyhim}) along with (\ref{eq:timetowos})--(\ref{eq:listeningtozeldeeee}) imply
\begin{align}
\bm{T}_k\rightarrow -\sum_{z=1}^{s}\sum_{\ell=0}^{n_(z)-1}r_{j(z)}\bm{b}^{(j(z))}{\text{ w.p.1.}}
\label{eq:nobrothersand}
\end{align}
This implies B3$'$ holds with $\bm{T}$ defined as in the right-hand side of (\ref{eq:nobrothersand}). Next we show B4$'$ holds. 

That $E[\bm{V}_k|\mathcal{F}_k]=\bm{0}$ follows immediately from C4. Next we show that there exists a constant $C'<\infty$ such that $C'>\|E[\bm{V}_k\bm{V}_k^\top|\mathcal{F}_k]-\bm\Sigma\|\rightarrow 0$, where $\bm\Sigma$ is as in C6. 
First, 
by C1, C4, and C5:
\begin{align}
\label{eq:avatarparticulatsfs}
C_1>\Bigg\|\Var{\Bigg[k^{(\upalpha+\upbeta)/2}\sum_{z=1}^{s}\sum_{\ell=0}^{n(z)-1} {a}_{k}^{(j(z))} {\bm{\upxi}}_k^{(j(z))}\left(\hat{\bm\uptheta}_k^{\left(I_{z,\ell}\right)}\right)\Bigg|\mathcal{F}_k\Bigg]}-\bm\Sigma\Bigg\|\rightarrow 0,
\end{align}
for some constant $C_1<\infty$. (Note that the vector inside the variance in the previous equation has mean zero by the first part of condition C4. Therefore, its variance is equal to its second moment matrix.)  Next, by the definition of $\bm\Delta_{\bm\upxi}^{(\cdot)}$ in (\ref{eq:futuramasohs}) and using conditions C0, C1, and C5 it follows that there must exist a constant $C_2<\infty$ such that:
\begin{align}
\label{eq:doughiesloon}
C_2>\Bigg\|\Var{\Bigg[k^{(\upalpha+\upbeta)/2}\sum_{z=1}^{s}\sum_{\ell=0}^{n(z)-1} {a}_{k}^{(j(z))} \bm{J}^{(j(z))}\left(\overline{\bm\uptheta}_k^{(I_{z,\ell})}\right) \bm\Delta_{\bm\upxi}^{(I_{z,\ell})}\Bigg|\mathcal{F}_k\Bigg]}\Bigg\|\rightarrow 0.
\end{align}
%
%
From (\ref{eq:avatarparticulatsfs}) we see that the second moment matrix of the first line of (\ref{eq:othingcoodsd}) is bounded and converges to $\bm\Sigma$. Additionally, (\ref{eq:doughiesloon}) says that the second moment matrix of the last two lines of (\ref{eq:othingcoodsd}) is bounded and converges to the zero matrix. Furthermore, any cross-term expectation involving a summand from the first line of (\ref{eq:othingcoodsd}) and a summand from the second or third line of (\ref{eq:othingcoodsd}) is also bounded and converges to the zero matrix due to the fact that:
\begin{align*}
k^{\upbeta-\upalpha} O(k^{-\upalpha})\left(E\Big\|{\bm{\upxi}}_k^{(j(z))}\Big(\hat{\bm\uptheta}_k^{\left(I_{z,\ell}\right)}\Big)\Big|\mathcal{F}_k\Big\|^2\right)^{1/2}\left(E\Big\|{\bm{\upxi}}_k^{(j(m))}\Big(\hat{\bm\uptheta}_k^{\left(I_{m,i}\right)}\Big)\Big|\mathcal{F}_k\Big\|^2\right)^{1/2}=o(1),
%
\end{align*}
a consequence of C5. In conclusion, $E[\bm{V}_k\bm{V}_k^\top|\mathcal{F}_k]$ is equal to the sum of finite number of bounded matrices, most of which converge to the zero matrix w.p.1 except for one matrix (the matrix in the variance of equation \ref{eq:avatarparticulatsfs}) which converges to $\bm\Sigma$ w.p.1. Therefore, there exists a constant $C'<\infty$ such that $C'>\|E[\bm{V}_k\bm{V}_k^\top|\mathcal{F}_k]-\bm\Sigma\|\rightarrow 0$ and B4$'$ holds.

It remains to show that B5$'$ and B6$'$ hold. First, conditions C0 and C7 imply $\bm{V}_k$ 
satisfies the conditions of Lemma \ref{lem:piccardo}. Therefore, by Lemma \ref{lem:piccardo} the vector $\bm{V}_k$ satisfies B5-(i) (a special case of B5$'$). Lastly, since C8 and B6$'$ are identical, condition B6$'$ is automatically satisfied. Since we have shown that conditions C0--C8 imply B0$'$--B6$'$ are satisfied, the conclusion of Theorem \ref{thm:generalizefabian} holds and the desired result follows from said theorem.
\end{proof}

\section[On the Conditions for Normality]{On the Conditions for Normality}
\label{sec:bonestrailseason}

Theorem \ref{thm:fnogg} stated a set of conditions for the asymptotic normality of the scaled iterates from Algorithm \ref{kirkey} (conditions C0--C8). This section discusses the validity of these conditions. 

{\underline{\it{Condition C0}}}. One of the assumptions of this condition is that the Hessian of $L(\bm\uptheta)$, $\bm{H}(\bm\uptheta)$, must be bounded uniformly over $\bm\uptheta$. This is certainly a stronger assumption than condition A0 which only requires $L(\bm\uptheta)$ to be twice-continuously differentiable. 
 Condition C0 also assumes that $\hat{\bm\uptheta}_k^{(I_{z,i})}\rightarrow \bm\uptheta^\ast$ w.p.1 uniformly over $z$ and $i$, this assumption is satisfied under the conditions of Theorem \ref{thm:hoeshoo} (see Corollary \ref{eq:idamagedittarantatan}). 

{\underline{\it{Condition C1}}}. Note that this condition is a special case of the assumption from A0$''$ requiring $a_k^{(j)}/a_k\rightarrow r_j$ (condition C1 replaces $a_k$ with $k^{-\upalpha}$). As discussed in Section \ref{sec:discussconvergence}, C1 would be satisfied by sequences of the form $a_k^{(j)}=a^{(j)}/(1+k+A^{(j)})^\upalpha$ where $a^{(j)}$ and $A^{(j)}$ are both strictly-positive constants. 

{\underline{\it{Condition C2}}. This condition requires $\bm\Gamma=\sum_{z=1}^{s}\sum_{\ell=0}^{n(z)-1} r_{j(z)} \bm{J}^{(j(z))}(\bm\uptheta^\ast)
$
to be equal to $\bm{SUS}^{-1}$ where $\bm{S}$ and $\bm{U}$ are both real matrices and $\bm{U}$ is positive-definite and upper-triangular. In order to discuss whether this is a reasonable assumption, note that $\bm{J}^{(j)}(\bm\uptheta^\ast)\in \mathbb{R}^{p\times p}$ is a matrix formed by taking $\bm{H}(\bm\uptheta^\ast)$ and replacing the rows in the set $\mathcal{S}_j$, as defined in (\ref{eq:notexclusive}), with zeros. If $\bm{H}(\bm\uptheta^\ast)$ is positive definite, one situation in which C2 holds is when $r_j>0$ for all $j$ and the $\mathcal{S}_j$ are disjoint sets (it is important to note that having this sets be disjoint is not a necessary condition). Then,
\begin{align}
\sum_{z=1}^{s}\sum_{\ell=0}^{n(z)-1} r_{j(z)} \bm{J}^{(j(z))}(\bm\uptheta^\ast)=\sum_{j=1}^d r_j d_j\bm{J}^{(j)}(\bm\uptheta^\ast)
\label{eq:parisorny}
\end{align}
where $d_j= \sum_{z=1}^{s}\sum_{\ell=0}^{n(z)-1}{}\chi\{j(z)=j\}$. The variable $d_j$ represents the (deterministic) number of times the $j$th subvector is updated during the $k$th iteration. Provided $r_jd_j>0$, (\ref{eq:parisorny}) can be rewritten as follows:
\begin{align}
\label{eq:sensorypinetree}
\sum_{j=1}^d r_j d_j\bm{J}^{(j)}(\bm\uptheta^\ast)=\bm{\Lambda}\bm{H}(\bm\uptheta^\ast),
\end{align}
where the matrix $\bm{\Lambda}$ is a positive definite diagonal matrix with $i$th diagonal entry equal to $r_jd_j$ where $j$ satisfies $i\in \mathcal{S}_j$ (recall that in this example $i$ lies in $\mathcal{S}_j$ for exactly one $j$ since the sets $\mathcal{S}_j$ are disjoint). Then, as proven in Proposition \ref{prop:sandpipercrossing}, condition C2 holds. Note that having $r_jd_j>0$ implies that $k^\upalpha a_k^{(j)}\rightarrow r_j>0$ and that each subvector must be updated at least once per iteration.

{\underline{\it{Condition C3}}. In the case where the noisy gradient estimates in Algorithm \ref{kirkey} are obtained in using the SG estimates (see, for example, Section \ref{sec:ay}) the bias term is equal to zero. Therefore, in this case we can assume ${\bm{\upbeta}}^{(j(z))}\big(\hat{\bm\uptheta}_k^{(I_{z,i})}\big)=\bm{0}$ so that C3 holds with $\bm{b}^{(j(z))}=\bm{0}$. When the noisy gradient estimates are obtained in an SPSA manner (see, for example, Section \ref{sect:cspsa}) the bias terms ${\bm{\upbeta}}^{(j(z))}\big(\hat{\bm\uptheta}_k^{(I_{z,i})}\big)$ cannot be assumed to be equal to the zero vector. However, analogous conditions to those of Proposition 2 in Spall (1992)\nocite{spall1992} (after an natural adaptation for the setting of Algorithm \ref{kirkey}) can be used to show that the bias terms satisfy  C3 when the noisy gradient updates are obtained in the manner of (\ref{eq:shmasdfg})--(\ref{eq:splice}) or in the manner of (\ref{eq:howeardspos1})--(\ref{eq:howeardspos2}).

{\underline{\it{Condition C4}}. The mean-zero assumption of the noise terms is a common assumption throughout the SA literature. The assumption essentially requires the conditional mean of the noise term (conditional on $\hat{\bm\uptheta}_k^{(I_{z,i})}$) to be absorbed into the bias term ${\bm{\upbeta}}_k^{(j(z))}\big(\hat{\bm\uptheta}_k^{(I_{z,i})}\big)$. Condition C3 would then require the conditional mean of the noise terms to converge to the zero vector.

{\underline{\it{Condition C5}}. Condition C5 allows for two possible scenarios. The first case, C5-(i), requires the noise vectors to have a bounded conditional covariance matrix which converges w.p.1. A simple scenario in which this holds is when the noise vectors are based on independent multivariate Gaussians with mean vector equal to zero and covariance matrix equal to $\bm{I}$. Condition C5-(ii), on the other hand, allows the noise terms to have a growing (in magnitude) conditional covariance matrix. The amount of growth, however, must be proportional to $k^{\upalpha-\upbeta}$. This condition is well suited for algorithms such as SPSA where the conditional covariance of the noise grows as $k\rightarrow \infty$ (e.g., Spall 1992)\nocite{spall1992}. Conditions C5-(iii) and C5-(iv) are similar to C5-(i) and C5-(ii), respectively. Loosely speaking, C5-(iii) requires any two different noise vectors to be asymptotically uncorrelated while C5-(iv) allows the magnitude of the correlation to increase at a rate proportional to $k^{\upalpha-\upbeta}$. One scenario where conditions C5-(iii) and C5-(iv) are satisfied is when the noise terms are independent of each other. Section \ref{sec:hannahashley} gives an example of where the noise terms are not independent but where C5-(iv) holds.

{\underline{\it{Condition C6}}}. This condition requires $\bm\Sigma=\sum_{z=1}^{s}\sum_{\ell=0}^{n(z)-1} r_{j(z)} {\bm\Sigma}^{(j(z))}$
to be positive definite. Let us give an example of when this condition may hold. If the sets $\mathcal{S}_j$ are disjoint (i.e., that the subvectors to update do not overlap) then with $d_j$ defined as in (\ref{eq:sensorypinetree}): 
\begin{align*}
\bm\Sigma=\sum_{j=1}^d r_j d_j\bm{\Sigma}^{(j)}
\end{align*}
(note that each matrix $\bm{\Sigma}^{(j)}$ has zeros in the columns and rows whose indices are not in $\mathcal{S}_j$). If the submatrix of $\bm{\Sigma}^{(j)}$ formed by taking the rows and columns with indices $\mathcal{S}_j$ is positive definite (i.e., assume the limiting conditional covariance matrix of the nonzero entries of $\bm\upxi_k^{(j)}$ is positive definite, a reasonable assumption) and if we also assume $r_jd_j>0$ then $\bm\Sigma$ would be a block-diagonal positive definite matrix and C6 would be satisfied. Having $r_jd_j>0$ requires each subvector to be updated at least once per iteration and $k^\upalpha a_k^{(j)}\rightarrow r_j>0$.

{\underline{\it{Conditions C7 and C8}}}. Conditions C7 and C8 are assumptions that are often applied to non-cyclic SA algorithms. C7 pertains to the uniform integrability of the squared magnitude of $\{{\bm{\upxi}}_k^{(j(z))}\big(\hat{\bm\uptheta}_k^{(I_{z,i})}\big)\}$ or $k^{(\upbeta-\upalpha)/2}{\bm{\upxi}}_k^{(j(z))}\big(\hat{\bm\uptheta}_k^{(I_{z,i})}\big)$ (see Definition \ref{def:unifintegrable}). A sufficient condition for C7 to hold is if the traces of the covariance matrices of the a noise terms are bounded uniformly over $k$. Condition C8 is impossible to verify in practice although we note that C8 is essentially identical to condition B6 of Fabian's theorem (see p. \pageref{page:scratchless}).

\section{Concluding Remarks}
\label{sec:dontmeanthings}

This chapter presented a generalization of Fabian's theorem (see Theorem 2.2 in Fabian 1968 or Theorem \ref{thm:fabian} in this dissertation)\nocite{fabian1968}. The resulting generalization (Theorem \ref{thm:generalizefabian}) was used to prove the asymptotic normality of the normalized iterates form Algorithm \ref{kirkey} (a few additional applications of Theorem \ref{thm:generalizefabian} are discussed in Appendix \ref{sec:fabiansecgeneralize}). Attempting to use Theorem \ref{thm:generalizefabian} to derive an asymptotic normality result for Algorithm \ref{beastwasdone} gives rise to a critical complication.
 First, note that when $L(\bm\uptheta)$ is twice continuously differentiable then:
%
\begin{align}
\label{eq:noealnavyyard}
 \hat{\bm{\uptheta}}_{k+1}=&\ \hat{\bm\uptheta}_k-\tilde{a}_{k}^{(j_k)} \bm{J}^{(j_k)}(\overline{\bm\uptheta}_k)(\hat{\bm{\uptheta}}_k-\bm\uptheta^\ast)- \tilde{a}_{k}^{(j_k)} {\bm{\upbeta}}_k^{(j_k)}(\hat{\bm{\uptheta}}_k)- \tilde{a}_{k}^{(j_k)} {\bm{\upxi}}_k^{(j_k)}(\hat{\bm{\uptheta}}_k),
 \end{align}
 where $\hat{\bm{\uptheta}}_k$ denotes an iterate from Algorithm \ref{beastwasdone} and where the $i$th row of $\bm{J}^{(j)}(\overline{\bm\uptheta}_k)$ is equal to the $i$th row of $\bm{J}^{(j)}(\bm\uptheta)$ evaluated at $\bm\uptheta=(1-\uplambda_i)\hat{\bm{\uptheta}}_k+\uplambda_i\bm\uptheta^\ast$ for some $\uplambda_i\in[0,1]$ which depends on $\hat{\bm{\uptheta}}_k$. Then, (\ref{eq:noealnavyyard}) may be written in the form of (\ref{eq:alien}) by letting $\bm{W}_k=\hat{\bm{\uptheta}}_k-\bm\uptheta^\ast$, $\bm\Gamma_k=k^\upalpha\tilde{a}_{k}^{(j_k)} \bm{J}^{(j_k)}(\overline{\bm\uptheta}_k)$, $\bm{T}_k=-k^{\upalpha+\upbeta/2}\tilde{a}_{k}^{(j_k)} {\bm{\upbeta}}_k^{(j_k)}(\hat{\bm{\uptheta}}_k)$, $\bm\Phi_k=\bm{I}$, and $\bm{V}_k=-k^{(\upalpha+\upbeta)/2}\tilde{a}_{k}^{(j_k)} {\bm{\upxi}}_k^{(j_k)}(\hat{\bm{\uptheta}}_k)$. Note, however, that $\bm\Gamma_k$ does not converge w.p.1 to any matrix since the random variable $j_k$ does not converge. Therefore, Theorem \ref{thm:generalizefabian} cannot be used to derive an asymptotic normality result for (\ref{eq:noealnavyyard}) for this definition of $\bm\Gamma_k$. Moreover, as discussed in pp. \pageref{eq:anotherpagerefgoeshere}--\pageref{prop:sandpipercrossing}, it is not generally possible to find an alternate definition of $\bm\Gamma_k$ (along with a redefinition of $\bm{T}_k$, $\bm{\Phi}_k$, and $\bm{V}_k$) so that all of the conditions of Fabian's conditions hold without imposing more assumptions on $\hat{\bm{\uptheta}}_k$. Future research could consider combining Algorithm \ref{kirkey} with iterate averaging, which may lead to the asymptotic normality of the normalized iterates.

%% file: chapter4p.tex

\chapter{Efficiency of GCSA}
\label{chap:imefficient}

 This chapter attempts to quantify the {{relative efficiency}} between a strictly cyclic version of GCSA and its non-cyclic counterpart, where efficiency is defined in terms of the mean squared error (MSE) after taking into consideration the cost of implementation (the term ``cost'' is used in a very general sense and the precise definition is highly problem-specific). Specifically, if  $\hat{\bm\uptheta}_k^{\text{cyc}}$ denotes the $k$th iterate of a strictly cyclic implementation of GCSA and $\hat{\bm\uptheta}_k^{\text{non}}$ denotes the $k$th iterate of its non-cyclic counterpart, the focus of this chapter is on the ratio:
 \begin{align}
\label{eq:ratio}
\frac{E\|\hat{\bm\uptheta}_{k_1}^{\text{cyc}}-\bm\uptheta^\ast\|^2}{E\|\hat{\bm\uptheta}_{k_2}^{\text{non}}-\bm\uptheta^\ast\|^2},
\end{align}
 where the cumulative cost of implementing the cyclic algorithm up to iteration $k_1$ is equal (or approximately equal) to that of its non-cyclic counterpart up to iteration $k_2$ ($k_2$ is then a function of $k_1$ and vice versa). An analogue to (\ref{eq:ratio}) would be to set the MSEs in the numerator and in the denominator to be equal and then study how $k_1$ and $k_2$ differ (e.g, Spall 1992\nocite{spall1992}). In general, the definition of per-iteration cost  is highly problem-specific and plays an important role in computing (\ref{eq:ratio}). Section \ref{sec:costoport} discusses the per-iteration cost issue in more detail. Section \ref{sec:aseficwiener} then obtains an asymptotic approximation to the ratio in (\ref{eq:ratio}). Section \ref{sec:apsecialcaserelefff} uses the approximation from Section \ref{sec:aseficwiener} to study the relative efficiency between a special case of Algorithm \ref{kirkey} and its non-cyclic counterpart. Lastly, Section \ref{eq:hodgmodgepodge} contains concluding remarks.


\section{Cost of Implementation}
\label{sec:costoport}

In order to perform a fair comparison between any two algorithms  it is necessary to  take the cost of implementation into consideration. Here, the precise definition of ``cost'' varies depending on the specific algorithm at hand. For example, cost could be a measure of the number of iterations, the actual run-time of an algorithm, the number of times a certain simulation has to be run, the amount of storage required, or a combination of these. 
As another example, Knight et al. (2014)\nocite{knightetal2013} consider a simplified machine model with several processors and represents run-time as a linear function of the arithmetic (flop) cost and other variables (e.g., number of messages sent). Thus, if cost is defined as run-time, the cost of implementation will be a function of the flop count as well as of other variables. This section focuses on comparing the cost of implementing the cyclic seesaw SA algorithm (the ideas can easily be generalized to the case where $\bm\uptheta$ is partitioned into more than two subvectors) to the cost of implementing its non-cyclic counterpart when cost is a measure of the arithmetic computations needed to perform a single iteration. Specifically, this section pinpoints the type of arithmetic computations that can result in a significant difference between the per-iteration cost of implementing a cyclic algorithm and the per-iteration cost of implementing its non-cyclic formulation.

 Letting $\hat{\bm{\uptheta}}_k^{\text{cyc}}$ denote an iterate of the cyclic seesaw SA algorithm (Section \ref{sec:keatonmodasd}) and using the notation in (\ref{eq:analgoysDRAGON}) we assume throughout this section that:
 \begin{subequations}
\begin{align}
{{\hat{\bm\uptheta}}}_{k}^{(I)} &=\left[\begin{array}{c}(\hat{\bm{\uptheta}}_{k+1}^{\text{cyc}})^{[\bm{1}]} \\ \relax (\hat{\bm{\uptheta}}_k^{\text{cyc}})^{[\bm{2}]}\end{array}\right]=\hat{\bm{\uptheta}}_k^{\text{cyc}}-\bm\Phi^{(1)}(\hat{\bm{\uptheta}}_k^{\text{cyc}},k,\bm{V}_{k})\in \mathbb{R}^p,\label{eq:huh0}\\
\hat{\bm\uptheta}_{k+1}^{\text{cyc}}&=\left[\begin{array}{c}(\hat{\bm{\uptheta}}_{k+1}^{\text{cyc}})^{[\bm{1}]} \\ \relax (\hat{\bm{\uptheta}}_{k+1}^{\text{cyc}})^{[\bm{2}]}\end{array}\right]=\hat{\bm{\uptheta}}_k^{(I)}-\bm\Phi^{(2)}(\hat{\bm{\uptheta}}_{k}^{(I)},k,\bm{V}_{k}^{(I)})\in \mathbb{R}^p,\label{eq:huh}
\end{align}
\end{subequations}
for a {{known}} vector-valued function satisfying $\bm\Phi(\bm\uptheta,k,\bm{V}){: \mathbb{R}^p\times \mathbb{Z}^+\times\Omega_{\bm{V}}\rightarrow \mathbb{R}^p}$ where $\bm\Phi^{(j)}(\cdot,\cdot,\cdot)$ is defined using the notation from (\ref{eq:electricpulsar}) and $\bm{V}_k$ and $\bm{V}_k^{(I)}$ are random variables with values in $\Omega_{\bm{V}}$.  We refer to the functions  $\bm\Phi^{(1)}(\cdot,\cdot,\cdot)$ and  $\bm\Phi^{(2)}(\cdot,\cdot,\cdot)$ as {\it{update functions}}. In the CSG algorithm (Section \ref{sec:ay}), for example, if $\hat{\bm{g}}^{\text{SP}}_k(\hat{\bm{\uptheta}}_k^{\text{cyc}})$ and $\hat{\bm{g}}^{\text{SP}}_k(\hat{\bm{\uptheta}}_k^{(I)})$ are obtained via the evaluation of an arithmetic function (e.g., the pure infinitesimal perturbation analysis (IPA) algorithm where the update direction is as in (\ref{eq:imanIPA})) then $\bm\Phi^{(1)}(\hat{\bm{\uptheta}}_k^{\text{cyc}},k,\bm{V}_{k})=a_{k}^{(1)}[\hat{\bm{g}}^{\text{SP}}_k(\hat{\bm{\uptheta}}_k^{\text{cyc}})]^{(1)}$ and $\bm\Phi^{(2)}(\hat{\bm{\uptheta}}_{k}^{(I)},k,\bm{V}_{k}^{(I)})=a_{k}^{(2)}[\hat{\bm{g}}^{\text{SG}}(\bm{{\hat{\uptheta}}}_{k}^{(I)})]^{(2)}$.
Similarly, if $\hat{\bm{\uptheta}}_k^{\text{non}}$ denotes an iterate of the basic non-cyclic SA algorithm (see Section \ref{sec:dummieyouask}) we assume:
\begin{align}
\label{eq:herculedherculess}
\hat{\bm{\uptheta}}_{k+1}^{\text{non}}&=\left[\begin{array}{c}(\hat{\bm{\uptheta}}_k^{\text{non}})^{[\bm{1}]} \\ \relax (\hat{\bm{\uptheta}}_k^{\text{non}})^{[\bm{2}]}\end{array}\right]\notag\\
&=\hat{\bm{\uptheta}}_k^{\text{non}}-\bm\Phi(\hat{\bm{\uptheta}}_k^{\text{non}},k,\bm{V}_k)\notag\\
&=\hat{\bm{\uptheta}}_k^{\text{non}}-\bm\Phi^{(1)}(\hat{\bm{\uptheta}}_k^{\text{non}},k,\bm{V}_k)-\bm\Phi^{(2)}(\hat{\bm{\uptheta}}_k^{\text{non}},k,\bm{V}_k)\in \mathbb{R}^p,
\end{align}
 where $\bm{V}_k$ is a random variable in $\Omega_{\bm{V}}$. 
Both (\ref{eq:huh0},b) and (\ref{eq:herculedherculess}) 
may be implemented in a distributed or non-distributed manner, in the sense that the vectors $\bm\Phi^{(1)}(\cdot,\cdot,\cdot)$ and $\bm\Phi^{(2)}(\cdot,\cdot,\cdot)$ are allowed (but not required) to be computed at different locations. Botts et al. (2016), for example, consider a multi-agent surveillance and tracking problem in which agent $j$ updates its state vector (corresponding to the vector $\bm\uptheta^{[\bm{j}]}$) using $\bm\Phi^{(j)}(\cdot,\cdot,\cdot)$.
In this multi-agent problem, each of the vector-valued functions $\bm\Phi^{(j)}(\cdot,\cdot,\cdot)$ is computed at a different location by a different agent. This multi-agent algorithm is, therefore, a distributed algorithm. 
Next we define the cost of implementing (\ref{eq:huh0},b) and (\ref{eq:herculedherculess}).

 The definition of cost to be used in this section is based on the observation that any function, regardless of complexity, can be evaluated by performing a sequence of elementary operations involving one or two arguments at a time (e.g., addition, multiplication, division, the power operation, trigonometric functions, exponential functions, and logarithmic functions). 
 Then, the cost of implementing an iteration of (\ref{eq:herculedherculess}), denoted by $c^{\text{non}}_k$, will be defined as the number of elementary  operations required to evaluate $\bm\Phi^{(1)}(\hat{\bm{\uptheta}}_k^{\text{non}},k,\bm{V}_k)$ and $\bm\Phi^{(2)}(\hat{\bm{\uptheta}}_k^{\text{non}},k,\bm{V}_k)$. Similarly, the cost of implementing an iteration of (\ref{eq:huh0},b), denoted by $c^{\text{cyc}}_k$,  will be defined as the number of elementary  operations required to evaluate both $\bm\Phi^{(1)}(\hat{\bm{\uptheta}}_k^{\text{cyc}},k,\bm{V}_{k})$ and $\bm\Phi^{(2)}(\hat{\bm{\uptheta}}_{k}^{(I)},k,\bm{V}_{k}^{(I)})$. Note that we are omitting the cost of performing the subtraction operations appearing in (\ref{eq:huh0},b) and the subtraction operation appearing in (\ref{eq:herculedherculess}). This was done due to the fact that performing these addition operations takes the same number of elementary operations, namely $p$, for both the cyclic and the non-cyclic algorithms. Therefore, the aforementioned subtraction operations are not significant given the focus of this section on pinpointing the sources of discrepancy between the cost of implementing a cyclic and the cost of implementing its non-cyclic counterpart.

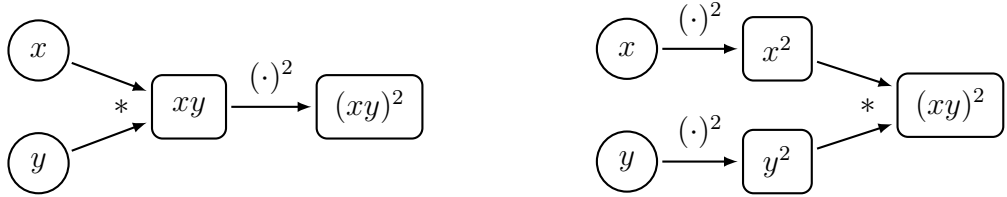
\begin{figure}
\centering
        \begin{subfigure}[b]{0.49\textwidth}
                \centering
\begin{tikzpicture}

\node[draw,circle,minimum size=.8cm,inner sep=0pt, thick] (x) at (0,0) {$x$};
\node[draw,circle,minimum size=.8cm,inner sep=0pt, thick] (y) at (0,-1.5) {$y$};

\node[draw,rectangle, rounded corners, minimum height=2em, minimum width=28pt, thick] (xy) at (2,-0.75) {$xy$};

\node[draw,rectangle, rounded corners, minimum height=2em, minimum width=40pt, thick] (xys) at (4.4,-0.75) {$(xy)^2$};

\draw[->, thick] (x)--(xy) ;
\draw[->,thick] (y)--(xy);
\node at (1.1,-0.75) {$\ast$};
\draw[->,thick] (xy)--(xys) node[above,midway, above=3pt] {$(\cdot)^2$};

\end{tikzpicture}
                \caption{Requires two elementary operations.}
                \label{fig:gull}
        \end{subfigure}\hfill
        \begin{subfigure}[b]{0.49\textwidth}
                \centering
\begin{tikzpicture}

\node[draw,circle,minimum size=.8cm,inner sep=0pt, thick] (x) at (0,0) {$x$};
\node[draw,circle,minimum size=.8cm,inner sep=0pt, thick] (y) at (0,-1.5) {$y$};

\node[draw,rectangle, rounded corners, minimum height=2em, minimum width=26pt, thick] (xs) at (2,0) {$x^2$};

\node[draw,rectangle, rounded corners, minimum height=2em, minimum width=26pt, thick] (ys) at (2,-1.5) {$y^2$};

\draw[->, thick] (x)--(xs) node[above,midway, above=3pt] {$(\cdot)^2$};
\draw[->,thick] (y)--(ys) node[above,midway, above=3pt] {$(\cdot)^2$};
\node at (3.2,-0.75) {$\ast$};

\node[draw,rectangle, rounded corners, minimum height=2em, minimum width=40pt, thick] (xys) at (4.3,-0.75) {$(xy)^2$};
\draw[->,thick] (xs)--(xys);
\draw[->,thick] (ys)--(xys);

\end{tikzpicture}
                \caption{Requires three elementary operations.}
                \label{fig:gull2}
        \end{subfigure}\hfill
        \caption[Two computational graphs (CGs) for computing $f(x,y)=(xy)^2$]{Two CGs for evaluating the function $\uprho(x,y)=(xy)^2$. 
        }\label{fig:oneatmosphere}
\end{figure}
 
 The sequence of elementary operations performed when evaluating any real-valued function can be represented by a computational graph (CG) (e.g., Nocedal and Wright, p. 205) which is generally not unique (see Figure \ref{fig:oneatmosphere}). In this section we will assume that each entry of $\bm\Phi(\cdot,\cdot,\cdot)\in \mathbb{R}^p$ is associated with a predetermined CG whose shape is independent of the arguments of $\bm\Phi(\cdot,\cdot,\cdot)$. In general, the elementary computations performed when evaluating $\bm\Phi(\bm\uptheta,k,\bm{V})$ must fall into one of the following (mutually exclusive) categories:
 \label{catpage}
 \begin{DESCRIPTION}
 \item[Cat. 1:] Computations involving $\bm\uptheta^{[\bm{1}]}$ but not $\bm\uptheta^{[\bm{2}]}$.
 \item[Cat. 2:] Computations involving $\bm\uptheta^{[\bm{2}]}$ but not $\bm\uptheta^{[\bm{1}]}$.
 \item[Cat. 3:]  Computations involving both $\bm\uptheta^{[\bm{1}]}$ and $\bm\uptheta^{[\bm{2}]}$.
 \item[Cat. 4:] Computations not involving $\bm\uptheta$.
 \end{DESCRIPTION}
%
With this in mind, an elementary operation performed when computing any of the entries of $\bm\Phi^{(j)}(\bm\uptheta,k,\bm{V})$ will be denoted by $f_{\ell}^{(j)}(\bm\uptheta^{[\bm{1}]})$, $s_{\ell}^{(j)}(\bm\uptheta^{[\bm{1}]})$, $b_{\ell}^{(j)}(\bm\uptheta^{[\bm{1}]},\bm\uptheta^{[\bm{2}]})$, or $n_{\ell}^{(j)}(k)$ depending on whether it belongs to category 1, 2, 3, or 4 above, respectively. The letter ``$f$'' was chosen to emphasize the fact that an operation depends only on the {\it{first}} part of $\bm\uptheta$. Similarly, the letters ``$s$'', ``$b$'', and ``$n$'', were used to emphasize the fact that an operation depends on the {\it{second}} part of $\bm\uptheta$, on {\it{both}} parts of $\bm\uptheta$, or {\it{neither}} on the first- {\it{nor}} on the second part of $\bm\uptheta$.
 The collection of all elementary operations performed when computing the vectors $\bm\Phi^{(j)}(\bm\uptheta,k,\bm{V})$ can then be summarized by the following sets whose elements can be thought of as labels of the operations to be performed:
%
%
%
\begin{subequations}
 \begin{align}
 &F^{(j)}(\bm\uptheta^{[\bm{1}]})\equiv \left\{ f_{\ell}^{(j)}(\bm\uptheta^{[\bm{1}]})\right\}_{\ell},\ \ \ S^{(j)}(\bm\uptheta^{[\bm{2}]})\equiv \left\{ s_{\ell}^{(j)}(\bm\uptheta^{[\bm{2}]})\right\}_{\ell}, \label{eq:fier0}\\
 &B^{(j)}(\bm\uptheta^{[\bm{1}]},\bm\uptheta^{[\bm{2}]})\equiv \left\{ b_{\ell}^{(j)}(\bm\uptheta^{[\bm{1}]},\bm\uptheta^{[\bm{2}]})\right\}_{\ell},\ \ \ N^{(j)}(k)\equiv \left\{ n_{\ell}^{(j)}(k)\right\}_{\ell}.
 \label{eq:fier}
\end{align}
\end{subequations}
Next we provide insight into the reason why (\ref{eq:huh0},b) and (\ref{eq:herculedherculess}) may have different per-iteration costs (i.e., why $c^{\text{non}}_k$ may not equal $c^{\text{cyc}}_k$).

Consider a simple situation where $\bm\uptheta=[x,y]^\top\in \mathbb{R}^2$ (so that $\bm\uptheta^{[\bm{1}]}=x$ and $\bm\uptheta^{[\bm{2}]}=y$), $V$ is a real-valued random variable whose distribution does not depend on $\bm\uptheta$, and assume that:
\begin{align}
\label{eq:blindmeyyyy}
\bm\Phi(\bm\uptheta,k,{V})=k^{-1}\left[\begin{array}{c}y(xy)^{9}+V \\0\end{array}\right]+k^{-1}\left[\begin{array}{c}0 \\ x (xy)^{9}+V\end{array}\right]
\end{align}
%
so that the first vector, respectively second vector, on the right-hand side of (\ref{eq:blindmeyyyy}) represents $\bm\Phi^{(1)}(\bm\uptheta,k,V)$, respectively $\bm\Phi^{(2)}(\bm\uptheta,k,V)$ (in the context of Section \ref{sec:sgfosa}, $k\bm\Phi(\bm\uptheta,k,{V})$ corresponds to the gradient of the noisy function $Q(\bm\uptheta,V)=(xy)^{10}/10+(x+y)V$).
The non-cyclic algorithm in (\ref{eq:herculedherculess}) with $\bm\Phi(\bm\uptheta,k,{V})$ as in (\ref{eq:blindmeyyyy}) would 
%
be implemented as follows:
\begin{align*}
\hat{\bm{\uptheta}}_{k+1}^{\text{non}}=\hat{\bm{\uptheta}}_k^{\text{non}}-k^{-1}\left[\begin{array}{c}y_k(x_ky_k)^{9}+V_k \\0\end{array}\right]-k^{-1}\left[\begin{array}{c}0 \\ x_k (x_ky_k)^{9}+V_k\end{array}\right],
\end{align*}
where $\hat{\bm{\uptheta}}_k^{\text{non}}=[x_k,y_k]^\top$. Because the possibly non-zero entries of $\bm\Phi^{(1)}(\hat{\bm{\uptheta}}_k^{\text{non}},k,V_k)$ and $\bm\Phi^{(2)}(\hat{\bm{\uptheta}}_k^{\text{non}},k,V_k)$ are very similar, it may be advantageous to share certain computations in order to minimize the overall arithmetic cost. For example, the quantity $(x_ky_k)^9$ appears in both  $\bm\Phi^{(1)}(\hat{\bm{\uptheta}}_k^{\text{non}},k,V_k)$ and $\bm\Phi^{(2)}(\hat{\bm{\uptheta}}_k^{\text{non}},k,V_k)$. In the multi-agent setting, if agents are allowed to communicate then the first agent could compute $(x_ky_k)^{9}$ and share it with the second agent. The ability to share the value of $(x_ky_k)^{9}$ is lost when considering a cyclic implementation (regardless of whether the implementation is distributed or not).\label{page:transferincossst} 

The implementation of the cyclic seesaw algorithm in (\ref{eq:huh0},b) is as follows:
\begin{align*}
\hat{\bm{\uptheta}}_{k+1}^{\text{cyc}}=\hat{\bm{\uptheta}}_{k}^{\text{cyc}}-k^{-1}\left[\begin{array}{c}y_k(x_ky_k)^{9}+V_k \\0\end{array}\right]-k^{-1}\left[\begin{array}{c}0 \\ x_{k+1} (x_{k+1}y_k)^{9}+V_k^{(I)}\end{array}\right],
\end{align*}
where $\hat{\bm{\uptheta}}_k^{\text{cyc}}=[x_k,y_k]^\top$ (in this case, $\hat{\bm{\uptheta}}_k^{(I)}$ would be equal to $[x_{k+1},y_k]^\top$). Since $(x_ky_k)^{9}\neq (x_{k+1}y_k)^{9}$, in a cyclic implementation it would not be possible to reduce cost by sharing the value of $(xy)^9$. However, since $(xy)^9=x^9y^9$, a reduction in the arithmetic cost could be obtained by sharing the values of $x_{k+1}^9$ and $y_k^9$. 
In general, the types of arithmetic computations that can be shared (and the associated reduction in cost) depends on the CGs of the entries of $\bm\Phi(\bm\uptheta,k,\bm{V})$.

In the remainder of this section we derive expressions for $c^{\text{non}}_k$ and $c^{\text{cyc}}_k$ that take into consideration a possible reduction in cost through the sharing of arithmetic computations.
We make the following assumptions:
 \begin{DESCRIPTION}
 \item[D0] The CG corresponding to the $i$th entry of $\bm\Phi(\bm\uptheta,k,\bm{V})$ is the same for both the cyclic- and non-cyclic implementations and does not depend on $\bm\uptheta$, $k$, or $\bm{V}$.
\item [D1] If the algorithm is implemented in a distributed manner, $\bm\uptheta$, $k$ and $\bm{V}$ are known to the agent computing $\bm\Phi^{(j)}(\bm\uptheta,k,\bm{V})$.
\item[D2] During iteration $k$, it is possible to share/reuse computations corresponding to nodes in the CGs which were computed during iterations $1,\dots,k$.
 \end{DESCRIPTION}
Under D0--D1 let $\text{Naive-Cost}^{(1)}$ be the total number of elementary operations required to evaluate the $p'$ CGs  associated with the evaluation of $\bm\Phi^{(1)}(\cdot,\cdot,\cdot)$. Similarly, let $\text{Naive-Cost}^{(2)}$ be the total number of elementary operations required to evaluate the $p'$ CGs  associated with the evaluation of $\bm\Phi^{(2)}(\cdot,\cdot,\cdot)$.
Then, under D0--D2 the following holds for both the cyclic- and non-cyclic algorithms:
\vspace{-.5in}
\begin{subequations}
\begin{align}
c_k^{\text{non}}&= \text{Naive-Cost}^{(1)}+\text{Naive-Cost}^{(2)}- c^{\text{non}}_{\text{share}}(k)- c^{\text{non}}_{\text{carry}}(k),\label{eq:opahhh0}\\
c_k^{\text{cyc}}&= \text{Naive-Cost}^{(1)}+\text{Naive-Cost}^{(2)}- c^{\text{cyc}}_{\text{share}}(k)- c^{\text{cyc}}_{\text{carry}}(k),\label{eq:opahhh}
\end{align}
\end{subequations}
where $c^{\text{non}}_{\text{share}}(k)$ and $c^{\text{cyc}}_{\text{share}}(k)$ denote the cost saved by sharing information on calculations performed during the $k$th iteration, and where $c^{\text{non}}_{\text{carry}}(k)$ and $c^{\text{cyc}}_{\text{carry}}(k)$ denote the cost saved by reusing and/or sharing computations carried from previous iterations.
Next we provide expressions for $c_{\text{share}}^{\text{non}}(k)$, $c_{\text{carry}}^{\text{non}}(k)$, $c_{\text{share}}^{\text{cyc}}(k)$, and $c_{\text{carry}}^{\text{cyc}}(k)$ based on the terminology from (\ref{eq:fier0},b).

%
%
%
 Let $\mathcal{A}_k$ be the set of computations performed during the $k$th iteration that are shared during iteration $k$. 
Similarly, let $\mathcal{B}_k$ be the set of computations performed during iterations $1,\dots, k-1$ that are reused and/or shared during iteration $k$. Additionally, let $\upalpha_{F,\ell}$, $\upalpha_{S,\ell}$, $\upalpha_{B,\ell}$, and $\upalpha_{N,\ell}$ represent the cost saved during the $k$th iteration by sharing $f_{\ell}^{(1)}(\bm\uptheta^{[\bm{1}]})$, $s_{\ell}^{(1)}(\bm\uptheta^{[\bm{2}]})$, $b_{\ell}^{(1)}(\bm\uptheta^{[\bm{1}]},\bm\uptheta^{[\bm{2}]})$, and $n_{\ell}^{(1)}(k)$, respectively. 
%
%
 Then, assuming $\upalpha_{F,\ell}$, $\upalpha_{S,\ell}$, $\upalpha_{B,\ell}$, and $\upalpha_{N,\ell}$ do not depend on $\bm\uptheta$ or $k$:
\begin{align}
c_{\text{share}}^{\text{non}}(k)=&\ \sum_{\ell}\upalpha_{F,\ell}\hspace{.03in} \chi\Big\{f_{\ell}^{(1)}\left((\hat{\bm{\uptheta}}_k^{\text{non}})^{[\bm{1}]}\right)\in F^{(2)}\left((\hat{\bm{\uptheta}}_k^{\text{non}})^{[\bm{1}]}\right)\cap \mathcal{A}_k\Big\}\notag\\
&+\sum_{\ell}\upalpha_{S,\ell}\hspace{.03in}\chi\Big\{s_{\ell}^{(1)}\left((\hat{\bm{\uptheta}}_k^{\text{non}})^{[\bm{2}]}\right)\in S^{(2)}\left((\hat{\bm{\uptheta}}_k^{\text{non}})^{[\bm{2}]}\right)\cap\mathcal{A}_k\Big\}\notag\\
&+\sum_{\ell}\upalpha_{B,\ell}\hspace{.03in}\chi\Big\{b_{\ell}^{(1)}\left((\hat{\bm{\uptheta}}_k^{\text{non}})^{[\bm{1}]},(\hat{\bm{\uptheta}}_k^{\text{non}})^{[\bm{2}]}\right)\in B^{(2)}\left((\hat{\bm{\uptheta}}_k^{\text{non}})^{[\bm{1}]},(\hat{\bm{\uptheta}}_k^{\text{non}})^{[\bm{2}]}\right)\cap\mathcal{A}_k\Big\}\notag\\
&+\sum_{\ell}\upalpha_{N,\ell}\hspace{.03in}\chi\Big\{n_{\ell}^{(1)}(k)\in N^{(2)}(k)\cap\mathcal{A}_k\Big\}.\label{eq:nonshared}
\end{align}

 
 To compute  $c_{\text{carry}}^{\text{non}}(k)$ let $\upbeta_{N,\ell}^{(j)}$ represent the arithmetic cost saved during the $k$th iteration by reusing the elements $n_{\ell}^{(j)}(i)$ for $i=0,\dots,k-1$. Then, assuming $\upbeta_{N,\ell}^{(j)}$ does not depend on $\bm\uptheta$ or $k$:
 %
%
\begin{align}
c_{\text{carry}}^{\text{non}}(k)=&\ \sum_{\ell}\upbeta_{N,\ell}^{(1)}\hspace{.03in}\chi\left\{n_{\ell}^{(1)}(k)\in\bigcup_{i=0}^{k-1}\big( N^{(1)}(i)\cup N^{(2)}(i)\big)\cap\mathcal{B}_k\right\}\notag\\
&+\sum_{\ell}\upbeta_{N,\ell}^{(2)}\hspace{.03in}\chi\left\{n_{\ell}^{(2)}(k)\in\bigcup_{i=0}^{k-1}\big( N^{(1)}(i)\cup N^{(2)}(i)\big)\cap\mathcal{B}_k\right\}\notag\\
&-\sum_{\ell}\upbeta_{N,\ell}^{(1)}\hspace{.03in}\chi\left\{n_{\ell}^{(1)}(k)\in\bigcup_{i=0}^{k-1}\big( N^{(1)}(i)\cup N^{(2)}(i)\big)\cap\mathcal{B}_k\right\}\times\notag\\
&\chi\left\{n_{\ell}^{(2)}\in\bigcup_{i=0}^{k-1}\big( N^{(1)}(i)\cup N^{(2)}(i)\big)\cap\mathcal{B}_k\right\},\label{eq:halfdome}
\end{align}
this follows after noting that the only terms from previous iterations that could be reused during iteration $k$ are terms that do not depend on $\bm\uptheta$.
For the cyclic seesaw algorithm, the cost of sharing computations is different than for the non-cyclic algorithm (as was discussed on p. \pageref{page:transferincossst}). Specifically, 
\begin{align}
c_{\text{share}}^{\text{cyc}}(k)=&\ \sum_{\ell}\upalpha_{S,\ell}\hspace{.03in} \chi\Big\{s_{\ell}^{(1)}\left((\hat{\bm{\uptheta}}_k^{\text{non}})^{[\bm{2}]}\right)\in S^{(2)}\left((\hat{\bm\uptheta}_{k}^{\text{non}})^{[\bm{2}]}\right)\cap \mathcal{A}_k\Big\}\notag\\
&+\sum_{\ell}\upalpha_{N,\ell}\hspace{.03in} \chi\Big\{n_{\ell}^{(1)}(k)\in N^{(2)}(k)\cap \mathcal{A}_k\Big\}.\label{eq:cycshared}
\end{align}
%
Next, letting $\upbeta_{F,\ell}^{(1)}$ represent the arithmetic cost saved during the $k$th iteration by reusing the element $f_{\ell}^{(1)}(\hat{\bm{\uptheta}}_k^{(1)})$ computed during the previous iteration and assuming $\upbeta_{F,\ell}^{(1)}$ does not depend on $\bm\uptheta$ or $k$ it follows that:
%
\begin{align}
&c_{\text{carry}}^{\text{cyc}}( k)=\sum_{\ell}\upbeta_{N,\ell}^{(1)}\hspace{.03in}\chi\left\{n_{\ell}^{(1)}(k)\in\bigcup_{i=0}^{k-1}\big( N^{(1)}(i)\cup N^{(2)}(i)\big)\cap\mathcal{B}_k\right\}\notag\\
&+\sum_{\ell}\upbeta_{N,\ell}^{(2)}\hspace{.03in}\chi\left\{n_{\ell}^{(2)}(k)\in\bigcup_{i=0}^{k-1}\big( N^{(1)}(i)\cup N^{(2)}(i)\big)\cap\mathcal{B}_k\right\}\notag\\
&-\sum_{\ell}\upbeta_{N,\ell}^{(1)}\hspace{.03in}\chi\left\{n_{\ell}^{(1)}(k)\in\bigcup_{i=0}^{k-1}\big( N^{(1)}(i)\cup N^{(2)}(i)\big)\cap\mathcal{B}_k\right\}\times\notag\\
&\chi\left\{n_{\ell}^{(2)}(k)\in\bigcup_{i=0}^{k-1}\big( N^{(1)}(i)\cup N^{(2)}(i)\big)\cap\mathcal{B}_k\right\}\notag\\
&+ \sum_{\ell}\upbeta_{F,\ell}^{(1)}\hspace{.03in}\chi\Big\{f_{\ell}^{(1)}\left((\hat{\bm{\uptheta}}_k^{\text{cyc}})^{[\bm{1}]}\right)\in F^{(2)}\left((\hat{\bm\uptheta}_{k}^{\text{cyc}})^{[\bm{1}]}\right)\cap\mathcal{B}_k\Big\}.\label{eq:tenk}
\end{align}
%
Note that (\ref{eq:tenk}) is different from (\ref{eq:halfdome}) since it contains an extra term (the last line in (\ref{eq:tenk})),
 which originates from the fact that the first $p'$ entries of $\hat{\bm{\uptheta}}_{k-1}^{(I)}$ are the same as the first $p'$ entries of $\hat{\bm{\uptheta}}_k^{\text{cyc}}$. Therefore, calculations involving $(\hat{\bm{\uptheta}}_{k-1}^{(I)})^{[\bm{1}]}$ which were computed during iteration $k-1$ could be used to evaluate the computations involving $(\hat{\bm{\uptheta}}_{k}^{\text{cyc}})^{[\bm{1}]}$ during the $k$th iteration.
The following Proposition compares $c_k^{\text{non}}$ and $c_k^{\text{cyc}}$.

\begin{proposition}
Assume D0--D2 hold along with the following conditions:
\begin{DESCRIPTION}
\item[D3] $\upalpha_{F,\ell}$, $\upalpha_{S,\ell}$, $\upalpha_{B,\ell}$, and $\upalpha_{N,\ell}$ do not depend on $\bm\uptheta$ or $k$.
\item[D4] $\upbeta_{N,\ell}^{(j)}$ does not depend on $\bm\uptheta$ or $k$.
\item[D5] $\upbeta_{F,\ell}^{(j)}$ does not depend on $\bm\uptheta$ or $k$.
\end{DESCRIPTION}
Then,
\begin{align}
\label{eq:core}
c_k^{\text{non}}-c_k^{\text{cyc}}= T_2(k)-T_1(k),
\end{align}
where $T_1(k)\geq 0$ is defined as:
\begin{align*}
T_1(k)&\equiv \sum_{\ell}\upalpha_{F,\ell}\hspace{.03in}\chi\Big\{f_{\ell}^{(1)}\left((\hat{\bm{\uptheta}}_k^{\text{non}})^{[\bm{1}]}\right)\in F^{(2)}\left((\hat{\bm{\uptheta}}_k^{\text{non}})^{[\bm{1}]}\right)\cap \mathcal{A}_k\Big\}\notag\\
&\sum_{\ell}\upalpha_{B,\ell}\hspace{.03in}\chi\Big\{b_{\ell}^{(1)}\left((\hat{\bm{\uptheta}}_k^{\text{non}})^{[\bm{1}]},(\hat{\bm{\uptheta}}_k^{\text{non}})^{[\bm{2}]}\right)\in B^{(2)}\left((\hat{\bm{\uptheta}}_k^{\text{non}})^{[\bm{1}]},(\hat{\bm{\uptheta}}_k^{\text{non}})^{[\bm{2}]}\right)\cap\mathcal{A}_k\Big\}
\end{align*}
and $T_2(k)\geq 0$ is defined as:
\begin{align*}
T_2(k)&\equiv \sum_{\ell}\upbeta_{F,\ell}^{(1)}\hspace{.03in}\chi\Big\{f_{\ell}^{(1)}\left((\hat{\bm{\uptheta}}_k^{\text{cyc}})^{[\bm{1}]}\right)\in F^{(2)}\left((\hat{\bm\uptheta}_{k}^{\text{cyc}})^{[\bm{1}]}\right)\cap\mathcal{B}_k\Big\}.
\end{align*}
\end{proposition}
\begin{proof}
First, (\ref{eq:nonshared}) and (\ref{eq:cycshared}) imply $T_1(k)=c_{\text{share}}^{\text{non}}(k)-c_{\text{share}}^{\text{cyc}}(k)$ while (\ref{eq:halfdome}) and (\ref{eq:tenk}) imply $T_2(k)=c_{\text{carry}}^{\text{cyc}}(k)-c_{\text{carry}}^{\text{non}}(k)$. The result then follows from (\ref{eq:opahhh0},b).
\end{proof}

To facilitate the interpretation of (\ref{eq:core}) we assume the following:
\begin{DESCRIPTION}
\item[D6]  $\upalpha_{F,\ell}^{(1)}=\upbeta_{F,\ell}^{(1)}$. 
This assumption is satisfied whenever the cost saved by sharing a computation involving only $\bm\uptheta^{[\bm{1}]}$ is independent of whether the computation was performed during the $k$th iteration or during an earlier iteration.
\end{DESCRIPTION}
Then, (\ref{eq:core}) becomes
\begin{align*}
&c_k^{\text{non}}-c_k^{\text{cyc}}= \sum_{\ell}\upalpha_{F,\ell}^{(1)}\hspace{.03in}\chi\Big\{f_{\ell}^{(1)}\left((\hat{\bm{\uptheta}}_k^{\text{cyc}})^{[\bm{1}]}\right)\in F^{(2)}\left((\hat{\bm\uptheta}_{k}^{\text{cyc}})^{[\bm{1}]}\right)\cap\mathcal{B}_k\Big\}\notag\\
& -\sum_{\ell}\upalpha_{F,\ell}\hspace{.03in}\chi\Big\{f_{\ell}^{(1)}\left((\hat{\bm{\uptheta}}_k^{\text{non}})^{[\bm{1}]}\right)\in F^{(2)}\left((\hat{\bm{\uptheta}}_k^{\text{non}})^{[\bm{1}]}\right)\cap \mathcal{A}_k\Big\}\notag\\
&-\sum_{\ell}\upalpha_{B,\ell}\hspace{.03in}\chi\Big\{b_{\ell}^{(1)}\left((\hat{\bm{\uptheta}}_k^{\text{non}})^{[\bm{1}]},(\hat{\bm{\uptheta}}_k^{\text{non}})^{[\bm{2}]}\right)\in B^{(2)}\left((\hat{\bm{\uptheta}}_k^{\text{non}})^{[\bm{1}]},(\hat{\bm{\uptheta}}_k^{\text{non}})^{[\bm{2}]}\right)\cap\mathcal{A}_k\Big\}.
\end{align*}
Furthermore, provided it is possible to share $f_{\ell}^{(1)}\left(\bm\uptheta^{[\bm{1}]}\right)$ if and only if it is possible to carry $f_{\ell}^{(1)}\left(\bm\uptheta^{[\bm{1}]}\right)$ from the previous iteration, then
 the difference between $c_k^{\text{cyc}}$ and $c_k^{\text{non}}$ under D0--D6 is given by:
\begin{align*}
c_k^{\text{cyc}}-c_k^{\text{non}}=\sum_{\ell}\upalpha_{B,\ell}\hspace{.03in}\chi\Big\{b_{\ell}^{(1)}\left((\hat{\bm{\uptheta}}_k^{\text{non}})^{[\bm{1}]},(\hat{\bm{\uptheta}}_k^{\text{non}})^{[\bm{2}]}\right)\in B^{(2)}\left((\hat{\bm{\uptheta}}_k^{\text{non}})^{[\bm{1}]},(\hat{\bm{\uptheta}}_k^{\text{non}})^{[\bm{2}]}\right)\cap\mathcal{A}_k\Big\}.
\end{align*}
Here, we can see that under conditions D0--D6 the main difference between the arithmetic cost of a cyclic algorithm and that of its non-cyclic counterpart is due to sharing elements that combine information regarding both subvectors of $\bm\uptheta$. Intuitively, this is due to the fact that by implementing an algorithm in a cyclic manner we lose the ability to reuse any computations involving both $\bm\uptheta^{[\bm{1}]}$ and $\bm\uptheta^{[\bm{2}]}$. Therefore, if a significant cost reduction is achieved by sharing this type of element then an iteration of the non-cyclic algorithm will have a significantly lower cost. 

In practice, the per-iteration cost of implementation will depend strongly on the specific algorithm considered. For the cyclic seesaw SPSA algorithm, for example, a more appropriate definition of cost might be the number of noisy loss function evaluations used per-iteration. If cyclic seesaw SPSA is implemented according to (\ref{eq:moumousika}) along with (\ref{eq:howeardspos1}) and (\ref{eq:howeardspos2}) then the per-iteration cost of implementation would be equal to 4 loss function measurements per-iteration. In the case of the regular SPSA algorithm (see (\ref{eq:pocres})), however, the per-iteration cost of implementation would be only 2 loss function measurements per-iteration. The following section presents some general results on the relative asymptotic efficiency in (\ref{eq:ratio}).

\section{Approximating Relative Efficiency}
\label{sec:aseficwiener}

 In the previous chapter (see Section \ref{sec:upperpotomacshell}) it was shown that (under certain assumptions) $k^{\upbeta/2}(\hat{\bm\uptheta}_k^{\text{cyc}}-\bm\uptheta^\ast)\xrightarrow{\text{\ dist\ }}\mathcal{N}({\bm\upmu}^{\text{cyc}},{\bm\Sigma}^{\text{cyc}})$
for some mean vector $\bm\upmu^{\text{cyc}}$ and covariance matrix $\bm\Sigma^{\text{cyc}}$, where $\hat{\bm\uptheta}_k^{\text{cyc}}$ denotes the $k$th iterate of Algorithm \ref{kirkey} and the notation ``$\xrightarrow{\text{\ dist\ }}$'' means convergence in distribution. A similar result is readily obtained for its non-cyclic counterpart. In particular, let $\hat{\bm{\uptheta}}_k^{\text{non}}$ denote the $k$th iterate obtained using the non-cyclic counterpart of Algorithm \ref{kirkey}. Then, under certain assumptions (see Fabian 1968)\nocite{fabian1968} it can be shown that $k^{\upbeta/2}(\hat{\bm\uptheta}_k^{\text{non}}-\bm\uptheta^\ast)\xrightarrow{\text{\ dist\ }} \mathcal{N}({\bm\upmu}^{\text{non}},{\bm\Sigma}^{\text{non}})$
for some mean vector ${\bm\upmu}^{\text{non}}$ and covariance matrix ${\bm\Sigma}^{\text{non}}$. This section shows how results on asymptotic normality can be used to approximate the relative efficiency in (\ref{eq:ratio}). 

 Following the ideas in (Spall, 1992) we assume:
\begin{align}
\label{eq:thressandfive}
\lim_{k\rightarrow \infty}E[k^{\upbeta/2}(\hat{\bm\uptheta}_k^{\text{cyc}}-\bm\uptheta^\ast)]=\bm\upmu^{\text{cyc}}, \ \ \ \lim_{k\rightarrow \infty}E[k^{\upbeta/2}(\hat{\bm\uptheta}_k^{\text{non}}-\bm\uptheta^\ast)]={\bm\upmu}^{\text{non}},
\end{align}
and that:
\begin{align}
\lim_{k\rightarrow \infty}\Var{[k^{\upbeta/2}(\hat{\bm\uptheta}_k^{\text{cyc}}-\bm\uptheta^\ast)]}=\bm\Sigma^{\text{cyc}},\ \ \ \lim_{k\rightarrow \infty}\Var{[k^{\upbeta/2}(\hat{\bm\uptheta}_k^{\text{non}}-\bm\uptheta^\ast)]}=\bm\Sigma^{\text{non}}.
\label{eq:suchawalkschon}
\end{align}
Spall (1992) discusses how (\ref{eq:thressandfive}) and (\ref{eq:suchawalkschon}) hold if the terms $\|k^{\upbeta/2}(\hat{\bm\uptheta}_k^{\text{cyc}}-\bm\uptheta^\ast)\|^2$ and $\|k^{\upbeta/2}(\hat{\bm\uptheta}_k^{\text{non}}-\bm\uptheta^\ast)\|^2$ are uniformly integrable (see also Laha and Rohatgi, pp. 138--140). Equations (\ref{eq:thressandfive}) and (\ref{eq:suchawalkschon}) allow us to obtain the following approximation for large $k$:
\begin{align*}
E\|\hat{\bm\uptheta}_{k}^{\text{cyc}}-\bm\uptheta^\ast\|^2=&\ \trace{\left[E(\hat{\bm\uptheta}_k^{\text{cyc}}-\bm\uptheta^\ast)(\hat{\bm\uptheta}_k^{\text{cyc}}-\bm\uptheta^\ast)^\top\right]}
\approx k^{-\upbeta}\trace{\left[\bm\Sigma^{\text{cyc}}+\bm\upmu^{\text{cyc}}(\bm\upmu^{\text{cyc}})^\top\right]}.
\end{align*}
Similarly, for large $k$ we approximate:
\begin{align*}
E\|\hat{\bm\uptheta}_{k}^{\text{non}}-\bm\uptheta^\ast\|^2
\approx k^{-\upbeta}\trace{\left[\bm\Sigma^{\text{non}}+\bm\upmu^{\text{non}}(\bm\upmu^{\text{non}})^\top\right]}.
\end{align*}
Then, for large $k_1$ and $k_2$ we may approximate (\ref{eq:ratio}) as follows:
\begin{align}
\frac{E\|\hat{\bm\uptheta}_{k_1}^{\text{cyc}}-\bm\uptheta^\ast\|^2}{E\|\hat{\bm\uptheta}_{k_2}^{\text{non}}-\bm\uptheta^\ast\|^2}\approx \left(\frac{k_2}{k_1}\right)^\upbeta\frac{\trace{\left[\bm\Sigma^{\text{cyc}}\right]}+[\bm\upmu^{\text{cyc}}]^\top\bm\upmu^{\text{cyc}}}{\trace{\left[\bm\Sigma^{\text{non}}\right]}+[{\bm\upmu}^{\text{non}}]^\top{\bm\upmu}^{\text{non}}}.
\label{eq:michbuuubl}
\end{align}
Here, we remind the reader that in order to perform a fair comparison between the cyclic and non-cyclic algorithm, the cumulative cost of implementing the cyclic algorithm up to iteration $k_1$ is assumed to be equal (or approximately equal) to that of its non-cyclic counterpart up to iteration $k_2$; an analysis of the asymptotic properties of (\ref{eq:michbuuubl}) would therefore require increasing $k_1$ and $k_2$ at a rate that ensures the comparison is fair at all times. With this in mind, we introduce the following proposition.
\begin{proposition}
\label{prop:omygodsthunder}
Let  $0<c^{\text{cyc}}$ and $0<c^{\text{non}}$ be finite constants such that
 $c^{\text{cyc}}/c^{\text{non}}$ is a rational number. Then, there exists an infinite sequence $\{(k_{1(i)},k_{2(i)})\}_{i=1}^\infty$ in $\mathbb{Z}^+\times\mathbb{Z}^+$ with $k_{1(i-1)}<k_{1(i)}$ and $k_{2(i-1)}<k_{2(i)}$ such that $c^{\text{cyc}}k_{1(i)}=c^{\text{non}}k_{2(i)}$.
\end{proposition}
\begin{proof}
Let $c^{\text{cyc}}/c^{\text{non}}=p/q$ where $p$ and $q$ are strictly positive integers. The desired sequence can be obtained by letting $k_{1(i)}=qi$ and $k_{2(i)}=pi$.
\end{proof}

The significance of Proposition \ref{prop:omygodsthunder} is that if the per-iteration costs $c_k^{\text{cyc}}$ and $c_k^{\text{non}}$ are nonzero, finite, independent of $k$ (so that $c_k^{\text{cyc}}=c^{\text{cyc}}$ and $c_k^{\text{non}}=c^{\text{non}}$), and if $c^{\text{cyc}}/c^{\text{non}}$ is a rational number, then the cost of computing $\hat{\bm{\uptheta}}_{k_{1(i)}}^{\text{cyc}}$ is exactly equal to the cost of computing $\hat{\bm{\uptheta}}_{k_{2(i)}}^{\text{non}}$ for any strictly-positive integer $i$ when $k_{1(i)}$ and $k_{2(i)}$ are as in Proposition \ref{prop:omygodsthunder}. We can then obtain the following result:
\begin{align}
\label{eq:wobble}
\lim_{i\rightarrow \infty}\frac{E\|\hat{\bm\uptheta}_{k_{1(i)}}^{\text{cyc}}-\bm\uptheta^\ast\|^2}{E\|\hat{\bm\uptheta}_{k_{2(i)}}^{\text{non}}-\bm\uptheta^\ast\|^2}= \left(\frac{c^{\text{cyc}}}{c^{\text{non}}}\right)^\upbeta\frac{\trace{\left[{\bm\Sigma}^{\text{cyc}}\right]}+[\bm\upmu^{\text{cyc}}]^\top\bm\upmu^{\text{cyc}}}{\trace{\left[{\bm\Sigma}^{\text{non}}\right]}+[{\bm\upmu}^{\text{non}}]^\top{\bm\upmu}^{\text{non}}}.
\end{align}
The following section uses the limit in (\ref{eq:wobble}) to study the relative efficiency between a special case of Algorithm \ref{kirkey} and its non-cyclic counterpart.

\section{Relative Efficiency: A Special Case}
\label{sec:apsecialcaserelefff}

This section computes the asymptotic relative efficiency between a non-cyclic algorithm and a cyclic algorithm which is a special case of Algorithm \ref{kirkey} in which the sets $\mathcal{S}_j$ do not intersect (this is equivalent to saying that any single entry of $\hat{\bm{\uptheta}}_k$ can only be updated by one processor/agent although each processor/agent may update several entries) and updates are produced in a strictly cyclic manner; this cyclic algorithm is described in more detail in Algorithm \ref{kirkeyspecial}. Given our desire to analyze the effect of a cyclic implementation on asymptotic efficiency, an appropriate non-cyclic counterpart to Algorithm \ref{kirkeyspecial} is given in {\ref{alg:apeparingairshow}}. Here, in order to to isolate the effects of a cyclic implementation on asymptotic efficiency, a diagonal matrix of gains was used in an attempt to make the non-cyclic algorithm as similar to the cyclic algorithm as possible. In the remainder of this section we obtain expressions for the terms $\bm\Sigma^{\text{cyc}}$, $\bm\Sigma^{\text{non}}$, $\bm\upmu^{\text{cyc}}$, and $\bm\upmu^{\text{non}}$ (defined in the opening paragraph of Section \ref{sec:aseficwiener}) in order to estimate the relative efficiency between Algorithms \ref{kirkeyspecial} and \ref{alg:apeparingairshow} in the manner of (\ref{eq:wobble}).

 \begin{algorithm}[!t]                     
\caption{Strictly Cyclic Version of Algorithm \ref{kirkey}}      
\label{kirkeyspecial}                
\begin{algorithmic} [1]                  
\setstretch{1.5} 
\REQUIRE    $\hat{\bm{\uptheta}}_0^{\text{cyc}}$, $\{a_k^{(j)}\}_{k\geq 0}$ for $j=1,\dots,d$, and let $\mathcal{S}_1,\dots, \mathcal{S}_d$ be such that (\ref{eq:scurry}) holds. Set $k=0$.
\WHILE{stopping criterion has not been reached}
\STATE{Set $\hat{\bm\uptheta}_k^{(I_{0})}=\hat{\bm{\uptheta}}_k^{\text{cyc}}$.}
         \FOR{$j=1,\dots,d$}
	\STATE{Define $\hat{\bm\uptheta}_k^{(I_{j})}\equiv \hat{\bm\uptheta}_k^{(I_{j-1})}- {a}_{k}^{(j)}\ \hat{\bm{g}}^{(j)}\left(\hat{\bm\uptheta}_k^{(I_{j-1})}\right)$.
	}
	 \ENDFOR
	 \STATE{Let $\hat{\bm{\uptheta}}_{k+1}^{\text{cyc}}\equiv\hat{\bm{\uptheta}}_k^{(I_{d})}$.
	 }
	\STATE{set $k=k+1$}
\ENDWHILE
\end{algorithmic}
\end{algorithm}

%

The following corollary to Theorem \ref{thm:fnogg} will be used to derive expressions for  $\bm\Sigma^{\text{cyc}}$ and $\bm\upmu^{\text{cyc}}$, the parameters of the asymptotic distribution of $k^{\upbeta/2}(\hat{\bm\uptheta}_k^{\text{cyc}}-\bm\uptheta^\ast)$.

\begin{corollary}
\label{eq:hjointcoffees}
Let $\hat{\bm{\uptheta}}_k^{\text{cyc}}$ be generated according to Algorithm \ref{kirkeyspecial}. Additionally, let $\mathcal{F}_k$ denote the sigma field generated by $\{\hat{\bm{\uptheta}}_\ell^{\text{cyc}}\}_{\ell=0}^k$ as well as by any random variables generated  by the algorithm in the production of $\hat{\bm{\uptheta}}_k^{\text{cyc}}$. 
Also, assume the following conditions hold (all scalars, matrices, and vectors are assumed to have real entries):
%
\begin{description}
\item[C0$'$] $L(\bm\uptheta)$ is twice continuously differentiable, $\bm{H}(\bm\uptheta)$ has bounded entries, and $\hat{\bm\uptheta}_k^{(I_j)}\rightarrow \bm\uptheta^\ast$ w.p.1 for $j=1,\dots, d$.
\item[C1$'$] The same as C1.
\begin{algorithm}[!t]                     
\caption{Non-Cyclic Counterpart to Algorithm \ref{kirkeyspecial}}      
\label{alg:apeparingairshow}                
\begin{algorithmic} [1]                  
\setstretch{1.5} 
\REQUIRE    $\hat{\bm{\uptheta}}_0^{\text{non}}$, $\{a_k^{(j)}\}_{k\geq 0}$ for $j=1,\dots,d$, let $\mathcal{S}_1,\dots, \mathcal{S}_d$ be such that (\ref{eq:scurry}) holds, and let $\bm{A}_k$ be a diagonal matrix with $i$th diagonal entry equal to $a_k^{(j)}$, where $j$ is such that $i\in \mathcal{S}_j$. Set $k=0$.
\WHILE{stopping criterion has not been reached}
	 \STATE{Let $\hat{\bm{\uptheta}}_{k+1}^{\text{non}}=\hat{\bm{\uptheta}}_k^{\text{non}}-\bm{A}_k\hat{\bm{g}}(\hat{\bm{\uptheta}}_k^{\text{non}})$.
	 }
	\STATE{set $k=k+1$}
\ENDWHILE
\end{algorithmic}
\end{algorithm}
\item[C2$'$]  There exists  an diagonal matrix $\bm{\Lambda}\in \mathbb{R}^{p\times p}$ with strictly positive eigenvalues and a nonsingular matrix $\bm{S}\in \mathbb{R}^{p\times p}$ such that
$\bm\Gamma\equiv \sum_{j=1}^{d} r_{j} \bm{J}^{(j)}(\bm\uptheta^\ast)=\bm{AH}(\bm\uptheta^\ast)=\bm{S\Lambda S}^{-1}$ where $\bm{A}$ is the diagonal matrix with $i$th diagonal entry equal to $r_i$, and where $\bm{J}^{(j)}(\bm\uptheta)$ denotes the Jacobian of $\bm{g}^{(j)}(\bm\uptheta)$.
\item[C3$'$]
 $\upbeta$ and $\upalpha$ satisfy $0<\upbeta\leq \upalpha$. Additionally, for all $j$ there exist a finite vector $\bm{b}^{(j)}\in \mathbb{R}^p$ such that 
  $k^{\upbeta/2}{\bm{\upbeta}}^{(j)}(\hat{\bm\uptheta}_k^{(I_{j})})\rightarrow \bm{b}^{(j)}$ w.p.1 for all $j$.
\item[C4$'$] For all $j$,
  $E[{\bm{\upxi}}^{(j)}(\hat{\bm\uptheta}_k^{(I_{j})})|\hat{\bm\uptheta}_k^{(I_{j})}]=\bm{0}$. 
\item[C5$'$]  There exists a constant $C>0$ and a matrix $\bm\Sigma^{(j)}$ such that one of the following holds w.p.1:
\begin{enumerate}[label=(\roman*)]
\item $\upbeta=\upalpha$, $C>\|E[{\bm{\upxi}}^{(j)}\big(\hat{\bm\uptheta}_k^{(I_{j})}\big)[{\bm{\upxi}}^{(j)}\big(\hat{\bm\uptheta}_k^{(I_{j})}\big)]^\top|\mathcal{F}_k]-\bm\Sigma^{(j)}\|\rightarrow 0$.
\item $\upbeta<\upalpha$, $C>\|k^{\upbeta-\upalpha}E[{\bm{\upxi}}^{(j)}\big(\hat{\bm\uptheta}_k^{(I_{j})}\big)[{\bm{\upxi}}^{(j)}\big(\hat{\bm\uptheta}_k^{(I_{j})}\big)]^\top|\mathcal{F}_k]-\bm\Sigma^{(j)}\|\rightarrow 0$.
\end{enumerate}
Additionally, if $\bm{v}_1$ and $\bm{v}_2$ are two different elements of $\{{\bm{\upxi}}_k^{(j)}\big(\hat{\bm\uptheta}_k^{(I_j)}\big)\}_{j}$, then one of the following holds w.p.1:
\begin{enumerate}[label=(\roman*)]
\setcounter{enumi}{2}
\item C5$'$-(i) holds and $C>\|E[\bm{v}_1\bm{v}_2^\top|\mathcal{F}_k]\|\rightarrow 0$.
\item C5$'$-(ii) holds and $C>\|k^{\upbeta-\upalpha}E[\bm{v}_1\bm{v}_2^\top|\mathcal{F}_k]\|\rightarrow 0$.
\end{enumerate}
\item[C6$'$] Let $\bm{M}_B\equiv \sum_{j=1}^d\bm\Sigma^{(j)}$. The matrix $\bm{A}\bm{M}_B\bm{A}$ is positive definite.
\item[C7$'$] One of the following holds:
\begin{enumerate}[label=(\roman*)]
\item $\upbeta=\upalpha$ and B5-(i) holds with $\bm{V}_k$ replaced by ${\bm{\upxi}}^{(j)}\big(\hat{\bm\uptheta}_k^{(I_{j})}\big)$.
\item $\upbeta<\upalpha$ and B5-(i) holds with $\bm{V}_k$ replaced by $k^{(\upbeta-\upalpha)/2}{\bm{\upxi}}^{(j)}\big(\hat{\bm\uptheta}_k^{(I_{j})}\big)$.
\end{enumerate}
\item[C8$'$] The same as B6$'$ with $U_{ii}$ replaced by $\Lambda_{ii}$.
\end{description}
Then, $k^{\upbeta/2}(\hat{\bm\uptheta}_k^{\text{cyc}}-\bm\uptheta^\ast)\xrightarrow{\text{\ dist\ }}\mathcal{N}({\bm\upmu}^\text{cyc},{\bm\Sigma}^\text{cyc})$
 with $\bm\upmu^{\text{cyc}}= -\bm{S}[\bm{\Lambda}-\upbeta_+/2\bm{I}]^{-1} \bm{S}^{-1}\bm{A}\bm{b}^\text{cyc}$ and $\bm\Sigma^{\text{cyc}}=\bm{SQ}^{\text{cyc}}\bm{S}^\top$, where
\begin{align}
\label{eq:jepanuj}
\bm{b}^\text{cyc}\equiv\sum_{j=1}^d \bm{b}^{(j)}, \ \ \ {Q}^\text{cyc}_{ij}&=\frac{\left[\bm{S}^{-1}\bm{A}\bm{M}_B\bm{A}(\bm{S}^{-1})^\top\right]_{ij}}{\Lambda_{ii}+\Lambda_{jj}-\upbeta_+}.
\end{align}
\end{corollary}
\begin{proof}
Since C0$'$--C8$'$ imply C0--C8, the result follows from Theorem \ref{thm:fnogg}.
\end{proof}

The following corollary to Theorem \ref{thm:fnogg} will be used to derive expressions for  $\bm\Sigma^{\text{non}}$ and $\bm\upmu^{\text{non}}$, the parameters of the asymptotic distribution of $k^{\upbeta/2}(\hat{\bm\uptheta}_k^{\text{non}}-\bm\uptheta^\ast)$.

\begin{corollary}
\label{eq:tarattatarantanand}
Let $\hat{\bm{\uptheta}}_k^{\text{non}}$ be generated according to Algorithm \ref{alg:apeparingairshow}. Additionally, let $\mathcal{F}_k$ denote the sigma field generated by $\{\hat{\bm{\uptheta}}_\ell^{\text{non}}\}_{\ell=0}^k$ as well as by any random variables generated  by the algorithm in the production of $\hat{\bm{\uptheta}}_k^{\text{non}}$. Also, assume the following conditions hold (all scalars, matrices, and vectors are assumed to have real entries):
%
\begin{description}
\item[C0$''$] $L(\bm\uptheta)$ is twice continuously differentiable, $\bm{H}(\bm\uptheta)$ has bounded entries, and $\hat{\bm\uptheta}_k^{\text{non}}\rightarrow \bm\uptheta^\ast$ w.p.1.
\item[C3$''$]  $\upbeta$ and $\upalpha$ satisfy $0<\upbeta\leq \upalpha$. Additionally, there exist a finite vector $\bm{b}^{\text{non}}\in \mathbb{R}^p$ such that  $k^{\upbeta/2}{\bm{\upbeta}}\big(\hat{\bm\uptheta}_k\big)\rightarrow \bm{b}^{\text{non}}$ w.p.1.
\item[C4$''$]  $E[{\bm{\upxi}}\big(\hat{\bm\uptheta}_k\big)|\hat{\bm{\uptheta}}_k]=\bm{0}$.
\item[C5$''$]  There exists a constant $C>0$ and a matrix $\bm{M}$ such that one of the following holds w.p.1:
\begin{enumerate}[label=(\roman*)]
\item $\upbeta=\upalpha$, $C>\|E[{\bm{\upxi}}\big(\hat{\bm\uptheta}_k\big)[{\bm{\upxi}}\big(\hat{\bm\uptheta}_k\big)]^\top|\mathcal{F}_k]-\bm{M}\|\rightarrow 0$.
\item $\upbeta<\upalpha$, $C>\|k^{\upbeta-\upalpha}E[{\bm{\upxi}}\big(\hat{\bm\uptheta}_k\big)[{\bm{\upxi}}\big(\hat{\bm\uptheta}_k\big)]^\top|\mathcal{F}_k]-\bm{M}\|\rightarrow 0$.
\end{enumerate}
\item[C6$''$] $\bm\Sigma\equiv \bm{A}{\bm{M}}\bm{A}$ is a positive definite matrix.
\item[C7$''$] One of the following holds:
\begin{enumerate}[label=(\roman*)]
\item $\upbeta=\upalpha$ and B5-(i) holds with $\bm{V}_k$ replaced by ${\bm{\upxi}}\big(\hat{\bm\uptheta}_k\big)$.
\item $\upbeta<\upalpha$ and B5-(i) holds with $\bm{V}_k$ replaced by $k^{(\upbeta-\upalpha)/2}{\bm{\upxi}}\big(\hat{\bm\uptheta}_k\big)$.
\end{enumerate}
\item[C1$''$, C2$''$, \& C8$''$] The same as C1$'$, C2$'$, and C8$'$, respectively.
\end{description}
Then, $k^{\upbeta/2}(\hat{\bm\uptheta}_k^{\text{non}}-\bm\uptheta^\ast)\xrightarrow{\text{\ dist\ }}\mathcal{N}({\bm\upmu}^\text{non},{\bm\Sigma}^\text{non})$ with ${\bm\upmu}^\text{non}= -\bm{S}[\bm{\Lambda}-\upbeta_+/2\bm{I}]^{-1} \bm{S}^{-1}\bm{A}\bm{b}^\text{non}$ and $\bm\Sigma^{\text{non}}=\bm{SQ}^{\text{non}}\bm{S}^\top$, where
%
\begin{align}
\label{eq:diamonddrone}
{Q}^\text{non}_{ij}=\frac{\left[\bm{S}^{-1}\bm{A}\bm{M}\bm{A}(\bm{S}^{-1})^\top\right]_{ij}}{\Lambda_{ii}+\Lambda_{jj}-\upbeta_+}.
\end{align}
\end{corollary}
\begin{proof}
Since C0$''$--C8$''$ imply C0--C8, the result follows from Theorem \ref{thm:fnogg}
\end{proof}

The validity of C0$'$--C8$'$ and C0$''$--C8$''$ is directly related to the validity of C0--C8 (see Section \ref{sec:bonestrailseason} for an in-depth discussion on this topic). Using the expressions for $\bm\Sigma^{\text{cyc}}$, $\bm\upmu^{\text{cyc}}$, $\bm\Sigma^{\text{non}}$, and $\bm\upmu^{\text{non}}$ from Corollaries \ref{eq:hjointcoffees} and \ref{eq:tarattatarantanand} along with the values of $c^\text{non}$ and $c^\text{cyc}$ we can compute the estimate in (\ref{eq:wobble}). In general, however, the study of this ratio is rather complicated. In order to obtain a simple expression for (\ref{eq:wobble}) we will consider a special case.
 
\begin{proposition}
\label{prop:exploringeoss}
Let $\hat{\bm{\uptheta}}_k^{\text{cyc}}$ be obtained according to Algorithm \ref{kirkeyspecial} and let $\hat{\bm{\uptheta}}_k^{\text{non}}$ be obtained according to Algorithm \ref{alg:apeparingairshow}. Assume the following:
\begin{DESCRIPTION}
\item[E0] Conditions C0$'$--C8$'$ and C0$''$--C8$''$ hold.
\item[E1]  $\upbeta>0$ and ${a}_k^{(j)}>0$ for $j=1,\dots,d$ are the same for both Algorithms \ref{kirkeyspecial} and \ref{alg:apeparingairshow}.
\item[E2] $\bm{b}^{\text{cyc}}=\bm{b}^{\text{non}}$ (e.g., as in the SG and CSG algorithms).
\item[E3] At least one of the following holds:
\begin{enumerate}[label=(\roman*)]
\item $\bm{S}=c\bm{I}$ for some constant $c$ (i.e., $\bm{H}(\bm\uptheta^\ast)$ is a diagonal matrix).
\item $\bm{M}_B=\bm{M}$ (e.g., when the entries of the $\bm\upxi_k$ are independent given $\mathcal{F}_k$).
\end{enumerate}
\end{DESCRIPTION}
Then, when $k_{1(i)}$ and $k_{2(i)}$ are as in Proposition \ref{prop:omygodsthunder} the limit in (\ref{eq:wobble}) becomes:
\begin{align}
\label{eq:horsesrrsesriding}
\lim_{i\rightarrow \infty}\frac{E\|\hat{\bm\uptheta}_{k_{1(i)}}^{\text{cyc}}-\bm\uptheta^\ast\|^2}{E\|\hat{\bm\uptheta}_{k_{2(i)}}^{\text{non}}-\bm\uptheta^\ast\|^2}= \left(\frac{c^{\text{cyc}}}{c^{\text{non}}}\right)^\upbeta.
\end{align}
\end{proposition}
\begin{proof}
First, note that E0 implies the results of Corollaries \ref{eq:hjointcoffees} and \ref{eq:tarattatarantanand} hold. Then, E1 and E2 imply $\bm\upmu^{\text{cyc}}=\bm\upmu^{\text{non}}$. Next, E2 along with (\ref{eq:jepanuj}) and (\ref{eq:diamonddrone}) implies that $\trace{[{\bm\Sigma}^{\text{cyc}}]}=\trace{[{\bm\Sigma}^{\text{non}}]}$. The relationship in (\ref{eq:horsesrrsesriding}) then follows from (\ref{eq:wobble}).
\end{proof}

An implication of Proposition \ref{prop:exploringeoss} is that the asymptotic relative efficiency is entirely determined by $c^{\text{cyc}}$ and $c^{\text{non}}$. If ${c^{\text{cyc}}}/{c^{\text{non}}}>1$ then the non-cyclic algorithm is more efficient, if ${c^{\text{cyc}}}/{c^{\text{non}}}<1$ then the cyclic algorithm is more efficient, and if ${c^{\text{cyc}}}/{c^{\text{non}}}=1$ then both algorithms are equally efficient. 
%
In general, however, it is not easy to see whether ratio in (\ref{eq:wobble}) is greater than one, less than one, or equal to one. Next we show that when E3 does not hold
 then having ${c^{\text{cyc}}}/{c^{\text{non}}}<1$ does not imply the cyclic algorithm is asymptotically more efficient. Similarly, having ${c^{\text{cyc}}}/{c^{\text{non}}}>1$ does not imply the non-cyclic algorithm is asymptotically more efficient and having ${c^{\text{cyc}}}/{c^{\text{non}}}=1$ does not imply both algorithms are asymptotically equally efficient.


Let us assume that condition E3 does not hold. Specifically, the matrix $\bm{S}$ is not a multiple of the identity matrix and that $\bm{M}_B\neq \bm{M}$. In this case, the limit of the ratio in (\ref{eq:wobble}) may be greater than one even when ${c^{\text{cyc}}}/{c^{\text{non}}}<1$, less than one even when ${c^{\text{cyc}}}/{c^{\text{non}}}>1$, and different from one even when ${c^{\text{cyc}}}/{c^{\text{non}}}=1$. To illustrate this fact we consider the special case where $\bm{b}^{\text{cyc}}$ and $\bm{b}^{\text{non}}$, the asymptotic values of normalized bias vectors (see (\ref{eq:jepanuj}) and conditions C3$'$ and C3$''$ for the formal definitions of these vectors), are identically $\bm{0}$. This is a valid assumption when the update directions are SG-based. Here,
\begin{align}
\label{eq:provewrongnewsjoey}
\lim_{i\rightarrow \infty}\frac{E\|\hat{\bm\uptheta}_{k_{1(i)}}^{\text{cyc}}-\bm\uptheta^\ast\|^2}{E\|\hat{\bm\uptheta}_{k_{2(i)}}^{\text{non}}-\bm\uptheta^\ast\|^2}= \left(\frac{c^{\text{cyc}}}{c^{\text{non}}}\right)^\upbeta\frac{\trace{\left[{\bm\Sigma}^{\text{cyc}}\right]}}{\trace{\left[{\bm\Sigma}^{\text{non}}\right]}}.
\end{align}
%
%
Furthermore, we assume $\bm{A}=\bm{I}$, $\upalpha<1$ (so that $\upbeta_+=0$ in condition B6$'$) and that the block-diagonal matrix $\bm{M}_B$ (appearing in condition C5$'$) coincides with $\bm{M}$ (appearing in condition  C5$''$) on the block-diagonal part of $\bm{M}_B$. Specifically, there exist block matrices $\bm{B}_1,\dots,\bm{B}_p$ such that:
\begin{align}
\label{eq:timewillshowr}
\setstretch{1.8} 
\bm{M}_B=\left[\begin{array}{ccc}\bm{B}_1 & \bm{0} & \bm{0} \\\bm{0} & \ddots & \bm{0} \\\bm{0} & \bm{0} & \bm{B}_d\end{array}\right], \ \ \ \bm{M}=\left[\begin{array}{ccc}\bm{B}_1 & \bm{\ast} & \bm{\ast} \\\bm{\ast} & \ddots & \bm{\ast} \\\bm{\ast} & \bm{\ast} & \bm{B}_d\end{array}\right],
\end{align}
where $\bm\ast$ denote possibly non-zero entries and $\bm{0}$ denotes a matrix of zeros of the appropriate dimension. Let us discuss a situation where the assumption in (\ref{eq:timewillshowr}) is satisfied. First, under conditions C4$'$ and C4$''$ it follows from (\ref{eq:avatarparticulatsfs}) that $\bm{M}_B$ corresponds to the limiting (conditional) covariance of $k^{(\upbeta-\upalpha)/2}[\bm\upxi_k^{(1)}(\hat{\bm{\uptheta}}_k^{(I_1)})+\cdots+\bm\upxi_k^{(d)}(\hat{\bm{\uptheta}}_k^{(I_d)})]$ while $\bm{M}$ represents the limiting (conditional) covariance of $k^{(\upbeta-\upalpha)/2}\bm\upxi_k(\hat{\bm{\uptheta}}_k^{\text{non}})$. Under condition C5$'$-(iii) and C5$'$-(iv), however,
 it follows that $\bm{M}_B$ must be a block diagonal matrix (i.e., $\bm{M}_B$ must have the form in (\ref{eq:timewillshowr})), this contrasts with the matrix $\bm{M}$ which may have extra nonzero entries due to the possible asymptotic correlation between the entries of $\bm\upxi_k$.

  Consider now the following example. Let :
\begin{align}
\label{eq:firstcontactwars}
\setstretch{1.25} 
\bm{H}(\bm\uptheta^\ast)=\left[\begin{array}{cc}2 & -1 \\-1 & 2\end{array}\right],\ \ \ \bm{M}_B=\left[\begin{array}{cc}2 & 0 \\0 & 2\end{array}\right],\ \ \  \bm{M}=\left[\begin{array}{cc}2 & 1 \\1 & 2\end{array}\right].
\end{align}
The positive definite matrix $\bm{H}(\bm\uptheta^\ast)$ (with eigenvalues 1 and 3) might correspond, for example, to the loss function $L(\bm\uptheta)=\bm\uptheta^\top \bm{H}(\bm\uptheta^\ast)\bm\uptheta$. Then $\bm\Gamma=\bm{AH}(\bm\uptheta^\ast)=\bm{H}(\bm\uptheta^\ast)$ is positive definite and satisfies:
\begin{align*}
\setstretch{1.25} 
\bm\Gamma=\bm{S}\left[\begin{array}{cc}1 & 0 \\0 & 3\end{array}\right]\bm{S}^{-1},\text{ with }\bm{S}=\frac{1}{\sqrt{2}}\left[\begin{array}{cc}-1 & -1 \\-1 & 1\end{array}\right].
\end{align*}
Under the above conditions and using the definitions of $\bm\Sigma^{\text{cyc}}$ and $\bm\Sigma^{\text{non}}$ given in Corollaries \ref{eq:hjointcoffees} and \ref{eq:tarattatarantanand}: $\trace{\left[\bm\Sigma^{\text{cyc}}\right]}/\trace{\left[\bm\Sigma^{\text{non}}\right]}=4/5$. Here, (\ref{eq:provewrongnewsjoey}) implies that the {\it{cyclic algorithm is asymptotically more efficient if and only if}} $c^{\text{cyc}}/c^{\text{non}}<(5/4)^{-\upbeta}$. 
 Now, say we modify $\bm{M}_B$ and $\bm{M}$ as follows (while keeping $\bm{H}(\bm\uptheta^\ast)$ as in (\ref{eq:firstcontactwars})):
\begin{align*}
\setstretch{1.25} 
\bm{M}_B=\left[\begin{array}{cc}2 & 0 \\0 & 2\end{array}\right], \ \ \ \bm{M}=\left[\begin{array}{cc}2 & -1 \\-1 & 2\end{array}\right]
\end{align*}
Then, $\trace{\left[\bm\Sigma^{\text{cyc}}\right]}/\trace{\left[\bm\Sigma^{\text{non}}\right]}=4/3$. In this case, (\ref{eq:provewrongnewsjoey}) implies that the {\it{cyclic}} {\it{algorithm is asymptotically more efficient if and only if}} $c^{\text{cyc}}/c^{\text{non}}<(3/4)^{-\upbeta}$. 
%
%
This example illustrates the fact that there is a delicate balance between $\bm{S}$, $\bm\Lambda$, $\bm{A}$, $\bm{M}$, $\upbeta$, and ${c^{\text{cyc}}}/{c^{\text{non}}}$ that determines which algorithm is more efficient. 

%

\section{Concluding Remarks}
\label{eq:hodgmodgepodge}

In order to perform a fair comparison between any two algorithms it is necessary to define the cost of implementing each algorithm. Generally, the definition of ``cost'' is highly problem-specific. This chapter begins by considering the case where cost is a measure of the number of arithmetic operations needed to perform a single iteration. Under certain assumptions, Section \ref{sec:costoport} pinpointed the type of arithmetic operations that result in a difference between the cost of implementing the cyclic seesaw algorithm and the cost of implementing its non-cyclic counterpart. This chapter also provided an approximation to the limiting value of the relative efficiency in (\ref{eq:ratio}), which takes into consideration the per-iteration costs of implementing the cyclic- and non-cyclic algorithms, for comparing Algorithm \ref{kirkeyspecial} to Algorithm \ref{alg:apeparingairshow}. The approximation to the asymptotic relative efficiency was based on the asymptotic normality results from Section \ref{sec:upperpotomacshell}. This chapter also showed that the approximate asymptotic relative efficiency is, under certain conditions, entirely determined by the per-iteration costs of implementing the two algorithms. In general, however, there is a delicate balance between the gain sequences, the covariance matrix of the noise terms, the Hessian matrix of $L(\bm\uptheta)$, and the per-iteration implementation costs that determines which method is more efficient.

%% file: chapterNUMERICS.tex

\chapter{Numerical Analysis}
\label{chap:numericschap}

Chapter \ref{sec:cyclicseesaw} of this dissertation derived conditions for the convergence of the GCSA algorithm, Chapter \ref{sec:ROC} derived conditions for the asymptotic normality of a special case of GCSA, and Chapter \ref{chap:imefficient} discussed the asymptotic relative efficiency between a special case of GCSA and its non-cyclic counterpart. This chapter contains numerical results that illustrate the theory of Chapters \ref{sec:cyclicseesaw}--\ref{chap:imefficient}. Specifically, Section \ref{sec:convergencenumericos} contains numerical results on the convergence of the GCSA algorithm, Section \ref{sec:normalitycenumericos} contains results on asymptotic normality, and Section \ref{sec:Numericssimples} contains results on relative efficiency.

\section{Numerical Examples on Convergence}
\label{sec:convergencenumericos}

This section provides a few numerical examples on the convergence of Algorithms \ref{beastwasdone} (where the subvector to update is selected at random) and \ref{kirkey} (where the subvector to update is selected following a deterministic pattern). Each of these two algorithms is implemented using SPSA-based noisy update directions (Section \ref{sec:clubsodeashrimp}) and SG-based noisy update directions (Section \ref{sec:omgimanothersectionsssd}). Algorithms \ref{beastwasdone} and \ref{kirkey} are also compared to two non-cyclic variants.

\subsection{Algorithms \ref{beastwasdone} \& \ref{kirkey} with SPSA-Based Gradient Estimates}
\label{sec:clubsodeashrimp}

This section considers the problem of stochastic optimization using Algorithms \ref{beastwasdone} and \ref{kirkey} (special cases of GCSA) with SPSA-based gradient estimates of the form in (\ref{eq:howeardspos1}) and (\ref{eq:howeardspos2}). Here only the subvector to update is perturbed at each subiteration. The loss function to minimize is the {\it{skewed quartic}} function (Spall 2003, Chapter 6):
\begin{align}
\label{func:skewedquartic}
L(\bm\uptheta)=(\bm{B}\bm\uptheta)^\top\bm{B}\bm\uptheta +0.1\sum_{i=1}^p(\bm{B}\bm\uptheta)_i^3+0.01\sum_{i=1}^p(\bm{B}\bm\uptheta)_i^4,
\end{align}
where $p=10$ is the dimension of $\bm\uptheta$, $p\bm{B}\in \mathbb{R}^{p\times p}$ is an upper-triangular matrix of ones, and $\bm{v}_i$ denotes the $i$th entry of the vector $\bm{v}$. $L(\bm\uptheta)$ achieves its minimum at $\bm\uptheta^\ast=\bm{0}$ with $L(\bm\uptheta^\ast)=\bm{0}$. Figure \ref{eq:whatyouwillnotbring} illustrates the function in (\ref{func:skewedquartic}) for the special case where $p=2$. We assume that the loss function is corrupted by i.i.d. Gaussian noise, that is $Q(\bm\uptheta,\bm{V})=L(\bm\uptheta)+\upvarepsilon$
where $\upvarepsilon\sim\mathcal{N}(0,0.1^2)$. 

\begin{figure}[!t]
\centering
\begin{tikzpicture}
  \definecolor{mynewcoloring}{RGB}{102, 128, 153}
\node[inner sep=0pt] (probe) at (0,0)
    {\includegraphics[scale=0.5]{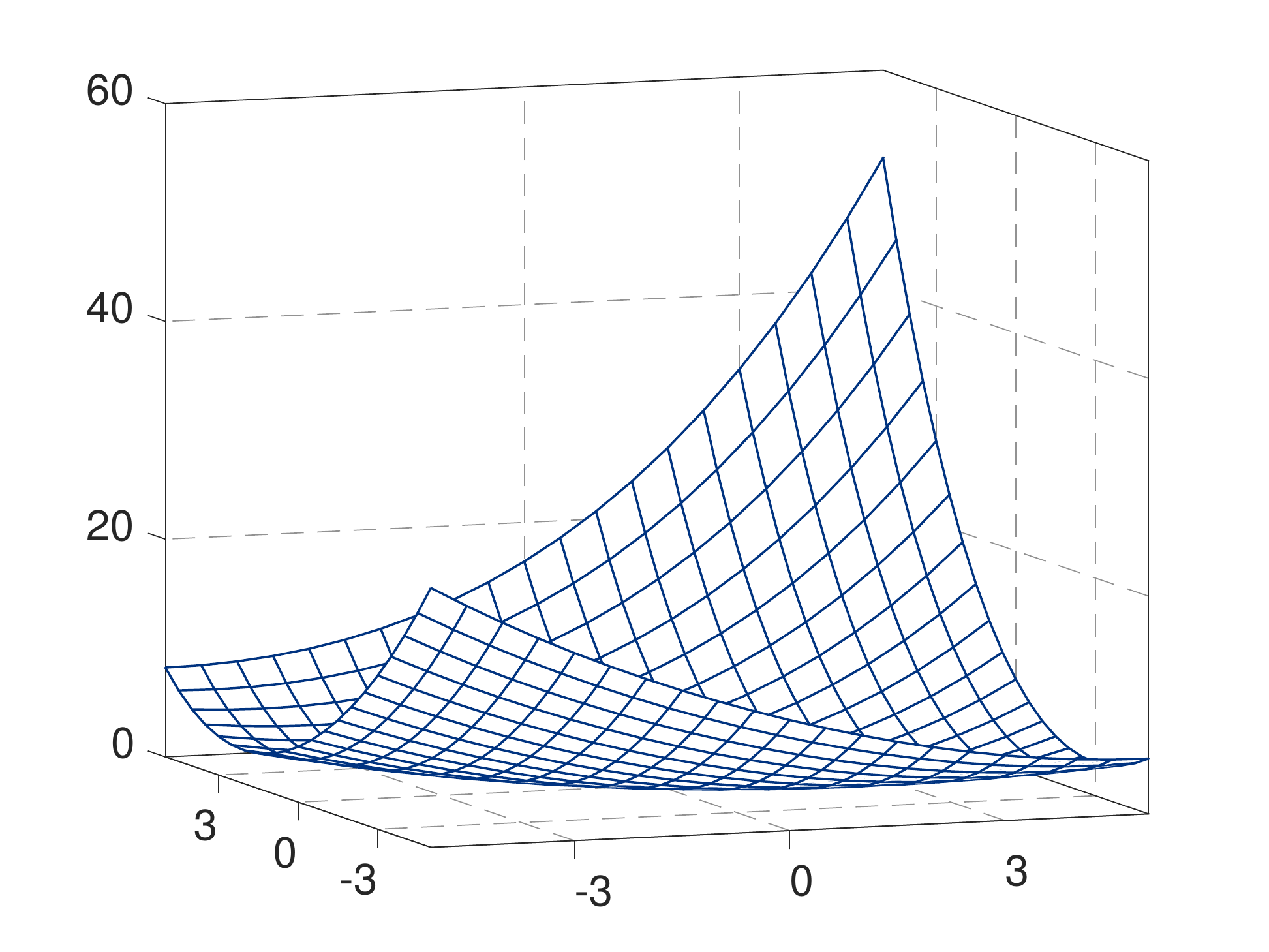}};  
     \node[left= -0.6cm of probe]{$L(\bm\uptheta)$};
      \node at (0.8,-3.7) {$\uptau_1$};
      \node at (-3.6,-3.4) {$\uptau_2$};
\end{tikzpicture}
%
\caption[The skewed quartic function in two dimensions]{The skewed quartic function in two dimensions. In this example, $\bm\uptheta=[\uptau_1,\uptau_2]^\top\in\mathbb{R}^2$ and $L(\bm\uptheta)$ is defined as in (\ref{func:skewedquartic}).}
\label{eq:whatyouwillnotbring}
\end{figure}
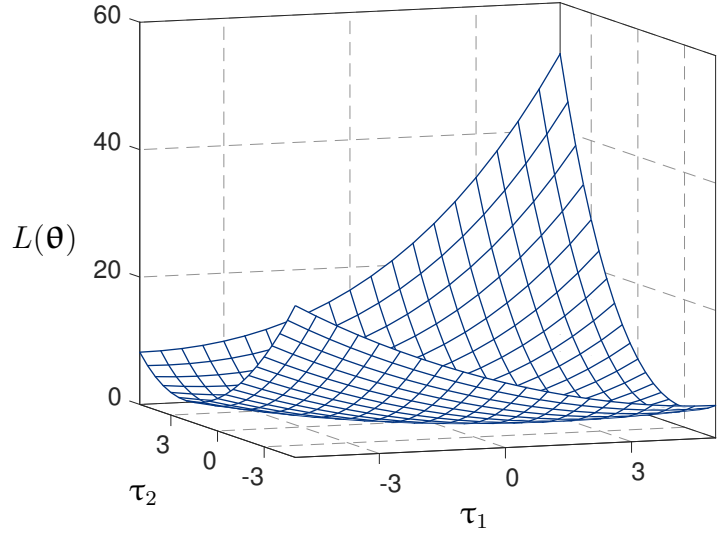

In order to implement Algorithms \ref{beastwasdone} and \ref{kirkey} the vector $\bm\uptheta=[\uptau_1,\dots,\uptau_{10}]^\top$ is divided into $d=5$ subvectors of length $2$ where the $j$th subvector is equal to $[\uptau_{2j-1},\uptau_{2j}]^\top$. The values for $a_k^{(j)}$ and $c_k^{(j)}$ are given by $a_k^{(j)}=a/(1+k)^{\upalpha}$ and $c_k^{(j)}=c/(1+k)^{\upgamma}$ for all $j$ with $a=1$, $\upalpha=0.602$, $c=1$, and $\upgamma=0.101$. These values of $\upalpha$ and $\upgamma$ are standard values for the SPSA algorithm while the values of $a$, $A$, and $c$ were {\it{lightly}} tuned. The non-zero entries of the perturbation vectors (the non-zero entries of $\bm\Delta$) are independent  and take the value $\pm1$ with equal probability. For Algorithm \ref{beastwasdone}, all subvectors were equally likely to be selected for updating (i.e., $q(j)=1/5$ for all $j$). We let $\hat{\bm{\uptheta}}_u^{\text{Alg \ref{beastwasdone}}}$ and $\hat{\bm{\uptheta}}_u^{\text{Alg \ref{kirkey}}}$ denote the iterates of Algorithms \ref{beastwasdone} and \ref{kirkey}, respectively, after having performed $u$ subvector updates (in general, the number of subvector updates is not necessarily equal to the number of iterations). Figure \ref{eq:iamapigeon2} presents the values of $\overline{\|\hat{\bm{\uptheta}}_u^{\text{Alg \ref{beastwasdone}}}-\bm\uptheta^\ast\|}/\|\hat{\bm{\uptheta}}_0^{\text{Alg \ref{beastwasdone}}}-\bm\uptheta^\ast\|$ and  $\overline{\|\hat{\bm{\uptheta}}_u^{\text{Alg \ref{kirkey}}}-\bm\uptheta^\ast\|}/\|\hat{\bm{\uptheta}}_0^{\text{Alg \ref{kirkey}}}-\bm\uptheta^\ast\|$ with $\hat{\bm{\uptheta}}_0\equiv\hat{\bm{\uptheta}}_0^{\text{Alg \ref{beastwasdone}}}=\hat{\bm{\uptheta}}_0^{\text{Alg \ref{kirkey}}}=[1,\dots,1]^\top$, where $\bar{\bm\uptheta}_u^{\text{Alg \ref{beastwasdone}}}$ and $\bar{\bm\uptheta}_u^{\text{Alg \ref{kirkey}}}$ denote the mean values of the algorithm iterates over 100 replications. It is seen that the mean distance between $\hat{\bm{\uptheta}}_u^{\text{Alg \ref{beastwasdone}}}$ and $\bm\uptheta^\ast$ (relative to $\|\hat{\bm{\uptheta}}_0^{\text{Alg \ref{beastwasdone}}}-\bm\uptheta^\ast\|$) and the mean distance between $\hat{\bm{\uptheta}}_u^{\text{Alg \ref{kirkey}}}$ and $\bm\uptheta^\ast$ (relative to $\|\hat{\bm{\uptheta}}_0^{\text{Alg \ref{kirkey}}}-\bm\uptheta^\ast\|$) are both decreasing in magnitude (i.e., the iterates approach $\bm\uptheta^\ast$). 

\begin{figure}[t]
\centering
\begin{tikzpicture}
  \definecolor{mynewcoloring}{RGB}{102, 128, 153}

\node[inner sep=0pt] (probe) at (0,0)
    {\includegraphics[scale=0.6]{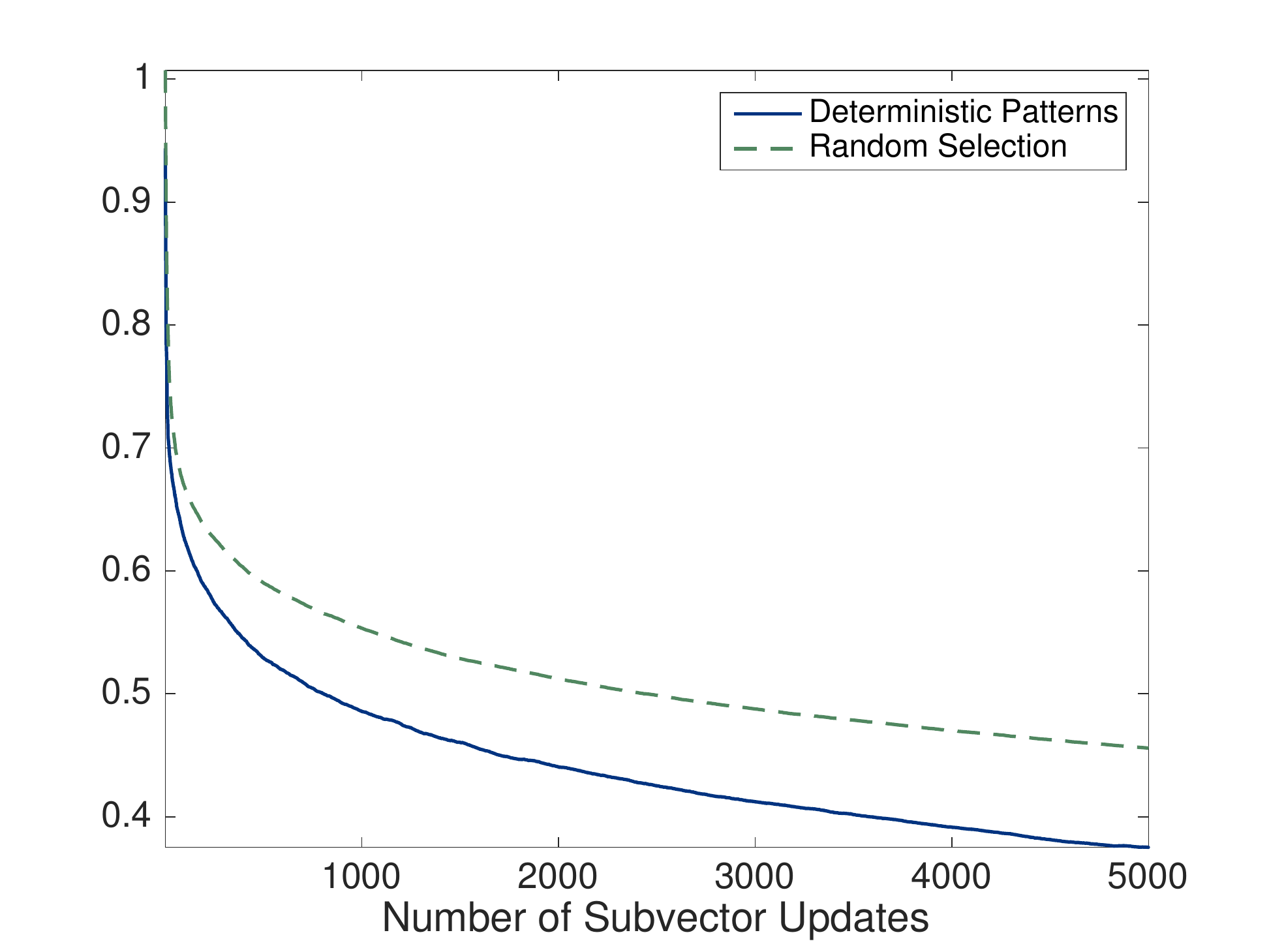}};  
     \node[left= -0.8cm of probe]{$\displaystyle{\frac{\overline{\|\hat{\bm{\uptheta}}_u-\bm\uptheta^\ast\|}}{\|\hat{\bm{\uptheta}}_0-\bm\uptheta^\ast\|}}$};
\end{tikzpicture}
\caption[Performance of SPSA-based GCSA with skewed quartic loss]{An illustration of Algorithms \ref{beastwasdone} (randomized selection) and \ref{kirkey} (deterministic patterns) on the noisy skewed quartic loss function when the gradient estimates are SPSA-based. Here, $\hat{\bm{\uptheta}}_u$ represents either $\hat{\bm{\uptheta}}_u^{\text{Alg \ref{beastwasdone}}}$ or $\hat{\bm{\uptheta}}_u^{\text{Alg \ref{kirkey}}}$.}
\label{eq:iamapigeon2}
\end{figure}

We now consider the same setting as in Figure \ref{eq:iamapigeon2} with two modifications. First, the sequences $a_k^{(1)}$ and $a_k^{(2)}$ are not required to be equal but both are assumed to have the form:
\begin{align*}
a_k^{(j)}=\frac{a^{(j)}}{(1+k+A^{(j)})^{0.602}}
\end{align*}
for $a^{(j)}>0$ and $A^{(j)}\geq 0$. Second, $\bm\uptheta$ is divided into only two subvectors corresponding to the first- and second halves of $\bm\uptheta$ (i.e., $d=2$). 
 We compare the following algorithms:
\begin{DESCRIPTION}
\item[Case 1:] The SPSA algorithm with $a_k=a/(1+k+A)^{0.602}$.
\item[Case 2:] The SPSA algorithm with $a_k$ replaced by a diagonal matrix, $\bm{A}_k$, such that the $i$th diagonal entry of $\bm{A}_k$ is equal to $a_k^{(j)}$ where $i\in \mathcal{S}_j$.\label{pageref:toccasesSPSA}
\item[Case 3:] Algorithm \ref{beastwasdone} (the subvector to update is selected according to a random variable) with SPSA-based gradient estimates. Each of the two subvectors is updated with equal probability.
\item[Case 4:] Algorithm \ref{kirkey} (the subvector to update is selected following a strictly alternating pattern) with SPSA-based gradient estimates.
\end{DESCRIPTION}
 For each case, the algorithms were run for a total of $T=5000$ subvector updates and each realization was initialized at a vector following a uniform distribution on $[-5,5]^{10}$. Thus, there were 2,500 iterations in Cases 1 and 2 (since each iteration requires updating both subvectors) and 5,000 iterations in Cases 3 and 4 (since each iteration requires updating only one subvector). We also set $c_k=c_k^{(j)}=0.1/(1+k)^{0.101}$ and use the same distribution for the perturbation vectors as before. Table \ref{table:omgmyfirsttable} presents the results of part of the tuning process for Cases 1--4. 
 Figure \ref{fig:gainyourownsga} shows the evolution of the loss function (for three replications) as the number of updates increases for Cases 1--4 using the best choice of parameters from Table \ref{table:omgmyfirsttable} (for all cases the same gain sequence yielded the smallest mean terminal loss function value). Table \ref{table:omgmyfirsttable} and Figure \ref{fig:gainyourownsga} indicate that all algorithms had comparable performance. Moreover, the results support the theory on convergence of the GCSA algorithm.

\begin{table}[p]
\centering
\begin{tabular}{SSSSSS} \toprule
    {Algorithm} & {$a^{(1)}$} & {$a^{(2)}$}  & {$A^{(1)}$} & {$A^{(2)}$} & {$\overline{L(\hat{\bm{\uptheta}}_T)}$} \\ \midrule
    {Case 1}  & 1  & {--} & 0  & {--}  & {$5.1781\times 10^{88}$}    \\
    {SPSA}  & 1 & {--} & 100 & {--} & 0.1733\ $\ast$   \\
    {(regular)}  & 0.1  & {--} & 0  & {--} & 0.5291    \\
    {}  & 0.1  & {--} & 100  & {--} &  0.5613       \\ \midrule
    {Case 2}  & 1  & 1 & 0  & 0  & {$6.7502\times 10^{119}$}    \\
    {SPSA}  & 1 & 1 & 100 & 100 & 0.1650\ $\ast$   \\
    {(diagonal gain)}  & 1  & 0.1 & 0  & 0 & 0.4335    \\
    {}  & 1  & 0.1 & 100  & 100 & 0.4236    \\
    {}  & 0.1  & 1 & 0  & 0 & 0.3941    \\
    {}  & 0.1  & 1 & 100  & 100 & 0.3691    \\
   {}  & 0.1  & 0.1 & 0  & 0 & 0.5060    \\
    {}  & 0.1  & 0.1 & 100  & 100 & 0.5727          \\ \midrule
    {Case 3}  & 1  & 1 & 0  & 0  &  0.2282    \\
    {SPSA}  & 1 & 1 & 100 & 100 & 0.1982\ $\ast$   \\
    {(random selection)}  & 1  & 0.1 & 0  & 0 & 0.2191    \\
    {}  & 1  & 0.1 & 100  & 100 & 0.2025    \\
    {}  & 0.1  & 1 & 0  & 0 & 0.6915    \\
    {}  & 0.1  & 1 & 100  & 100 &  0.5772    \\
   {}  & 0.1  & 0.1 & 0  & 0 & 0.6109   \\
    {}  & 0.1  & 0.1 & 100  & 100 & 0.6477          \\ \midrule
    {Case 4}  & 1  & 1 & 0  & 0  & 0.1992    \\
    {SPSA}  & 1 & 1 & 100 & 100 & 0.1732\ $\ast$   \\
    {(deterministic patterns)}  & 1  & 0.1 & 0  & 0 & 0.3406    \\
    {}  & 1  & 0.1 & 100  & 100 & 0.3796    \\
    {}  & 0.1  & 1 & 0  & 0 & 0.3841    \\
    {}  & 0.1  & 1 & 100  & 100 & 0.3750    \\
   {}  & 0.1  & 0.1 & 0  & 0 & 0.4812
    \\
    {}  & 0.1  & 0.1 & 100  & 100 & 0.5286         \\ \bottomrule
\end{tabular}
\caption[Tuning the gain sequence parameters for Cases 1--4 from p. \pageref{pageref:toccasesSPSA} where the noisy gradient updates are SPSA-based]{Tuning the gain sequence parameters. Note that for Case 1 there is only one gain sequence, $a_k$, which we have set equal to $a^{(1)}/(1+k+A^{(1)})^{0.602}$. See p. \pageref{pageref:toccasesSPSA} for a description of each of the four cases above. Here, $\overline{L(\hat{\bm{\uptheta}}_T)}$ represents an estimate of the loss function at the terminal value obtained by averaging the terminal loss values of 100 replications. For comparison, $L(\bm\uptheta^\ast)=0$. An asterisk marks the smallest mean terminal loss function value for each case.}
\label{table:omgmyfirsttable}
\end{table}


\begin{figure}[p]
\centering
  \begin{subfigure}[b]{0.5\textwidth}
 \centering
\begin{tikzpicture}

  \definecolor{mynewcoloring}{RGB}{102, 128, 153}
\node[inner sep=0pt] (probe) at (0,0) 
{\includegraphics[scale=0.35]{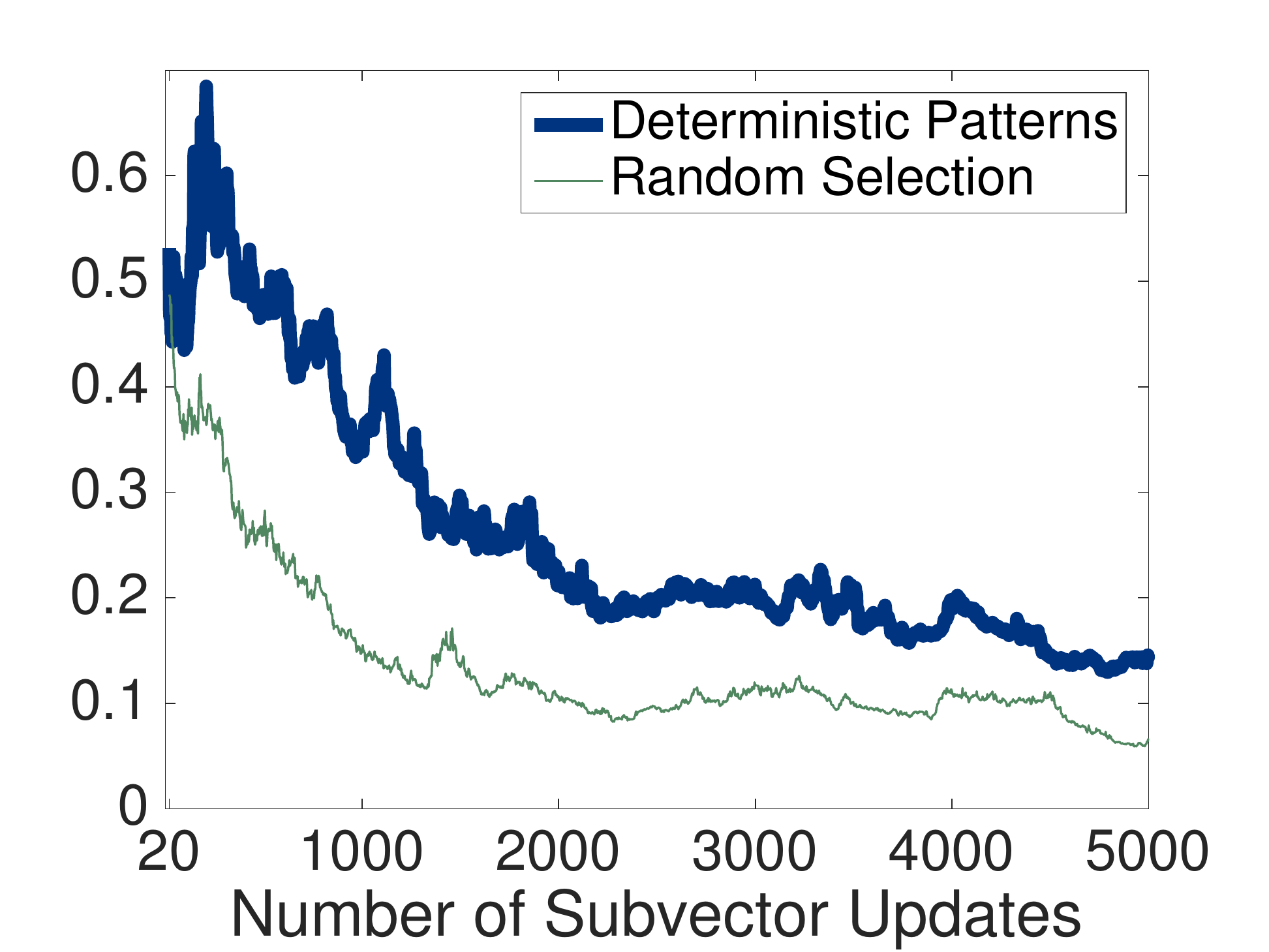}};
 \node[rotate=90] at (-3.5,0) {\small{$L(\hat{\bm{\uptheta}}_u)$}};

\end{tikzpicture}
        \end{subfigure}\hfill
        \begin{subfigure}[b]{0.5\textwidth}
                \centering
\begin{tikzpicture}

  \definecolor{mynewcoloring}{RGB}{102, 128, 153}
\node[inner sep=0pt] (probe) at (0,0)
    {\includegraphics[scale=0.35]{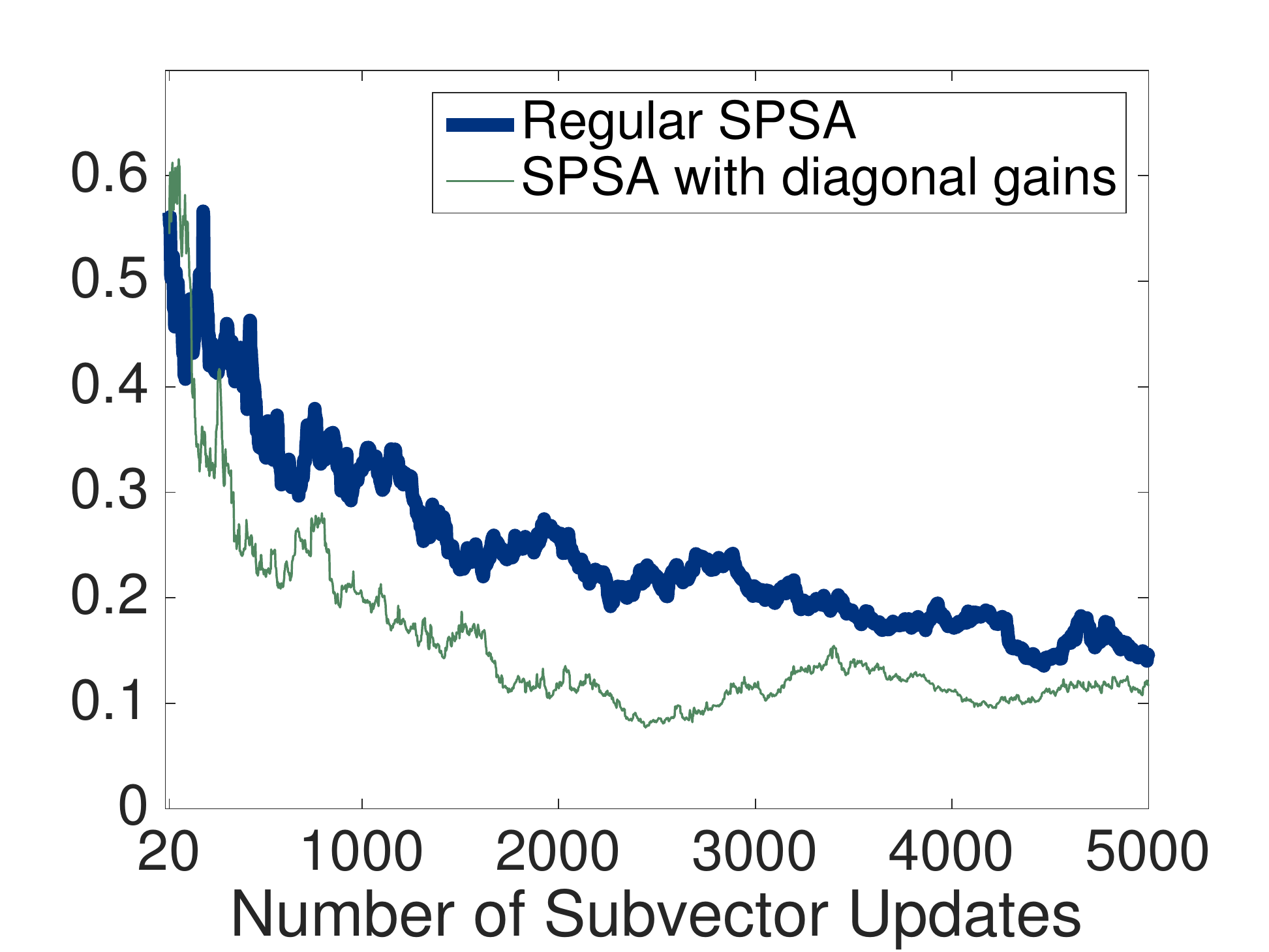}};
     \node[rotate=90] at (-3.5,0) {\small{$L(\hat{\bm{\uptheta}}_u)$}};

\end{tikzpicture}
 \end{subfigure}
 \caption*{Above: the 1st replication. }


 \vspace{.03in}
 
 
    \begin{subfigure}[b]{0.5\textwidth}
 \centering
\begin{tikzpicture}

  \definecolor{mynewcoloring}{RGB}{102, 128, 153}
\node[inner sep=0pt] (probe) at (0,0) 
{\includegraphics[scale=0.35]{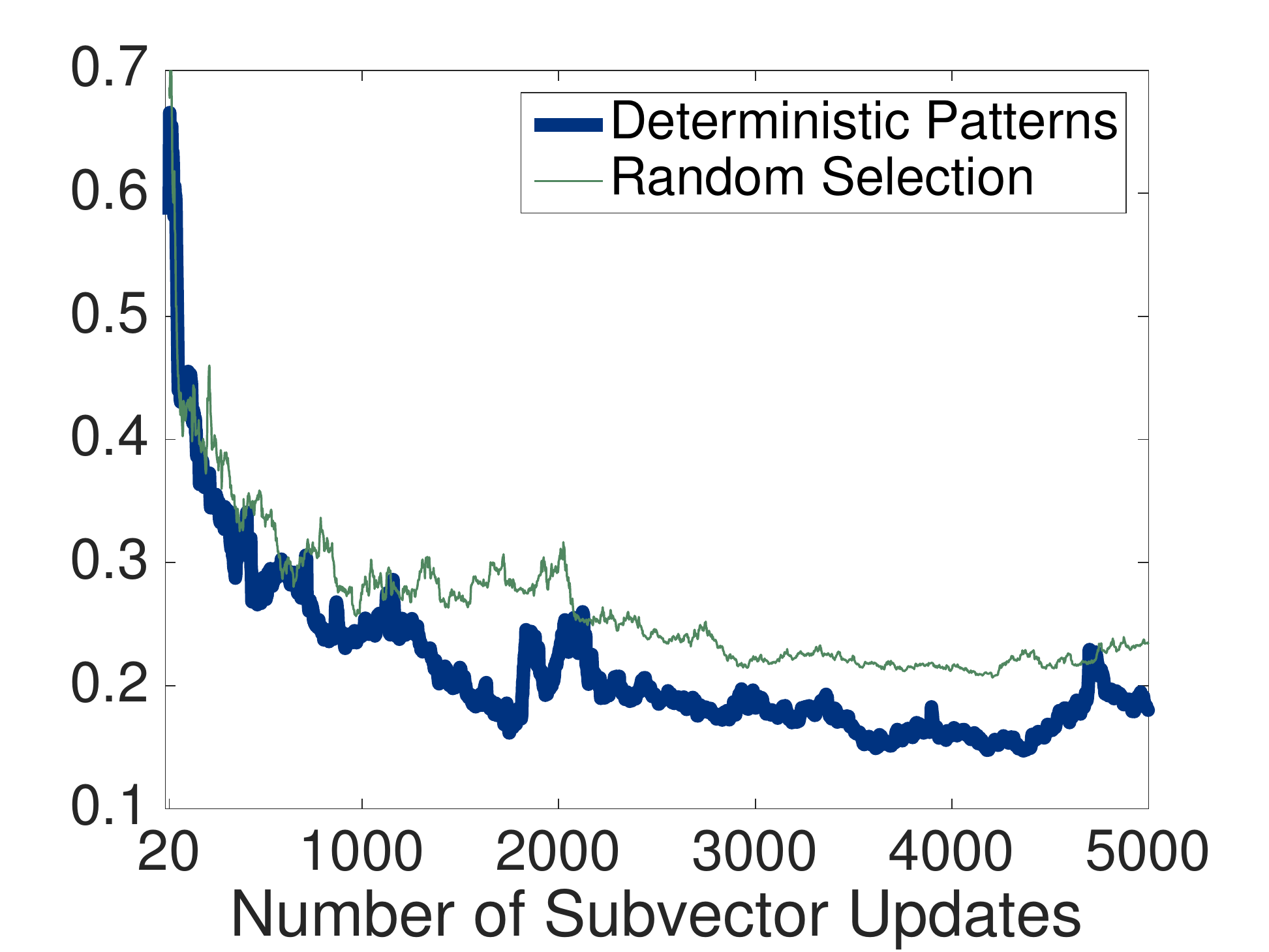}};
 \node[rotate=90] at (-3.5,0) {\small{$L(\hat{\bm{\uptheta}}_u)$}};

\end{tikzpicture}
        \end{subfigure}\hfill
        \begin{subfigure}[b]{0.5\textwidth}
                \centering
\begin{tikzpicture}

  \definecolor{mynewcoloring}{RGB}{102, 128, 153}
\node[inner sep=0pt] (probe) at (0,0)
    {\includegraphics[scale=0.35]{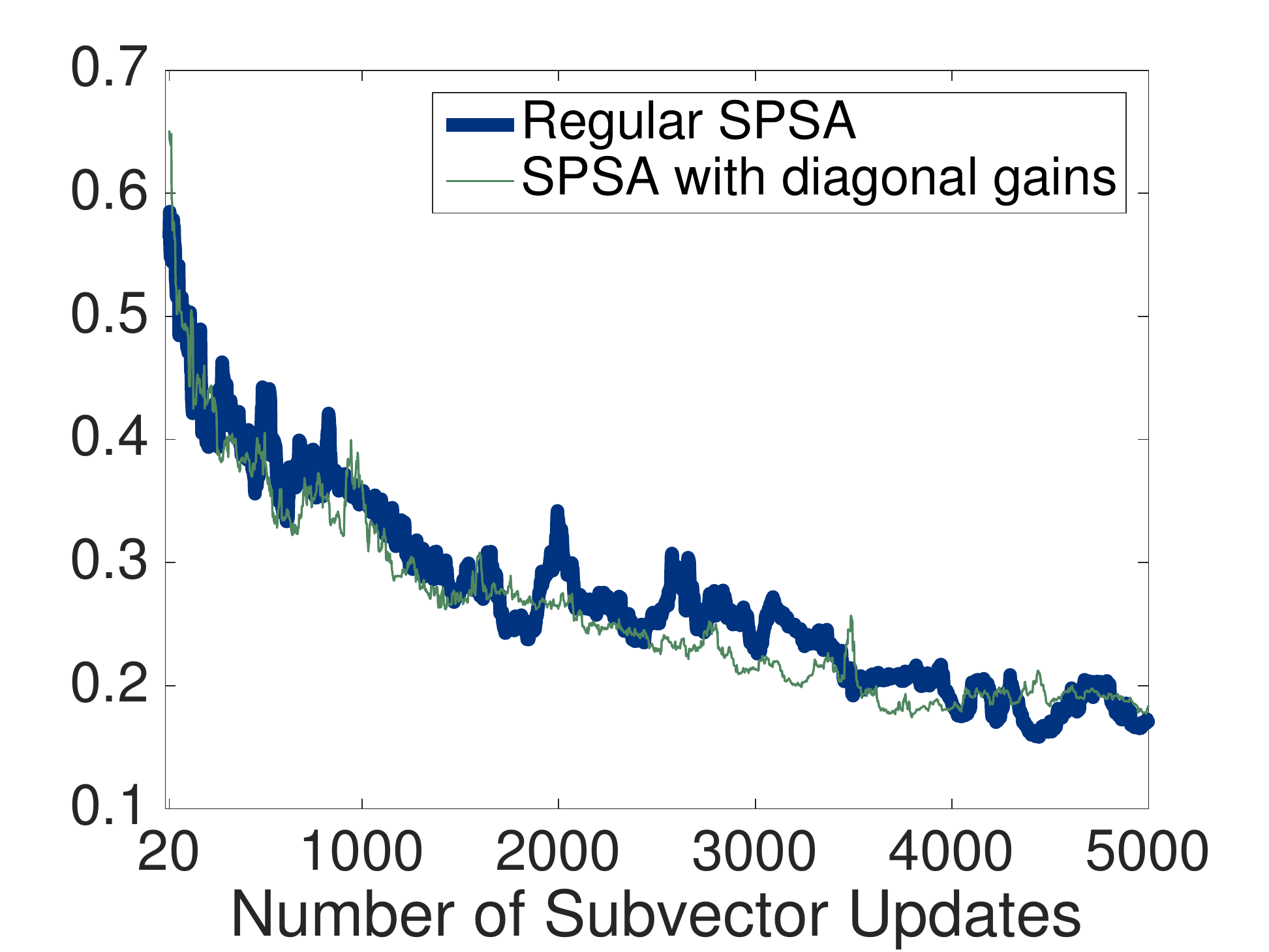}};
     \node[rotate=90] at (-3.5,0) {\small{$L(\hat{\bm{\uptheta}}_u)$}};

\end{tikzpicture}
 \end{subfigure}%
 \caption*{Above: the 2nd replication.}

 
  \vspace{.03in}
 
 
    \begin{subfigure}[b]{0.5\textwidth}
 \centering
\begin{tikzpicture}

  \definecolor{mynewcoloring}{RGB}{102, 128, 153}
\node[inner sep=0pt] (probe) at (0,0) 
{\includegraphics[scale=0.35]{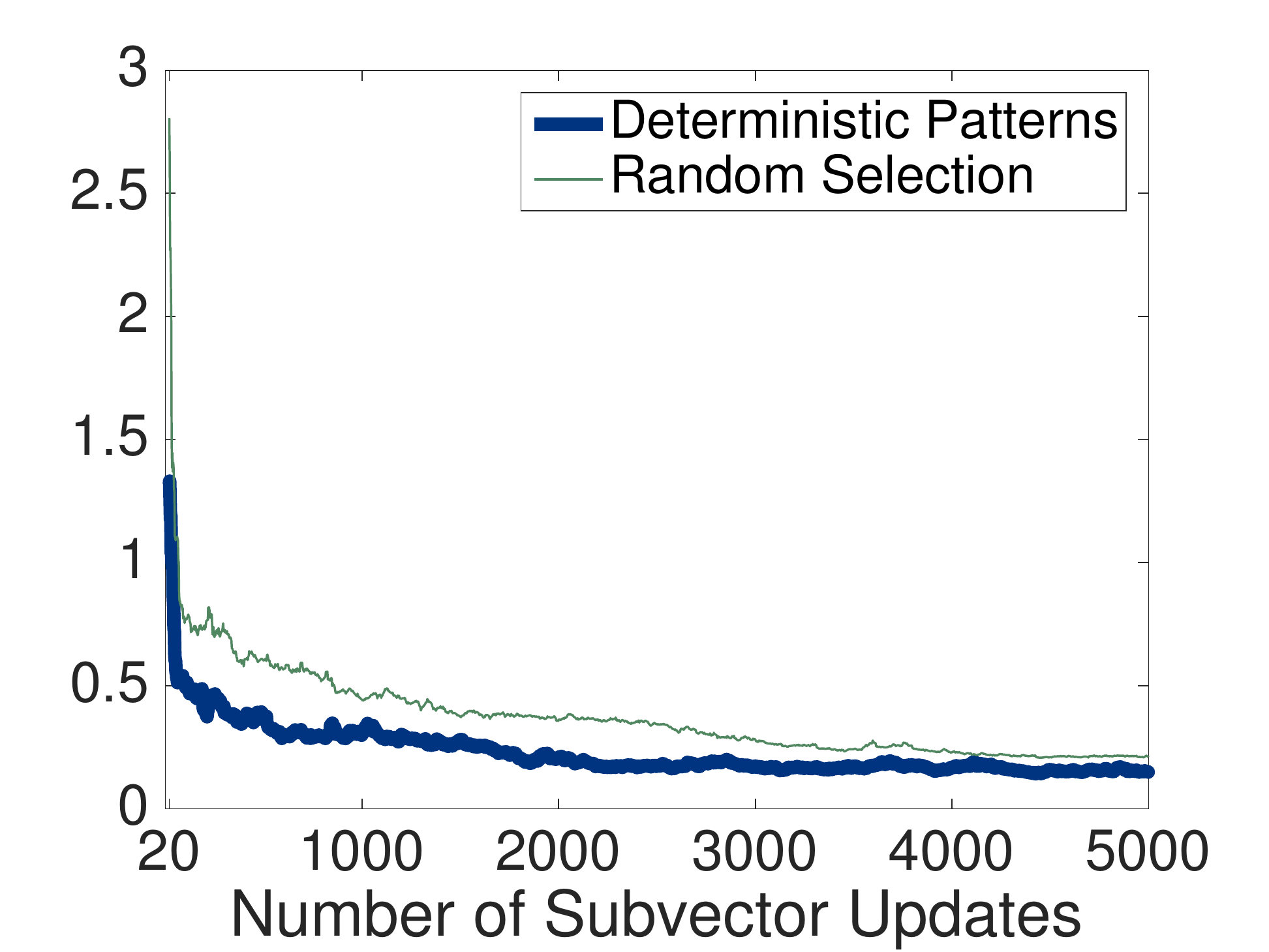}};
 \node[rotate=90] at (-3.5,0) {\small{$L(\hat{\bm{\uptheta}}_u)$}};

\end{tikzpicture}
        \end{subfigure}\hfill
        \begin{subfigure}[b]{0.5\textwidth}
                \centering
\begin{tikzpicture}

  \definecolor{mynewcoloring}{RGB}{102, 128, 153}
\node[inner sep=0pt] (probe) at (0,0)
    {\includegraphics[scale=0.35]{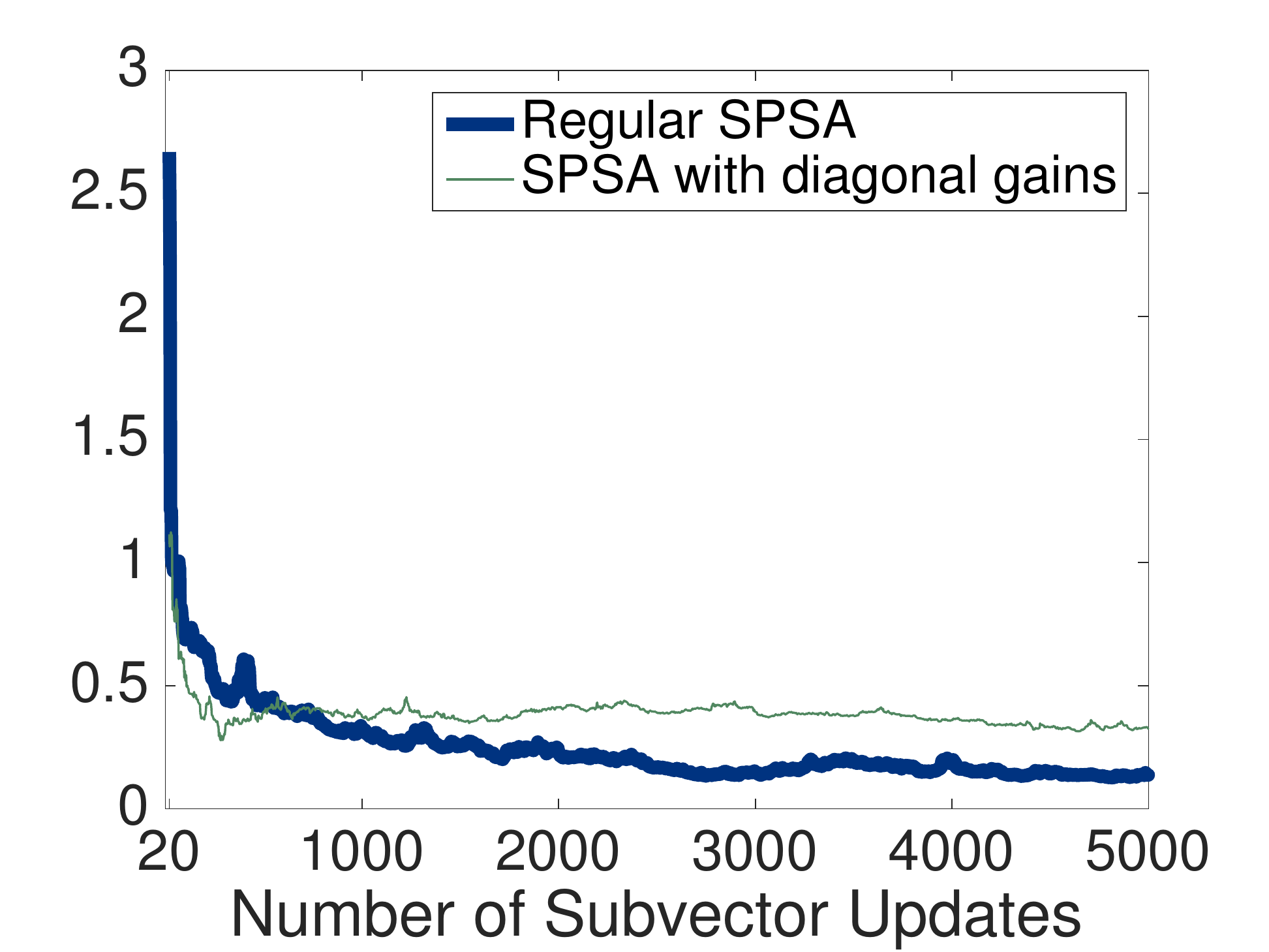}};
     \node[rotate=90] at (-3.5,0) {\small{$L(\hat{\bm{\uptheta}}_u)$}};

\end{tikzpicture}
 \end{subfigure}
  \caption*{Above: the 3rd replication.}
 
 
 \vspace{.03in}
         
        
\caption[Performance of Cases 1--4 from p. \pageref{pageref:toccasesSPSA} where the noisy gradient estimates are SPSA-based]{Performance of Cases 1 (deterministic patterns), 2 (random selection), 3 (regular SPSA), and 4 (SPSA with diagonal gains) for three i.i.d. replications (i.e., three i.i.d. realizations of each of the four cases) using the tuned parameters from Table \ref{table:omgmyfirsttable}. The term $\hat{\bm{\uptheta}}_u$ is a generic vector representing the vector obtained after having performed $u$ subvector updates.}
\label{fig:gainyourownsga}
\end{figure}

\subsection{Algorithms \ref{beastwasdone} \& \ref{kirkey} with SG-Based Gradient Estimates}
\label{sec:omgimanothersectionsssd}
This section is concerned with an implementation of Algorithms \ref{beastwasdone} and \ref{kirkey} where the noisy gradient estimates are SG-based. Here, we consider the problem of linear regression with scalar output. We assume that the output, $z_k$, of some random process is given by $z_k=\bm{h}_k^\top \bm\uptheta^\ast+\upvarepsilon_k$
for a sequence $\bm{h}_k\in\mathbb{R}^p$ of inputs and an i.i.d. sequence of mean-zero random variables $\upvarepsilon_k$. The objective is to recover the true value of $\bm\uptheta$, denoted by $\bm\uptheta^\ast$, using a set of input-output pairs $\{(\bm{h}_k,z_k)\}_{k\geq 1}$. An online approach to solving this problem is the following:
\begin{align}
\label{eq:colorlocalcouncil}
\hat{\bm{\uptheta}}_{k+1}=\hat{\bm{\uptheta}}_k-a_k\left[(\bm{h}_{k+1}^\top \bm\uptheta-z_{k+1})\bm{h}_{k+1}\right]_{\bm\uptheta=\hat{\bm{\uptheta}}_k}.
\end{align}
The algorithm in (\ref{eq:colorlocalcouncil}) is known as the least-mean-squares (LMS) algorithm.

Before proceeding with the numerical experiments let us briefly motivate (\ref{eq:colorlocalcouncil}). First, note that if $\{\upvarepsilon_k\}_{k\geq 0}$ and $\{\bm{h}_{k}\}_{k\geq 0}$ are each i.i.d. sequences (with $\bm{h}_k$ being independent of $\upvarepsilon_k$) then, under certain conditions (e.g., Spall 2003, p. 136), the loss function $L(\bm\uptheta)\equiv (1/2)E[(z_k-\bm{h}_{k}^\top\bm\uptheta)^2]$ is independent of $k$, $L(\bm\uptheta)$ has a unique minimum at $\bm\uptheta^\ast$, and the vector in square brackets in (\ref{eq:colorlocalcouncil}) is an unbiased measurement of the gradient of $L(\bm\uptheta)$. Therefore, under certain conditions the LMS algorithm in (\ref{eq:colorlocalcouncil}) is a special case of the SG algorithm (which requires an unbiased estimate of the gradient of the loss function). Figure \ref{fig:substitutionoperation} illustrates $L(\bm\uptheta)$ for the special case where $\bm\uptheta\in \mathbb{R}^2$, the entries of $\bm{h}_k$ are independent and uniformly distributed in $[-3,3]$, $\upvarepsilon_k\sim\mathcal{N}(0,1)$, and $\bm\uptheta^\ast=[1,1]^\top$.

\begin{figure}[!t]
\centering
\begin{tikzpicture}
  \definecolor{mynewcoloring}{RGB}{102, 128, 153}
\node[inner sep=0pt] (probe) at (0,0)
    {\includegraphics[scale=0.5]{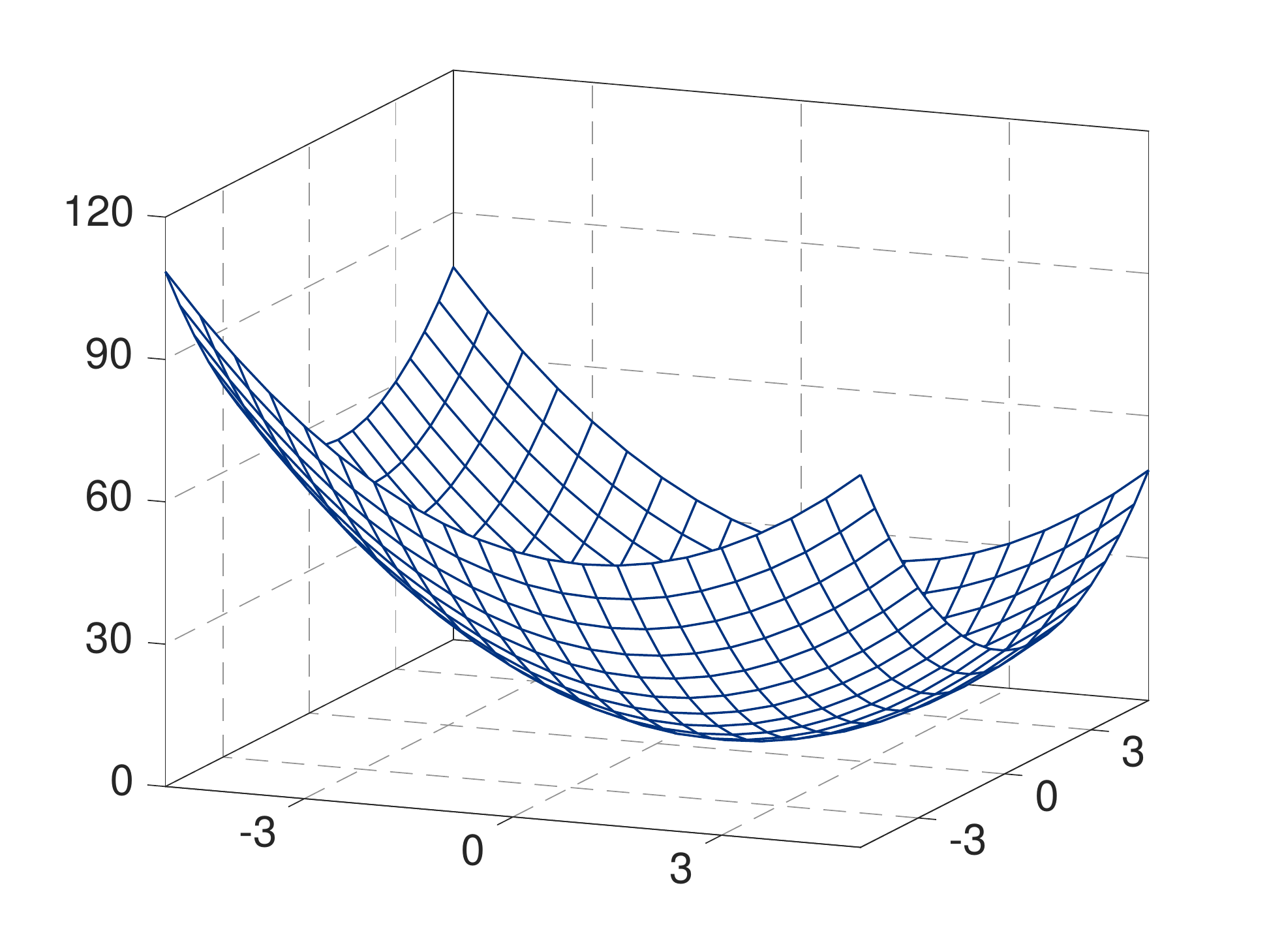}};  
     \node[left= -0.6cm of probe]{$L(\bm\uptheta)$};
      \node at (4.6,-2.5) {$\uptau_2$};
      \node at (-2.1,-3.4) {$\uptau_1$};
\end{tikzpicture}
\caption[The loss function associated with the LMS algorithm in (\ref{eq:colorlocalcouncil}).]{The loss function, $L(\bm\uptheta)= (1/2)E{[}(z_k-\bm{h}_{k}^\top\bm\uptheta)^2{]}$, associated with the least-mean-squares (LMS) algorithm in (\ref{eq:colorlocalcouncil}) for the special case where $\bm\uptheta\in \mathbb{R}^2$,  $\bm\uptheta^\ast=[1,1]^\top$, the entries of $\bm{h}_k$ are uniformly distributed in $[-3,3]$, and $\upvarepsilon_k\sim\mathcal{N}(0,1)$ is independent of $\bm{h}_k$. Here, $L(\bm\uptheta)=(1/2)[1+3(\bm\uptheta-\bm\uptheta^\ast)^\top(\bm\uptheta-\bm\uptheta^\ast)]$.}
\label{fig:substitutionoperation}
\end{figure}

A cyclic implementation of (\ref{eq:colorlocalcouncil}) in the manner of Algorithms \ref{beastwasdone} and \ref{kirkey} would involve defining the subvectors, defining the gains sequences $a_k^{(j)}$, determining how the subvector to update in Algorithm \ref{beastwasdone} is selected, and determining how the data $\{(\bm{h}_k,z_k)\}_{k\geq 1}$ are to be generated and processed. In regards to the subvector definition, we assume $p=10$ and use the same partition as in Section \ref{sec:clubsodeashrimp} (here there are 5 subvectors of length 2 each). In regards to the gain sequences, we set $a_k^{(j)}=a/(1+k+A)^{\upalpha}$ for all $j$ where $a=1$, $A=100$, and $\upalpha=1$. The value of $\upalpha$ is a standard choice for the SG algorithms and the values of $a$ and $A$ were {\it{lightly}} tuned. Next, for Algorithm \ref{beastwasdone} each of the 5 subvectors has equal probability of being selected ($q(j)=q$ for $j=1,\dots,5$). The input-output pair $(\bm{x}_k,z_k)$ was generated letting $\bm\uptheta^\ast=[1,\dots,1]^\top$, $\upvarepsilon_k\sim \mathcal{N}(0,1)$, and letting the entries of $\bm{h}_k$ be independent and uniformly distributed in $[-3,3]$. Finally, each input-output pair was used to perform 5 subvector updates. Once 5 updates have been performed using the same input-output pair, a new pair is generated. Using the notation of Section \ref{sec:clubsodeashrimp}, Figure \ref{fig:feedhimabunchofapples} presents the value of $\|\hat{\bm{\uptheta}}_u-\bm\uptheta^\ast\|/\|\hat{\bm{\uptheta}}_0-\bm\uptheta^\ast\|$ for 5 replications of Algorithms \ref{beastwasdone} and \ref{kirkey} initialized at $-\bm\uptheta^\ast$. It can be seen that the iterates of both algorithms appear to converge to $\bm\uptheta^\ast$.



\begin{figure}[!t]
\centering
\begin{tikzpicture}
  \definecolor{mynewcoloring}{RGB}{102, 128, 153}
\node[inner sep=0pt] (probe) at (0,0)
    {\includegraphics[scale=0.6]{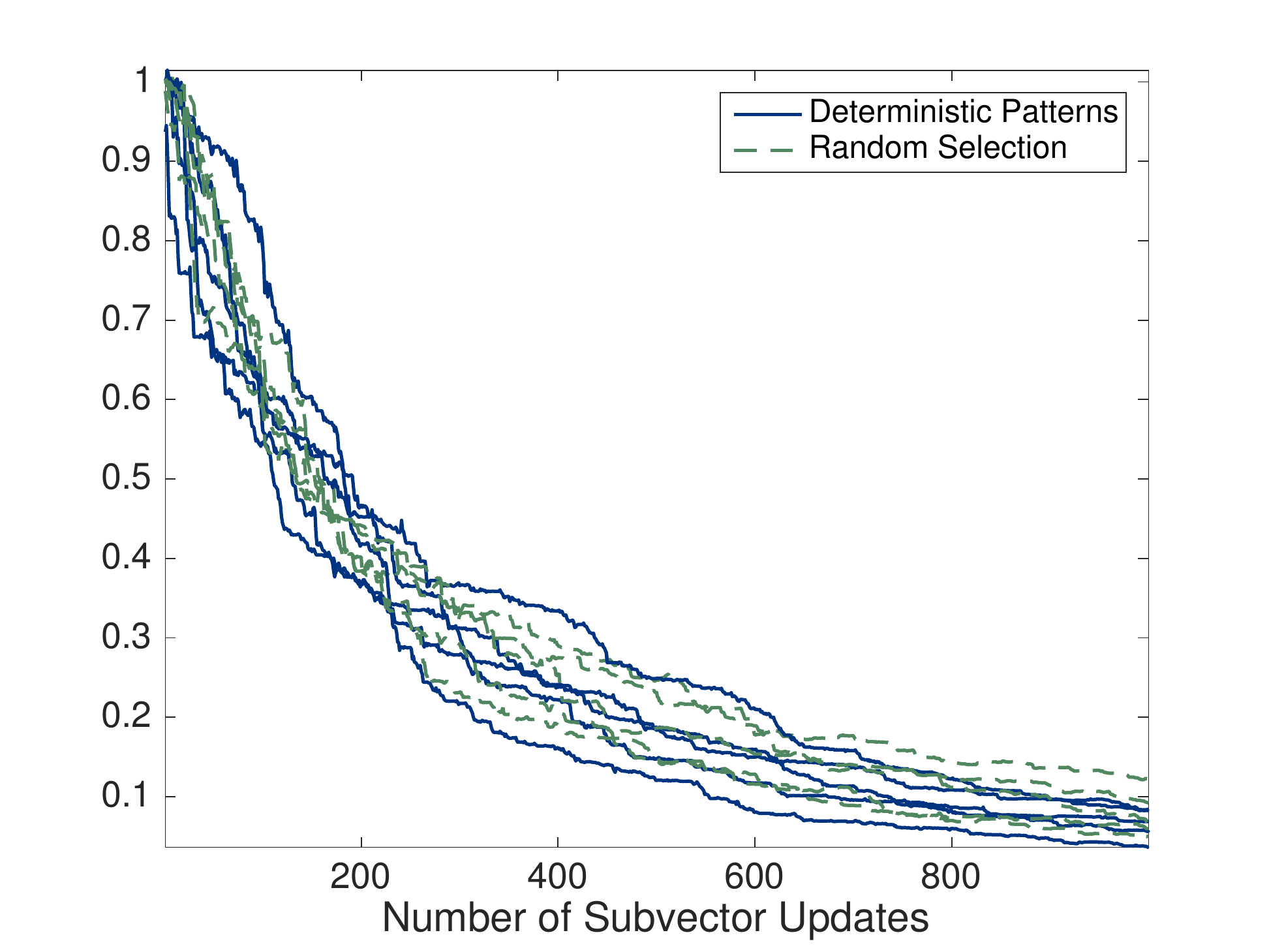}};  
        \node[left= -0.8cm of probe]{$\displaystyle{\frac{\|\hat{\bm{\uptheta}}_u-\bm\uptheta^\ast\|}{\|\hat{\bm{\uptheta}}_0-\bm\uptheta^\ast\|}}$};
\end{tikzpicture}
\caption[Performance of SG-based GCSA implementations of LMS]{Cyclic implementations, in the manner of Algorithms \ref{beastwasdone} (random selection) and \ref{kirkey} (deterministic patterns), of the LMS algorithm in (\ref{eq:colorlocalcouncil}). Pictured are 5 realizations of each algorithm. Here, $\hat{\bm{\uptheta}}_u$ represents either $\hat{\bm{\uptheta}}_u^{\text{Alg \ref{beastwasdone}}}$ (dashed lines) or $\hat{\bm{\uptheta}}_u^{\text{Alg \ref{kirkey}}}$ (solid lines).}
\label{fig:feedhimabunchofapples}
\end{figure}

We now consider the same setting as in Figure \ref{fig:feedhimabunchofapples} with three modifications. First, the gain sequences take the form:
\begin{align*}
a_k^{(j)}=\frac{a^{(j)}}{(1+k+A^{(j)})^{0.501}},
\end{align*}
for $a^{(j)}>0$ and $A^{(j)}\geq 0$. Second, the vector $\bm\uptheta$ is divided into only two subvectors corresponding to the first- and second halves of $\bm\uptheta$ (i.e., $d=2$). Third,  we assume $\upvarepsilon_k\sim\mathcal{N}(0,0.1^2)$. 
 We compare the following algorithms:
\begin{DESCRIPTION}
\item[Case 1:] The SG algorithm from (\ref{eq:colorlocalcouncil}) with $a_k=a/(1+k+A)^{0.501}$.
\item[Case 2:] The SG algorithm from (\ref{eq:colorlocalcouncil}) with $a_k$ replaced by a diagonal matrix, $\bm{A}_k$, such that the $i$th diagonal entry of $\bm{A}_k$ is equal to $a_k^{(j)}$ where $i\in \mathcal{S}_j$.\label{pageref:toccasesSG}
\item[Case 3:] Algorithm \ref{beastwasdone} (the subvector to update is selected according to a random variable) with SG-based gradient estimates. Each of the two subvectors is updated with equal probability.

\item[Case 4:] Algorithm \ref{kirkey} (the subvector to update is selected following a strictly alternating pattern) with SG-based gradient estimates.
\end{DESCRIPTION}
 Each algorithm was run for a total of $T=5000$ subvector updates and each realization was initialized at a vector uniformly distributed on $[-4,6]^{10}$.  Thus, there were 2,500 iterations in Cases 1 and 2 (since each iteration requires updating both subvectors) and 5,000 iterations in Cases 3 and 4 (since each iteration requires updating only one subvector). The entries of  $\bm{h}_k$ were independent and uniformly distributed in $[-3,3]$. 
 Table \ref{table:omgmySECONDtable} presents the results of part of the tuning process for Cases 1--4 and Figure \ref{fig:betterplaybetstts} shows the evolution of the loss function as the number of updates increases. 
Table \ref{table:omgmySECONDtable} and Figure \ref{fig:betterplaybetstts} indicate that all algorithms had a comparable performance (for all cases the same gain sequence yielded the smallest mean terminal loss function value). Moreover, the results support the theory on convergence of the GCSA algorithm. 

\begin{table}[p]
\centering
\begin{tabular}{SSSSSS} \toprule
    {Algorithm} & {$a^{(1)}$} & {$a^{(2)}$}  & {$A^{(1)}$} & {$A^{(2)}$} & {$\overline{L(\hat{\bm{\uptheta}}_T)}$} \\ \midrule
    {Case 1}  & 1  & {--} & 0  & {--}  & 0.0073    \\
    {SG}  & 1 & {--} & 100 & {--} & 0.0072   \\
    {(regular)}  & 0.1  & {--} & 0  & {--} & 0.0052  \\
    {}  & 0.1  & {--} & 100  & {--} & 0.0052\ $\ast$       \\ \midrule
    {Case 2}  & 1  & 1 & 0  & 0  & 0.0072    \\
    {SG}  & 1 & 1 & 100 & 100 & 0.0071   \\
    {(diagonal gain)}  & 1  & 0.1 & 0  & 0 & 2.6669    \\
    {}  & 1  & 0.1 & 100  & 100 & 0.0060    \\
    {}  & 0.1  & 1 & 0  & 0 & 0.7573    \\
    {}  & 0.1  & 1 & 100  & 100 & 0.0060    \\
   {}  & 0.1  & 0.1 & 0  & 0 & 0.0052    \\
    {}  & 0.1  & 0.1 & 100  & 100 & 0.0052\ $\ast$          \\ \midrule
    {Case 3}  & 1  & 1 & 0  & 0  & 0.0064    \\
    {SG}  & 1 & 1 & 100 & 100 & 0.0063   \\
    {(random selection)}  & 1  & 0.1 & 0  & 0 & 0.0258    \\
    {}  & 1  & 0.1 & 100  & 100 & 0.0056    \\
    {}  & 0.1  & 1 & 0  & 0 & 0.0081    \\
    {}  & 0.1  & 1 & 100  & 100 & 0.0056    \\
   {}  & 0.1  & 0.1 & 0  & 0 & 0.0051   \\
    {}  & 0.1  & 0.1 & 100  & 100 & 0.0051\ $\ast$          \\ \midrule
    {Case 4}  & 1  & 1 & 0  & 0  & {$1.3221\times 10^{8}$}    \\
    {SG}  & 1 & 1 & 100 & 100 & 0.0067   \\
    {(deterministic patterns)}  & 1  & 0.1 & 0  & 0 & {$1.5906\times 10^{17}$}    \\
    {}  & 1  & 0.1 & 100  & 100 & 0.0060    \\
    {}  & 0.1  & 1 & 0  & 0 & 0.0060    \\
    {}  & 0.1  & 1 & 100  & 100 & 0.0060 \\
   {}  & 0.1  & 0.1 & 0  & 0 & 0.0051    \\
    {}  & 0.1  & 0.1 & 100  & 100 & 0.0051\ $\ast$         \\ \bottomrule
\end{tabular}
\caption[Tuning the gain sequence parameters for Cases 1--4 from p. \pageref{pageref:toccasesSG} where the noisy gradient updates are SG-based]{Tuning the gain sequence parameters. Note that for Case 1 there is only one gain sequence, $a_k$, which we have set equal to $a^{(1)}/(1+k+A^{(1)})^{0.501}$. See p. \pageref{pageref:toccasesSG} for a description of each of the four cases above. Here, $\overline{L(\hat{\bm{\uptheta}}_T)}$ represents an estimate of the loss function at the terminal value obtained by averaging the terminal loss values of 100 replications. For comparison, $L(\bm\uptheta^\ast)=0$. An asterisk marks the smallest mean terminal loss function value for each case.}
\label{table:omgmySECONDtable}
\end{table}


\begin{figure}[p]
\centering
  \begin{subfigure}[b]{0.5\textwidth}
 \centering
\begin{tikzpicture}

  \definecolor{mynewcoloring}{RGB}{102, 128, 153}
\node[inner sep=0pt] (probe) at (0,0) 
{\includegraphics[scale=0.35]{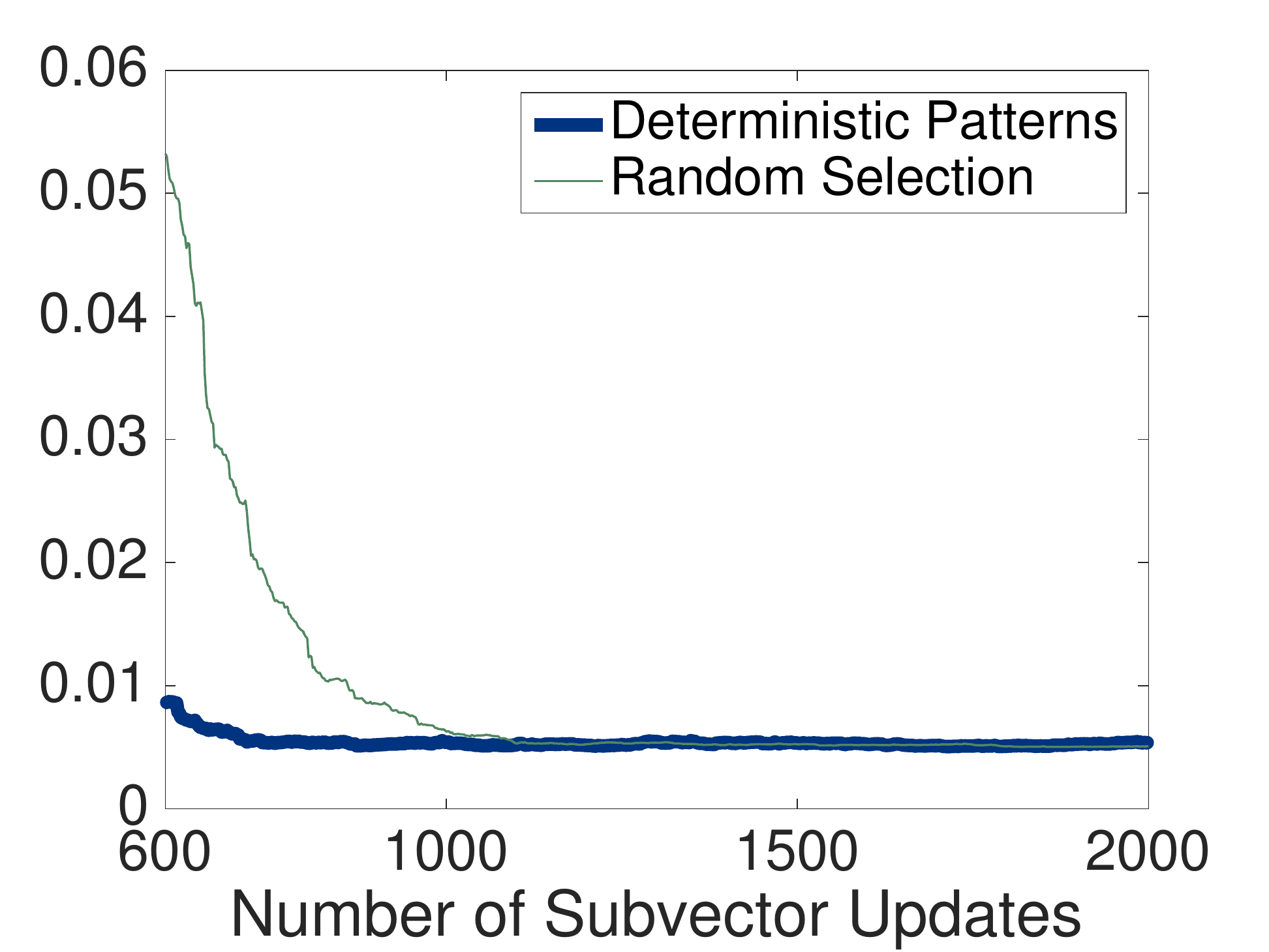}};
\node[rotate=90] at (-3.75,0) {\small{$L(\hat{\bm{\uptheta}}_u)$}};


\end{tikzpicture}
        \end{subfigure}\hfill
        \begin{subfigure}[b]{0.5\textwidth}
                \centering
\begin{tikzpicture}

  \definecolor{mynewcoloring}{RGB}{102, 128, 153}
\node[inner sep=0pt] (probe) at (0,0)
    {\includegraphics[scale=0.35]{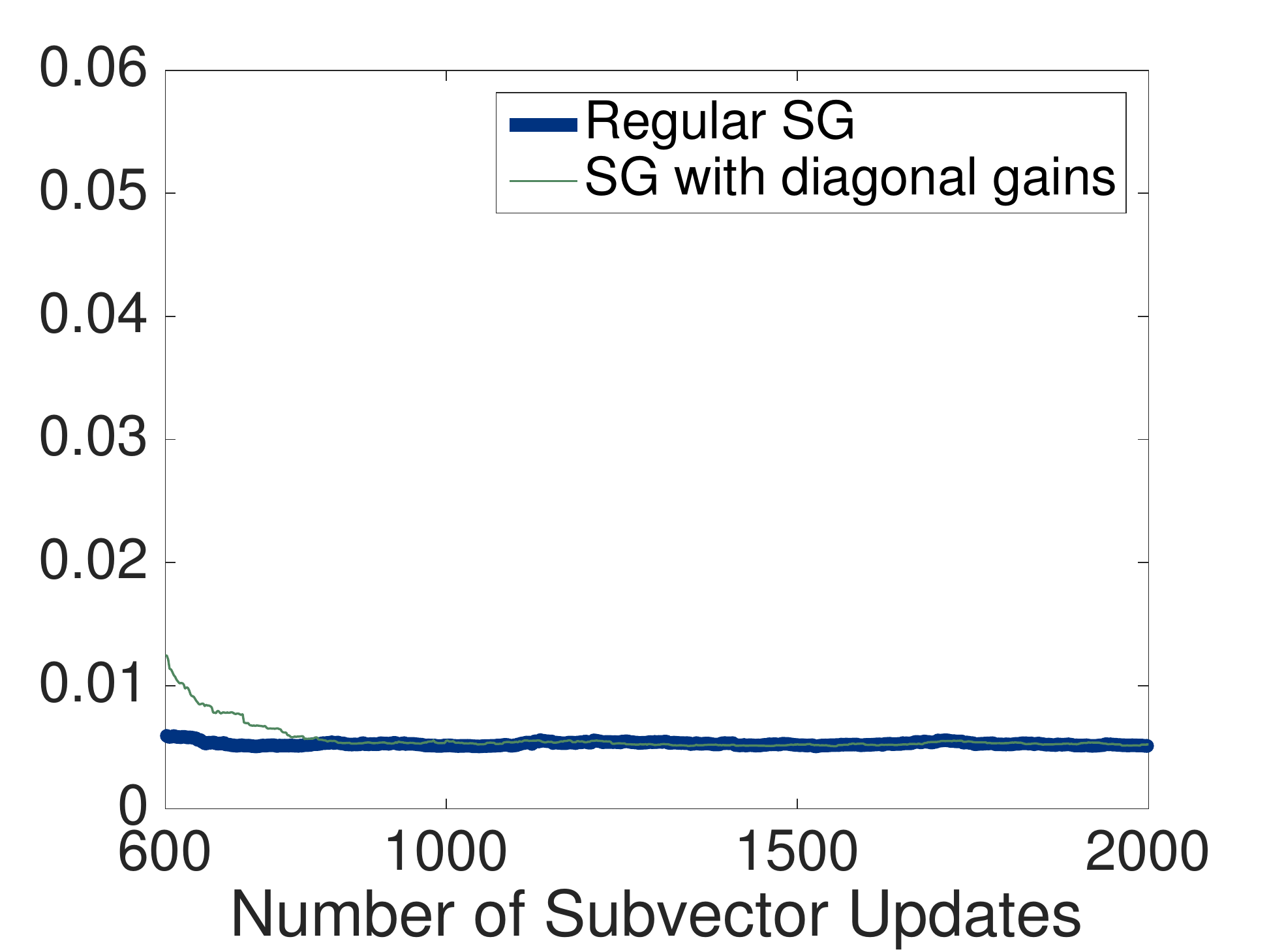}};
    \node[rotate=90] at (-3.75,0) {\small{$L(\hat{\bm{\uptheta}}_u)$}};

\end{tikzpicture}
 \end{subfigure}
 \caption*{Above: the 1st replication.}


 \vspace{.03in}
 
 
    \begin{subfigure}[b]{0.5\textwidth}
 \centering
\begin{tikzpicture}

  \definecolor{mynewcoloring}{RGB}{102, 128, 153}
\node[inner sep=0pt] (probe) at (0,0) 
{\includegraphics[scale=0.35]{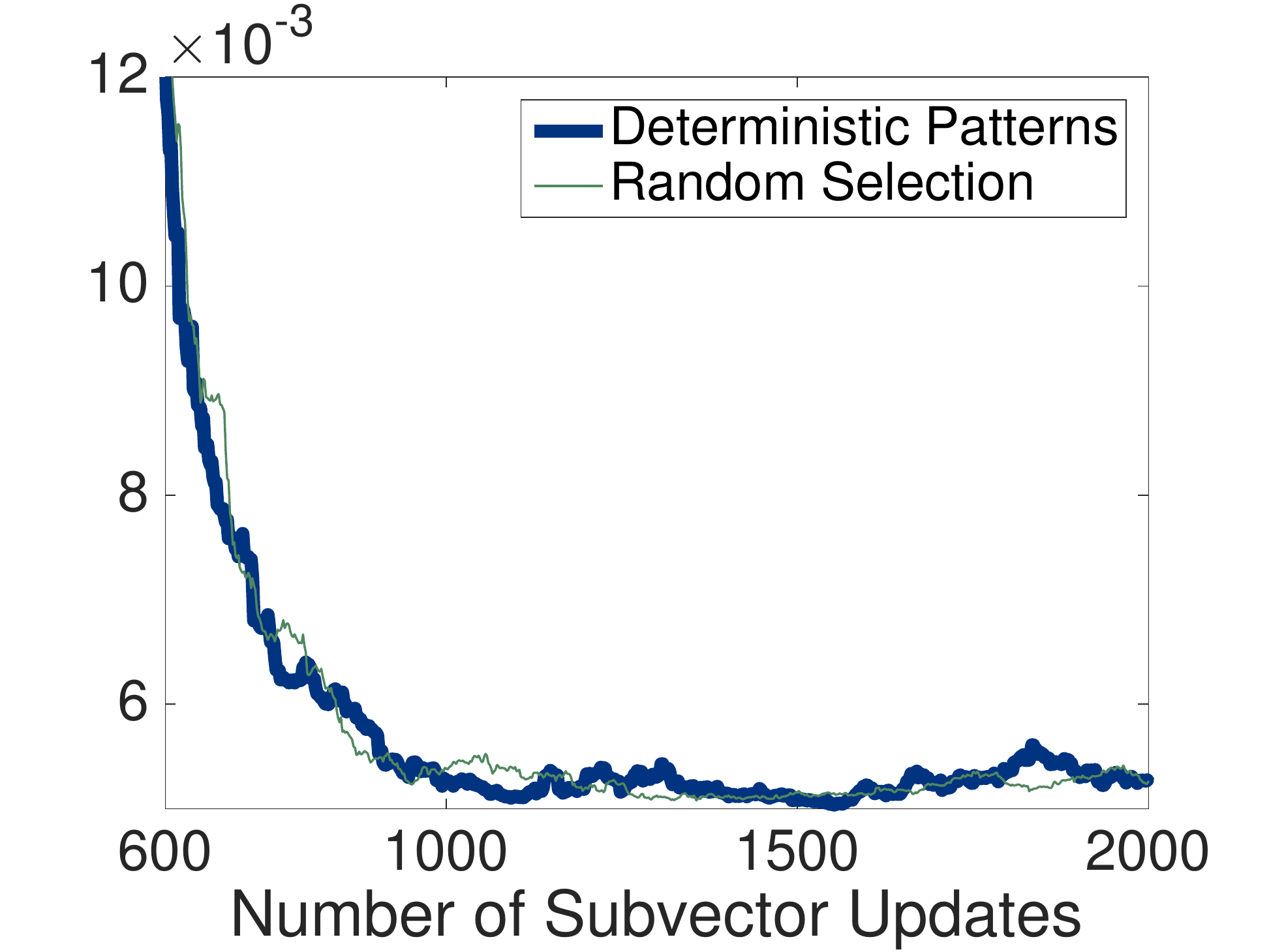}};
\node[rotate=90] at (-3.8,0) {\small{$L(\hat{\bm{\uptheta}}_u)$}};


\end{tikzpicture}
        \end{subfigure}\hfill
        \begin{subfigure}[b]{0.5\textwidth}
                \centering
\begin{tikzpicture}

  \definecolor{mynewcoloring}{RGB}{102, 128, 153}
\node[inner sep=0pt] (probe) at (0,0)
    {\includegraphics[scale=0.35]{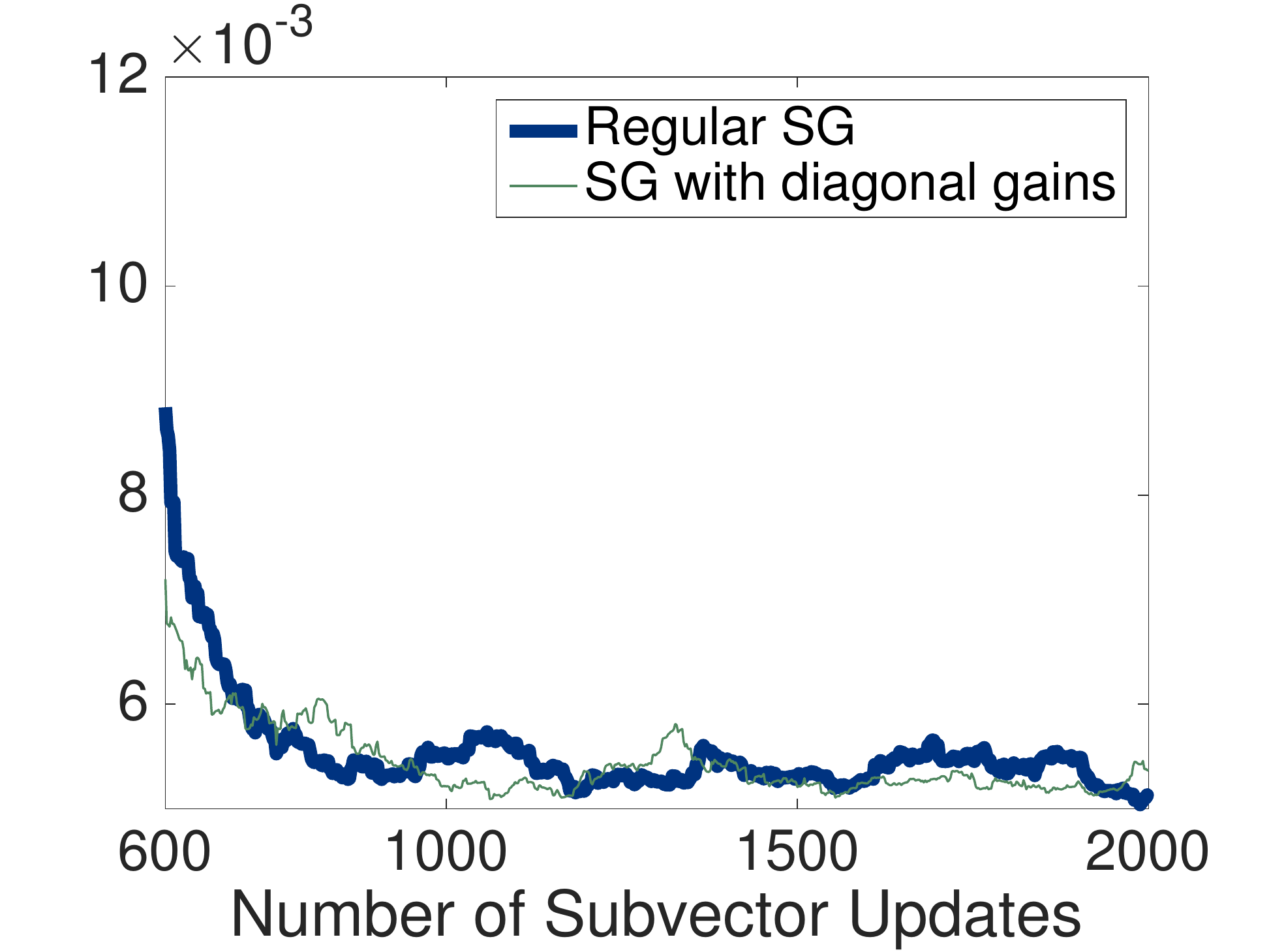}};
    \node[rotate=90] at (-3.35,0) {\small{$L(\hat{\bm{\uptheta}}_u)$}};

\end{tikzpicture}
 \end{subfigure}%
 \caption*{Above: the 2nd replication.}

 
  \vspace{.03in}
 
 
    \begin{subfigure}[b]{0.5\textwidth}
 \centering
\begin{tikzpicture}

  \definecolor{mynewcoloring}{RGB}{102, 128, 153}
\node[inner sep=0pt] (probe) at (0,0) 
{\includegraphics[scale=0.35]{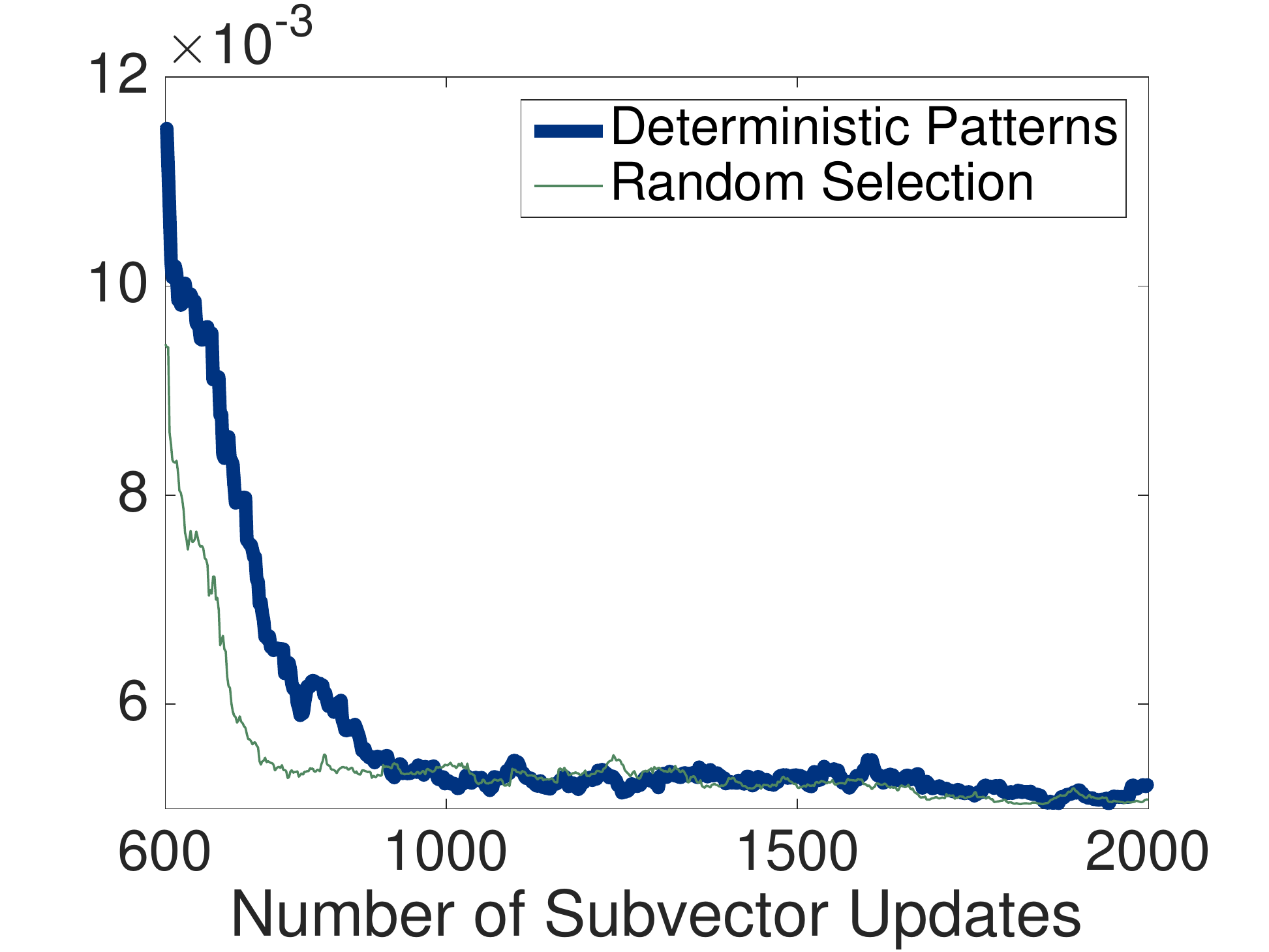}};
\node[rotate=90] at (-3.8,0) {\small{$L(\hat{\bm{\uptheta}}_u)$}};

\end{tikzpicture}
        \end{subfigure}\hfill
        \begin{subfigure}[b]{0.5\textwidth}
                \centering
\begin{tikzpicture}

  \definecolor{mynewcoloring}{RGB}{102, 128, 153}
\node[inner sep=0pt] (probe) at (0,0)
    {\includegraphics[scale=0.35]{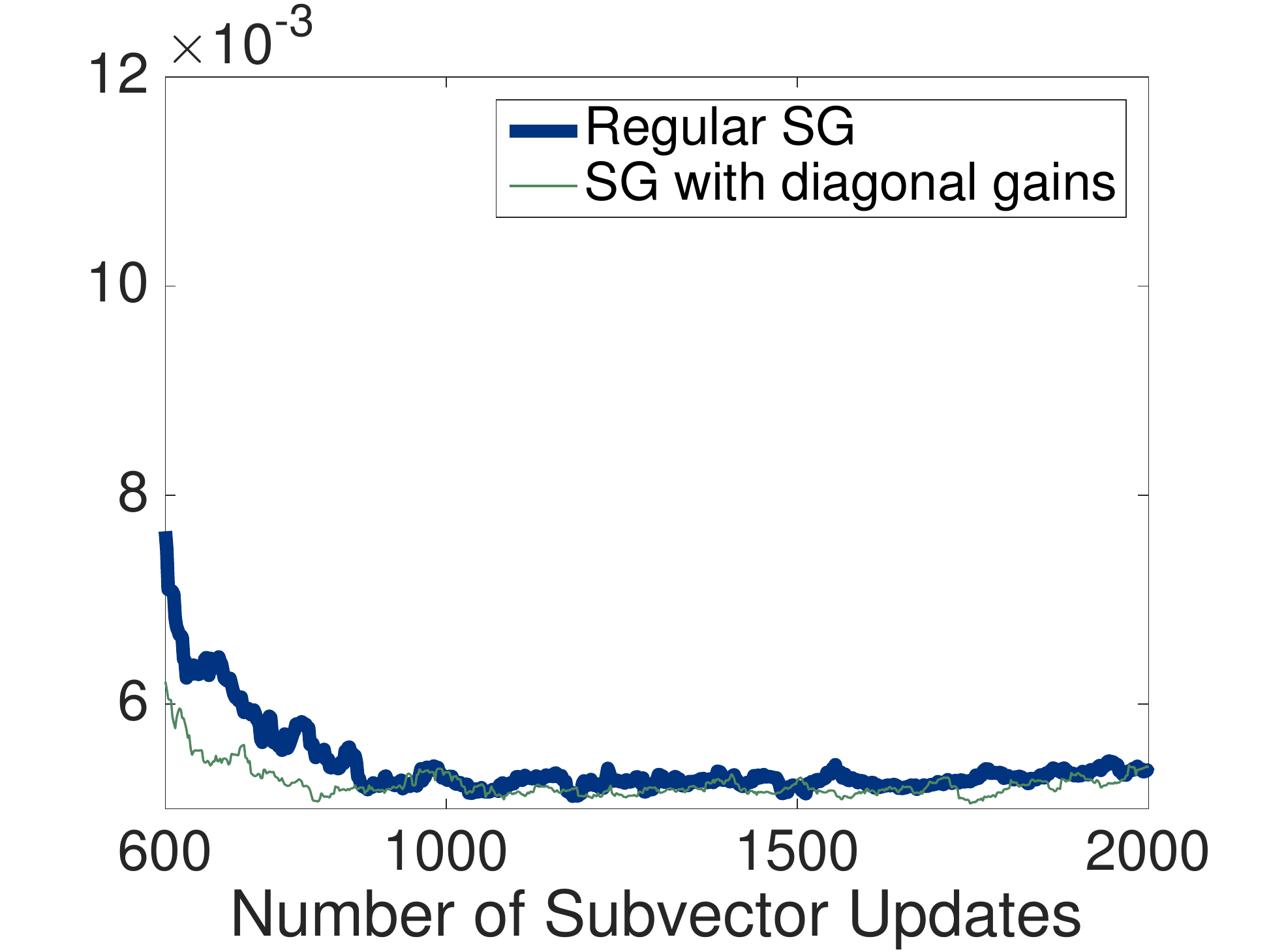}};
    \node[rotate=90] at (-3.5,0) {\small{$L(\hat{\bm{\uptheta}}_u)$}};

\end{tikzpicture}
 \end{subfigure}
  \caption*{Above: the 3rd replication.}
 
 
 \vspace{.03in}
         
        
\caption[Performance of Cases 1--4 from p. \pageref{pageref:toccasesSG} where the noisy gradient updates are SG-based]{Performance of Cases 1 (deterministic patterns), 2 (random selection), 3 (regular SG), and 4 (SG with diagonal gains) for three i.i.d. replications (i.e., three i.i.d. realizations of each of the four cases) using the tuned parameters from Table \ref{table:omgmySECONDtable}. The term $\hat{\bm{\uptheta}}_u$ is a generic vector representing the vector obtained after having performed $u$ subvector updates.}
\label{fig:betterplaybetstts}
\end{figure}

\section{Numerical Examples on Normality}
\label{sec:normalitycenumericos}

This section provides a few simple numerical examples regarding the asymptotic normality of the iterates from Algorithm \ref{kirkey} (where the subvector to updates is selected following a deterministic pattern). This algorithm is implemented using SG-based noisy update directions (Section \ref{subsec:normalitySG}) and SPSA-based noisy update directions (Section \ref{sec:hannahashley}).

\subsection{Algorithm \ref{kirkey} with SG-Based Gradient Estimates}
\label{subsec:normalitySG}

Consider the minimization of $L(\bm\uptheta)=E[Q(\bm\uptheta,\bm{V})|\bm\uptheta]$ where $Q(\bm\uptheta,\bm{V})=\uprho(\bm\uptheta)+\bm\uptheta^\top\bm{V}$ for some real-valued function $\uprho(\bm\uptheta)$ and a vector-valued random variable $\bm{V}$ that is independent of $\bm\uptheta$. Then, $L(\bm\uptheta)=\uprho(\bm\uptheta)+\bm\uptheta^\top E[\bm{V}]$. If the mean of $\bm{V}$ is unknown and $\uprho(\bm\uptheta)$ is continuously differentiable, an SG-based approach to solving this problem is obtained by letting $\hat{\bm{g}}_k(\hat{\bm{\uptheta}}_k)= [\partial \uprho(\bm\uptheta)/\partial \bm\uptheta]_{\bm\uptheta=\hat{\bm{\uptheta}}_k} + \bm{V}$. Here, $\hat{\bm{g}}_k(\hat{\bm{\uptheta}}_k)$ is an unbiased estimate of the gradient of $L(\bm\uptheta)$ at $\hat{\bm{\uptheta}}_k$. If it is possible to sample from $\bm{V}$ then it is possible to obtain $\hat{\bm{g}}_k(\hat{\bm{\uptheta}}_k)$ and update $\hat{\bm{\uptheta}}_k$ using SG:
\begin{align}
\label{eq:firfoes}
\hat{\bm{\uptheta}}_{k+1}=\hat{\bm{\uptheta}}_k-a_k\left(\left[\frac{\partial \uprho(\bm\uptheta)}{\partial \bm\uptheta}\right]_{\bm\uptheta=\hat{\bm{\uptheta}}_k}+\bm{V}_k\right),
\end{align}
where $\bm{V}_k$ denotes a measurement of $\bm{V}$. A strictly cyclic implementation of (\ref{eq:firfoes}) (i.e, an implementation resembling the cyclic seesaw SG algorithm except with $d$ subvectors) with non-overlapping subvectors is as follows: 
 \begin{align}
 \label{eq:tinkertonpinks}
\hat{\bm{\uptheta}}_{k}^{(I_j)}=\hat{\bm{\uptheta}}_k^{(I_{j-1})}-a_k^{(j)}\left(\left[\frac{\partial \uprho(\bm\uptheta)}{\partial \bm\uptheta}\right]_{\bm\uptheta=\hat{\bm{\uptheta}}_k^{(I_{j-1})}}+\bm{V}_k^{(I_{j-1})}\right)^{(j)}
\end{align}
for $1\leq j\leq d$, where $\hat{\bm{\uptheta}}_k^{(I_0)}\equiv \hat{\bm{\uptheta}}_k^{\text{cyc}}$, $\hat{\bm{\uptheta}}_{k+1}^{\text{cyc}}\equiv\hat{\bm{\uptheta}}_k^{(I_d)}$, and $\bm{V}_k^{(I_j)}$ is a vector-valued random variable representing a measurement of $\bm{V}$. Here, Theorem \ref{thm:fnogg} could be used to derive conditions for the asymptotic normality of $k^{\upbeta/2}(\hat{\bm{\uptheta}}_k^{\text{cyc}}-\bm\uptheta^\ast)$ where $\upbeta>0$ is some constant and $\bm\uptheta^\ast$ is a minimizer of $L(\bm\uptheta)$. We do this next.

Let us discuss conditions under which C0--C8, the conditions of Theorem \ref{thm:fnogg}, are satisfied. First note that if $\uprho(\bm\uptheta)$ is twice continuously differentiable with a bounded Jacobian matrix $\bm{J}(\bm\uptheta)$ and if $\hat{\bm{\uptheta}}_k^{(I_j)}\rightarrow \bm\uptheta^\ast$ w.p.1 for all $j$ then condition C0 is satisfied. Next, if $a_k^{(j)}=a_k=a/(1+k+A)^\upalpha$ for some $a, A >0, \upalpha>0$ then C1 is also satisfied with $r_j=a$. Moreover, since we are assuming the $d$ subvectors do not overlap when the matrix $\bm\Gamma$ in (\ref{eq:macnealhuh}) is equal to $a \bm{J}(\bm\uptheta^\ast)$. Therefore, if $\bm{J}(\bm\uptheta^\ast)$ is positive definite then condition C2 would also be satisfied with $a\bm{J}(\bm\uptheta^\ast)=\bm{P\Lambda P}^\top$ where $\bm{P}\in\mathbb{R}$ is an orthogonal matrix and $\bm\Lambda$ is a positive definite diagonal matrix. Next, letting $\mathcal{F}_k$ be the $\upsigma$-field generated by $\{\hat{\bm{\uptheta}}_i^{\text{cyc}}\}_{i=1}^{k}$, the bias term in condition C3 is identically zero (i.e., equal to the zero vector). Therefore, C3 is automatically satisfied. Finally, if the variables $\bm{V}_k^{(I_j)}$ for $j=1,\dots,d$ and $k\geq 0$ are i.i.d. and have (finite) positive definite covariance matrix, then conditions C4, C5-(i), C5-(iii), C6, and C7-(i) hold with $\bm\Sigma=a^2\sum_{j=1}^d \Var{([\bm{V}_k^{(I_j)}]^{(j)})}$, a positive definite block-diagonal matrix. Finally, if $0<\upalpha\leq 1$ (note that $\upbeta=\upalpha$ by C5-(i)), then a {\it{sufficient}} condition for B6 to be satisfied is if  the smallest diagonal entry of $\bm\Lambda$ is strictly greater than $1/2$ (if $\upalpha<1$ then this condition may be relaxed to having the smallest diagonal entry of $\bm\Lambda$ be strictly positive).

The discussion above gave conditions for the asymptotic normality of the vector $k^{\upbeta/2}(\hat{\bm{\uptheta}}_k^{\text{cyc}}-\bm\uptheta^\ast)$ with $\upbeta=\upalpha$. One important condition was the convergence w.p.1 of $\hat{\bm{\uptheta}}_k^{(I_j)}$ to $\bm\uptheta^\ast$. Corollary \ref{thm:hoeshooHarry} could be used to show convergence of $\hat{\bm{\uptheta}}_k^{\text{cyc}}$ to $\bm\uptheta^\ast$. Then, Corollary \ref{eq:idamagedittarantatan} would imply the desired result. Conditions A0$''$, A3$''$, A5$''$,  and A6$''$, are automatically satisfied given the assumptions of the previous paragraph. It remains to discuss the validity of A4$''$, A7$''$, and A8$''$. Condition A4$''$ requires the algorithm's iterates to be bounded w.p.1. Verifying this in practice is impossible (see Section \ref{sec:discussconvergence}). 
Next, A7$''$ requires $\bm{\uptheta^{\ast}}$ to be a locally asymptotically stable (in the sense of Lyapunov) solution of:
\begin{align}
\label{eq:chirrinpinpin4}
\dot{\bm{Z}}(t)=- a \left(\left[\frac{\partial \uprho(\bm\uptheta)}{\partial \bm\uptheta}\right]_{\bm\uptheta=\bm{Z}(t)}+E[\bm{V}]\right).
\end{align}
One example of a situation in which $\bm\uptheta^\ast$ is a locally asymptotically stable solution to (\ref{eq:chirrinpinpin4}) is when $\uprho(\bm\uptheta)=[\bm\uptheta^\top\bm{H}\bm\uptheta]/2$ for some positive definite matrix $\bm{H}$. Throughout the remainder of this section we will assume $\uprho(\bm\uptheta)$ has this form. In this case, $\dot{\bm{Z}}(t)=- a\bm{H}(\bm\uptheta+\bm{H}^{-1}E[\bm{V}])$. Because $\bm\uptheta^\ast=-\bm{H}^{-1}E[\bm{V}]$, $\bm\uptheta^\ast$ is an equilibrium point of $\dot{\bm{Z}}(t)$.
The desired Lyapunov stability follows from Lyapunov's second method for stability (e.g., Cronin 2007\nocite{cronin2007}), and therefore A7$''$ holds. Condition A8$''$ is also impossible to verify in practice. It requires the iterates to be contained within the domain of attraction of (\ref{eq:chirrinpinpin4}) infinitely often w.p.1. Next we conduct a small numerical experiment to illustrate convergence and asymptotic normality under the assumptions made thus far. Aside from A4$''$ and A8$''$, all the conditions for convergence w.p.1 and asymptotic normality are satisfied.

For our numerical experiments we implement (\ref{eq:tinkertonpinks}) assuming $p=2$ (i.e., $\bm\uptheta\in\mathbb{R}^2$) and that the vectors $\bm{V}_k^{(I_j)}$ are i.i.d. with $\bm{V}_k^{(I_j)}\sim\mathcal{N}(\bm\upmu,\bm{I})$, where $\bm{\upmu}=[5, 5]^\top$. In this setting there are two subvectors to update corresponding to the first- and second entries of $\bm\uptheta$. The matrix $\bm{H}$ was set to be equal to the identity matrix. For the gain sequences the values used were $\upalpha=0.501$, $a=1$, and $A=100$.  Under these assumptions, $\bm\uptheta^\ast=-[5,5]^\top$ and $k^{0.501/2}(\hat{\bm{\uptheta}}_k^{\text{cyc}}-\bm\uptheta^\ast)$ should converge in distribution to a random variable with distribution $\mathcal{N}(\bm{0},0.5\bm{I})$. Figure \ref{fig:figgyanotherorbs} gives the results for $\hat{\bm{\uptheta}}_0^{\text{cyc}}=[1,1]^\top$. The results in Figure \ref{fig:lilly1} appear to support the fact that the values of $T^{0.501/2}(\hat{\bm{\uptheta}}_T^{\text{cyc}}-\bm\uptheta^\ast)$ are expected to have a multivariate normal distribution with a diagonal covariance matrix (this can be seen since the cluster of points in Figure \ref{fig:lilly1} is roughly symmetric) for $T=1000$. Additionally, because the sum of the entries of a random variable with distribution $\mathcal{N}(\bm{0},0.5\bm{I})$ should have a normal distribution with mean zero and variance 1, a Q-Q plot could be used to compare the quantiles of a standard normal random variable to the quantiles of the sum of the entries of $T^{0.501/2}(\hat{\bm{\uptheta}}_T^{\text{cyc}}-\bm\uptheta^\ast)$. Figure \ref{fig:lilly2} shows the resulting Q-Q plot. It appears that the sum of the entries of the normalized iterates closely approximates a standard normal random variable as desired.

\begin{figure}
\centering
        \begin{subfigure}[b]{0.5\textwidth}
                \centering
\begin{tikzpicture}

  \definecolor{mynewcoloring}{RGB}{102, 128, 153}
\node[inner sep=0pt] (probe) at (0,0) 
{\includegraphics[scale=0.35]{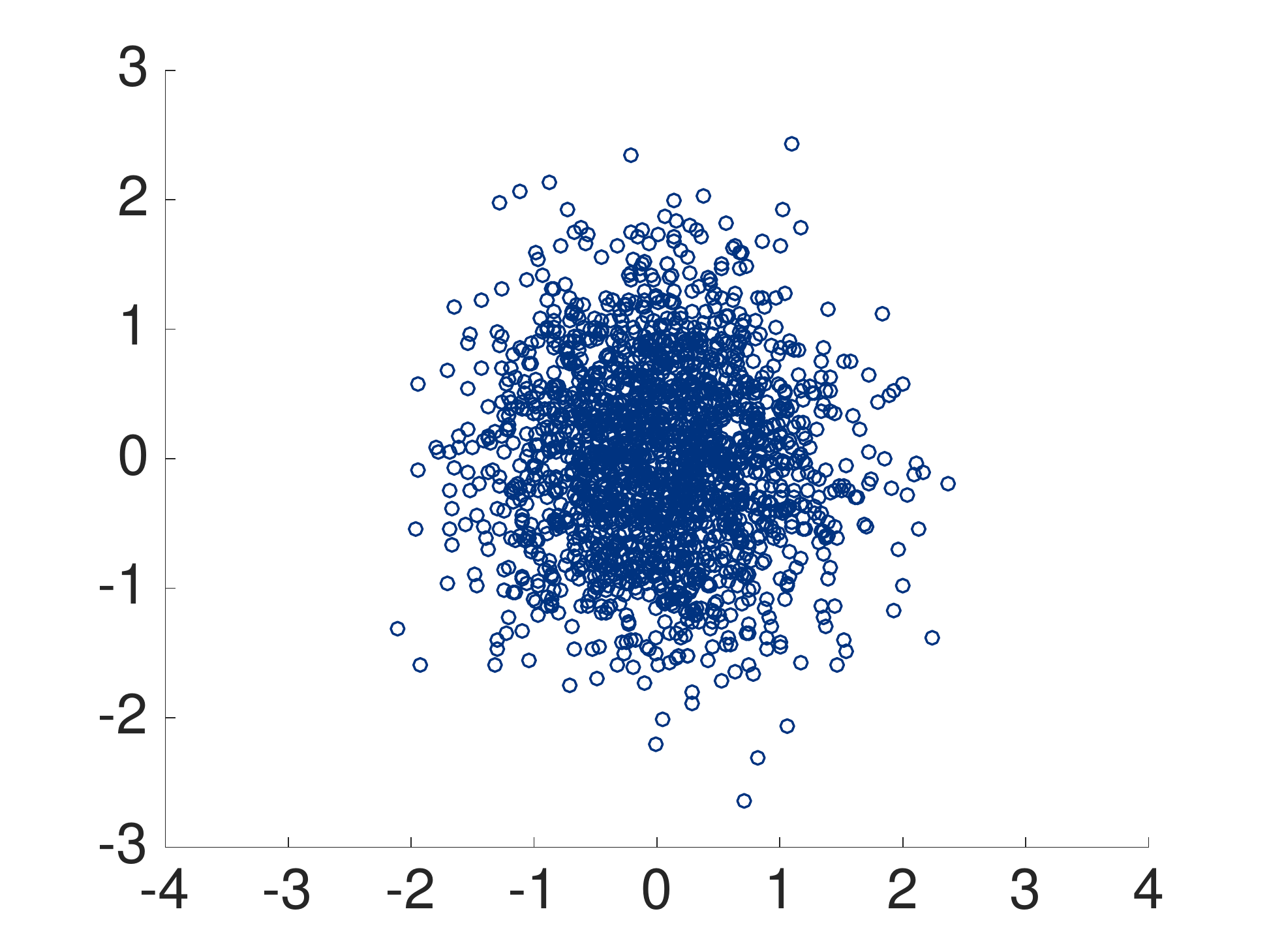}};
\node at (0,-2.6) {\textcolor{white}{$\uptheta_1$}};

\end{tikzpicture}
                \caption{Scatter plot of $T^{\upbeta/2}(\hat{\bm{\uptheta}}_{T}^{\text{cyc}}-\bm\uptheta^\ast)$.}
                \label{fig:lilly1}
        \end{subfigure}\hfill
        \begin{subfigure}[b]{0.5\textwidth}
                \centering
\begin{tikzpicture}

  \definecolor{mynewcoloring}{RGB}{102, 128, 153}
\node[inner sep=0pt] (probe) at (0,0)
    {\includegraphics[scale=0.35]{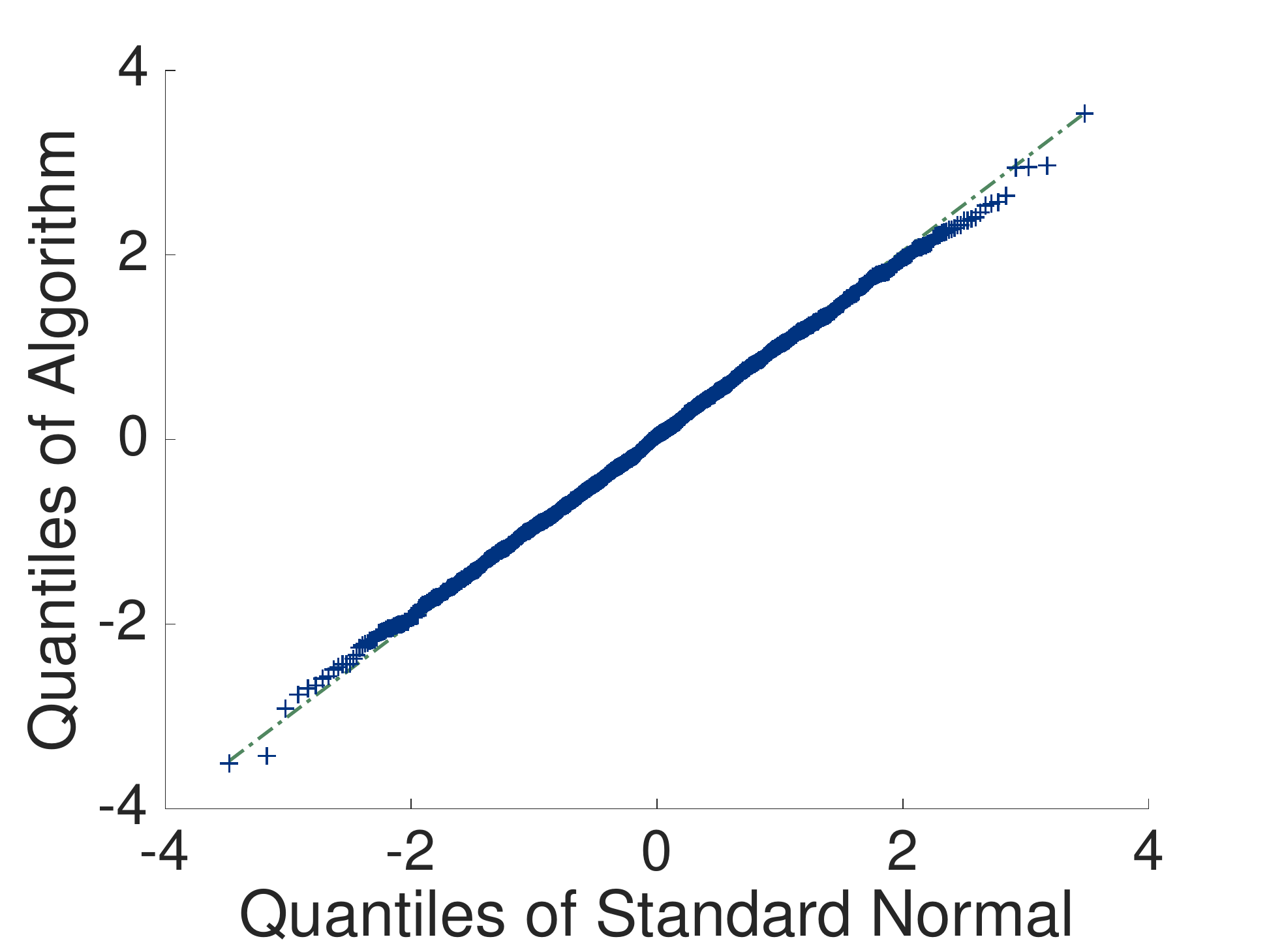}};

\end{tikzpicture}
                \caption{Q-Q plot of $[T^{\upbeta/2}(\hat{\bm{\uptheta}}_{T}^{\text{cyc}}-\bm\uptheta^\ast)]^\top\bm{1}$.}
                \label{fig:lilly2}
        \end{subfigure}\hfill
        \caption[Scatter- and Q-Q plots supporting the asymptotic normality of the normalized cyclic seesaw SG iterates]{In the plots above the term ``$\bm{1}$'' denotes the vector of ones, $\upbeta=0.501$, and $T=1000$. The plots are the result of 2,000 i.i.d. replications.}
        \label{fig:figgyanotherorbs}
\end{figure}

\subsection{Algorithm \ref{kirkey} with SPSA-Based Gradient Estimates}
\label{sec:hannahashley}

This section considers the same problem from Section \ref{subsec:normalitySG}, that is the minimization of $L(\bm\uptheta)=\bm\uptheta^\top\bm{H}\bm\uptheta/2+\bm\uptheta^\top E[\bm{V}]$. However, rather than using noisy gradient measurements of the form  used in (\ref{eq:tinkertonpinks}), this section relies only on noisy measurements, $Q(\bm\uptheta,\bm{V})=\bm\uptheta^\top\bm{H}\bm\uptheta/2+\bm\uptheta^\top\bm{V}$, in the optimization process. This is done using the cyclic seesaw SPSA algorithm from Section \ref{sect:cspsa}. Specifically, an iteration of the algorithm is given by:
\begin{subequations}
\begin{align}
\label{eq;tratarakatnatantan0}
\hat{\bm{\uptheta}}_{k}^{(I)}&=\hat{\bm{\uptheta}}_k^\text{cyc}-a_k^{(1)}\hat{\bm{g}}^{\text{SP}}_k(\hat{\bm{\uptheta}}_k^\text{cyc}),\\
 \hat{\bm{\uptheta}}_{k+1}^{\text{cyc}}&=\hat{\bm{\uptheta}}_k^{(I)}-a_k^{(2)}\hat{\bm{g}}^{\text{SP}}_k(\hat{\bm{\uptheta}}_k^{(I)}),\label{eq;tratarakatnatantan}
\end{align}
\end{subequations}
where $\hat{\bm{g}}^{\text{SP}}_k(\hat{\bm{\uptheta}}_k^\text{cyc})$ and $\hat{\bm{g}}^{\text{SP}}_k(\hat{\bm{\uptheta}}_k^{(I)})$ are as  in (\ref{eq:howeardspos1}) and (\ref{eq:howeardspos2}), respectively. In order to implement the cyclic seesaw SPSA algorithm in (\ref{eq;tratarakatnatantan0},b) it is necessary to specify the values of $a_k^{(j)}$, $c_k^{(j)}$, $\bm\Delta_k$, and $\bm\Delta_k^{(I)}$ (see p. \pageref{eq:howeardspos2} for the definition of the last three terms). In this section's numerical experiment the non-zero entries of $\bm\Delta_k$ and $\bm\Delta_k^{(I)}$ were set to be i.i.d. random variables taking the values $\pm 1$ with equal probability. We also set $a_k^{(j)}=1/(1+k+A)^\upalpha$ and $c_k^{(j)}=1/(1+k)^\upgamma$ for all $j$, where $\upalpha=0.602$, $\upgamma=0.101$, and $A=500$. All measurements of the random variable $\bm{V}$ are assumed to be i.i.d. with distribution $\mathcal{N}(\bm\upmu,0.1\bm{I})$ where $\bm\upmu$ is the vector of fives. Next, we show convergence w.p.1 of $\hat{\bm{\uptheta}}_k^{\text{cyc}}$ to $\bm\uptheta^\ast$ (the minimizer of $L(\bm\uptheta)$) and we derive the asymptotic distribution of $k^{\upbeta/2}(\hat{\bm{\uptheta}}_k^{\text{cyc}}-\bm\uptheta^\ast)$ for $\upbeta=\upalpha-2\upgamma$.

To prove convergence of $\hat{\bm{\uptheta}}_k^{\text{cyc}}$ to $\bm\uptheta^\ast$, we verify that the conditions of Corollary \ref{thm:hoeshooHarry} hold. The convergence of $\hat{\bm{\uptheta}}_k^{(I)}$ to $\bm\uptheta^\ast$ follows from Corollary \ref{eq:idamagedittarantatan}. First, it is easy to see that condition A0$''$ holds with $r_j=1$ for all $j$. Similarly, it is easy to see that condition A3$''$ is satisfied. Condition A4$''$ is impossible to guarantee in practice and, in our numerical experiment, we do not impose any bounds forcing A4$''$ to hold. In order to verify A5$''$ and A6$''$ we must express the noisy gradient vectors $\hat{\bm{g}}^{\text{SP}}_k(\hat{\bm{\uptheta}}_k^\text{cyc})$ and $\hat{\bm{g}}^{\text{SP}}_k(\hat{\bm{\uptheta}}_k^{(I)})$ in terms of their bias- and noise terms given in (\ref{eq:savemebias0},b) and (\ref{eq:savemenoise0},b), respectively. 
In our example ${\bm{\upbeta}}^{(1)}_k(\hat{\bm{\uptheta}}_k^{\text{cyc}})={\bm{\upbeta}}^{(1)}_k(\hat{\bm{\uptheta}}_k^{(I)})=\bm{0}$ (due to the fact that the loss function is quadratic in $\bm\uptheta$) and 
the $m$th entry of the noise term ${\bm{\upxi}}^{(1)}_k(\hat{\bm{\uptheta}}_k^{\text{cyc}})$ is equal to:
\begin{align*}
\left[{\bm{\upxi}}^{(1)}_k(\hat{\bm{\uptheta}}_k^{\text{cyc}})\right]^{[m]}&=\frac{\bm\Delta_k^\top(\bm{V}_k^++\bm{V}_k^-)}{2\Delta_k^{[m]}}+\frac{\bm\Delta_k^\top\bm{H}\hat{\bm{\uptheta}}_k^{\text{cyc}}}{\Delta_k^{[m]}}+\frac{(\hat{\bm{\uptheta}}_k^{\text{cyc}})^\top(\bm{V}_k^+-\bm{V}_k^-)}{2c_k\Delta_k^{[m]}}-\left[\bm{g}(\hat{\bm{\uptheta}}_k^{\text{cyc}})\right]^{[m]}
\end{align*}
for $m\in \mathcal{S}_1$ (see p. \pageref{eq:notexclusive} for the definition of $\mathcal{S}_j$) and $[{\bm{\upxi}}^{(1)}_k(\hat{\bm{\uptheta}}_k^{\text{cyc}})]^{[m]}=0$ otherwise. Similarly, the $m$th entry of the noise term ${\bm{\upxi}}^{(2)}_k(\hat{\bm{\uptheta}}_k^{\text{cyc}})$ is equal to:
\begin{align*}
&\left[{\bm{\upxi}}^{(2)}_k(\hat{\bm{\uptheta}}_k^{(I)})\right]^{[m]}=\notag\\
&\frac{(\bm\Delta_k^{(I)})^\top(\bm{V}_k^{(I)+}+\bm{V}_k^{(I)-})}{2(\bm\Delta_k^{(I)})^{[m]}}+\frac{(\bm\Delta_k^{(I)})^\top\bm{H}\hat{\bm{\uptheta}}_k^{(I)}}{(\bm\Delta_k^{(I)})^{[m]}}+\frac{(\hat{\bm{\uptheta}}_k^{(I)})^\top(\bm{V}_k^{(I)+}-\bm{V}_k^{(I)-})}{2c_k(\bm\Delta_k^{(I)})^{[m]}}-\left[\bm{g}(\hat{\bm{\uptheta}}_k^{(I)})\right]^{[m]}
\end{align*}
for $m\in \mathcal{S}_2$ and $[{\bm{\upxi}}^{(2)}_k(\hat{\bm{\uptheta}}_k^{(I)})]^{[m]}=0$ otherwise.
Given that the bias terms are identically zero, condition A5$''$ is satisfied. Using the independence of the $\bm\Delta$ and $\bm{V}$ and the fact that both these variables have finite variance, condition A6$''$ follows from the discussion on p. \pageref{eq:ranch} (using the fact that $\sum_{k=1}^\infty a_k^2/c_k^2<\infty$). The validity of condition A7$''$ has already been shown in Section \ref{subsec:normalitySG}. Finally, condition A8$''$, like  condition A4$''$, is impossible to verify. Aside from conditions A4$''$ and A8$''$, all the conditions for convergence w.p.1 hold. Next, we verify the conditions for the asymptotic normality of $k^{(\upalpha-2\upgamma)/2}(\hat{\bm{\uptheta}}_k^{\text{cyc}}-\bm\uptheta^\ast)$ and compute the parameters of the asymptotic distribution.

To prove the desired result on asymptotic normality we show that the conditions of Theorem \ref{thm:fnogg} are satisfied. For simplicity, we will assume $\bm\uptheta\in \mathbb{R}^2$. The validity of C0--C3 has already been shown (see Section \ref{subsec:normalitySG} and the previous paragraph). Condition C4 follows immediately by construction. Let us now show that condition C5-(iv) holds.
First, because both noise terms have mean zero (conditionally on $\mathcal{F}_k$, where $\mathcal{F}_k$ was defined in Section \ref{subsec:normalitySG}), the conditional covariance matrix of ${\bm{\upxi}}^{(1)}_k(\hat{\bm{\uptheta}}_k)$ and ${\bm{\upxi}}^{(2)}_k(\hat{\bm{\uptheta}}_k^{(I)})$ is $E[\bm\upxi_k^{(1)}(\hat{\bm{\uptheta}}_k)\bm\upxi_k^{(2)}(\hat{\bm{\uptheta}}_k^{(I)})^\top|\mathcal{F}_k]$. Using the convergence (w.p.1) of the iterates to $\hat{\bm{\uptheta}}_k^\ast$ along with the continuity of the gradient of $L(\bm\uptheta)$  it is easy to show that $E[\bm\upxi_k^{(1)}(\hat{\bm{\uptheta}}_k)\bm\upxi_k^{(2)}(\hat{\bm{\uptheta}}_k^{(I)})^\top|\mathcal{F}_k]=O(1/c_k)$ w.p.1. However, $c_k=O(1/k^{\upgamma})$ which implies $\lim_{k\rightarrow \infty}k^{\upbeta-\upalpha}/c_k= 0$. Therefore, $k^{\upbeta-\upalpha}E[\bm\upxi_k^{(1)}(\hat{\bm{\uptheta}}_k)\bm\upxi_k^{(2)}(\hat{\bm{\uptheta}}_k^{(I)})^\top|\mathcal{F}_k]\rightarrow \bm{0}$ w.p.1. Next we verify the remaining conditins for asymptotic normality.

\begin{figure}
\centering
        \begin{subfigure}[b]{0.5\textwidth}
                \centering
\begin{tikzpicture}

  \definecolor{mynewcoloring}{RGB}{102, 128, 153}
\node[inner sep=0pt] (probe) at (0,0) 
{\includegraphics[scale=0.35]{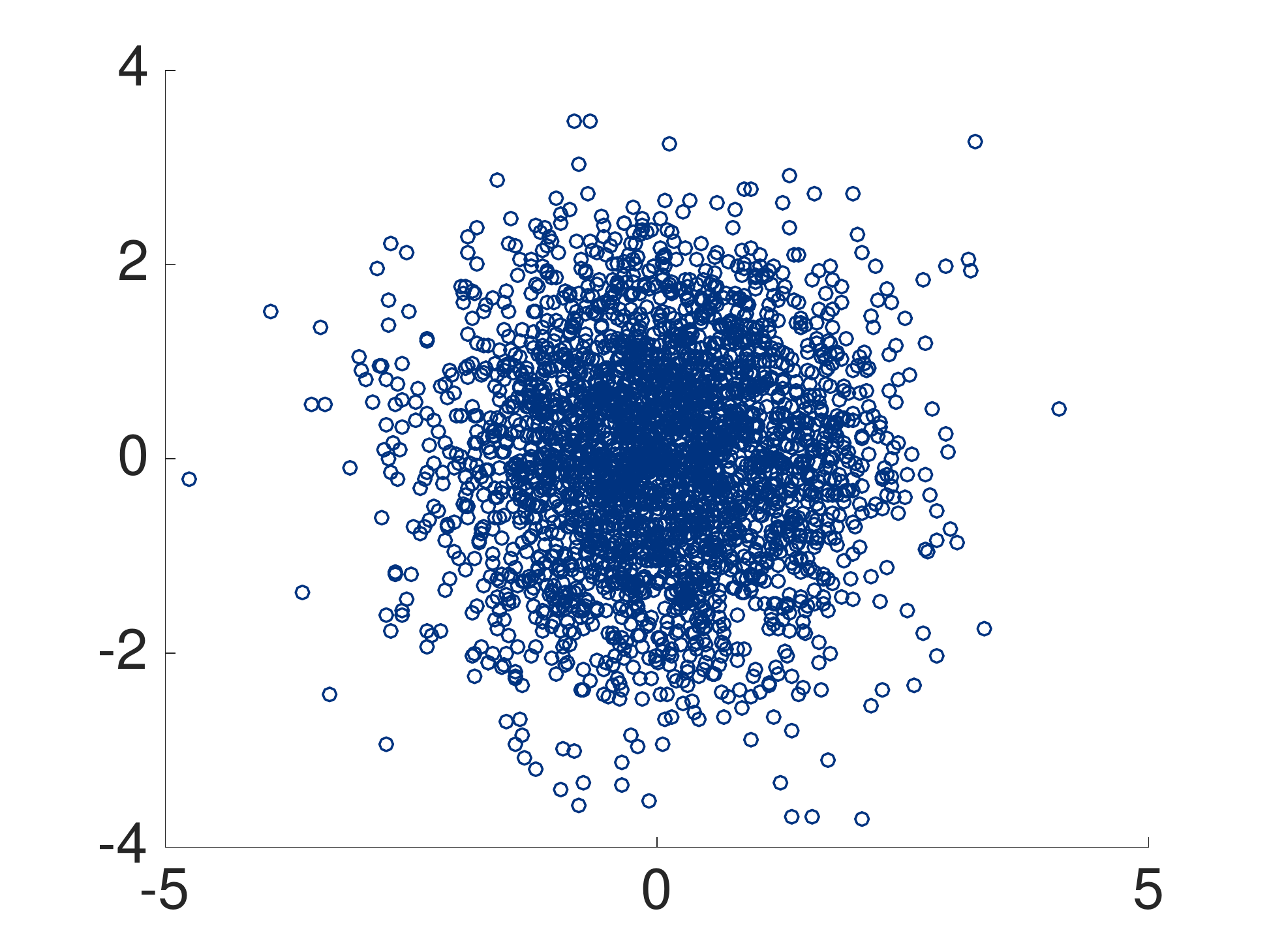}};
\node at (0,-2.6) {\textcolor{white}{$\uptheta_1$}};

\end{tikzpicture}
                \caption{Scatter plot of $T^{\upbeta/2}(\hat{\bm{\uptheta}}_{T}^{\text{cyc}}-\bm\uptheta^\ast)$.}
                \label{fig:lorberton1}
        \end{subfigure}\hfill
        \begin{subfigure}[b]{0.5\textwidth}
                \centering
\begin{tikzpicture}

  \definecolor{mynewcoloring}{RGB}{102, 128, 153}
\node[inner sep=0pt] (probe) at (0,0)
    {\includegraphics[scale=0.35]{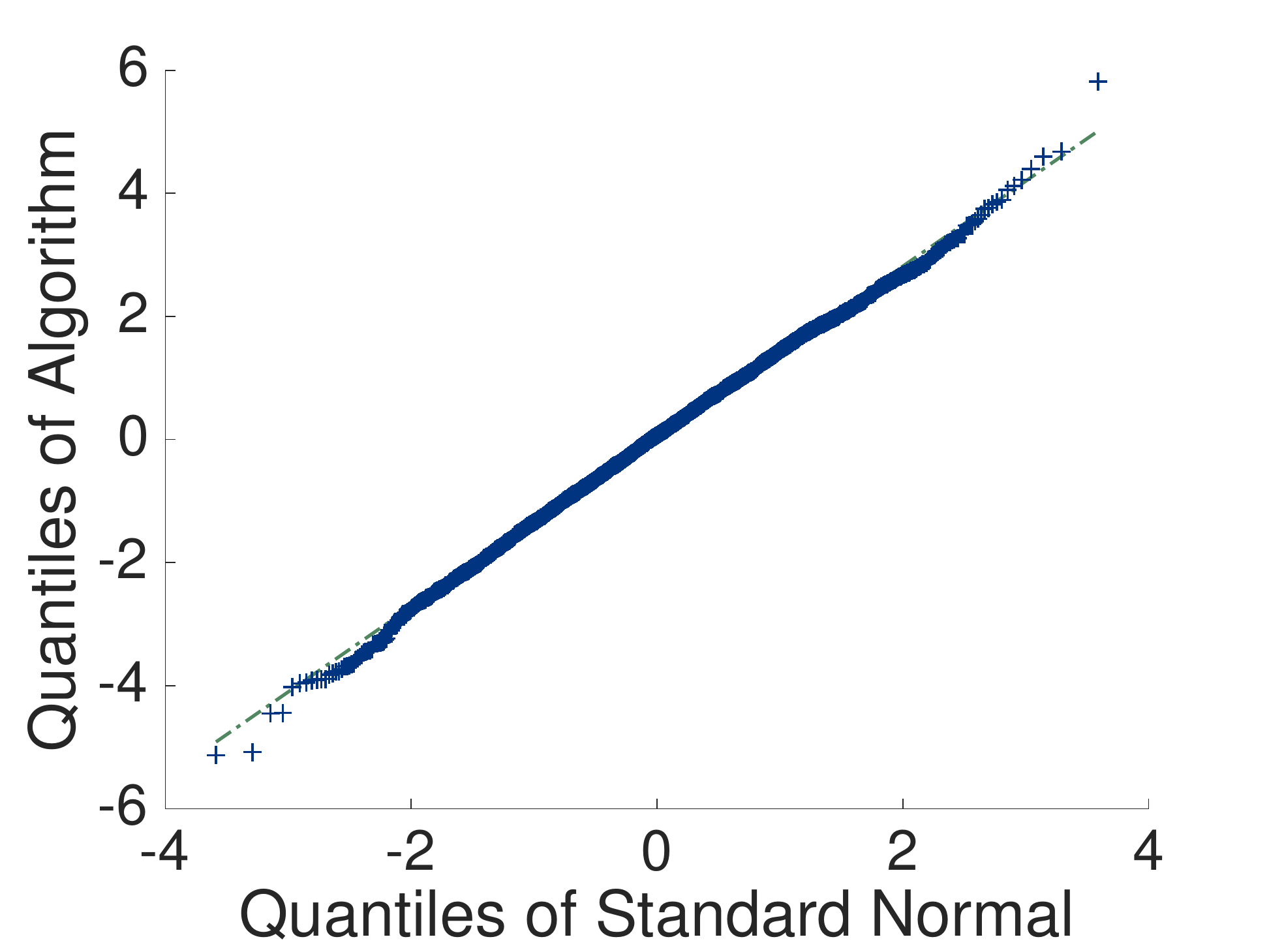}};

\end{tikzpicture}
                \caption{Q-Q plot of $(8/10)^{1/2}[T^{\upbeta/2}(\hat{\bm{\uptheta}}_{T}^{\text{cyc}}-\bm\uptheta^\ast)]^\top\bm{1}$.}
                \label{fig:lorberton2}
        \end{subfigure}\hfill
        \caption[Scatter- and Q-Q plots supporting the asymptotic normality of the normalized cyclic seesaw SPSA iterates]{In the plots above the term ``$\bm{1}$'' denotes the vector of ones, $\upbeta=\upalpha-2\upgamma=0.4$, and $T=3000$. The plots are the result of 3,000 i.i.d. replications.}
        \label{fig:flyingorber}
\end{figure}


After a direct calculation of the conditional covariance matrices of the two noise terms it follows that C5-(ii) and C6 hold with $\bm\Sigma=(0.1/4)(\bm\uptheta^\ast)^\top\bm\uptheta^\ast\bm{I}$. Condition C7--(ii) follows form the uniform integrability of the weighted (by $k^{\upbeta-\upalpha}$) noise terms.
  Condition C8 follows from the fact that $\bm{H}$ is positive definite. If we assume $\bm{H}=\bm{I}$ for simplicity, the asymptotic distribution of $k^{(\upalpha-2\upgamma)/2}(\hat{\bm{\uptheta}}_k^{\text{cyc}}-\bm\uptheta^\ast)$ according to Theorem \ref{thm:fnogg} is then $\mathcal{N}(\bm{0},(5/8)\bm{I})$. Furthermore, $(8/10)^{1/2}[T^{\upbeta/2}(\hat{\bm{\uptheta}}_{T}^{\text{cyc}}-\bm\uptheta^\ast)]^\top\bm{1}\sim \mathcal{N}(0,1)$. Figure \ref{fig:flyingorber} gives the results for $\hat{\bm{\uptheta}}_0^{\text{cyc}}=[1,1]^\top$. The results in Figure \ref{fig:lorberton1} appears to support the fact that the values of $T^{(\upalpha-2\upgamma)/2}(\hat{\bm{\uptheta}}_T^{\text{cyc}}-\bm\uptheta^\ast)$ are expected to have a multivariate normal distribution with a diagonal covariance matrix (this can be seen since the cluster of points in Figure \ref{fig:lorberton1} is roughly symmetric) for $T=3000$. Moreover, Figure \ref{fig:lorberton2} shows that $(8/10)^{1/2}[T^{\upbeta/2}(\hat{\bm{\uptheta}}_{T}^{\text{cyc}}-\bm\uptheta^\ast)]^\top\bm{1}$ appears to have a standard normal distribution as desired.

  \section{Numerical Analysis on Efficiency}
\label{sec:Numericssimples}

This section contains simulation-based estimates of  the relative efficiency ratio (\ref{eq:wobble}) for comparing the SG- and cyclic seesaw SG algorithms when $c^{\text{cyc}}$ and $c^{\text{non}}$ are defined as in Section \ref{sec:costoport} (this is done in Section \ref{eq:Indyiedie1}) and for comparing the SPSA- and cyclic seesaw SPSA algorithms under a few different definitions of $c^{\text{cyc}}$ and $c^{\text{non}}$ (this is done in Section \ref{eq:Indyiedie12}).

\subsection{Efficiency: SG versus Cyclic Seesaw SG}
\label{eq:Indyiedie1}

This section compares the SG algorithm to the cyclic seesaw SG algorithm by estimating the relative efficiency given by the ratio in (\ref{eq:wobble}). The SG algorithm treated is a distributed version of the LMS algorithm in (\ref{eq:colorlocalcouncil}) where measurements of the pair $(\bm{h}_k,z_k)$ with $z_k=\bm{h}_k^\top \bm\uptheta^\ast+\upvarepsilon_k$ are used to estimate $\bm\uptheta^\ast$. The non-cyclic implementation of LMS considered is the following:
\begin{subequations}
\begin{align}
[\hat{\bm{\uptheta}}_{k+1}^{\text{non}}]^{(1)}=[\hat{\bm{\uptheta}}_k^{\text{non}}]^{(1)}-a_k^{(1)}\left[(\bm{h}_{k+1}^\top \hat{\bm{\uptheta}}_k^{\text{non}}-z_{k+1})\bm{h}_{k+1}\right]^{(1)},\label{eq:fafafafafafufahfah0}\\
[\hat{\bm{\uptheta}}_{k+1}^{\text{non}}]^{(2)}=[\hat{\bm{\uptheta}}_k^{\text{non}}]^{(2)}-a_k^{(2)}\left[(\bm{h}_{k+1}^\top \hat{\bm{\uptheta}}_k^{\text{non}}-z_{k+1})\bm{h}_{k+1}\right]^{(2)}, \label{eq:fafafafafafufahfah}
\end{align}
\end{subequations}
where $\bm\uptheta\in \mathbb{R}^p$ for $p=12$ even and $p'=p/2=6$ is the subvector length. In the distributed implementation considered here, (\ref{eq:fafafafafafufahfah0}) and (\ref{eq:fafafafafafufahfah}) are computed at physically different locations. To highlight the distributed nature of the implementation, throughout the remainder of this section we use the multi-agent analogy from Section \ref{sec:costoport} (see p. \pageref{eq:herculedherculess}). In other words, we assume that each of the two lines in (\ref{eq:fafafafafafufahfah0},b) is computed by a different agent.

The cyclic seesaw counterpart to (\ref{eq:fafafafafafufahfah0},b), which is also assumed to be implemented in a distributed manner, is defined as follows:
\begin{subequations}
\begin{align}
\hat{\bm{\uptheta}}_k^{(I)}&=\hat{\bm{\uptheta}}_k^{\text{cyc}}-a_k^{(1)}\left[(\bm{h}_{k+1}^\top \hat{\bm{\uptheta}}_k^{\text{cyc}}-z_{k+1})\bm{h}_{k+1}\right]^{(1)},\label{eq:addalabeltomee0}\\
\hat{\bm{\uptheta}}_{k+1}^{\text{cyc}}&=\hat{\bm{\uptheta}}_k^{(I)}-a_k^{(2)}\left[(\bm{h}_{k+1}^\top \hat{\bm{\uptheta}}_k^{(I)}-z_{k+1})\bm{h}_{k+1}\right]^{(2)}.
\label{eq:addalabeltomee}
\end{align}
\end{subequations}
Once again, we use the multi-agent analogy from Section \ref{sec:costoport} (see p. \pageref{eq:herculedherculess}) to highlight the fact that the cyclic seesaw algorithms is assumed to be implemented in a distributed manner. Note that 
$(\bm{h}_{k+1},z_{k+1})$ is used to perform two subvector updates 
in  (\ref{eq:addalabeltomee0},b): once to obtain $\hat{\bm{\uptheta}}_k^{(I)}$ and once to obtain $\hat{\bm{\uptheta}}_{k+1}^{\text{cyc}}$. The pair $(\bm{h}_{k+1},z_{k+1})$ is also used to perform two subvector updates in 
 (\ref{eq:fafafafafafufahfah0},b). 

\subsubsection{Cost as a Measure of Arithmetic Computations}
\label{sec:SGarithmetic}

The first definition of cost we consider for comparing (\ref{eq:fafafafafafufahfah0},b) to  (\ref{eq:addalabeltomee0},b) via  (\ref{eq:ratio}) is the arithmetic cost of Section \ref{sec:costoport}. Here, the per-iteration costs $c^{\text{non}}$ and $c^{\text{cyc}}$ denote the number of basic arithmetic operations required to obtain $\hat{\bm{\uptheta}}_{k+1}^\text{non}$ from $\hat{\bm{\uptheta}}_{k}^\text{non}$ and to obtain $\hat{\bm{\uptheta}}_{k+1}^\text{cyc}$ from $\hat{\bm{\uptheta}}_{k}^\text{cyc}$, respectively. In order to compute these quantities it is necessary to specify the computational graphs (CGs) associated with the algorithms, these are given in Figures \ref{fig:figgystuff1} and \ref{fig:figgystuff2}. From Figure \ref{fig:figgystuff1} we see that the per-iteration cost associated with the SG algorithm is $c^{\text{non}}=4p+3$ operations ($3p+1$ operations from the first graph and $p+2$ from the second). In contrast, Figure \ref{fig:figgystuff2} implies $c^{\text{cyc}}=5p+3$ operations ($3p+1$ operations from the first graph and $2p+2$ from the second).

In our numerical examples we use $p=12$ so that $c^\text{non}=51$, $c^\text{cyc}=63$, and $c^\text{cyc}/c^\text{non}=21/17\approx 5/4$. The significance of this ratio, as was explained in the discussion following Proposition \ref{prop:omygodsthunder}, is that the cost of computing $\hat{\bm{\uptheta}}_{4i}^{\text{cyc}}$ is approximately equal to the cost of computing $\hat{\bm{\uptheta}}_{5i}^{\text{non}}$ for any strictly-positive integer $i$. Therefore, for every four iterations of the cyclic seesaw SG algorithm in (\ref{eq:addalabeltomee0},b) we may run five iterations of the SG algorithm in (\ref{eq:fafafafafafufahfah0},b) at the same computational cost. Consequently, $k_1=4i$ and $k_2=5i$ are valid values for $k_1$ and $k_2$ in (\ref{eq:ratio}) for any strictly-positive integer $i$. Figure \ref{eq:GOTITPEEPOPAH} estimates the relative efficiency ratio in (\ref{eq:ratio}) with $k_1=4i$ and $k_2=5i$. The estimates were obtained by using $\bm\uptheta^\ast=[1,\dots,1]^\top$, letting the entries of $\bm{h}_k$ be independent and uniformly distributed in $[-3,3]$, $\upepsilon_k\sim \mathcal{N}(0,0.5^2)$, and setting $a_k=a_k^{(j)}=0.1/(1+k+100)$.  The expectations in (\ref{eq:ratio}) were estimated using the average of 100 i.i.d. replications. For each replication, both algorithms were initialized at $-\bm\uptheta^\ast$.

Figure \ref{eq:GOTITPEEPOPAH} shows that the cyclic implementation was asymptotically less efficient (by approximately 15\%).
 While the cyclic algorithm was more efficient for small values of $i$ (small values of $i$ correspond to earlier iterations of the algorithms), the relative efficiency for small $i$ was found to be highly sensitive to the initialization of the algorithms while the asymptotic relative efficiency
 appeared to be more robust to
  initialization; 
this pattern was observed based on 10 other initializations of the algorithms.

\begin{figure}[p]
\centering
  \begin{subfigure}[b]{1\textwidth}
 \centering
\begin{tikzpicture}

\node[draw,circle,minimum size=1.2cm,inner sep=0pt, thick] (x) at (0,0) {$\hat{\bm{\uptheta}}_k^\text{non}$};
\node[draw,circle,minimum size=1.2cm,inner sep=0pt, thick] (y) at (0,-1.5) {$\bm{h}_{k+1}$};
\node[draw,circle,minimum size=1.2cm,inner sep=0pt, thick] (z) at (6.55,1.05) {$z_k$};
\node[draw,circle,minimum size=1.2cm,inner sep=0pt, thick] (a) at (11.4,1.05) {$a_k^{(1)}$};
\node[draw,rectangle, rounded corners, minimum height=2.3em, minimum width=90pt, ultra thick] (newinput) at (2.6,1.05) {$(\hat{\bm{\uptheta}}_k^\text{non})^\top\bm{h}_{k+1}$};
\node [above of=newinput] {Share this Value};

\node[draw,rectangle, rounded corners, minimum height=2.3em, minimum width=74pt, thick] (xy) at (2.6,-0.75) {$(\hat{\bm{\uptheta}}_k^\text{non})^\top\bm{h}_{k+1}$};

\node[draw,rectangle, rounded corners, minimum height=2.3em, minimum width=110pt, thick] (xyz) at (6.55,-0.75) {$(\hat{\bm{\uptheta}}_k^\text{non})^\top\bm{h}_{k+1}-z_{k+1}$};

\node[draw,rectangle, rounded corners, minimum height=2.3em, minimum width=128pt, thick] (xyza) at (11.4,-0.75) {$a_k^{(1)}(\hat{\bm{\uptheta}}_k^\text{non})^\top\bm{h}_{k+1}-z_{k+1}$};

\node[draw,rectangle, rounded corners, minimum height=2.8em, minimum width=168pt, thick] (xyzah) at (11.4,-2.55) {$a_k^{(1)}\left[(\hat{\bm{\uptheta}}_k^\text{non})^\top\bm{h}_{k+1}-z_{k+1}\right]\bm{h}_{k+1}^{(1)}$};

\node[draw,rectangle, rounded corners, minimum height=2.3em, minimum width=55pt, ultra thick] (final) at (11.4,-4.35) {$[\hat{\bm{\uptheta}}_{k+1}^\text{non}]^{(1)}$};

\draw[->, thick] (x)--(xy) ;
\draw[->, thick] (xy)--(newinput) ;
\draw[->,thick] (y)--(xy);
\draw[->,thick] (z)--(xyz);
\draw[->,thick] (xy)--(xyz);
\draw[->,thick] (a)--(xyza);
\draw[->,thick] (xyz)--(xyza);
\draw[->,thick] (xyza)--(xyzah);
\draw[->,thick] (y)--(0,-2.55)--(xyzah);
\draw[->,thick] (xyzah)--(final);
\draw[->,thick] (x)--(-.8,0)--(-.8,-4.35)--(final);

\end{tikzpicture}
 \end{subfigure}
 \caption*{The CG for computing (\ref{eq:fafafafafafufahfah0}).}


 \vspace{.5in}
 
 
    \begin{subfigure}[b]{1\textwidth}
 \centering
\begin{tikzpicture}

\node[draw,circle,minimum size=1.2cm,inner sep=0pt, thick] (x) at (0,0) {$\hat{\bm{\uptheta}}_k^{\text{non}}$};
\node[draw,circle,minimum size=1.2cm,inner sep=0pt, thick] (y) at (0,-1.5) {$\bm{h}_{k+1}$};
\node[draw,circle,minimum size=1.2cm,inner sep=0pt, thick] (z) at (6.55,1.05) {$z_k$};
\node[draw,circle,minimum size=1.2cm,inner sep=0pt, thick] (a) at (11.4,1.05) {$a_k^{(2)}$};
\node[draw,rectangle, rounded corners, minimum height=2.3em, minimum width=90pt, thick] (newinput) at (2.6,1.05) {$(\hat{\bm{\uptheta}}_k^\text{non})^\top\bm{h}_{k+1}$};
\node [above of=newinput] {Borrow this Value};

\node[draw,rectangle, rounded corners, minimum height=2.3em, minimum width=74pt, thick] (xy) at (2.6,-0.75) {$(\hat{\bm{\uptheta}}_k^{\text{non}})^\top\bm{h}_{k+1}$};

\node[draw,rectangle, rounded corners, minimum height=2.3em, minimum width=110pt, thick] (xyz) at (6.55,-0.75) {$(\hat{\bm{\uptheta}}_k^{\text{non}})^\top\bm{h}_{k+1}-z_{k+1}$};

\node[draw,rectangle, rounded corners, minimum height=2.3em, minimum width=128pt, thick] (xyza) at (11.4,-0.75) {$a_k^{(2)}(\hat{\bm{\uptheta}}_k^{\text{non}})^\top\bm{h}_{k+1}-z_{k+1}$};

\node[draw,rectangle, rounded corners, minimum height=2.8em, minimum width=168pt, thick] (xyzah) at (11.4,-2.55) {$a_k^{(2)}\left[(\hat{\bm{\uptheta}}_k^{\text{non}})^\top\bm{h}_{k+1}-z_{k+1}\right]\bm{h}_{k+1}^{(2)}$};

\node[draw,rectangle, rounded corners, minimum height=2.3em, minimum width=55pt, ultra thick] (final) at (11.4,-4.35) {$[\hat{\bm{\uptheta}}_{k+1}^\text{non}]^{(2)}$};

\draw[->, thick] (x)--(xy) ;
\draw[->, thick] (newinput)--(xy) ;
\draw[->,thick] (y)--(xy);
\draw[->,thick] (z)--(xyz);
\draw[->,thick] (xy)--(xyz);
\draw[->,thick] (a)--(xyza);
\draw[->,thick] (xyz)--(xyza);
\draw[->,thick] (xyza)--(xyzah);
\draw[->,thick] (y)--(0,-2.55)--(xyzah);
\draw[->,thick] (xyzah)--(final);
\draw[->,thick] (x)--(-.8,0)--(-.8,-4.35)--(final);

\end{tikzpicture}
 \end{subfigure}%
 \caption*{The CG for computing (\ref{eq:fafafafafafufahfah}).}

 
  \vspace{.03in}
 
    
\caption[CGs for computing the distributed SG update in (\ref{eq:fafafafafafufahfah0},b)]{CGs for computing the distributed SG update in (\ref{eq:fafafafafafufahfah0},b).  In the distributed setting considered (i.e., in the multi-agent setting) the first graph, respectively the second graph, corresponds to computations performed by Agent 1, respectively Agent 2. Note that we are assuming both agents have access to $\hat{\bm{\uptheta}}_k^\text{non}$, $\bm{h}_{k+1}$, and $z_k$. In general these CGs are not unique.}
\label{fig:figgystuff1}
\end{figure}


\begin{figure}[p]
\centering
  \begin{subfigure}[b]{1\textwidth}
 \centering
\begin{tikzpicture}

\node[draw,circle,minimum size=1.2cm,inner sep=0pt, thick] (x) at (0,0) {$\hat{\bm{\uptheta}}_k^\text{cyc}$};
\node[draw,circle,minimum size=1.2cm,inner sep=0pt, thick] (y) at (0,-1.5) {$\bm{h}_{k+1}$};
\node[draw,circle,minimum size=1.2cm,inner sep=0pt, thick] (z) at (6.55,1.05) {$z_k$};
\node[draw,circle,minimum size=1.2cm,inner sep=0pt, thick] (a) at (11.4,1.05) {$a_k^{(1)}$};
\node[draw,rectangle, rounded corners, minimum height=2.3em, minimum width=90pt, ultra thick] (newinput) at (2.6,1.05) {$[\bm{h}_{k+1}^{(2)}]^\top[\hat{\bm{\uptheta}}_k^\text{cyc}]^{(2)}$};
\node [above of=newinput] {Share this Value};

\node[draw,rectangle, rounded corners, minimum height=2.3em, minimum width=74pt, thick] (xy) at (2.6,-0.75) {$(\hat{\bm{\uptheta}}_k^\text{cyc})^\top\bm{h}_{k+1}$};

\node[draw,rectangle, rounded corners, minimum height=2.3em, minimum width=110pt, thick] (xyz) at (6.55,-0.75) {$(\hat{\bm{\uptheta}}_k^\text{cyc})^\top\bm{h}_{k+1}-z_{k+1}$};

\node[draw,rectangle, rounded corners, minimum height=2.3em, minimum width=128pt, thick] (xyza) at (11.4,-0.75) {$a_k^{(1)}(\hat{\bm{\uptheta}}_k^\text{cyc})^\top\bm{h}_{k+1}-z_{k+1}$};

\node[draw,rectangle, rounded corners, minimum height=2.8em, minimum width=168pt, thick] (xyzah) at (11.4,-2.55) {$a_k^{(1)}\left[(\hat{\bm{\uptheta}}_k^\text{cyc})^\top\bm{h}_{k+1}-z_{k+1}\right]\bm{h}_{k+1}^{(1)}$};

\node[draw,rectangle, rounded corners, minimum height=2.3em, minimum width=35pt, ultra thick] (final) at (11.4,-4.35) {$\hat{\bm{\uptheta}}_{k}^{(I)}$};

\draw[->, thick] (x)--(xy) ;
\draw[->, thick] (xy)--(newinput) ;
\draw[->,thick] (y)--(xy);
\draw[->,thick] (z)--(xyz);
\draw[->,thick] (xy)--(xyz);
\draw[->,thick] (a)--(xyza);
\draw[->,thick] (xyz)--(xyza);
\draw[->,thick] (xyza)--(xyzah);
\draw[->,thick] (y)--(0,-2.55)--(xyzah);
\draw[->,thick] (xyzah)--(final);
\draw[->,thick] (x)--(-.8,0)--(-.8,-4.35)--(final);

\end{tikzpicture}
 \end{subfigure}
 \caption*{The CG for computing (\ref{eq:addalabeltomee0}).}


 \vspace{.5in}
 
 
    \begin{subfigure}[b]{1\textwidth}
 \centering
\begin{tikzpicture}

\node[draw,circle,minimum size=1.2cm,inner sep=0pt, thick] (x) at (0,0) {$\hat{\bm{\uptheta}}_k^{(I)}$};
\node[draw,circle,minimum size=1.2cm,inner sep=0pt, thick] (y) at (0,-1.5) {$\bm{h}_{k+1}$};
\node[draw,circle,minimum size=1.2cm,inner sep=0pt, thick] (z) at (6.55,1.05) {$z_k$};
\node[draw,circle,minimum size=1.2cm,inner sep=0pt, thick] (a) at (11.4,1.05) {$a_k^{(2)}$};
\node[draw,rectangle, rounded corners, minimum height=2.3em, minimum width=90pt, thick] (newinput) at (2.6,1.05) {$[\bm{h}_{k+1}^{(2)}]^\top[\hat{\bm{\uptheta}}_k^\text{cyc}]^{(2)}$};
\node [above of=newinput] {Borrow this Value};

\node[draw,rectangle, rounded corners, minimum height=2.3em, minimum width=74pt, thick] (xy) at (2.6,-0.75) {$(\hat{\bm{\uptheta}}_k^{(I)})^\top\bm{h}_{k+1}$};

\node[draw,rectangle, rounded corners, minimum height=2.3em, minimum width=110pt, thick] (xyz) at (6.55,-0.75) {$(\hat{\bm{\uptheta}}_k^{(I)})^\top\bm{h}_{k+1}-z_{k+1}$};

\node[draw,rectangle, rounded corners, minimum height=2.3em, minimum width=128pt, thick] (xyza) at (11.4,-0.75) {$a_k^{(2)}(\hat{\bm{\uptheta}}_k^{(I)})^\top\bm{h}_{k+1}-z_{k+1}$};

\node[draw,rectangle, rounded corners, minimum height=2.8em, minimum width=168pt, thick] (xyzah) at (11.4,-2.55) {$a_k^{(2)}\left[(\hat{\bm{\uptheta}}_k^{(I)})^\top\bm{h}_{k+1}-z_{k+1}\right]\bm{h}_{k+1}^{(2)}$};

\node[draw,rectangle, rounded corners, minimum height=2.3em, minimum width=35pt, ultra thick] (final) at (11.4,-4.35) {$\hat{\bm{\uptheta}}_{k+1}^{\text{cyc}}$};

\draw[->, thick] (x)--(xy) ;
\draw[->, thick] (newinput)--(xy) ;
\draw[->,thick] (y)--(xy);
\draw[->,thick] (z)--(xyz);
\draw[->,thick] (xy)--(xyz);
\draw[->,thick] (a)--(xyza);
\draw[->,thick] (xyz)--(xyza);
\draw[->,thick] (xyza)--(xyzah);
\draw[->,thick] (y)--(0,-2.55)--(xyzah);
\draw[->,thick] (xyzah)--(final);
\draw[->,thick] (x)--(-.8,0)--(-.8,-4.35)--(final);

\end{tikzpicture}
 \end{subfigure}%
 \caption*{The CG for computing (\ref{eq:addalabeltomee}).}

 
  \vspace{.03in}
 
    
\caption[CGs for computing the distributed cyclic SG update in (\ref{eq:addalabeltomee0},b)]{CGs for computing the distributed cyclic seesaw SG update in (\ref{eq:addalabeltomee0},b). In the multi-agent setting the first graph, respectively the second graph, corresponds to computations performed by Agent 1, respectively Agent 2. Note that we are assuming both agents have access to $\bm{h}_{k+1}$ and $z_k$. We are also assuming Agent 1 has access to $\hat{\bm{\uptheta}}_k^\text{cyc}$ and that Agent 2 has access to $\hat{\bm{\uptheta}}_k^{(I)}$. In general these CGs are not unique. }
\label{fig:figgystuff2}
\end{figure}

\begin{figure}[!t]
\centering
\begin{tikzpicture}
  \definecolor{mynewcoloring}{RGB}{102, 128, 153}
\node[inner sep=0pt] (probe) at (0,0)
    {\includegraphics[scale=0.6]{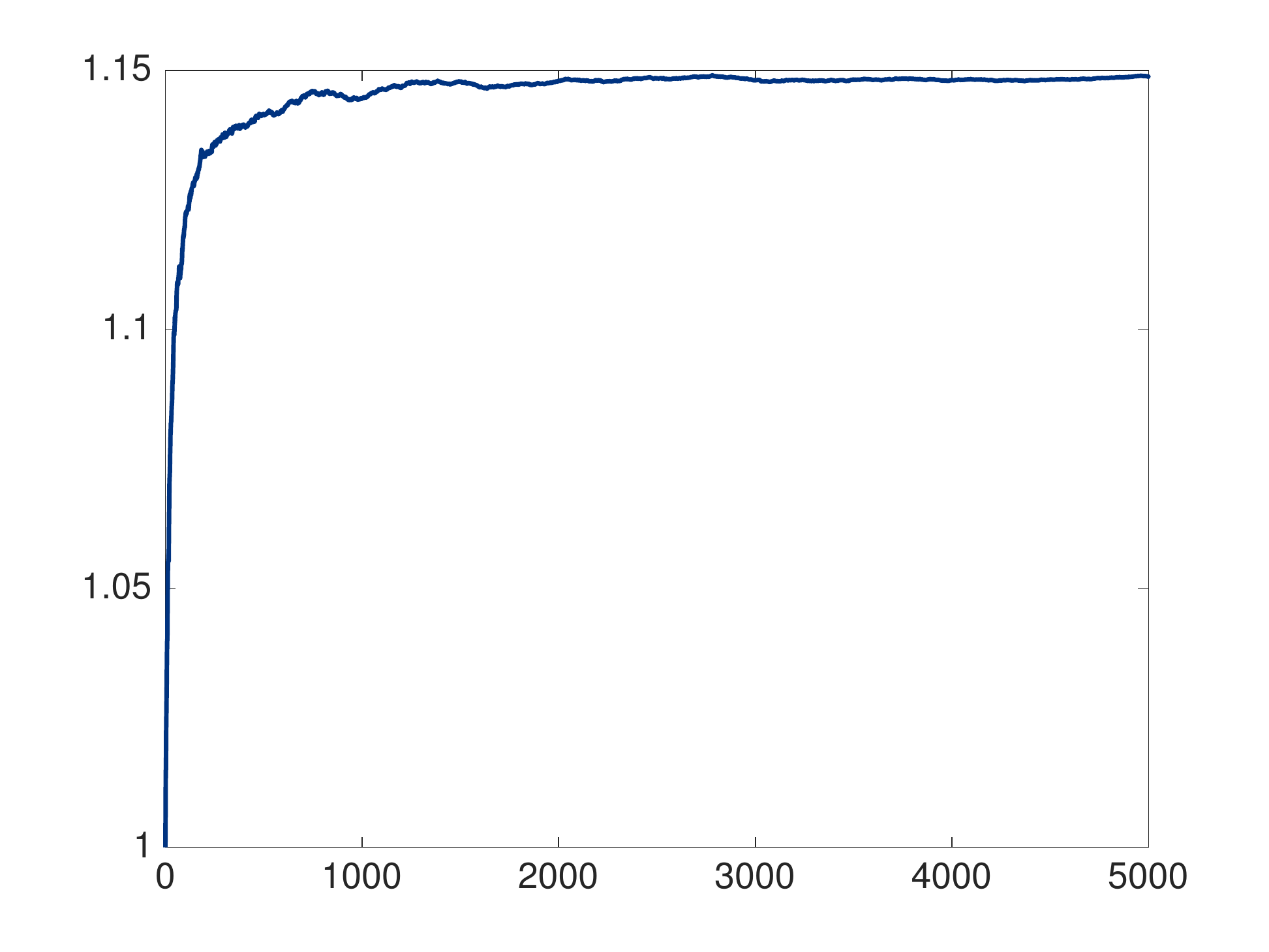}};  
        \node[left= -0.6cm of probe]{$\displaystyle{\frac{\overline{\|\hat{\bm\uptheta}_{4i}^{\text{cyc}}-\bm\uptheta^\ast\|^2}}{\overline{\|\hat{\bm\uptheta}_{5i}^{\text{non}}-\bm\uptheta^\ast\|^2}}}$};
      \node at (0.3,-4.3) {$i$ (a strictly-positive integer)};
\end{tikzpicture}
\caption[Relative efficiency between SG and cyclic seesaw SG when cost is a measure of the number of arithmetic computations]{An estimate of the relative efficiency in (\ref{eq:ratio}) for comparing the cyclic seesaw SG algorithm to the regular (i.e., non-cyclic) SG algorithm when cost is a measure of the number of arithmetic computations.  }
\label{eq:GOTITPEEPOPAH}
\end{figure}

\subsubsection{Cost as the Number of Subvector Updates}

The second type of cost we consider is the number of subvector updates performed. Under this definition of cost the SG algorithm and its cyclic counterpart have exactly the same cost since two subvector updates are performed every iteration. In other words, when cost is defined as the per-iteration number of subvector updates, it follows that $c^\text{non}=c^\text{cyc}$. Therefore, $k_1=i=k_2$ are valid values for $k_1$ and $k_2$ in (\ref{eq:ratio}) for any strictly-positive integer $i$. Figure \ref{fig:2013eura100k} estimates (\ref{eq:ratio}) for multiple values of $i$ using the same settings from Section \ref{sec:SGarithmetic}. The expectations in (\ref{eq:ratio}) were estimated using the average of 100 i.i.d. replications. For each replication, both the cyclic and non-cyclic algorithms were initialized at $-\bm\uptheta^\ast$.
 Asymptotically, 
the cyclic algorithm was 5\% less efficient. Based on 10 other initializations of the algorithms, it was observed that the relative efficiency was more sensitive to initialization for early iterations than for later iterations.
 \begin{figure}[!t]
\centering
\begin{tikzpicture}
  \definecolor{mynewcoloring}{RGB}{102, 128, 153}
\node[inner sep=0pt] (probe) at (0,0)
    {\includegraphics[scale=0.6]{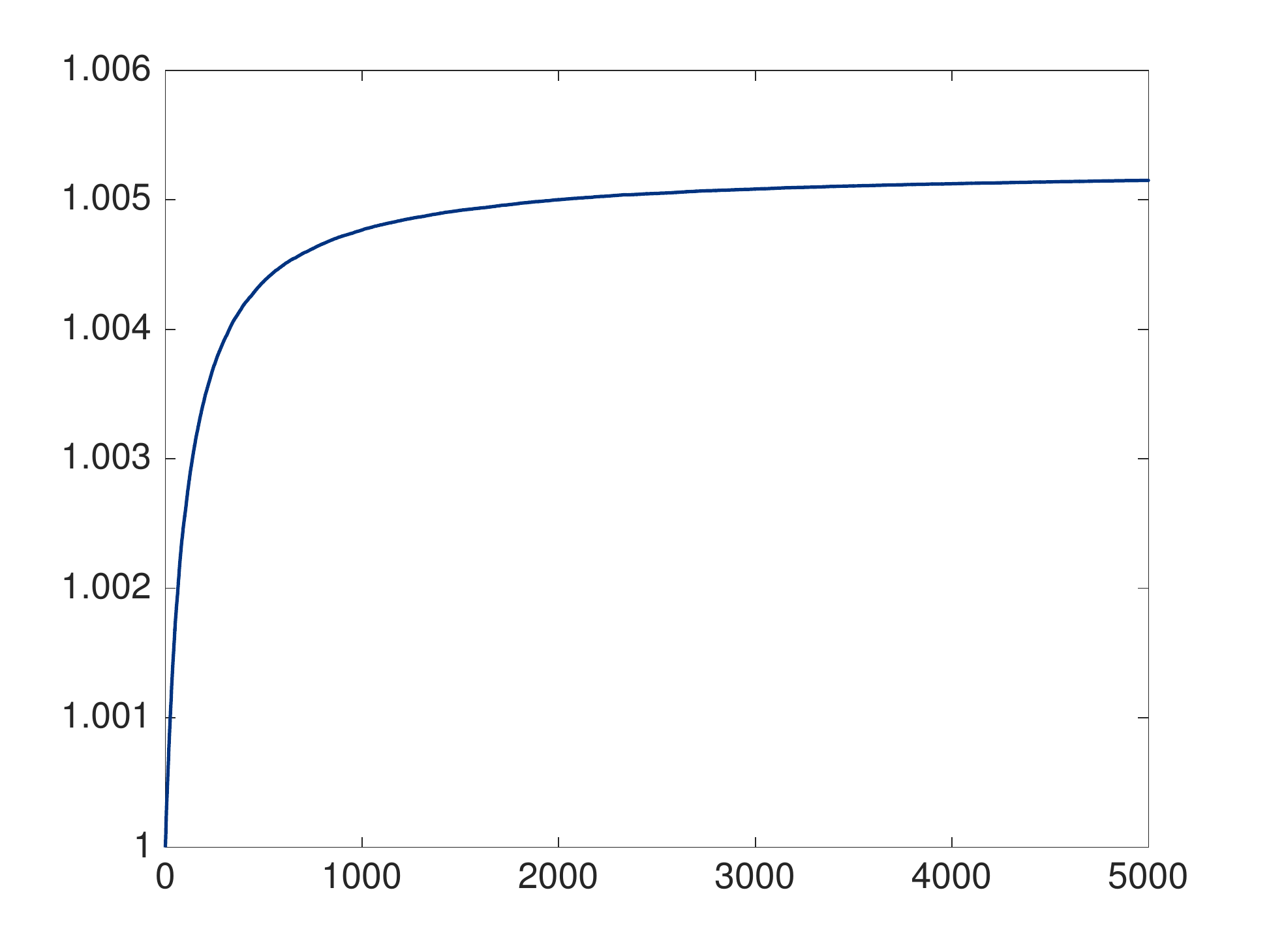}};  
        \node[left= -0.3cm of probe]{$\displaystyle{\frac{\overline{\|\hat{\bm\uptheta}_{i}^{\text{cyc}}-\bm\uptheta^\ast\|^2}}{\overline{\|\hat{\bm\uptheta}_{i}^{\text{non}}-\bm\uptheta^\ast\|^2}}}$};
      \node at (0.3,-4.3) {$i$ (a strictly-positive integer)};
\end{tikzpicture}
\caption[Relative efficiency between SG and cyclic seesaw SG when cost is a measure of the number of subvector updates]{An estimate of the relative efficiency in (\ref{eq:ratio}) for comparing the cyclic seesaw SG algorithm to the regular (i.e., non-cyclic) SG algorithm when cost is a measure of the number of subvector updates. 
}
\label{fig:2013eura100k}
\end{figure}

\subsubsection[Cost as Time allowing for Agent Unavailability]{Cost as Time allowing for Agent Unavailability} 
\label{secsec:originalunavailability}

The efficiency analysis in this section once again relies on drawing parallels between the distributed implementations of (\ref{eq:fafafafafafufahfah0},b) and (\ref{eq:addalabeltomee0},b) and the multi-agent setting. In the multi-agent setting, the first and second subvectors of $\bm\uptheta$ correspond, respectively, to parameters that Agent No. 1 and Agent No. 2 can update. Implementing the SG algorithm in (\ref{eq:fafafafafafufahfah0},b) in a distributed manner would require both agents to be available at the same time. However, in practice it is not uncommon for agents to be tasked with secondary objectives that may make them temporarily unavailable to update their parameter vector.
A slight modification to the cyclic seesaw SG algorithm in (\ref{eq:addalabeltomee0},b) would be to allow a single agent, while available, to continuously update its parameters until the moment when the other agent becomes available. The resulting algorithm would be a special case of Algorithm \ref{findme}. This section constructs a simple model of agent availability in order to estimate  (\ref{eq:ratio}) when cost is a measure of time.


To model agent availability we begin by defining a unit of time as the time it takes to perform a subvector update, this time is assumed to be equal for all subvectors independently of whether we are implementing the SG algorithm or its modified cyclic version and independently of the magnitude of the update (note that for the regular SG algorithm the time needed to perform a single iteration is one time unit since the two subvector updates are performed simultaneously).   We will assume that any agent can become unavailable with probability $q$ after completing its latest subvector update (for simplicity we assume an agent cannot become unavailable while in the middle of performing an update). Furthermore, at the next time unit an agent that was previously unavailable will remain unavailable with probability $q$. Because implementing the SG algorithm in (\ref{eq:fafafafafafufahfah0},b) requires both agents to be available simultaneously, no updates can be performed using the non-cyclic algorithm if at least one of the agents is unavailable. In contrast, the proposed modification to the cyclic algorithm (where an available agent continuously updates its parameters) could still be implemented provided at least one of the agents is available. 

Given our definition of a time unit and our model of agent unavailability, the relative efficiency ratio in (\ref{eq:ratio}) is estimated as the ratio:
\begin{align}
\label{eq:datamissinon}
\frac{E\|\hat{\bm\uptheta}^{\text{cyc}}(t)-\bm\uptheta^\ast\|^2}{E\|\hat{\bm\uptheta}^{\text{non}}(t)-\bm\uptheta^\ast\|^2},
\end{align}
where $t$ denotes time and where $\hat{\bm{\uptheta}}^{\text{cyc}}(t)$ and $\hat{\bm{\uptheta}}^{\text{non}}(t)$ denote algorithm iterate values at time $t$ (we deliberately avoid writing $t$ as a subscript of $\hat{\bm{\uptheta}}$ due to the fact that subscripts in (\ref{eq:fafafafafafufahfah0},b) and (\ref{eq:addalabeltomee0},b) are reserved for the iteration number, which is not the same as $t$). In order to select a value of $q$ for our numerical examples we note the following. At any given time the probability of updating the parameter vector using the SG algorithm is $P_1\equiv (1-q)^2$ whereas the probability of updating the parameter vector using the modified cyclic seesaw algorithm is $P_2\equiv 2(1-q)-(1-q)^2$. For $q\in [0,1]$, $P_2\geq P_1$ and the maximum difference between $P_1$ and $P_2$ occurs at $q=1/2$.  Figure \ref{fig:failureagentsgraphy} estimates (\ref{eq:datamissinon}) using $q=1/2$ 
 with $a_k=a_k^{(j)}=0.1/(1+100+k)^{0.501}$. The expectations in (\ref{eq:datamissinon}) were estimated using the average of 100 i.i.d. replications. For each replication, both algorithms were initialized at $-\bm\uptheta^\ast$. We use the same distributions for  $\bm{h}_k$ and $\upvarepsilon_k$ from Section \ref{sec:SGarithmetic}. It was assumed that a new input-output pair $(\bm{h}_k,z_k)$ became available with the passing of every time unit. Agents were assumed to use the latest input-output pair to perform their updates.  The results of Figure \ref{fig:failureagentsgraphy} indicate that the cyclic variant was approximately 10\% more efficient (asymptotically), likely due to the fact that using a cyclic implementation implies the parameter vector is updated more frequently than when using a non-cyclic implementation.
%
Based on 10 other initializations of the algorithms, it was observed that the relative efficiency was more sensitive to initialization for small values of $t$ than for larger values.
\begin{figure}[!t]
\centering
\begin{tikzpicture}
  \definecolor{mynewcoloring}{RGB}{102, 128, 153}
\node[inner sep=0pt] (probe) at (0,0)
    {\includegraphics[scale=0.6]{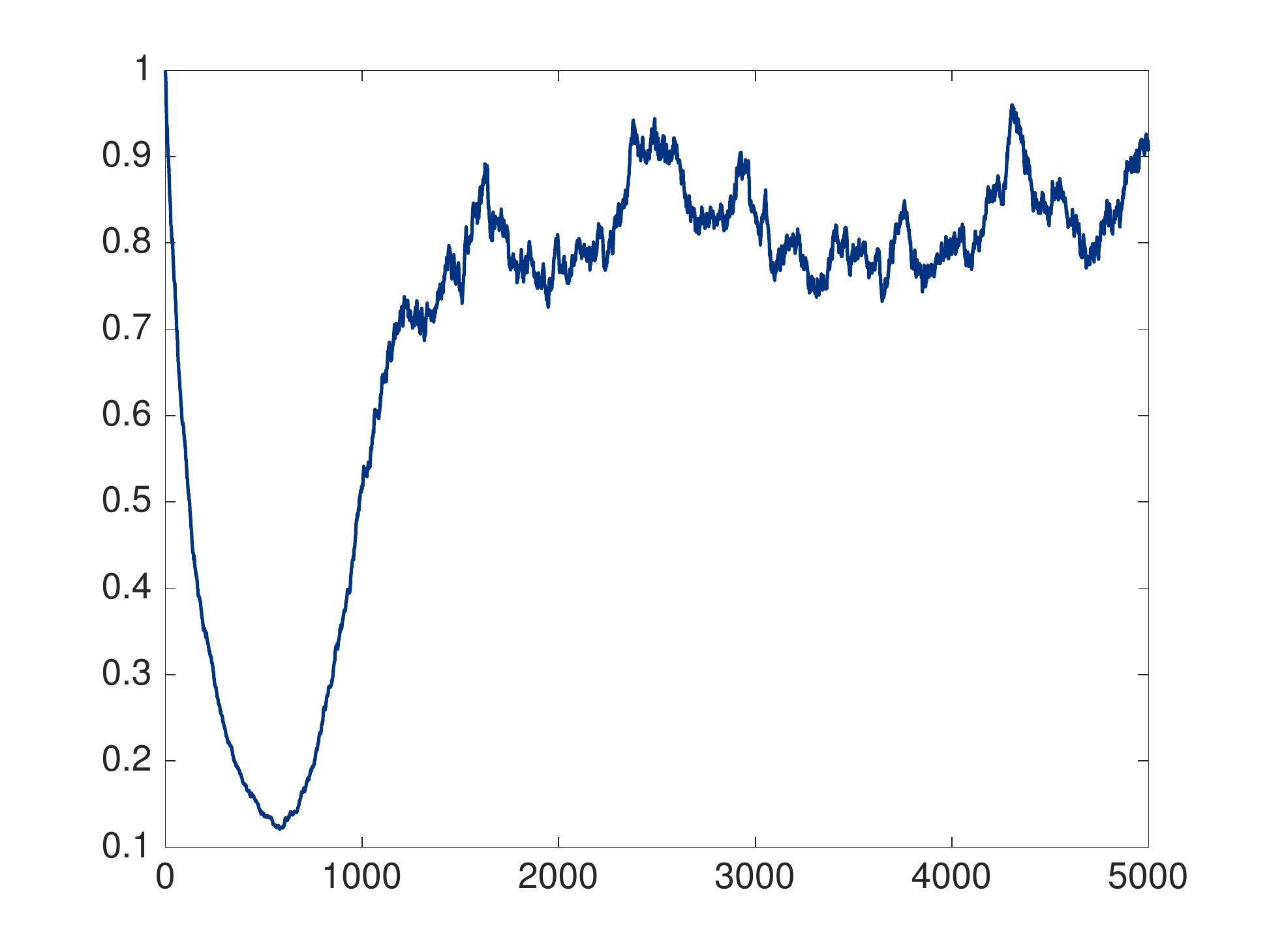}};  
        \node[left= -0.6cm of probe]{$\displaystyle{\frac{\overline{\|\hat{\bm\uptheta}^{\text{cyc}}(t)-\bm\uptheta^\ast\|^2}}{\overline{\|\hat{\bm\uptheta}^{\text{non}}(t)-\bm\uptheta^\ast\|^2}}}$};
      \node at (0.3,-4.3) {$t$};
\end{tikzpicture}
\caption[An estimate of the relative efficiency in (\ref{eq:datamissinon}) for comparing the SG algorithm to a generalized cyclic seesaw SG algorithm where agents may be unavailable to perform updates]{An estimate of the relative efficiency in (\ref{eq:datamissinon}) for comparing the SG to a generalized cyclic seesaw SG algorithm where agents may be unavailable to perform updates with probability $q=1/2$. }
\label{fig:failureagentsgraphy}
\end{figure}

\subsection{Efficiency: SPSA versus Cyclic Seesaw SPSA}
\label{eq:Indyiedie12}

This section compares the SPSA algorithm to the cyclic seesaw SPSA algorithm by estimating the relative efficiency in (\ref{eq:ratio}). The SPSA algorithm treated is the SPSA algorithm from Section \ref{sec:clubsodeashrimp} where the objective is to minimize the skewed quartic loss function with $p=10$.
  We assume $Q(\bm\uptheta,\bm{V})=L(\bm\uptheta)+\upvarepsilon$ where $\upvarepsilon\sim\mathcal{N}(0,0.1^2)$. For the cyclic seesaw implementation we let $p'=5$ 
  and use noisy gradient estimates of the form in (\ref{eq:howeardspos1}) and (\ref{eq:howeardspos2}), SPSA-based noisy gradient estimates where only the subvector to update is perturbed.

\subsubsection{Cost as the Number of Noisy Loss Measurements}
\label{secsec:whennoticreep}

The first definition of cost we consider is the number of noisy loss function measurements required per-iteration. Under this definition $c^\text{non}=2$ for the the regular SPSA algorithm whereas $c^\text{cyc}=4$ for the cyclic seesaw SPSA algorithm. Therefore, in order to conduct a fair comparison between  SPSA and cyclic SPSA we would need to run twice as many iterations for the regular SPSA algorithm.  Therefore, $k_1=i$ and $k_2=2i$ are valid values for $k_1$ and $k_2$ in (\ref{eq:ratio}) for any strictly-positive integer $i$. As such, Figure \ref{fig:sassyconcREG} estimates the relative efficiency in (\ref{eq:ratio}) with $k_1=i$ and $k_2=2i$ for different values of $i$. The expectations in (\ref{eq:ratio}) were estimated using the average of 500 i.i.d. replications. For each replication, both algorithms were initialized at $-[1,\dots,1]^\top$. As suggested by Table \ref{table:omgmyfirsttable}, we use the gain sequences $a_k=a_k^{(j)}=1/(1+k+100)^{0.602}$ and set $c_k=c_k^{(j)}=0.1/(1+k)^{0.101}$. We let the non-zero entries of the $\bm\Delta$ be independent and take the value $\pm1$ with equal probability. We can see that the relative efficiency favors the cyclic algorithm for early iterations (although this behavior was, once again, sensitive to the initialization of the algorithms) while the non-cyclic algorithm becomes more efficient (by about 8\%) in the long run. Based on 10 different initializations of the other, it was observed that the relative efficiency was more sensitive to initialization for small values of $t$ than for larger values. 

\subsubsection{Cost as the Number of Subvector Updates}

\begin{figure}[!t]
\centering
\begin{tikzpicture}
  \definecolor{mynewcoloring}{RGB}{102, 128, 153}
\node[inner sep=0pt] (probe) at (0,0)
    {\includegraphics[scale=0.6]{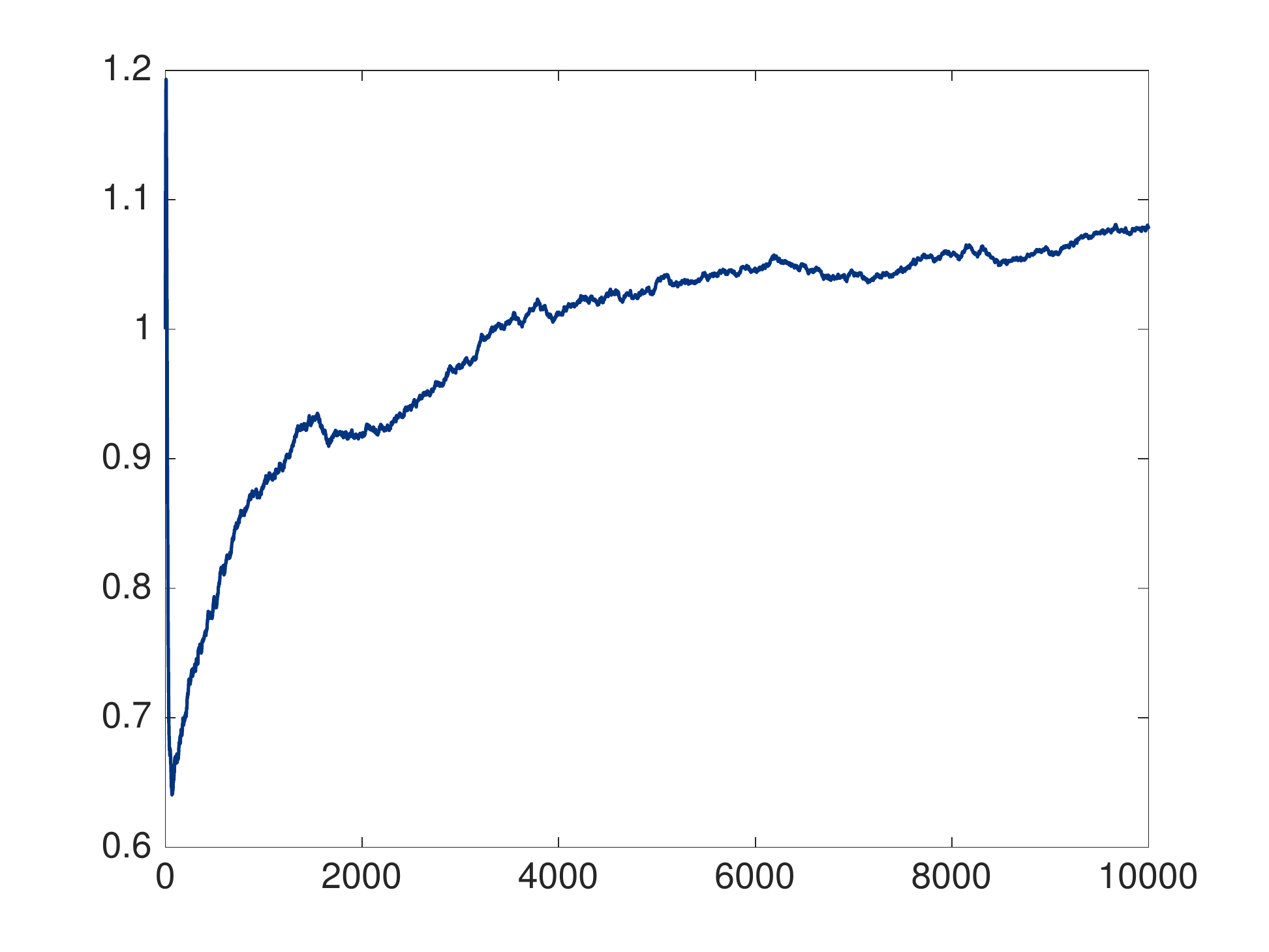}};  
        \node[left= -0.6cm of probe]{$\displaystyle{\frac{\overline{\|\hat{\bm\uptheta}^{\text{cyc}}_i-\bm\uptheta^\ast\|^2}}{\overline{\|\hat{\bm\uptheta}^{\text{non}}_{2i}-\bm\uptheta^\ast\|^2}}}$};
      \node at (0.3,-4.3) {$i$ (a strictly-positive integer)};
\end{tikzpicture}
\caption[Relative efficiency between SPSA and cyclic seesaw SPSA when cost is a measure of the number of noisy loss measurements]{An estimate of the relative efficiency in (\ref{eq:ratio}) for comparing the cyclic seesaw SPSA algorithm to the regular (i.e., non-cyclic) SPSA algorithm when cost is a measure of the number of noisy loss function evaluations.  }
\label{fig:sassyconcREG}
\end{figure} 

\begin{figure}[!t]
\centering
\begin{tikzpicture}
  \definecolor{mynewcoloring}{RGB}{102, 128, 153}
\node[inner sep=0pt] (probe) at (0,0)
    {\includegraphics[scale=0.6]{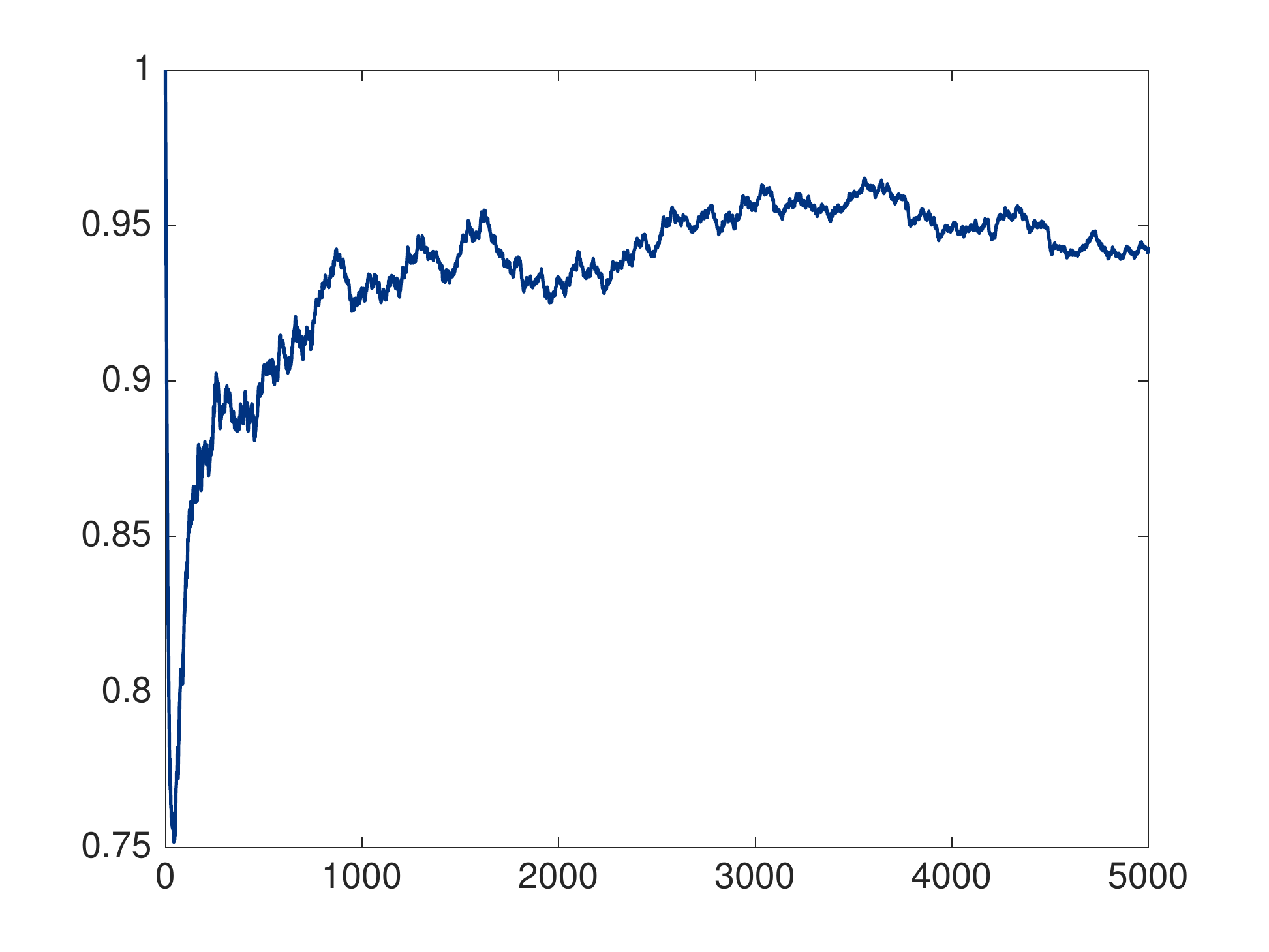}};  
        \node[left= -0.6cm of probe]{$\displaystyle{\frac{\overline{\|\hat{\bm\uptheta}^{\text{cyc}}_i-\bm\uptheta^\ast\|^2}}{\overline{\|\hat{\bm\uptheta}^{\text{non}}_{i}-\bm\uptheta^\ast\|^2}}}$};
      \node at (0.3,-4.3) {$i$ (a strictly-positive integer)};
\end{tikzpicture}
\caption[Relative efficiency between SPSA and cyclic seesaw SPSA when cost is a measure of the number of subvector updates]{An estimate of the relative efficiency in (\ref{eq:ratio}) for comparing the cyclic seesaw SPSA algorithm to the regular (i.e., non-cyclic) SPSA algorithm when cost is a measure of the number of subvector updates. }
\label{fig:sassyconc}
\end{figure} 

This section estimates the ratio in (\ref{eq:ratio}) by viewing cost as the number of subvector updates per-iteration, which implies $c^\text{non}=c^\text{cyc}=2$. Therefore, $k_1=i=k_2$ are valid values for $k_1$ and $k_2$ in (\ref{eq:ratio}) for any strictly-positive integer $i$. Figure \ref{fig:sassyconc} presents the results with $k_1=i=k_2$ using the settings from Section \ref{secsec:whennoticreep}. The expectations in (\ref{eq:ratio}) were estimated using the average of 500 i.i.d. replications. For each replication, both the cyclic and non-cyclic algorithms were initialized at $-[1,\dots,1]^\top$.  It can be seen that cyclic algorithm slightly outperforms its non-cyclic counterpart (by about 6\%) when the number of iterations is large. Based on 10 other initializations of the algorithms, the relative efficiency for small values of $i$ appeared to be more sensitive to initialization  than for large values of $i$.

\subsubsection[Cost as Time allowing for Agent Unavailability]{Cost as Time allowing for Agent Unavailability}

Similarly to Section \ref{secsec:originalunavailability}, this section estimates relative efficiency in a multi-agent setting where agents can be unavailable with probability $q=1/2$. The difference between the experiment in this section and the experiment from Section \ref{secsec:originalunavailability} is that SPSA-based gradient estimates, instead of SG-based estimates, are used to perform the subvector updates. For this reason we refer the reader to Section \ref{secsec:originalunavailability} for a more detailed description of the model for agent unavailability. Figure \ref{fig:eleveatedday} presents the estimate of (\ref{eq:datamissinon}) as a function of $t$, denoting time, using the settings from Section \ref{secsec:whennoticreep}. The expectations in (\ref{eq:datamissinon}) were estimated using the average of 500 i.i.d. replications. For each replication, both algorithms were initialized at the vector $-[1,\dots,1]^\top$. 
\begin{figure}[!t]
\centering
\begin{tikzpicture}
  \definecolor{mynewcoloring}{RGB}{102, 128, 153}
\node[inner sep=0pt] (probe) at (0,0)
    {\includegraphics[scale=0.6]{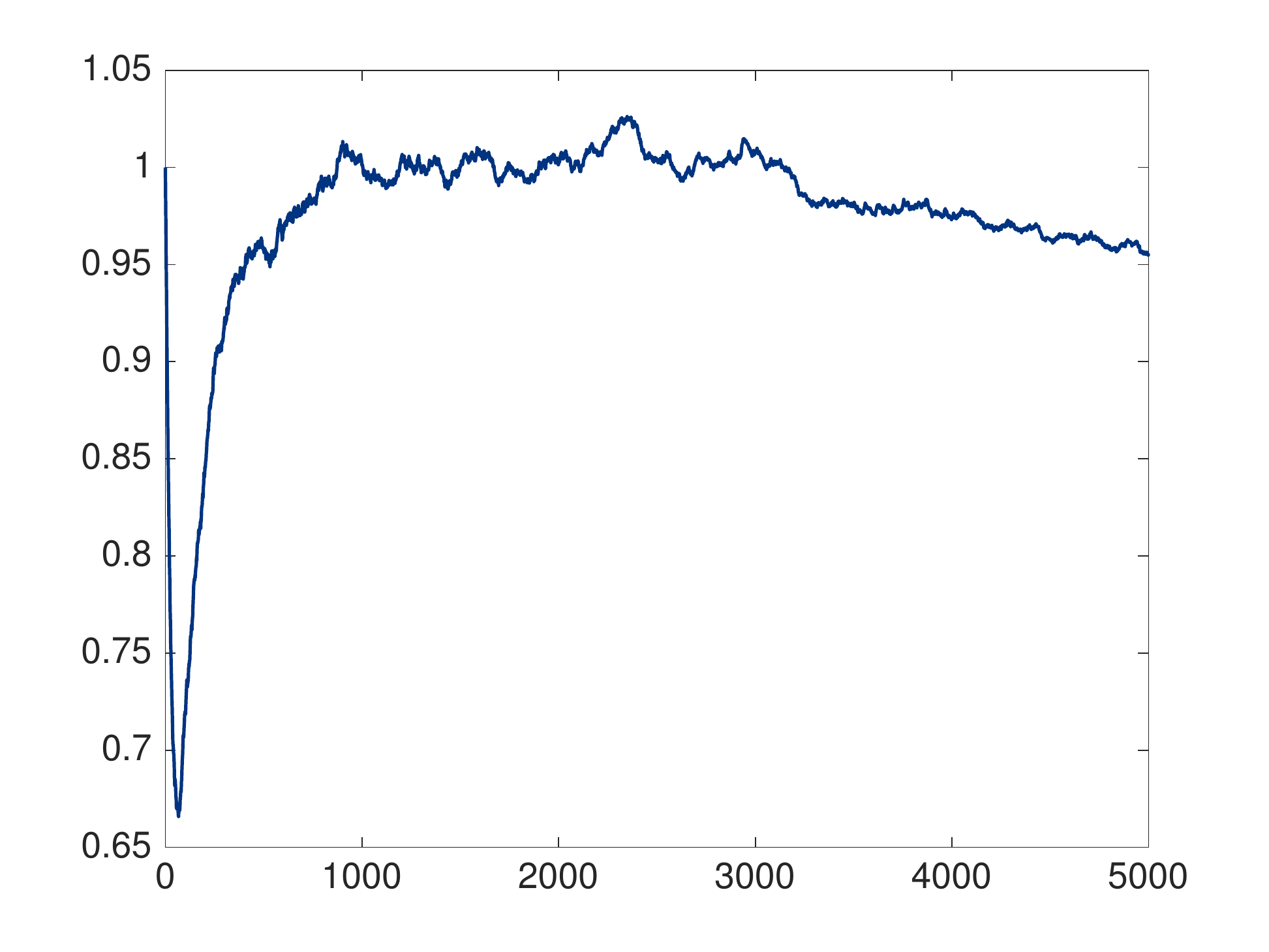}};  
        \node[left= -0.6cm of probe]{$\displaystyle{\frac{\overline{\|\hat{\bm\uptheta}^{\text{cyc}}(t)-\bm\uptheta^\ast\|^2}}{\overline{\|\hat{\bm\uptheta}^{\text{non}}(t)-\bm\uptheta^\ast\|^2}}}$};
      \node at (0.3,-4.3) {$t$};
\end{tikzpicture}
\caption[An estimate of the relative efficiency in (\ref{eq:datamissinon}) for comparing the SPSA algorithm to a generalized cyclic seesaw SPSA algorithm where agents may be unavailable to perform updates]{An estimate of the relative efficiency in (\ref{eq:datamissinon}) for comparing the SPSA to a generalized cyclic seesaw SPSA algorithm where agents may be unavailable to perform updates with probability $q=1/2$. }
\label{fig:eleveatedday}
\end{figure} 
We see that asymptotically
 the cyclic algorithm appeared to be only slightly more efficient than its non-cyclic counterpart (based on 10 other initializations, this appeared to be independent of the initialization of the cyclic and non-cyclic algorithms).

\section{Concluding Remarks}

This chapter contains numerical results illustrating the theory on convergence from Chapter \ref{sec:cyclicseesaw} and the theory on asymptotic normality from Chapter \ref{sec:ROC}. Additionally, numerical experiments were conducted to investigate the relative efficiency (as defined in (\ref{eq:ratio})) between cyclic algorithms and their non-cyclic counterparts. For the numerical experiments that were conducted, the cyclic and non-cyclic algorithms had comparable asymptotic efficiency. However, it is important to keep in mind that there are many factors that affect asymptotic relative efficiency (e.g., the shape of the loss function to minimize, the number of subvectors, the definition of cost) and the results of this Chapter do not imply that the cyclic and non-cyclic algorithms are equally efficient in general. For a theoretical analysis of asymptotic relative efficiency we refer the reader to Sections \ref{sec:aseficwiener} and \ref{sec:apsecialcaserelefff}.

%% file: chapter5.tex

\chapter[A Zero-Communication Multi-Agent Problem]{A Zero-Communication Multi-Agent Problem}
\label{chap:multiagent}
\chaptermark{A Zero-Communication Multi-Agent Problem}

This chapter investigates the numerical performance of cyclic SA applied to a multi-agent optimization problem for tracking and surveillance with no communication between agents. Here, a group of agents equipped with finite-range sensors and no communication abilities can adjust their positions in order to collectively accomplish the following objectives:
\begin{DESCRIPTION}
\label{eq:sandywpusstod}
\item[Objective 1:] Track any detected targets within a certain area.
\item[Objective 2:] Maximize the probability of detecting other targets in the area, that is maximize the collective coverage of the area of interest.
\end{DESCRIPTION}
To address the two points above, it is assumed that each agent can obtain noisy information regarding the current position of nearby agents and targets and, using this noisy information, each agent decides what its next action should be. The outline of this chapter is as follows. Section \ref{sec:probdescription} describes the details behind the multi-agent optimization problem. Section \ref{sec:proposedapproach} then proposes a cyclic SA approach for solving the optimization problem from Section \ref{sec:probdescription}. Lastly, Section \ref{sec:numericsamultiagent} contains numerical results and Section \ref{sec:finalcommentsagent} contains concluding remarks.

\section{Problem Description}
\label{sec:probdescription}

Throughout this chapter we let $\bm{x}_j(t)=[x_j^N(t),x_j^E(t), \dot{x}_j^N(t), \dot{x}_j^E(t)]^\top\in \mathbb{R}^4$ denote the vector of positions (in $\mathbb{R}^2$) and velocity (in $\mathbb{R}^2$) of agent $j$ at time $t$. The four-dimensional $\bm{x}_j(t)$ is said to be the state vector of agent $j$ at time $t$. Here, the superscripts ``E'' and ``N'' are used to denote agent $j$'s position in the east and north directions, respectively, relative to some artificial center (corresponding to the point $(0,0)$ on the Cartesian plane) that is universal to all agents and targets (this notation is used as a proxy for the standard $(x,y)$ coordinate notation due to the fact that the letter ``$\bm{y}$'' will be used to denote the state vector of a target). 
The position and velocity of the $i$th target at time $t$ will be represented by the state vector $\bm{y}_i(t)=[y_i^N(t),y_i^E(t), \dot{y}_i^N(t), \dot{y}_i^E(t)]^\top\in \mathbb{R}^4$ 
(the range of $i$ is unspecified since the number of targets is unknown and may change with time). Each agent is assumed to be equipped with a sensor capable of estimating the positions of nearby agents and targets. The sensor model is described next.

\subsection{The Sensor Model}

When the target location is unknown, each agent is allowed to take a readings from its surroundings and collect estimated positions of any detected agents and targets. To model a sensor's probability of detection we use a variant of the sensor model from Kim et al. (2005)\nocite{kimsensormodel2005}. Here, it is assumed that there exists a constant $r>0$, referred to as the range of the sensor, such that an agent cannot detect any target/agent located at a distance greater than $r$ from the sensor. It is also assumed that there exists a constant $r'>0$ with $r'<r$ such that any target/agent located within a distance of $r'$ from the sensor's location will be detected (w.p.1). A target/agent located within the range of the sensor but at a distance larger than $r'$ will be detected with a probability that decreases linearly as the distance to the target approaches $r$. Specifically, if $d$ is the distance from a target/agent to a sensor, then the probability of detecting said target/agent is given by: 
\begin{align}
{\text{Probability of Detection}} = \begin{cases} 1 &\mbox{if } d \leq  r', \\
1-\frac{d-r'}{r-r'} &\mbox{if } r'<d \leq  r,\\
 0 & \mbox{if } d>r. \end{cases}
 \label{eq:sensorequationmodel}
\end{align}
\afterpage{%
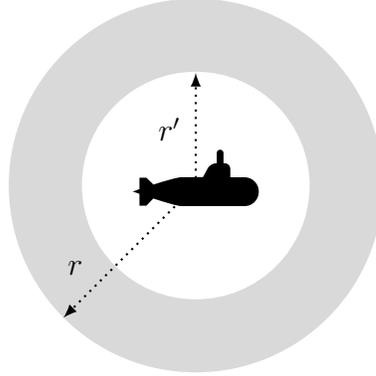
\begin{figure}[!t]
\centering
\begin{tikzpicture}
  \definecolor{mynewcoloring}{RGB}{242, 242, 242}
   \definecolor{mynewcoloring2}{RGB}{217, 217, 217}
\draw[thick, draw=white, fill=mynewcoloring2] (0,0) circle (2.5cm);
\draw[thick, draw=white, fill=white] (0,0) circle (1.5cm);

\draw[->, dotted, thick] (0,0)--(0,1.5);
\draw[->, dotted, thick] (0,0)--(-1.76,-1.76);
\node at (-1.6, -1.1) {$r$};
\node at (-0.35,0.75) {$r'$};

\node[inner sep=0pt] (probe) at (0,0.1) {\includegraphics[scale=0.2]{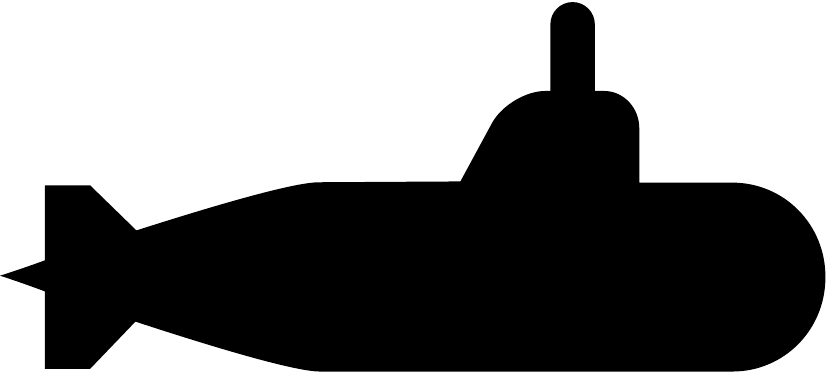}};

\end{tikzpicture}
\caption[The sensor model]{The probability of detection is 1 for agents/targets located within a radius of $r'$. The probability of detection in the gray annulus decreases as the distance to the agent's sensor increases.  An agent cannot detect any agent/target located at a radius greater than $r$, the sensor's range.}
\label{fig:sensormodel}
\end{figure} 
}
Figure \ref{fig:sensormodel} illustrates (\ref{eq:sensorequationmodel}).{\footnote{The submarine graphic in Figure \ref{fig:sensormodel} was created by Freepik (www.freepik.com) and is licensed under the Attribution 3.0 Unported (CC BY 3.0) license.}} The difference between the sensor model in sensor model in (\ref{eq:sensorequationmodel}) and the sensor model in Kim et al. (2005)\nocite{kimsensormodel2005} is that (\ref{eq:sensorequationmodel}) assumes the probability of detection decreases linearly within the gay annulus in Figure \ref{fig:sensormodel}, whereas Kim et al. (2005)\nocite{kimsensormodel2005} allow this decrease to be nonlinear.

When an agent's sensor detects a target or another agent, we assume the sensor can also collect noisy information regarding the position of said target/agent. In the case where target $i$ is detected by agent $j$ at time $t$, then agent $j$ 
%
%
can also obtain noisy measurements of
\begin{align}
\upvarphi(\bm{x}_j(t),\bm{y}_i(t))\equiv&\ \text{tan}^{-1}{\left[\frac{y_i^{N}(t)-x_j^{N}(t)}{y_i^{E}(t)-x_j^{E}(t)}\right]}\notag\\
&+\chi\{y_i^{E}(t)-x_j^{E}(t)<0\} \uppi,
\label{eq:badabada1}
\end{align}
 the angle to the target (here $\chi\{\mathcal{E}\}$ is the indicator function of the event $\mathcal{E}$), and
 \begin{align}
 \label{eq:badabada2}
 \uprho(\bm{x}_j(t),\bm{y}_i(t))\equiv\sqrt{(y_i^{N}(t)-x_j^{N}(t))^2+(y_i^{E}(t)-x_j^{E}(t))^2},
 \end{align}
  the Euclidean distance to the target. The noisy measurements regarding the position of target $i$ that are available to agent $j$ are assumed to be of the form:
\begin{align}
\label{eq:elevatormusic}
\bm{z}_{j:i}(t)\equiv \left[\begin{array}{c}\upvarphi(\bm{x}_j(t),\bm{y}_i(t)) \\\uprho(\bm{x}_j(t),\bm{y}_i(t))\end{array}\right]+\bm{v}\in \mathbb{R}^2,
\end{align}
where $\bm{v}\sim \mathcal{N}(\bm{0},\bm{R})$ for a diagonal covariance matrix $\bm{R}$. A similar model to  (\ref{eq:elevatormusic}) is assumed for an agent detecting another agent. 
In the sensor model above, agents do now know the true position or velocity of the target(s) or of any other agent; only noisy observations of the distance and angle to any detected agent/target are available. However, it is assumed that an agent can correctly classify any detected individual as either an agent or a target. It is also assumed that an agent can temporarily assign names or ``tags'' to detected agents/targets  order to match an agent/target's previous observed position to its current observed position. If an agent ever loses track (i.e., fails to detect) a previously detected agent/target, it is assumed that the tag for that agent/target resets, that is the agent can no longer match the previously detected agent/target to any agent/target detected in the future.

\subsection{The Loss Function to Minimize}
\label{sec:idealloss}

In order to motivate the choice of loss function to minimize, let us begin by considering a simplified setting where there is only one agent and one target. Here, if the agent detects the target at time $t$, a natural strategy for the agent to implement is to attempt to have $[x_1^E(t+\updelta t),x_1^N(t+\updelta t)]$, the agent's position at time $t+\updelta t$ with $\updelta t>0$, resemble $[y_1^E(t+\updelta t),y_1^N(t+\updelta t)]$, the target's position at time $t+\updelta t$. If $\bm{y}_1(t+\updelta t)$ were known, the agent could attempt to minimize the function $\|[x_1^E(t+\updelta t),x_1^N(t+\updelta t)]-[y_1^E(t+\updelta t),y_1^N(t+\updelta t)]\|^2/2$ with respect to $[x_1^E(t+\updelta t),x_1^N(t+\updelta t)]$. 
 Consider now the case where there are $n$ agents and one target and assume agent $j$ detects the target at time $t$. Furthermore, let us temporarily assume that agent $j$ knows the values of $[y_1^E(t+\updelta t),y_1^N(t+\updelta t)]$ and  $[x_i^E(t+\updelta t),x_i^N(t+\updelta t)]$ for $i\neq j$, that is agent $j$ knows the future positions of all agents and targets (this assumption will be weakened later on). If agent $j$ determines that several agents will be close to the target at time $t+\updelta t$ then, in light of Objective 2 on p. \pageref{eq:sandywpusstod}, it may no longer be desirable for agent $j$ to move towards the target's position. Instead, a more appropriate strategy for agent $j$ might be to move away from the target's position in an attempt to maximize the overall probability of detecting other targets that may be in the area of interest (AOI). Even in this simplified setting where we are assuming that the future positions of the target and of all other agents are known to agent $j$, it is unclear {\it{where specifically}} agent $j$ should move if it decides to move away from the target's position. Motivated by the paper by Lee et al. (2015)\nocite{leeetal2015robotics}, we propose to let agent $j$ move to the center of mass of its estimated Voronoi cell. This approach is based on the idea of attempting to have agents ``spread out'' when not actively tracking a target. Details are given next.

\afterpage{%
\begin{figure}[!t]
\centering
\begin{tikzpicture}
  \definecolor{mynewcoloring}{RGB}{102, 128, 153}
\node[inner sep=0pt] (probe) at (0,0)
    {\includegraphics[scale=0.6]{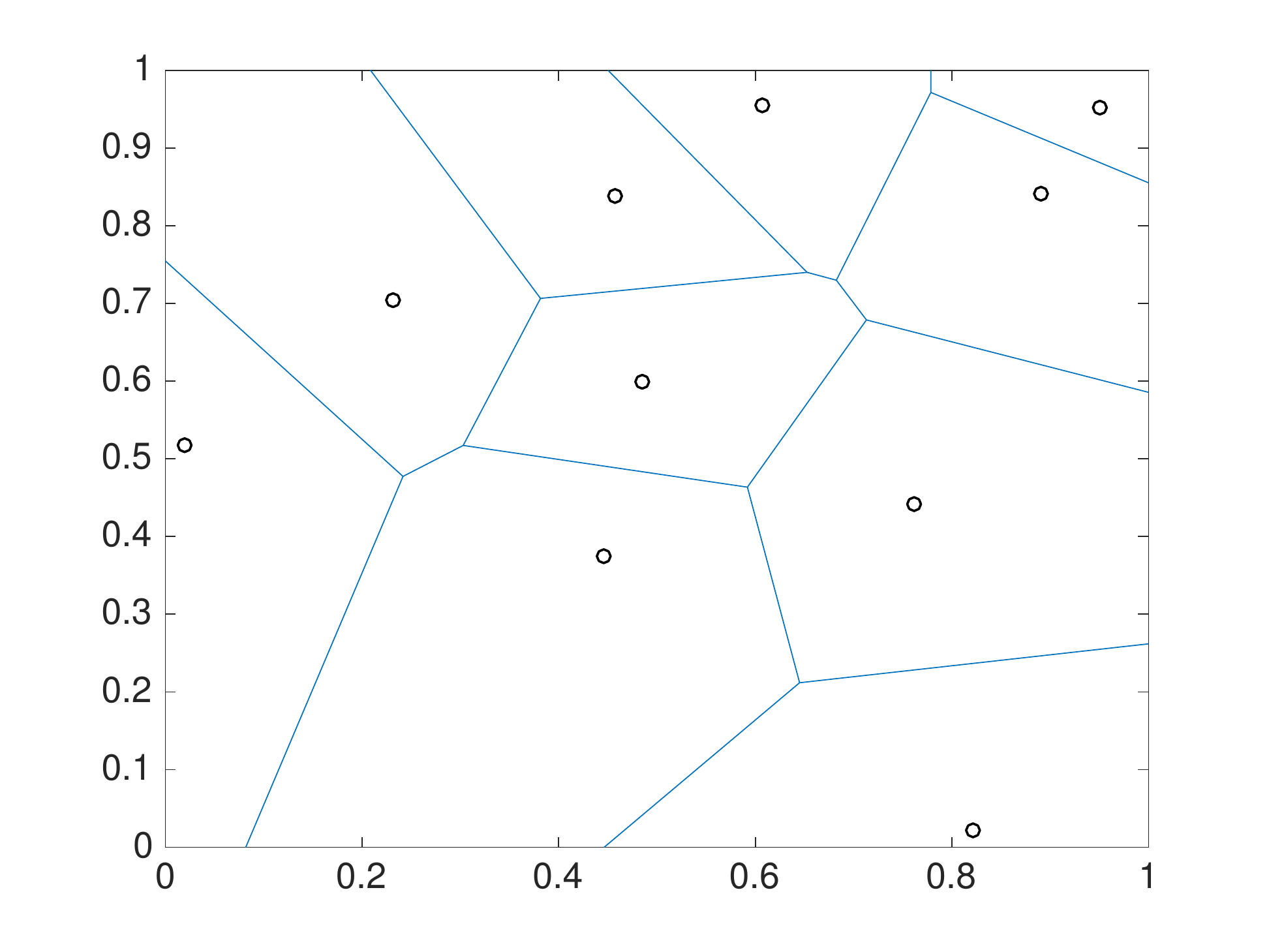}};  
\end{tikzpicture}
\caption[A Voronoi tesselation based on the positions of the agents]{Voronoi tessellation where circles mark the positions of the agents. The Voronoi tessellation is unique. Cells depend on the positions of the agents.}
\label{fig:voronoi}
\end{figure} 
}

When agent $j$ is not actively tracking a target, that is when the agent is not addressing Objective 1, the agent should be attempting to address Objective 2. In this light, agent $j$ should take action to increase the overall coverage of the AOI (loosely speaking, we want the agents to ``spread out'' when not actively tracking a target). For this we define the following coverage function:
\begin{align}
\label{eq:coveragefunc}
H(P_1,\dots,P_n,[x_1^E,x_1^N]^\top,\dots, [x_n^E,x_n^N]^\top)\equiv \sum_{j=1}^n\int_{P_j}\|\bm{q}-[x_j^E,x_j^N]^\top\|^2\  d\bm{q},
\end{align}
where $P_1,\dots, P_n$ are subsets of $\mathbb{R}^2$ that form a partition of the AOI (a partition of the AOI is a collection of disjoint cells whose union is equal to the entire AOI). Minimizing (\ref{eq:coveragefunc}) is equivalent to maximizing the coverage of the AOI. Moreover, it is known (see, for example, Lee et al. 2015)\nocite{leeetal2015robotics} that minimizing (\ref{eq:coveragefunc}) is equivalent to minimizing the following function:
\begin{align}
\label{eq:minimizemeImhard}
F([x_1^E,x_1^N]^\top,\dots, [x_n^E,x_n^N]^\top)\equiv  \sum_{j=1}^n\int_{V_j}\|\bm{q}-[x_j^E,x_j^N]^\top\|^2\  d\bm{q},
\end{align}
where $V_j\subset \mathbb{R}^2$ is the $j$th Voronoi cell defined as the set of points that are at least as close (in Euclidean distance) to agent $j$ than to any other agent (Figure \ref{fig:voronoi} gives an example of a Voronoi tessellation, a partition where the cells are Voronoi cells). While minimizing (\ref{eq:minimizemeImhard}) is difficult, a critical point of (\ref{eq:minimizemeImhard}) is obtained when $[x_j^E,x_j^N]^\top=\bm{c}_j$, where
\begin{align*}
\bm{c}_j=[c_j^E,c_j^N]^\top\equiv \frac{\int_{\bm{q}\in V_j}\bm{q}\ d\bm{q}}{\text{Area of $V_j$}}
\end{align*}
is the center of mass of $V_j$. This suggests that 
a possible strategy for agent $j$ to implement is to move towards $\bm{c}_j$ when not actively tracking a target.

So far, the discussion in this section suggests that if $[y_1^E(t+\updelta t),y_1^N(t+\updelta t)]$ and  $\{[x_i^E(t+\updelta t),x_i^N(t+\updelta t)]\}_{i\neq j}$ are known to agent $j$ at time $t$ (a generally unreasonable assumption), then agent $j$ should adjust its vector of velocities to attempt to minimize the function:
\begin{align}
& L_j(\dot{x}_j^E(t),\dot{x}_j^N(t))\equiv\notag\\
& \chi\{\upgamma_1(t+\updelta t)< \uptau\}\left[\frac{\left\|[{x}_j^E(t+\updelta t),{x}_j^N(t+\updelta t)]-[{y}_1^E(t+\updelta t),{y}_1^E(t+\updelta t)]\right\|^2}{2}\right]\notag\\
&+\chi\{\upgamma_1(t+\updelta t)\geq \uptau\}\left[\frac{\left\|[{x}_j^E(t+\updelta t),{x}_j^N(t+\updelta t)]-\bm{c}_j(t)^\top\right\|^2}{2}\right],
\label{eq:losssunknown}
\end{align}
where $\upgamma_1(t+\updelta t)$ is the number of agents (not counting agent $j$) that will be close  to target 1 at time $t+\updelta t$ (the definition of ``close'' is flexible), $\uptau$ is a predetermined nonnegative integer, $\bm{c}_j(t)$ is the center of mass of the Voronoi cell, $V_j$, computed based on $\{[{x}_i^E(t+\updelta t),{x}_i^N(t+\updelta t)]\}_{i\neq j}$ and $[{x}_j^E(t),{x}_j^N(t)]$, and
\begin{align}
\label{eq:azucarmoreno}
\left[\begin{array}{c}{x}_j^E(t+\updelta t) \\{x}_j^N(t+\updelta t)\end{array}\right]=\left[\begin{array}{c}{x}_j^E(t) \\{x}_j^N(t)\end{array}\right]+\updelta t \left[\begin{array}{c}\dot{x}_j^E(t) \\\dot{x}_j^N(t)\end{array}\right].
\end{align} 
The function:
\begin{align}
&L_j(\dot{x}_j^E(t),\dot{x}_j^N(t))\equiv \Bigg[\sum_{i=1}^{\text{\# Targets}} \chi\{i=i^\ast\}\chi\{\upgamma_i(t+\updelta t)< \uptau\}\times\notag\\
&\frac{\left\|[{x}_j^E(t+\updelta t),{x}_j^N(t+\updelta t)]-[{y}_i^E(t+\updelta t),{y}_i^N(t+\updelta t)]\right\|^2}{2}\Bigg]\notag\\
&+\Bigg[ \chi\{\upgamma_i(t+\updelta t)\geq \uptau \text{ for all } i{\text{ or }} \text{\# Targets} = 0\}\times \notag\\
&\frac{\left\|[{x}_j^E(t+\updelta t),{x}_j^N(t+\updelta t)]-\bm{c}_j(t)^\top\right\|^2}{2}\Bigg]
\label{eq:proactivefeet}
\end{align}
provides a natural generalization to (\ref{eq:losssunknown}) for the case where there may be zero or multiple targets; in (\ref{eq:proactivefeet}), $\upgamma_i(t+\updelta t)$ is the number of agents (other than agent $j$) that will be close  to target $i$ at time $t+\updelta t$, and $i^\ast\in S\equiv \{i{\text{ such that }} \upgamma_i(t+\updelta t)< \uptau\}$ is such that out of all $i\in S$, $i^\ast$ is the label/name of the target whose future position (at time $t+\updelta t$) is closest to agent $j$'s current position (at time $t$).\label{page:defineS} If multiple targets in the set $S$ are to be equally close to agent $j$ at time $t+\updelta t$, select one of these targets at random and let the selected target be target $i^\ast$. The indicator functions in (\ref{eq:proactivefeet}) indicate whether agent $j$ believes target $i$ is already sufficiently tracked/covered by other agents. 


 Let us comment on the loss function from (\ref{eq:proactivefeet}). First, note that this loss function is time-varying and possibly random since it depends on the time-varying position of the target and of all other agents. Second, $L_j(\dot{x}_j^E(t),\dot{x}_j^N(t))$ is agent-specific in that each agent is minimizing a different function (this can be seen by noting the presence of $\bm{c}_j(t)$ in the loss function, a term that is agent-dependent).  Another thing to note is that although we have used the number of agents near the target as a measure of how well the target's location is covered, other measurements of target coverage could be used such as the probability that the target will be detected at time $t+\updelta t$. The last observation we make pertains to the fact that the values of $[y_1^E(t+\updelta t),y_1^N(t+\updelta t)]$ and  $\{[x_i^E(t+\updelta t),x_i^N(t+\updelta t)]\}_{i\neq j}$ are unknown to agent $j$ at time $t$, as are the velocities of all other agents and targets. Therefore, the value of the function (\ref{eq:proactivefeet}) is unknown to agent $j$ (recall that the variable $\upgamma_i(t+\updelta t)$ appearing in (\ref{eq:proactivefeet}) depends on the state vectors of the other agents). The following section uses Kalman Filter-based estimates of agent- and target positions to approximate (\ref{eq:proactivefeet}). 
 
 Using estimated values of the other agents' positions introduces an important issue. Suppose that agent $j$ were to predict that $\upgamma_i(t+\updelta t)$ other agents would be located near target $i$ at time $t+\updelta t$ where $\upgamma_i(t+\updelta t)\geq \uptau$. Then,the decision of agent $j$ will be to move towards the center of mass of its Voronoi cell which could result in the agent moving away from the target. Note, however, that if the other $n-1$ agents were to also predict that at least $\uptau$ agents would be located near the target at time $t+\updelta t$, then all $n$ agents would move towards the centers of their Voronoi cells and away from the target, which  could result in a problematic situation where the target is left ``unguarded''. To avoid this problem, we consider a setting in which agents update their velocity vectors in a strictly cyclic manner.
 As a result, when the time comes for agent $j$ to update its velocity, the other agents will have already begun moving towards- or away from the target, giving agent $j$ information about their intentions.
 %
 %
 %


\section{Proposed Cyclic SA-Based Approach}
\label{sec:proposedapproach}
This section describes the precise way in which an agent estimates the future positions of detected agents and targets and how the agent uses this information to adjust its state vector based on the loss function in (\ref{eq:proactivefeet}). Specifically, this section answers the following questions:
\begin{enumerate}
\item How does agent $j$ estimate $\{\bm{x}_i(t+\updelta t)\}_{i\neq j}$ and $\bm{y}_i(t+\updelta t)$ (i.e., the future state vectors of other agents and the state vectors of targets)?
\item How does agent $j$ use the estimates of $\{\bm{x}_i(t+\updelta t)\}_{i\neq j}$ and $\bm{y}_i(t+\updelta t)$ to update its own state vector (position and velocity)?
\end{enumerate}
 Section \ref{sec:estimatingstate} addresses the first question and Section \ref{sec:algooutline} addresses the second question.

\subsection{Extended Kalman Filter for Estimating State Vectors}
\label{sec:estimatingstate}

At time $t$, agent $j$ uses the extended Kalman filter (EKF) to obtain an estimate for $\bm{y}_i(t+\updelta t)$, the true state vector of target $i$ at time $t+\updelta t$. The EKF is used due to the fact that (\ref{eq:elevatormusic}) is nonlinear in the state vector. To use the EKF, agent $j$ first models the dynamics of target $i$ as follows:
\begin{align}
\label{eq:qbert}
\bm{y}_i(t+\updelta t)=\bm\Phi\bm{y}_i(t)+\bm{w}(t)
\end{align}
for $t$ in a discrete set of points $\updelta t$ time units apart, where 
\begin{align}
\label{eq:qbert2mass}
\setstretch{1.25} 
\bm{\Phi}\equiv \left[\begin{array}{cccc}1 & 0 & \updelta t & 0 \\0 & 1 & 0 & \updelta t \\0 & 0 & 1 & 0 \\0 & 0 & 0 & 1\end{array}\right]
\end{align}
and where $\bm{w}$ is a vector of zero-mean Gaussian white noise with covariance matrix $\bm{Q}$. 
 The EKF algorithm used by agent $j$ for estimating the target's state vector is illustrated in Algorithm \ref{alg:EKF}. Loosely speaking, at time $t-\updelta t$ agent $j$ predicts the position of target $i$ at time $t$ (Line \ref{line:predictline} of Algorithm \ref{alg:EKF}). The predicted state estimate is denoted by $\hat{\bm{y}}_i(t|t-\updelta t)$. Then, at time $t$, agent $j$ uses the latest noisy information on the position of target $i$ to correct its previous estimate (Line \ref{line:correctline} of Algorithm \ref{alg:EKF}) . The corrected state estimate is denoted by $\hat{\bm{y}}_i(t|t)$ (although $\bm{H}$, $\bm{K}$, and $\bm{\upeta}$ are functions of time, this dependance has been omitted for simplicity).  We assume that agent $j$ also estimates the state vector of another detected agent, say agent $i$, using Algorithm \ref{alg:EKF} after a natural modification to replace $\hat{\bm{y}}_i(t|t)$ and $\hat{\bm{y}}_i(t|t-\updelta t)$ with $\hat{\bm{x}}_i(t|t)$ and $\hat{\bm{x}}_i(t|t-\updelta t)$, respectively.
 
   \begin{algorithm}[!t]                     
\caption{The Extended Kalman Filter (EKF) for Estimating $\bm{y}_i(t)$}      
\label{alg:EKF}                
\begin{algorithmic} [1]                  
\setstretch{1.5} 
\REQUIRE    $\hat{\bm{y}}_i(t-\updelta t|t-\updelta t)\in \mathbb{R}^4$ and $\bm{P}_{t-\updelta t|t-\updelta t}\in \mathbb{R}^{4\times 4}$.
\STATEx{\bf{Predict}} (execute at time $t-\updelta t$):
\STATE{$\hat{\bm{y}}_i(t|t-\updelta t)= \bm\Phi\hat{\bm{y}}_i(t-\updelta t|t-\updelta t)$, where $\bm\Phi$ is defined in (\ref{eq:qbert2mass}).}\label{line:predictline}
\STATE{$\bm{P}_{t|t-\updelta t}=\bm\Phi \bm{P}_{t-\updelta t|t-\updelta t}\bm\Phi^\top+\bm{Q}$, where $\bm{Q}$ is defined below (\ref{eq:qbert2mass}).}
\STATEx{\bf{Correct}} (execute at time $t$):
\STATE{$\bm\upeta = \bm{z}_{j:i}(t)-[{\upvarphi}(\bm{x}_j(t),\hat{\bm{y}}_i(t|t-\updelta t)), {\uprho}(\bm{x}_j(t),\hat{\bm{y}}_i(t|t-\updelta t))]^\top$, where ${\upvarphi}(\cdot,\cdot)$ and ${\uprho}(\cdot,\cdot)$ are defined in (\ref{eq:badabada1}) and (\ref{eq:badabada2}).}
\STATE{$\bm{S}=\bm{H}\bm{P}_{t|t-\updelta t}\bm{H}^\top+\bm{R}$, where $\bm{R}$ is defined below (\ref{eq:elevatormusic}) and:\label{eq:linedefineH}
\begin{align*}
\setstretch{3} 
\bm{H}\equiv \left[\begin{array}{cccc}\displaystyle{\frac{x^N-y^N}{\uprho^2(\bm{x},\bm{y})}}  & \displaystyle{\frac{-(x^E-y^E)}{\uprho^2(\bm{x},\bm{y})}}  & 0 & 0 \\ \displaystyle{\frac{-(x^E-y^E)}{\uprho(\bm{x},\bm{y})}}  & \displaystyle{\frac{-(x^N-y^N)}{\uprho(\bm{x},\bm{y})}}  & 0 & 0\end{array}\right] {\text{ with $\bm{x}=\bm{x}_j(t)$ and $\bm{y}=\hat{\bm{y}}_i(t|t-\updelta t)$.}}
\end{align*}
}
\STATE{$\bm{K}=\bm{P}_{t|t-\updelta t}\bm{H}^\top\bm{S}^{-1}$.}
\STATE{$\hat{\bm{y}}_i(t|t)=\hat{\bm{y}}_i(t|t-\updelta t)+\bm{K}\bm\upeta$.}\label{line:correctline}
\STATE{$\bm{P}_{t|t}=(\bm{I}-\bm{K}\bm{H})\bm{P}_{t|t-\updelta t}$.}
\end{algorithmic}
\end{algorithm}

  One important thing to note is that the EKF algorithm defines $\bm{H}$ (appearing in Line \ref{eq:linedefineH} of Algorithm \ref{alg:EKF}) as the Jacobian matrix of the vector:
  \begin{align}
\left[\begin{array}{c}\upvarphi(\bm{x}_j(t),\hat{\bm{y}}_i(t|t-\updelta t)) \\\uprho(\bm{x}_j(t),\hat{\bm{y}}_i(t|t-\updelta t))\end{array}\right]
\label{eq:imighthaveto}
  \end{align}
   with respect to $\bm{x}_j(t)$. However, the vector in (\ref{eq:imighthaveto}) is not differentiable if and only if at least one of the following holds:
   \begin{enumerate}
\item $x_j^E(t)=\hat{y}_j^E(t|t-\updelta t)$ and $x_j^N(t)>\hat{y}_j^N(t|t-\updelta t)$, where the vector $[\hat{y}_j^E(t|t-\updelta t),\hat{y}_j^N(t|t-\updelta t)]^\top$ corresponds to the first two entries of the vector $\hat{\bm{y}}_i(t|t-\updelta t)$.
 \item $\uprho(\bm{x}_j(t),\hat{\bm{y}}_i(t|t-\updelta t))=0$.
 \end{enumerate}
 In other words, a successful implementation of the EKF algorithm requires avoiding scenarios 1 and 2 above, possibly by perturbing $\hat{\bm{y}}_i(t|t-\updelta t)$.

\subsection{Algorithm Description}
\label{sec:algooutline}

We consider a setting in which the $n$  agents update their directional velocities in a strictly cyclic manner. 
At time $t_0$, agent 1 uses its sensor to detect any nearby agents and targets. Agent 1 then estimates the future state vectors (at time $t_0+\updelta t$) of all detected agents and targets using the EKF algorithm (Algorithm \ref{alg:EKF} ). Next, using the  estimated future states of detected agents and targets, agent 1 updates its velocity vector (more details about this step will be given in the sequel). In the same manner, agent 2 updates its velocity vector at time $t_0+\updelta t/n$. In general, agent $j$ updates its directional velocity at the times $\{t_0+(j-1)\updelta t/n+k \updelta t \}_{k\geq0}$ using the predicted positions of agents and targets $\updelta t$ time units into the future. 
Before implementing the algorithm, it is necessary to specify exactly how agent $j$ updates its state vector. This is done next.

The objective of Agent $j$ is to minimize the loss function in (\ref{eq:proactivefeet}). As discussed in Section \ref{sec:idealloss}, however, agent $j$ does not have access to a closed form expression for the loss function because the function depends on the unknown future state vectors of targets and of other agents. In our proposed approach, agent $j$ uses the predicted agent- and target state vectors to estimate the gradient of the loss function (the differentiability of the loss function is addressed in Section \ref{sec:connectiontocyclicSA}). Agent $j$ then updates its state vector using the noisy gradient estimate in a steepest-descent manner. Specifically, agent $j$ updates its state vector at times $t\in \{t_0 + (j-1)\updelta t/n+k\updelta t\}_{k\geq0}$ by following the steps below.
\begin{DESCRIPTION}
\item[Step 1:] Use sensor to detect nearby agents and targets.
\item[Step 2:] Estimate the future (at time $t+\updelta t$) state vectors of all detected agents and targets using the EKF (Algorithm \ref{alg:EKF}).
\item[Step 3:] Let $\uplambda\geq 0$ be a number such that an agent is considered to be close to the target if the distance between the agent and the target is less than $\uplambda$ (n general $\uplambda\neq r$). For each detected target, estimate $\upgamma_i(t+\updelta t)$ (defined below (\ref{eq:proactivefeet})), based on $\uplambda$ and
denote the estimate by $\hat{\upgamma}_i(t+\updelta t)$.
\item[Step 4:] Find the set of targets for which $\hat{\upgamma}_i(t+\updelta t)<\uptau$ (with $\uptau>0$). 
Denote this set by $\hat{S}$ ($\hat{S}$ is therefore an estimate of the set $S$ defined on p. \pageref{page:defineS}).
\item[Step 5:] If $\hat{S}=\emptyset$ (here $\emptyset$ denotes the empty set), use the estimated future positions (at time $t+\updelta t$) of the other agents along with the  current position (at time $t$) of agent $j$ to estimate $\bm{c}_j(t)$. Denote the estimate by $\hat{\bm{c}}_j(t)$. If $S\neq \emptyset$, pick the target in $\hat{S}$ whose predicted position is closest to agent $j$ (in the event of a tie, choose a target at random). Denote that target ``target $\hat{i}^\ast$''  ($\hat{i}^\ast$ is an estimate for the integer $i^\ast$ defined on p. \pageref{page:defineS}).
\item[Step 6:] Select $a>0$ and update agent $j$'s velocity vector as follows:
\begin{align}
&[\dot{x}_j^E(t), \dot{x}_j^N(t)]^\top=(1-a)[\dot{x}_j^E(t-\updelta t),\dot{x}_j^N(t-\updelta t)]^\top\notag\\
 &+a\Bigg[ \chi\{\hat{\upgamma}_i(t+\updelta t)\geq \uptau \text{ for all } i {\text{ or \# Detected Targets}}=0\} \times\notag\\
&\frac{\hat{\bm{c}}_j(t)-[x_j^E(t),x_j^N(t)]^\top}{\updelta t}\Bigg]+ \sum_{i=1}^{\text{\# Detected Targets}} a \Bigg[\chi\{\hat{\upgamma}_i(t+\updelta t)< \uptau\}\chi\{i=\hat{i}^\ast\}\times \notag\\
& \frac{[\hat{y}_i^E(t+\updelta t|t),\hat{y}_i^N(t+\updelta t|t)]^\top-[x_j^E(t),x_j^N(t)]^\top}{\updelta t} \Bigg].
\label{eq:step6}
\end{align}
\item[Step 7:] Update agent $j$'s position using (\ref{eq:azucarmoreno}) and (\ref{eq:step6}).
\end{DESCRIPTION}
%

\subsection{Connection to Cyclic SA}
\label{sec:connectiontocyclicSA}

First, note that the loss function in (\ref{eq:proactivefeet}) is differentiable
%
 with gradient:
\begin{align}
\frac{\partial L_j(\dot{x}_j^E,\dot{x}_j^N)}{\partial [\dot{x}_j^E,\dot{x}_j^N])^\top}=&\ \Bigg[\chi\{\upgamma_i(t+\updelta t)\geq \uptau \text{ for all } i {\text{ or \# Targets}}=0\}\updelta t\times\notag\\
& \left([x_j^E,x_j^N]^\top+\updelta t[\dot{x}_j^E,\dot{x}_j^N]^\top-\bm{c}_j(t)^\top\right) \Bigg]\notag\\
&+ \sum_{i=1}^{\text{\# Targets}} \Bigg[\chi\{\upgamma_i(t+\updelta t)< \uptau\}\chi\{i=i^\ast\}\updelta t\times \notag\\
& \left([x_j^E,x_j^N]^\top+\updelta t[\dot{x}_j^E,\dot{x}_j^N]^\top-[{y}_i^E(t+\updelta t),{y}_i^N(t+\updelta t)]^\top\right) \Bigg],
\label{eq:truegradientIam}
\end{align}
(assuming the future positions of the target do not depend on the positions of the agents, a simplifying consequence of (\ref{eq:qbert})). 
The update in (\ref{eq:step6}) is based on the idea that the gradient in (\ref{eq:truegradientIam}) can be estimated using the values of $\hat{\bm{x}}_i(t+\updelta t|t)$, and $\hat{\bm{y}}_i(t+\updelta t|t)$. Specifically, define the stochastic gradient:
\begin{align}
&\hat{\bm{g}}_j(\dot{x}_j^E(t-\updelta t),\dot{x}_j^N(t-\updelta t))\equiv\notag\\
 &\Bigg[ \chi\{\hat{\upgamma}_i(t+\updelta t)\geq \uptau \text{ for all } i {\text{ or \# Detected Targets}}=0\}\updelta t \times\notag\\
&\left([x_j^E(t),x_j^N(t)]+\updelta t[\dot{x}_j^E(t-\updelta t),\dot{x}_j^N(t-\updelta t)]-\hat{\bm{c}}_j(t)^\top\right)\Bigg]^\top\notag\\
&+ \sum_{i=1}^{\text{\# Detected Targets}} \Bigg[\chi\{\hat{\upgamma}_i(t+\updelta t)< \uptau\}\chi\{i=\hat{i}^\ast\}\updelta t\times \notag\\
& \left([x_j^E(t),x_j^N(t)]+\updelta t[\dot{x}_j^E(t-\updelta t),\dot{x}_j^N(t-\updelta t)]-[\hat{y}_i^E(t+\updelta t|t),\hat{y}_i^N(t+\updelta t|t)]\right) \Bigg]^\top.
\label{eq:kettoout}
\end{align}
The vector in (\ref{eq:kettoout}) can be seen as a stochastic estimate of the gradient in (\ref{eq:truegradientIam}). Moreover, the update in (\ref{eq:step6}) can be rewritten as:
\begin{align*}
[\dot{x}_j^E(t), \dot{x}_j^N(t)]^\top=[\dot{x}_j^E(t-\updelta t), \dot{x}_j^N(t-\updelta t)]^\top-a(\updelta t)^{-2}\hat{\bm{g}}_j(\dot{x}_j^E(t-\updelta t),\dot{x}_j^N(t-\updelta t)).
\end{align*}
Thus, the update in (\ref{eq:step6}) has the general form of the SA update in (\ref{eq:youknownothing}) with $a_k= a(\updelta t)^{-2}$. A constant gain sequence was used due to the time-varying nature of the loss function. One important thing to note is that the noisy gradient estimate in (\ref{eq:kettoout}) is not an unbiased estimate of  the true gradient in (\ref{eq:truegradientIam}). The bias is due to the fact that the EKF does not generally produce unbiased predictions of the future positions of any detected agents and targets. The magnitude of the bias will increase or decrease as the bias in the EKF estimates increases or decreases, respectively.

Because agents take turns performing the update in (\ref{eq:step6}), the resulting algorithm resembles a distributed implementation of Algorithm \ref{kirkey}, a cyclic algorithm which is a special case of the GCSA algorithm (Algorithm \ref{findme}). However, the multi-agent algorithm from this chapter does not fit perfectly into the framework of Algorithm \ref{findme} due to the fact that each agent is optimizing a different loss function and each of these loss functions is time-varying. As a result, the theory from Chapters \ref{sec:cyclicseesaw}--\ref{chap:imefficient} is not directly applicable. The following section investigates the {\it{numerical}} performance of the 
algorithm 
in Section \ref{sec:algooutline}.

\section{Numerical Analysis}
\label{sec:numericsamultiagent}

This section investigates the numerical performance of the algorithm from Section \ref{sec:algooutline}. 
The numerical results assume that the AOI is always a $\text{6} \times \text{6}$ square centered at the origin and that the value of $a$ in (\ref{eq:step6}) is such that $0<a<1$. The reason for forcing $a$ to be strictly less than 1 stems from the fact that setting $a=1$ implies that the distance measure $\uprho(\bm{x}_j(t),\hat{\bm{y}}_i(t|t-\updelta t))=0$, which implies that the matrix $\bm{H}$ (see Line \ref{eq:linedefineH} of Algorithm \ref{alg:EKF}) as well as the angle to the target's predicted position are not well-defined. Having $a<1$ guarantees that $\uprho(\bm{x}_j(t),\hat{\bm{y}}_i(t|t-\updelta t))\neq0$, avoiding the aforementioned complications. 

\subsection{Software}

The numerical results in this Section are all performed in MATLAB. The MATLAB Multi-Parametric Toolbox 3.0 (Herceg et all. 2013) \nocite{MPT3} was used to compute the vertices defining the Voronoi cells, based on the current positions of the $n$ agents, while simultaneously constraining the cells to be contained within the AOI. Then, the centers of mass of the resulting Voronoi cells (i.e., the $\hat{\bm{c}}_j(t)$s) were computed using the function {\verb|polygeom|} by Sommer (2016)\nocite{polygeom}. 


\subsection{Effect of the Sensor Range}
\label{subsec:effsensorrange}

\begin{figure}
\centering
\begin{subfigure}[b]{0.5\textwidth}
 \centering
\begin{tikzpicture}

\node[inner sep=0pt] (probe) at (0,0) 
{\includegraphics[scale=0.35]{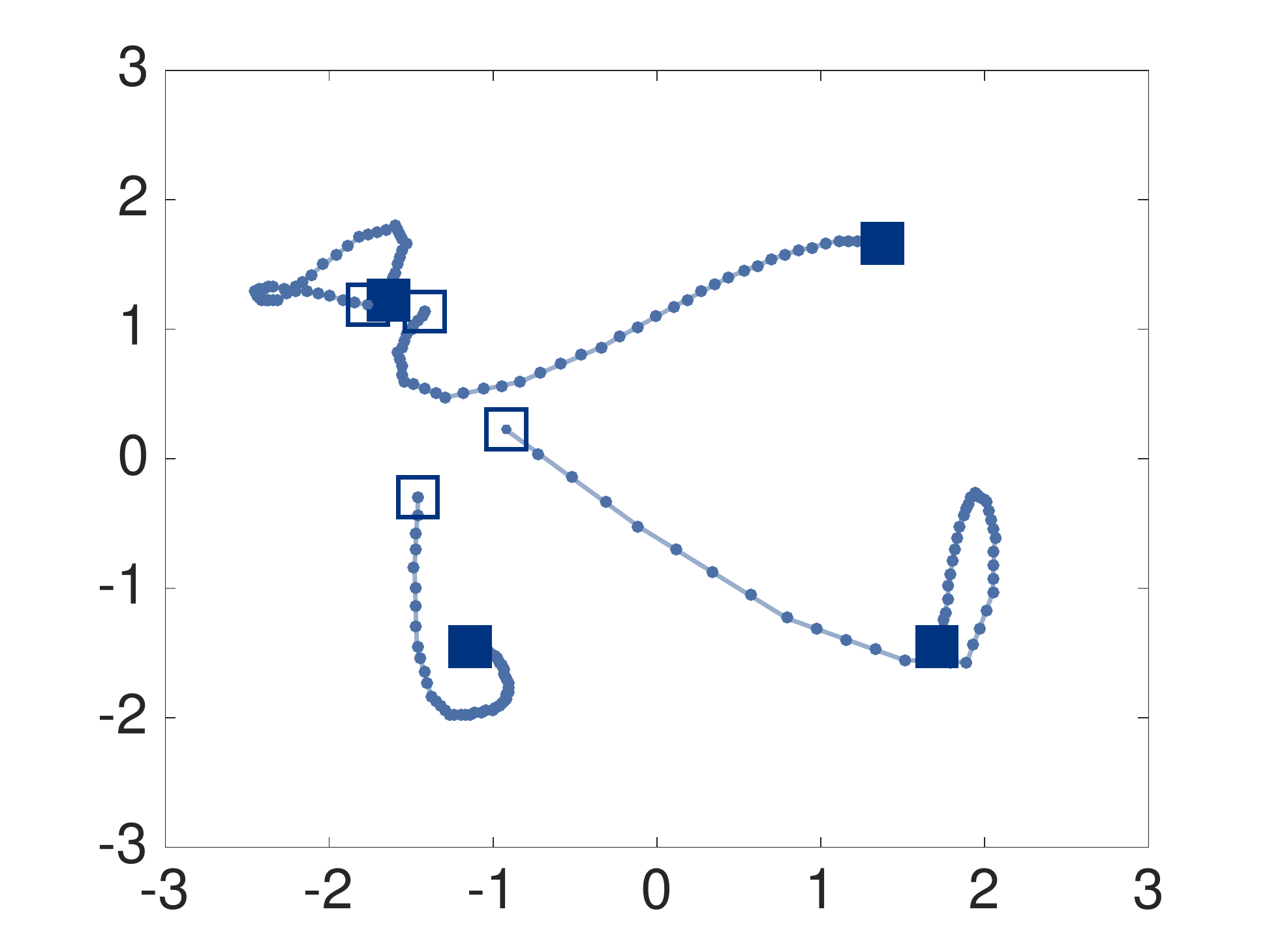}};

\end{tikzpicture}
 \caption{$r=4$, $r'=3.8$.}
 \label{fig:sensorrange1}
\end{subfigure}\hfill
\begin{subfigure}[b]{0.5\textwidth}
\centering
\begin{tikzpicture}

\node[inner sep=0pt] (probe) at (0,0)
{\includegraphics[scale=0.35]{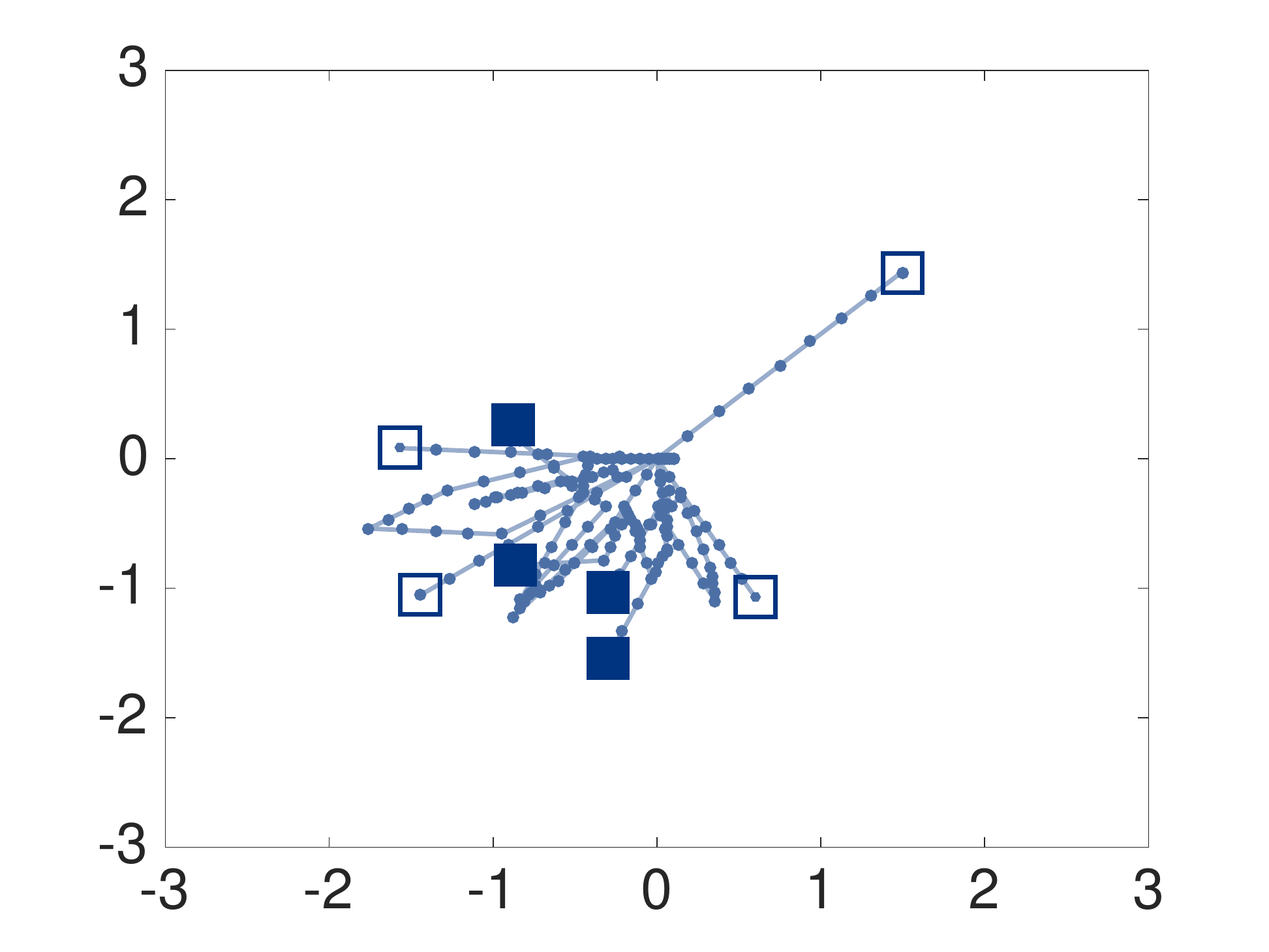}};
\end{tikzpicture}
\caption{$r=0.5$, $r'=0.4$.}
\label{fig:sensorrange2}
\end{subfigure}\hfill
\caption[The effect of the sensor range on area coverage]{The effect of the sensor range on area coverage for the case where there are 4 agents and 0 targets. Each empty square denotes the starting position of an agent. Similarly, each filled square denotes the final position of an agent. Note that in Figure \ref{fig:sensorrange2} agents fail to spread out.}
\label{fig:effectofrange}
\end{figure}

Because the algorithm in Section \ref{sec:algooutline} involves zero communication between agents, the quality of the agents' sensors directly affects the performance of the algorithm. Figure \ref{fig:effectofrange} presents two cases illustrating the effect of sensor range on area coverage for a setting where there are four agents and zero targets (both Figures \ref{fig:sensorrange1} and \ref{fig:sensorrange2} are the result of the agents updated their velocity 50 times). The first case, presented in Figure \ref{fig:sensorrange1}, shows the evolution of four agents' positions when the radii $r=2$ and $r'=1.5$ (see (\ref{eq:sensorequationmodel}) for the definitions of $r$ and $r'$). Here, the four agents are frequently able to detect each other and can therefore spread out to maximize area coverage (i.e., minimize (\ref{eq:minimizemeImhard})). In contrast, Figure \ref{fig:sensorrange2} shows the evolution of four agents' positions when $r=0.5$ and $r'=0.4$. Here, the agents start relatively far apart and fail to detect each other due to the limited range of the sensors. Consequently, the agents tend to gravitate towards the center of the AOI, an action that is not optimal when attempting to maximizing area coverage. In summary, Figure \ref{fig:effectofrange} illustrates the fact that a larger sensor range provides agents with more of the information they need to distribute their positions optimally relative to the objective function in (\ref{eq:minimizemeImhard}). Figure \ref{fig:effectofrange} was produced using $\bm{R}=0.05\bm{I}$, $\updelta t=0.4$, and $a=0.5$.

\begin{figure}
\centering
\begin{subfigure}[b]{0.5\textwidth}
 \centering
\begin{tikzpicture}

\node[inner sep=0pt] (probe) at (0,0) 
{\includegraphics[scale=0.35]{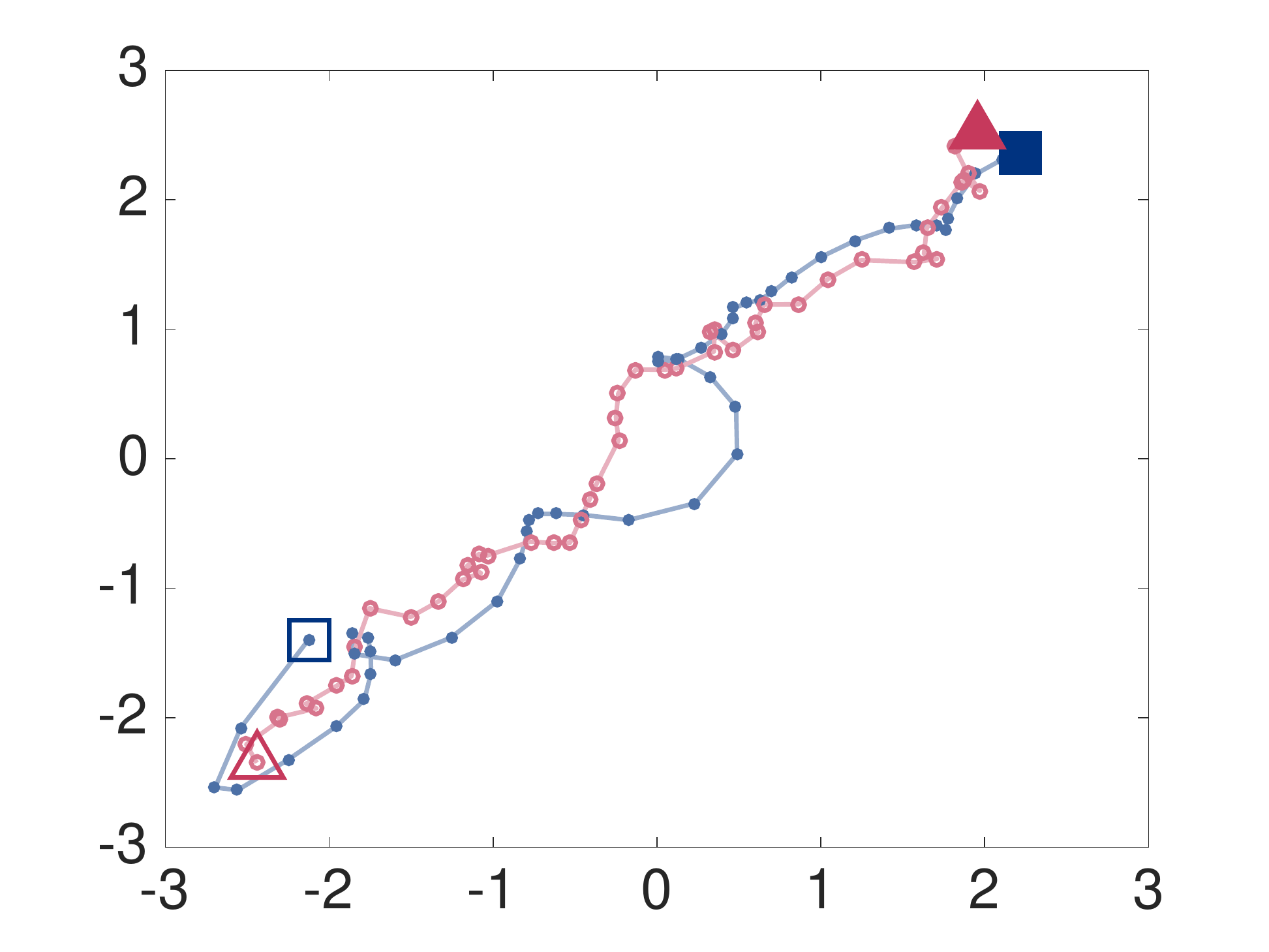}};

\end{tikzpicture}
 \caption{$r=2$, $r'=1.5$.}
 \label{fig:rangetrack1}
\end{subfigure}\hfill
\begin{subfigure}[b]{0.5\textwidth}
\centering
\begin{tikzpicture}

\node[inner sep=0pt] (probe) at (0,0)
{\includegraphics[scale=0.35]{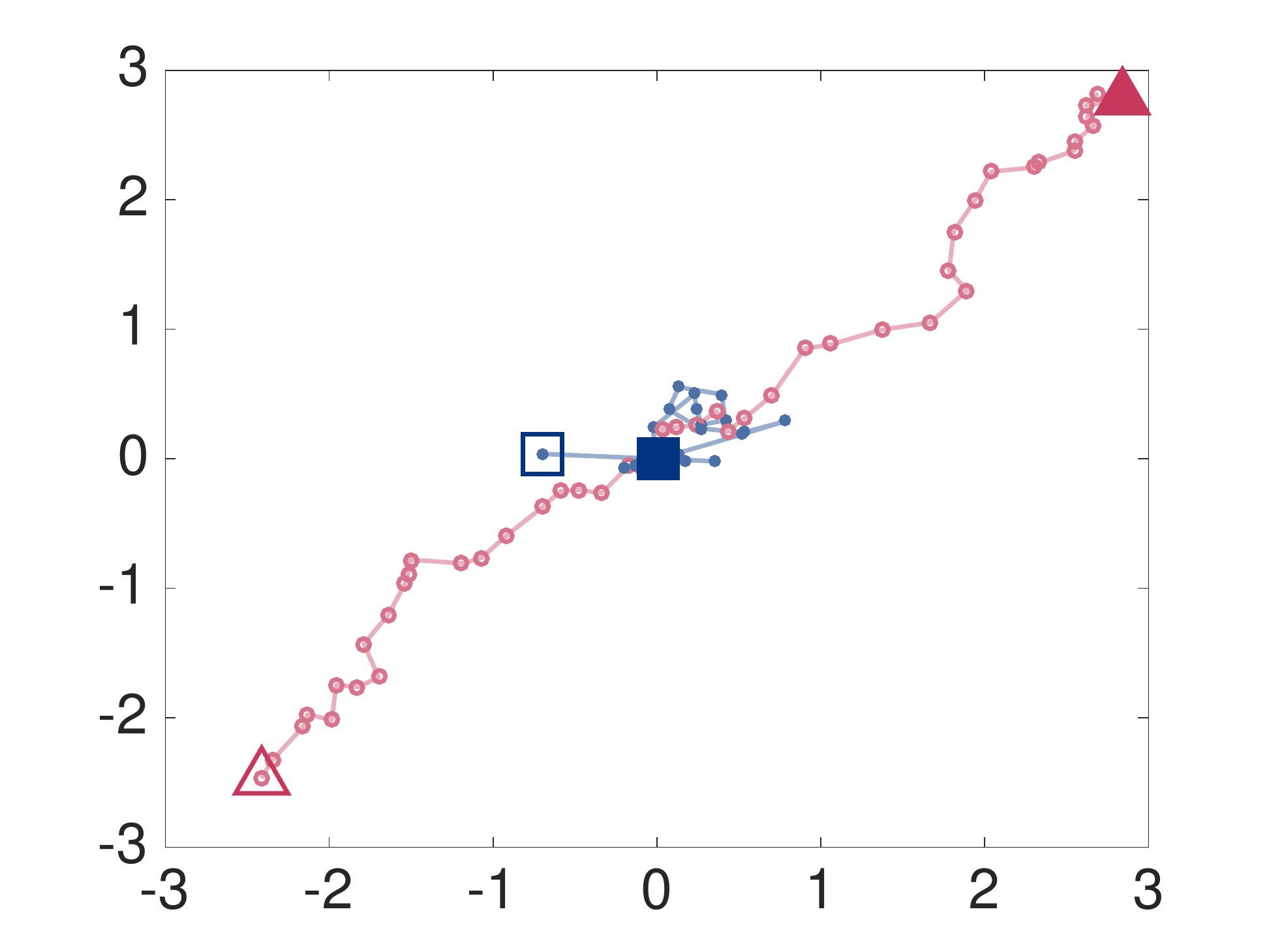}};
\end{tikzpicture}
\caption{$r=0.5$, $r'=0.4$.}
\label{fig:rangetrack2}
\end{subfigure}\hfill
\caption[The effect of the sensor range on target tracking]{The effect of sensor range on target tracking for the case where there is 1 agent and 1 target. An empty square (respectively empty triangle) denotes the starting position of the agent (respectively target). The filled square (respectively filled triangle) denotes the final position of the agent (respectively target). Note that in Figure \ref{fig:rangetrack2} the agent fails to track the target.}
\label{fig:rangetrack}
\end{figure}

Aside from improving area coverage, a larger sensor range has the (not surprising) benefit of also improving target tracking. In Figure \ref{fig:rangetrack1}, for example, a single agent whose sensor has a large range is able to track a single target fairly well. When the sensor's range is significantly decreased, however, Figure \ref{fig:rangetrack2} shows that the agent is no longer able to continuously detect and keep track of the target. The simulations in Figure \ref{fig:rangetrack} were produced using $\bm{R}=0.05\bm{I}$, $\updelta t =0.1$, and $a=0.5$; and the target was assumed to move according to the linear model in (\ref{eq:qbert}) with $\bm{Q}=0.01\bm{I}$ and velocity equal to 1 (i.e., the last two entries of $\bm{y}_i(t)$ are equal to 1). Note that Figure \ref{fig:effectofrange}  uses $\updelta t=0.4$ while Figure \ref{fig:rangetrack} uses $\updelta t=0.1$. This was done with the objective of simulating a setting which the agents begin updating their parameter vectors at times that are 0.1 time units apart, which implies $\updelta t=0.n$ when there are $n$ agents.

\begin{figure}
\centering
\begin{subfigure}[b]{0.5\textwidth}
 \centering
\begin{tikzpicture}

\node[inner sep=0pt] (probe) at (0,0) 
{\includegraphics[scale=0.35]{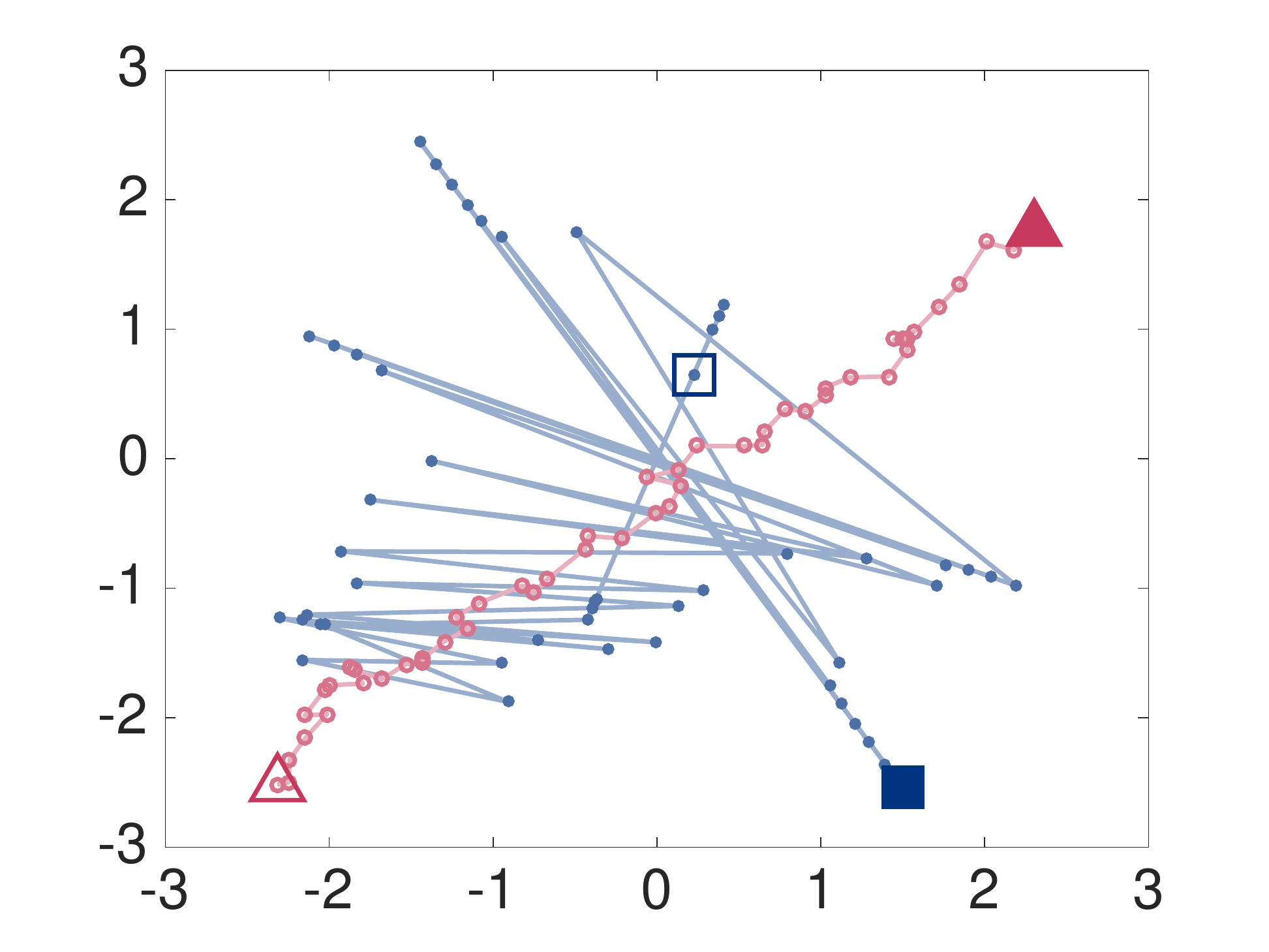}};

\end{tikzpicture}
 \caption{$a=1.3$.}
 \label{fig:gaintoolarge}
\end{subfigure}\hfill
\begin{subfigure}[b]{0.5\textwidth}
\centering
\begin{tikzpicture}

\node[inner sep=0pt] (probe) at (0,0)
{\includegraphics[scale=0.35]{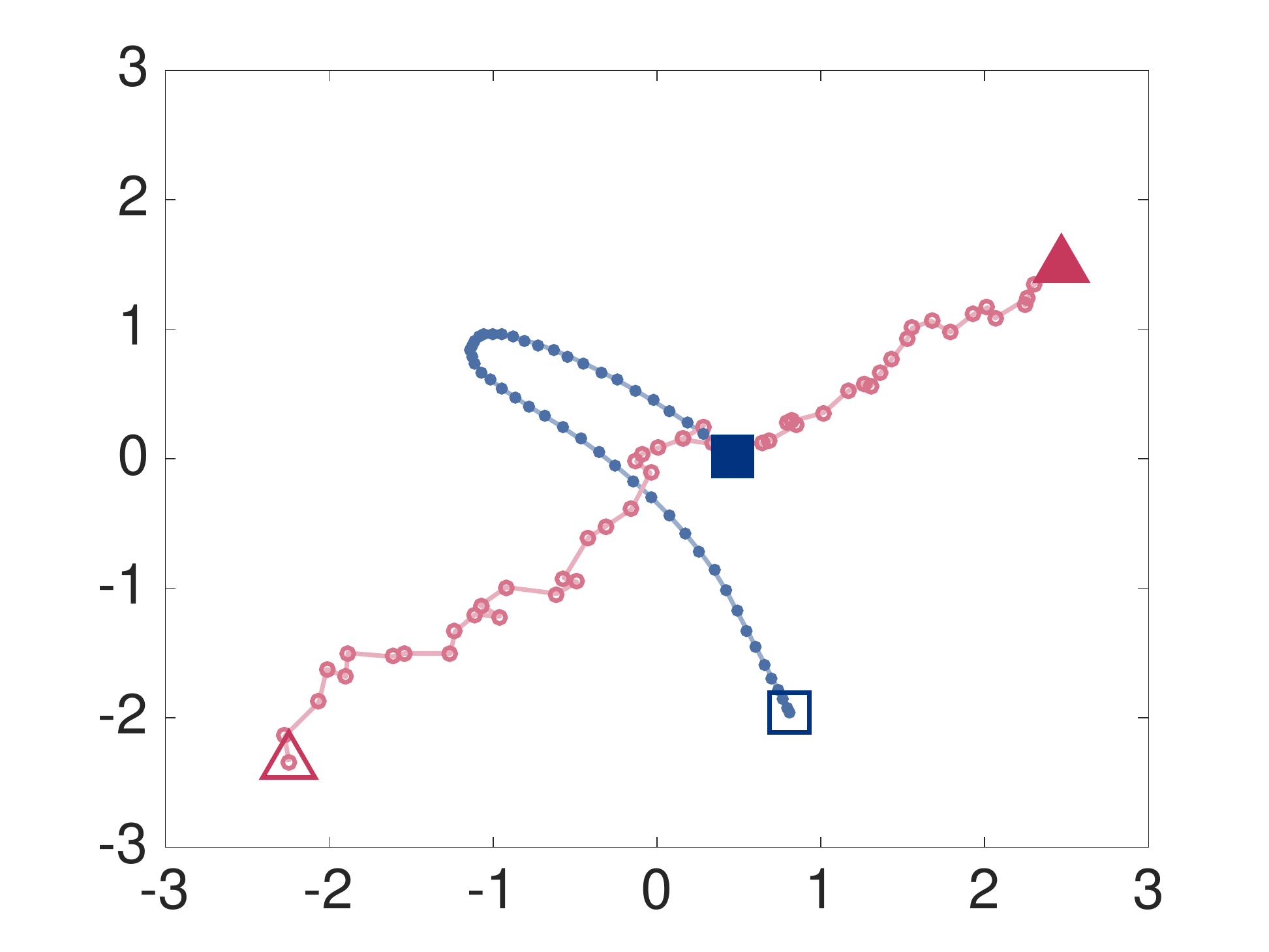}};
\end{tikzpicture}
\caption{$a=0.01$.}
\label{fig:gaintoosmall}
\end{subfigure}\hfill
\caption[The effect of the gain sequence on target tracking]{The effect of the gain sequence on target tracking for the case where there is 1 agent and 1 target. An empty square (respectively empty triangle) denotes the starting position of the agent (respectively target). The filled square (respectively filled triangle) denotes the final position of the agent (respectively target). }
\label{fig:magnitudepfthegain}
\end{figure}

\subsection{Effect of the Gain Sequence}
\label{subsec:effofgan}

Like the sensor range, the value of $a$ (the gain appearing in Step 6 of the multi-agent algorithm) plays a crucial role in the performance of the algorithm. Figure \ref{fig:gaintoolarge}, for example, shows an agent attempting to track a target when the gain sequence is too large. Here, the agent's velocity vector is highly sensitive to small changes in the target's predicted position, this eventually results in the agent moving too far from the target to the point where the agent is no longer able to detect it. Similarly, Figure \ref{fig:gaintoosmall} shows how using a gain sequence that is too small also results in the agent not being able to remain close enough the target. Consequently, the agent eventually loses track of the target. The simulations in Figure \ref{fig:magnitudepfthegain} were produced using $\bm{R}=0.05\bm{I}$,  $\updelta t =0.1$, $a=0.5$, $r=2$, and $r'=1.5$; and the target was assumed to move according to the linear model in (\ref{eq:qbert}) with $\bm{Q}=0.01\bm{I}$ and velocity equal to 1 (i.e., the last two entries of $\bm{y}_i(t)$ are equal to 1). The agent and target updated their velocities 50 times.

Unfortunately, there is no general way to determine the best value of $a$ in practice. One observation regarding the update in (\ref{eq:step6}) that can be helpful in selecting the value of $a$ is the following: setting $a=1$ implies that $[{x}_j^{E}(t+\updelta t),{x}_j^{N}(t+\updelta t)]^\top$, the position of agent $j$ at time $t+\updelta t$, will be exactly equal to one of $[\hat{y}_i^E(t+\updelta t|t),\hat{y}_i^N(t+\updelta t|t)]^\top$ and $\hat{\bm{c}}_j(t)$, depending on whether the agent decides to move towards the target or towards the center of mass of its Voronoi cell. Therefore, when an agent is actively tracking a target, setting $a=1$ implies that $\uprho(\bm{x}_j(t+\updelta t),\hat{\bm{y}}_i(t+\updelta t|t))=0$. This observation implies that setting $a=1$ is not an ideal strategy for two reasons: 
\begin{enumerate}
\item The EKF algorithm is not well-defined when $\uprho(\bm{x}_j(t+\updelta t),\hat{\bm{y}}_i(t+\updelta t|t))=0$ (see the discussion following (\ref{eq:imighthaveto})).
\item In reality, two bodies cannot occupy the same space so that having $\uprho(\bm{x}_j(t+\updelta t),\hat{\bm{y}}_i(t+\updelta t|t))=0$ is not physically possible when $\hat{\bm{y}}_i(t+\updelta t|t)$ is close to the target's true position at time $t+\updelta t$.
\end{enumerate}
In light of the discussion above, a general guideline for choosing $a$ is to pick a value that is only slightly less than 1 if the target's predicted state is expected to be somewhat accurate. For our numerical experiments we select values of $a$ that are strictly less than one in an attempt to ensure that the algorithm is not too sensitive to changes in the EKF's predictions.

\subsection{Maximum Spread and Minimum Distance to Target}
\label{sec:spreadanddistance}

Sections \ref{subsec:effsensorrange} and \ref{subsec:effofgan} explained the significance of $r$, $r'$, and $a$ when a single agent is attempting to track a target. It was observed that if the sensor range is too small or if the value of $a$ is either too small or too large, the terminal distance between the agent and the target tends to be large. When there are multiple agents in the AOI, two other variables that govern the algorithm's performance are $\uplambda$,  the distance at which a target is assumed to be close to an agent (see Step 3 of the algorithm), and $\uptau$,  the number of agents that need to be near a target for that target to be considered to be well-guarded (see Step 4 of the algorithm). A large value of $\uptau$ combined with a small value of $\uplambda$ implies that an agent will only move away from a detected target if it predicts that several agents will be extremely close to the target. In contrast, a small value of $\uptau$ combined with a large value of $\uplambda$ implies that agents very easily move away from a detected target. Thus, a poor choice of $\uptau$ and $\uplambda$ could result in either all agents moving away from the target or all agents gathering near the target; neither of these scenarios is ideal given objectives 1 and 2 on p. \pageref{eq:sandywpusstod}. With this in mind,
Tables \ref{eq:jonnsm1}--\ref{eq:jonnsm4} compute: 1) the mean terminal distance between the target and the nearest agent divided by the initial value of this distance (see the column  ``Dist. to Nearest Agent''), and 2) the mean terminal maximum distance between agents divided by the initial value of this distance (see the column ``Max. Dist. Between Agents''), for different combinations of $r$, $r'$, $\uptau$, and $\uplambda$ when there are 4 agents. Small values in the first column indicate agents are tracking the target while large values of the second column indicate agents are spreading out. The means in Tables \ref{eq:jonnsm1}--\ref{eq:jonnsm4} are computed from 20 i.i.d. replications, each consisting of 50 velocity updates per-agent. Each replication was produced by initializing the positions of agents and targets uniformly at random within the AOI; using $\bm{R}=0.05\bm{I}$, $\updelta t =0.1$, and $a=0.5$; and the target was assumed to move according to the linear model in (\ref{eq:qbert}) with $\bm{Q}=0.01\bm{I}$, velocity equal to 1 (i.e., the last two entries of $\bm{y}_i(t)$ are equal to 1).

\begin{table}[p]
\centering
\begin{tabular}{c c c c c c} \toprule
   {$\uptau$} & {$\uplambda$}  & {Dist. to Nearest Agent} & {Max. Dist. Between Agents}  \\ \midrule
     $1$  & $0.1$ & $0.1569\ (\hat{\upsigma}^2=0.0124)$  & $0.9537\ (\hat{\upsigma}^2=0.1967)$      \\
     $1$ & $0.5$ &  $0.3348\ (\hat{\upsigma}^2=0.1029)$ & $1.1006\ (\hat{\upsigma}^2=0.3250)$   \\
     $1$  & $2$ & $0.3048\ (\hat{\upsigma}^2=0.0546)$  & $1.4699\ (\hat{\upsigma}^2=0.1375)$     \\
   $1$  & $3$ & $0.4875\ (\hat{\upsigma}^2=0.1947)$  & $1.4973\ (\hat{\upsigma}^2=0.5830)$       \\ \midrule
   $2$  & $0.1$ & $ 0.2964\ (\hat{\upsigma}^2=0.0592)$  & $0.7556\ (\hat{\upsigma}^2=0.1740)$      \\
    $2$ & $0.5$ & $0.1976 \ (\hat{\upsigma}^2=0.0265)$ & $1.1087\ (\hat{\upsigma}^2=0.4836)$    \\
    $2$  & $2$ & $0.3080\ (\hat{\upsigma}^2=0.0471)$  & $ 1.2940\ (\hat{\upsigma}^2=0.3126)$     \\
    $2$  & $3$ & $0.2400\ (\hat{\upsigma}^2=0.0308)$  & $1.5787\ (\hat{\upsigma}^2=0.6409)$      \\ \midrule
    $3$  & $0.1$ & $0.1949\ (\hat{\upsigma}^2=0.0191)$  & $0.7750\ (\hat{\upsigma}^2=0.2952)$    \\
    $3$  & $0.5$ & $0.2330\ (\hat{\upsigma}^2=0.0648)$  & $1.1551\ (\hat{\upsigma}^2=1.0617)$     \\
    $3$  & $2$ & $0.2703\ (\hat{\upsigma}^2=0.0347)$  & $1.1701\ (\hat{\upsigma}^2=0.3006)$   \\
    $3$  & $3$ & $0.3261\ (\hat{\upsigma}^2=0.1056)$  & $1.0590\ (\hat{\upsigma}^2=0.1192)$          \\ \bottomrule
\end{tabular}
\caption[Normalized mean terminal distance from the target to the nearest agent and normalized mean terminal maximum distance between agents when there are four agents, one target, $r=3$, and $r'=2.5$]{Mean terminal distance from the target to the nearest agent and mean terminal maximum distance between agents when there are four agents, one target, $r=3$, and $r'=2.5$ (distances are divided by their values at the beginning of the replication prior to averaging). $\hat{\upsigma}^2$ is the sample variance. The variables $\uptau$ and $\uplambda$ have a significant effect over the behavior of the agents.}
\label{eq:jonnsm1}
\end{table}

\begin{table}[p]
\centering
\begin{tabular}{c c c c c c} \toprule
   {$\uptau$} & {$\uplambda$}  & {Dist. to Nearest Agent} & {Max. Dist. Between Agents}  \\ \midrule
     $1$  & $0.1$ & $0.3557\ (\hat{\upsigma}^2=0.0860)$  & $0.9675\ (\hat{\upsigma}^2=0.2911)$      \\
     $1$ & $0.5$ &  $0.3261\ (\hat{\upsigma}^2= 0.0597)$ & $1.2445\ (\hat{\upsigma}^2=0.2934)$   \\
     $1$  & $2$ & $0.6423\ (\hat{\upsigma}^2=0.3227)$  & $1.0706\ (\hat{\upsigma}^2=0.1149)$     \\
   $1$  & $3$ & $0.7195\ (\hat{\upsigma}^2=0.3410)$  & $0.9462\ (\hat{\upsigma}^2=0.0800)$       \\ \midrule
   $2$  & $0.1$ & $ 0.3513\ (\hat{\upsigma}^2=0.1028)$  & $1.1958\ (\hat{\upsigma}^2=0.3893)$      \\
    $2$ & $0.5$ & $0.2993 \ (\hat{\upsigma}^2=0.0726)$ & $1.2100\ (\hat{\upsigma}^2=0.1166)$    \\
    $2$  & $2$ & $0.3120\ (\hat{\upsigma}^2=0.0771)$  & $ 1.3101\ (\hat{\upsigma}^2=0.2050)$     \\
    $2$  & $3$ & $0.5354\ (\hat{\upsigma}^2=0.3261)$  & $1.3315\ (\hat{\upsigma}^2=0.4669)$      \\ \midrule
    $3$  & $0.1$ & $0.3448\ (\hat{\upsigma}^2=0.3936)$  & $1.3020\ (\hat{\upsigma}^2=0.3143)$    \\
    $3$  & $0.5$ & $0.2593\ (\hat{\upsigma}^2=0.0673)$  & $0.8850\ (\hat{\upsigma}^2=0.1520)$     \\
    $3$  & $2$ & $0.2373\ (\hat{\upsigma}^2=0.0310)$  & $1.2419\ (\hat{\upsigma}^2=0.1599)$   \\
    $3$  & $3$ & $0.3159\ (\hat{\upsigma}^2=0.1042)$  & $1.2420\ (\hat{\upsigma}^2=0.1427)$          \\ \bottomrule
\end{tabular}
\caption[Normalized mean terminal distance from the target to the nearest agent and normalized mean terminal maximum distance between agents when there are four agents, one target, $r=2$, and $r'=1.5$]{Mean terminal distance from the target to the nearest agent and mean terminal maximum distance between agents when there are four agents, one target, $r=2$, and $r'=1.5$ (distances are divided by their values at the beginning of the replication prior to averaging). $\hat{\upsigma}^2$ is the sample variance. The variables $\uptau$ and $\uplambda$ have a significant effect over the behavior of the agents.}
\label{eq:jonnsm2}
\end{table}

\begin{table}[p]
\centering
\begin{tabular}{c c c c c c} \toprule
   {$\uptau$} & {$\uplambda$}  & {Dist. to Nearest Agent} & {Max. Dist. Between Agents}  \\ \midrule
     $1$  & $0.1$ & $1.0161\ (\hat{\upsigma}^2=0.6779)$  & $0.7132\ (\hat{\upsigma}^2=0.1144)$      \\
     $1$ & $0.5$ &  $1.2612\ (\hat{\upsigma}^2=0.7190)$ & $0.6005\ (\hat{\upsigma}^2=0.1207)$   \\
     $1$  & $2$ & $1.1480\ (\hat{\upsigma}^2=0.8477)$  & $0.6960\ (\hat{\upsigma}^2=0.1133)$     \\
   $1$  & $3$ & $1.0841\ (\hat{\upsigma}^2=0.3906)$  & $0.6978\ (\hat{\upsigma}^2=0.2923)$       \\ \midrule
   $2$  & $0.1$ & $ 0.8769\ (\hat{\upsigma}^2=0.4911)$  & $0.7907\ (\hat{\upsigma}^2=0.1680)$      \\
    $2$ & $0.5$ & $1.3092 \ (\hat{\upsigma}^2=0.8296)$ & $0.6923\ (\hat{\upsigma}^2=0.1428)$    \\
    $2$  & $2$ & $1.2736\ (\hat{\upsigma}^2=0.3907)$  & $ 0.6780\ (\hat{\upsigma}^2=0.1095)$     \\
    $2$  & $3$ & $1.2974\ (\hat{\upsigma}^2=1.2430)$  & $0.6500\ (\hat{\upsigma}^2=0.0818)$      \\ \midrule
    $3$  & $0.1$ & $0.9303\ (\hat{\upsigma}^2=0.6797)$  & $0.5529\ (\hat{\upsigma}^2= 0.0540)$    \\
    $3$  & $0.5$ & $0.9133\ (\hat{\upsigma}^2=0.4332)$  & $0.6480\ (\hat{\upsigma}^2=0.1503)$     \\
    $3$  & $2$ & $1.0092\ (\hat{\upsigma}^2=0.4750)$  & $0.6979\ (\hat{\upsigma}^2=0.0554)$   \\
    $3$  & $3$ & $1.0755\ (\hat{\upsigma}^2=0.5873)$  & $0.6411\ (\hat{\upsigma}^2=0.0706)$          \\ \bottomrule
\end{tabular}
\caption[Normalized mean terminal distance from the target to the nearest agent and normalized mean terminal maximum distance between agents when there are four agents, one target, $r=1$, and $r'=0.5$]{Mean terminal distance from the target to the nearest agent and mean terminal maximum distance between agents when there are four agents, one target, $r=1$, and $r'=0.5$ (distances are divided by their values at the beginning of the replication prior to averaging). $\hat{\upsigma}^2$ is the sample variance. The variables $\uptau$ and $\uplambda$ have a significant effect over the behavior of the agents.}
\label{eq:jonnsm3}
\end{table}

\begin{table}[p]
\centering
\begin{tabular}{c c c c c c} \toprule
   {$\uptau$} & {$\uplambda$}  & {Dist. to Nearest Agent} & {Max. Dist. Between Agents}  \\ \midrule
     $1$  & $0.1$ & $1.4262\ (\hat{\upsigma}^2=0.5683)$  & $0.4505\ (\hat{\upsigma}^2=0.0398)$      \\
     $1$ & $0.5$ &  $1.0955\ (\hat{\upsigma}^2=0.3439)$ & $0.4335\ (\hat{\upsigma}^2=0.0288)$   \\
     $1$  & $2$ & $1.4668\ (\hat{\upsigma}^2=0.1747)$  & $0.4351\ (\hat{\upsigma}^2=0.0359)$     \\
   $1$  & $3$ & $1.3079\ (\hat{\upsigma}^2=0.5449)$  & $0.4665\ (\hat{\upsigma}^2=0.0306)$       \\ \midrule
   $2$  & $0.1$ & $ 1.3213\ (\hat{\upsigma}^2=0.5002)$  & $0.4293\ (\hat{\upsigma}^2=0.0227)$      \\
    $2$ & $0.5$ & $1.4793 \ (\hat{\upsigma}^2=0.9726)$ & $0.3979\ (\hat{\upsigma}^2=0.0312)$    \\
    $2$  & $2$ & $1.4770\ (\hat{\upsigma}^2=0.3656)$  & $ 0.3648\ (\hat{\upsigma}^2=0.0234)$     \\
    $2$  & $3$ & $1.2614\ (\hat{\upsigma}^2=0.6410)$  & $0.4079\ (\hat{\upsigma}^2=0.0158)$      \\ \midrule
    $3$  & $0.1$ & $ 1.2997\ (\hat{\upsigma}^2=0.3756)$  & $0.4262\ (\hat{\upsigma}^2=0.0330)$    \\
    $3$  & $0.5$ & $1.4249\ (\hat{\upsigma}^2=0.8190)$  & $0.4505\ (\hat{\upsigma}^2=0.0338)$     \\
    $3$  & $2$ & $1.0037\ (\hat{\upsigma}^2=0.1506)$  & $0.4659\ (\hat{\upsigma}^2=0.0946)$   \\
    $3$  & $3$ & $ 1.4076\ (\hat{\upsigma}^2=0.7358)$  & $0.4224\ (\hat{\upsigma}^2=0.0145)$          \\ \bottomrule
\end{tabular}
\caption[Normalized mean terminal distance from the target to the nearest agent and normalized mean terminal maximum distance between agents when there are four agents, one target, $r=0.5$, and $r'=0.4$]{Mean terminal distance from the target to the nearest agent and mean terminal maximum distance between agents when there are four agents, one target, $r=0.5$, and $r'=0.4$ (distances are divided by their values at the beginning of the replication prior to averaging). $\hat{\upsigma}^2$ is the sample variance. The variables $\uptau$ and $\uplambda$ have a significant effect over the behavior of the agents.}
\label{eq:jonnsm4}
\end{table}

An overall trend that can be observed in Tables \ref{eq:jonnsm1}--\ref{eq:jonnsm4} is that larger values of $r$ are correlated with a decrease in the mean terminal distance between the target and the nearest agent. 
Additionally, larger values of $r$ are associated with larger values of the mean maximum distance between agents. 
In contrast, small values of $r$ are correlated with a large mean terminal distance between the target and the nearest agent and a small mean terminal maximum distance between agents. 
Larger values of $r$ allow agents to minimize the distance between the target and the nearest agent while simultaneously maximizing the maximum distance between agents.

For the case where $r=2$ and $r'=1.5$, $\uptau=1$, and $\uplambda=0.5$, Figure \ref{fig:toysstory1} shows the evolution of the mean distance between the target and the nearest agent relative to the initial value of this distance using the settings from Tables \ref{eq:jonnsm1}--\ref{eq:jonnsm4}.  The result in Figure \ref{fig:toysstory1} is consistent with the results in Table \ref{eq:jonnsm2}. Figure \ref{fig:toysstory2} repeats the experiment of Figure \ref{fig:toysstory1} for the case where there is one agent and one target. The mean terminal distance in Figure \ref{fig:toysstory2} seems to be approximately the same as the mean terminal distance in Figure \ref{fig:toysstory1}, although it is important to note the high variance of the individual realizations

\begin{figure}[!p]
\centering
\begin{subfigure}[b]{1\textwidth}
 \centering
\begin{tikzpicture}

\node[inner sep=0pt] (probe) at (0,0) 
{\includegraphics[scale=0.6]{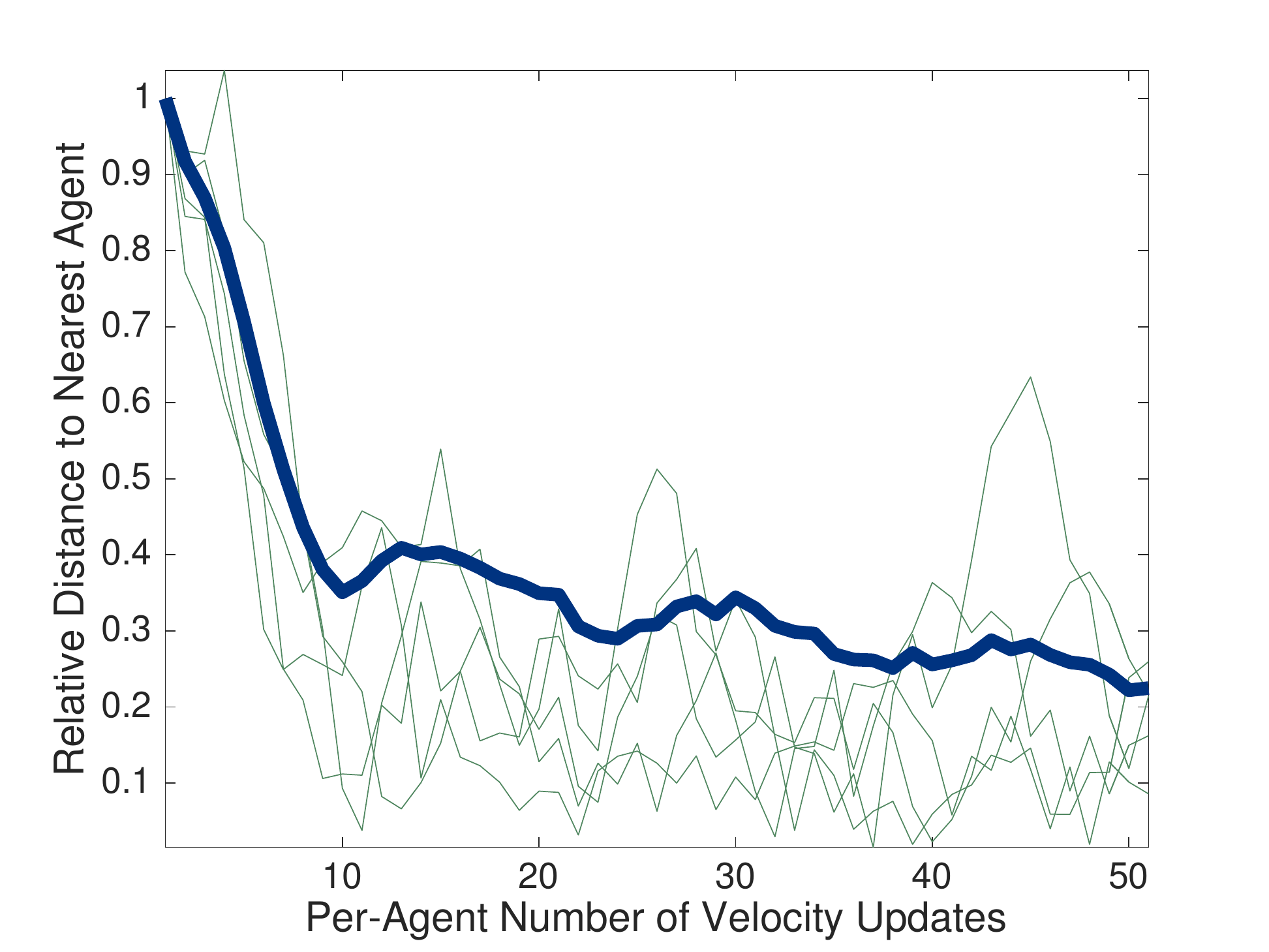}};

\end{tikzpicture}
 \caption{Four agents, one target.}
 \label{fig:toysstory1}
\end{subfigure}\hfill

\begin{subfigure}[b]{1\textwidth}
\centering
\begin{tikzpicture}

\node[inner sep=0pt] (probe) at (0,0)
{\includegraphics[scale=0.6]{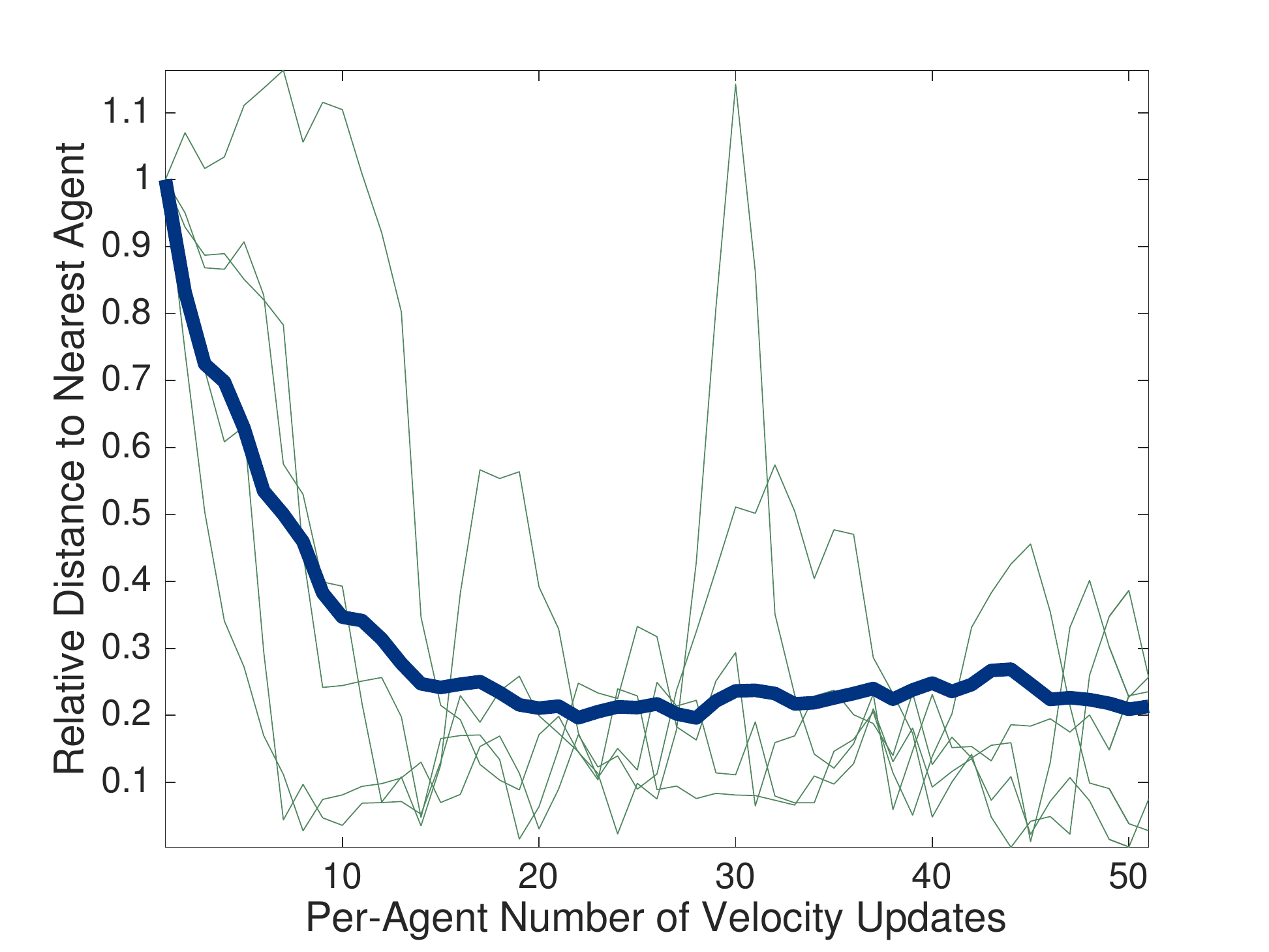}};
\end{tikzpicture}
\caption{One agent, one target.}
\label{fig:toysstory2}
\end{subfigure}\hfill
\caption[Mean distance from the target to the nearest agent divided by the initial value of this distance]{Mean distance from the target to the nearest agent divided by the initial value of this distance. The thick line represents the mean relative distance obtained by averaging 20 i.i.d. realizations while each of the thinner lines represents one of the first 5 realizations. }
\label{fig:toysstory}
\end{figure}

\section{Concluding Remarks}
\label{sec:finalcommentsagent}

This chapter investigates the numerical performance of cyclic SA applied to a multi-agent optimization problem for tracking and surveillance. An attractive feature of the resulting multi-agent optimization algorithm is that it can be implemented when agents are not allowed to communicate with each other, although it is entirely possible to extend the algorithm to include communication between agents. One assumption made throughout this section is that each agent knows their own position and velocity at any point in time. Through a natural modification, the algorithm in this chapter could also be implemented in the same manner when these variables are estimated. The numerical results indicate that, when the sensor radius is sufficiently large, the proposed algorithm allows agents to track the target while simultaneously increasing the distance between agents. 
%

A few directions for future research regarding this chapter's multi-agent algorithm include investigating other methods for predicting target/agent state positions (aside from the EKF), considering adaptive gain sequences (instead of the using the constant gain $a$), and considering other objective functions. With regards to this last topic, we remind the reader that  in this chapter each agent is minimizing a different loss function, which is a different setting than the setting from Chapters \ref{chap:descriptionofGCSA}--\ref{chap:numericschap} where there is a universal (i.e., common to all agents) loss function to minimize. A universal loss function can be constructed for the multi-agent problem in this chapter by simply adding the individual loss functions. Specifically, the universal loss function could be defined as $L(\bm\uptheta)\equiv \sum_{j=1}^n L_j(\dot{x}_j^E(t),\dot{x}_j^N(t))$, where $\bm\uptheta$ is a vector containing the velocity vectors of all $n$ agents. Note that the value of $[\dot{x}_j^E(t),\dot{x}_j^N(t)]^\top$ affects the value of $L_i(\dot{x}_i^E(t),\dot{x}_i^N(t))$ even when $i\neq j$ (through the terms $\bm{c}_i(t)$ and $\upgamma_\ell(t+\updelta t)$ for $\ell=1,\dots,{\text{\# Agents}}$).
Furthermore, if agent $j$ is located at a distance of exactly $\uplambda$ from any target (under the usual Gaussian state model this occurs with probability zero) then $L(\bm\uptheta)$ is not differentiable with respect to $[\dot{x}_j^E(t),\dot{x}_j^N(t)]^\top$. Future work could focus on studying this more complex scenario.

%% file: chapter6.tex

\chapter[Overall Concluding Remarks and Future Work]{Overall Concluding Remarks\\ and Future Work}

This dissertation investigated the asymptotic properties of the generalized cyclic stochastic approximation (GCSA) algorithm (Algorithm \ref{findme}) including the convergence (w.p.1) of the iterates, the asymptotic normality of the normalized iterates, and the asymptotic efficiency of GCSA relative to its non-cyclic counterpart. Next we review some of the contributions of this dissertation. 

The main theorem on convergence, Theorem \ref{thm:hoeshoo}, provided sufficient conditions for the convergence w.p.1 of the GCSA iterates to a zero of the gradient of the function to minimize. Although some of the convergence conditions closely resemble well-known conditions for the convergence of the standard (i.e., non-cyclic) SA algorithms, a few conditions are inherent only to the GCSA algorithm (see Section \ref{sec:discussconvergence} for a discussion on the validity of the convergence conditions in Theorem \ref{thm:hoeshoo}). Section \ref{sec:convergencenumericos} contains numerical examples illustrating the convergence of GCSA using both SG- and SPSA-based gradient estimates.

With the goal of computing the rate of convergence of the GCSA algorithm, a generalization to Theorem 2.2 in Fabian (1968)\nocite{fabian1968} (regarding the asymptotic normality of SA procedures) was provided in Chapter \ref{sec:ROC}  (see Theorem \ref{thm:generalizefabian}). Theorem \ref{thm:fnogg} then used the generalization from Theorem \ref{thm:generalizefabian} to derive conditions for the asymptotic normality of the normalized iterates of a special case of GCSA (which helps define the rate of convergence of the algorithm) as well as to obtain expressions for the parameters of its asymptotic distribution (see Section \ref{sec:bonestrailseason} for a discussion on the conditions of Theorem \ref{thm:fnogg} for asymptotic normality).  A few other applications of our generalization to Theorem 2.2 in Fabian (1968)\nocite{fabian1968} are discussed in Appendix \ref{sec:fabiansecgeneralize}, including an application to an adaptive SA algorithm.  Numerical examples supporting the theory on asymptotic normality are given in Section \ref{sec:normalitycenumericos}.  

Using the result on asymptotic normality from Theorem \ref{thm:fnogg}, Chapter \ref{chap:imefficient} provided an analytical estimate for the asymptotic efficiency (in the MSE sense) of a special case of GCSA relative to the efficiency of its non-cyclic counterpart after taking into consideration the cost of implementing each algorithm. It was shown that, in general, either algorithm may be (asymptotically) more efficient than the other (it is important to keep in mind, however, that the relative efficiency in Chapter \ref{chap:imefficient} is only one possible way to compare the cyclic and non-cyclic algorithms). Section \ref{sec:Numericssimples} contains numerical experiments estimating the aforementioned relative efficiency under different definitions of cost when the update directions are either SG- or SPSA-based.

 While the numerical examples in Chapter \ref{chap:numericschap} focused on cases where the conditions for convergence and/or asymptotic normality were met, Chapter \ref{chap:multiagent} investigated the numerical performance of a cyclic SA algorithm applied to a zero-communication multi-agent optimization problem where the loss function is time-varying. As a result, the theory of Chapters \ref{sec:cyclicseesaw}--\ref{chap:imefficient} (which assumes that the loss function is not time-varying) does not apply. It was observed that the performance of the cyclic SA approach had a strong correlation with the quality of the agents' sensors: the better the sensors' quality, the better the cyclic algorithm performed (this behavior is not surprising and, more importantly, is not specific to the cyclic algorithm).
 
 We end this chapter by giving a few topics for future research.
%
One natural direction for future research pertains to the convergence of a constrained variant of GCSA. A possible approach to constructing a constrained version of GCSA could involve projecting $\hat{\bm{\uptheta}}_k$ at every step (we refer the reader to the comment on p. \pageref{page:onconditionA4444} regarding condition A4)\nocite{ISSO}. Another approach could consist on constructing a variant of GCSA for which the constraints are satisfied asymptotically (e.g., Wang and Spall 2011, where the iterates of the algorithm asymptotically satisfy the requirement of lying in a discrete set)\nocite{dspsa}. The work of Wang and Spall (2008)\nocite{otherwang} presents an SA algorithm based on the penalty functions method  for solving stochastic optimization problems with general inequality constraints, the ideas of this paper may also be of interest when considering constrained versions of the GCSA algorithm.

 A second avenue for future research involves the generalization of the GCSA algorithm to an asynchronous setting (i.e., parallelizing GCSA), this could possibly be done using ideas similar to those in Tsitsiklis (1994)\nocite{tsitsiklis1994}. Additionally, it would also be interesting to study the behavior of cyclic procedures for time-varying loss functions (many multi-agent optimization problems, such as the problem in Chapter \ref{chap:multiagent}, have a time-varying loss function). 

A few other topics for future research are the following: 
\begin{enumerate}
\item Can the boundedness assumption of condition A1 (see p. \pageref{cond:conditiona1}) be weakened? (Answering this question may prove useful in when obtaining practical conditions for the convergence of an asynchronous variant of GCSA.)
\item   Borkar and Meyn (2000) show that:
\begin{quote}
. . . the ODE method can be extended to establish both the stability and convergence of the stochastic approximation method, as opposed to only the latter. (Borkar and Meyn 2000)
\end{quote}
Might the ideas of the paper by Borkar and Meyn (2000)\nocite{borkarmeyn2000} be used to extend the theory of convergence of the GCSA iterates in an analogous manner?
\item Could iterate averaging be combined with Algorithm \ref{beastwasdone} (see p. \pageref{beastwasdone}) so that the iterates of the resulting algorithm are asymptotically normally distributed (after an appropriate centering and scaling)?
\item What finite-time performance guarantees can be obtained for the GCSA algorithm (e.g., in terms of bounds on the MSE as a function of the iteration number)?
\item As discussed in Sections \ref{sec:costoport} and \ref{sec:aseficwiener}, the cost of implementation is an important aspect to consider when comparing a cyclic algorithm to its non-cyclic counterpart. Future work might consist of a more detailed investigation into how the cost of implementing a cyclic algorithm compares to the cost of implementing a non-cyclic algorithm.
\end{enumerate}
In summary, this dissertation addressed several important (and previously unanswered) questions regarding cyclic implementations of SA procedures. Still, several interesting questions remain unanswered (as indicated in the discussion above). It is expected that some of the directions for future work mentioned above will make heavy use of the results and ideas in this dissertation.

%% file: AppendixFabian.tex
\chapter{A Generalization of Fabian (1968) and Applications}
\label{sec:fabiansecgeneralize}
\chaptermark{A Generalization of Fabian and Applications}

Stochastic approximation (SA) is a general framework for analyzing the convergence of a large collection of stochastic root-finding algorithms. The Kiefer--Wolfowitz and stochastic gradient algorithms are two well known (and widely used) examples of SA. Because of their applicability to a wide range of problems, many results have been obtained regarding the convergence properties of SA procedures. One important reference in the literature, Fabian (1968)\nocite{fabian1968}, derives general conditions for the asymptotic normality of the SA iterates. Since then, many results regarding the asymptotic normality of SA procedures have relied heavily on Theorem 2.2  in Fabian (1968)\nocite{fabian1968} (which we refer to as ``Fabian's theorem''). Unfortunately, some of the assumptions of Fabian's theorem are not applicable to some modern implementations of SA in control and learning (Section  \ref{subsec:primadonna}, for example, explained why Fabian's theorem could not be applied to the GCSA algorithm). This chapter explains in detail the nature of this incompatibility and shows how Fabian's theorem can be generalized to address the issue.{\footnote{This chapter is largely based on the paper by Hernandez and Spall (2017)\nocite{hernandeznspall2017}.}} While the main theorem in this chapter has already been presented in Chapter \ref{sec:ROC} (the main result in this chapter is the same as Theorem \ref{thm:generalizefabian} in Chapter \ref{sec:ROC}), the theorem was previously presented in the context of cyclic SA. This chapter discusses other applications of the theorem for a more general class of SA algorithms (due to an effort to make this chapter somewhat self-contained, there is some overlap between this chapter and Chapter \ref{sec:ROC}). Furthermore, this appendix contains a proof of the generalization to Fabian's theorem that is more detailed than the proof provided in Section \ref{sec:genfabiangenfabian}.

 The main result in Fabian (1968)\nocite{fabian1968}, which we will refer to as ``Fabian's theorem,'' is concerned with the following recursion:
\begin{align}
\label{eq:moon8}
\bm{W}_{k+1}=(\bm{I}-k^{-\upalpha}\bm\Gamma_k)\bm{W}_k+\frac{\bm{T}_k}{k^{\upalpha+\upbeta/2}}+\frac{\bm\Phi_k\bm{V}_k}{k^{(\upalpha+\upbeta)/2}},
\end{align}
for $\upalpha, \upbeta>0$, random random vectors $\bm{W}_k, \bm{T}_k, \bm{V}_k \in \mathbb{R}^p$, and random matrices $\bm\Phi_k, \bm\Gamma_k\in \mathbb{R}^{p\times p}$. Under certain conditions, Fabian (1968)\nocite{fabian1968} shows that $k^{\upbeta/2}\bm{W}_k$ is asymptotically normally distributed (with mean vector and covariance matrix determined by $\upalpha$, $\upbeta$, $\bm{T}_k$, $\bm{V}_k$, $\bm\Phi_k$, and $\bm\Gamma_k$. Undoubtedly, Fabian's theorem is already applicable to a variety of SA algorithms (see, for example, Zhou and Hu 2014\nocite{zhounhu2014},  Kar et al. 2013\nocite{karetal2013}, Hu et al. 2012\nocite{huetal2012}, and Zorin et al. 2000\nocite{zorinetal200} to name a few applications). However, a critical assumption in Fabian's theorem is that $\bm\Gamma_k\rightarrow \bm\Gamma$ with probability one (w.p.1) for some real, positive definite matrix $\bm\Gamma$. This introduces an important limitation, the nature of which we review next.
 
  Let us begin by describing the standard SA recursion. Given a vector $\bm\uptheta\in \mathbb{R}^p$ and a vector-valued function $\bm{f}(\bm\uptheta)\in \mathbb{R}^n$ the basic SA algorithm for solving $\bm{f}(\bm\uptheta)=\bm{0}$ is given by the recursion:
\begin{align}
\label{eq:formof8}
\hat{\bm{\uptheta}}_{k+1}=\hat{\bm{\uptheta}}_k-a_k{\bm{Y}}_k(\hat{\bm{\uptheta}}_k),
\end{align}
where ${\bm{Y}}_k(\bm\uptheta)$ is a vector-valued random variable representing a noisy observation of $\bm{f}(\hat{\bm{\uptheta}}_k)$, and $a_k>0$ with $a_k\rightarrow 0$ is the {{gain sequence}} of the algorithm. Under certain conditions (which vary depending on the specific SA algorithm considered) we have $\hat{\bm{\uptheta}}_k\rightarrow \bm\uptheta^\ast$ w.p.1, where $\bm\uptheta^\ast$ is a zero of $\bm{f}(\bm\uptheta)$. Letting $\bm{W}_k=\hat{\bm{\uptheta}}_k-\bm\uptheta^\ast$ in (\ref{eq:moon8}), Fabian's theorem can be used to derive conditions under which $k^{\upbeta/2}(\hat{\bm{\uptheta}}_k-\bm\uptheta^\ast)$ has a limiting multivariate normal distribution. 
For some SA algorithms, however, requiring $\bm\Gamma$ to be {\it{symmetric}} (which is necessary in order for $\bm\Gamma$ to be real and positive definite) is incompatible with Fabian's other assumptions regarding $\upalpha$, $\upbeta$, $\bm{T}_k$, $\bm{V}_k$, and $\bm\Phi_k$. Specifically, for some SA algorithms it impossible to find a parametrization of (\ref{eq:moon8}) such that $\bm\Gamma$ is symmetric and all of Fabian's assumptions on $\upalpha$, $\upbeta$, $\bm{T}_k$, $\bm{V}_k$, $\bm\Phi_k$, and $\bm\Gamma_k$ also hold. Consequently, Fabian's theorem cannot be applied. The following are examples of such algorithms:
\begin{enumerate}
\item SA algorithms for root-finding where $\bm{f}(\bm\uptheta)$ does not represent a gradient 
so that the Jacobian of $\bm{f}(\bm\uptheta)$ is typically non-symmetric (e.g., Blum 1954)\nocite{blum1954}.
\item Second order or adaptive SA algorithms for stochastic optimization where $\bm{f}(\bm\uptheta)$ is the gradient of a function to minimize and $a_k$ is replaced by a diagonal matrix that is not necessarily a multiple of the identity matrix (e.g., Renotte and Wouwer 2003)\nocite{renottenwouwer2003}.
\item The Generalized Cyclic SA (GCSA) algorithm in which only a subset of the parameter vector is updated at any given time. The subvector to update can be selected following a deterministic pattern (Hernandez and Spall 2014\nocite{hernandeznspall2014} and 2016\nocite{hernandeznspall2016}) or according to a random variable (Hernandez 2016)\nocite{hernandez2016}.
\item Randomized coordinate descent algorithms where some coordinates may be updated with a higher probability than others (e.g., Nesterov 2010).
\end{enumerate}
 The main contribution of this work is to provide a theorem that relaxes Fabian's restriction that $\bm\Gamma$ must be symmetric. The generalization makes the theorem more applicable to some modern applications of SA in control and learning including special cases of examples 1--4 above.

It is important to note that there do exist results showing the asymptotic normality of certain classes of SA procedures for which $\bm\Gamma$ may not be a symmetric matrix. In general, however, explicit expressions for the parameters of the asymptotic distribution are only available for specific SA algorithms.  For the special case where $a_k=1/k$ with $k\geq 1$, for example, Nevel'son and Has'minskii (1973, Chapter 6)\nocite{nevelsonandhas1973} show asymptotic normality for the Robbins--Monro procedure (Robbins and Monro 1951)\nocite{robbinsmonro1951} where $\bm\Gamma$ is not required to be symmetric and is allowed to have complex eigenvalues. The authors give a closed-form expression for the mean vector and covariance matrix of the asymptotic distribution.  

One common class of results closely related to asymptotic normality are the stochastic differential equation (SDE)-based definitions of rate of convergence for SA algorithms (e.g., Chapter 10 in Kushner and Yin 1997,\nocite{kushnyin1997}  Chapter 4 in part II of Benveniste et al. 1990\nocite{benveniste1990}, and Chapter 7 in Kushner and Clark 1978)\nocite{kushnclark1978}. Here, the authors first normalize the SA iterates and rewrite the resulting normalized iterates as a continuous time-dependent processes. The authors then show that the normalized continuous process converges weakly to an SDE and the rate of convergence is defined as the inverse of the scaling coefficient when normalizing the SA iterates. In certain cases, such as when the resulting SDE corresponds to a stationary Gauss--Markov process, these results imply the asymptotic normality of the normalized SA iterates and $\bm\Gamma$ is not required to be symmetric. Generally, however, the connection between asymptotic normality and the SDE-based definition of rate of convergence is not explicitly made. Moreover, in the instances where a connection to asymptotic normality is made, the parameters of the asymptotic distribution are not typically provided. Benveniste et al. (1990, Part II, Section 4.5, Theorem 13)\nocite{benveniste1990}, for example, use the SDE-based method to show asymptotic normality of normalized SA iterates where $\bm\Gamma$ is not required to be symmetric and may have complex eigenvalues. However,  to find the covariance matrix of the asymptotic distribution one must solve the Lyapunov equation on p. 334 of the same reference. We also note that Benveniste's result
 is limited to the case where the mean of the asymptotic distribution is zero (e.g., does not include the algorithms from Section \ref{sec:coffee100}).

 A few attractive features of our proposed generalization are that 1) the result is not restricted to any specific SA algorithm,  
 2) aside from guaranteeing asymptotic normality, we also present the explicit solution for the asymptotic mean and covariance matrix that is not generally provided in the SDE-based approach above, and 3) our generalization makes Fabian's theorem applicable to some SA algorithms for which asymptotic normality 
 has not been proven.


\section{Preliminaries}
\label{sec:prems}

While we refer the reader to Theorem \ref{thm:fabian} of this dissertation for the full set of conditions of Fabian's theorem, a few key assumptions are that $\bm\Gamma_k$ must converge to a real positive definite (and therefore symmetric) matrix $\bm\Gamma$ w.p.1, $\bm{T}_k$ must converge to some vector $\bm{T}$ either w.p.1 (with probability one) or in expectation, $\bm{\Phi}_k$ must converge to $\bm\Phi$ w.p.1, $\bm{V}_k$ must have mean zero (conditionally on $\mathcal{F}_k$), and the covariance matrix of $\bm{V}_k$ (conditional on $\mathcal{F}_k$) must be uniformly bounded over $k$ and $\mathcal{F}_k$ and must converge (w.p.1) to some positive definite matrix $\bm\Sigma$. When $\bm{f}(\bm\uptheta)=\bm{g}(\bm\uptheta)$ is the gradient of a function to minimize, say $L(\bm\uptheta)$, the assumption that $\bm\Gamma=\lim_{k\rightarrow \infty}\bm\Gamma_k$ may be reasonable. To see this, note that if $\bm{g}(\bm\uptheta)$ is continuously differentiable then (\ref{eq:formof8}) may be written in the form of (\ref{eq:moon8}) by letting $\bm\Gamma_k=k^\upalpha a_k\tilde{\bm{H}}_k$, where the $i$th row of $\tilde{\bm{H}}_k$ is equal to the $i$th row of the Hessian matrix, $\bm{H}(\bm\uptheta)$,  of $L(\bm\uptheta)$ evaluated at $\bm\uptheta=(1-\uplambda_i)\hat{\bm{\uptheta}}_k+\uplambda_i\bm\uptheta^\ast$ for some $\uplambda_i\in[0,1]$ that depends on $\hat{\bm{\uptheta}}_k$.  If $\bm{H}(\bm\uptheta)$ is continuous at $\bm\uptheta^\ast$, $\hat{\bm{\uptheta}}_k\rightarrow \bm\uptheta^\ast$ w.p.1, and $k^\upalpha a_k\rightarrow a>0$, then requiring $\bm\Gamma$ to be real and positive definite translates into $a\bm{H}(\bm\uptheta^\ast)$ being real and positive definite, a common assumption in many minimization problems.
Suppose now that a modification is made to (\ref{eq:moon8}) in which $\bm{f}(\bm\uptheta)=\bm{g}(\bm\uptheta)$ and $a_k$ is replaced by a diagonal matrix $\bm{A}_k$ (e.g., Renotte and Wouwer 2003 where $\bm{A}_k$ is taken to be a diagonal matrix with the same diagonal entries as an estimate of $a_k\bm{H}(\hat{\bm{\uptheta}}_k)^{-1}$). Here, since $\bm{Y}_k(\hat{\bm{\uptheta}}_k)$ denotes a noisy estimate of the gradient it is common to replace the notation $\bm{Y}_k(\hat{\bm{\uptheta}}_k)$ in (\ref{eq:formof8}) with $\hat{\bm{g}}_k(\hat{\bm{\uptheta}}_k)$.
 This modification gives rise to the algorithm:
\begin{align}
\label{eq:willser8}
\hat{\bm{\uptheta}}_{k+1}=\hat{\bm{\uptheta}}_k-\bm{A}_k\hat{\bm{g}}_k(\hat{\bm{\uptheta}}_k).
\end{align}
It is easy to verify that  (\ref{eq:willser8}) is a special case of (\ref{eq:moon8}) in which $\bm\Gamma_k=k^\upalpha \bm{A}_k\tilde{\bm{H}}_k$. Therefore, if $\bm{H}(\bm\uptheta)$ is continuous, $\hat{\bm{\uptheta}}_k\rightarrow \bm\uptheta^\ast$ w.p.1, and $k^\upalpha\bm{A}_k\rightarrow \bm{A}$ w.p.1 then $\bm\Gamma_k\rightarrow\bm\Gamma=\bm{A}\bm{H}(\bm\uptheta^\ast)$ w.p.1. Generally the matrix $\bm{A}\bm{H}(\bm\uptheta^\ast)$ would not be symmetric and, therefore, Fabian's theorem would not be applicable. Because there is no {\it{unique}} way to define the variables $\bm\Gamma_k$, $\bm{T}_k$, $\bm{\Phi}_k$, and $\bm{V}_k$ in (\ref{eq:moon8}), it may be tempting to think that it is always possible to redefine these variables so that $\bm\Gamma$ is symmetric. Next we show that such a redefinition often leads to very strong assumptions on $\hat{\bm{\uptheta}}_k$.

Suppose an SA algorithm can be written in the form of (\ref{eq:moon8}) and that all of the conditions of Fabian's theorem are satisfied {\it{with the exception that $\bm\Gamma$ is not symmetric}}. For simplicity, we consider a special case of (\ref{eq:moon8}) where $\upbeta\neq 0$, $\bm{T}_k=\bm{T}$, $\bm\Phi_k=\bm{I}$, and $\bm\Gamma_k=\bm\Gamma$. Now, assume the matrices $\bm\Gamma_k'$, $\bm\Phi_k'$ and vectors $\bm{T}_k'$, $\bm{V}_k'$ provide an alternative way to write the SA algorithm in the form of (\ref{eq:moon8}) and assume $\bm\Gamma_k'$ converges w.p.1 to a symmetric matrix. Then, since
\begin{align*}
\bm{W}_{k+1}=&\ (\bm{I}-k^{-\upalpha}\bm\Gamma_k')\bm{W}_k+\frac{\bm{T}}{k^{\upalpha+\upbeta/2}}+\frac{\bm{V}_k}{k^{(\upalpha+\upbeta)/2}}+k^{-\upalpha}(\bm\Gamma_k'-\bm\Gamma)\bm{W}_k,
\end{align*}
 it is known that either $\bm{T}_k'$ must depend on $(\bm\Gamma'_k-\bm\Gamma)\bm{W}_k$ or $\bm\Phi_k'\bm{V}_k'$ must depend on $(\bm\Gamma_k'-\bm\Gamma)\bm{W}_k$. However, the assumptions on $\bm{T}_k'$, $\bm\Phi_k'$, and $\bm{V}_k'$ are typically incompatible with the term $k^{\upbeta/2}(\bm\Gamma'_k-\bm\Gamma)\bm{W}_k$. Say, for example, that we let $\bm\Phi_k'=\bm{I}$, $\bm{V}_k'=\bm{V}_k$, and $\bm{T}_k'=\bm{T}+k^{\upbeta/2}(\bm\Gamma'_k-\bm\Gamma)\bm{W}_k$. Here, having $\bm{T}_k'$ converge to some finite vector $\bm{T}'$ w.p.1 or in expectation (as required by Fabian's theorem) would impose a priori conditions on the stochastic rate at which $\hat{\bm{\uptheta}}_k$ converges to $\bm\uptheta^\ast$.  However, such a condition violates the very purpose of Fabian's theorem, which is to establish such a rate of convergence. Alternatively, having $\bm\Phi_k'\bm{V}_k'=\bm{V}_k+k^{(\upbeta-\upalpha)/2}(\bm\Gamma_k'-\bm\Gamma)\bm{W}_k$ does not lead to an appropriate definition of $\bm\Phi_k'$ and $\bm{V}_k'$ given the restriction that $\bm{V}_k'$ must have mean zero conditionally on $\mathcal{F}_k$. By generalizing Fabian's conditions to relax the symmetry condition on $\bm\Gamma$ we avoid the need for imposing any additional restrictions on $\hat{\bm{\uptheta}}_k$. Next we give a few other examples of algorithms for which $\bm\Gamma$ cannot be assumed to be symmetric.
 
Another algorithm for which $\bm\Gamma$ may not be symmetric is a variant of the adaptive stochastic approximation method for stochastic optimization in Spall (2000). Given a function $L(\bm\uptheta)$ to minimize, the algorithm in Spall (2000) is as follows:
\begin{align}
&\hat{\bm{\uptheta}}_{k+1}=\hat{\bm{\uptheta}}_k-a_k\overline{\overline{\bm{H}}}_k^{-1}\hat{\bm{g}}_k(\hat{\bm{\uptheta}}_k),\notag\\
&\overline{\bm{H}}_k\equiv\frac{k}{k+1}+\overline{\bm{H}}_{k-1}+\frac{1}{k+1}\hat{\bm{H}}_k,\notag\\
 &\overline{\overline{\bm{H}}}_k\equiv \bm{\uppi}_k(\overline{\bm{H}}_k),\label{eq:presidents8}
\end{align}
where $\hat{\bm{H}}_k$ is an estimate of the Hessian of $L(\bm\uptheta)$ evaluated at $\hat{\bm{\uptheta}}_k$ and $\bm{\uppi}_k({\overline{\bm{H}}}_k)$ maps ${\overline{\bm{H}}}_k$ onto the set of positive definite matrices. When the dimension of $\overline{\bm{H}}_k$ is high, it may be advantageous to let $\overline{\overline{\bm{H}}}_k$ be a diagonal matrix. However, since $\bm\Gamma_k=k^\upalpha a_k\overline{\overline{\bm{H}}}_k^{-1}\tilde{\bm{H}}_k$ for algorithm (\ref{eq:presidents8}) (see Spall 2000, Theorems 3a and 3b) then  the requirement that $\bm\Gamma$ must be symmetric generally prevents us from using Fabian's theorem when $\bm\uppi_k$ is a mapping onto the set of diagonal matrices (here we note that the conditions for asymptotic normality given in Spall 2000, which are based on Fabian's theorem, require that $\overline{\overline{\bm{H}}}_k$ converge to $\bm{H}(\bm\uptheta^\ast)$ precisely so that $\bm\Gamma$ is symmetric and this condition also prevents $\overline{\overline{\bm{H}}}_k$ from being a diagonal matrix unless $\bm{H}(\bm\uptheta^\ast)$ is also a diagonal matrix).
%
%
 Requiring $\bm\Gamma$ to be symmetric is also an unreasonable assumption for the cyclic seesaw SA algorithm where $\bm\Gamma_k=k^\upalpha \bm{A}_k\tilde{\bm{H}}_k$ as in the case of algorithm (\ref{eq:willser8}) (see (6) in Hernandez 2016). In cases where SA is used for general root-finding, the symmetry of $\bm\Gamma$ is an especially unreasonable assumption. Here, 
  $\bm\Gamma_k=k^\upalpha a_k\tilde{\bm{J}}_k$ where the $i$th row of $\tilde{\bm{J}}_k$ is equal to the $i$th row of the Jacobian of $\bm{f}(\bm\uptheta)$ evaluated at $\bm\uptheta=(1-\uplambda_i)\hat{\bm{\uptheta}}_k+\uplambda_i\bm\uptheta^\ast$ for some $\uplambda_i\in[0,1]$.
 
 In the proof of Fabian's theorem, the symmetry of $\bm\Gamma$ is used to write $\bm\Gamma=\bm{P\Lambda P}^\top$ for a real orthogonal matrix $\bm{P}$ and a real diagonal matrix $\bm\Lambda$ with strictly positive eigenvalues. Both $\bm{P}$ and $\bm\Lambda$ affect the parameters of the limiting distribution of $k^{\upbeta/2}(\hat{\bm{\uptheta}}_k-\bm\uptheta^\ast)$. When $\bm\Gamma$ is not a symmetric matrix, writing $\bm\Gamma=\bm{P\Lambda P}^\top$ is clearly not possible. This is a problem for all previous examples where $\bm\Gamma$ can reasonably be assumed to satisfy $\bm\Gamma=\bm{A}\bm{M}$ for a positive definite diagonal matrix $\bm{A}$ and a matrix $\bm{M}$ which represented either a Hessian or a Jacobian. Although this $\bm\Gamma$ is not generally symmetric, it may still be real diagonalizable and have strictly positive eigenvalues (i.e., $\bm\Gamma=\bm{S}\bm{\Lambda}\bm{S}^{-1}$ for a nonsingular, not necessarily orthogonal, real matrix $\bm{S}$ and a positive definite diagonal matrix $\bm{\Lambda}$). Such is the case, for example, if $\bm{M}=\bm{H}(\bm\uptheta^\ast)$ is real and positive definite. 
  For nonlinear root-finding SA algorithms, however, $\bm{M}=\bm{J}(\bm\uptheta^\ast)$ represents the Jacobian of $\bm{f}(\bm\uptheta)$ evaluated at $\bm\uptheta^\ast$. Here, it is possible for $\bm\Gamma=\bm{A}\bm{J}(\bm\uptheta^\ast)$ not to be diagonalizable even if $\bm{J}(\bm\uptheta^\ast)$ is a square matrix. In the special case where $\bm\Gamma$ is a square matrix, however, $\bm\Gamma=\bm{S}\bm{U}\bm{S}^{-1}$ for a nonsingular real matrix $\bm{S}$ and a real upper triangular matrix $\bm{U}$ (here it is said that $\bm\Gamma$ is real upper-triangularizable) with strictly positive diagonal entries if and only if $\bm\Gamma$ is real and has strictly positive eigenvalues (see, for example, Horn and Johnson p. 82)\nocite{hornnjohnson2010}. Generalizing Fabian's result to allow $\bm\Gamma$ to be any real upper-triangularizable matrix with strictly-positive eigenvalues would then allow for treatment of the aforementioned ``non-standard'' algorithms, which is precisely the generalization we introduce here. 

\section{The Generalized Theorem}
\label{sec:ourgeneralization}
This section contains a generalization to Fabian's theorem derived by replacing Fabian's Assumption 2.2.1 with a weaker assumption. Following the proof of Fabian's theorem (Fabian 1968)\nocite{fabian1968}, we begin by showing that $k^{\upbeta/2}\bm{W}_k$ is asymptotically normally distributed if and only if a much simpler process is also asymptotically normally distributed. After showing that the simpler process does, in fact, converge in distribution to a multivariate normal random variable, the parameters of its asymptotic distribution will uniquely determine the parameters of the asymptotic distribution of $k^{\upbeta/2}\bm{W}_k$.

  Throughout this Chapter $M_{ij}$ and $M_{k(i,j)}$ denote the $(i,j)$th entries of the matrices $\bm{M}$ and $\bm{M}_k$, respectively, and ${v}_i$ and $v_{k(i)}$ denote the $i$th entries of the vectors $\bm{v}$ and $\bm{v}_k$, respectively. Furthermore, $k\geq 1$ denotes a strictly positive integer; $\xrightarrow{\text{\ dist\ }}$ means convergence in distribution; $\bm{V}_k$, $\bm{W}_k$, $\bm{T}_k$, and $\bm{T}$ are vectors in $\mathbb{R}^p$; $\bm\Gamma_k$, $\bm\Phi_k$, $\bm\Sigma$, $\bm\Gamma$, $\bm\Phi$, and $\bm{P}$ are matrices in $\mathbb{R}^{p\times p}$; and $\mathcal{N}(\bm\upmu,\bm{M})$ denotes a multivariate normal random variable with mean $\bm\upmu$ and covariance $\bm{M}$. Furthermore, throughout this section we assume the recursion for ${\bm{W}}_k$ given in (\ref{eq:moon8}) satisfies the following conditions:
\begin{DESCRIPTION}
\item[B0$'$]   $\bm\Gamma_k$, $\bm\Phi_{k-1}$, and $\bm{V}_{k-1}$ are $\mathcal{F}_k$-measurable where $\mathcal{F}_k$ is a non-decreasing sequence of $\upsigma$-fields with $\mathcal{F}_k\subset \mathcal{S}$ for some $\upsigma$-field $\mathcal{S}$.
\item[B1$'$] There exists an upper triangular $\bm{U}\in \mathbb{R}^{p\times p}$ with strictly positive eigenvalues and a nonsingular $\bm{S}\in \mathbb{R}^{p\times p}$ such that $\bm\Gamma_k\rightarrow \bm\Gamma=\bm{S}\bm{U}\bm{S}^{-1}$ w.p.1.
\item[B2$'$] $\bm\Phi_k\rightarrow \bm\Phi$ w.p.1.
\item[B3$'$] Either $\bm{T}_k\rightarrow \bm{T}$ w.p.1 or $E\|\bm{T}_k-\bm{T}\|\rightarrow 0$.
\item[B4$'$] $E[\bm{V}_k|\mathcal{F}_k]=\bm{0}$, there exists a constant $C>0$ such that $C>\|E[\bm{V}_k\bm{V}_k^\top|\mathcal{F}_k]-\bm\Sigma\|$, and $\|E[\bm{V}_k\bm{V}_k^\top|\mathcal{F}_k]-\bm\Sigma\|\rightarrow 0$ w.p.1.
\item[B5$'$] For some $0< \upalpha\leq 1$ define $\upsigma_{k,r}^2\equiv E\chi\{\|\bm{V}_k\|^2\geq rk^\upalpha\}\|\bm{V}_k\|^2$ where $\chi\{\mathcal{E}\}$ is the indicator function of the event $\mathcal{E}$. For every $r>0$ either:
\begin{align*}
{\text{$\lim_{k\rightarrow \infty}\upsigma_{k,r}^2=0$ or $\upalpha=1$ and $\lim_{n\rightarrow \infty}n^{-1}\sum_{k=1}^n\upsigma^2_{k,r}=0$.}}
\end{align*}
\item[B6$'$] Define $\uplambda\equiv \min_i{\{U^{[i,i]}\}}$. Let $\upalpha$ and $\upbeta$ be constants such that $0<\upalpha\leq 1$ and $0\leq \upbeta$. Define $\upbeta_+\equiv\upbeta$ if $\upalpha=1$ and $\upbeta_+\equiv 0$ otherwise. Then, $\upbeta_+<2\uplambda$.
\end{DESCRIPTION}
Conditions B0$'$ and B2$'$--B6$'$ are identical to the conditions of Fabian's theorem (note that B6$'$ is obtained by rewriting B6 in the notation of B1$'$). On the other hand, B1$'$ is the relaxed version of Fabian's  corresponding condition (condition 2.2.1 in Fabian 1968) requiring symmetry of $\bm\Gamma$. As discussed in the comment following Proposition \ref{prop:sandpipercrossing}, any real square matrix with real eigenvalues satisfies B1$'$. 
Next we use conditions B0$'$--B6$'$ to relate the asymptotic distribution of $k^{\upbeta/2}\bm{W}_k$ to that of a much simpler process. 

In order to compute the asymptotic distribution of $k^{\upbeta/2}\bm{W}_k$ we begin by constructing a slightly different process, $\tilde{\bm{W}}_k$, defined as follows:
\begin{align}
\label{eq:linkssfton8}
\tilde{\bm{W}}_{k}=(k-1)^{\upbeta/2}\bm{S}^{-1}\bm{W}_k,
\end{align}
where $\bm{S}$ is the matrix from B1$'$. It is easy to see that $\bm{W}_1=\bm{0}$. Moreover, using Slutsky's theorem it follows that if $\tilde{\bm{W}}_k\xrightarrow{\text{\ dist\ }}\mathcal{N}(\bm\upmu,\bm{M})$ 
then $k^{\upbeta/2}\bm{W}_k\xrightarrow{\text{\ dist\ }}\mathcal{N}(\bm{S}\bm\upmu,\bm{SM S}^\top)$.
Thus, proving that the process $\tilde{\bm{W}}_k$ is asymptotically normally distributed (with certain mean vector and covariance matrix) is sufficient for computing the asymptotic distribution of $k^{\upbeta/2}\bm{W}_k$.
%
Moreover, after some algebraic manipulation it can be shown that $\tilde{\bm{W}}_k$ is a special case of (\ref{eq:moon8}):
\begin{align}
\label{eq:graham38848}
 \tilde{\bm{W}}_{k+1}=(\bm{I}-k^{-\upalpha}\tilde{\bm{\Gamma}}_k)\tilde{\bm{W}}_k+k^{-\upalpha}\bm{S}^{-1}\bm{T}_k+k^{-\upalpha/2}\bm{S}^{-1}\bm\Phi_k\bm{V}_k,
\end{align}
where $\tilde{\bm{W}}_{1}=\bm{0}$, 
\begin{align*}
\tilde{\bm\Gamma}_k\equiv \left(\frac{k}{k-1}\right)^{\upbeta/2}\bm{S}^{-1}\bm\Gamma_k\bm{S}-\left[\left(\frac{k}{k-1}\right)^{\upbeta/2}-1\right]k^\upalpha\bm{I},
\end{align*}
and $\tilde{\bm\Gamma}_k\rightarrow \tilde{\bm{U}}\equiv \bm{U}-(\upbeta_+/2)\bm{I}$ w.p.1. Using the same arguments as those in Fabian (1968, proof of Theorem 2.2\nocite{fabian1968}) it can be shown that replacing ${\bm{T}}_k$ with ${\bm{T}}$ and $\bm\Phi_k$ with $\bm\Phi$  in (\ref{eq:graham38848}) does not change the asymptotic distribution of $\tilde{\bm{W}}_k$ (note that the recursive nature of $\tilde{\bm{W}}_k$ means this is not immediate). Therefore, in order to show that $\tilde{\bm{W}}_k$ is asymptotically normally distributed we may assume, without loss of generality, that ${\bm{T}}_k={\bm{T}}$ and $\bm\Phi_k=\bm\Phi$ so that:
\begin{align}
\label{eq:jaredboothsssd8}
 \tilde{\bm{W}}_{k+1}=(\bm{I}-k^{-\upalpha}\tilde{\bm{\Gamma}}_k)\tilde{\bm{W}}_k+k^{-\upalpha}\tilde{\bm{T}}+k^{-\upalpha/2}\tilde{\bm{V}}_k,
\end{align}
where $\tilde{\bm{T}}\equiv \bm{S}^{-1}\bm{T}$ and $\tilde{\bm{V}}_k\equiv \bm{S}^{-1}\bm{\Phi V}_k$. The following lemma facilitates future analysis by relating the asymptotic distribution of $\tilde{\bm{W}}_k$, as described by (\ref{eq:jaredboothsssd8}), to that of an even simpler process. 

\begin{lemma}
\label{claim:coughsyrup8}
Consider the following process:
\begin{align}
\label{eq:norahjoness7728}
 \tilde{\bm{W}}_{k+1}'=(\bm{I}-k^{-\upalpha}\tilde{\bm{\Gamma}}_k)\tilde{\bm{W}}_k'+k^{-\upalpha/2}\tilde{\bm{V}}_k,
\end{align}
where $\tilde{\bm{W}}_1'\equiv\bm{0}$. Assume the process  ${\bm{W}}_k$ in (\ref{eq:moon8}) satisfies B0$'$--B6$'$. If $\tilde{\bm{W}}_{k}'\xrightarrow{\text{\ dist\ }}\mathcal{N}(\bm\upmu,\bm{M})$ 
then $\tilde{\bm{W}}_{k}\xrightarrow{\text{\ dist\ }}\mathcal{N}(\bm\upmu+\bm\upnu,\bm{M})$,
where
 \begin{align}
 \label{eq:okayfine8}
 \upnu_i\equiv (\tilde{{U}}_{ii})^{-1}\tilde{{T}}_i-\Bigg[(\tilde{{U}}_{ii})^{-1}\sum_{j=i+1}^p \tilde{U}_{ij}\upnu_j\Bigg]
 \end{align}
\end{lemma}

\begin{proof}
 The result of the lemma holds if $\tilde{\bm{W}}_k-\tilde{\bm{W}}_k'\rightarrow \bm\upnu$ w.p.1. We show this next. 
First, using the triangle inequality it follows that
\begin{align}
\label{eq:glights8}
\| \tilde{\bm{W}}_{k+1}-\tilde{\bm{W}}_{k+1}'\|&\leq(1-k^{-\upalpha}[\lambda-\upbeta_+/2+o(1)])\|\tilde{\bm{W}}_k-\tilde{\bm{W}}_k'\|+\|k^{-\upalpha}\tilde{\bm{T}}\|,
\end{align}
where $o(\cdot)$ is the standard little-{\it{o}} notation.
Then, (\ref{eq:glights8}) and Lemma 4.2 in Fabian (1967)\nocite{fabian1967} imply $\limsup_{k\rightarrow \infty} \|\tilde{\bm{W}}_k-\tilde{\bm{W}}_k'\|<\infty$ w.p.1 (we write $\|\tilde{\bm{W}}_k-\tilde{\bm{W}}_k'\|=O(1)$ w.p.1 where $O(\cdot)$ is the standard big-{\it{O}} notation). Next, note that
\begin{align}
\label{eq:asfjjjtrees8}
 \tilde{\bm{W}}_{k+1}-\tilde{\bm{W}}_{k+1}'=(\bm{I}-k^{-\upalpha}\tilde{\bm{U}})(\tilde{\bm{W}}_k-\tilde{\bm{W}}_k')+k^{-\upalpha}[\tilde{\bm{T}}+(\tilde{\bm{U}}-\tilde{\bm\Gamma}_k)(\tilde{\bm{W}}_k-\tilde{\bm{W}}_k')].
\end{align}
Then, since $\|\tilde{\bm{W}}_k-\tilde{\bm{W}}_k'\|=O(1)$ w.p.1 and $\tilde{\bm\Gamma}_k\rightarrow \tilde{\bm{U}}$ w.p.1 we may assume each entry of the term $(\tilde{\bm{U}}-\tilde{\bm\Gamma}_k)(\tilde{\bm{W}}_k-\tilde{\bm{W}}_k')$ is $o(1)$ w.p.1. Applying Lemma 4.2 (Fabian 1967)\nocite{fabian1967} to the last entry of $\tilde{\bm{W}}_k-\tilde{\bm{W}}_k'$ as determined by recursion (\ref{eq:asfjjjtrees8}) implies:
\begin{align*}
\tilde{\bm{W}}_{k(p)}-\tilde{\bm{W}}_{k(p)}'\rightarrow \upnu_p=  (\tilde{{U}}_{pp})^{-1}\tilde{{T}}_p{\text{ w.p.1}}.
\end{align*}
This is consistent with (\ref{eq:okayfine8}). In general, Lemma 4.2 in Fabian (1967) can be used to show that $\tilde{\bm{W}}_{k(i)}-\tilde{\bm{W}}_{k(i)}'\rightarrow\upnu_i$ w.p.1 where $\upnu_i$ is as in (\ref{eq:okayfine8}). Thus,  $\upnu_i$ can be computed once the values $\{\upnu_\ell\}_{\ell=i+1}^p$ have been computed (i.e., once the last $p-i$ entries of $\bm\upnu$ have been computed).
\end{proof}

 Lemma \ref{claim:coughsyrup8} implies that deriving the asymptotic distribution of $\tilde{\bm{W}}_{k}'$ is sufficient for deriving the asymptotic distribution of $\tilde{\bm{W}}_{k}$ and, therefore, of $k^{\upbeta/2}\bm{W}_k$.  At this point, the same arguments as those in Fabian (1968, proof of Theorem 2.2)\nocite{fabian1968} can be used to show the following two results under B0$'$--B6$'$:
 \begin{enumerate}
 \item Replacing $\tilde{\bm{\Gamma}}_k$ with $\tilde{\bm{U}}$ in (\ref{eq:norahjoness7728}) does not change the asymptotic distribution of $\tilde{\bm{W}}_k'$ which depends on the limit (w.p.1) of $\tilde{\bm{\Gamma}}_k$ but not on $\tilde{\bm{\Gamma}}_k$ itself. Therefore, without loss of generality we may assume that $\tilde{\bm{\Gamma}}_k=\tilde{\bm{U}}$ so that:
 \begin{align}
\label{eq:utildenew8}
 \tilde{\bm{W}}_{k+1}'=(\bm{I}-k^{-\upalpha}\tilde{\bm{U}})\tilde{\bm{W}}_k'+k^{-\upalpha/2}\tilde{\bm{V}}_k.
\end{align}

 \item The characteristic function of the asymptotic distribution of $\tilde{\bm{W}}_{k}'$ evaluated at $\bm{t}\in \mathbb{R}^p$ depends on $\bm{t}$, $\upalpha$, $\tilde{\bm{U}}$, and on $\tilde{\bm{\Sigma}}\equiv \lim_{k\rightarrow \infty}\text{var}{[\tilde{\bm{V}}_k|\mathcal{F}_k]}$ but does not depend on other aspects of the distribution of $\tilde{\bm{V}}_k$ (provided B0$'$--B6$'$ hold).  Therefore, we may assume that the vectors $\tilde{\bm{V}}_k$ are i.i.d. $\mathcal{N}(\bm{0},\tilde{\bm{\Sigma}})$.
 \end{enumerate}
 The following Lemma derives the asymptotic distribution of $\tilde{\bm{W}}_k'$.

 \begin{lemma}
\label{eq:losorto8}
Assume $\bm{W}_k$ (defined in Theorem \ref{thm:fabian}) satisfies B0$'$--B6$'$. Then, $\tilde{\bm{W}}_k'$ converges in distribution to a multivariate normal random variable with mean zero and covariance matrix $\bm{Q}$ whose entries are the unique solution to:
\begin{align}
\label{eq:nolookback88}
Q_{ij}=\frac{[\bm{S}^{-1}\bm{\Phi\Sigma\Phi}^\top(\bm{S}^{-1})^\top]_{ij}}{\tilde{U}_{ii}+\tilde{U}_{jj}}-\Bigg[\frac{\sum_{\ell=j+1}^{p} \tilde{U}_{j\ell}\ Q_{i\ell}+\sum_{\ell=i+1}^{p} \tilde{U}_{i\ell}\ Q_{\ell j}}{\tilde{U}_{ii}+\tilde{U}_{jj}}\Bigg].
\end{align}
\end{lemma}
\begin{proof}
First, since $\tilde{\bm{V}}_k= \bm{S}^{-1}\bm{\Phi V}_k$ then $\tilde{\bm{\Sigma}}=\bm{S}^{-1}\bm{\Phi\Sigma\Phi}^\top(\bm{S}^{-1})^\top$ by condition B4$'$. Next, by the discussion immediately preceding this lemma we may assume (without loss of generality) that the vectors $\tilde{\bm{V}}_k$ are i.i.d. $\mathcal{N}(\bm{0},\tilde{\bm{\Sigma}})$.
Consequently, for each $k$ the distribution of the random vector $\tilde{\bm{W}}_k'$ from (\ref{eq:utildenew8}) must be a multivariate normal random variable with mean zero and covariance 
 $\bm{Q}_k\equiv E[\tilde{\bm{W}}_{k}\tilde{\bm{W}}_{k}^\top]$. Using (\ref{eq:utildenew8}), the entries of $\bm{Q}_k$ can be shown to satisfy:
\begin{align}
Q_{k+1(ij)}=&\ \left(1-k^{-\upalpha}\left[\tilde{U}_{ii}+\tilde{U}_{jj}-k^{-\upalpha} \tilde{U}_{ii}\hspace{.2em} \tilde{U}_{jj}\right]\right)Q_{k(ij)}\notag\\
&-k^{-\upalpha}\left[\sum_{\ell=j+1}^{p} \tilde{U}_{j\ell}\hspace{.2em} Q_{k(i\ell)}+\sum_{\ell=i+1}^{p} \tilde{U}_{i\ell}\hspace{.2em} Q_{k(\ell j)}\right]+k^{-\upalpha}\tilde{\Sigma}_{ij}\notag\\
&+k^{-2\upalpha}\sum_{\ell=i}^p \tilde{U}_{i\ell}\sum_{s=j+1}^p\tilde{U}_{js}\hspace{.2em} Q_{\ell s}+k^{-2\upalpha}\sum_{\ell=i+1}^p \tilde{U}_{i\ell}\hspace{.2em} \tilde{U}_{jj}\hspace{.2em} Q_{\ell j}.
\label{eq:shortybrowlowsky8}
\end{align}
Note that for the special case where $i=j=p$ we have:
\begin{align*}
Q_{k+1(pp)}=&\ \left(1-k^{-\upalpha}\left[\tilde{U}_{pp}+\tilde{U}_{pp}-o(1)\right]\right)Q_{k(pp)}+k^{-\upalpha}\tilde\Sigma_{pp}.
\end{align*}
Then, Lemma 4.2 (Fabian 1967) implies that $Q_{pp}=\tilde{\Sigma}_{pp}/[\tilde{U}_{pp}+\tilde{U}_{pp}]$. More generally, assume $Q_{mn}$ has already been computed for all tuples $(m, n)$ such that either $m\geq i+1$ and $n\geq j+1$, $m=i$ and $n\geq j+1$, or $m\geq i+1$ and $n=j$. Then, (\ref{eq:shortybrowlowsky8}) implies:
\begin{align*}
Q_{k+1(ij)}=&\ (1-k^{-\upalpha}[\tilde{U}_{ii}+\tilde{U}_{jj}-o(1)])Q_{k(ij)}\notag\\
&-k^{-\upalpha}\left[\sum_{\ell=j+1}^{p} \tilde{U}_{j\ell}\hspace{.2em} Q_{k(i\ell)}+\sum_{\ell=i+1}^{p} \tilde{U}_{i\ell}\hspace{.2em} Q_{k(\ell j)}-o(1)-\tilde{\Sigma}_{ij}\right].
\end{align*}
 Then, Lemma 4.2 in Fabian (1967) implies $\bm{Q}=\lim_{k\rightarrow \infty}\bm{Q}_k$ satisfies (\ref{eq:nolookback88}). Therefore, the characteristic function of $\bm{W}_k'$ converges to the characteristic function of a multivariate normal random variable with mean zero and covariance $\bm{Q}$. This implies the desired result.
 \end{proof}
The following theorem uses the result of Lemmas \ref{claim:coughsyrup8} and \ref{eq:losorto8} to compute the asymptotic distribution of $k^{\upbeta/2}\bm{W}_k$.

  \begin{theorem}
[This is identical to Theorem \ref{thm:generalizefabian}]
\label{thm:generalizefabian8}
Assume the recursion for ${\bm{W}}_k$ given in (\ref{eq:moon8}) satisfies conditions B0$'$--B6$'$.  Then, the asymptotic distribution of $k^{\upbeta/2}\bm{W}_k$ is a multivariate normal random variable with mean $\bm{S}\bm\upnu$ and covariance matrix $\bm{SQS}^\top$, where the entries of $\bm\upnu$ are the unique solution to (\ref{eq:okayfine8}) and the entries of $\bm{Q}$ are the unique solution to (\ref{eq:nolookback88}).
 Entry $Q_{ij}$ can be computed once $Q_{mn}$ has been computed for all tuples $(m, n)$ such that either $m\geq i+1$ and $n\geq j+1$, $m=i$ and $n\geq j+1$, or $m\geq i+1$ and $n=j$. Therefore, the entries of $\bm{Q}$ can be computed beginning with $Q_{pp}$ and using the symmetry of $\bm{Q}$. Similarly, the entries of $\bm\upnu$ can be computed beginning with $\upnu_p$. 
\end{theorem}

\begin{proof}
First, Lemma \ref{eq:losorto8} shows $\tilde{\bm{W}}_k'\xrightarrow{\text{\ dist\ }}\mathcal{N}(\bm{0},\bm{Q})$.
 Next, Lemma \ref{claim:coughsyrup8} implies that $\tilde{\bm{W}}_k\xrightarrow{\text{\ dist\ }}\mathcal{N}(\bm\upnu,\bm{Q})$. 
Finally, by the discussion following (\ref{eq:linkssfton8}) we see that
$\bm{W}_k\xrightarrow{\text{\ dist\ }}\mathcal{N}(\bm{S}\bm\upnu,\bm{S}\bm{Q}\bm{S}^\top)$ as desired.
\end{proof}

Note that if $\bm\Gamma$ is real and symmetric then the terms in square brackets in (\ref{eq:okayfine8}) and (\ref{eq:nolookback88}) disappear since $\tilde{\bm{U}}$ may be taken to be a diagonal matrix.

%% file: AppendixSensor.tex
\chapter[System Identification for Multi-Sensor Data Fusion]{System Identification\\ for Multi-Sensor Data Fusion}
\label{appen:systemid}
\chaptermark{System Identification for Sensor Data Fusion}

This chapter considers the problem of determining the presence and location of a static object within an area of interest (AOI) by combining information from multiple sensors.{\footnote{This chapter is largely based on the paper by Hernandez (2015)\nocite{hernandez2015CISS}. \copyright \ 2015 IEEE. Personal use of this material is permitted. Permission from IEEE must be
obtained for all other uses, in any current or future media, including
reprinting/republishing this material for advertising or promotional purposes, creating new
collective works, for resale or redistribution to servers or lists, or reuse of any copyrighted
component of this work in other works.}} A simple maximum-likelihood (ML) approach is investigated. We consider a setting in which there exist two main types of sensors, namely: ``small'' and ``large'' sensors. Essentially, it is assumed that small sensors can inspect an area that is relatively small in comparison to that which the large sensor can inspect. By deriving a relationship between small and large sensor measurements we combine data using the aforementioned ML-based approach. In particular, each detection problem is initially formulated as a system identification problem. Here, the large sensor collects data on the full system while small sensors collect data on subsystems. By establishing a connection of this identification problem to existing literature, we can obtain asymptotic convergence and asymptotic normality results.  It is important to mention that the terms ``object'' and ``AOI'' will be used in a very general sense. Two examples will be considered. The first example is the search for a static object within a given area (e.g., Chung and Burdick 2012)\nocite{chungnburdick2012}. Here, the terms object and AOI can be taken literally. The second example, however, is the detection of faulty tanks in a three-tank system (TTS) (e.g., Zhou et al. 2012)\nocite{zhouLeakage2012}. In this case, detecting a faulty tank within the TTS can also be thought of as detecting an object within an AOI. 

  \begin{figure}[!t]
 \centering
\begin{tikzpicture}

  \definecolor{newblue}{RGB}{41, 67, 86}
 
  \definecolor{mycoloring}{RGB}{183,208,225}
   \definecolor{mynewcoloring}{RGB}{215,230,244}
     \definecolor{mygrey}{RGB}{179, 179, 179}

\draw[thick, dashed] (0,0) circle (2cm);

\draw[thick, fill=mygrey] (-1,-0.5) circle (0.15cm);

\draw[thick, fill=mygrey] (1,1) circle (0.15cm);

\draw[thick, fill=mygrey] (0.7,-0.5) circle (0.15cm);

\draw[thick, fill=mygrey] (0.-0.3,0.5) circle (0.15cm);

\end{tikzpicture}
\caption[An illustration of the fact that the area of interest (AOI) can be seen as the union of the $C_i$s]{The large sensor is capable of inspecting the entire area within the dotted region (the AOI) while each of the small sensors inspects one of the shaded regions. In the notation of this chapter, each of the shaded regions corresponds to a $C_i$. Without loss of generality we assume $p=5$ with one cell left unexplored (the complement of the small sensor search areas within the AOI) which is denoted by $C_5$. Therefore, the $\text{AOI}=\cup_{i=1}^5 C_i$.}
\label{eq:soupsonsoupydelicious}
\end{figure}
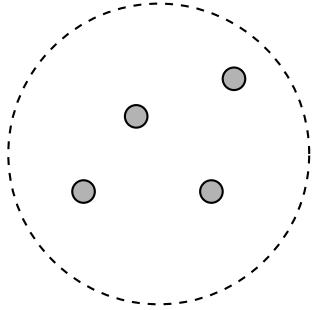

An important assumption in our proposed approach is that there are two classes of sensors: a large sensor $S$ capable of searching the entire AOI for an object and a set of small sensors $\{S_1,\dots,S_p\}$. It is assumed that the small sensors can (collectively) search only a subset of the AOI for the object (Figure \ref{eq:soupsonsoupydelicious}). In general, the measurements collected from $S$ are allowed to follow any exponential family distribution while measurements from a small sensor $S_i$ will be required to follow a Bernoulli distribution. 
The idea of sensor output following a Bernoulli distribution can be a natural assumption. Some sensors produce measurements with a natural Bernoulli interpretation. EGO sensors, for example, measure the air-to-fuel ratio in the exhaust gas and determine when this ratio crosses a threshold. As another example, Kim et al. (2005)\nocite{kimsensormodel2005} consider binary proximity sensors with varying ranges. Other times, sensor output is quantized and transmitted to a data fusion center in the form of binary data (Shen et al. 2014)\nocite{shenetal2014robust}.

To combine the information obtained from all sensors we propose a ML-based approach. In particular, it is assumed that $n_i$ i.i.d. measurements are collected from $S_i$ and that $n$ i.i.d. measurements are collected from $S$. Estimating the parameters governing measurement distributions will give us an indication of whether an object is present or not; details are given in Section \ref{sec:imthemoethossecs}. Based on the theory by Spall (2014)\nocite{spall2014sensor}, Maranzano and Spall (2011)\nocite{maranzanonspall2011}, and Spall (2009)\nocite{spall2009}, convergence and asymptotic normality results are readily available. It is important to mention that the theory does not require sample sizes to be the same for all sensors (nor to increase at the same rate).

In practice, the existence of small and large sensors may be a reasonable assumption to make. For example, there exist sensors with different ranges (e.g., Kim et al. 2005)\nocite{kimsensormodel2005} which allows for a natural interpretation of small vs. large sensors. As another (more abstract) example, Duffield (2005)\nocite{duffield2006} considers a network tomography problem where the objective is to detect faulty links between nodes. Sampling from the large sensor here could correspond to collecting measurements of global network performance (i.e., whether the network can be classified as faulty or not). For example, it might be possible to detect whether there is an unusual mount of messages being transmitted between two (not necessarily adjacent) nodes. A faulty system might then correspond to one where too many messages are reaching a given node. The state of the full system could then be modeled using a Bernoulli random variable. Furthermore, each link in the network could also be tested to determine whether it is functioning properly; tests on individual links would correspond to small sensor measurements. Similar to the fault detection problem in an TTS, several applications to system reliability could fit into this small-- and large sensor setting (e.g., Spall 2009)\nocite{spall2009}.

In the remainder of this chapter: Section \ref{sec:rubbisssssshhhh} introduces the general detection problem and the two motivating applications, Section \ref{sec:imthemoethossecs} discusses how how sensor information is combined and used for detection, Section \ref{sec:sensornumerics} contains numerical results, and Section \ref{sec:discussionsensosrs} contains some final comments.

\section{Preliminaries}
\label{sec:rubbisssssshhhh}

 Two main detection examples will be treated. The first is a search problem where the objective is to test for the presence and location of a static object. The second problem is a fault detection problem for an TTS. It is important to note that the main detection problem is more general than either of these examples and is described next.

\subsection{The General Detection Problem}
\label{sec:youre2a}

 In general AOI can be thought of as a grid with $p$ disjoint cells $\{C_1,\dots,C_p\}$. Cell $C_i$ may be inspected by sensor $S_i$ while the large sensor $S$ may inspect the entire AOI. Additionally, one of the assumptions made is that measurements from $S_i$ follow a Bernoulli distribution. Letting $\uprho_i$ be the parameter governing said distribution we define the vector $\bm\uptheta\equiv[\uprho_1,\dots,\uprho_p]^\top$. Similarly, if $Y$ is a measurement from sensor $S$ then $\uprho\equiv E[Y]$. It is assumed that $n_i$ i.i.d. samples are collected from $S_i$ and $n$ i.i.d. samples are collected from $S$.   Furthermore, samples ontained from different sensors are independent of each other. We let $X_i$ denote the sum of the $n_i$ measurements from $S_i$ and $Y_k$ denote the $k$th measurement from $S$ (for $k=1,\dots,n$).  The goal is to detect the presence (or absence) of a static object within the AOI. Specifically, we would like to know whether an object is present in any of the given cells. The two detection examples are described next.

  \subsection{The Search Problem}
  \label{sec:searchproblemsection}

Here the objective is to determine whether certain static object is present within a given search area. In this scenario the AOI is a physical space (for simplicity we let it be two-dimensional). It should be noted that the grid representation of the AOI is merely conceptual and the location of the different sensors may resemble
Figure \ref{eq:soupsonsoupydelicious}. It is, however, assumed that cells do not intersect; a
reasonable assumption when the small sensors are placed far
enough apart. Setting sensors so that they are situated away
from each other leads to situations where some cell is never
inspected; this is permitted by the asymptotic theory in Spall
(2014)\nocite{spall2014sensor}. Without loss of generality we can then assume:
\begin{align*}
\bigcup_{i=1}^{p} C_i=\text{AOI}.
\end{align*}
The case where the union of all cells is not equal to the AOI will not be considered. Observe that if an object is located within the AOI, then it is located in at least one cell and vice versa. Here the distribution of $Y$ (a measurement from the large sensor $S$) is also assumed to have a Bernoulli distribution.

The second example to consider is the detection of leaks in an TTS. The TTS considered by Zhou et al. (2012)\nocite{zhouLeakage2012} is as follows: three fluid-filled tanks ($T_1, T_2,$ and  $T_3$) are connected in series via a set of pipes. Water flows from $T_1$ to $T_2$ to $T_3$ and finally exits into a reservoir. All tanks are $T_1$ and $T_3$ are equipped with pressure sensors whose measurements are converted into liquid levels. Furthermore, two pumps, $Q_1$ and $Q_3$, are fitted into the same two tanks and provide the flow rate. All tanks are fitted with adjustable valves to simulate clogs and leaks. The problem is then that of leakage-fault diagnosis (the detection of leaky tanks in the system) based on liquid level observations. The task is designing a test for fault detection that is sensitive to the fault and yet tolerant to noise and errors in the system's dynamics. 

%
%
%
%
%
%
%
%
%

\subsection{The Leakage Fault Diagnosis Problem}
\label{sec:fluxybuxy}

Liquid levels alone are not sufficient for detecting leaky tanks; the dynamics governing the levels must be modeled. Specifically, let $h_k^{(i)}$ denote the liquid level of tank $T_i$ at time $k$ and define:
\begin{align*}
\setstretch{1.5} 
\bm{x}_k\equiv \left[\begin{array}{c}h_k^{(1)} \\h_k^{(2)} \\h_k^{(3)}\end{array}\right], \ \ \ \bm{y}_k\equiv  \left[\begin{array}{ccc}1 & 0 & 0 \\0 & 0 & 1\end{array}\right]\bm{x}_k
\end{align*}
where $\bm{y}_k$ denotes the observable liquid levels (only $T_1$ and $T_3$ are equipped with pressure sensors). The discrete-time dynamics of the state vector $\bm{x}_k$ are modeled according to:
\begin{equation}
\begin{split}
\label{eq:model}
\bm{x}_{k+1}=\bm{A}\bm{x}_k+\bm{B}_u\bm{u}_k+ \bm{B}_d\bm{d}_k+\bm{B}_f\bm{f}_k
\end{split}
\end{equation}
where $\bm{A}, \bm{B}_u,\bm{B}_d, \bm{B}_f$ are matrices, $\bm{u}_k$ is a set of controls (essentially this controls the flow rate of $Q_1$ and $Q_3$), $\bm{d}_k$ is a disturbance caused by liquid fluctuations and $\bm{f}_k$ is a possible leakage fault (if $\bm{f}_k=\bm{0}$ then there are no faults in the system). $\bm{x}_k, \bm{d}_k$ and $\bm{f}_k$ are three-dimensional vectors while $\bm{u}_k$ is two-dimensional (there are only two pumps). Theoretically, one could compare the observed state vector $\bm{x}_{k+1}$ to a predicted $\tilde{\bm{x}}_{k+1}$ generated using (\ref{eq:model}) with $\bm{f}_k=\bm{0}$. If the error between the predicted and observed state-vectors is too large it would give an indication that there exists a fault somewhere in the system. This is the general idea behind the work in Zhou et al. (2012)\nocite{zhouLeakage2012}. The authors give a precise definition if the error between predicted and observed states. This error is referred to as the {\it{residual}}. The authors also provide specific thresholds that determine when the residual is too large. Furthermore, it is assumed that at most one tank is faulty at any given time and a precise threshold-based methodology for detecting the faulty tank is given.

The relationship between the work by Zhou et al. (2012)\nocite{zhouLeakage2012} and our search problem can now be illustrated. At a given time the residual is tested against a threshold. If the residual exceeds the threshold then the system is classified as faulty (this is analogous to a large sensor measurement). Furthermore, when attempting to locate the faulty tank a three-dimensional vector $\bm{v}$ of Bernoulli responses is produced. For example, if $T_1$ is classified as faulty then $\bm{v}=[v^{(1)},v^{(2)},v^{(3)}]^\top=[1,0,0]^\top$. Note, however, that the $v^{(i)}$ are not independent of each other due to the assumption that at most one tank is faulty. Recall that in this paper we require samples from different small sensor to be independent. For this reason, we assume any single tank may be faulty independently of whether the other two tanks are faulty. 

To illustrate our proposed methodology for combining sensor information we consider a simplified problem. We will assume that all tanks are equipped with pressure sensors and that the state-vector $\bm{x}_k$ evolves according to (\ref{eq:model}) with $\bm{d}_k=\bm{0}$:
\begin{equation}
\begin{split}
\label{eq:model2}
\bm{x}_{k+1}=\bm{A}\bm{x}_k+\bm{B}_u\bm{u}_k+\bm{B}_f\bm{f}_k.
\end{split}
\end{equation}
 In addition, we let $\bm{z}_k$ be an observable random variable such that:
 \begin{align*}
 \bm{z}_k=\left[\begin{array}{c}z_k^{(1)} \\z_k^{(2)} \\z_k^{(3)}\end{array}\right]\equiv\bm{x}_k+\bm\upvarepsilon_k
 \end{align*}
where $\bm\upvarepsilon_k\sim \mathcal{N}(\bm{0},\bm\Sigma)$. Lastly, $\bm{x}_{k}$ is not observable.

If we assume that $\bm{f}_k=\bm{0}$ and that $\bm{u}_k$ is known and constant during a given time period (constant for $k\in T\equiv \{t,\dots,t+\uptau\}$ for some $t, \uptau\geq 1$) it is easy to obtain confidence intervals for $z_{k+1}^{(i)}$ and for $\ell_{k+1}\equiv\sum_{i=1}^3z_{k+1}^{(i)}$. If $\ell_{k+1}$ lies outside its confidence interval we conclude $\bm{f}_k\neq \bm{0}$ and a fault is detected. Furthermore, if $h_{k+1}^{(i)}$ lies outside its confidence interval then a fault is detected in tank $T_i$. Note that for all $k\in T$ the probability of setting of the fault detection alarm is constant whenever $\bm{f}_k=\bm{0}$. However, the same cannot be said when $\bm{f}_k\neq \bm{0}$. For this reason, we cannot assume that $\uprho_i$ is independent of $k$. Motivated by the fact that the final goal is that of classification and not estimation, we choose to ignore this fact; our numerical experiments are performed assuming $\uprho_i$ is constant for $k\in T$.

In the notation of the general detection problem (Section \ref{sec:youre2a}) we define $Y_k$ as the indicator of whether $\ell_k$ is outside its confidence interval and let $X_i$ be the number of times a fault was detected in $T_i$ for $k\in T$. Therefore, $n=n_i=|T|$.

\section{Methodology}
\label{sec:imthemoethossecs}
The ML based approach for detection consists of estimating the true parameters $\bm\uptheta^\ast$ and $\uprho^\ast$ and then using these estimates for detection. Establishing a  relationship between estimation and object detection relies on an important assumption: $\uprho_i>0.5$ whenever the object is present in $C_i$ and $\uprho_i \leq 0.5$ otherwise. This bound is an arbitrary but intuitive choice. The bound implies that if a sensor appears to detect an object more often then not then it must be that an object is actually present. Without any other information about $S_i$ this is the least we can hope for. If $Y$ also follows a Bernoulli distribution (e.g., the search problem from Section \ref{sec:searchproblemsection} and the detection problem from Section \ref{sec:fluxybuxy}) then the same logic applies to $\uprho$. In the two examples to be considered here, the false positive (FP) and true positive (TP) rates are assumed known for all sensors. In this, case an object is said to be detected in $C_i$ if the estimated value of $\uprho_i$ is closer to the TP than to the FP rate of $S_i$.

To further illustrate the basic idea behind our classification task, consider two coins, the first coin has a known probability of heads equal to 0.8 while the second has a known probability of heads equal to 0.3. If one is to choose a single coin with equal probability then the unconditional probability of heads after a single flip is 0.55. However, when collecting data from sample tosses it is implicit we must first choose a coin. As sample sizes increase the MLE of the data will converge to either 0.8 or 0.3. Therefore, once the data has been collected we can use this information to infer which coin has been chosen. Here, choosing the first coin may be analogous to having an object be present while choosing the second is the complementary situation.
 
 Given the collection of measurements obtained from the different sensors, estimating $\bm\uptheta$ and $\uprho$ could be done using two fundamentally different approaches:
 \begin{itemize}
\item {\bf{Independent Estimation:}} Here $\uprho$ as well as each of the parameters $\{\uprho_1,\dots,\uprho_p\}$ are estimated independently.
\item{\bf{Joint Estimation:}} A relationship between $\uprho$ and $\bm\uptheta$ is exploited so that joint parameter estimation can be used.
\end{itemize}
 Here we will focus on the case of joint parameter estimation. Joint estimation will allow us to obtain detection decisions that are consistent between the small and large sensors. As an example, consider the search problem (Section \ref{sec:searchproblemsection}) where $Y\sim \text{Bernoulli}(\uprho)$ and with $n_i=n=1$ for all $i$. Furthermore, assume all of the small sensors failed to detect an object while the large sensor succeeded in detecting it (i.e., $X_i=0$ for all $i$ and $Y_1=1$). The ML estimates when using independent estimation would be $\hat{\bm\uptheta}_{\text{ind}}=[0,\dots,0]^\top$ and $\hat{\uprho}_{\text{ind}}=1$ (the subindex ``$\text{ind}$'' is used to indicate that independent estimation was used). This might lead us to conclude that an object is present within the AOI but not located within any of the cells which is a contradiction. Using joint estimation implies we wish to find and exploit a relationship between $\uprho$ and $\bm\uptheta$. In particular, we will assume there exists a function $h$ such that $\uprho=h(\bm\uptheta)$. We will derive this function for two special cases. Next we give the details behind the joint estimation approach.
 
  Since the distribution $p(Y=y|\uprho)$ of $Y$ belongs to the exponential family, we have $p(Y=y|\uprho)=\exp[a(\uprho)+b(\uprho)+c(y)]$
for some functions $a(\uprho)$, $b(\uprho)$ and $c(y)$.  In addition, the log-likelihood function with respect to $\bm\uptheta$ is given by:
\begin{align*}
&\mathcal{L}(\bm\uptheta)=\sum_{k=1}^n [a(\uprho)Y_k+b(\uprho)]+\sum_{i=1}^p[X_i\log (\uprho_i)+(n_i-X_i)\log(1-\uprho_i)]+ {\textnormal{constant}},
\end{align*}
where the constant term does not depend on $\bm\uptheta$. In some situations the MLE for $\bm\uptheta$ can be obtained by finding the roots of the score function (the score function is the gradient of the log-likelihood function):
\begin{align}
\label{eq:score}
\frac{\mathcal{L}(\bm\uptheta)}{\partial \bm\uptheta}=\left[\sum_{k=1}^na'(\uprho)(Y_k-\uprho)\right]\bm{h}'(\bm\uptheta)+\left[\begin{array}{c}\frac{X_1}{\uprho_1}-\frac{n_1-X_1}{1-\uprho_1} \\\vdots \\\frac{X_p}{\uprho_p}-\frac{n_p-X_p}{1-\uprho_p}\end{array}\right].
\end{align}
However, it is not generally true that the root of the score equation coincides with the MLE or that the root necessarily exists nor is unique. Still, it is common practice to solve the score equation to attempt to obtain the MLE. This is the approach taken here. Next we briefly review some of the existing theory regarding a solution to (\ref{eq:score}).


Define $N\equiv n+n_1+\cdots+n_p$ and let $\hat{\bm\uptheta}^{(N)}$ be the value of $\bm\uptheta$ that solves (\ref{eq:score}) assuming it equals the MLE; the value of $\hat{\bm\uptheta}^{(N)}$ can be obtained by using the relationship $\uprho=h(\bm\uptheta)$. Let $\bm\uptheta^\ast$ and $\uprho^\ast$ be the true parameters to be estimated. Furthermore, define $\hat{\uprho}^{(N)}\equiv h(\hat{\bm\uptheta}^{(N)})$. Spall (2014)\nocite{spall2014sensor} shows that under some regularity conditions we have the following results: Spall (2014 Theorem 3.1)\nocite{spall2014sensor}  defines $n_s$ as the slowest increasing sample size among the small sensors and shows that if $n+n_s\rightarrow \infty$ then $\hat{\uprho}^{(N)}\rightarrow \uprho^\ast$ w.p.1. A similar result (Spall 2014, Theorem 3.2)\nocite{spall2014sensor} establishes the convergence w.p.1 of $\hat{\bm\uptheta}^{(N)}$ to $\bm\uptheta^\ast$ whenever $n_i\rightarrow \infty$ for at least $p-1$ of the small sensors. Asymptotic normality results are also presented in (Spall 2014, Theorems 4.1 and 4.2)\nocite{spall2014sensor}. Specifically,
$\sqrt{n^+}(\hat{\uprho}^{(N)}-\uprho^\ast)\xrightarrow{\text{\ dist\ }} N\left(0,\sigma(n_{.s}^\ast)^2\right)$ as $N\rightarrow \infty$ where $n^+\equiv n+n_s$;
we refer the reader to (Spall 2014)\nocite{spall2014sensor} for the definition of $\sigma(\cdot)$ and $n_{.s^\ast}$. Similarly:
\begin{align*}
\bm{S}_{N}(\hat{\bm\uptheta}^{(N)}-\bm\uptheta^\ast)\xrightarrow{\text{\ dist\ }}N(\bm{0},{\bm{\Sigma}_{\bm\uptheta^\ast}})
\end{align*}
as $N\rightarrow \infty$ where $\bm{S}_N$ is a matrix depending on the $n_i$ and $\bm\Sigma_{\bm\uptheta^\ast}$ is a limiting function of $\bm{S}_N$ and of the Fisher information matrix at $\bm\uptheta^\ast$. The reason we include these results here is to highlight the dependence on the sample sizes (which may vary among sensors) and the rate at which they increase. 


An important part in implementing ML estimation via (\ref{eq:score}) is the identification of the function $h$ relating $\bm\uptheta$ to $\uprho$. This function is problem dependent. Next we show how to obtain $h$ for the search-  and fault detection problems introduced in Sections (Section \ref{sec:searchproblemsection}) and (Section \ref{sec:fluxybuxy}), respectively.

\subsection[Determining $h$ for the search problem]{Determining $h$ for the search problem (Section \ref{sec:searchproblemsection})}

Here, the large sensor $S$ has a ${\textnormal{Bernoulli}}(\uprho)$ distribution where $0<\uprho<1$. Let $\uprho_i^{\text{FP}}, \uprho_i^{\text{TP}},\uprho^{\text{FP}}$ and $\uprho^{\text{TP}}$ denote the FP and TP rates for sensors $S_i$ and $S$. We have the following relationships:
\begin{subequations}
\begin{align}\label{eq:eeples}
\uprho_i&=\uppi_i \uprho_i^{\textnormal{TP}}+(1-\uppi_i)\uprho_i^{\textnormal{FP}},\\
\label{eq:ese}
\uprho&= \uppi\uprho^{\textnormal{TP}}+(1-\uppi)\uprho^{\textnormal{FP}}.
\end{align}
\end{subequations}
where $\uppi_i$ is the indicator function of whether an object is located in $C_i$ and $\uppi$ is the indicator function of whether and object is located somewhere within the AOI (see Table \ref{tab:notation}). The TP and FP rates are assumed known, an assumption often satisfied when there exists a probabilistic model (such as for the aforementioned acoustic proximity sensors). Alternatively, these rates may also be obtained via experimentation. Niu et al. (2005)\nocite{niuetal2005} address the issue of estimating false positive detection rates for a particular data fusion technique.
\begin{table}
\centering
\label{tab:onlytable}
\begin{tabular}{r c p{10cm} }
\toprule
$\{\uprho_i^{\text{FP}}, \uprho_i^{\text{TP}}\}$ & $\equiv$ & False positive \& true positive rates of $S_i$.\\
$\uppi_i$ &$\equiv$& Indicator of whether an object is present in $C_i$.\\
$\{\uprho^{\text{FP}}, \uprho^{\text{TP}}\}$ &$\equiv$& False positive \& true positive rates of $S$.\\
$\uppi$ &$\equiv$& Indicator of whether an object is present in the AOI.\\
$\bm\uptheta$ &$\equiv$& $[\uprho_1,\dots,\uprho_p]^\top$.\\
$\uprho_i$ &$\equiv$& Mean of the distribution of measurements from $S_i$.\\
$\uprho$ &$\equiv$& Mean of the distribution of measurements from $S$.\\
$\bm\uptheta^\ast$ &$\equiv$& True small sensor parameter vector.\\
$\uprho^\ast$ &$\equiv$& True large sensor parameter.\\
\bottomrule
\end{tabular}
\caption[A list of useful notation for Appendix \ref{appen:systemid}]{A list of useful notation.}
\label{tab:notation}
\end{table}

Via a simple manipulation of (\ref{eq:eeples}) we have:
\begin{align}
\label{eq:savy}
\uppi_i=\frac{\uprho_i-\uprho_i^{\textnormal{FP}}}{\uprho_i^{\textnormal{TP}}-\uprho_i^{\textnormal{FP}}};
\end{align}
we will also assume that $\uprho_i^{\textnormal{TP}}\neq\uprho_i^{\textnormal{FP}}$ and that $\uprho^{\textnormal{TP}}\neq\uprho^{\textnormal{FP}}$, otherwise sensors would not be providing the information we require for our final classification task.

Next, the indicator function of an object being in the entire search area is equal to $\uppi=1-\prod_{i=1}^{p}(1-\uppi_i)$
so that using (\ref{eq:savy}) we obtain the following relationship between $\uppi$, $\uprho_i^{\textnormal{TP}}$, and $\uprho_i^{\textnormal{FP}}$:
\begin{align*}
\uppi&=1-\prod_{i=1}^p\left(1-\frac{\uprho_i-\uprho_i^{\textnormal{FP}}}{\uprho_i^{\textnormal{TP}}-\uprho_i^{\textnormal{FP}}}\right).
\end{align*}
Consequently, combining the previous result with (\ref{eq:ese}) gives the desired $h$:
\begin{equation}
\begin{split}
\label{eq:htheta1}
\uprho
&={\underbrace{\uprho^{\textnormal{TP}}+(\uprho^{\textnormal{FP}}-\uprho^{\textnormal{TP}})\prod_{i=1}^p\left(1-\frac{\uprho_i-\uprho_i^{\textnormal{FP}}}{\uprho_i^{\textnormal{TP}}-\uprho_i^{\textnormal{FP}}}\right)}_{\equiv h(\bm\uptheta)}}.
\end{split}
\end{equation}

\subsection[Determining $h$ for the fault detection problem]{Determining $h$ for the fault detection problem (Section \ref{sec:fluxybuxy})}

Finding a function to relate $\uprho$ to $\bm\uptheta$ is complicated. However, if the TP rates are small, (\ref{eq:htheta1}) yields the approximation:
\begin{equation}
\begin{split}
\label{eq:htheta2}
\uprho\approx{\underbrace{\uprho^{\textnormal{FP}}\prod_{i=1}^p\left(1+\frac{\uprho_i-\uprho_i^{\textnormal{FP}}}{\uprho_i^{\textnormal{FP}}}\right)}_{\equiv h(\bm\uptheta)}}.
\end{split}
\end{equation}
Therefore, (\ref{eq:htheta2}) is the approximation to be used in the numerical experiments. It is important to notice that the true positive rate may not be small enough for the approximation to hold. This is something we investigate numerically.

\section{Numerical Analysis}
\label{sec:sensornumerics}

For the search problem two disjoint cells constituted the AOI.  ($p=2$). Each of the cells corresponds to one of the $C_i$s. An object was assumed to be located within cell $C_1$. Furthermore, the FP and FN rates were assumed equal for all of the small sensors. Sample sizes of $n=n_i=5,10,15,20$ were used and four estimates were produced: $\hat{\bm\uptheta}^{(N)}$, $\hat{\uprho}^{(N)}$, $\hat{\bm\uptheta}_{\text{ind}}$, and $\hat{\uprho}_{\text{ind}}$. The first two variables are estimates for $\bm\uptheta^\ast$ and $\uprho^\ast$ obtained using joint estimation while the last two variables are estimates obtained via independent estimation. To compare performance in estimating $\uprho^\ast$ a relative error was employed:
\begin{align*}
&e^{(1)}_\uprho\equiv \frac{|\hat{\uprho}^{(N)}-\uprho^\ast|}{\uprho^\ast}, \ e^{(2)}_\uprho\equiv \frac{|\hat{\uprho}_{\text{ind}}-\uprho^\ast|}{\uprho^\ast}.
\end{align*}
Similarly, the error in estimating $\bm\uptheta^\ast$ was measured using the errors:
\begin{align*}
&e^{(1)}_{\bm\uptheta}\equiv \frac{\|\hat{\bm\uptheta}^{(N)}-\bm\uptheta^\ast\|}{\|\bm\uptheta^\ast\|}, \ e^{(2)}_{\bm\uptheta}\equiv \frac{\|\hat{\bm\uptheta}_{\text{ind}}-\bm\uptheta^\ast\|}{\|\bm\uptheta^\ast\|}.
\end{align*}
Table \ref{tab:poorhealthzelda} compares the mean (taken over 100 replications) and standard deviation (SD) of $e^{(1)}_\uprho$ and $e^{(2)}_\uprho$ for $\uprho^\ast=0.8, \uprho_1^\ast=0.6$ and $\uprho_i^\ast=0.4$ for $i\neq 1$. Furthermore, the FP and TP rates used were $\uprho^{\text{FP}}=0.2$, $\uprho_i^{\text{FP}}=0.4$, $\uprho^{\text{TP}}=0.8$ and $\uprho_i^{\text{TP}}=0.6$. It is shown that joint parameter estimation outperforms independent estimation. Table \ref{tab:poorhealthzelda} also compares the mean and SD of $e^{(1)}_{\bm\uptheta}$ and $e^{(2)}_{\bm\uptheta}$ for the same parameter settings. Once again it appears joint estimation outperforms independent estimation. Still, more extensive numerical experiments are required.
\begin{table}[t]
\centering
\begin{tabular}{SSSSSS} \toprule
    {$n=n_i$} &   { Mean of $e^{(1)}_{\uprho}$}  & { Mean of $e^{(2)}_{\uprho}$} & { SD of $e^{(1)}_{\uprho}$} & { SD of $e^{(2)}_{\uprho}$} \\ \midrule
    {5}    & {$0.09$} & {$0.18$}  & {$0.13$} & {$0.15$}\\
    {10}  & {$0.06$} & {$0.12$} & {$0.08$} &  {$0.10$}\\
    {15}    & {$0.05$}  & {$0.09$} & {$0.06$} & {$0.08$}\\
    {20}    & {$0.05$}  & {$0.09$} & {$0.06$} & {$0.08$}      \\ \midrule
    {$n=n_i$} & {Mean of $e^{(1)}_{\bm\uptheta}$}  & {Mean of $e^{(2)}_{\bm\uptheta}$} & {SD of $e^{(1)}_{\bm\uptheta}$} & {SD of $e^{(2)}_{\bm\uptheta}$} \\ \midrule
    {5}    & {$0.16$} & {$0.30$} & {$0.13$} & {$0.18$}\\
    {10}  & {$0.12$} & {$0.24$} & {$0.11$} & {$0.13$}\\
    {15}    & {$0.11$}  & {$0.20$} & {$0.12$} & {$0.11$}\\
    {20}    & {$0.10$}  & {$0.18$} & {$0.10$} & {$0.10$}     \\ \bottomrule
\end{tabular}
\caption[Mean and standard devition of $e^{(1)}_{\uprho}$, $e^{(2)}_{\uprho}$, $e^{(1)}_{\bm\uptheta}$, and $e^{(2)}_{\bm\uptheta}$]{Mean and SD of $e^{(1)}_{\uprho}$, $e^{(2)}_{\uprho}$, $e^{(1)}_{\bm\uptheta}$, and $e^{(2)}_{\bm\uptheta}$.}
\label{tab:poorhealthzelda}
\end{table}

Because estimating the true parameters is important in determining if the object is present or not it appears that joint parameter estimation would improve the classification task. However, it turns out that the number of misclassified cells was (on average) comparable for both methods (i.e., joint- and independent estimation) whenever the small sensors were ``good enough''. For the TTS leak detection problem the classification and estimation errors were comparable for both methods.

\section{Concluding Remarks}
\label{sec:discussionsensosrs}

We have presented a general idea of how to combine information from small and large sensors to aid detection. For the search area problem the numerical experiments indicate that estimation of the parameters governing the sensors was improved using our methodology. Still, the task of determining whether an object was present or not in the AOI presented a similar behavior whenever the small sensors were informative enough (i.e., when the small sensors had high TP and low FP rates). However, one important thing to take into consideration is that jointly estimating the measurement parameters, as in our proposed methodology, could help infer whether an object is present in regions that are never inspected by any small sensor.  For the fault-detection problem the numerical experiments did not indicate improved estimation- or classification performance. The most likely explanation is that there is only a significant advantage when sample sizes are very small and the error rates for many sensors are large. However, for small sample sizes the score function may have multiple solutions. Further analysis could include theoretical analysis regarding whether combining information can significantly improve classification, especially when the large sensor measurements are not binary.

\newpage

%% file: notation.tex
\fancyhead{} 
\fancyhead[L]{FREQUENTLY USED NOTATION} 
\chapter*{Frequently Used Notation} 
\setcounter{chapter}{2}
\renewcommand{\thechapter}{\Alph{chapter}}%
\addcontentsline{toc}{chapter}{Frequently Used Notation} 
\label{chap:fun}

\noindent A list of the most commonly used notation arranged by category. The notation in this list is consistent throughout the entire dissertation (with the exception of Section \ref{sec:literaturesurveyrev} where some of the notation was borrowed from the references being surveyed).

\vspace{.2in}
\noindent{\bf{\large{General Notation}}}

\noindent $\equiv$ is the symbol for ``defined as''.

\noindent $A \cup B$ denotes the union of the sets $A$ and $B$.

\noindent $A \cap B$ denotes the intersection of the sets $A$ and $B$.

\noindent $S^c$ denotes the complement of the set $S$.

\noindent $\emptyset$ denotes the empty set.

\noindent $\Box$ denotes the end of a proof.

\noindent $a\choose{b}$ represents the quantity ${a!}/[{(a-b)!b!}]$.

\noindent $\mathbb{Z}^+$ denotes the set of strictly positive integers.

\noindent $\mathbb{R}^p$ denotes the Euclidean space of dimension $p$.

\noindent w.l.o.g. means ``without loss of generality''.

\noindent $\ceil{a}$ is the ceiling function applied to a real number $a$.

\noindent $\floor{a}$ is the floor function applied to a real number $a$.

\noindent $I_T$ denotes the interval $[-T,T]$ for $T> 0$.



\noindent $O(\cdot)$ denotes the standard big-{\it{O}} notation.

\noindent $o(\cdot)$ denotes the standard little-{\it{o}} notation.

\vspace{.2in}
\noindent{\bf{\large{Matrices and Vectors}}}

\noindent Bolded variables represent either vectors or matrices.

\noindent $\bm{A}^\top$ denotes the transpose of the matrix $\bm{A}$.

%
%

\noindent $\trace{(\bm{A})}$ is the trace of the matrix $\bm{A}$.

%


\noindent $\|\cdot\|$ represents the Frobenius norm.


\noindent $\bm{I}$ denotes the identity matrix of unspecified dimensions.


\noindent ${M}_{ij}$ (also $(\bm{M})_{ij}$ or $[\bm{M}]_{ij}$) denotes the $(i,j)$th entry of the matrix $\bm{M}$.

\noindent ${M}_{k(ij)}$ denotes the $(i,j)$th entry of the matrix $\bm{M}_k$.

\noindent ${v}_j$ (also $(\bm{v})_j$ or $[\bm{v}]_j$) denotes the $j$th entry of the vector $\bm{v}$.

\noindent ${v}_{k(j)}$ denotes the $j$th entry of the vector $\bm{v}_k$.

\noindent $p$ is reserved for the dimension of $\bm\uptheta$ so that $\bm\uptheta=[\uptau_1,\dots,\uptau_p]^\top\in \mathbb{R}^p$.

\vspace{.2in}
\noindent{\bf{\large{Probability}}}

\noindent $\chi\{\mathcal{E}\}$ denotes the indicator function of the event $\mathcal{E}$.

\noindent $P(\mathcal{E})$ denotes the probability of the event $\mathcal{E}$.

\noindent $E[\mathcal{X}]$ represents the expectation of the random variable $\mathcal{X}$.

\noindent $\Var{(\mathcal{X})}$ is the variance or covariance matrix of $\mathcal{X}$.

\noindent $\mathcal{N}(\bm\upmu,\bm\Sigma)$ is a normal random variable with mean $\bm\upmu$ and variance/covariance $\bm\Sigma$.

\noindent $\text{Bernoulli}(p)$ denotes a Bernoulli random variable with parameter $p$.

\noindent i.i.d. means ``independent and identically distributed''.

\noindent w.p.1 means ``with probability one''.

\noindent $\xrightarrow{\text{\ dist\ }}$ means convergence in distribution.

\noindent $\mathcal{X} \sim$ is used to indicate that the random variable $\mathcal{X}$ has certain distribution.

\noindent $(\Omega, \mathcal{F}, P)$ is a probability space with sample space $\Omega$, $\upsigma$-field $\mathcal{F}$, and measure $P$.

\vspace{.2in}
\noindent{\bf{\large{Sequences and Sums}}}

\noindent $g_n=o(h_n)$ if $h_n^{-1}g_n\rightarrow 0$ as $n\rightarrow \infty$.

\noindent $G_n=O(h_n)$ if $\|h_n^{-1}G_n\|\leq c$ for some $c\in \mathbb{R}$ and all $n$.


\noindent $\sum_{i=a}^b(\cdot)_i = 0$ whenever $b<a$.

\vspace{.2in}
\noindent{\bf{\large{Stochastic Optimization}}}

\noindent $L(\bm\uptheta)$ denotes a loss function to minimize.

\noindent $\bm\uptheta$ denotes the vector of parameters being estimated.

\noindent $\bm{g}(\bm\uptheta)$ denotes the gradient vector of $L(\bm\uptheta)$.

\noindent $\bm{H}(\bm\uptheta)$ denotes the Hessian matrix of $L(\bm\uptheta)$.

\noindent $\bm\uptheta^\ast$ denotes a minimizer of $L(\bm\uptheta)$ or a solution to $\bm{g}(\bm\uptheta)=\bm{0}$.

\noindent $\hat{\bm{\uptheta}}_k$ is the iterate produced in the $k$th iteration of an algorithm searching for $\bm\uptheta^\ast$.

\noindent $\hat{\bm{g}}_k(\bm\uptheta)$ denotes a noisy gradient measurement obtained during iteration $k+1$.

\noindent $Q(\bm\uptheta,\bm{V})$ denotes a noisy measurement of $L(\bm\uptheta)$.

\vspace{.2in}
\noindent{\bf{\large{GCSA-Specific Notation}}}


\noindent $\mathcal{S}_j$ for $j=1,\dots, d$ is a set of coordinates defining a subvector of $\bm\uptheta$ (see p. \pageref{eq:notexclusive}).

\noindent $j_k(m)$ is a random variable in $\{1,\dots,d\}$. 

\noindent $\bm{v}^{(j)}$ for $\bm{v}\in\mathbb{R}^p$ forces certain entries of $\bm{v}$ to be zero according to (\ref{eq:alonesong}).

\noindent $\bm{\hat{\bm{\uptheta}}_k}^{(I_{m,i})}$ denotes an intermediate step within an iteration (see p. \pageref{def:intermediate}).

\noindent $\bm{J}^{(j)}(\bm\uptheta)$ denotes the Jacobian of $\bm{g}^{(j)}(\bm\uptheta)$.

\noindent $a_k^{(j)}$ for $j=1,\dots,d$ denotes a gain sequence.\label{faq:GCSAspecific}


\noindent $s_k$ is the random number of blocks in the $(k+1)$st iteration.

\noindent $n_k(m)$ is the number of updates in the $m$th block of the $(k+1)$st iteration.


\noindent $\tilde{A}_{k}(m)=\sum_{j=1}^d{\mathcal{X}}\{j_k(m)=j\}\tilde{a}_k^{(j)}$ (see p. \pageref{eq:attaritapo}).

\noindent $\tilde{a}_{k+1}^{(j)}=a_0^{(j)}+\sum_{i=0}^k {\mathcal{X}}\{\upvarphi_k^{(j)}\geq i\}(a_{i+1}^{(j)}-a_{i}^{(j)})$ (see p. \pageref{eq:saropian}).

\noindent $\upvarphi_k^{(j)}=\sum_{i=0}^k {\mathcal{X}}\left\{\left(\sum_{m=1}^{s_i}{\mathcal{X}}\{j_i(m)=j\}\right)>0\right\}-1$ (see p. \pageref{eq:wheretheasaredefinedd}).

\noindent $x_k(j)= ({\tilde{a}_{k}^{(j)} }/{a_{k}})\sum_{m=1}^{s_k}{\mathcal{X}}\{j_k(m)=j\}n_k(m)$ (see p. \pageref{eq:tinkelpan}).

\noindent $\upmu_k(j) =E[x_k(j)]$ (see p. \pageref{eq:tinkelpan}).

\noindent $\upmu(j)=\lim_{k\rightarrow \infty} \upmu_k(j)$ (see p. \pageref{eq:tinkelpan}).

\noindent $X_k= \sum_{j=1}^dx_k(j)$ (see p. \pageref{eq:apoint}).

\noindent $\bm{h}_k(\bm\uptheta)= \sum_{j=1}^d \upmu_k(j)\bm{g}^{(j)}(\bm\uptheta)$ (see p. \pageref{eq:unrealantartica}).

\noindent $\bm{h}(\bm\uptheta)= \sum_{j=1}^d \upmu(j)\bm{g}^{(j)}(\bm\uptheta)$ (see p. \pageref{eq:unrealantartica}).\label{def:lastpage}

%% file: root.bbl
\begin{thebibliography}{82}
\expandafter\ifx\csname natexlab\endcsname\relax\def\natexlab#1{#1}\fi
\expandafter\ifx\csname url\endcsname\relax
  \def\url#1{{\tt #1}}\fi
\expandafter\ifx\csname urlprefix\endcsname\relax\def\urlprefix{URL }\fi

\bibitem[{Abounadi et~al.(2002)Abounadi, Bertsekas, \&
  Borkar}]{aboundaietal2002}
Abounadi, J., Bertsekas, D.~P., \& Borkar, V. (2002).
\newblock Stochastic approximation for nonexpansive maps: application to
  {Q}-learning algorithms.
\newblock {\em SIAM Journal on Control and Optimization\/}, {\em 41\/}(1), pp.
  1--22.

\bibitem[{Absil \& Kurdyka(2006)}]{ansilandkurd2006}
Absil, P.~A., \& Kurdyka, K. (2006).
\newblock On the stable equilibrium points of gradient systems.
\newblock {\em Systems \& Control Letters\/}, {\em 55\/}(7), pp. 573--577.

\bibitem[{Andrews(2006)}]{andrews2006congestion}
Andrews, M. (2006).
\newblock Joint optimization of scheduling and congestion control in
  communication networks.
\newblock In {\em Proceedings of the Conference on Information Sciences and
  Systems (CISS)\/}, (pp. 1572--1577).

\bibitem[{Benveniste et~al.(1990)Benveniste, M\'etivier, \&
  Priouret}]{benveniste1990}
Benveniste, A., M\'etivier, M., \& Priouret, P. (1990).
\newblock {\em Adaptive Algorithms and Stochastic Approximations\/}.
\newblock Springer-Verlag.

\bibitem[{Bertsekas \& Tsitsiklis(1989)}]{bertsekasandtsitsiklis1989}
Bertsekas, D.~P., \& Tsitsiklis, J.~N. (1989).
\newblock {\em Parallel and Distributed Computation: Numerical methods\/}.
\newblock Prentice-Hall.

\bibitem[{Bhatnagar(2011)}]{bhatnagar2011}
Bhatnagar, S. (2011).
\newblock The {B}orkar--{M}eyn theorem for asynchronous stochastic
  approximations.
\newblock {\em Systems \& Control Letters\/}, {\em 60\/}(7), pp. 472--478.

\bibitem[{Bhatnagar et~al.(2013)Bhatnagar, Prasad, \&
  Prashanth}]{bhatnagaretal2013}
Bhatnagar, S., Prasad, H.~L., \& Prashanth, L.~A. (2013).
\newblock {\em Stochastic Recursive Algorithms for Optimization: Simultaneous
  Perturbation Methods\/}.
\newblock Springer.

\bibitem[{Bianchi \& Jakubowicz(2013)}]{bianchiandjiakubowicz2013}
Bianchi, P., \& Jakubowicz, J. (2013).
\newblock Convergence of a multi-agent projected stochastic gradient algorithm
  for non-convex optimization.
\newblock {\em IEEE Transactions on Automatic Control\/}, {\em 58\/}(2), pp.
  391--405.

\bibitem[{Bickel \& Doksum(2007)}]{bickelndoksum2007}
Bickel, P.~J., \& Doksum, K.~A. (2007).
\newblock {\em Mathematical Statistics\/}.
\newblock New Jersey: Pearson, 2nd ed.

\bibitem[{Blum(1954)}]{blum1954}
Blum, J.~R. (1954).
\newblock Multidimensional stochastic approximation methods.
\newblock {\em Annals of Mathematical Statistics\/}, {\em 25\/}(4), pp.
  737--744.

\bibitem[{Borkar(1997)}]{borkar1997}
Borkar, V.~S. (1997).
\newblock Stochastic approximation with two time scales.
\newblock {\em Systems \& Control Letters\/}, {\em 29\/}(5), pp. 291--194.

\bibitem[{Borkar(1998)}]{borkar1998}
Borkar, V.~S. (1998).
\newblock Asynchronous stochastic approximations.
\newblock {\em SIAM Journal on Control and Optimization\/}, {\em 36\/}(3), pp.
  840--851.

\bibitem[{Borkar \& Meyn(2000)}]{borkarmeyn2000}
Borkar, V.~S., \& Meyn, S.~P. (2000).
\newblock The {ODE} method for convergence of stochastic approximation and
  reinforcement learning.
\newblock {\em SIAM Journal on Control and Optimization\/}, {\em 38\/}(2), pp.
  447--469.

\bibitem[{Botts et~al.(2016)Botts, Spall, \& Newman}]{bottsetal2016}
Botts, C.~H., Spall, J.~C., \& Newman, A.~J. (2016).
\newblock Multi-agent surveillance and tracking using cyclic stochastic
  gradient.
\newblock In {\em Proceedings of the American Control Conference (ACC)\/}, (pp.
  270--275).

\bibitem[{Box \& Wilson(1951)}]{boxwilson1951}
Box, G. E.~P., \& Wilson, K.~B. (1951).
\newblock On the experimental attainment of optimum conditions.
\newblock {\em Journal of the Royal Statistical Society\/}, {\em 13 (Series
  B)\/}, pp. 1--38.

\bibitem[{Canutescu \& Dunbrack(2003)}]{canutescu2003}
Canutescu, A.~A., \& Dunbrack, R.~L. (2003).
\newblock Cyclic coordinate descent: a robotics algorithm for protein loop
  closure.
\newblock {\em Protein Science\/}, {\em 12\/}(5), pp. 963--972.

\bibitem[{Chung(2001)}]{chung2001}
Chung, K.~L. (2001).
\newblock {\em A Course in Probability Theory\/}.
\newblock Elsevier, 3rd ed.

\bibitem[{Chung \& Burdick(2012)}]{chungnburdick2012}
Chung, T.~H., \& Burdick, J.~W. (2012).
\newblock Analysis of search decision making using probabilistic search
  strategies.
\newblock {\em IEEE Transactions on Robotics\/}, {\em 28\/}(1), pp. 132--144.

\bibitem[{Cronin(2007)}]{cronin2007}
Cronin, J. (2007).
\newblock {\em Ordinary Differential Equations: Introduction and Qualitative
  Theory\/}.
\newblock Florida: CRC Press, 3rd ed.

\bibitem[{Dennis \& Schnabel(1989)}]{dennisnschnabel1989}
Dennis, J.~E., \& Schnabel, R.~B. (1989).
\newblock A view of unconstrained optimization.
\newblock {\em Handbooks in Operations Research and Management Science\/}, {\em
  1\/}, pp. 1--72.

\bibitem[{Doob(1953)}]{doob1953}
Doob, J.~L. (1953).
\newblock {\em Numerical Optimization\/}.
\newblock New York: John Wiley \& Sons, 2nd ed.

\bibitem[{Duffield(2006)}]{duffield2006}
Duffield, N. (2006).
\newblock Network tomography of binary network performance characteristics.
\newblock {\em IEEE Transactions on Information Theory\/}, {\em 52\/}(12), pp.
  5373--5388.

\bibitem[{Fabian(1967)}]{fabian1967}
Fabian, V. (1967).
\newblock Stochastic approximation of minima with improved asymptotic speed.
\newblock {\em Annals of Mathematical Statistics\/}, {\em 38\/}(1), pp.
  191--200.

\bibitem[{Fabian(1968)}]{fabian1968}
Fabian, V. (1968).
\newblock On asymptotic normality in stochastic approximation.
\newblock {\em Annals of Mathematical Statistics\/}, {\em 39\/}(4), pp.
  1327--1332.

\bibitem[{Herceg et~al.(2013)Herceg, Kvasnica, Jones, \& Morari}]{MPT3}
Herceg, M., Kvasnica, M., Jones, C., \& Morari, M. (2013).
\newblock Multi-parametric toolbox 3.0.
\newblock In {\em Proceedings of the European Control Conference (ECC)\/}, (pp.
  502--510). Z\"urich, Switzerland.
\newblock \url{http://control.ee.ethz.ch/~mpt}.

\bibitem[{Hernandez(2015)}]{hernandez2015CISS}
Hernandez, K. (2015).
\newblock Combined sensor information for detection.
\newblock In {\em Proceedings of the Conference on Information Sciences and
  Systems (CISS)\/}, (pp. 1--6).
\newline\urlprefix\url{http://ieeexplore.ieee.org/abstract/document/7086857}

\bibitem[{Hernandez(2016)}]{hernandez2016}
Hernandez, K. (2016).
\newblock Cyclic stochastic optimization via arbitrary selection procedures for
  updating parameters.
\newblock In {\em Proceedings of the Conference on Information Sciences and
  Systems (CISS)\/}, (pp. 349--354).

\bibitem[{Hernandez \& Spall(2014)}]{hernandeznspall2014}
Hernandez, K., \& Spall, J.~C. (2014).
\newblock Cyclic stochastic optimization with noisy function measurements.
\newblock In {\em Proceedings of the American Control Conference (ACC)\/}, (pp.
  5204--5209).

\bibitem[{Hernandez \& Spall(2016)}]{hernandeznspall2016}
Hernandez, K., \& Spall, J.~C. (2016).
\newblock Asymptotic normality and efficiency analysis of the cyclic seesaw
  stochastic optimization algorithm.
\newblock In {\em Proceedings of the American Control Conference (ACC)\/}, (pp.
  7255--7260).

\bibitem[{Hernandez \& Spall(2017)}]{hernandeznspall2017}
Hernandez, K., \& Spall, J.~C. (2017).
\newblock A complete result on the asymptotic normality of stochastic
  approximation algorithms.
\newblock {\em (Under review)\/}.

\bibitem[{Hoeffding(1963)}]{hoeffding}
Hoeffding, W. (1963).
\newblock Probability inequalities for sums of bounded random variables.
\newblock {\em Journal of the American statistical association\/}, {\em
  58\/}(301), pp. 13--30.

\bibitem[{Horn \& Johnson(2010)}]{hornnjohnson2010}
Horn, R.~A., \& Johnson, C.~R. (2010).
\newblock {\em Matrix Analysis\/}.
\newblock New York: Cambridge University Press.

\bibitem[{Hu et~al.(2012)Hu, Hu, \& Chang}]{huetal2012}
Hu, J., Hu, P., \& Chang, H.~S. (2012).
\newblock A stochastic approximation framework for a class of randomized
  optimization algorithms.
\newblock {\em IEEE Transactions on Automatic Control\/}, {\em 57\/}(1), pp.
  165--178.

\bibitem[{Kar et~al.(2013)Kar, Moura, \& Poor}]{karetal2013}
Kar, S., Moura, J.~M., \& Poor, H.~V. (2013).
\newblock Distributed linear parameter estimation: asymptotically efficient
  adaptive strategies.
\newblock {\em SIAM Journal on Control and Optimization\/}, {\em 51\/}(3), pp.
  2200--2229.

\bibitem[{Kim et~al.(2005)Kim, Mechitov, Choi, \& Ham}]{kimsensormodel2005}
Kim, W., Mechitov, K., Choi, J.~Y., \& Ham, S. (2005).
\newblock On target tracking with binary proximity sensors.
\newblock In {\em Proceedings of the International Symposium on Information
  Processing in Sensor Networks (IPSN)\/}, (pp. 301--308).

\bibitem[{Knight et~al.(2013)Knight, Carson, \& Demmel}]{knightetal2013}
Knight, N., Carson, E., \& Demmel, J. (2013).
\newblock Exploiting data sparsity in parallel matrix powers computations.
\newblock In {\em Proceedings of the International Conference on Parallel
  Processing and Applied Mathematics\/}, (pp. 15--25).

\bibitem[{Konda \& Tsitsiklis(1997)}]{kondaandtsitsiklis2004}
Konda, V.~R., \& Tsitsiklis, J.~N. (1997).
\newblock Convergence rate of linear two-time-scale stochastic approximation.
\newblock {\em Annals of Applied Probability\/}, {\em 14\/}(2), pp. 796--819.

\bibitem[{Kushner \& Yin(1997)}]{kushnyin1997}
Kushner, H., \& Yin, G. (1997).
\newblock {\em Stochastic Approximation Algorithms and Applications\/}.
\newblock New York: Springer-Verlag.

\bibitem[{Kushner \& Clark(1978)}]{kushnclark1978}
Kushner, H.~J., \& Clark, D.~S. (1978).
\newblock {\em Stochastic Approximation Methods for Constrained and
  Unconstrained Systems\/}.
\newblock New York: Springer-Verlag, 2nd ed.

\bibitem[{Lee et~al.(2015)Lee, Diaz-Mercado, \&
  Egerstedt}]{leeetal2015robotics}
Lee, S., Diaz-Mercado, Y., \& Egerstedt, M. (2015).
\newblock Multirobot control using time-varying density functions.
\newblock {\em IEEE Transactions on Robotics\/}, {\em 31\/}(2), pp. 489--493.

\bibitem[{Lee \& Park(2008)}]{leenpark2008}
Lee, S., \& Park, F.~C. (2008).
\newblock Cyclic optimization algorithms for simultaneous structure and motion
  recovery in computer vision.
\newblock {\em Engineering Optimization\/}, {\em 40\/}(5), pp. 403--419.

\bibitem[{Li \& Petropuli(2014)}]{liandpetriouli2014}
Li, B., \& Petropuli, A. (2014).
\newblock Efficient target estimation in distributed {MIMO} radar via the
  {ADMM}.
\newblock In {\em Proceedings of the Conference on Information Sciences and
  Systems (CISS)\/}, (pp. 1--5).

\bibitem[{Li \& Osher(2009)}]{liandosher2009}
Li, Y., \& Osher, S. (2009).
\newblock Coordinate descent optimization for $\ell_1$ minimization with
  application to compressed sensing; a greedy algorithm.
\newblock {\em Inverse Problems and Imaging\/}, {\em 3\/}(3), pp. 487--503.

\bibitem[{Ljung(1977)}]{ljung1977}
Ljung, L. (1977).
\newblock Analysis of recursive stochastic algorithms.
\newblock {\em IEEE Transactions on Automatic Control\/}, {\em 22\/}(4), pp.
  551--575.

\bibitem[{Luo \& Tseng(1992)}]{luptseng1992}
Luo, Z.~Q., \& Tseng, P. (1992).
\newblock On the convergence of the coordinate descent method for convex
  differentiable minimization.
\newblock {\em Journal of Optimization Theory and Applications\/}, {\em
  72\/}(1), pp. 7--35.

\bibitem[{Luo \& Tseng(1993)}]{luptseng1993}
Luo, Z.~Q., \& Tseng, P. (1993).
\newblock Error bounds and convergence analysis of feasible descent methods: A
  general approach.
\newblock {\em Annals of Operations Research\/}, {\em 46\/}(1), pp. 157--178.

\bibitem[{Maranzano \& Spall(2011)}]{maranzanonspall2011}
Maranzano, C.~J., \& Spall, J.~C. (2011).
\newblock Framework for estimating system reliability from full system and
  subsystem tests with dependence on dynamic inputs.
\newblock In {\em Proceedings of the Conference on Decision and Control
  (CDC)\/}, (pp. 6666--6671).

\bibitem[{Necoara(2013)}]{necoara2013}
Necoara, I. (2013).
\newblock Random coordinate descent algorithms for multi-agent convex
  optimization over networks.
\newblock {\em IEEE Transactions on Automatic Control\/}, {\em 58\/}(8), pp.
  2001--2012.

\bibitem[{Necoara \& Petrascu(2014)}]{recoaraandpetrascu2014}
Necoara, I., \& Petrascu, A. (2014).
\newblock A random coordinate descent algorithm for optimization problems with
  composite objective function and linear coupled constraints.
\newblock {\em Computational Optimization and Applications\/}, {\em 57\/}(2),
  pp. 307--337.

\bibitem[{Nedi\'c \& Bertsekas(2010)}]{nedicandbertsekas2010}
Nedi\'c, A., \& Bertsekas, D.~P. (2010).
\newblock The effect of deterministic noise in subgradient methods.
\newblock {\em Mathematical Programming\/}, {\em 125\/}(1), pp. 75--99.

\bibitem[{Nevel'son \& Has'minskii(1973)}]{nevelsonandhas1973}
Nevel'son, M.~B., \& Has'minskii, R.~Z. (1973).
\newblock {\em Stochastic Approximation and Recursive Estimation\/}.
\newblock Rhode Island: American Mathematical Society.

\bibitem[{Niu et~al.(2005)Niu, Varshney, \& Cheng}]{niuetal2005}
Niu, R., Varshney, P.~K., \& Cheng, Q. (2005).
\newblock Distributed detection in a large wireless sensor network.
\newblock {\em Information Fusion\/}, {\em 7\/}(4), pp. 380--394.

\bibitem[{Nocedal \& Wright(2006)}]{nocedalnwright2006}
Nocedal, J., \& Wright, S.~J. (2006).
\newblock {\em Numerical Optimization\/}.
\newblock New York: Springer, 2nd ed.

\bibitem[{{\O}ksendal(2003)}]{oksendal2003}
{\O}ksendal, B. (2003).
\newblock {\em Stochastic Differential Equations: an Introduction with
  Applications\/}.
\newblock Vancouver: Springer, sixth ed.

\bibitem[{Peterson et~al.(2014)Peterson, Newman, \& Spall}]{spie2014}
Peterson, C.~K., Newman, A.~J., \& Spall, J.~C. (2014).
\newblock Simulation-based examination of the limits of performance for
  decentralized multi-agent surveillance and tracking of undersea targets.
\newblock In {\em Proceedings of the SPIE Defense + Security Conference\/},
  (pp. 90910--90915).

\bibitem[{Ram et~al.(2009a)Ram, Nedi\'c, \& Veeravalli}]{rametal2009a}
Ram, S.~S., Nedi\'c, A., \& Veeravalli, V.~V. (2009a).
\newblock Incremental stochastic subgradient algorithms for convex
  optimization.
\newblock {\em SIAM Journal on Optimization\/}, {\em 20\/}(2), pp. 691--717.

\bibitem[{Ram et~al.(2009b)Ram, Nedi\'c, \& Veervalli}]{rametal2009b}
Ram, S.~S., Nedi\'c, A., \& Veervalli, V.~V. (2009b).
\newblock Asynchronous gossip algorithms for stochastic optimization.
\newblock In {\em Proceedings of the Conference on Decision and Control held
  jointly with the 28th Chinese Control Conference (CDC/CCC)\/}, (pp.
  3581--3586).

\bibitem[{Ram et~al.(2010)Ram, Nedi\'c, \& Veervalli}]{rametal2010}
Ram, S.~S., Nedi\'c, A., \& Veervalli, V.~V. (2010).
\newblock Distributed stochastic subgradient projection algorithms for convex
  optimization.
\newblock {\em Journal of Optimization Theory and Applications\/}, {\em
  147\/}(3), pp. 516--545.

\bibitem[{Renotte \& Wouwer(2003)}]{renottenwouwer2003}
Renotte, C., \& Wouwer, A.~V. (2003).
\newblock Stochastic approximation techniques applied to parameter estimation
  in a biological model.
\newblock In {\em Proceedings of the International Workshop on Intelligent Data
  Acquisition and Advanced Computing Systems: Technology and Applications\/},
  (pp. 261--265).

\bibitem[{Robbins \& Monro(1951)}]{robbinsmonro1951}
Robbins, H., \& Monro, S. (1951).
\newblock A stochastic approximation method.
\newblock {\em Annals of Mathematical Statistics\/}, {\em 22\/}(3), pp.
  400--407.

\bibitem[{Rudin(1976)}]{Rudin1976}
Rudin, W. (1976).
\newblock {\em Principles of Mathematical Statistics\/}.
\newblock McGraw-Hill, 3rd ed.

\bibitem[{Shen et~al.(2014)Shen, Varshney, \& Zhu}]{shenetal2014robust}
Shen, X., Varshney, P.~K., \& Zhu, Y. (2014).
\newblock Robust distributed maximum likelihood estimation with dependent
  quantized data.
\newblock {\em Automatica\/}, {\em 50\/}(1), pp. 169--174.

\bibitem[{Singh et~al.(2014)Singh, Nedi\'c, \& Srikant}]{singhetal2014}
Singh, C., Nedi\'c, A., \& Srikant, R. (2014).
\newblock Random block-coordinate gradient projection algorithms.
\newblock In {\em Proceedings of the Conference on Decision and Control
  (CDC)\/}, (pp. 185--190).

\bibitem[{Solodov \& Zavriev(1998)}]{solodovzavriev1998}
Solodov, M.~V., \& Zavriev, S.~K. (1998).
\newblock Error stability properties of generalized gradient-type algorithms.
\newblock {\em Journal of Optimization Theory and Applications\/}, {\em
  98\/}(3), pp. 663--680.

\bibitem[{Sommer(Retrieved May 15 2017)}]{polygeom}
Sommer, H.~J. (Retrieved May 15 2017).
\newblock polygeom.m.
\newblock In {\em MATLAB Central File Exchange\/}.
\newblock Last updated December 9 2016.
  \url{http://www.mathworks.com/matlabcentral/fileexchange/319-polygeom-m}.

\bibitem[{Spall(1992)}]{spall1992}
Spall, J.~C. (1992).
\newblock Multivariate stochastic approximation using a simultaneous
  perturbation gradient approximation.
\newblock {\em IEEE Transactions on Automatic Control\/}, {\em 37\/}(3), pp.
  332--341.

\bibitem[{Spall(2003)}]{ISSO}
Spall, J.~C. (2003).
\newblock {\em Introduction to Stochastic Search and Optimization\/}.
\newblock New Jersey: John Wiley \& Sons.

\bibitem[{Spall(2009)}]{spall2009}
Spall, J.~C. (2009).
\newblock System reliability estimation and confidence regions from subsystem
  and full system tests.
\newblock In {\em Proceedings of the American Control Conference (ACC)\/}, (pp.
  5067--5072).

\bibitem[{Spall(2012)}]{spallcyclicseesaw2012}
Spall, J.~C. (2012).
\newblock Cyclic seesaw process for optimization and identification.
\newblock {\em Journal of Optimization Theory and Applications\/}, {\em
  154\/}(1), pp. 187--208.

\bibitem[{Spall(2014)}]{spall2014sensor}
Spall, J.~C. (2014).
\newblock Identification for systems with binary subsystems.
\newblock {\em IEEE Transactions on Automatic Control\/}, {\em 59\/}(1), pp.
  3--17.

\bibitem[{Tseng(2001)}]{tseng2001}
Tseng, P. (2001).
\newblock Convergence of a block coordinate descent method for
  non-differentiable minimization.
\newblock {\em Journal of Optimization Theory and Applications\/}, {\em
  109\/}(3), pp. 475--494.

\bibitem[{Tseng \& Yun(2009)}]{tsengandyun2009}
Tseng, P., \& Yun, S. (2009).
\newblock Block-coordinate gradient descent method for linearly constrained
  nonsmooth separable optimization.
\newblock {\em Journal of Optimization Theory and Applications\/}, {\em
  140\/}(3), pp. 513--535.

\bibitem[{Tsitsiklis(1984)}]{tsitsiklis1984}
Tsitsiklis, J.~N. (1984).
\newblock {\em Problems in Decentralized Decision making and Computation\/}.
\newblock Ph.D. thesis, Massachusetts Institute of Technology, Cambridge Lab
  for Information and Decision Systems.

\bibitem[{Tsitsiklis(1994)}]{tsitsiklis1994}
Tsitsiklis, J.~N. (1994).
\newblock Asynchronous stochastic approximation and {Q}-learning.
\newblock {\em Machine Learning\/}, {\em 16\/}(3), pp. 185--202.

\bibitem[{Tsitsiklis et~al.(1986)Tsitsiklis, Bertsekas, \&
  Athans}]{tsitsiklisetal1986}
Tsitsiklis, J.~N., Bertsekas, D.~P., \& Athans, M. (1986).
\newblock Distributed asynchronous deterministic and stochastic gradient
  optimization algorithms.
\newblock {\em IEEE Transactions on Automatic Control\/}, {\em 31\/}(9), pp.
  803--812.

\bibitem[{Wang \& Spall(2008)}]{otherwang}
Wang, I.-J., \& Spall, J.~C. (2008).
\newblock Stochastic optimisation with inequality constraints using
  simultaneous perturbations and penalty functions.
\newblock {\em International Journal of Control\/}, {\em 81\/}(8), pp.
  1232--1238.

\bibitem[{Wang \& Spall(2011)}]{dspsa}
Wang, Q., \& Spall, J.~C. (2011).
\newblock Discrete simultaneous perturbation stochastic approximation on loss
  function with noisy measurements.
\newblock In {\em Proceedings of the American Control Conference (ACC)\/}, (pp.
  4520--4525).

\bibitem[{Wright(2015)}]{wright2015}
Wright, S.~J. (2015).
\newblock Coordinate descent algorithms.
\newblock {\em Mathematical Programming\/}, {\em 151\/}(1), pp. 3--34.

\bibitem[{Xu \& Yin(2015)}]{xyandyin2015}
Xu, Y., \& Yin, W. (2015).
\newblock Block stochastic gradient iteration for convex and nonconvex
  optimization.
\newblock {\em SIAM Journal on Optimization\/}, {\em 25\/}(3), pp. 1686--1716.

\bibitem[{Zhou et~al.(2012)Zhou, He, Wang, Liu, \& Ji}]{zhouLeakage2012}
Zhou, D., He, X., Wang, Z., Liu, G., \& Ji, Y. (2012).
\newblock Leakage fault diagnosis for an internet-based three-tank system: an
  experimental study.
\newblock {\em IEEE Transactions on Control Systems Technology\/}, {\em
  20\/}(4), pp. 857--870.

\bibitem[{Zhou \& Hu(2014)}]{zhounhu2014}
Zhou, E., \& Hu, J. (2014).
\newblock Gradient-based adaptive stochastic search for non-differentiable
  optimization.
\newblock {\em IEEE Transactions on Automatic Control\/}, {\em 59\/}(7), pp.
  1818--1832.

\bibitem[{Zorin et~al.(2000)Zorin, Mayer-Proschel, Noble, \&
  Yakovlev}]{zorinetal200}
Zorin, A., Mayer-Proschel, M., Noble, M., \& Yakovlev, A.~Y. (2000).
\newblock Estimation problems associated with stochastic modeling of
  proliferation and differentiation of {O-2A} progenitor cells in vitro.
\newblock {\em Mathematical Biosciences\/}, {\em 167\/}(2), pp. 109--121.

\end{thebibliography}
